\documentclass[12pt,reqno]{amsbook}

\usepackage{amsmath,a4wide,wasysym}
\usepackage{xfrac}
\usepackage{stmaryrd,mathrsfs,bm,amsthm,mathtools,yfonts,amssymb,color,braket,booktabs,graphicx,graphics,amsfonts}
\usepackage{xcolor}
\usepackage{tikz}
\usepackage{xcolor}
\usepackage{imakeidx}
\usepackage[utf8]{inputenc}
\usepackage[T1]{fontenc}
\usepackage{lmodern}

\usepackage{esint}
\usepackage{latexsym,microtype,indentfirst}
\usepackage[pagebackref,citecolor=black,linkcolor=blue,colorlinks=true]{hyperref}
\usepackage[colorinlistoftodos]{todonotes}
\usepackage{courier}

\numberwithin{equation}{chapter}

\newtheorem{theorem}{Theorem}[section]
\newtheorem{lemma}[theorem]{Lemma}
\newtheorem{proposition}[theorem]{Proposition}
\newtheorem{corollary}[theorem]{Corollary}
\newtheorem{conjecture}[theorem]{Conjecture}

\theoremstyle{definition}
\newtheorem{definition}[theorem]{Definition}
\newtheorem{remark}[theorem]{Remark}

\newtheorem{ipotesi}[theorem]{Assumption}

\numberwithin{equation}{section}

\newcommand{\exc}{\mathrm{exc}}
\newcommand{\VV}{\mathscr{V}}
\newcommand{\ello}{\mathscr{V}_0}
\newcommand{\To}{\mathcal{T}_0}
\newcommand{\baraj}{a_j}
\newcommand{\barbj}{b_j}
\newcommand\supp{{\rm spt}\,}

\newcommand\MM{\mathcal{M}}
\newcommand\R{\mathbb{R}}
\newcommand\HH{\mathcal{H}}

\newcommand\N{{\mathbb N}}
\newcommand\e{\varepsilon}

\newcommand\dist{{\rm dist}}
\newcommand\Lip{{\rm Lip}}

\newcommand{\restr}{\mathbin{\vrule height 1.6ex depth 0pt width
0.13ex \vrule height 0.13ex depth 0pt width 1.3ex}}
\newcommand{\restrsmall}{\mathbin{\vrule height 1.1ex depth 0pt width
0.07ex \vrule height 0.07ex depth 0pt width 0.9ex}}

\newcommand\param{\lambda}

\newcommand\Rad{\mathrm{Rsq}}
\newcommand\rad{\mathrm{rsq}}
\newcommand\varthetad{\dot{\vartheta}}
\newcommand\varthetadd{\ddot{\vartheta}}

\newcommand\etad{\dot{\varrho}}
\newcommand\etadd{\ddot{\varrho}}
\newcommand\Isq{\mathrm{Isq}}
\newcommand\isq{\mathrm{isq}}

\title[Regularity theory for Mumford-Shah functional]{The regularity theory for the Mumford-Shah functional on the plane}

\author[C. De Lellis]{Camillo De Lellis}
\address{School of Mathematics, Institute for Advanced Study, 1 Einstein Dr., Princeton NJ 05840, USA}
\email{camillo.delellis@math.ias.edu}
\author[M. Focardi]{Matteo Focardi}
\address{Dipartimento di Matematica e Informatica ``Ulisse Dini'', Universit\`a degli Studi di Firenze, Viale Morgagni 67/A, 50134 Firenze, Italy}
\email{matteo.focardi@unifi.it}

\makeindex
\makeindex[name=simb,title=List of Symbols]

\begin{document}

\frontmatter

\maketitle


\hspace{7cm} {\small to Laura and Lorenza}

\medskip

\hspace{7cm} {\small for allowing us to share so much time}

\hspace{7cm} {\small in the company of the Mumford and Shah problem}

\tableofcontents

\mainmatter

\chapter*{Preface}

This book delves into the fascinating world of the Mumford-Shah functional, a cornerstone of image processing and a rich source of challenging mathematical problems. Introduced by Mumford and Shah in their seminal 1989 paper, this functional seeks to capture the essence of an image by decomposing it into a piecewise smooth function and a set of discontinuities. While its practical applications in image segmentation and reconstruction are undeniable, the functional also presents profound theoretical questions. Many of these regard the regularity of its minimizers: the central theme of this book is the exploration of these regularity questions, focusing specifically on the planar case. We embark on a comprehensive journey through the existing literature, revisiting classic results with fresh perspectives and incorporating new advancements. Our goal is to provide a self-contained and detailed account of the state-of-the-art, accessible to both aspiring researchers seeking a thorough introduction and experts interested in the latest developments. 

\medskip

The Mumford-Shah conjecture, a long-standing open problem, predicts a remarkably simple structure for the discontinuity set of minimizers: a collection of smooth curves meeting at triple junctions forming equal angles. While this conjecture remains elusive in its full generality, significant progress has been made. We present a detailed exposition of the $\varepsilon$-regularity theory, a powerful tool that establishes the local regularity of the discontinuity set under certain assumptions. This theory allows us to understand the behavior of minimizers near points that resemble the conjectured structures. Beyond $\varepsilon$-regularity, we explore the classification of global minimizers, the properties of critical points, and various structural results that shed light on the nature of the discontinuity set. 

Throughout the book we emphasize clarity and completeness, providing detailed proofs for all the presented theorems. We assume a solid background in measure theory and basic elliptic PDEs, while we provide supplementary material in the appendices to aid readers less familiar with specific technical aspects. We hope that this book will serve as a valuable resource for anyone intrigued by the beauty and complexity of the Mumford-Shah functional, inspiring further research and deeper understanding of this remarkable mathematical object. 

\chapter{Introduction}

The aim of these notes is to give a complete self-contained account of the current state of the art in the regularity for planar minimizers and critical points of the Mumford-Shah functional. {The latter is an energy functional introduced by Mumford and Shah in their seminal paper \cite{MS} on image reconstruction. Assuming that $g$ is some function on the plane whose values represent a noisy greyscale image, Mumford and Shah proposed to identify the geometric content of the image as a pair $(u,K)$, given by a set $K$ of finite length and a smooth function $u$ on the complement $K^c$ of $K$, which minimizes the functional
\[
\int_{K^c} |\nabla u|^2 + {\rm length}\, (K) + \int_{K^c} (u-g)^2\, .
\]
They also advanced the conjecture that, for such minimizers, the set $K$ is described by a collection of finitely many $C^1$ arcs which do not cross and with the property that if an endpoint belongs to more than one arc, then it is in fact a common endpoint of three arcs meeting at equal angles. This conjecture is to date still open, even though there has been considerable progress towards its resolution.

\medskip

The question of existence of minimizers has been solved in the mathematical literature by allowing suitable weaker formulations than the original one of Mumford and Shah. One approach, due to De Giorgi, is to consider $K$ as the set of ``jump discontinuities'' of $u$ and then to apply the direct methods of the calculus of variations in an appropriate functional-analytic setting to an appropriate generalization of the Mumford-Shah functional. This led De Giorgi and Ambrosio to introduce the space of ``special functions of bounded variations'' in \cite{DGA88}, and we refer the reader to the monograph \cite{AFP00} for a detailed account of the various variational problems that found proper treatment in this fruitful setting, which is not limited to the planar case. The other approach is to allow $K$ to be a general closed rectifiable set with finite Hausdorff measure and $u$ to be a $W^{1,2}$ function on its complement, and then to work out a suitable compactness theory for appropriate minimizing sequences, see the work of Dal Maso, Morel, and Solimini \cite{DME92}. The two theories are in fact equivalent thanks to a fundamental result of De Giorgi, Carriero, and Leaci, cf. \cite{DGCL89}, a fact which is also carefully explained in one of the appendices of this book.\index{SBV@SBV}\index{special functions of bounded variations@special functions of bounded variation}
\index{functions of bounded variation, special@functions of bounded variation, special}

Having established existence, the natural question is now whether minimizers fulfill the prediction of Mumford and Shah. Note that over $K^c$ the function $u$ is a minimizer of the quadratic functional $\int (|\nabla u|^2 + (u-g)^2)$ and thus a solution to a very elementary linear elliptic differential equation. Therefore, the relevant question here is how regular the set $K$ is and how $u$ behaves ``at $K$'': this is what we understand as the {\em regularity theory}. 

\medskip

Our ambition is to cover all the major results in the planar setting. However, we will not only include the most recent theorems but also give some new takes on the older ones. In several places we will present alternative proofs (e.g., we will give several simplified arguments for many portions of \cite{bonnetdavid2001}) and other parts will include material which is, strictly speaking, completely new (e.g., we will generalize the results of \cite{DFG} to include the fidelity term in the energy). 

At the same time, we will include detailed and complete arguments for all theorems. This book can therefore be used both by readers with a solid background in measure theory and elementary elliptic PDE who intend to learn the theory in detail for the first time, and by experts who are interested in the most recent developments. Concerning the first category of readers, we, however, assume that they are familiar with the notion of rectifiable $1$-dimensional sets and with their most basic properties (for instance the sections 2.8, 2.9, 2.10, 2.11, and 2.12 in \cite{AFP00}, specialized to the case of $1$-dimensional sets in the plane, are more than enough).}

\section{Setup} 
As already mentioned, the relevant energies are computed on pairs $(u, K)$\index{admissible pairs@admissible pairs}\index[simb]{aaluK@$(u, K)$} where $K$ is a closed $1$-rectifiable subset of some planar open set $\Omega$ and $u$ is an element of $W^{1,2}_{{\rm loc}} (\Omega\setminus K)$. For every bounded open set $U\subset \Omega$, every $g\in L^\infty (\Omega)$\index[simb]{aalg@$g$}, and every $\param \geq 0$\index[simb]{aagl@$\param$} we define \index{Mumford Shah energy@ Mumford Shah energy}\index[simb]{aalE_p@$E_\param$}\index[simb]{aalE_0@$E_0$}
\begin{align}
E_\param (u, K, U, g) &:= \int_{U\setminus K} |\nabla u|^2 +
\mathcal{H}^1 (K\cap U) + \lambda \int_U |u-g|^2 \, .\label{e:MS+fidelity}
\end{align}
where $\mathcal{H}^1$ denotes the $1$-dimensional Hausdorff measure.\index[simb]{aalHcal@$\mathcal{H}^1$}\index{Hausdorff measure@Hausdorff measure} The open set $U$ will be omitted if it coincides with $\Omega$ or it is obvious from the context and likewise $g$ will be omitted when $\param =0$ or when it is obvious from the context.

The term 
\[
\mathscr{F} (u, U, g) = \int_U (u-g)^2
\]
will be often called ``fidelity term'', \index[simb]{aalfscr@$\mathscr{F}$}\index{fidelity term@fidelity term}
while the remaining portion of the functional will be consistently denoted by $E_0$. 
In particular, we have the obvious identity $E_\lambda = E_0 + \lambda \mathscr{F}$. 
Under our assumption (i.e. $g\in L^\infty$) $E_0$ dominates the fidelity term at small scales. 
In particular, the case $\param>0$ adds mainly technical complications and the reader who wishes to avoid these technicalities could just focus on the case $\param =0$. 
Quite a few constants in the statements need to be adjusted if $\|g\|_{\infty}$ becomes high. 
In order to avoid discussing the size of these constants we will assume once and for all that 
\begin{equation}\label{e:g-and-lambda}
\|g\|_{\infty} \leq M_0 \footnote{Throughout the notes we will omit the reference domain in the $L^\infty$ norm if it will be clear from the context.}
\qquad \mbox{and} \qquad 0\leq \lambda \leq 1\, ,    
\end{equation}
where $M_0$\index[simb]{aalM@$M_0$} is any given positive number. In particular, in the rest of the notes we will omit the dependence of the constants from $M_0$.

The energies can be naturally extended to the class of functions of special bounded variation (shortly $SBV (\Omega)$) and in that case we can regard $K$ as the jump set of the $SBV$ function $u$ (cf. the monograph \cite{AFP00} for the definitions). This is not really relevant for our discussions, as long as the reader is willing to accept some compactness statements which pertain to the {\em existence theory} (of minimizers and critical points) of the functionals. Such statements will play anyway an important role and for the sake of completeness we will include short proofs in the appendix: since we believe that it makes them conceptually easier and cleaner, those arguments take advantage of the $SBV$ theory. \index{SBV@SBV}
\index{special functions of bounded variation@special functions of bounded variation}\index{functions of bounded variation, special@functions of bounded variation, special}

\section{Minimizers} 

We will consider three different notions of minimizers associated with the above energy functionals. We start by defining the first two.

\begin{definition}\label{d:definitions}
A pair $(u,K)$ in an open set $\Omega\subset \mathbb R^2$ will be called:
\begin{itemize}
\item an {\em (absolute) minimizer} \index{absolute minimizer@absolute minimizer} of the functional $E_\param$ if for every bounded open $U\subset \Omega$ and for every pair $(v, J)$ such that $\overline{J\Delta K} \cup \supp (u-v) \subset \subset U$ we have
\begin{align}\label{e:minimizer}
E_\param (u,K,U, g) \leq E_\param  (v,J,U, g)\, ;
\end{align}
\item a {\em restricted minimizer} \index{restricted minimizer@restricted minimizer} if \eqref{e:minimizer} holds for those $(v,J)$ with the additional property that the number of connected components of $J$ does not exceed that of $K$.  
\end{itemize}
\end{definition}

\begin{remark}
Strictly speaking our definition of restricted minimizers in $\Omega$ can be made sense even if $K$ has an infinite number of connected components: in the latter case we understand that any competitor is allowed, and thus $(u,K)$ is an absolute minimizer. Clearly the interest in restricted minimizers lies in the cases where the number of connected components is finite. In many statements this assumption is however not needed and we will omit it: it will instead be spelled out explicitly when it is needed.  
\end{remark}
Absolute minimizers are generally just called minimizers in the literature and we will often use the same convention: we will add the adjective absolute only when we want to emphasize that we can compare its energy with any competitor.
Obviously a minimizer is also a restricted minimizer, while the converse is instead false in general. Note that the value of the functional is invariant under the addition of $\mathcal{H}^1$-null sets, which creates an annoying technicality when dealing with regulary issues. Assume for instance that 
$(u,\emptyset)$ is an absolute minimizer of $E_0$. Then $u$ is a classical harmonic function and so our pair is as smooth as desirable. However, if $K$ is any $\mathcal{H}^1$ null set, then $(u,K)$ is as well an absolute minimizer, so that we have several other minimizers with sets $K$ which can be quite irregular. In order to avoid this issue we wish to normalize minimizers in an appropriate way, throwing away ``unnecessary pieces'' of $K$.

\begin{definition}\label{d:normalized}
Following \cite{DavidBook} a pair $(u,K)$ will be called reduced if, for every open set $U$, $\mathcal{H}^1 (K\cap U)=0$ implies $K\cap U = \emptyset$. \index{reduced pair@reduced pair}
\end{definition}

It is easy to see (cf. Corollary~\ref{c:normalized} below) that if $(u,K)$ is a minimizer or a restricted minimizer in $\Omega$, then it can ``be normalized'' to a reduced pair in the following sense: if we set $\mathcal{N} := \{U \mbox{ open }: \mathcal{H}^1 (U\cap K)=0\}$ and let 
\[
K':=K \setminus \bigcup_{U\in \mathcal{N}} U\, , 
\] 
then $u$ can be extended to a $C^1$ function $u'$ on $\Omega\setminus K'$ (in fact an harmonic function when $\param =0$). Clearly the energy of $(u',K')$ coincides with that of $(u,K)$. For this reason, from now on we will always assume that a minimizing pair (or restricted minimizing pair) is a reduced, unless we explicitly say that it might not be.

\section{Epsilon regularity}\label{ss:eps reg}

The celebrated Mumford-Shah conjecture, which is still unsolved, states the following.

\begin{conjecture}[Mumford-Shah conjecture]\label{c:MS}
If $(u,K)$ is a minimizer of $E_\lambda$ then $K$ can be described as the union of finitely many closed $C^1$ arcs $\gamma_i$ which do not cross but can meet at their endpoints at 120 degrees in ``triple junctions''. In particular, if we fix a point $x\in K$, then in any sufficiently small disk $B_r (x)$ the set $K\cap B_r (x)$ is diffeomorphic to one of the following special types of closed sets:
\begin{itemize}
\item[(a)] a radius of $B_r (x)$;
\item[(b)] a diameter of $B_r (x)$;
\item[(c)] the union of three radii of $B_r (x)$ forming angles of $120$ degrees.
\end{itemize} 
\end{conjecture}

The main conclusion of the regularity theory developed thus far is that if $K$ is close in the Hausdorff distance\index{Hausdorff distance@Hausdorff distance}\footnote{Due to suitable compactness properties of minimizers, it actually suffices that $K$ is close in a weaker sense.} to one of the three examples (a), (b) and (c) at some given scale, then it is indeed diffeomorphic to the corresponding model at a slightly smaller scale: a precise statement is given in Theorem~\ref{t:main} below. These results are typically called ``$\varepsilon$-regularity theorems'', the $\varepsilon$ referring to the fact that the key assumption is the closeness (at a certain scale and in an appropriate sense) to a given model. {In geometric analysis $\varepsilon$-regularity theorems appeared first in the pioneering work of De Giorgi on area-minimizing hypersurfaces, cf. \cite{DG61}. In fact many of our arguments in this book are heavily inspired by analogous ones in the theory of minimal surfaces and we will highlight many parallels in the sequel.}

Before coming to the precise statements we introduce the following terminology.

\begin{definition}\label{d:regular-points}
Let $x\in K$. The point will be called, respectively, a {\em regular loose end}\index{regular loose end@regular loose end}\index{loose end, regular@loose end, regular}, a {\em pure jump}\index{pure jump@pure jump}, or a {\em triple junction}\index{triple junction@triple junction} if for some $r>0$ there is a $C^1$ diffeomorphism $\Phi$ of $B_r (x)$ with $\Phi (x)=x$, $D\Phi (x) = {\rm Id}$, and such that:
\begin{itemize} 
\item[(i)] $\Phi (B_r (x)\cap K)$ is a radius of $B_r (x)$;
\item[(ii)] $\Phi (B_r (x)\cap K)$ is a diameter of $B_r (x)$;
\item[(iii)] $\Phi (B_r (x) \cap K)$ is the union of three radii of $B_r (x)$ meeting at equal angles.
\end{itemize}
A maximal $C^1$ arc of $K$ is a closed (relatively to the domain $\Omega$) $C^1$ arc \index{maximal $C^1$ arc@maximal $C^1$ arc} contained in $K$ which is not contained in any larger $C^1$ arc. An extremum \index{maximal $C^1$ arc, extremum of@maximal $C^1$ arc, extremum of}\index{extremum of a maximal $C^1$ arc@extremum of a maximal $C^1$ arc} of a maximal $C^1$ arc is an endpoint belonging to the domain $\Omega$.
\end{definition}

If $K\cap B_r (x)$ is a single continuous arc with endpoints $x$ and $y\in \partial B_r (x)$ which is $C^1$ away from $\{x\}$ but not necessarily up to $x$, we then call the point $x$ a {\em loose end}\index{loose end@loose end}. In fact a substantial portion of these notes will be dedicated to show that, if $(u,K)$ is a minimizer (no matter whether absolute or restricted), then every loose end is regular. Likewise, an outcome of the regularity theory is also that, if $K\cap B_r (x)$ consists of three continuous arcs which meet at $x$, then these arcs are necessarily $C^1$ up to their common endpoint $x$. 

\medskip

The focus of this book is then proving the following theorem which summarizes several results in the literature. { Proper credits about the various statements will be given in Section~\ref{s:plan of the notes}, while in the remainder of this introduction we will discuss the content, in particular its motivations and consequences}.

\begin{theorem}\label{t:main}\label{T:MAIN}
There are $\alpha>0$ and $\varepsilon >0$ with the following property. Assume:
\begin{enumerate}
\item $(u,K)$ is an absolute minimizer of $E_\param$ in $B_r (x)\subset \mathbb R^2$;
\item \eqref{e:g-and-lambda} holds and $r\leq 1$;
\item the Hausdorff distance of $K\cap B_{2r} (x)$ to a set $K_0$ as in Conjecture~\ref{c:MS}(a), (b) or (c) is smaller than $\varepsilon r$. 
\end{enumerate}
Then:
\begin{itemize}
\item[(i)] there is a $C^{1,\alpha}$ diffeomorphism $\Phi$ of $B_r (x)$ such that $\Phi (K\cap B_r (x)) = K_0$;
\item[(ii)] if $\param =0$ or $g\in C^0$, then $K$ is $C^2$ in some neighborhood of any pure jump;
\item[(iii)] if $\param =0$ each maximal $C^1$ arc is $C^2$ up to its extrema, where its curvature vanishes. 
\end{itemize}
Moreover:
\begin{itemize}
\item[(iv)] the same conclusions hold in cases (b) and (c) if $(u,K)$ is a restricted minimizer (in particular $\varepsilon$ is independent of the number of connected components of $K$);
\item[(v)] in case (a) the same conclusions hold if $(u,K)$ is a restricted minimizer and $K\cap B_{2r} (x)$ is connected;
\item[(vi)] in case (c) the three maximal $C^1$ arcs forming $K\cap B_r (x)$ meet at equal angles.
\end{itemize}
\end{theorem}

Note that in case (c) the meeting point $y$ of the three maximal $C^1$ arcs forming $K\cap B_r (x)$ is not necessarily $x$. Likewise, in case (a) the interior regular loose end is not necessarily located at $x$.

As it is known from the work of Bonnet \cite{B96}, if $(u,K)$ is a restricted minimizer of $E_\lambda$ and $K$ has a finite number of connected components, then for every $x\in K$ the assumptions of Theorem~\ref{t:main} are satisfied in a sufficiently small disk $B_r (x)$. We thus get a quite satisfactory description of the regularity of the set $K$, namely the conclusions of Conjecture~\ref{c:MS} hold. 

\begin{corollary}\label{c:Bonnet-plus}
Conjecture~\ref{c:MS} hold for restricted minimizers $(u,K)$ such that $K$ has a finite number of connected components. 
\end{corollary}

A proof of the latter corollary will be given in Section \ref{s:Bonnet-plus}.
If $g$ is more regular, the regularity of $K$ can be bootstrapped to higher regularity at regular jump points: loosely speaking the regularity of $K$ at jump points is $2$ derivatives more than the regularity of $g$, see \cite{AFP99}. In fact, even the real analiticity of $K$ holds if $g$ is real analytic, cf. \cite{KLM}. However, we will not pursue this issue in these notes.

We finally observe that the second part of conclusion (ii) in Theorem~\ref{t:main} has a curious outcome: for restricted minimizers of $E_0$ with finitely many connected components or anyway for any minimizer for which the Mumford-Shah conjecture holds, the points of maximal curvature of any maximal arc forming $K$ must be pure jump points, unless the arc is a straight segment. 

\section{Global generalized minimizers and classification}
Absolute minimizers enjoy a natural elementary energy upper bound\index{energy upper bound@energy upper bound} on any disk $B_r (x)\subset \Omega$, more precisely 
\begin{equation}\label{e:upper bound}
E_\lambda (u, K, B_r (x), g) \leq 2\pi r  + \lambda \pi \|g\|_\infty^2 r^2\, ,
\end{equation}
cf. Lemma~\ref{l:upper-bound} below. Restricted minimizers satisfy the same bound under the additional assumption that $\overline{B}_r (x) \cap K\neq \emptyset$. If we therefore consider the following rescalings\index{rescaling@rescaling}\index[simb]{aalK_r@$K_{x,r}$}\index[simb]{aalu_u_{x,r}@$u_{x,r}$}
\begin{align*}
u_{x,r} (y) &:= r^{-\sfrac{1}{2}} u (x+ry)\\
K_{x,r} &:= \{{\textstyle{\frac{y-x}{r}}}: y\in K\}\, ,
\end{align*}
then $E_0 (u_{x,r}, K_{x,r}, B_1)$ enjoy a uniform bound. Since under the scaling detailed above the ``fidelity term'' becomes negligible as $r\downarrow 0$, the latter bound allows for a suitable compactness statement for $r\downarrow 0$, provided we introduce an appropriate concept of ``generalized minimizer of $E_0$''\index{generalized minimizer@generalized minimizer}, cf. Definition~\ref{d:generalized} (this is done elegantly in \cite{B96} for the case of restricted minimizers). 
Note that in general, we do not have uniform bounds on the size of $u$, but only on the size of $\nabla u$, hence the space of generalized minimizers of $E_0$ must allow for a suitable normalization of infinities. An important point is that the regularity theory detailed so far remains valid for them.

\begin{theorem}\label{t:main-generalized}
The conclusions of Theorem~\ref{t:main} remain valid for generalized and generalized restricted minimizers of $E$.
\end{theorem}

Coming back to the scaling procedure outlined above, as $r\downarrow 0$ the domain of definition of $(u_{x,r}, K_{x,r})$ becomes larger and in the limit we attain ``global generalized minimizers'', namely generalized minimizers on the entire $\mathbb R^2$\index{global generalized minimizer@global generalized minimizer}\index{generalized minimizer, global@generalized minimizer, global}\index{global minimizer, generalized@global minimizer, generalized}. We warn the reader that in the literature these are typically called just ``global minimizers''. Our rationale for deviating slightly from the usual terminology and adding the adjective ``generalized'' is explained in the next chapter, see in particular Remark~\ref{r:generalized-vs-topological}.

{Looking at limits of these rescalings for $r\downarrow 0$, which will often refer to as ``blow-ups'' \index{blow-up@blow-up}, is reminiscent of the concept of tangent cones in the theory of minimal surfaces, see e.g. \cite{Simon83}. There is however a notable ingredient missing in our context, namely a powerful monotonicity formula as the one on the mass ratio for minimal surfaces. We will come back to this later in this introduction and at several point in the rest of the notes.}

If for any such global minimizer $(u,K)$ of $E_0$ the set $K$ is empty, then $u$ must be an harmonic function and the upper bound \eqref{e:upper bound} implies, by Liouville's theorem, that $u$ is constant. This is the easiest among the ``elementary'' global minimizer\index{elementary global minimizer@elementary global minimizer}\index{global minimizer, elementary@global minimizer, elementary}, which we define as generalized global minimizers whose Dirichlet energy vanishes identically on all compact sets, cf. Theorem~\ref{t:class-global}. Note that, since for these global minimizers the Dirichlet energy is zero, not surprisingly they end up matching ``global minimal $1$-dimensional sets'' in $\mathbb R^2$, which is a more common terminology for them. We however prefer the term elementary global minimizer because the values of the functions still play an important role (for instance we will see that there must be an ``infinite jump'' across the interfaces subdividing different connected components).

Two elementary types of global generalized minimizers correspond to the cases (b) and (c) explained above: in one case the set $K$ consists of an infinite line and in the other case it consists of three halflines meeting at a common origin at equal ($120$ degrees) angles. In both cases the function is piecewise constant (in fact, in an appropriate sense, the difference between all constant values must be $\infty$). This list exhausts the case of elementary global minimizers.

If the Mumford-Shah Conjecture~\ref{c:MS} holds, there is then only another possible type of generalized global minimizer, namely $K$ can only be a half line.
It can be proved that the corresponding $u$ must then take a rather special form, cf. Chapter 4.\index{cracktip@cracktip}

\begin{definition}\label{d:crack-tip}
We will call {\em cracktip} \index{cracktip@cracktip} any pair $(u,K)$ in $\mathbb R^2$ which, up to translations, rotations, subtraction of a constant and change of sign can be described as follows:
\begin{itemize}
    \item $K$ is the set $\mathbb R^+ := \{(t,0): t\geq 0\}$;
    \item $u (x_1, x_2) = \sqrt{\frac{2}{\pi}} {\rm Re}\, \sqrt{x_1+ix_2}$ , where we use the convention that $\sqrt{z}$ is the branch of the square root yielding $i$ on $-1$, with branch cut on $\mathbb R^+$. In polar coordinates the function is given by 
    \[
    u (\theta, \rho) = \sqrt{\frac{2\rho}{\pi}} \cos \frac{\theta}{2}\, ,
    \]
    while in cartesian coordinates we have
    \[
    u(x_1,x_2)=\frac1{\sqrt{\pi}}\frac{x_2}{|x_2|}\sqrt{{x_1+\sqrt{x_1^2+x_2^2}}}
    \qquad\textrm{ if $x_2\neq 0$,}
    \]
    and $u(x_1,0)=0$ if $x_1\leq0$.
\end{itemize}
\end{definition}

With a slight abuse of notation, sometimes we will also use cracktip for points $p\in K$ (in an admissible pair $(u, K)$ which is not necessarily globally defined) where $K$ is locally diffeomorphic to a global cracktip as in the above definition.

\begin{remark}\label{r:convention}
In the literature it is more common to put the halfline of discontinuity at the negative real line and hence to use the formula 
\[
u (\theta, \rho) = \sqrt{\frac{2\rho}{\pi}} \sin \frac{\theta}{2}\, .
\]
The two conventions are obviously equivalent up to symmetries along the imaginary axis.
Since a lot of our analysis in Chapter~\ref{ch:crack} will be based on the paper \cite{DFG}, which fixes the discontinuity set on the positive real axis, we have opted for the less common convention of Definition~\ref{d:crack-tip}. 
\end{remark}

\begin{remark}\label{r:elementary-exotic}
Given that the Mumford-Shah conjecture is still open, currently the existence of other global minimizers is not excluded. These are termed {\em exotic} by Bonnet and David (see e.g. \cite{bonnetdavid2001} and \cite{DavidBook}). We will use the term ``not elementary'' for global minimizers such that $\nabla u$ does not vanish, hence for both the cracktips and the (conjecturally nonexisting) exotic minimizers. 
\end{remark}

Theorem~\ref{t:main} implies that if the list of generalized global minimizers is exhausted by the elementary ones and by cracktips, then the Mumford-Shah conjecture holds. Except for the determination of what happens at loose ends, the latter conclusion was in fact proved as a combination of the two remarkable works \cite{david1996} and \cite{bonnetdavid2001} (see also \cite{DavidBook}). The only point left by those references was the possibility that the set $K$ forms a rather slow spiral at loose ends: Theorem~\ref{t:main} excludes the latter possibility. This was first proved by Andersson and Mikayelyan for minimizers of $E_0$ under the assumption of connectedness of $K$ in \cite{AndMyk}. A different argument for the same theorem was then given by the two authors and Ghinassi in \cite{DFG}, while the present reference is the first to handle the case of $E_\lambda$ for $\lambda>0$. 

The Mumford-Shah conjecture is indeed fully equivalent to the classification of generalized global minimizers, a fact which we record in the following theorem.

\begin{theorem}\label{t:equivalence}
Conjecture~\ref{c:MS} holds for absolute minimizers if and only if any global generalized minimizer is either one of the elementary minimizers listed in Theorem~\ref{t:class-global} or a cracktip.
\end{theorem}

Note at this point that we would not need a classification of all global generalized minimizers, but only of those which are arising as blow-ups. Moreover the elementary minimizers and cracktips are all ``homogeneous'', more precisely they are invariant under the rescalings $(K_{0,r}, u_{0,r})$, if we assume that the origin is appropriately chosen (in the case of pure jumps it must be contained in $K$, while in the case of triple junctions and cracktips it must be, respectively, the common endpoint of the three halflines forming $K$ and the loose end of the halfline forming $K$). On the other hand if we had a suitable monotonicity formula guaranteeing such homogeneity for the blow-ups, then it is not difficult to see that the classification would be rather simple. There are interesting monotonicity formulae available under certain assumptions and in this book we will highlight their usefulness, on the other hand none of them is powerful enough to guarantee the homogeneity of blow-ups in full generality.

A notable feature of the monotonicity formula for minimal surfaces is the fact that the total volume of a minimal surface of dimension $m$ in a ball $B_r (p)$ centered at some point $p$ on the surface is always bounded below by the volume of the corresponding $m$-dimensional disk of radius $r$ (as long as $B_r (p)$ does not intersect the boundary of the surface). This ``density lower bound'' \index{density lower bound@density lower bound} is in fact valid for the set $K$ as well, as it was proved in the important work \cite{DGCL89}, and it plays a fundamental role in the regularity theory.

Observe that Theorem~\ref{t:main} does not guarantee that cracktips are indeed global minimizers: Theorem~\ref{t:main} would imply the latter conclusion only if one could find an absolute minimum $(u,K)$ such that $K$ is close, in some disk $B_r (x)$, to a single straight segment which terminates at $x$. 

The global minimality of cracktip was however proved in the book \cite{bonnetdavid2001}. While these notes will not cover the proof of the latter fact, we will however cover the proof of one fundamental conclusion of \cite{bonnetdavid2001}, which will be instrumental in showing case (c) of Theorem~\ref{t:main}, and which is interesting in its own right. The relevant statement is recorded in the first half of the following theorem, while the second half is another remarkable result for global generalized minimizers proved in \cite{DL02}, of which we will report as well the proof.

\begin{theorem}\label{t:uniqueness}
Let $(u,K)$ be a generalized global minimizer in the sense of Definition~\ref{d:generalized}. If all but at most one connected component of $K$ are contained in a compact set of $\mathbb R^2$, then $(u,K)$ is either one of the elementary global minimizers described in Theorem~\ref{t:class-global} or it is a cracktip. 

Let $(u,K)$ be a generalized global minimizer for which $K$ disconnects $\mathbb R^2$. Then $(u,K)$ is an elementary global minimizer.
\end{theorem}

\section{Structural results} 

Even though the Mumford-Shah conjecture is widely open for absolute minimizers with an infinite number of connected components, Theorem~\ref{t:main} and Theorem~\ref{t:uniqueness} tell us a lot about the structure of the set $K$ for minimizers $(u,K)$. We first introduce the following terminology.

\begin{definition}\label{d:internal points}
Consider a closed set $K\subset \mathbb R^2$. A point $p\in K$ is {\em nonterminal} \index{nonterminal point@nonterminal point} if there is a continuous injective map $\gamma: (-1,1) \to K$ such that $\gamma (0) = p$. Otherwise $p\in K$ is called {\em terminal}\index{terminal point@terminal point}.

Given a pair $(u,K)$ which is an absolute, restricted, generalized, or generalized restricted minimizer, we define:
\begin{itemize}
    \item the regular part $K^{(r)}$ of $K$\index{regular part of $K$@regular part of $K$}\index[simb]{aalK^r@$K^{(r)}$}, which consists of those points which are either regular loose ends, or pure jump points, or triple junctions;
    \item the irregular part $K^{(i)}$ of $K$\index{irregular part of $K$@irregular part of $K$}\index[simb]{aalK^i@$K^{(i)}$}, which is the complement of $K^{(r)}$ in $K$;
    \item the points $K^\sharp$ of high energy\index{points of high energy@points of high energy}\index{high energy, points of@high energy, points of}\index[simb]{aalK^sharp@$K^\sharp$}, which is the union of $K^{(i)}$ and regular loose ends. 
\end{itemize}
\end{definition}

The following theorem summarizes all the structural results available in the literature (at least to our knowledge).

\begin{theorem}\label{t:structure}
Assume $(u,K)$ is an absolute or a generalized minimizer of $E_\lambda$. Then
\begin{itemize}
    \item[(i)] $K^{(i)}$ is a relatively closed subset, and its Hausdorff dimension\index{Hausdorff dimension@Hausdorff dimension} is at most $1-\varepsilon$, where $\varepsilon$ is a positive geometric constant;
    \item[(ii)] Every nonterminal point of $K$ is either a triple junction or a pure jump point;
    \item[(iii)] If $U$ is an open set such that $U\cap K$ consists of finitely many connected components, then $K^{(i)}\cap U =\emptyset$;
    \item[(iv)] A point $p\in K$ belongs to $K^\sharp$ if and only if 
    \begin{equation}\label{e:high-energy}
    \limsup_{r\downarrow 0} \frac{1}{r} \int_{B_r (x)} |\nabla u|^2 >0\, ,
    \end{equation}
    which occurs if and only if
    \begin{equation}\label{e:high-energy-2}
    \liminf_{r\downarrow 0} \frac{1}{r} \int_{B_r (x)} |\nabla u|^2 >0\, ;
    \end{equation}
    \item[(v)] Triple junctions and regular loose ends form a discrete set;
    \item[(vi)] A terminal point which is not the accumulation point of infinitely many connected components of $K$ is necessarily a regular loose end.
\end{itemize}
\end{theorem}

Note that $K^{(i)}$ could be further subdivided in the union of those points $\{p\}$ of $K$ which are connected components of $K$ (``singletons'') and of the irregular terminal points of the connected components of $K$ with positive length. The current state of the art in the regularity theory does not allow to conclude that singletons are absent, or that terminal points of connected components of $K$ with positive length are {\em always} regular loose ends. In both cases the scenario that cannot be excluded is that of points $p$ which are the accumulation of an infinite sequence of connected components of $K$ with positive length. It seems simpler to exclude that this might happen for terminal points of connected components with positive length: the latter would be implied by the following strengthening of the first part of Theorem~\ref{t:uniqueness}, which is interesting in its own right. 

\begin{conjecture}\label{c:BD-forte}
Let $(u,K)$ be a global generalized minimizer with the property that $K$ has one unbounded connected component. Then $(u,K)$ is either an elementary global minimizer or a cracktip.
\end{conjecture}

We finally record here that the structure theorem yields another equivalent formulation of the Mumford-Shah conjecture in terms of the optimal summability of $\nabla u$ (cf. \cite{DLF13}).

\begin{theorem}\label{t:summability}
Let $(u,K)$ be an absolute or generalized minimizer of $E_\param$ in $\Omega$. The set $K^{(i)}$ is empty if and only if $\nabla u \in L^{4,\infty}_{{\rm loc}}$,
\index[simb]{aalL^4,inftyloc@$L^{4,\infty}_{{\rm loc}}$} namely if and only if for every compact set $U\subset \Omega$ there is a constant $C = C(U)$ such that
\begin{equation}\label{e:L quattro debole}
|\{x\in U : |\nabla u (x)|^4\leq M\}|\leq \frac{C}{M}
\qquad \forall M \geq 1\, .
\end{equation}
\end{theorem}

\section{Critical points}\label{ss:critical-points}
Restricted, absolute, generalized, and generalized restricted minimizers are naturally critical points in the following sense (cf. Proposition~\ref{p:variational identities}). We start by defining two classes of appropriate variations of the relevant functionals.

\begin{definition}\label{d:variations}
Let $(u,K)$ be an admissible pair in an open set $U$.
\begin{itemize}
\item[(Out)] If $\varphi\in W^{1,2} (U\setminus K)$ with $\supp (\varphi) \Subset U$, then the $1$-parameter family of competitors $(u_t, K_t)=(u+t\varphi, K)$ will be called an {\em outer variation} \index{outer variation@outer variation}in $U$.
\item[(In)] Consider any one-parameter family of diffeomorphisms $(-\varepsilon, \varepsilon) \times \mathbb{R}^2 \ni (t, x)\mapsto \Phi_t (x)= \Phi (t,x)$ of $U$, which is $C^1$ in both variables and such that $\Phi_t (x)=x$ for all $x$ outside a compact subset of $U$. Then $(u_t, K_t)=(u\circ \Phi_t^{-1},\Phi_t (K))$ will be called an {\em inner variation}\index{inner variation@inner variation}. 
\end{itemize}
\end{definition}

As it can be naturally expected, for minimizers we will show that the condition
\begin{equation}\label{e:critical}
\left.\frac{d}{dt}\right|_{t=0} E_\lambda (u_t,K_t,U,g) = 0 \, 
\end{equation}
holds for every inner and outer variation in $U$.
Equation \eqref{e:critical} gives corresponding Euler-Lagrange conditions, which in turn play a crucial role in the regularity theory. In fact several conclusions can be inferred from such conditions alone and in these notes we will make an effort to keep track of them. 

The first Euler-Lagrange condition, derived from outer variations is given by\index{Euler-Lagrange conditions@Euler-Lagrange conditions}
\begin{align}
&- \int_{\Omega\setminus K} \nabla u \cdot \nabla \varphi = \lambda \int_\Omega \varphi (u-g)\qquad\mbox{$\forall \varphi\in W^{1,2} (\Omega\setminus K)$ with $\supp (\varphi) \subset\subset \Omega$.} \label{e:outer}
\end{align}
The second, derived from inner variations, takes a particularly simple form when $g\in C^1$:
\begin{align}
&\int_{\Omega\setminus K} \big(|\nabla u|^2 {\rm div}\, \psi -2 \nabla u^T
\cdot D \psi\, \nabla u\big)+
\int_K e^T \cdot D\psi\, e \, d\mathcal{H}^1\nonumber\\
&= \lambda\int_{\Omega\setminus K} (2 (u-g) \nabla g \cdot \psi - |u-g|^2 {\rm div}\, \psi)\qquad\qquad\qquad\qquad \forall \psi \in C^1_c (\Omega, \mathbb R^2)\, ,\label{e:inner-C1}
\end{align}
where $e(x)$\index[simb]{aale@$e$} is a unit tangent vector field to $K$ and the notations $a\cdot b$ and $A b$ refer to, respectively, the scalar product\index{scalar product@scalar product} between two vectors $a$ and $b$ and the usual product of a matrix $A$ and a (column) vector $b$, $\nabla u$ denotes the gradient of the function $u$, and $D\psi$ the Jacobian matrix \index{Jacobian matrix@Jacobian matrix} of the vector field $\psi$ (this convention will be followed through the rest of these notes). The test $\varphi$ in the outer variation relates to $u_t$ through the relation $u_t = u + t\varphi$ when it is $C^1$, while the test $\psi$ relates to the one parameter family of diffeomorphisms in Definition~\ref{d:variations}(In) via $\psi (x) = \partial_t \Phi (0,x)$.

Note that \eqref{e:inner-C1} does not make sense when $g\in L^\infty$. In fact it is not at all obvious that $t\mapsto E_\lambda (u_t,K_t,\Phi_t(U),g\circ\Phi_t^{-1})$ is differentiable at $t=0$ in that case. This can however be shown for minimizers and corresponding Euler-Lagrange conditions can be derived. In order to state it, we first introduce the normal $\nu$\index[simb]{aagn@$\nu$}, which is the counterclockwise rotation of $e$ by an angle of $90$ degrees. For minimizers we can then show the existence of suitable traces $u^+$, $u^-$ and $g_K$ on $K$\index[simb]{aalu^+@$u^+$}\index[simb]{aalu^-@$u^-$}\index[simb]{aalg_K@$g_K$}, which are locally bounded Borel functions and such that
\begin{align}\label{e:inner}
& \int_{\Omega\setminus K} \big(|\nabla u|^2 {\rm div}\, \psi 
-2 \nabla u^T \cdot D \psi\, \nabla u\big)
+ \int_K e^T \cdot D\psi\,  e \, d\mathcal{H}^1 \nonumber\\ 
& =2 \lambda \int_{\Omega\setminus K}(u-g)\nabla u\cdot\psi
+\lambda \int_{K}(|u^+-g_K|^2-|u^--g_K|^2) \psi\cdot\nu d\mathcal{H}^1
\quad \forall \psi \in C^1_c (\Omega, \mathbb R^2)\, .
\end{align}
The latter identity is in fact equivalent to \eqref{e:inner-C1} when $g\in C^1$.

Motivated by the above discussion we introduce the following notions.

\begin{definition}\label{d:critical-points}
Consider $g\in C^1$. A pair $(u,K)$ in $\Omega$  is a critical point\index{critical point@critical point} of $E_\lambda$ if \eqref{e:outer}-\eqref{e:inner-C1} hold for every $\varphi\in W^{1,2} (\Omega\setminus K)$ with $\supp (\varphi) \subset\subset \Omega$ and every $\psi\in C^1_c (\Omega, \mathbb R^2)$. 

Consider $g\in L^\infty$. A pair $(u,K)$ in $\Omega$ is a critical point of $E_\lambda$ if \eqref{e:outer} holds for every $\varphi\in W^{1,2} (\Omega\setminus K)$ with $\supp (\varphi) \subset\subset \Omega$
and if there are bounded Borel functions $u^+, u^-$, and $g_K$ such that \eqref{e:inner} holds for every $\psi\in C^1_c (\Omega, \mathbb R^2)$.
\end{definition}

The variational identities \eqref{e:outer}-\eqref{e:inner} are proved for minimizers in Proposition~\ref{p:variational identities}. Both \eqref{e:outer} and \eqref{e:inner-C1} are well-known and their proof are reported, for instance, in \cite{AFP00}. However we have not been able to find a reference in the literature for \eqref{e:inner} and we thus give its derivation in the appendix.

\begin{remark}
In these notes we will take advantage of the identity \eqref{e:inner} to prove several monotonicity formulae which we will freely use in the proofs of the regularity results of Theorem~\ref{t:main} in cases (b) and (c). On the other hand, those $\varepsilon$-regularity statements hold more generally for what in the literature are sometimes called quasi-minimizers (see \cite{AFP00}) or almost-minimizers\footnote{We warn the reader that \cite{DavidBook} uses both the terms almost-minimizers and quasi-minimizers, but the quasi-minimziers of \cite{DavidBook} are not the ones of \cite{AFP00}.} (see \cite{DavidBook}) \index{almost-minimizer@almost-minimizer} \index{quasi-minimizer@quasi-minimizer} of the functional $E_0$, 
of which minimizers of $E_\lambda$ are a distinguished example. In particular our use of the monotonicity formulae can be avoided. 

More precisely, the conclusions of Theorem~\ref{t:main} in cases (b) and (c) of Conjecture~\ref{c:MS} 
are proved in \cite{david1996,AP97,AFP97} (cf. also \cite{AFP00,DavidBook}) for 
pairs $(u,K)$ satisfying the following weakened version of \eqref{e:minimizer}: 
there are $\omega,\delta>0$ such that
\[
E_0 (u,K,B_r(y))\leq E_0 (v,J,B_r(y))+\omega r^{1+\delta}
\]
(where $x,r$ and $(v,J)$ are arbitrary and satisfy the conditions for competitors given in Definition~\ref{d:definitions}, with $U=B_r (x)$).
Note that minimizers of $E_\lambda$ are indeed almost-minimizers of $E$ with $\delta=1$ and 
$\omega=2\pi$ according to a simple comparison estimate. 

However, the conclusion of Theorem~\ref{t:main} for Conjecture~\ref{c:MS}(a)  
seems to require crucially \eqref{e:inner} or anyway a suitable approximate version. At present we do not know how to prove the same conclusion for almost-minimizers or even whether to expect it to be true for them.

We also believe that the use of the variational identities in the proofs of the cases (b) and (c) has its own interest: besides providing a different point of view which covers the most important examples of almost-minimizers of $E_0$ (i.e. minimizers of $E_\lambda$), it simplifies considerably some of the arguments. 
\end{remark}

\section{Plan of the notes}\label{s:plan of the notes}

We start the notes with Chapter~\ref{ch:preliminaries}, which will collect a number of preliminary results which have their own independent interest but will also be instrumental in the proofs of the main theorems. We will in particular:
\begin{itemize}
    \item state and prove some elementary bounds and maximum principles;
    \item state a fundamental density lower bound on the discontinuity set $K$, due to De~Giorgi, Carriero, and Leaci, whose proof is deferred to the appendix;
    \item introduce the ``blow-up procedure'' and state some relevant compactness properties of minimizers, with most of the proofs deferred to the appendix;
    \item detail the variational identities stated in Section~\ref{ss:critical-points} (whose proofs are given in the appendix) and derive a few important corollaries, in particular most of the monotonicity formulae.
\end{itemize}
Chapter~\ref{ch:salti_e_tripunti} will prove the cases (b) and (c) of Theorem~\ref{t:main} 
(cf. Subsection~\ref{ss:proof of (b)-(c) Theorem main}, see also Theorems~\ref{t:eps_salto_puro} and \ref{t:eps_tripunto} for slightly different statements). Case (b) (e.g. the regularity at ``pure jumps'') was proved independently by David in \cite{david1996} and Ambrosio, Fusco, and Pallara in \cite{AFP97}.
We essentially follow the approach of Ambrosio, Fusco, and Pallara, cf. the book \cite{AFP00}, but it must be noted that, while their arguments are written to cover the case of all dimensions, in our setting we exploit the $1$-dimensionality of the set $K$ to take several shortcuts. The case (c) of triple junctions was proved in \cite{david1996}. Here we propose an approach which is different and new, in particular we take advantage of the monotonicity formulae of the first chapter (due to David and L\'eger in \cite{DL02}) and of a ``blow-in'' argument which considerably simplifies the overall proof. 

Chapter~\ref{c:Bonnet David} is a self-contained proof of Theorem~\ref{t:uniqueness}. The arguments for the first part of the theorem are borrowed from the book \cite{bonnetdavid2001} while the arguments for the second part are due to \cite{DL02}. Here and there we propose alternative derivations and we streamline and simplify a few steps of the original arguments. Moreover, we prove first the second part of the theorem and hence take advantage of it whenever possible to prove the first part (a possibility which was not at disposal of the authors in \cite{bonnetdavid2001}, given that the paper \cite{DL02} appeared afterwards). 

The first part of Theorem~\ref{t:uniqueness} is in fact a stepping stone for case (c) of Theorem~\ref{t:main}, more precisely it implies the intermediate result of Corollary~\ref{c:connectedness}, namely that when for an absolute minimizer $(u,K)$ the set $K$ is sufficiently close to a radius in the disk $B_r (x)$, then it consists of a single connected component in $B_{r/2} (x)$. Chapter~\ref{ch:crack} starts from this conclusion  and leads to a complete proof of case (c) of Theorem~\ref{t:main} (cf. Theorem~\ref{t:final-cracktip}). We follow the arguments of our previous work \cite{DFG} with some appropriate modifications to account for the fidelity term in the energy $E_\lambda$ (the work \cite{DFG} addressed the case of $E_0$). 

Chapter~\ref{ch:finale} examines a few important consequences of the $\varepsilon$-regularity theory. It first studies which structural conclusions can be derived from it. More precisely it proves all the structural results of Theorem~\ref{t:structure}, apart from 
item (ii), which is exactly the content of Corollary~\ref{c:non-terminal=regular}: a quantitative version of item (i) is given in Corollary~\ref{c:semplice-2}, and
items (iii) to (vi) are implied by  
Theorem~\ref{t:structure-2}. 
Theorem~\ref{t:summability} is the content of 
Theorem~\ref{t:summability-2}. Moreover the chapter addresses some other consequences of the $\varepsilon$-regularity theory, like the porosity of the singular portion of $K$, and the higher integrability of $\nabla u$, conjectured by De~Giorgi and first proved by the authors in \cite{DLF13} in the setting of this book (i.e. in $2$-dimensions) and hence settled in all dimensions by De~Philippis and Figalli in \cite{DePFig14}. In fact we present both arguments. 

\subsection{Acknowledgments} Both authors are very thankful to Silvia Ghinassi, for helping with a preliminary version of Chapter~\ref{ch:salti_e_tripunti}, and to Francesco Deangelis, for reading carefully the first draft of the book. The first author acknowledges the support of the National Science Foundation through the grant FRG-1854147. The second author has been supported by the European Union - Next Generation EU, Mission 4 Component 1 CUP 2022J4FYNJ, PRIN2022 project "Variational methods for stationary and evolution problems with singularities and interfaces".

\chapter{Density bounds, compactness, variations, and monotonicity}\label{ch:preliminaries}

In this chapter, we will collect some preliminary important considerations, before delving into the $\varepsilon$-regularity theory in the next four chapters. For some aspects which are not our main focus, but which are needed to understand the rest of the notes, we will nonetheless include the arguments for the reader's convenience (and because in some cases it has been difficult to track in the literature appropriate arguments that apply to our precise statements): however, when the proofs are long and technical we will move them to the appendix. 

After collecting some preliminary lower and upper bounds on the energy of minimizers and on the length of $K$, we will use them to introduce a pivotal procedure in the regularity theory, that of ``blow-up''. This procedure will allow us to zoom around a point $x\in K$ and extract meaningful limits (global minimizers) which in turn will give us a first rough picture of the local behavior around the point $x$. The second part of the chapter is devoted to important variational identities that can be derived by plugging suitable tests in \eqref{e:inner}. We will review several consequences of these identities (the determination of the cracktip factor, L\'eger's magic formula, the David-L\'eger-Maddalena-Solimini identity, and the truncated test identities). Finally, in the last section, we will introduce two important monotonicity formulae due, respectively, to Bonnet, and to David and L\'eger. While the first has a straightforward proof, which will be given in this chapter, the latter requires a more involved argument, which will be detailed later in Chapter 4.

Although the concept of ``generalized minimizer'' for $E_0$ will be given in Definition~\ref{d:generalized}, some preliminary statements (which will be formulated earlier than Definition~\ref{d:generalized}) hold for generalized minimizers as well and we point this out when needed: the proof will either become obvious once the concept is defined, or it is postponed to the appendix, where we will treat all cases.

\section{Preliminaries} 

In this section, we collect some preliminary considerations. 

\subsection{A maximum principle} We first give the following elementary maximum principle (see Figure~\ref{f:max-principle}).\index{maximum principle for minimizers@maximum principle for minimizers}

\begin{lemma}\label{l:maximum}
Let $(u,K)$ be an absolute or restricted minimizer of $E_\lambda$ in $\Omega$, or a generalized (resp. generalized restricted) minimizer of $E_0$ (cf. Definition~\ref{d:generalized}) and let $V$ be a connected component of $U\setminus K$ for some open $U\subset\subset \Omega$.
\begin{itemize}
\item[(a)] If $V\cap \partial U =\emptyset$, then $\|u\|_{L^\infty (V)}\leq \|g\|_{\infty}$ when $\lambda >0$ and $\nabla u|_V =0$ when $\lambda=0$;
\item[(b)] If $V$ intersects $\partial U$, then if $\lambda>0$ 
\[
\min \left\{\inf_V g, \inf_{V\cap \partial U} u\right\} \leq \inf_V u \leq \sup_V u \leq \max \left\{\sup_V g, \sup_{V\cap \partial U} u\right\}\,, 
\]
while if $\lambda=0$
\[
\inf_{V\cap \partial U} u \leq \inf_V u \leq \sup_V u \leq \sup_{V\cap \partial U} u\,.
\]
\end{itemize}
\end{lemma}

\begin{figure}
\begin{tikzpicture}
\draw (-6,0) circle [radius =2];
\draw[very thick] (-6, 2) to [out = 240, in = 120] (-6, 1.5) to [out = 300, in = 0] (-6, -1.5) to [out = 180, in = 210] (-6, 1.5);
\node at (-6,0) {$V$};
\draw (0,0) circle [radius = 2];
\draw[very thick] (-0.5,2.4) to [out=290, in=180] (0.2,-1) to [out=0, in=180] (0.7,-0.5) to [out=0, in=240] (2.2,0.5);
\node at (0.3,0.3) {$V$};
\node[above right] at ({sqrt(2)},{sqrt(2)}) {$V\cap \partial U$};
\end{tikzpicture}
\caption{On the left case (a) and on the right case (b) of Lemma~\ref{l:maximum}. In these examples $U$ is a disk.\label{f:max-principle}}
\end{figure}
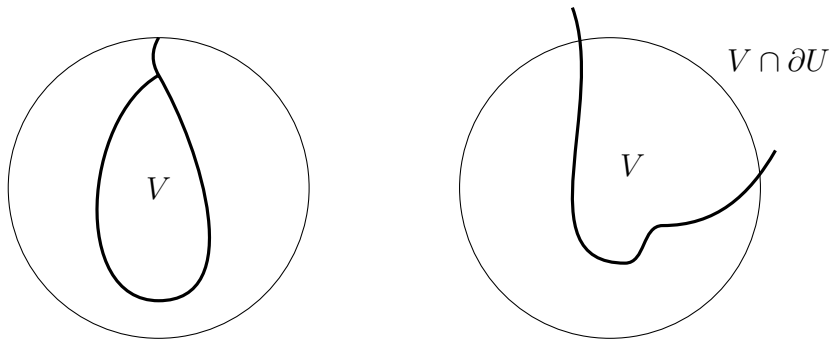

\begin{proof} Assume (a) is false. For $\lambda=0$ we could define $\bar K =K$, $\bar u = u$ on $\Omega\setminus V$ and $\bar u=0$ on $V$ when $(u,K)$ minimizes $E_0$: the pair $(\bar u, \bar K)$ would then have a lower energy. Similarly, when the pair minimizes $E_\lambda$ with $\lambda$ positive, $u$ is the absolute minimizer of $\int_V (|\nabla u|^2+\lambda (u-g)^2)$ in $V$ and hence its $L^\infty$ norm is bounded by $\|g\|_\infty$ by an obvious truncation argument.
As for (b), if $\lambda>0$ set
\begin{align*}
m:= \min \left\{\inf_V g, \inf_{V\cap \partial U} u\right\}\, ,\qquad
M:= \max \left\{\sup_V g, \sup_{V\cap \partial U} u\right\}\, .
\end{align*}
If (b) were false, we could just define $\bar K = K$, $\bar u = u$ on $\Omega\setminus V$ and $\bar u = \min \{M, \max \{m, u\}\}$ on $V$: the pair $(\bar u , \bar K)$ would contradict the minimality. Similarly, if $\lambda=0$ we conclude by comparison with
$\bar u = \min\{\sup_{V\cap \partial U} u , \max\{\inf_{V\cap \partial U} u, u\}\}$.
\end{proof}

\subsection{Lower and upper bounds} Next, we point out that it is rather trivial to come up with upper bounds for the full energy in a disk $B_r (x)$ using a second simple comparison argument\index{energy upper bound@energy upper bound}.

\begin{lemma}\label{l:upper-bound}
Assume $(u,K)$ is an absolute or restricted minimizer of $E_\lambda$ in $\Omega\subset \mathbb R^2$ or a generalized (resp generalized restricted) minimizer of $E_0$ (cf. Definition~\ref{d:generalized}). Let $B_r (x)\subset\subset \Omega$ and in case $(u,K)$ is a restricted minimizer assume in addition $\overline{B}_r (x)\cap K\neq \emptyset$. Then the estimate \eqref{e:upper bound} holds. 
\end{lemma}
\begin{proof}
Compare the energy of $(u,K)$ to the competitor $(v,J)$ defined by setting
$J= \partial B_r (x) \cup (K\setminus B_r (x))$, $v=0$ on $B_r (x)$ and $v= u$ on $\Omega\setminus (\overline{B_r (x)}\cup K)$. See Figure~\ref{f:upper-bound}.
\end{proof}

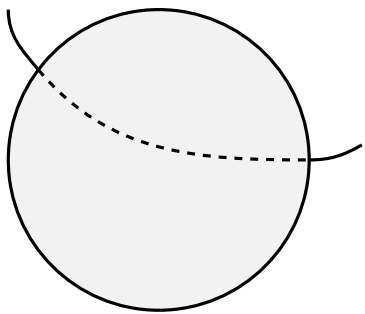
\begin{figure}
\begin{tikzpicture}
\fill[gray!10] (0,0) circle [radius =2];
\draw[very thick] (0,0) circle [radius = 2];
\draw[very thick] (-2,2) to [out = 270, in=130] (-8/5, 6/5);
\draw[very thick, dashed] (-8/5, 6/5) to [out = 310, in = 180] (2, 0);
\draw[very thick] (2,0) to [out = 0, in = 210] (2.7, 0.2);
\end{tikzpicture}
\caption{The competitor $(v,J)$ in Lemma~\ref{l:upper-bound}: we remove the dashed part of the set $K$, we add the circle $\partial B_r (x)$, we set $v=0$ in the shaded disk and we keep $v=u$ outside of it.\label{f:upper-bound}}
\end{figure}

A much more interesting and nontrivial fact is that it is possible to bound  uniformly from below on any disk which is centered at a point of $K$ the ratio of the length of $K$ and the radius of the disk itself. This fundamental discovery, which is at the foundation of all the regularity theory, is due to De~Giorgi, Carriero and Leaci \cite{DGCL89} (and an appropriate generalization appropriately holds in any dimension). Since then different proofs have appeared in the literature (see \cite{CL90,DME92,david1996,Siaudeau,DF,BL}). 
However, the statement below, which is valid across different formulations and for restricted minimizers independently of the number of the connected components of $K$, seems to be new and we include its proof in the appendix.\index{density lower bound@density lower bound} \index{lower bound, density@lower bound, density}

\begin{theorem}\label{t:dlb}
There exists a geometric constant $\epsilon>0$ with the following property.
\begin{itemize}
 \item[(a)] If $(u,K)$ is a (absolute, restricted, generalized, or generalized restricted) minimizer of $E_0$ on $\Omega$, 
 then
\begin{equation}\label{e:dlbE}
\mathcal{H}^1(K\cap B_\rho(x))\geq\epsilon\rho\qquad
\forall x \in K,\, \forall \rho\in(0,\mathrm{dist}(x,\partial\Omega))\, .
\end{equation}
\item[(b)] If $(u,K)$ is a (absolute or restricted) minimizer of $E_\lambda$ on $\Omega$, then 
\begin{equation}\label{e:dlbEg}
\mathcal{H}^1(K\cap B_\rho(x))\geq\epsilon\rho\qquad
\forall x\in K,\, \forall \rho\in(0,\min\{1,\mathrm{dist}(x,\partial\Omega)\})\, .
\end{equation}
\end{itemize}
\end{theorem}

We remark that the constant $\epsilon$ in \eqref{e:dlbEg} depends in fact on $\lambda$ and $\|g\|_\infty$, a dependence that we can ignore thanks to \eqref{e:g-and-lambda}.
Note that the two bounds \eqref{e:upper bound}-\eqref{e:dlbE} (resp. \eqref{e:dlbEg}) 
taken together describe a property of $K$ which is often termed {\em Ahlfors-David regularity} in the literature\index{Ahlfors-David regular@Ahlfors-David regular}, namely the property that for some exponent $\alpha$ (in our case $\alpha=1$) and some constant $C>0$, the closed set $K$ satisfies
\begin{equation}\label{e:Ahlfors-David regularity}
C^{-1} r^\alpha \leq \mathcal{H}^\alpha (K\cap B_r (x)) \leq C r^\alpha
\end{equation}
for all $x\in K$ and for all positive radii $r< \min \{1, \dist (x, \partial \Omega)\}$. 

\subsection{Normalization and equivalence between classical and SBV formulations}
The density lower bound in Theorem~\ref{t:dlb} and Lemma~\ref{l:maximum}, together with an elementary fact about SBV functions observed by De~Giorgi, Carriero and Leaci, gives the following corollary. 
We do not dwell here on the definition of BV and SBV functions and on the definition of the set $S_u$, but we refer the reader to the textbook \cite{AFP00}.

\begin{corollary}\label{c:normalized}
Let $(u,K)$ be a (absolute, restricted, generalized, or generalized restricted) minimizer of $E_\lambda$. 
Then $u\in SBV(\Omega)$ with  $\overline{S}_u\subseteq K$. If $u$ is an absolute minimizer, then we have in addition $K\triangle\overline{S}_u=\emptyset$.

Moreover, if $\mathcal{H}^1 (K\cap U) = 0$ for some open domain $U\subset\subset\Omega$,
then $u$ extends to a function $u\in C^{1,\alpha}_{{\rm loc}} \cap W^{2,p}_{{\rm loc}}(U)$ for every $p<\infty$ and every $\alpha\in (0,1)$ which solves 
$\Delta u = \lambda (u-g)$.
\end{corollary}

{A proof of the corollary is given in Section \ref{s:SBV-equivalence} of the appendix, where the first part is used to conclude easily the second. However, it is possible to give the following elementary argument for the second part, without resorting to SBV functions. First of all, once we show that $u\in W^{1,2}_{loc} (U)$, $u$ is a local minimizer of $\int (|\nabla u|^2 +\lambda (u-g)^2)$ in $U$: $u$ will then be a weak solution of $\Delta u = \lambda (u-g)$ and recalling that $u$ is locally bounded, while $g$ is bounded, the regularity of $u$ is a consequence of the Schauder and Calderon-Zygmund estimates for the Laplace equation. Our goal is thus to prove that $\nabla u$ is the weak derivative of $u$ in $U$, namely that 
\[
\int u\, {\rm div}\, X = -\int X\cdot \nabla u \qquad \mbox{for every $X\in C^\infty_c (U, \mathbb R^2)$.}
\] 
Since $K\cap {\rm spt}\, (X)$ is compact, set $\tau:=\min \{\dist (y, \partial U) : y\in K\cap {\rm spt}\, (X)\} >0$. Fix now any $\varepsilon >0$ and choose a covering $B_{r_i} (x_i)$ of $K\cap {\rm spt}\, (X)$ with $x_i\in K$, $2r_i < \tau$ and $\sum_i r_i < \varepsilon$. By compactness we can assume that the cover is finite. For each $i$ we let $\varphi_i$ be a smooth cut-off function with $0\leq \varphi_i\leq 1$, $\varphi_i \equiv 1$ on $\mathbb R^2\setminus B_{2r_i} (x_i)$ and $\varphi_i\equiv 0$ on $B_{r_i} (x_i)$. We also require that $|\nabla \varphi_i|\leq C r_i^{-1}$, where the constant $C$ is just dimensional. We consider then 
$\psi_\varepsilon := \Pi_i \varphi_i$. Given that ${\rm spt}\,  (X\psi_\varepsilon) \subset \subset U\setminus K$ and that $u\in W^{1,2}_{loc} (U\setminus K)$, we have 
\[
\int \psi_\varepsilon X \cdot \nabla u = - \int u \psi_\varepsilon {\rm div}\, X - \int u\, \nabla \psi_\varepsilon \cdot X\, . 
\]
Observe that $\psi_\varepsilon$ converges monotonically to $1$ almost everywhere in $\mathbb R^2$ and thus we conclude 
\[
\lim_{\varepsilon\downarrow 0} \int \psi_\varepsilon (X \cdot \nabla u + u\, {\rm div}\, X) = \int (X \cdot \nabla u+ u\, {\rm div}\, X) \, .
\]
To prove our claim we therefore just need to show that $\int u \nabla \psi_\varepsilon \cdot X$ converges to $0$.
Recalling, however, that $u$ and $X$ are both bounded on compact subsets of $U$ and using the properties of $\psi_\varepsilon$ we conclude
\[
\left| \int u \nabla \psi_\varepsilon \cdot X \right| \leq C \sum_i \int_{B_{2r_i} (x_i)} |\nabla \varphi_i|
\leq C \sum_i r_i \leq C \varepsilon\, ,
\]
for a constant $C$ which is independent of $\varepsilon$.
}

\begin{remark}\label{r:pointwise}
Corollary~\ref{c:normalized} implies that the function $u$ and its gradient $\nabla u$ are pointwise defined on any open subset $U$ in which $\mathcal{H}^1 (K\cap U) =0$, since $u$ is continuous and continuously differentiable. This easily implies that any minimizer can be ``reduced'' in the sense of Definition~\ref{d:normalized}, as explained in the introduction. In the rest of these notes we will make a similar assumption whenever dealing with general critical points, i.e. we will assume that they are reduced and that they belong to $C^{1,\alpha}_{{\rm loc}} \cap W^{2,p}_{{\rm loc}} (\Omega\setminus K)$.
\end{remark}

\section{Blow-up of minimizers}

In this section we assume that $(u,K)$ is an {\em absolute} or a {\em restricted} minimizer of the energy functional $E_\param$ (and we recall that $g$ is bounded).
Observe that \eqref{e:upper bound} gives locally an apriori estimate on the energy of $u$. Fix a point $x\in K$ and a sequence of radii $r_j\downarrow 0$. The bound suggests to consider the rescaled functions\index{rescaling@rescaling}\index[simb]{aalK_r@$K_{x,r}$}\index[simb]{aalu_u_{x,r}@$u_{x,r}$}
\begin{align}
K_j &:= \left\{\frac{y-x}{r_j} : y\in K \right\}\label{e:riscalamenti K}\\
u_j (y) &:= r_j^{-\frac{1}{2}} u (x + r_j y)\,.
\label{e:riscalamenti u}
\end{align}
When $\param =0$ the pair $(u_j, K_j)$ is then still a minimizer of the functional $E_0$ if $(u,K)$ is. In case of $\param >0$, the density lower bound (cf. Theorem~\ref{t:dlb}) ensures that 
\[
\epsilon\, r\leq\int_{B_r (x)} |\nabla u|^2 + \mathcal{H}^1 (K\cap B_r (x))
\]
for every sufficiently small radius $r$, while the maximum principle gives
\[
\mathscr{F} (u, B_r (x), g) = \int_{B_r (x)} |u-g|^2 \leq 4 \pi\|g\|^2_{L^\infty(B_r)} r^2\, ,
\]
in particular the fidelity term $\mathscr{F}$ becomes negligible compared to $E_0$ at small scales.

It would be desirable to use now the upper bound on $E_\lambda$ in order to provide a suitable compactness result. Note however that $E_0$ controls the values of the function only up to an additive constant in each connected component of $\Omega\setminus K$, while when $\param>0$ the fidelity term does not help because it is a lower order perturbation. Thus the space of rescalings of minimizers is in general not (pre)compact in a classical sense: roughly speaking to find a meaningful compactification we must allow the limits of the rescalings to take infinite values, and in fact we need to distinguish between infinities of ``different size''. 

We next deal with a slightly more general situation in which we make the following assumptions.

\begin{ipotesi}\label{a:blow-up}
We assume that
\begin{itemize}
    \item[(i)] $\lambda_j$ is a sequence of numbers in $[0,1]$;
    \item[(ii)] $g_j$ is a sequence of bounded functions with $\|g_j\|_\infty \leq M_0$;
    \item[(iii)] $(u^j, K^j)$ is a sequence of absolute or restricted minimizers of 
    \[
    E_{\lambda_j} (\cdot, \cdot, U_j, g_j) = E_0 (\cdot,\cdot, U_j) + \lambda_j \mathscr{F} (\cdot, U_j, g_j)
    \]
    on a sequence of domains $U_j$;
    \item[(iv)] $\lim_{j\uparrow\infty} \lambda_j r_j = 0$ and $\limsup_j \|u^j\|_\infty < \infty $;
    \item[(v)] a certain open domain $U$ satisfies $U'\subset \frac{U_j-x_j}{r_j}$ for every $U'\subset U$ compact and for $j$ large enough (depending on $U'$).
\end{itemize}
We define the rescaled sets and functions
\begin{align*}
 K_j = K^j_{x_j, r_j} &:= \frac{K^j-x_j}{r_j}\\
 u_j(y) = u^j_{x_j, r_j} (y) &:= 
 \frac{u^j (r_j y + x_j)}{r_j^{\sfrac{1}{2}}}\, .
\end{align*}
Up to subsequences (and using standard arguments) we can further assume the following.
\begin{itemize}
    \item[(a)] $K_j\cap U$ converges locally in the Hausdorff distance\index{Hausdorff distance@Hausdorff distance} to a closed set $K\cap U$, namely, for every open set $W\subset \subset U$ and every $\varepsilon >0$, the following holds for every sufficiently large $j$:
    \begin{itemize}
    \item $K_j \cap W \subset \{x : {\rm dist}\, (x, K) < \varepsilon\}$,
    \item $K\cap W \subset \{x: {\rm dist}\, (x, K_j)<\varepsilon\}$.
    \end{itemize}
    \item[(b)] We enumerate the connected components $\{\Omega_k\}_{k\in \mathscr{I}}$ of $U\setminus K$ and for each we select a point $z_k\in \Omega_k$. The real numbers 
\begin{equation}\label{e:rinormalizzazione-1}
\{u_j (z_k) - u_j (z_l)\}_{j\in \mathbb N} 
\end{equation}
converge to some elements $p_{kl}$\index[simb]{aalpkl@$p_{kl}$} of the extended real line $[-\infty, \infty]$.
\item[(c)] The functions
\begin{equation}\label{e:rinormalizzazione-2}
v^k_j := u_j - u_j (z_k)
\end{equation}
converge locally in $W^{1,2}_{{\rm loc}} (\Omega_k)$ to an harmonic function $v^k$ (this requires classical estimates for solution of the Laplace equation, for the details we refer the reader to the appendix).
\end{itemize}
We introduce further the function $v$ on $U\setminus K$ by setting $v= v^k$ on each $\Omega_k$.
\end{ipotesi}

{When $x_j=x$, $r_j\downarrow 0$, $\lambda_j=\lambda$}, $u^j\equiv u^0$ and
$K^j \equiv K^0$, the above sequence will be called a {\em blow-up sequence}\index{blow-up sequence@blow-up sequence}. Observe in addition that, because of (i) and (ii), (iv) is always satisfied for a blow-up sequence.
It is elementary to see that Theorem~\ref{t:dlb} and Lemma~\ref{l:upper-bound} imply that $K$ has locally finite $\mathcal{H}^1$ measure (cf. the proof in the appendix). The rectifiability of $K$ is instead more delicate. Leaving that aside for the moment, we introduce the following terminology.
\index[simb]{aaluKpk@$(v,K,\{p_{kl}\})$}

\begin{definition}\label{d:competitors}
The triple $(v,K,\{p_{kl}\})$ as above will be called a {\em limit point} of (an appropriate subsequence of) $(u_j, K_j)$ in the set $U$. In the case of a blow-up sequence, such triple will be called a {\em blow-up of $(u^0, K^0)$ at $x$}. \index{blow-up@blow-up} A pair $(w,J)$ will be called a ``topological competitor'' \index{topological comptetitor@topological competitor}\index{competitor, topological@competitor, topological} for the pair $(v,K)$ (cf. \cite{DavidBook}) if it satisfies the following assumptions:
\begin{itemize}
    \item[(i)] $(w,J)$ coincides with $(v,K)$ outside of an open set $O\subset\subset U$;
    \item[(ii)] if $x,y\in U \setminus (O\cup K)$ belong to distinct connected components of $U\setminus K$, then they belong to distinct connected components of $U\setminus J$ (see Figure~\ref{figura-1} for a simple illustration of the latter condition). 
\end{itemize}
\end{definition}

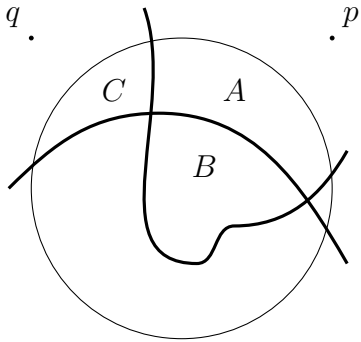
\begin{figure}
\begin{tikzpicture}
\draw (0,0) circle [radius = 2];
\draw[very thick] (-2.3,0) to [out=45, in=180] (-0.3,1) to [out=0, in=120] (2.2, -1);
\draw[very thick] (-0.5,2.4) to [out=290, in=180] (0.2,-1) to [out=0, in=180] (0.7,-0.5) to [out=0, in=240] (2.2,0.5);
\node at (0.7,1.3) {$A$};
\node at (0.3,0.3) {$B$};
\node at (-0.9,1.3) {$C$};
\draw[fill] (2,2) circle [radius=0.025];
\node[above right] at (2,2) {$p$}; 
\node[above left] at (-2,2) {$q$};
\draw[fill] (-2,2) circle [radius=0.025];
\end{tikzpicture}
\caption{A visual explanation of the admissibility in point (ii) of Definition~\ref{d:competitors}. The set $K$ is given by the thick lines, while the open set $O$ is the interior of the circle. In $O$ we are allowed to change $K$ to a new set $J$ under the condition that any two points outside $O$ which belong to distinct components of $U\setminus K$ will still belong to distinct connected components of $U\setminus J$. For instance, we cannot remove from $K$ any arc which lies between the regions $A$ and $C$, since such operation would ``connect'' the points $p$ and $q$. However we are allowed to remove from $K$ an arc which lies between $A$ and $B$.
The picture is also a good illustration of Lemma~\ref{l:ciao_componenti}. As long as $J\cap O$ ``separates'' the four arcs which are the connected components of $\partial O\setminus K$ in $O$, the pair $(w, J)$ is certainly a topological competitor for $(v,K)$.}\label{figura-1}
\end{figure}

It turns out that this is indeed a good candidate for a suitable variational problem in the limit.
The proof of the following theorem will be given in the appendix for the reader's convenience.

\begin{theorem}\label{t:minimizers compactness}\label{T:MINIMIZERS COMPACTNESS}
Let $(u_j, K_j)$ be a sequence as in Assumption~\ref{a:blow-up} and let $(v, K, \{p_{kl}\})$ be as in Definition~\ref{d:competitors}. Then $K$ is rectifiable and has locally finite $\mathcal{H}^1$ measure, while $v\in W^{1,2} (U'\setminus K)$ for every $U'\subset\subset U$. Moreover the following holds.
\begin{itemize}
\item[(i)] For every $O\subset\subset U$ such that $\mathcal{H}^1 (\partial O \cap K) = 0$, we have
\begin{align*}
&\lim_{j\to\infty} \mathcal{H}^1 (K_j\cap \overline{O}) = \mathcal{H}^1 (K\cap \overline{O})\\
&\lim_{j\to\infty} \int_{O\setminus K_j} |\nabla u_j|^2 = \int_{O\setminus K} |\nabla v|^2\, .
\end{align*}
For every continuous compactly supported function $\varphi:\mathbb P^1 \mathbb R\times U \to \mathbb R$ and for every $O$ bounded open set with $\mathcal{H}^1 (O\cap K)>0$ we have
\begin{equation}\label{e:varifold-convergence}
\lim_{j\to\infty} \int_{K_j\cap O} \varphi (T_x K_j, x) d\mathcal{H}^1 (x) = \int_{K\cap O} \varphi (T_x K, x)\, d\mathcal{H}^1 (x)\, .
\end{equation} 
\item[(ii)] In the case of absolute minimizers, if $(w, J)$ is a topological competitor as in Definition~\ref{d:competitors}, then
$E_0 (w, J) \geq E_0 (v, K)$. In the case of restricted minimizers the same holds if $J$ does not increase the number of connected components of $K$.
\item[(iii)] Let $\mathscr{A}\subseteq\mathscr{I}$\index[simb]{aalAscr@$\mathscr{A}$} be a set of indices with the property that $-\infty < p_{kl} < \infty$ for every $k,l\in \mathscr{A}$. Define
\begin{itemize}
\item $k_0:= \min \mathscr{A}$;
\item $U_\mathscr{A}:= \cup_{k\in \mathscr{A}} \Omega_k$ and $\Omega_\mathscr{A}:=\mathrm{int}\big(\overline{U_\mathscr{A}}\big)$;
\item $u_\mathscr{A}:= v^{k_0}$ on $\Omega_{k_0}$ and $u_\mathscr{A}:=v^k + p_{kk_0}$ on $\Omega_k$ for any $k\in \mathscr{A}$.
\end{itemize}
Then $(u_{\mathscr{A}}, K)$ is an absolute (resp. restricted) minimizer of $E_0$ on the set $\Omega_\mathscr{A}$.
\end{itemize}
\end{theorem}

\index{generalized minimizer@generalized minimizer}\index{global generalized minimizer@global generalized minimizer}\index{generalized restricted minimizer@generalized restricted minimizer}\index{global generalized restricted minimizer@global generalized restricted minimizer}

\begin{definition}\label{d:generalized}
A triple $(v, K, \{p_{kl}\})$ on $U$ as in the theorem above will be called a ``generalized minimizer'' resp. ``generalized restricted minimizer'' if it is the limit of absolute resp. restricted minimizers \index{generalized minimizer@generalized minimizer}\index{generalized minimizer, global@generalized minimizer, global}\index{global minimizer, generalized@global minimizer, generalized}. If $U=\mathbb R^2$, then the triple will be called ``global generalized minimizer'', resp. ``global generalized restricted minimizer''. Finally, an ``elementary global generalized minimizer'' is a global generalized minimizer whose Dirichlet energy vanishes identically on compact sets. \index{global minimizer, generalized@global minimizer, generalized}
\end{definition}

\begin{remark}\label{r:generalized-vs-topological}
It is worth emphasizing that the generalized minimizers as defined above are not the same as the ``topological minimizers''\index{topological minimizer@topological minimizer} defined by David in \cite{DavidBook}. The energy of the latter can only be compared to that of topological competitors, while the sense of point (iii) in Theorem~\ref{t:minimizers compactness} is that we have some more competitors which are not topological: in particular if $\Omega_k$ and $\Omega_l$ are two connected components of $\mathbb R^2\setminus K$ with the property that $p_{kl}\in \mathbb R$ sharing a common arc $\gamma\subset K$, we are allowed to use a competitor which removes a portion of $\gamma$, thus connecting the two regions. We would not be allowed to do this if instead $p_{kl}\in \{-\infty, \infty\}$. In particular, the role of the constants $p_{kl}$ is ultimately to tell us whether across a particular interface dividing two connected components we have an ``infinite jump''.
\end{remark}

In a variety of situations, it will be useful to have a quick way to identify whether a competitor $(w,J)$ satisfies the requirement (ii) of Definition~\ref{d:competitors}. The following simple remark will be widely used in that sense.

\begin{lemma}\label{l:ciao_componenti}
Let $(v,K)$ and $(w,J)$ be pairs in $\Omega$ which coincide outside of an open disk $B_r (x)\subset\subset \Omega$. Let $\{\gamma_i\}_i$ be the connected components of $\partial B_r (x)\cap K = \partial B_r (x) \cap J$. If no distinct pairs $\gamma_i$ and $\gamma_j$ are contained in the closure of the same connected component of $B_r (x)\setminus J$, then the pair $(w,J)$ is a topological competitor in the sense of Definition~\ref{d:competitors}(ii).
\end{lemma}

The proof is left as a simple exercise to the reader, see again Figure~\ref{figura-1}.

\section{Compactness for generalized minimizers}

The theorem above can be extended to generalized minimizers of $E_0$ (absolute or restricted). This remark will be especially used in ``blow-down'' \index{blow-down@blow-down} procedures for global generalized minimizers $(u, K, \{p_{ij}\})$, i.e. any limit of a subsequence of the rescalings $(u_{0,R}, K_{0,R}, \{p_{ij}\})$ as the radius diverges, where recall that $K_{0,R}=\frac KR$ and $u_{0,R}(y)=R^{-\frac12}u(Ry)$.

\begin{ipotesi}\label{a:compattezza}
We assume that
\begin{itemize}
    \item[(i)] $U_j$ is a sequence of monotonically increasing domains which converge to some domain $\Omega$;
    \item[(ii)] $(u_j, K_j)$ is part of a triple $(u_j, K_j, \{p_{kl,j}\}_j)$ of generalized (restricted) minimizers.
\end{itemize}
Up to subsequences we further assume the following.
\begin{itemize}
    \item[(a)] $K_j$ converges locally in the Hausdorff distance to a closed set $K$;
    \item[(b)] We enumerate the connected components $\{\Omega_k\}_{k\in \mathscr{I}}$ of $\Omega\setminus K$, for each we select a point $x_k\in \Omega_k$. For each fixed $x_k, x_l$ will belong to appropriate connected components $\Omega_{k,j}$ of $U_j \setminus K_j$ for $j$ large enough and we will assume that 
\begin{equation}\label{e:rinormalizzazione-4}
\{u_j (x_k) - u_j (x_l) - p_{kl,j}\}_{j\in \mathbb N}\,  
\end{equation}
converge to some elements $p_{kl}$ of the extended real line;
\item[(c)] The functions
\begin{equation}\label{e:rinormalizzazione-5}
v^k_j := u_j - u_j (x_k)
\end{equation}
converge locally in $W^{1,2}_{{\rm loc}} (\Omega_k)$ to an harmonic function $v^k$.
\end{itemize}
We introduce further the function $v$ on $\Omega\setminus K$ by setting $v= v^k$ on each $\Omega_k$.
\end{ipotesi}

\begin{theorem}\label{t:compactness-2}\label{T:COMPACTNESS-2}
Under the Assumption~\ref{a:compattezza} all the conclusions of Theorem~\ref{t:minimizers compactness} apply to the triple $(v, K, \{p_{kl}\})$ and the corresponding sequence $\{(u_j, K_j)\}$. 
\end{theorem}

\section{Elementary global generalized minimizers} We next turn to the simplest type of global generalized minimizers, i.e. those for which the Dirichlet energy vanishes identically. \index{elementary global minimizer@elementary global minimizer}\index{global minimizer, elementary@global minimizer, elementary}\index{pure jump, global@pure jump, global}\index{global pure jump@ global pure jump}\index{triple junction, global@triple junction, global}\index{global triple junction@global triple junction}

\begin{theorem}[Classification of elementary global generalized minimizers]\label{t:class-global}
Let $(v, K, \{p_{kl}\})$ be a global generalized minimizer of $E_0$ and assume that $\int |\nabla v|^2 =0$. Then $(v, K, \{p_{kl}
\})$ is either
\begin{itemize}
\item[(a)] A {\em constant}, namely $K=\emptyset$ and $v$ is a constant.
\item[(b)] A {\em global pure jump}, namely 
\begin{itemize}
\item[(b1)] $K$ is a straight line,
\item[(b2)] $v$ is constant on each connected component $\Omega_1$ and $\Omega_2$ of $\mathbb R^2\setminus K$, 
\item[(b3)] and $|p_{12}|=\infty$.
\end{itemize}
\item[(c)] A {\em global triple junction}, namely: 
\begin{itemize}
\item[(c1)] $K$ is the union of three half lines originating at a common point where they form equal angles,
\item[(c2)] $v$ is constant on each of the three connected components $\Omega_1, \Omega_2$ and $\Omega_3$ of $\mathbb R^2\setminus K$,
\item[(c3)] and $|p_{12}|=|p_{13}|=|p_{23}|= \infty$.
\end{itemize}
\end{itemize}
For global generalized restricted minimizers such that $K\neq \emptyset$, then $K$ is as in (b1) or as in (c1). The conclusions (b2) and (c2) hold as well. The conclusions (b3) and (c3) do not hold for all generalized restricted minimizers, but do hold for those satisfying the following stronger variational property:
\begin{itemize}
    \item $E_0 (v, K, U) \leq E_0 (w, J, U)$ for any bounded open set $U$ and any pair $(w, J)$ such that $\{v\neq w\} \subset \subset U$ and for which $J$ consists of at most two connected components.
\end{itemize}
Finally, if $(v, K, \{p_{kl}\})$ is a global generalized minimizer, then:
\begin{itemize}
\item[(i)] If $K$ is empty and $(v, K, \{p_{kl}\})$ is an {\em absolute} minimizer, then it is necessarily a constant;
\item[(ii)] If $K$ is a straight line, then it is necessarily a pure jump;
\item[(iii)] If $K$ is the union of three half lines originating at a common point, then it is necessarily a triple junction.
\end{itemize}
\end{theorem}
\begin{proof} We focus on the case of generalized minimizers, leaving the analogous one of generalized restricted minimizers to the reader.

We start off proving items (i)-(iii) in the second part of the statement. 

Assume first $K=\emptyset$, then $v$ is harmonic on $\mathbb{R}^2$. 
In view of the density upper bound (cf. \eqref{e:upper bound}) 
and the mean value property for harmonic functions we conclude 
that $|\nabla v|=0$ on $\mathbb{R}^2$. 

If $K$ is a line, without loss of generality, we may assume $K=\{x\in \mathbb{R}^2:\,x_2=0\}$
and set $H^\pm:=\{x\in \mathbb{R}^2:\,\pm x_2>0\}$.
Then $v$ is harmonic on $H^+\cup H^-$ and $\frac{\partial v}{\partial\nu}=0$ on $K$. By the Schwartz reflection principle, 
the even extensions $v_\pm$ of $v|_{H^\pm}$ across $K$ are harmonic on the whole of $\mathbb{R}^2$ with \[
\int_{B_\rho}|\nabla v_\pm|^2dx=2\int_{H^\pm\cap B_\rho}|\nabla v|^2dx\leq 4\pi\rho.
\]
Arguing as above, we infer that $|\nabla v|=0$ on $H^+\cup H^-$, therefore $v$ is locally constant on 
$\mathbb{R}^2\setminus K$. To conclude the proof of item (ii) we have to show that $|p_{12}|=\infty$.
Otherwise setting $\mathscr{A} = \{1,2\}$ and $w= u_\mathscr{A}$, we can apply Theorem~\ref{t:compactness-2} to the rescalings
$u_{0,R}(x):=R^{-\frac12}v(Rx)$ and $K_{0,R}$. Obviously they converge to $(0,K, \{0\})$ as $R\uparrow \infty$, and so $(\tilde{u}, \tilde{K}) = (0, K)$ would have to be an absolute minimizer on any bounded open subset of $\mathbb R^2$. The latter assertion is false, as we can remove any compact subset of $K$ and extend the function $\tilde{u}$ to $0$ on it. 

In case (iii), $\mathbb{R}^2\setminus K=\Omega_1\cup\Omega_2\cup\Omega_3$, each $\Omega_i$ being a convex cone with 
vertex $p$ and opening $\alpha_i\in(0,2\pi)$. 
Recalling that $\triangle v=0$ on $\Omega_i$ and $\frac{\partial v}{\partial\nu}=0$ on $\partial\Omega_i\cap K$, 
we may expand $w_i:=v|_{\Omega_i}$ in Fourier series.
In particular, given a point $x\in \Omega_i$, we set $r=\mathrm{dist}(x,p)$ and we let $\vartheta\in[0,\alpha_i]$ be the angle formed be the segment $[p,x]$ and one of the two half-lines delimiting $\Omega_i$ (the choice is not important). 
Hence, if we denote by  $a_{i,k}$ the Fourier coefficients of $v|_{\partial B_1\cap\Omega_i}$ in the angle $\vartheta$, we can write
\[
w_i(r,\vartheta)=\sum_{k=0}^\infty a_{i,k}r^{\frac{k\pi}{\alpha_i}}\cos\Big(\frac{k\pi}{\alpha_i}\vartheta\Big)\,,
\] 
on $\Omega_i$, while we can compute
\begin{equation}\label{e:enrg harmonic}
\int_{\Omega_i\cap B_r (p)}|\nabla w_i|^2=\sum_{k=1}^\infty\frac{k\pi}{2}a_{i,k}^2r^{2\frac{k\pi}{\alpha_i}}\,.
\end{equation}
Since $2k\pi\geq2\pi>\alpha_i$, $k\geq1$, we conclude that $a_{i,k}=0$ for all $k\geq 1$ and $i\in\{1,2,3\}$, thanks again to inequality
\eqref{e:upper bound}. 
Thus, $v$ is locally constant on $\mathbb{R}^2\setminus K$. 
In turn, being $v$ a generalized global minimizer, $K\cap B_1 (p)$ is a set with minimal length 
connecting three points on $\partial B_1 (p)$, in particular the angles in $p$ must be all equal to $\frac23\pi$.

Assume now $p_{ij}$ is finite for some $i\neq j$. Let $\mathscr{A} = \{i,j\}$ and consider $w= u_\mathscr{A}$ on $\Omega_{\mathscr{A}} = {\rm int} (\overline{\Omega_i \cup \Omega_j})$. Without loss of generality we can assume that $\partial \Omega_i \cap \partial \Omega_j = \{x_2=0, x_1\geq 0\}$. We can then choose the points $y_k := (k^2,0)$ and the radii $R_k := k$ and consider the pairs $(w_{y_k,R_k},K_{y_k, R_k})$ on the domains $B_{\sqrt{k}} (0)$ obtained by rescaling and translating $\Omega_{\mathscr{A}}$. We can again apply Theorem~\ref{t:compactness-2} and we would conclude as above that $(0,K)$ is a generalized minimizer, which is a contradiction. 

Finally, we prove the classification of generalized global minimizers with null gradient energy as stated in (a)-(c). In this case $K$ is a minimal Caccioppoli partition and then a minimal connection of 
$\partial B_R\cap K$ for all $R$ (cf. \cite[Lemma 12]{DLF13}). 
In particular, $R\mapsto\frac{\mathcal{H}^1(\partial B_R\cap K)}{R}$ is nondecreasing and $\frac KR$ is converging to a minimal cone $K_\infty$ as $R\uparrow \infty$. Therefore, $K_\infty$ is either a line or a propeller\index{propeller@propeller}, the union of three half-lines meeting at a common point with equal angles. In turn, this implies that $\mathcal{H}^0(\partial B_R\cap K)\leq 3$ at least for some sequence of $R$'s converging to infinity by the coarea formula \cite[Theorem 2.93]{AFP00} (because $R^{-1} \mathcal{H}^1 (B_R\cap K) \leq 3$ for every $R>0$). In particular, either 
$\mathcal{H}^0(\partial B_\rho\cap K)=3$ for some $\rho$, and in this case $B_\rho\cap K$ is a propeller, or $\mathcal{H}^0(\partial B_\rho\cap K)=2$ for some $\rho$ and in this case $B_\rho\cap K$ is a segment.
Moreover, in the first case $\mathcal{H}^0(\partial B_R\cap K)=3$ for all $R\geq\rho$, and thus $K$ is a propeller, while in the second case 
$\mathcal{H}^0(\partial B_R\cap K)=2$ for all $R\geq\rho$, and $K$ is a line (cf. \cite[Sections 2.3 and 3]{DLF13}).

We conclude the proof of (a)-(c) by appealing to the classification result contained in (i)-(iii).
\end{proof} 

A related classification result, expressed in terms of one-dimensional minimal sets in 
$\mathbb R^n$, is contained in \cite[Theorem 10.1]{David2009}.

\section{Variational identities}

As we have already remarked, restricted and absolute minimizers $(u,K)$ of $E_\param$
are critical points, namely they satisfy the identities  \eqref{e:outer} and \eqref{e:inner}. The same applies to generalized minimizers of $E_0$. We can thus summarize our conclusions in the following statement (for the proof see Appendix~\ref{a:variational identities}).

\begin{proposition}\label{p:variational identities}
Assume $(u,K)$ is a (restricted or absolute) minimizer of $E_\lambda$ or a generalized (restriced or absolute) minimizer of $E_0$. Then
$(u,K)$ is a critical point, namely there exists suitable traces $u^+$, $u^-$, and $g_K\in L^\infty(\Omega,\mathcal H^1\restr K)$ such that the identities \eqref{e:outer} and \eqref{e:inner} hold, with the properties that $\|g_K\|_\infty \leq \|g\|_\infty$ and that $u^\pm$ are the classical one-sided traces of $K$ in the sense of Sobolev-space theory. If $g$ is, in addition, $C^1$, \eqref{e:inner} is indeed equivalent to \eqref{e:inner-C1}. 
\end{proposition}
A more explicit form of the Euler-Lagrange conditions can be devised in case $K$ is a smooth graph. To this aim we denote by
$\nu$ the counterclockwise rotation by $90$ degrees of a $C^0$ unit tangent vector $e$ locally orienting $K$, 
while we denote by $\kappa$\index[simb]{aagk@$\kappa$} the curvature of a local compatible parametrization, namely 
\begin{equation}\label{e:rigorous-def-curvature}
\kappa = \ddot\gamma \cdot \nu\, ,
\end{equation}
for an arclength parametrization $\gamma$ such that $\dot\gamma = e$. Such classical definition of the curvature $\kappa$ assumes in general $C^2$ regularity. We will use it under the assumption that $K$ is $C^{1,1}$: the reader can check that local arc-length parametrizations belong to $W^{2,\infty}$ and thus $\ddot\gamma$ is interpreted as an $L^\infty$ function of $t$. Under such assumptions the curvature $\kappa$ is then a bounded Borel function defined $\mathcal{H}^1$-a.e. on $K$.
Finally, $w^+$ and $w^-$ are the one-sided traces of the relevant function $w$ on $K$ (following the obvious convention that $w^+$ is the trace on the side which $\nu$ is pointing to, cf. Figure~\ref{figura-3}). 

\begin{figure}
\begin{tikzpicture}
\draw[>=stealth,->] (-0.5,0) -- (3.5,0);
\node[below] at (3.5,0) {$x_1$};
\draw[>=stealth,->] (0,-0.5) -- (0,2.9);
\node[left] at (0,2.9) {$x_2$};
\draw[very thick] (0,0) to [out=0, in=200] (1.5,0.7) to [out=20, in=180] (2.7, 0.3) to [out = 0, in=225] (3.3,0.6);
\draw[>=stealth,->] (1.5,0.7) -- ({1.5-1.2*cos(20)},{0.7-1.2*sin(20)});
\draw[fill] (1.5,0.7) circle [radius=0.05];
\draw[>=stealth,->] (1.5,0.7) -- ({1.5+1.2*sin(20)},{0.7-1.2*cos(20)});
\node[right] at (1.5,1.1) {$p = \gamma (t)$};
\node[above] at ({1.5-1.2*cos(20)},{0.7-1.2*sin(20)}) {$e(p)$};
\node[below] at  ({1.5+1.2*sin(20)},{0.7-1.2*cos(20)}) {$\nu (p)$};
\node[below left] at (1.5,0.5) {$+$};
\node[above left] at (1.5,0.7) {$-$};
\end{tikzpicture}
\caption{The tangent vector $e(p) = \dot\gamma (t)$ (for an-arc length parametrization) and the normal vector $\nu (p)$. The picture illustrates the convention for the symbols $\pm$ on traces of functions over $\gamma$.\label{figura-3}}
\end{figure}
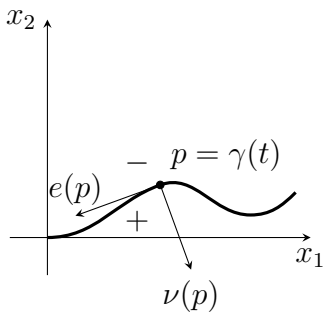

\begin{proposition}\label{p:variational identities 2}
Let $(u, K)$ be a critical point of $E_\lambda$ in $U$ and assume that $K\cap U$ consists only of regular jump points (i.e. $K\cap U$ is a $C^1$ submanifold\footnote{In particular we are assuming that $U\cap K$ does not contain any loose end of $K$.}). Then
\begin{itemize}
\item[(a)] $u$ has $C^{1,\alpha}$ extensions on each side of $K\cap A$, for every $\alpha<1$, 
\item[(b)]$K\cap A$ is locally $C^{1,1}$,
\item[(c)] the variational identities \eqref{e:outer}-\eqref{e:inner} are equivalent to the following three conditions
\begin{align}
&\Delta u = \lambda (u-g) \qquad \mbox{on $\Omega\setminus K$}\label{e:Euler harmonic g}\\
&\frac{\partial u}{\partial \nu} = 0 \qquad \mbox{on $K$}\label{e:Euler Neumann g}\\
& \kappa = - (|\nabla u^+|^2- |\nabla u^-|^2) - \lambda (|u^+-g_K|^2-|u^--g_K|^2)\qquad 
\mbox{$\mathcal{H}^1$ a.e. on $K$\,,} \label{e:Euler curvature g}
\end{align}
\end{itemize}
where $\kappa$ is the curvature of $K$ and $g_K$ is the function in Proposition~\ref{p:variational identities}. 
\end{proposition}
The proof is given in the appendix for the reader's convenience.
Therefore, in view of Remark~\ref{r:pointwise} and item (a) above, in the rest of the notes we can simply assume that $u$ is continuously differentiable in $\Omega\setminus K$. 

We record here an important elementary consequence of \eqref{e:outer} which will be used throughout the notes 
\begin{corollary}\label{c:Bonnet}
Assume $(u,K)$ is a critical point of $E_\lambda$ in some domain $\Omega$ and fix $x\in \Omega$. For a.e. $r\in (0, \dist (x, \partial \Omega))$ the following holds. First of all we have
\begin{equation}\label{e:int by parts}
\int_U |\nabla u|^2 + \lambda \int_U (u-g) u= \int_{\partial B_r (x)\cap \overline{U}} u \frac{\partial u}{\partial n}    
\end{equation}
for every connected component $U$ of $B_r (x)\setminus K$ (where $n$ is the unit normal to $\partial B_r (x)$). 

Moreover if $\lambda=0$ and $\gamma$ is a connected component of $\partial B_r (x)\setminus K$ such that the endpoints of $\gamma$ belong to the same connected component of $K$, then
\begin{equation}\label{e:Bonnet}
\int_\gamma \frac{\partial u}{\partial n} = 0\, 
\end{equation}
(see Figure~\ref{figura-4} for an illustration of the two conclusions).
\end{corollary}

\begin{figure}
\begin{tikzpicture}
\draw[>=stealth,->] (0,2) -- (0,2.8);
\node[right] at (0,2.4) {$n$};
\draw (0,0) circle [radius = 2];
\draw[very thick] (-0.5,2.4) to [out=290, in=180] (0.2,-1) to [out=0, in=180] (0.7,-0.5) to [out=0, in=240] (2.2,0.5);
\node at (0.3,0.3) {$U$};
\node[above right] at ({sqrt(2)},{sqrt(2)}) {$\gamma$};
\end{tikzpicture}
\begin{tikzpicture}
\draw[>=stealth,->] (0,2) -- (0,2.8);
\node[right] at (0,2.4) {$n$};
\draw (0,0) circle [radius = 2];
\draw[very thick] (-0.5,1.5) to [out=135, in=180] (0.2,4) to [out=0, in=180] (0.7,3.5) to [out=45, in=0] (1.5,0.5);
\node at (1,2.3) {$U$};
\node[above right] at ({sqrt(2)},{sqrt(2)}) {$\gamma$};
\end{tikzpicture}

\caption{The arc $\gamma$ is the set $\overline{U}\cap \partial B_r (x)$ and since its two endpoints belong to the same connected component of $K$, both conclusions of Corollary~\ref{c:Bonnet} apply.
On the left the case $U\subset B_r(x)$, on the right the case 
$U\subset\Omega\setminus \overline{B}_r(x)$.}
\label{figura-4}
\end{figure}
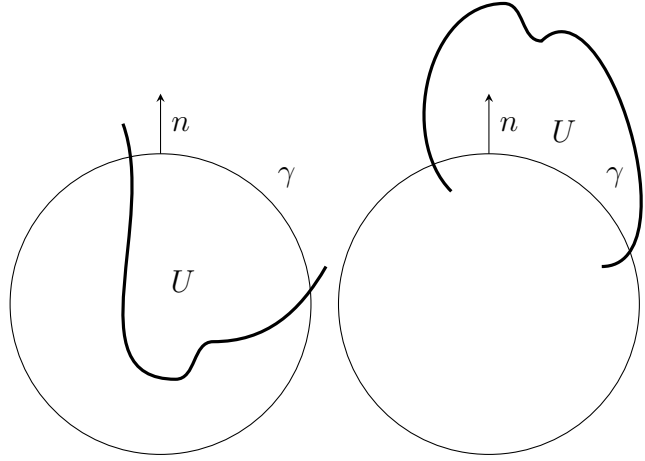

\begin{proof} Without loss of generality we assume $x=0$. We will prove the claims for any radius $r$ with the following property. First of all define the domain
\[
\hat{K}:= K \cup \{y\in \Omega: {\textstyle{\frac{ry}{|y|}}}\in K\}
\]
and we require that it is a Lebesgue null set. This is certainly the case for all $r$ such that $\partial B_r\cap K$ is finite. Additionally we require that $u|_{\partial B_r (x)\setminus K}\in W^{1,2} (\partial B_r (x)\setminus K)$ and that
\begin{equation}\label{e:coroncina}
\lim_{\delta\downarrow 0} \frac{1}{\delta} \int_{B_{r+\delta}\setminus (B_{r-\delta}\cup \hat{K})} \left(|\nabla u (x) - 
\nabla u ({\textstyle{\frac{rx}{|x|}}})|^2 + |u (x) - u ({\textstyle{\frac{rx}{|x|}}})|^2\right)\, dx =0\, .
\end{equation}
This is again a property which certainly holds for a.e. $r$.

Let us first deal with the second statement and fix a corresponding $\gamma$. Then there is one connected component $U$ of $\Omega\setminus (\gamma \cup K)$ with the property that $U\subset \subset \Omega$ and $\partial U$ contains $\gamma$. 

There are now two possibilities, either $U\subset B_r (x)$, or $U\subset \Omega\setminus \overline{B}_r (x)$. In the first we case we define 
\begin{equation}\label{e:psi_delta}
\psi_\delta (x)
:= \left\{\begin{array}{ll}
1 & \mbox{for $|x|\leq 1-\delta$}\\
1-\frac1\delta (|x|-1 +\delta) &\mbox{for $1-\delta \leq |x| \leq 1$}\\
0 &\mbox{otherwise,}
\end{array}
\right.
\end{equation}
while in the second we define
\[
\psi_\delta (x)
:= \left\{\begin{array}{ll}
0 & \mbox{for $|x|\leq 1$}\\
\frac1\delta (|x|-1) &\mbox{for $1 \leq |x| \leq 1 + \delta$}\\
1 &\mbox{otherwise.}
\end{array}
\right.
\]
We then define the function $\chi$ on $\Omega\setminus (K\cup \gamma)$ as constantly equal to $1$ on $U$ and constantly equal to $0$ otherwise, and we set $\varphi = \chi \psi_\delta(\frac{\cdot}{r})$. 
We then test \eqref{e:outer} with $\varphi$ and let $\delta\downarrow 0$. Using \eqref{e:coroncina} we
easily get \eqref{e:Bonnet} in the limit.

As for the first statement, we test \eqref{e:outer} with $\chi u \psi_\delta(\frac{\cdot}{r})$ and let $\delta\downarrow 0$, with $\psi_\delta$ as in \eqref{e:psi_delta}. The proof is entirely analogous and we leave the details to the reader.
\end{proof}

\subsection{Truncated tests}
While obviously we cannot plug in \eqref{e:inner} a test field which is not compactly supported, a standard approximation argument allows to derive an equivalent identity in that case too. This fact will play a pivotal role in the proof of several theorems described in these notes: specific choices of the test field will in fact deliver some remarkable identities. 
Analogous results can be inferred from \eqref{e:outer}, as done for instance in 
Corollary~\ref{c:Bonnet}. We do not work out the details since we do not need those 
results in these notes.   
\begin{proposition}\label{p:boundary-variations}
Let $(u,K)$ be a critical point of $E_\param$ in $\Omega$ and $y\in \Omega$. For a.e. $r\in (0, \dist (y, \partial \Omega))$ the following identity holds for every $\eta\in C^1 (\overline{B}_r(y), \mathbb R^2)$ 
\begin{align}\label{e:int-variation-bdry} 
&\int_{B_r (y) \setminus K} \big(|\nabla u|^2\operatorname{div} \eta 
-2 \nabla u^T \cdot D \eta\, \nabla u\big)
+ \int_{B_r (y)\cap K} e^T \cdot D \eta\, e   \, d\mathcal{H}^1 \, \nonumber\\
= &\int_{\partial B_r (y) \setminus K}\left(|\nabla u|^2 \eta \cdot n - 2\frac{\partial u}{\partial n}\nabla u\cdot\eta
\right) d\mathcal{H}^1 +
\sum_{p\in K \cap \partial B_r (y)} e(p) \cdot \eta (p)\nonumber\\
& +2\param \int_{B_r(y)\setminus K}(u-g)\nabla u\cdot\eta
+\param \int_{B_r(y)\cap K}\big(|u^+-g_K|^2-|u^--g_K|^2\big)
\eta\cdot \nu\, d\mathcal{H}^1 \, ,
\end{align}
where $n (x) = \frac{x-y}{|x-y|}$\index[simb]{aaln@$n$} is the exterior unit normal to the circle centered in $y$,
$e(p)$ is the tangent unit vector to $K$ at $p$ such that $e (p)\cdot n (p) >0$, and $\nu(p)=e(p)^\perp$.
\end{proposition}

\begin{proof}
We follow essentially the proof of \cite{DFG}. In that reference we take advantage of the regularity theory to avoid technicalities. In principle we could still consistently follow the same approach: the regularity needed to carry over the proof of \cite{DFG} relies only on case (b) of Theorem~\ref{t:main}, while Proposition~\ref{p:boundary-variations} will be used only in the proof of the cases (a) and (c). However, in order to keep our notes streamlined, we give here a proof which does not rely on any regularity. Moreover, since the proof in the case of $\param >0$ is just a minor adjustment, we will focus on the case of $E_0$. 

Fix $r>0$ and $\eta\in C^\infty_c (\Omega, \mathbb R^2)$ and without loss of generality assume $y=0$. Consider $\psi_\delta$ as in \eqref{e:psi_delta}. We would like to plug $\psi_{\delta}(\frac{\cdot}{r}) \eta$ in \eqref{e:inner}, however the latter is only Lipschitz continuous. We therefore first plug in \eqref{e:inner} a suitable smoothed version $\psi_{\delta, \varepsilon}(\frac{\cdot}{r}) \eta$, where $\psi_{\delta,\varepsilon}$ could for instance be the convolution of $\psi_{\delta}$ with a standard smooth kernel $\varphi_\epsilon$. We then wish to pass to the limit as $\varepsilon\downarrow 0$. Observe first of all that the first summand in both the left and the right hand sides of \eqref{e:inner} would carry to the obvious limits, respectively.
Next recall that, by the coarea formula \cite[Theorem 2.93]{AFP00}, for $\mathcal{H}^1$-a.e. $r$ the intersection of $K$ with $\partial B_r$ is finite. In particular, for $\mathcal{H}^1$-a.e. $\delta$ the function $\psi_{\delta}(\frac{\cdot}{r}) \eta$ is differentiable $\mathcal{H}^1$-a.e. on $K$ and the
second term on the left hand side of \eqref{e:inner} would then make sense. It is also a simple measure theoretic exercise to see that the identity \eqref{e:inner} in fact holds under that assumption. We then can explicitly compute
\begin{align*} 
&\int_{B_r\setminus K} \psi_{\delta}({\textstyle{\frac{x}{r}}}) (|\nabla u|^2 \operatorname{div} \eta - 2 \nabla u^T \cdot D \eta\, \nabla u) + 
\int_{B_r \cap K} \psi_{\delta}({\textstyle{\frac{x}{r}}})
e^T \cdot D \eta\, e 
\, d\mathcal{H}^1 \,\nonumber\\
=&\frac{1}{r\delta} \int_{B_{r}\setminus B_{r(1-\delta)}} \left(|\nabla u|^2 \eta\cdot {\textstyle{\frac{x}{|x|}}}
- 2 (\eta \cdot\nabla u)(\nabla u\cdot {\textstyle{\frac{x}{|x|}}})\right)
+ \frac{1}{r\delta} \int_{K\cap B_{r}\setminus B_{r (1-\delta)}} (e \cdot \eta) (e\cdot {\textstyle{\frac{x}{|x|}}})\, d\mathcal{H}^1 \nonumber\\
& +2\param \int_{B_r\setminus K}\psi_{\delta}({\textstyle{\frac{x}{r}}})(u-g)\nabla u\cdot\eta
+\param\int_{B_r\cap K}\psi_{\delta}({\textstyle{\frac{x}{r}}})
\big(|u^+-g_K|^2-|u^--g_K|^2\big)\eta\cdot
\nu\,d\mathcal{H}^1 \, .
\end{align*}
The terms on the left hand side in the equation above and the third and fourth integrals on the right hand side converge as $\delta\downarrow 0$ to the obvious limits by dominated convergence.
The first term on the right hand side converge to the obvious limit for every $r$ which satisfies \eqref{e:coroncina} as in the proof of Corollary~\ref{c:Bonnet}. 
The remaining term on the right hand side can be written in the following way using the coarea formula \cite[Theorem 2.93]{AFP00}:
\[
\frac{1}{r\delta}\int_{r(1-\delta)}^r \Big(\sum_{p\in K \cap \partial B_s} e(p) \cdot \eta (p)\Big)
\, ds\, .
\]
We then observe that, by standard measure theory, for a.e. $r$ the integral converges, as $\delta\downarrow 0$, to the corresponding one in \eqref{e:int-variation-bdry}. We thus conclude that \eqref{e:int-variation-bdry} holds for a.e. $r$. 

The argument above makes, however, the choice of the radii dependent on the fixed smooth vector field $\eta$. In order to gain a set of full measure for which \eqref{e:int-variation-bdry} is valid for every test $\eta$ we use the following standard argument. We first select a countable family $\{\eta_j\}\subset C^\infty (\Omega, \mathbb R^2)$ which is dense in the $C^1$ topology in the space $C^1_c (\Omega, \mathbb R^2)$. Hence we observe that the above argument gives a set of full measure of radii $r$ in $(0, \dist (y, \partial \Omega))$ for which \eqref{e:int-variation-bdry} is valid for every $\eta_j$. Next we fix an $r$ in this set and an $\eta\in C^1 (\overline{B}_r (y), \mathbb R^2)$. We extend the latter to a vector field in $C^1_c (\Omega,\mathbb R^2)$, hence select a sequence $\eta_{j_n}$ such that $\|\eta_{j_n}-\eta\|_{C^1}\to 0$ and derive \eqref{e:int-variation-bdry} as limit of the corresponding identities for $\eta_{j_n}$.
\end{proof}
We collect next specific choices of the test field which will be crucial in several instances (cf. Lemma~\ref{l:DLMS}, Proposition~\ref{p:monotonia-1}, Proposition~\ref{p:linearizzazione}, Proposition~\ref{p:tre-anelli_odd}).
\begin{corollary}\label{c:boundary-variations}
Let $(u,K)$, $y$, and $r$ be a as in Proposition~\ref{p:boundary-variations}.
If $\eta\in C^1 (\overline{B}_r, \mathbb R^2)$ is conformal, then
\begin{align} 
&\int_{B_r(y)\cap K} e^T \cdot D\eta \,  e \, d\mathcal{H}^1
= \int_{\partial B_r(y) \setminus K} \left(|\nabla u|^2 \eta \cdot n - 2\frac{\partial u}{\partial n}
\nabla u\cdot\eta\right)d\mathcal{H}^1 +
\sum_{p\in K \cap \partial B_r(y)} e(p) \cdot \eta (p)\notag\\
&+2\param \int_{B_r(y)\setminus K}(u-g)\nabla u\cdot\eta
+\param \int_{B_r(y)\cap K} \big(|u^+-g_K|^2-|u^--g_K|^2\big)\eta\cdot \nu\, d\mathcal{H}^1 \, .\label{e:conformal} 
\end{align}
In particular, for every constant vector $v\in\mathbb{R}^2$ we have 
\begin{align}\label{e:translations}
0 &= \int_{\partial B_r(y) \setminus K}  \left(|\nabla u|^2 v \cdot n - 2\frac{\partial u}{\partial n}\frac{\partial u}{\partial v}\right)d\mathcal{H}^1 +\sum_{p\in K \cap \partial B_r(y)} e(p) \cdot v\notag\\
& +2\param \int_{B_r(y)\setminus K}(u-g)\nabla u\cdot v + \param \int_{\partial B_r(y)\cap K}\big(|u^+-g_K|^2-|u^--g_K|^2\big)u\cdot \nu\, d\mathcal{H}^1\, ,
\end{align}
and if $\tau(x)=\frac{(x-y)^\perp}{|x-y|}$, then
\begin{align}\label{e:rotations}
&0= - 2 \int_{\partial B_r(y)\setminus K}\frac{\partial u}{\partial n}\frac{\partial u}{\partial \tau}d\mathcal{H}^1
+\sum_{p \in K \cap \partial B_r(y)} e(p) \cdot \tau(p)
+ 2 \param \int_{B_r(y)\setminus K}(u-g)\nabla u\cdot \tau
\,.
\end{align}
\end{corollary}
\begin{proof} Recall that, if 
 $\eta\in C^1 (\overline{B}_r, \mathbb R^2)$ is conformal, then $D\eta(x)=\vartheta(x)\mathbb{O}(x)$,
 with $\vartheta(x)>0 $ and
 $\mathbb{O}(x)\in O(2)$ for all $x\in \overline{B}_r$: in particular $|\xi|^2 {\rm div}\, \eta - 2 \xi^T \, D\eta \, \xi$ vanishes identically for every vector $\xi$, leading immediately to formula \eqref{e:conformal}. The identities in
 \eqref{e:translations} and \eqref{e:rotations} are then simple consequences of \eqref{e:conformal} after we choose $\eta(x)\equiv v$ and
 $\eta(x)=(x-y)^\perp$, respectively.
\end{proof}

\subsection{The factor of the cracktip}
Consider the pair $(u,K)$ given by $K= \mathbb R^+$ and, in polar coordinates, the function on
$\mathbb R^2\setminus K$ given by
\begin{equation}\label{e:radice}
u (\theta, r) = br^{\frac{1}{2}} \cos \frac{\theta}{2}\, .
\end{equation}
It is straightforward to see that $(u,K)$ is a critical point of $E_0$ on any set $\Omega = B_r\setminus B_\delta$ for positive $r$ and $\delta$ by taking advantage of Proposition~\ref{p:variational identities 2}
Note however that $K$ is not any more smooth at the tip: in a suitable variational sense the curvature of $K$ is singular at the origin. This singularity must be somewhat balanced by the variation of the Dirichlet energy and, remarkably, one outcome is that the constant $b$ is then determined up to sign. Proposition~\ref{p:boundary-variations} gives indeed a very short proof.

\begin{proposition}\label{p:constant-crack-tip}
Assume $K = \mathbb R^+$ and $u$ is given by \eqref{e:radice}. If $(u,K)$ is a critical point of $E_0$, 
then $b^2=\frac{2}{\pi}$.
\end{proposition}
\begin{proof}
We use identity \eqref{e:translations} in Corollary~\ref{c:boundary-variations} 
on $B_r (x)= B_1 (0)$ with vector $v=(1,0)$.
On the other hand, using polar coordinates, $|\nabla u|^2 (\theta, 1)= \frac{b^2}{4}$, while $\eta\cdot n = \cos \theta$ and hence \eqref{e:int-variation-bdry} becomes
\begin{equation}\label{e:semplice-1}
1 = 2 \int_{\partial B_1\setminus \{(1,0)\}} \frac{\partial u}{\partial n} \frac{\partial u}{\partial \eta}\, .
\end{equation}
We then compute
\begin{align}
\frac{\partial u}{\partial n} (1, \theta) & = \frac{b}{2} \cos \frac{\theta}{2}\\
\frac{\partial u}{\partial \theta} (1, \theta) &=-\frac{b}{2} \sin \frac{\theta}{2}\\
\frac{\partial u}{\partial \eta} (1, \theta) &= \cos \theta \frac{\partial u}{\partial n} (1, \theta) - \sin \theta \frac{\partial u}{\partial \theta} (1, \theta)=
\frac{b}{2} \left(\cos \theta \cos \frac{\theta}{2} + \sin \theta \sin \frac{\theta}{2}\right)\, .
\end{align}
So the right hand side of \eqref{e:semplice-1} equals
\begin{align*}
& 2 \int_0^{2\pi} \frac{b^2}{4} \left(\cos^2 \frac{\theta}{2} \cos \theta + \cos \frac{\theta}{2} \sin \frac{\theta}{2} \sin \theta\right)\,  d\theta
= \frac{b^2}{2} \int_0^{2\pi} \left(\cos \theta \frac{\cos \theta+1}{2} + \frac{\sin^2 \theta}{2}\right)\, d\theta\\
= & \frac{b^2}{4} \int_0^{2\pi} (1+\cos \theta)\, d\theta = \frac{\pi b^2}{2}\, ,
\end{align*}
and inserting it in \eqref{e:semplice-1} we achieve $b^2 = \frac{2}{\pi}$.
\end{proof}
We remark that \cite[Proposition~2.1]{LemMik18} provides a proof that
$(u,K)$, with $u$ defined in \eqref{e:radice} and $K=\mathbb R^+$, is a critical point of $E_0$ 
on $\mathbb{R}^2$.

\subsection{L\'eger's ``magic formula''}

A remarkable discovery of L\'eger in \cite{Leger} is a closed singular integral formula for 
\[
\left(\frac{\partial u}{\partial x}-i\frac{\partial u}{\partial y}\right)^2
\]
when $(u,K)$ is a global critical point of the Mumford-Shah functional. This formula will not be used in our notes, but since it can be seen as a simple and direct consequence of the inner variation identity, we include a very short proof in this section.

\begin{proposition}\label{p:Leger-magic}
Assume $(u,K)$ is a global generalized restricted minimizer. Then the following formula holds for every $(x_0,y_0)\not \in K$:
\begin{equation}\label{e:Leger-magic}
\left(\frac{\partial u}{\partial x} - i \frac{\partial u}{\partial y}\right)^2 (x_0,y_0) = - \frac{1}{2\pi} \int_K \frac{d\mathcal{H}^1 (x,y)}{((x-x_0) + i (y - y_0))^2} \, .
\end{equation}
\end{proposition}
\begin{remark}
Introducing the complex coordinate $z= x+iy$ the formula \eqref{e:Leger-magic} can be elegantly rewritten as:
\begin{equation}\label{e:Leger-magic-2}
\left(\frac{\partial u}{\partial z}\right)^2 (z_0) 
= - \frac{1}{8\pi} \int_K \frac{d\mathcal{H}^1 (w)}{(w-z_0)^2}\, .
\end{equation}
\end{remark}
\begin{remark}
The assumptions on the global minimality of $(u,K)$ can be considerably relaxed. The proof only uses the facts that the pair $(u,K)$ is a critical point and that the growth estimate
\begin{equation}\label{e:growth-Leger}
\mathcal{H}^1 (K\cap B_r) + \int_{B_r\setminus K} |\nabla u|^2 \leq C r
\end{equation}
holds for all sufficiently large disks.
\end{remark}

\begin{proof}
First of all, by translation invariance it suffices to prove the formula when $(x_0, y_0) = 0$. Next observe that the real part of \eqref{e:Leger-magic} reads
\begin{equation}\label{e:Leger-magic-real}
\left[\left(\frac{\partial u}{\partial x}\right)^2 - 
\left(\frac{\partial u}{\partial y}\right)^2\right] (0) = \frac{1}{2\pi} \int_K \frac{y^2 -x^2}{(x^2 + y^2)^2}d \mathcal{H}^1\, ,
\end{equation}
while the imaginary part is in fact equivalent to \eqref{e:Leger-magic-real} in the system of coordinates which results from a $45$ degrees counterclockwise rotation of the standard one. We focus therefore on the proof of \eqref{e:Leger-magic-real}. Since $(u,K)$ is a critical point of the Mumford-Shah functional without fidelity term, the inner variations \eqref{e:inner-C1} reads
\begin{equation}\label{e:internal-2.5.3}
\int_{\mathbb R^2\setminus K} (2 \nabla u^T\cdot D\psi \nabla u - |\nabla u|^2 {\rm div}\, \psi) = \int_K e^T \cdot D\psi\, e\, d\mathcal{H}^1\, .
\end{equation}
We fix positive radii $\rho<R$ and consider the vector field $\psi (x,y) = \varphi (|(x,y)|) (x, -y)$ where 
\[
\varphi (t) = \left\{
\begin{array}{ll}
\rho^{-2} - R^{-2} \qquad &\mbox{if $t\leq \rho$}\\
t^{-2} - R^{-2} \qquad &\mbox{if $\rho \leq t \leq R$}\\
0 &\mbox{otherwise}\, .
\end{array}\right.
\]
Strictly speaking the latter is not a valid test in \eqref{e:internal-2.5.3} because it is not continuously differentiable. However, if we assume that $\mathcal{H}^1 (K\cap (\partial B_\rho \cup \partial B_R))=0$, it is easily seen that the right hand side \eqref{e:internal-2.5.3} makes sense because $\psi$ is $\mathcal{H}^1$-a.e. differentiable on $K$, while a standard regularization argument, analogous to the ones already used in the previous sections, shows the validity of the formula. Next we compute $D\psi$ in the two relevant domains where it does not vanish:
\begin{align*}
D\psi & = \left(\frac{1}{\rho^2}- \frac{1}{R^2}\right)
\left(\begin{array}{ll}
1 & 0\\
0 & -1
\end{array}\right)
&\mbox{on $B_\rho$}\\
D\psi &=  -\frac{1}{R^2} \left(\begin{array}{ll}
1 & 0\\
0 & -1
\end{array}\right) + \frac{1}{(x^2+y^2)^2} 
\left(\begin{array}{ll}
y^2-x^2 & 2 xy\\
- 2 xy & y^2-x^2
\end{array}
\right) \qquad &\mbox{on $B_R\setminus \bar B_\rho$}\, .
\end{align*}
Choose then $\rho$ sufficiently small to have $B_\rho\cap K = \emptyset$ and obtain
\begin{align*}
&\frac{2}{\rho^2}\int_{B_\rho} \left(\left(\frac{\partial u}{\partial x}\right)^2 - \left(\frac{\partial u}{\partial y}\right)^2\right) -\frac{2}{R^2} \int_{B_R} \left(\left(\frac{\partial u}{\partial x}\right)^2 - \left(\frac{\partial u}{\partial y}\right)^2\right)\\
= & \frac{1}{R^2} \int_{K\cap B_R} (e_2^2-e_1^2)d\mathcal{H}^1 
+ \int_{K\cap B_R} \frac{y^2-x^2}{(x^2+y^2)^2}\, d\mathcal{H}^1\, .
\end{align*}
Then, we first let $R\uparrow \infty$ and use \eqref{e:growth-Leger} to obtain
\[
\frac{2}{\rho^2} \int_{B_\rho} \left(\left(\frac{\partial u}{\partial x}\right)^2 - \left(\frac{\partial u}{\partial y}\right)^2\right)
= \int_K \frac{y^2-x^2}{(x^2+y^2)^2}\, d\mathcal{H}^1\, ,
\]
and hence we let $\rho\downarrow 0$ and use the regularity of $u$ at $0$ to infer \eqref{e:Leger-magic-real}.
\end{proof}

\section{Monotonicity formulae}

An important role in our arguments will be played by three monotonicity statements, valid for minimizers of $E_\lambda$ (irrespectively whether they are absolute, restricetd, generalized, or generalized restricted). The first statement was discovered by Bonnet in \cite{B96}, while the other two were discovered and proved by David and L\'eger in \cite{DL02}. The first and the second are considerably easier to prove and we will show them in this section. 

\begin{proposition}\label{p:monotonia-0}
Assume $(u,K)$ is a critical point of $E_0$. Fix $x\in \Omega$ and consider
\begin{equation}
d(x,r):=\frac{D(x,r)}{r} := \frac{1}{r} \int_{B_r (x)\setminus K} |\nabla u|^2\, .
\end{equation}
Then $r\mapsto D(x,r)$ is an absolutely continuous function and $\frac{d}{dr} 
d(x,r)\geq 0$ for a.e. $r\in (0, \dist (x, \partial \Omega))$ which satisfies the following property:
\begin{itemize}
    \item[(i)] The set $K\cap \partial B_r (x)$ belongs to the same connected component of $K$.
\end{itemize}
Moreover, if (i) holds, $d(x,r)$ 
is constant on the interval $(0, r_0)$ and in addition $(u,K)$ is a restricted, an absolute, or a generalized minimizer, then either $\nabla u=0$ $\mathcal{L}^2$-a.e. on $B_{r_0}(x)$ or $(u,K)$ is a cracktip with loose end located at $x$.
\end{proposition}

\begin{proposition}\label{p:monotonia-1}
Let $(u,K)$ be an absolute, restricted, generalized, or generalized restricted minimizer of $E_\lambda$ in $\Omega$. Fix $x\in \Omega$ and consider
\begin{equation}\label{e:funzionale-F}
F (x,r) := \frac{2}{r} \int_{B_r (x)\setminus K} |\nabla u|^2 + \frac{1}{r} \mathcal{H}^1 (B_r (x)\cap K) =: 2d(x,r)
+ \frac{\ell (x,r)}{r}\, .
\end{equation}
Then $r\mapsto F (x,r)$ is a function of bounded variation. The singular part of its derivative is a nonnegative measure, while there is a constant $C$ such that the absolutely continuous part $F'(r)$ 
satisfies
\begin{equation}\label{e:DLMS monotonia}
F'(x,r)\geq\min\Big(1,3-\frac{2\alpha}\pi\Big)
\frac{D'(x,r)}{r} - C \param
\end{equation}
at a.e. radius $r\in(0,\dist(x,\partial\Omega))$ such that
\begin{itemize}
\item[(i)] each connected component of $\partial B_r (x)\setminus K$ has length less or equal to $\alpha r$, with $\alpha\leq\frac{3}{2}\pi$.
\end{itemize}
Moreover, if $\param =0$, $r\mapsto F (x,r)$ is constant on $(0, r_0)$ and (i) holds for a.e. $r\in (0,r_0)$, then $(u,K)$ coincides in $B_{r_0} (x)$ with an elementary global minimizer as in Theorem~\ref{t:class-global}(b) and (c). In both cases necessarily $x\in K$ and in case (c) it must in fact be a triple junction.  
\end{proposition}
\index[simb]{aalD(x,r)@$D(x,r)$}
\index[simb]{aald(x,r)@$d(x,r)$}
\index[simb]{aall(x,r)@$\ell(x,r)$}
\index[simb]{aalF(x,r)@$F(x,r)$}
The third is more laborious and will only be used in the second part of these notes. We therefore postpone its proof to Section~\ref{s:monotonia-2}.

\begin{proposition}\label{p:monotonia-2} Let $(u,K)$, $\Omega$, $x$, and $F$ be as in Proposition~\ref{p:monotonia-1} with $\param =0$.
 Then $F'(x,r)\geq 0$ at a.e. point $r$ such that
\begin{itemize}
\item[(i)] $N (r) := \sharp (K\cap \partial B_r (x)) \in \{0\} \cup [3, \infty)$
\item[(ii)] or $N(r)=2$ and the two points of $K\cap \partial B_r (x)$ belong to the same connected component of $K$.
\end{itemize}
Moreover, if $r\mapsto F (x,r)$ is constant on $(0, r_0)$ and (i) or (ii) hold for a.e. $r\in (0,r_0)$, then $(u,K)$ coincides in $B_{r_0} (x)$ with an elementary global minimizer as in Theorem~\ref{t:class-global}(a), (b), or (c). In both cases (b) and (c) necessarily $x\in K$ and in case (c) it must in fact be a triple junction.  
\end{proposition}

\begin{remark}\label{r:possible-monotonicity}
While it is plausible that (ii) in Proposition~\ref{p:monotonia-2} could be weakened to the single assumption $N(r)=2$, the monotonicity of $F$ is certainly false when $N(r)=1$. We indeed show here that it fails for any functional of type $F_c (x,r) := c\, d(x,r) + \frac{\ell (x,r)}{r}$.
Consider the cracktip pair $(u,K)$ of Definition~\ref{d:crack-tip} and change it to the pair $(u_a,K_a)$ where $K_a = K + (a,0)$ and $u_a (x_1, x_2) = u (x_1-a, x_2)$. We then introduce the function
\[
f (a,r) := F_c((-a,0),r)=
\frac{c}{r} \int_{B_r} |\nabla u_a|^2 + \frac{\mathcal{H}^1 (K_a\cap B_r)}{r}=:
D (a,r) + L (a,r)
\]
and we study it on the domain $\Lambda := \{|a|<\frac{1}{2}, |r-1|< \frac{1}{2}\}$. First of all observe that $L(a,r) = \frac{r-a}{r}$
and it is thus a smooth function on $\Lambda$. Next we integrate by parts and write 
\[
D(a,r) = \frac{c}{r} \int_{\partial B_r\setminus (1,0)}
u_a \frac{\partial u_a}{\partial n}\, .
\]
The latter formula shows immediately that $D$ is smooth as well on $\Lambda$. Next observe that
$|\nabla u_a|^2 (x) = \frac{1}{2\pi|x-(a,0)|}$
and in particular by symmetry we conclude
$D (a,r) = D (-a,r)$. We thus infer 
$\frac{\partial^2 D}{\partial a \partial r} (0,r)= 0$.
On the other hand we can explicitly compute
$\frac{\partial^2 L}{\partial a \partial r} = \frac{1}{r^2}$ and thus we get
$\frac{\partial^2 f}{\partial a \partial r} (0,r) = \frac{1}{r^2}$.
Since it is obvious that $\frac{\partial f}{\partial r} (0,r)=0$, 
there is $\delta >0$ such that
\[
\frac{\partial f}{\partial r} (a,r) <0
\qquad \forall r\in [{\textstyle{\frac{3}{4}}},
{\textstyle{\frac{5}{4}}}], \forall a\in [-\delta, 0)\, .
\]
This shows that the function $r\to F_c((-a,0),r)$
is certainly not monotone 
for $a<0$ sufficiently small. In fact a simple scaling argument shows that the monotonicity fails on some interval for every $a$ negative. 
\end{remark}

\subsection{The David-L\'eger-Maddalena-Solimini identity} A first important ingredient in the proofs of both Proposition~\ref{p:monotonia-1} and \ref{p:monotonia-2} is given by an interesting identity discovered independently by David and L\'eger in \cite{DL02} and Maddalena and Solimini in \cite{MadSol01} for critical points of $E_0$. In the following we state its version for critical points of $E_\lambda$

\begin{lemma}\label{l:DLMS}
Let $F$ be as in Proposition~\ref{p:monotonia-1}. If $(u,K)$ is an absolute, restricted, generalized, or a generalized restricted minimizer of $E_\param$ in $\Omega$, then for every $y\in\Omega$ and for a.e. $r\in(0,\dist(y,\partial\Omega))$ we have
\begin{align}\label{e:DLMS}
&\int_{\partial B_r (y)\setminus K} \left(\frac{\partial u}{\partial n}\right)^2
= \int_{\partial B_r (y)\setminus K} \left(\frac{\partial u}{\partial \tau}\right)^2 -
\frac{\mathcal{H}^1 (K\cap B_r (y))}{r} + \sum_{p\in \partial B_r (y)\cap K} e(p) \cdot n(p)\notag\\
& +\frac{2\param}{r} \int_{B_r(y)\setminus K}(u-g) \nabla u \cdot (x-y)
+\frac{\param}{r} \int_{B_r(y)\cap K}\big(|u^+-g_K|^2-|u^--g_K|^2\big)
(x-y)\cdot\nu\, d\mathcal{H}^1
\end{align}
where the vectors $n$ and $e$ are as in Proposition~\ref{p:boundary-variations}, in particular $n (p) \cdot e (p)>0$ for all $p\in \partial B_r (y)\cap K$,
and $\tau(p)=n^\perp(p)$\index[simb]{aagtau@$\tau$}.
\end{lemma}

\begin{proof} Test with $\eta (x) = \frac{x-y}{r}$ the equation \eqref{e:int-variation-bdry}.
\end{proof}

\subsection{Elementary estimate on harmonic extensions} A second important ingredient is an elementary estimate on harmonic extensions.

\begin{lemma}\label{l:extension}
Let $U\subset B_r (x)$ be an open set such that
\begin{itemize}
    \item $\partial B_r (x)\cap \partial U$ is a connected arc $\gamma$;
    \item $U$ is contained in a circular sector of angle $\alpha < 2\pi$.
\end{itemize}
Then for every $g\in W^{1,2} (\gamma)$ there is an harmonic extension $v\in W^{1,2} (U)$ such that
\begin{equation}\label{e:extension-estimate}
\int_U |\nabla v|^2 \leq \frac{\alpha r}{\pi} \int_\gamma \left(\frac{\partial g}{\partial \tau}\right)^2\, .
\end{equation}
\end{lemma}
\begin{proof} Since the statement is invariant under rotations, translations, and dilations, without loss of generality we assume $x=0$, $r=1$, and, using polar coordinates $(\theta, \rho)$, $U\subset V := \{\rho<1, 0<\theta < \alpha\}$. Observe that $\gamma = \{\rho =1, a\leq \theta\leq b\}$ with $0\leq a < b\leq \alpha$. Extend now $g$ to $\{\rho =1, 0\leq \theta\leq \alpha\}$ by setting it to be constant on the arcs $\{\rho =1, 0\leq \theta\leq a\}$ and
$\{\rho =1, b\leq \theta\leq \alpha\}$: recalling that $g$ is continuous by Morrey's embedding, it is obvious that the latter extension can be achieved in $W^{1,2}$. It then suffices to find a $W^{1,2}$ extension to the whole sector $V$ which enjoys the desired bound. In other words, we can assume without loss of generality that $U$ is itself the sector $V := \{\rho<1, 0<\theta < \alpha\}$. Consider now $g$ as a function on the interval $[0,\alpha]$ and extend it to an even $W^{1,2}$ function on $[-\alpha, \alpha]$, which can be thought as a periodic function on $\mathbb R$ with period $2\alpha$. In particular we can write its Fourier series as
\[
g (\theta) = a_0 + \sum_{k\geq 1} a_k \cos k \frac{\pi}{\alpha} \theta\, .
\]
We then consider the harmonic extension
\[
v (\theta, \rho) = a_0 + \sum_{k\geq 1} a_k \rho^{k\pi/\alpha} \cos k \frac{\pi}{\alpha} \theta
\]
and standard computations yield
\[
\int_V |\nabla v|^2 = \sum_{k\geq 1} \frac{k\pi}{2} a_k^2
\leq \sum_{k\geq 1} \frac{k^2\pi}{2} a_k^2 = \frac{\alpha}{\pi} \int_\gamma \left(\frac{\partial g}{\partial \tau}\right)^2\, .\qedhere
\]
\end{proof}

\subsection{Proof of Proposition~\ref{p:monotonia-1}} We assume without loss of generality that $x=0$, thus we drop the dependence on the base point in all the relevant quantities, i.e. for instance 
$D(r) = \int_{B_r\setminus K} |\nabla u|^2$ and $\ell (r) = \mathcal{H}^1 (B_r\cap K)$. 
Observe that $r\mapsto D (r)$ is an absolutely continuous function with
\begin{equation}\label{e:E'}
D'(r) = \int_{\partial B_r\setminus K} |\nabla u|^2 = \int_{\partial B_r \setminus K} \left(\frac{\partial u}{\partial n}\right)^2 + \int_{\partial B_r \setminus K} \left(\frac{\partial u}{\partial \tau}\right)^2\, .
\end{equation}
$r\mapsto \ell (r)$ is a monotone nondecreasing function, and hence a function of bounded variation. Moreover, the absolutely continuous part of is derivative equals, by the coarea formula \cite[Theorem 2.93]{AFP00},
\begin{equation}\label{e:ell'}
\ell' (r) = \sum_{p\in \partial B_r \cap K} \frac{1}{e(p)\cdot n (p)}\, 
\end{equation}
(where we follow the notation of Proposition~\ref{p:boundary-variations}). 
The first claim of the proposition is thus obvious, while for a.e. $r$ we have
\begin{equation}\label{e:F'}
r^2 F' (r) = 2 r \int_{\partial B_r\setminus K} |\nabla u|^2 + r \sum_{p\in \partial B_r \cap K} \frac{1}{e(p)\cdot n (p)} - 2 D(r) - \ell (r)\, .
\end{equation}
We first prove the conclusion for $\param =0$.
Using \eqref{e:DLMS} we get
\begin{align}
r^2 F' (r) =& 3 r \int_{\partial B_r \setminus K} \left(\frac{\partial u}{\partial \tau}\right)^2 + r \int_{\partial B_r \setminus K} \left(\frac{\partial u}{\partial n}\right)^2\nonumber\\
& + r \sum_{p\in \partial B_r \cap K} \left(\frac{1}{e(p)\cdot n (p)} + e (p)\cdot n (p)\right)
-2 (D(r)+\ell (r))\nonumber\\
\geq & 3 r \int_{\partial B_r \setminus K} \left(\frac{\partial u}{\partial \tau}\right)^2
+ r \int_{\partial B_r \setminus K} \left(\frac{\partial u}{\partial n}\right)^2
+ 2 r N(r) -2 E_0 (u,K,B_r)\, ,\label{e:3-tau-bound}
\end{align}
where $N(r) := \sharp (K\cap \partial B_r)$\index[simb]{aalNr@$N(r)$}. Consider next a competitor $(w,J)$ for $(u,K)$ in the following fashion: 
\begin{itemize}
    \item $(w,J) = (u,K)$ outside $B_r$;
    \item $J\cap B_r$ consists of $N(r)$ straight segments joining each point of $\partial B_r\cap K$ with the origin; 
    \item on each connected component of $B_r\setminus J$, which according to assumption (i) is a circular sector with angle $\alpha\leq\frac{3\pi}{2}$, we let $w$ be the extension of the trace of $u$ on the corresponding circular arc given by Lemma~\ref{l:extension}.
\end{itemize}    
Observe that $J$ does not increase the number of connected components and that we can apply Lemma~\ref{l:ciao_componenti}. In particular $(w,J)$ is a topological competitor for restricted and generalized restricted minimizers as well. We can use estimate 
\eqref{e:extension-estimate} in Lemma~\ref{l:extension} to infer \begin{align*}
    E_0 (u,K,B_r) \leq & E_0 (w,J,B_r) = \int_{B_r\setminus J} |\nabla w|^2 + r N(r)
    \leq  \frac{\alpha}{\pi}r \int_{\partial B_r\setminus K} \left(\frac{\partial u}{\partial \tau}\right)^2 + r N(r)\, .
\end{align*}
Combined with \eqref{e:3-tau-bound} we then conclude 
\[
r^2F'(r)\geq 
\Big(3-\frac{2\alpha}{\pi}\Big) r \int_{\partial B_r \setminus K} \left(\frac{\partial u}{\partial \tau}\right)^2+
r \int_{\partial B_r \setminus K} \left(\frac{\partial u}{\partial n}\right)^2.
\] 
In particular, we deduce from \eqref{e:E'} that
\[
F'(r)\geq
\min\Big(1,3-\frac{2\alpha}{\pi}\Big) \frac{D'(r)}r.
\]
Finally, if $F$ is constant on $(0, r_0)$ and (i) holds for a.e. $r\in (0, r_0)$, we would necessarily conclude from the last but one inequality that, for a.e. $r\in (0, r_0)$,
\[
\int_{\partial B_r \setminus K} \left(\frac{\partial u}{\partial n}\right)^2 = 0\, .
\]
Equality \eqref{e:int by parts} easily implies that $u$ must be locally constant on $B_{r_0}\setminus K$, and thus $K$ in $\partial B_r$ is a minimizing network. More precisely, if $r$ is a radius such that $N(r) <\infty$, $K\cap \overline{B}_r$ consists of finitely many connected components $K_1, \ldots, K_j$ and each $K_j$ is a connected set which minimizes the length among all closed connected sets $\hat{K}\subset \overline{B}_r$ with $\hat{K}\cap \partial B_r = K_j \cap B_r$. In particular we conclude that $K_j\cap B_r$ is either a single segment or it consists of a network of a finite number of segments joining at triple junction.

Observe next that from \eqref{e:3-tau-bound}, being $u$ locally constant on 
$B_{r_0}\setminus K$, the constancy of $F$ implies 
\[
2 N(r) = \sum_{p\in \partial B_r \cap K} \left(\frac{1}{e (p) \cdot n (p)} + e(p)\cdot n (p)\right)\, 
\]
for a.e. $r\in(0,r_0)$. For every such $r$ we must then have 
$e (p) = n (p)$ for every $p\in K \cap \partial B_r$. 
In particular we conclude that $K\cap B_{r_0}$ is a cone centered at the origin. But then it is either a straight segment or it is a collection of three radii meeting at the origin, or it is the empty set. The latter is excluded by assumption (i).

We now come to the monotonicity statement for $\param >0$. We plug \eqref{e:DLMS} in \eqref{e:F'} to infer the following inequality from Cauchy-Schwartz and the energy 
upper bound in \eqref{e:upper bound}:
\begin{align}
r^2 F' (r) =& 3 r \int_{\partial B_r \setminus K} \left(\frac{\partial u}{\partial \tau}\right)^2 + r \int_{\partial B_r \setminus K} \left(\frac{\partial u}{\partial n}\right)^2\nonumber\\
& + r \sum_{p\in \partial B_r \cap K} \left(\frac{1}{e(p)\cdot n (p)} + e (p)\cdot n (p)\right)
-2 (D(r)+\ell (r))\nonumber\\
    & +2\param \int_{B_r\setminus K}(u-g)\nabla u\cdot y
+\param \int_{B_r\cap K}\big(|u^+-g_K|^2-|u^--g_K|^2\big) y\cdot \nu\, d\mathcal{H}^1
    \nonumber\\
\geq & 3 r \int_{\partial B_r \setminus K} \left(\frac{\partial u}{\partial \tau}\right)^2
+ r \int_{\partial B_r \setminus K} \left(\frac{\partial u}{\partial n}\right)^2
+ 2 r N(r)
-2 E_\param (u,K,B_r)-C\param r^2\,,
    \label{e:3-tau-bound-g}
\end{align}
where $C$ depends on $\|g\|_{\infty}$ thanks to Lemma~\ref{l:maximum}.
We use the competitor $(w,J)$ defined above to get
\begin{align*}
    E_\param (u,K,B_r) \leq & E_\param (w,J,B_r) =   \int_{B_r\setminus J} |\nabla w|^2 
    + r N(r) +\param \int_{B_r}(w-g)^2\\
    &\leq  \frac{\alpha}{\pi}r \int_{\partial B_r\setminus K} \left(\frac{\partial u}{\partial \tau}\right)^2 + r N(r) + 4\pi\param \|g\|^2_{\infty}r^2\,
\end{align*}
(note that $w$ consists of harmonic extensions in $B_r\setminus J$ and since the trace of $w$ on $\partial B_r$ is bounded by $\|g\|_{\infty}$, we see immediately that $\|w\|_{\infty}\leq \|g\|_{\infty}$).
The latter estimate combined with \eqref{e:3-tau-bound-g} gives
\[
r^2F'(r)\geq 
\Big(3-\frac{2\alpha}{\pi}\Big) r \int_{\partial B_r \setminus K} \left(\frac{\partial u}{\partial \tau}\right)^2+
r \int_{\partial B_r \setminus K} \left(\frac{\partial u}{\partial n}\right)^2-C\param r^2\,,
\] 
and \eqref{e:DLMS monotonia} follows at once.

\subsection{Proof of Proposition~\ref{p:monotonia-0}} First of all we start by assuming, without loss of generality, that $\nabla u$ does not vanish identically. 

We follow the same approach of the previous section and will just carry on our computations assuming that we have selected a good radius. We thus have 
\begin{equation}\label{e:D'}
\frac{d}{dr}d(r)=
\frac{d}{dr} \frac{D(r)}{r} = \frac{1}{r} \int_{\partial B_r (x)\setminus K} |\nabla u|^2 - \frac{D(r)}{r^2}\,. 
\end{equation}
Observe that we are in a position to apply both statements in Corollary~\ref{c:Bonnet}. In particular, if we let $c_i$ be the average of $u$ on any connected component $\gamma_i$ of $\partial B_r (x)\setminus K$ we can write
\begin{align}\label{e:caso costante}
D(r) \stackrel{\eqref{e:int by parts}}{=} 
& \int_{\partial B_r (x)\setminus K} u \frac{\partial u}{\partial n} \stackrel{\eqref{e:Bonnet}}{=}
\sum_i \int_{\gamma_i} (u-c_i) \frac{\partial u}{\partial n}
\leq  \sum_i \left(\int_{\gamma_i} (u-c_i)^2\right)^{\frac{1}{2}}
\left(\int_{\gamma_i} \left(\frac{\partial u}{\partial n}\right)^2 \right)^{\frac{1}{2}} \nonumber\\
\leq & \sum_i \frac{\mathcal{H}^1 (\gamma_i)}{\pi} 
\left(\int_{\gamma_i} \left(\frac{\partial u}{\partial \tau}\right)^2\right)^{\frac{1}{2}}
\left(\int_{\gamma_i} \left(\frac{\partial u}{\partial n}\right)^2 \right)^{\frac{1}{2}}\nonumber\\
\leq & 2 r \left(\int_{\partial B_r\setminus K} \left(\frac{\partial u}{\partial \tau}\right)^2\right)^{\frac{1}{2}}
\left(\int_{\partial B_r\setminus K} \left(\frac{\partial u}{\partial n}\right)^2 \right)^{\frac{1}{2}}
\leq r \int_{\partial B_r\setminus K} \left(\left(\frac{\partial u}{\partial \tau}\right)^2+\left(\frac{\partial u}{\partial n}\right)^2\right)\nonumber\\
= & r\int_{\partial B_r\setminus K} |\nabla u|^2\, ,
\end{align}
where he have used that the sharp constant in the $1$-dimensional Poincar\'e-Wirtinger inequality on an interval of length $L$ is $\frac{L^2}{\pi^2}$.

Next notice that, if $\frac{D(r)}{r}$ is constant on $(0,r_0)$ then the equality holds for a.e. $r\in(0,r_0)$ in all the inequalities above. First of all observe that, if $\partial B_r\cap K$ consists of more than one point, then all the connected components of $\partial B_r\setminus K$ have length strictly less than $2\pi$. This would give a strict inequality sign at the beginning of the third line, unless one of the two functions $\frac{\partial u}{\partial n}$ and $\frac{\partial u}{\partial \tau}$ vanish identically on $\partial B_r\setminus K$. If however only one of them vanishes identically on $\partial B_r\setminus K$, then the second inequality in the third line is strict. So they would have to vanish both identically on $\partial B_r\setminus K$. 
Then necessarily from \eqref{e:D'} we would get
\[
D(r)=r\int_{\partial B_r\setminus K} |\nabla u|^2=0.
\]
So, if there is a set of radii  of positive measure such that $\partial B_r\cap K$ consists of more than one point, then $\nabla u$ vanishes identically on $B_r$ for any such radius $r$. 

Observe that the same argument implies that there cannot be a set of radii of positive measure for which $\partial B_r \cap K$ is empty, because the constant in the sharp Poincar\'e-Wirtinger inequality on the unit circle equals the constant of the sharp Poincar\'e-Wirtinger inequality on an interval of length $2\pi$, and we know that on the interval $[0, 2\pi[$ the latter is achieved by the functions of type
\begin{equation}\label{e:optimal function PW}
a + b \cos \frac{\theta}{2}
\end{equation}
where $a$ and $b$ are constants.

In particular, we conclude that there exists a subset $\mathcal{R}$ of $(0, r_0)$ of full measure, functions $a,b: \mathcal{R} \to \mathbb R$, and a function $c: \mathcal{R}\to \mathbb S^1$ with the following properties 
\begin{itemize}
\item[(a)] $K\cap \partial B_r = \{(r\cos c (r), r \sin c (r))\} =: \{p (r)\}$ for all $r\in \mathcal{R}$;
\item[(b)] $u (\theta, r) = a (r) + b (r) \cos \left(\frac{\theta - c (r)}{2}\right)$ in polar coordinates\footnote{In the latter formula we understand that the angle of the polar coordinates is taken to vary in $[c(r), c(r)+2\pi[$, so that the $\theta \mapsto u (\theta, r)$ has at most a jump discontinuity at $\theta = c(r)$ and is smooth everywhere else.};
\item[(c)] $b$ never vanishes on $\mathcal{R}$.
\end{itemize}
Note in particular that $K\cap \partial B_r$ consists of exactly one point for a.e. $r$. 
Fix next $r\in \mathcal{R}$ and
next define the function
\[
\bar{u} (\theta, \rho) = a(r) + b(r) (r^{-1}\rho) ^{\frac{1}{2}} \cos \frac{\theta -c(r)}{2}
\]
and let $\bar{K}$ 
be the segment with endpoints the origin and $p(r)$.

This is obviously a competitor for $(u,K)$. Moreover, a direct computation gives immediately 
\[
\int_{B_r\setminus\bar K} |\nabla \bar u|^2 = 2r \int_{\partial B_r\setminus \bar K} \left(\frac{\partial \bar{u}}{\partial \tau}\right)^2 = 2r \int_{\partial B_r\setminus K} \left(\frac{\partial u}{\partial \tau}\right)^2 =
\int_{B_r\setminus K} |\nabla u|^2\,, 
\]
where in the last equality we use the optimality conditions which can be derived from the constancy of $\frac{D(r)}{r}$ (cf. \eqref{e:caso costante}).
So, by minimality of $(u,K)$, $\mathcal{H}^1 (K) \leq \mathcal{H}^1 (\bar K)$. In particular, since we already know that $\mathcal{H}^1 (K\cap B_r)\geq r$ we actually conclude
\begin{equation}
\mathcal{H}^1 (K\cap B_r) = r\, .    
\end{equation}
This holds for a.e. $r$ and thus implies that the approximate tangent to the rectifiable set $K$ must in fact be orthogonal to the circle $\partial B_r$ at the point $(r\cos c(r), r \sin c(r))$ for a.e. $r\in (0, r_0)$. It also implies that, if we define
\[
R:= \{(r\cos c(r), r\sin c(r)): r\in \mathcal{R}\}\, ,
\]
then $\mathcal{H}^1 (K\setminus R)=0$. On the other hand, by the density lower bound, this also means that $R$ is dense in $K$.

Consider again a radius $r\in \mathcal{R}$ and, after applying a rotation, assume without loss of generality that $c(r) = \frac{\pi}{2}$. Notice that, since $K$ is closed, there is a positive $\delta$ with the property that, if $\rho \in (r-\delta, r+\delta)$, the open set $U := \{(\rho \cos \phi, \rho \sin \phi) : r-\delta < \rho < r+\delta, 0 < \phi < \frac{\pi}{4}\}$ does not intersect $K$. We use the addition formula for the cosine to write 
\[
u (\theta, r) = a(r) + b (r) \cos \frac{c(r)}{2} \cos \frac{\theta}{2} + b(r) \sin \frac{c(r)}{2} \sin \frac{\theta}{2} 
\]
on the set $U$. Note that $u$ is smooth over $U$ and in particular the map
\[
r\mapsto u (\theta, r)
\]
must be smooth on the interval $I = (r-\delta, r+\delta)$ for every fixed $\theta \in (\frac{2\pi}{4}, \pi)$. This immediately implies that the three functions $a$, $b \cos \frac{c}{2}$ and $b \sin \frac{c}{2}$ have smooth extensions from $\mathcal{R}\cap I$ to $I$, because varying $\theta$ it is easy to find three linearly independent linear combinations of these functions which are smooth. Using that $\cos^2 + \sin^2=1$, we then conclude that also $b^2$ has a smooth extension to $I$. Now, if such smooth extension were to vanish at some point $\rho$, we then would have that the trace of $u$ is constant on $\partial B_\rho$. But this would immediately imply the constancy of the function $u$ in $B_\rho$ and then, as already argued, that $\nabla u \equiv 0$ on $B_{r_0}$. 

Being that $b^2$ has a smooth extension which is bounded away from zero, $b$ itself has a smooth extension if it does not change sign over $I$.
Now, again because $K$ is closed, for any fixed $\delta$ we can assume that $c(\rho)\in (\frac{\pi}{2}-\delta, \frac{\pi}{2}+\delta)$ for all $\rho\in I$, so that in particular $\cos \frac{c}{2}$ is positive and bounded away from zero on $I$. But then the existence of a smooth extension of $b\cos \frac{c}{2}$ over $I$ would preclude $b$ from changing sign. We thus conclude that $b$ has a smooth extension. Moreover, over the interval $(\frac{\pi}{2}-\delta, \frac{\pi}{2} + \delta)$ the function $\cos \frac{\cdot}{2}$ has a smooth inverse, which allows us to conclude that $c$ has a smooth extension to $I$ as well. 

Therefore, we conclude that $K$ is a smooth curve in $B_{r+\delta}\setminus B_{r-\delta}$. But from the discussion above we also know that $K$ intersects $\partial B_\rho$ transversally for a.e. $\rho \in (r-\delta, r+\delta)$. This means that $K$ is a straight segment, or in other words $c$ is constant. We can now integrate by parts to conclude that 
\[
a' (r) = \frac{d}{dr} \frac{1}{2\pi r} \int_{\partial B_r} u = \frac{1}{2 \pi r} \int_{\partial B_r} \frac{\partial u}{\partial n} = 0 \, ,
\]
so that $a$ is constant over $I$.
Moreover we can use the fact that 
\[
\frac{1}{r} \int_{B_r} |\nabla u|^2 = \int_{\partial B_r} \left(\frac{\partial u}{\partial \tau}\right)^2
\]
is constant in $r$ to conclude that $b(r) = \beta r^{1/2}$ for some nonzero $\beta$. 

Summarizing, for every $r\in \mathcal{R}$ we have concluded that there is an interval $I$ containing it over which $a$ and $c$ are constant and $b$ takes the form $\beta r^{1/2}$ for some constant $\beta$. Without loss of generality assume $a=c=0$. Let $I$ be a maximal such interval, which we denote by $(s,t)$, and assume that $s >0$. Then the trace of $u$ on $\partial B_s$ on the exterior of the disk $B_s$ is of the form $ \beta s^{1/2} \cos \frac{\theta}{2}$. We now claim that $K\cap \partial B_s$ must consist of the single point $(s, 0)$ (in cartesian coordinates). Indeed, if $K\cap \partial B_s$ contains some other point, then there is a sequence of radii $\{r_k\} \subset \mathcal{R}$ with $r_k \to s$ and $c(r_k) \to \phi\neq 0$, because we know that $R$ is dense in $K$. But then it would follow from our formulas that $u|_{\partial B_{r_k}}$ converges to a function of the form $a' + b' \cos \frac{\theta-\phi}{2}$, which disagrees with $\beta s^{1/2} \cos \frac{\theta}{2}$ on the whole circle $\partial B_s$ minus a discrete set of points. This would only be possible if $\partial B_s \subset K$, which however is excluded from the fact that $\mathcal{H}^1 (K\setminus R)=0$.

Now, if $K\cap \partial B_s$ consists only of the point $(s, 0)$, then it turns out that $u$ is smooth in a neighborhood of $\partial B_s\setminus K$. We already know that the trace from the exterior of the disk $B_s$ must be $\beta s^{1/2} \cos \frac{\theta}{2}$, in particular we can conclude that $s\in \mathcal{R}$. But then we can iterate the argument above and show that in fact there is an interval $(s-\delta, s+\delta)$ over which $a$ and $c$ are constant and $b$ takes the form $\beta s^{1/2}$. This then shows that the same is true on the interval $(s-\delta, t)$, thereby contradicting the maximality of $(s, t)$.

The conclusion is that in fact $s$ must be $0$. Likewise, an entirely analogous argument shows $t= r_0$. We thus have conclude that over $B_{r_0}$ the discontinuity set $K$ is a radius and $u$ is given by the formula of the statement of the proposition.

\chapter{Pure jumps and triple junctions}\label{ch:salti_e_tripunti}

This chapter is devoted to proving the cases (b) and (c) of Theorem~\ref{t:main}.

{We first state, in Section \ref{s:epsilon reg} the two key $\varepsilon$-regularity theorems proved by David in his pioneering work \cite{david1996} (see also \cite{DavidBook}). The only difference with the statements in Theorem~\ref{t:main} is that we will make the additional assumption that the Dirichlet energy is also small, while the statements in Theorem~\ref{t:main} assume only the smallness of the Hausdorff distance to the model cases. We will show that it is rather straightforward to remove the smallness of the Dirichlet energy with a ``blow-up'' argument, thanks to Theorem~\ref{t:class-global}.

The sections~\ref{s:salto_preliminari}, \ref{s:approssimazione}, \ref{s:lemmi-decadimento}, and \ref{s:salto-conclusione} are then devoted to prove the $\varepsilon$-regularity at pure jumps, essentially following the approach by \cite{AFP97} (see also \cite[Chapter~8]{AFP00}). More precisely, the argument is based on a suitable decay lemma which is very close to the pioneering lemma of De~Giorgi in the regularity theory of area-minimizing hypersurfaces, but it is conceptually more complicated because we will need to juggle two quantities. One, cf. \eqref{e:mean flatness}, measures, in a scaling-invariant fashion, the $L^2$-closeness of the jump set $K$ to a flat line in a disk of radius $r$. The other, cf. \eqref{e:Dirichlet-scaled}, is the natural scaling-invariant Dirichlet energy in a disk of radius $r$. The decaying quantity is the maximum of the two, cf. Proposition \ref{p:salto_decay} (in reality, the actual decaying quantity, defined in \eqref{e:quantita-m}, takes also into account the fidelity term when $\lambda >0$). The key strategy is to then split the decay in two cases: the flatness of $K$ decays if it is at least comparable to the Dirichlet energy and likewise the Dirichlet energy decays if it is at least comparable to the flatness of $K$, cf. Lemmas \ref{l:decay1} and \ref{l:decay2}.\index{flatness@flatness}

A pivotal technical tool to prove the decay Lemmas is Proposition \ref{p:lip-approx} which shows how, when the Dirichlet energy and the flatness are both small, the set $K$ coincides, up to a small error, with the graph of a Lipschitz function. Section \ref{s:lemmi-decadimento} will exploit the Lipschitz approximation to prove the decay lemmas. When the flatness is comparable to the Dirichlet energy, the Lipschitz approximation will be shown to be very close to an affine map. This is a manifestation that $K$ is close to minimize the length: it corresponds to the one-dimensional case of De Giorgi's harmonic approximation in the case of area-minimizing hypersurfaces. When the Dirichlet energy is comparable to the flatness, we will instead show that $u$ is close to an harmonic function on a half disk, satisfying a Neumann boundary condition on the flat part of boundary. This is a manifestation of the fact that $u$ minimizes the Dirichlet energy and satisfies a Neumann boundary condition at $K$. An important part of the proof is to estimates the size of the ``possible holes of $K$.'' 

The final sections~\ref{s:tripunto-inizio} and \ref{s:triple-junction-final} are devoted to prove $\varepsilon$-regularity at triple junctions by taking advantage of the case of jump points and of the monotonicity formula contained in Proposition~\ref{p:monotonia-1}. This argument is new and quite different to David's original one. Section~\ref{s:tripunto-inizio} uses a ``blow-down'' argument and leverages to $\varepsilon$-regularity theorem for pure jumps to show that, once the set $K$ is sufficiently close to a triple junction in some disk $B_{2r} (x_0)$, then it stays close to a triple junction in every smaller disk centered at some point $y$ not too far from $x$. In particular the argument implies that in $B_r (y)$ the set $K$ is the union of three arcs joining at $y$ (but it does not imply that each arc is regular {\em up to and including} the extremum $y$). In Section~\ref{s:triple-junction-final} the regularity up to $y$ we will then be concluding thanks to the monotonicity formula of 
Proposition~\ref{p:monotonia-1}}

\section{Epsilon-regularity statements}\label{s:epsilon reg}

We start off by stating the case in which $K$ is close in Hausdorff distance to a line. First, we introduce some useful notation. We let:
\begin{itemize}
    \item $\mathcal{R}_\theta$  be the counterclockwise rotation of angle $\theta\in[0,2\pi]$ in $\mathbb R^2$. \item $\ello$ be the infinite line $\{(t, 0) : t\in \mathbb R\}$.
\end{itemize}
\index[simb]{aalV_0@$\ello$}\index[simb]{aalR_t@$\mathcal{R}_\theta$}\index[simb]{aagZ^j(t,x,r)@$\Omega^j (\theta,x,r)$}
Hence we combine in a single quantity the measure of the Hausdorff distance, $\dist_H$,\index[simb]{aalHdist@$\dist_H$}\index{Hausdorff distance@Hausdorff distance} of $K$ from a diameter 
and of the smallness of the Dirichlet energy
\begin{align}
\Omega^j (\theta,x,r) &:= r^{-1}\dist_H (K \cap \overline{B}_{2r} (x), (x+\mathcal{R}_\theta (\ello))\cap \overline{B}_{2r} (x)) + r^{-1} \int_{B_{2r} (x)\setminus K} |\nabla u |^2 \, .
\end{align}
Moreover, we will use the notation ${\rm gr}\, (f)$ for the graph of a given function $f$. \index[simb]{aalgr(f)@${\rm gr}\, (f)$@}
Finally, to make our statements less cumbersome, we agree to use ``minimizer'' whenever the statement holds for absolute, restricted, generalized, and generalized restricted minimizers (i.e. for all type of minimizers considered in these notes).

\begin{theorem}\label{t:eps_salto_puro}\label{T:EPS_SALTO_PURO}
There are constants $\varepsilon, \alpha, C > 0$ with the following property. Assume
\begin{itemize}
\item[(i)] $(u,K)$ is a minimizer of $E_\lambda$ on $B_{2r} (x)$ with $r\leq 1$;
\item[(ii)] There is $\theta \in [0, 2\pi]$ such that 
\[
\Omega^j (\theta,x,r) 
+\lambda \|g\|_\infty^2 r^{\frac12} < \varepsilon\, .
\]
\end{itemize}
Then $K\cap B_r (x)$ is the graph of a $C^{1,\alpha}$ function $f$. More precisely there is $f:[-r,r] \to \mathbb R$ such that $K\cap B_r (x) =
(x+\mathcal{R}_\theta ({\rm gr}\, (f)))\cap B_r (x)$ and 
\begin{equation}\label{e:estimate_jump Eg}
\|f\|_{0} + r \|f'\|_0 + r^{1+\alpha} [f']_\alpha \leq 
C r \big(\Omega^j (\theta,x,r)+\lambda \|g\|_\infty^2 r^{\frac12}\big)^{\frac{1}{2}}\, .
\end{equation}
\end{theorem}
The proof we will give shows that a weaker assumption suffices 
for Theorem~\ref{t:eps_salto_puro}. Indeed, the Hausdorff distance in the definition of $\Omega^j$ can be substituted by an $L^2$, one-sided analogue (cf. the definition of the mean flatness $\beta$ in \eqref{e:mean flatness} afterwards).\index{flatness@flatness}

A higher dimensional version of the previous result has been established contemporarily and independently 
by Ambrosio, Fusco and Pallara \cite{AP97,AFP97} (see also \cite[Chapter 8]{AFP00}). The proof 
we provide is inspired by that in \cite[Chapter 8]{AFP00}, despite the several shortcuts we can take due 
to the $2$d setting.

Next, we state the case in which $K$ is close to a propeller. To that end we introduce further the notation\index[simb]{aalT_0@$\To$}\index[simb]{aagZ^t(t,x,r)@$\Omega^t(\theta,x,r)$}\index[simb]{aalV_0^+@$\ello^+$}
\begin{itemize}
    \item $\ello^+$ for the halfline $\{(t,0): t\geq 0\}$;
    \item $\To$ for the global triple junction 
    \begin{equation}
        \To:= \ello^+ \cup \mathcal{R}_{\frac{2\pi}3} (\ello^+) \cup \mathcal{R}_{\frac{4\pi}3} (\ello^+)\, ;\label{e:To}
\end{equation}
\item $\Omega^t (\theta, x,r)$ for the analog of $\Omega^j (\theta, x,r)$:
\begin{equation}
\Omega^t (\theta,x,r) := r^{-1}\dist_H (K \cap \overline{B}_{2r} (x), (x+\mathcal{R}_\theta (\To))\cap \overline{B}_{2r} (x)) + r^{-1}\int_{B_{2r} (x)\setminus K} |\nabla u |^2\, .
\end{equation}
\end{itemize}

\begin{theorem}\label{t:eps_tripunto}\label{T:EPS_TRIPUNTO}
There are constants $\varepsilon, \alpha, C > 0$ with the following property. Assume:
\begin{itemize}
\item[(i)] $(u,K)$ is a minimizer  of $E_\lambda$ in $B_{2r} (x)$ with $r\leq 1$;
\item[(ii)] There is $\theta\in [0, 2\pi]$ such that 
\[
\Omega^t (\theta,x,r) 
+\lambda \|g\|_\infty^2 r^{\frac12}< \varepsilon\, .
\]
\end{itemize}
Then there is a $C^{1,\alpha}$ diffeomorphism $\Phi: B_r\to B_r (x)$ such that $K\cap B_r (x) = \Phi (\To \cap B_r)$ and
\begin{align}
|\Phi (0)-x| & + r\left(\|D\Phi - \mathcal{R}_\theta\|_0+ \|D\Phi^{-1} - \mathcal{R}_{-\theta}\|_0\right)\nonumber\\
&+ r^{1+\alpha} \left([D\Phi]_\alpha + [D\Phi^{-1}]_\alpha\right) \leq C r \big(\Omega^t (\theta,x,r)+ \lambda \|g\|_\infty^2 r^{\frac{1}{2}}\big)^{\frac12}\, .\label{e:estimate_triple}
\end{align}
\end{theorem}
In $3$d Lemenant \cite{lemenant} has proven an analogous statement provided that $K$ in $B_{2r}(x)$ is close either to the union of three half-planes meeting along their edges by $120$ degree angles 
(a $\mathbb{Y}$-cone), or to a cone over the union of the edges of a regular tetrahedron (a $\mathbb{T}$-cone).

We finally record an interesting corollary, which implies that, under the assumptions (i)-(ii) above, if $\lambda=0$, then the three arcs forming $K$ are in fact $C^2$ up to the triple junction and their respective curvatures vanish there.

\begin{corollary}\label{c:curvatura=0 nel tripunto}
Let $(u,K)$ be a minimizer of $E_0$ satisfying the assumptions of Theorem~\ref{t:eps_tripunto}. Then the three arcs forming $K\cap B_r(x)$ are $C^2$ up to the junction point $\bar x = \Phi (0)$ included and moreover $\kappa_i(\bar x)=0$ for 
every $i\in\{1,2,3\}$.
\end{corollary}

\subsection{Useful corollaries of the epsilon-regularity at pure jumps and triple junctions} We observe here that, as it is standard for $\varepsilon$-regularity statements as Theorem~\ref{t:eps_salto_puro} in geometric analysis, we can infer from these theorems a number of consequences which will be useful in the sequel, namely:
\begin{itemize}
    \item an improved convergence result of the jump sets when the limit is itself smooth; 
    \item a partial regularity result, which will
be extensively used in the rest of the notes;
\item a rigidity at ``infinity'' of certain global minimizers.
\end{itemize}

We start with the improved convergence.

\begin{corollary}\label{c:salto-puro-curvo}
Let $V\subset\subset U$ be two open planar domains and let:
\begin{itemize}
    \item[(a)] $K\subset U$ be a set which is the union of finitely many nonintersecting $C^1$ arcs (with endpoints in $\partial U$) and finitely many $C^1$ simple closed curves, all pairwise disjoint;
    \item[(b)] $u: U\setminus K \to \mathbb R$ be a $C^1$ function with the property that, for every connected component $A$ of $U\setminus K$, $u|_{A}$ has a $C^1$ extension to $\overline{A}$.
\end{itemize}
Then, for every $\delta>0$ there is a $\varepsilon (u, K, V, U, \delta)>0$ with the following property. If $(v,J)$ is a minimizer of $E_\lambda$ and
\begin{equation}
\dist_H (J\cap U, K\cap U) + \int_{U\setminus (J\cup K)} |\nabla v - \nabla u|^2 + \lambda {\rm diam}\, (U) \|g\|_\infty^2 < \varepsilon\, ,    \end{equation}
then $J\cap V$ is $C^{1,\alpha}$ close to $K\cap V$, where $\alpha$ is the constant of Theorem~\ref{t:eps_salto_puro}.
\end{corollary}

The proof follows from a simple application of Theorem~\ref{t:main} (and Theorem~\ref{t:main-generalized}) for the case of a pure jump and a standard covering argument. An analogous statement, which we leave to the reader, holds if $K$ is allowed to have a finite number of triple junctions.

\medskip

We next state the partial regularity result.

\begin{corollary}\label{c:a.e.regularity}
Assume $(u,K)$ is a minimizer of $E_\lambda$ in some open domain $U$. Then the subset of $K$ of pure jump points is relatively open and has full $\mathcal{H}^1$ measure.
\end{corollary}

\begin{proof}
The openness follows immediately from Corollary~\ref{c:salto-puro-curvo}. Observe moreover that, since $K$ is rectifiable, at $\mathcal{H}^1$-a.e. point $p\in K$ the sets 
\[
K_{p,r} := \frac{K-p}{r}
\]
converge to the approximate tangent line $\ell_p$ to $K$ at $p$. Although such convergence is just in a measure theoretic sense, the compactness Theorem~\ref{t:minimizers compactness} upgrades it to local Hausdorff convergence. By the case (b) of Theorem~\ref{t:main} every such $p$ is a pure jump point and a point of regularity for $K$.
\end{proof}
Actually, it was established by David \cite{david1996} (see also \cite[Theorem 51.20]{DavidBook}) that the Hausdorff dimension of the complement in $K$ of pure jump points is strictly less than $1$. We will give a different proof of this fact in Corollary~\ref{c:semplice-2} as a consequence of a higher integrability result for the gradient.

We close this section with the following rigidity theorem.

\begin{corollary}\label{c:unicita-blow-down}
Assume that $(u, K, \{p_{kl}\})$ is a global generalized (or generalized restricted) minimizer and that, for some sequence of radii $r_j\uparrow \infty$, a subsequence of rescalings $K_{0,r_j}$ converge to a pure jump or to a triple junction. Then $(u,K)$ itself is, respectively, a pure jump or a triple junction.
\end{corollary}

\begin{proof}
By the $\varepsilon$-regularity theory, we conclude that in each disk $B_{r_j}$ with $r_j$ sufficiently large, $K\cap B_{r_j}$ is diffeomorphic either to a straight line or to a triple junction, and in the latter case the point of junction must be contained in $B_{r_j/2}$. It is then simple to see that $K$ is connected.
 In particular, by the Bonnet's monotonicity formula Proposition~\ref{p:monotonia-0}, $\frac{1}{r} \int_{B_r} |\nabla u|^2$ is monotone nondecreasing in $r$. Since the limit of $\frac{1}{r_j} \int_{B_{r_j}} |\nabla u|^2$ is $0$, we conclude that $\nabla u$ vanishes identically. But then by Theorem~\ref{t:class-global}, $(u,K)$ itself is either a pure jump or a triple junction. 
\end{proof}

\section{Regularity at pure jumps: preliminaries}\label{s:salto_preliminari}

The proof of Theorem~\ref{t:eps_salto_puro} is based on a suitable decay proposition 
for which we need some notation and terminology. Let\index[simb]{aalL@$\mathcal{L}$}\index[simb]{aalA@$\mathcal{A}$}
\begin{align*}
\mathcal{L}&:= \{\VV \mid \VV \subset \R^2 \text{ is a linear $1$-dimensional subspace}\}; \\
\mathcal{A} &:=\left\{ z + \VV \mid z \in \R^2 \text{ and } \VV \in\mathcal{L}\right\}.
\end{align*}
Let $(u,K)$ be an admissible pair
in $B_1$. 
For all $x\in B_1$ and $0<r<1-|x|$ recall that
\begin{align}
d(x,r) & = \frac{D(x,r)}{r} =\frac{1}{r} \int_{B_r(x) \setminus K} |\nabla u|^2, \label{e:Dirichlet-scaled}
\end{align}
and define \index[simb]{aagb(x,r)@$\beta (x,r)$}\index[simb]{aalm(x,r)@$m(x,r)$}
\begin{align}
\beta(x,r) & := \min_{\VV \in \mathcal{A}}\int_{B_r(x) \cap K} \frac{\dist^2(y,\VV)}{r^2} \, \frac{d \mathcal{H}^1(y)}{r}\label{e:mean flatness}\\
m(x,r) & := \max \{d(x,r), \beta(x,r), \lambda \|g\|_\infty^2 r^{\frac{1}{2}}\}.\label{e:quantita-m}
\end{align}\index{flatness@flatness}
In what follows in case $x=0$ we shall drop the dependence on the base point from the notation introduced above. 
Observe that the following crude estimates are a simple consequence of our definitions:
\begin{align}
\beta(x,\tau r) & \leq \tau^{-3}\beta(x,r) \label{beta_scaling} \\ 
d(x,\tau r) & \leq \tau^{-1} d(x,r). \label{d_scaling}
\end{align}
The starting point to prove $\varepsilon$-regularity for $K$ is the decay of the energy under a smallness condition at a certain radius.
\begin{proposition}\label{p:salto_decay}
There are geometric constants $\varepsilon, \tau >0$ such that, if $(u,K)$ is a minimizer of $E_\param $ on $B_1$ and $0\in K$, 
then
\begin{equation}\label{e:salto_decay}
m (r) \leq \varepsilon\Rightarrow
m (\tau r) \leq \tau^{\frac{1}{2}} m (r)\, .
\end{equation}
\end{proposition}

The proof is based on two lemmas.

\begin{lemma}\label{l:decay1}
There exists $\tau_1\in (0,1)$ such that for any $\tau\leq \tau_1$ and for any $\overline{\delta}>0$ we can choose $\eta_1 =\eta_1 (\overline{\delta}, \tau)>0$ with the following property. If $(u,K)$ is a minimizer of $E_\param$ on $B_1$ and $0\in K$, then for all $r\in(0,1)$ such that
\begin{equation*} \label{decay1_eq1-g}
\overline{\delta} \max\{d(r),\param \|g\|_\infty^2 r^{\frac12}\} \leq \beta(r) \leq \eta_1 ,
\end{equation*}
we have
\begin{equation} \label{decay1_eq2-g}
\beta(\tau r) \leq \tau \beta(r).
\end{equation}
\end{lemma}

\begin{lemma}\label{l:decay2}
There exists $\tau_2\in (0,1)$  such that for any $\tau\leq \tau_2$ and for any $\overline{\delta}>0$ we can choose $\eta_2 = \eta_2 (\overline{\delta}, \tau)>0$ with the following property.
 If $(u,K)$ is a minimizer  of $E_\param$ on $B_1$ and $0\in K$, then for all $r\in(0,1)$ such that
\begin{equation*} \label{decay2_eq1-g}
\overline{\delta}\max\{\beta (r),\param \|g\|_\infty^2 r^{\frac12}\} \leq d (r) \leq \eta_2 ,
\end{equation*}
we have
\begin{equation} \label{decay2_eq2-g}
d (\tau r) \leq \tau^{\frac{1}{2}} d (r).
\end{equation}
\end{lemma}

Before coming to the proofs of the two Lemmas we show how the decay proposition can be easily concluded from them.

\begin{proof}[Proof of Proposition~\ref{p:salto_decay}] First of all fix $\tau = \min\{\tau_1, \tau_2\}$, denote by $\bar\delta$ a positive parameter smaller than $1$ to be chosen later in terms of $\tau$ and let $\varepsilon = \min\{\eta_1 (\tau, \bar\delta), \eta_2 (\tau, \bar\delta)\}$. We next claim that, for an appropriate choice of $\bar\delta\in(0,1)$, the conclusion of the proposition holds (in fact we will see below that $\bar \delta \leq \tau^{\frac{7}{2}}$ suffices). 

We distinguish four cases (actually, Case 4 below is not possible if $\lambda=0$). 
\begin{itemize}
\item [Case 1:] $\bar{\delta}\max\{d(r),\param \|g\|_\infty^2 r^{\frac12}\} \leq \beta (r)$ and 
$\bar\delta\max\{\beta(r),\param \|g\|_\infty^2 r^{\frac12}\}\leq d (r)$. 
In this case $\beta (\tau r) \leq \tau \beta (r)$ by \eqref{decay1_eq2-g} and 
$d (\tau r) \leq \tau^{\frac{1}{2}} d (r)$, by \eqref{decay2_eq2-g}, and thus 
$m (\tau r) \leq \tau^{\frac{1}{2}} m (r)$.
\item [Case 2:]  
$\bar\delta\max\{d(r),\param \|g\|_\infty^2 r^{\frac12}\}\leq \beta (r)$ and $d (r)<\bar{\delta}\max\{\beta(r),\param \|g\|_\infty^2 r^{\frac12}\}$.
Observe that we have
\begin{align*}
d(\tau r) &\stackrel{\eqref{d_scaling}}{\leq} \tau^{-1} d (r) <
\tau^{-1}\bar\delta\max\{\beta (r),\param \|g\|_\infty^2 r^{\frac12}\}\\
\beta (\tau r) &\stackrel{\eqref{decay1_eq2-g}}{\leq} \tau \beta (r)\, .
\end{align*}
Hence, provided $\bar\delta \leq \tau^{\frac{3}{2}}$, we have 
$m (\tau r) \leq \tau^{\frac{1}{2}} m (r)$.
\item [Case 3:] $\beta (r)<\bar{\delta}\max\{d(r),\param \|g\|_\infty^2 r^{\frac12}\}$ and 
$\bar\delta\max\{\beta(r),\param \|g\|_\infty^2 r^{\frac12}\}\leq d (r)$.
In this case
\begin{align*}
\beta (\tau r) &\stackrel{\eqref{beta_scaling}}{\leq} \tau^{-3} \beta (r) < \tau^{-3}\bar{\delta}\max\{d(r),r^{\frac12}\}\\
d (\tau r) &\stackrel{\eqref{decay2_eq2-g}}{\leq} \tau^{\frac{1}{2}} d (r)\, ,
\end{align*}
Hence, by choosing $\bar\delta \leq \tau^{\frac{7}{2}}$ we conclude.

\item [Case 4:] $\beta (r)<\bar{\delta}\max\{d(r),\param \|g\|_\infty^2 r^{\frac12}\}$ and 
$d (r)<\bar\delta\max\{\beta(r),\param \|g\|_\infty^2 r^{\frac12}\}$. In this last case $m(r)= \param \|g\|_\infty^2 r^{\frac12}$, being $\bar\delta<1$, and 
$\max\{\beta(r),d(r)\}\leq\bar\delta \param \|g\|_\infty^2r^{\frac12}$. Hence, we conclude that 
\begin{align*}
\beta (\tau r) &\stackrel{\eqref{beta_scaling}}{\leq} \tau^{-3} \beta (r) \leq \tau^{-3}\bar{\delta} \param \|g\|_\infty^2 r^{\frac12}\\
d(\tau r) &\stackrel{\eqref{d_scaling}}{\leq} \tau^{-1} d (r) \leq 
\tau^{-1}\bar\delta \param \|g\|_\infty^2 r^{\frac12}
\end{align*}
It suffices to choose $\bar\delta \leq \tau^{\frac{7}{2}}$ to infer the desired decay.\qedhere
\end{itemize}
\end{proof}

\section{Lipschitz approximation}\label{s:approssimazione}


A key ingredient in the proof of both decay lemmas is the Lipschitz approximation for the set $K$. Before stating it we introduce a suitable notion of  ``excess'' \index{excess@excess} which, like the quantity $\beta$, measures the flatness of the set $K$. Given $\VV\in \mathcal{L}$, we denote by $\pi_{\VV}:\mathbb{R}^2\to\VV$ and $\pi_{\VV}^{\perp}:\mathbb{R}^2\to \VV^{\perp}$ \index[simb]{aagp_V@$\pi_{\VV}$}\index[simb]{aagp_Vperp@$\pi_{\VV}^\perp$} the orthogonal projections onto $\VV$ and onto its orthogonal complement $\VV^{\perp}$, respectively. 
Given $\VV, \VV' \in \mathcal{L}$, consider the linear map $L: \mathbb R^2 \to \mathbb R^2$, $L:= \pi_{\VV} - \pi_{\VV'}$, 
and denote by $|L|$ its Hilbert-Schmidt norm.
In particular, 
\[
|L|^2=2-2 \pi_{\VV}:\pi_{\VV'}\, , 
\]
where $:$ denotes the scalar product between $2\times2$ matrices.\index[simb]{aalA:B@$A:B$}\index[simb]{aalexc(x,r)@$\exc(x,r)$}\index[simb]{aalexc(x,r)@$\exc_\VV (x,r)$}
The excess is then given for $x\in B_1$ and $0<r<1-|x|$ by
\begin{equation}\label{e:excess}
\exc(x,r):= 
\min_{\VV\in \mathcal{L}}\exc_\VV(x,r) 
\end{equation}
where for every $\VV\in \mathcal{L}$
\begin{equation}\label{e:excess V}
\exc_{\VV}(x,r):=\frac{1}{r}\int_{B_r(x)\cap K}|\VV-T_y K|^2 \, d\mathcal{H}^1(y),
\end{equation}
here $T_y K$\index[simb]{aalTK@$T_y K$} denotes the approximate tangent plane to $K$ at $y$ which exists for $\mathcal H^1$-a.e. $y\in K$ (cf. \cite[Theorem 2.83]{AFP00}).
It is also convenient to introduce a variant of the $\beta$-number: for all $\mathscr{A} \in \mathcal{A}$, \index[simb]{aagb_A(x,r)@$\beta_{\mathscr{A}}(x,r)$}
we set
\[
\beta_{\mathscr{A}}(x,r):= \int_{B_r(x) \cap K} \frac{\dist^2(y,\mathscr{A})}{r^2} \, \frac{d \mathcal{H}^1(y)}{r}\, .
\]
The basic idea behind the Lipschitz approximation is that the subset of $K$ consisting of those points $y$ where $\exc_{\VV} (y,r)$ is controlled by a (sufficiently small geometric) constant $c_0$ {\em at all scales} is contained in the graph of a Lipschitz function. This is analogous to the ``Lipschitz truncation'' of Sobolev functions $\varphi\in W^{1,1}$: the restriction of $\varphi$ to the lower level set $\{M |D\varphi| \leq \lambda\}$ of the Hardy-Littlewood maximal function of $|D\varphi|$ is Lipschitz with a constant comparable to $\lambda$, cf. \cite[Section 6.6]{EG}. At the technical level, in order to implement the latter idea we will need a suitable ``vertical separation'' lemma, which will be proved in Section \ref{s:vertical}.

However, in order to be useful in our context, we also need to show that the excess can be controlled by the flatness $\beta$. This will be accomplised in Section \ref{ss:tilt}.

\begin{proposition} \label{p:lip-approx}
There exist $C,\delta, \sigma, \alpha >0$ geometric constants with the following properties.
Assume that
\begin{itemize}
\item[(a)] $(u,K)$ is a minimizer of $E_\param$ in $B_r$;
\item[(b)] $0 \in K$ and $\ello$ is the horizontal axis;
\item[(c)] there exists $c\in\mathbb{R}$ such that setting $\bar\beta(r):=\beta_{(0,c)+\ello}(r)$, 

\[
d(r) + \bar\beta(r)+ \lambda \|g\|^2_{L^{\infty}(B_r)} r < \delta\,.
\]
\end{itemize}
Then, 
\begin{itemize}
\item[(i)] $K \cap B_{\frac r2} \subset \{|x_2| \leq Cr(\bar\beta(r))^{\alpha}\}$;
\end{itemize}
and there exists $f:[-\sigma r,\sigma r]\to\mathbb{R}$ 
Lipschitz such that 
\begin{itemize}
\item[(ii)] $\|f\|_{C^0}\leq Cr\bar\beta(r)^{\alpha}$ and $\operatorname{Lip}(f) \leq 1$; 
\item[(iii)] the following estimates hold
\[
\mathcal{H}^1\left (({\rm gr}\, (f) \triangle K)\cap 
[-\sigma r,\sigma r]^2\right) \leq Cr(d(r)+\bar\beta(r)+\param \|g\|^2_{L^\infty(B_r)} r
)
\] 
and
\begin{align*}
\int |f'|^2 &\leq \mathcal{H}^1({\rm gr}\, (f) \setminus K) 
+ C \int_{{\rm gr}\, (f) \cap K} |\ello-T_xK|^2 \, d\mathcal{H}^1(x) \\
&\leq Cr(d(r)+\bar\beta(r)+ \param\|g\|^2_{L^\infty(B_r)} r
);
\end{align*}
\item [(iv)] for any $\Lambda>0$, $\delta$ can be chosen so that for any $\varepsilon >Cr(\bar\beta(r))^{\alpha}$ (cf. item (i) above),
\[
\mathcal{H}^1([-\sigma r,\sigma r] \setminus \pi_{\ello}(K \cap \{|x_2| <\varepsilon r\}))\leq \Lambda rd(r)\, .
\]
\end{itemize}
Finally, there are constants $\bar \tau>0 , \bar\varepsilon > 0$ with the property that 
\begin{equation}\label{e: tilde K}
\tilde K := \left\{x \in K \cap B_{\bar \tau r} \mid \sup_{0 < \rho < \frac r4} \left( d(x,\rho)
+ \exc_{\ello}(x,\rho)\right)<\bar \varepsilon\right\}\subseteq {\rm gr}\,(f).
\end{equation}
\end{proposition}

\subsection{Tilt Lemma}\label{ss:tilt} The first basic tool to prove Proposition~\ref{p:lip-approx} is the following ``tilt lemma'', heavily inspired by a similar result in the theory of minimal surfaces due to Allard, cf. the seminal paper \cite{Allard}. In turn the lemma can be thought as a ``geometric counterpart'' of the classical Caccioppoli inequality for solutions of elliptic partial differential equations, in which the $L^2$ norm of the derivative of the solution in a given ball is controlled with the $L^2$ norm of the solution in a slightly larger ball. As in Allard's tilt lemma, our proof is based in plugging in the ``internal'' first variation a suitable test vector field.

\begin{lemma} \label{l:tilt}
There exists a universal constant $C>0$ such that, if $(u,K)$ is a minimizer of $E_\param$ in $B_1$ and $0\in K$, 
\begin{equation}\label{e:exc-nel-tilt-lemma}
\textstyle{\exc\left(\frac{r}{4}\right) \leq C\left(d(r) + \beta(r)
+\param \|g\|^2_{L^\infty(B_r)} r
\right)} \qquad \mbox{for all $r\leq 1$}.
\end{equation}
More precisely, if $\mathscr{A}=(0,c)+\VV_0$ for some $c\in\mathbb{R}$, we have the more accurate estimate
\begin{equation}\label{e:tilt-g}
{\textstyle{\exc_{\VV_0} \left(\frac{r}{4}\right)}}
\leq C \left(d(r) + \beta_{\mathscr{A}} (r)+\param \|g\|^2_{L^\infty(B_r)} r
\right).
\end{equation}
\end{lemma}

\begin{proof}
We start off noting that \eqref{e:tilt-g} implies \eqref{e:exc-nel-tilt-lemma}. 
By rotating we can assume that $\mathscr{A}=(0,c)+\ello$, for some constant 
$c\in \R$. 
By the density upper bound \eqref{e:upper bound}  we have $\exc_{\ello} (\frac r4)\leq \frac{16}r\mathcal{H}^1 (K \cap B_{\frac r4}) \leq 8\pi + \lambda \pi \|g\|_\infty^2 r$. 
Let $\gamma>0$ be a fixed a parameter which will be chosen appropriately later. If $\beta_{\mathscr{A}} (r) \geq \gamma$, then $\exc_{\ello} (\frac{r}{4}) \leq \frac{8\pi}{\gamma} \beta_{\mathscr{A}} (r) + \lambda \pi \|g\|_\infty^2 r$, and thus \eqref{e:tilt-g}
follows in this case. Therefore, we may additionally suppose that 
\begin{equation}\label{e:no-restriction}
\beta_{\mathscr{A}}(r) \leq \gamma.
\end{equation}
Next, we claim that
\begin{equation}\label{e:c estimate}
|c| \leq 2 \left(\frac{\beta_{\mathscr{A}}(r)}{\epsilon}\right)^{\frac 13}r. 
\end{equation}
Indeed, first observe that $|c|<\frac r2$ for $\gamma$ small enough. 
Otherwise, being $0 \in K$, the density lower bound inequality in \eqref{e:dlbEg} would imply
\[
\frac\epsilon{64}\leq\frac\epsilon4\left(\frac{|c|}{2r}\right)^2\leq
\frac1{r^3} \int_{B_{\frac{r}{4}}\cap K} 
|x_2-c|^2 \, d\mathcal{H}^1(x)\leq \beta_{\mathscr{A}}(r),
\] 
where in the second inequality we have used that $|x_2-c|\geq\frac{|c|}{2}$ for all $x \in B_{\frac{r}{4}}$.
By \eqref{e:no-restriction} this would give a contradiction for 
$\gamma$ sufficiently small.
Then, as $|c|<\frac r2$, arguing similarly we deduce that
\[
\epsilon\left(\frac{|c|}{2r}\right)^3 \leq 
\frac1{r^3} \int_{B_{\frac{|c|}{2}}\cap K} 
|x_2-c|^2 \, d\mathcal{H}^1(x)\leq \beta_{\mathscr{A}}(r),
\] 
as $|x_2-c|\geq\frac{|c|}{2}$ for all $x \in B_{\frac{|c|}{2}}$.

We next use \eqref{e:no-restriction} and \eqref{e:c estimate} 
to choose $\gamma>0$ small enough to have 
$B_{\frac r4} \subset B_{\frac r3}((0,c)) \subset B_{\frac{2r}3}((0,c)) \subset B_r$. Hence, we conclude that it is enough to prove
\begin{align*}
& \frac 1{r}\int_{B_{\frac r3}((0,c))\cap K} |\ello - T_xK|^2 \, d\mathcal{H}^1 \\
\leq & C \Big(d\left(\textstyle{\frac{2r}3}\right)+ 
\frac 1{r^3}\int_{ B_{\frac{2r}3}((0,c))\cap K}|x_2-c|^2 d\mathcal{H}^1
+ \lambda \|g\|^2_{L^\infty(B_{\frac{2r}3}((0,c)))}r\Big),
\end{align*}
which by translating is implied by
\begin{equation} \label{e:excess-bound}
\frac 1{r}\int_{B_{\frac r3}\cap K} |\ello - T_xK|^2 \, d\mathcal{H}^1 
\leq C \Big( d(r) + \frac 1{r^3}\int_{B_{\frac{2r}{3}}\cap K} |x_2|^2 \, d\mathcal{H}^1 + \lambda \|g\|^2_{L^\infty(B_{\frac{2r}{3}})}r \Big).
\end{equation}
Let $e \colon K \to \mathbb{S}^1$ be a tangent vector field to $K$ with $e(x)=(e_1(x),e_2(x))$. Then
\begin{align*}
|\ello - T_xK|^2& = 
\left|
\begin{pmatrix}
1 & 0\\
0&0
\end{pmatrix}
-
\begin{pmatrix}
e_1^2 & e_1e_2\\
e_1e_2& e_2^2
\end{pmatrix}
\right|^2 
= 2-2e_1^2=2e_2^2\,.
\end{align*}
In particular, \eqref{e:excess-bound} is equivalent to
\begin{equation} \label{e:excess-bound2}
\frac 1{r}\int_{K \cap B_{\frac r3}}e_2^2(x) \, d\mathcal{H}^1 \leq 
C \left( d(r) + \frac 1{r^3}\int_{K \cap B_{\frac{2r}{3}}} |x_2|^2 \, d\mathcal{H}^1
+ \lambda r \|g\|^2_{L^\infty(B_{\frac{2r}{3}})}
\right).
\end{equation}
Let $\eta(x)=\varphi^2(x)(0,x_2)$, $\varphi \in C_c^{\infty}(B_{\frac{2r}{3}},[0,1])$ 
with $\varphi \equiv 1$ on $B_{\frac r 3}$ and 
$\|\nabla\varphi\|_{L^\infty(B_{\frac{2r}{3}})}\leq \frac6r$. We have 
\[
D\eta(x)=2\varphi(x)(0,x_2) \otimes \nabla \varphi(x) +\varphi^2(x) \begin{pmatrix}
0 & 0\\
0 & 1
\end{pmatrix},
\]
and thus
\begin{align*}
\varphi^2 e_2^2 &\leq e^T \, D\eta \, e + 2 \varphi |\nabla \varphi| |x_2||e_2|. 
\end{align*}
The internal variation formula \eqref{e:inner} for critical points of $E_\lambda$, namely
\begin{align*} 
\int_{B_r\cap K} &e^T \cdot D\eta\,  e \, d\mathcal{H}^1=
-\int_{B_r \setminus K} \big(|\nabla u|^2 \operatorname{div} \eta 
+ 2\nabla u^T\cdot D\eta\, \nabla u\big)\\
&+2\param \int_{B_r \setminus K} (u-g)\eta\cdot\nabla u
+\param\int_{B_r\cap K} \big(|u^+-g_K|^2-|u^--g_K|^2\big)\eta\cdot\nu\, d\mathcal{H}^1
\end{align*}
yields 
\begin{equation*}
\int_{K\cap B_r} \varphi^2 e_2^2 \, d\mathcal{H}^1\leq 
2\int_{K\cap B_r} \varphi |\nabla \varphi| |x_2| |e_2|\,d\mathcal{H}^1 
+ C \int_{B_{\frac{2r}{3}} \setminus K} |\nabla u|^2 +C \lambda \|g\|^2_{L^\infty (B_{\frac{2r}{3}})} r^2\, ,
\end{equation*}
where we have used that $|D\eta|\leq C$ and that $\|u\|_{L^\infty(B_{\frac{2}{3}r})}\leq\|g\|_{L^\infty(B_{\frac{2}{3}r})}$. So we conclude 
\[
\int_{K\cap B_r} \varphi^2 e_2^2 \,d\mathcal{H}^1 \leq 
C\int_{K\cap B_r} |\nabla \varphi|^2 |x_2|^2\,d\mathcal{H}^1 
+ C\left(rd\left(\frac{2r}{3}\right )+\lambda r^2 \|g\|^2_{L^\infty(B_{\frac{2r}{3}})}
\right).
\]
Since $\varphi \in C_c^{\infty}(B_{\frac{2r}{3}},[0,1])$ with $\varphi \equiv 1$ 
on $B_{\frac r 3}$ and $\|\nabla\varphi\|_{L^\infty(B_{\frac{2r}{3}})}\leq \frac6r$, 
\eqref{e:excess-bound2} follows at once.
\end{proof}

\subsection{Vertical separation}\label{s:vertical} The second main ingredient to prove Proposition~\ref{p:lip-approx} is the following ``vertical separation lemma''.

\begin{lemma} \label{l:vertical-sep}
There exist $\varepsilon, \tau >0$ such that:
\begin{itemize}
    \item[(a)] if $(u,K)$ is a minimizer of $E_\param$ in $B_1$;
    \item[(b)] $z^1, z^2 \in K$ and $z^1\in B_{\frac12}$;
    \item[(c)] $|z^1-z^2|< \frac{\tau}{2}$ and
\begin{equation}
\sup_{\frac{|z^1-z^2|}2 < \rho < \frac{|z^1-z^2|}{\tau}} 
\left(d(z^1,\rho)+\exc_{\ello}(z^1,\rho)\right)<\varepsilon\, .
\end{equation}
\end{itemize}
Then $|z_2^1-z_2^2| \leq| z_1^1-z_1^2|$.
\end{lemma}

\begin{proof}
We first consider the case of absolute and generalized absolute minimizers.
For the sake of contradiction, assume the conclusion does not hold: in particular set $\varepsilon = \frac{1}{j}$ and $\tau = \frac{1}{j}$ and let $\{(u_j,K_j)\}$ and $z^1_j, z^2_j\in K_j$ be a sequence of minimizers and pairs of points which contradict the statement. Rescale the minimizers and translate so that $|z^1_j-z^2_j|=1$ and $z_j^1=0$. 
Thanks to Theorem~\ref{t:minimizers compactness}, up to subsequences assume that $(u_j, K_j)$ converges to a generalized global minimizer $(\bar u, \bar K)$ and that $z^i_j$ converges to some $z^i$, with $z^1=0$ and $|z^2|=1$. From the hypothesis we know that $\int_{B_r} |\nabla \bar u|^2=0$ for every $r$. By Theorem~\ref{t:class-global}, the pair $(\bar u, \bar K)$ can only be a constant, a pure jump, or a triple junction. We know $z^1, z^2 \in K$, so $K$ is non-empty and $u$ cannot be a constant. Consider now that, by \eqref{e:varifold-convergence} in Theorem~\ref{t:minimizers compactness}, the set $\bar K$ has tangent $\ello$ at $\mathcal{H}^1$-a.e. point. In particular, as $\exc_{\ello}(0,\rho)=0$
for all $\rho\in(\frac12,+\infty)$ (and then for all $\rho>0$), $\bar K$ cannot be the triple junction. So $(u,K)$ must be a pure jump with a horizontal discontinuity. But from the contradiction assumption we also deduce that $|z_2^2| \geq |z_1^2|$. Since $z^1=0\neq z^2$ and $|z^2|=1$, the two points cannot belong to the same horizontal line. This gives a contradiction.

In the case of restricted and generalized restricted minimizers, note that we can again make a blow-up argument where we know that $K_j$ is converging to an unbounded closed connected set $\bar K$ of locally finite length in $\mathbb R^2$. As above we conclude that the tangent to $\bar K$ is $\ello$ $\mathcal{H}^1$-a.e., in particular we conclude that $\bar K$ is either a halfline or a line, both contained in $\ello$. At any rate this implies that $z^1$ and $z^2$ both belong to the horizontal line $\ello$, which is the same contradiction reached above. 
\end{proof}

\subsection{Proof of Proposition~\ref{p:lip-approx}}\label{s:proof-Lipschitz}
We may argue as in Lemma~\ref{l:tilt} (cf. \eqref{e:c estimate}) to get $|c| \leq 2 \left(\frac{\bar\beta(r)}{\epsilon}\right)^{\frac 13}r$. In fact the same argument implies that
\begin{equation*}
K\cap B_{\frac r2} \subset \{x: |x_2-c|\leq 2 \left(\frac{\bar\beta(r)}{\epsilon}\right)^{\frac 13}r\}\,.
\end{equation*}
To this aim, 
assume that on the contrary there is $z\in K\cap B_{\frac r2}$ such that 
$|z_2-c|>2 \left(\frac{\bar\beta(r)}{\epsilon}\right)^{\frac 13}r$.
Set $\rho := \left(\frac{\bar\beta(r)}{\epsilon}\right)^{\frac 13}\frac r2$, if $\delta<\epsilon$ then $B_\rho(z)\subset B_r$ and we reach a contradiction as:
\begin{align*}
\frac98\bar{\beta}(r)=\frac9{r^3}\epsilon\rho^3\leq\frac1{r^3} \int_{B_\rho(z)\cap K} 
|x_2-c|^2 \, d\mathcal{H}^1(x)\leq\bar{\beta}(r),    
\end{align*}
since $|x_2-c|\geq 3\rho$ for every $x\in B_\rho(z)$.

In particular, we deduce that
\begin{equation}\label{e:reduction bis}
K \cap B_{\frac r2} \subset \{|x_2| \leq C r(\bar\beta(r))^{\frac 13}\}.
\end{equation}

Fix $\varepsilon$ and $\tau$ as in Lemma~\ref{l:vertical-sep} and let
\[
\tilde K := \left\{x \in K \cap B_{\frac{\tau}{4}r} \mid \sup_{0 < \rho < \frac r2}\left(d(x,\rho)+\exc_{\ello}(x,\rho)\right) 
<\varepsilon\right\}.
\]
On setting $\sigma:=\frac\tau4$, using Lemma~\ref{l:vertical-sep} we can define a $1$-Lipschitz function $f:[-\sigma r,\sigma r]\to\mathbb{R}$ such that
\[
\tilde K \subset \left\{(t,f(t)) \mid |t| \leq \sigma r \right\}
={\rm gr}\,(f).
\]
In particular, by \eqref{e:reduction bis} we get conclusions 
(i) and (ii) with $\alpha=\frac13$, as well as \eqref{e: tilde K}. 

In addition, for what conclusion (iii) is concerned, being $\operatorname{Lip}(f)\leq 1$, the second estimate there follows immediately from Lemma~\ref{l:tilt} provided that the first estimate in conclusion (iii) itself is established.
To this aim, by using Besicovitch covering theorem (see for instance \cite[Theorem 2.18]{AFP00}), one can cover $(K\setminus\tilde K)\cap B_{\sigma r}$ with a countable family of balls with controlled overlapping such that in each ball the defining condition of 
$\tilde K$ does not hold, so that the ensuing estimate follows easily
\begin{equation}\label{e:Besicovitch}
\mathcal{H}^1((K \setminus \tilde K)\cap B_{\sigma r}) 
\leq \frac{C}{\varepsilon} r\left({\textstyle{d \left(\frac34r\right)+
\exc_{\ello}\left(\frac34r\right)}}\right)
\leq \frac{C}{\varepsilon} r\big(d(r)+\bar\beta (r)+ \lambda r \|g\|^2_{L^\infty(B_r)}
\big)\,,
\end{equation}
where in the last inequality we have used \eqref{e:tilt-g}
in the Tilt Lemma.
Therefore, to conclude item (iii) we need to estimate $\mathcal{H}^1\left(({\rm gr}\,(f) \setminus K ) \cap \left[-\sigma r,\sigma r\right]^2\right)$. To this aim, note that 
\[
{\textstyle{\pi_{\ello}({\rm gr}\,(f) \setminus K ) \cap \left[-\sigma r,\sigma r\right] \subset \left[-\sigma r,\sigma r\right] \setminus \pi_{\ello}(\tilde K)}},
\]
so that (recalling $\operatorname{Lip}(f) \leq 1$) we get
\begin{align*}
\mathcal{H}^1&\left(({\rm gr}\,(f) \setminus K ) \cap {\textstyle{\left[-\sigma r,\sigma r\right]^2}}\right)\\
\leq & \sqrt{2}
\mathcal{H}^1\left(\pi_{\ello}({\rm gr}\,(f) \setminus K ) \cap 
{\textstyle{\left[-\sigma r,\sigma r\right]}}\right)\leq \sqrt{2}
\mathcal{H}^1\left( \textstyle{{\left[-\sigma r,\sigma r\right] \setminus \pi_{\ello}(\tilde K)}}\right)\\ 
\leq & \sqrt{2}
\mathcal{H}^1\left( \textstyle{{\left[-\sigma r,\sigma r\right] \setminus \pi_{\ello}(K)}}\right)+ \sqrt{2}
\mathcal{H}^1({\textstyle{\left[-\sigma r,\sigma r\right]}}\cap\pi_{\ello}(K \setminus \tilde K))\\ \leq
& \sqrt{2}
\mathcal{H}^1\left( \textstyle{{\left[-\sigma r,\sigma r\right] \setminus \pi_{\ello}(K)}}\right)+ \sqrt{2}
\mathcal{H}^1((K \setminus \tilde K)\cap B_{\sigma r}),
\end{align*}
In view of \eqref{e:Besicovitch} we are left with estimating the measure of the set $A:=\left[-\sigma r,\sigma r\right] \setminus \pi_{\ello}(K)$. 
In this respect, consider the vertical segments 
$\mathscr{W}_t:= \{t\} \times \left[-\frac12, \frac12\right]$, then
\begin{equation}
\mathscr{W}_t \cap K 
= \emptyset \qquad \forall t \in A.
\end{equation}
Theorem~\ref{t:class-global} and a compactness argument show that, choosing $\delta$ small enough 
in the assumption (c) of Proposition~\ref{p:lip-approx}, $(u',K')$ must be close to a pure jump $(v, J)$, where $J=\{x_2=0\}$ and $v=v^+\chi_{\{x_2>0\}} + v^- \chi_{\{x_2<0\}}$, with $|v^+ - v^-|\geq 2C_0>0$, for some universal constant $C_0$.
In particular, $|u'(t,\frac12)-u'(t,-\frac12)|\geq C_0$ and then from Jensen inequality we get 
\[
\int_{\mathscr{W}_t} |\nabla u'|^2dx_2 \geq \left(\int_{\mathscr{W}_t} |\nabla u'|dx_2\right)^2 \geq C_0^2 \qquad \forall t \in A' = \frac{A}{r}.
\]
From the latter we infer 
\[
\frac{1}{r} \mathcal{H}^1(A) =\mathcal{H}^1 \left(A'\right)\leq \frac1{C_0^2}\int_{B_1}|\nabla u'|^2 = \frac{1}{C_0^2 r} \int_{B_r} |\nabla u|^2,
\]
that concludes the proof of (iii).

To prove (iv), we argue as above with
\[
\tilde{A}:=[-\sigma r,\sigma r] \setminus 
\pi_{\ello}(K \cap \{|x_2| <\varepsilon\})
=[-\sigma r,\sigma r] \setminus \pi_{\ello}(K).
\]
Set
\[
\mathscr{W}_t = \{t\} \times \left[-\gamma, \gamma \right], \qquad  
\forall t \in \tilde{A}\,,
\]
for a fixed $\gamma = O(\lambda^{\frac12})$. 
As before, by Jensen inequality we get
\[
\frac1{2\gamma} \int_{\mathscr{W}_t} |\nabla u|^2dx_2 \geq \left(\frac1{2\gamma}\int_{\mathscr{W}_t} |\nabla u|dx_2\right)^2 \geq {\left(\frac{C_0}{2\gamma}\right)^2},
\]
and hence
\[
\mathcal{H}^1(\tilde A) \leq \frac{2\gamma}{C_0^2} \int_{B_r} 
|\nabla u|^2,
\] 
which implies (iv).
\endproof

\section{Regularity at pure jumps: decay lemmas}\label{s:lemmi-decadimento}

We are now ready to prove the two main decay lemmas.

\subsection{Proof of Lemma~\ref{l:decay1}} 
For the sake of contradiction, let $\tau_1>0$ to be chosen appropriately in what follows: then there  
are $\tau \in (0, \tau_1)$ and $\overline{\delta}>0$ such that there exist sequences $(u_j, K_j)$ of 
minimizers of $E_\lambda$, radii $r_j\in(0,1)$, and real numbers $c_j$ such that
\begin{align}
\overline{\delta}\max\{d_j, \lambda_j \|g_j\|^2_{\infty} r_j^{\frac12}\} &:={\overline{\delta}}\max\left\{\frac 1{r_j}\int_{B_{r_j} \setminus K_j}|\nabla u_j|^2, \lambda_j \|g_j\|^2_{\infty} r_j^{\frac12}\right\}\nonumber\\
&\leq\frac1{r_j^3}\int_{B_{r_j}\cap K_j} |x_2-c_j|^2\,d\mathcal{H}^1 =: \beta_j \to 0,\label{e:d-beta}
\end{align}
\begin{equation}\label{e:up-to-rotations}
\min_{\mathcal{A}}\int_{B_{r_j}\cap K_j}\dist^2(x,\mathscr{A})d\mathcal{H}^1
=\int_{B_{r_j}\cap K_j} |x_2-c_j|^2 \, d\mathcal{H}^1,
\end{equation}
and that
\begin{equation} \label{e:contradiction}
\int_{B_{\tau r_j} \cap K_j} \dist^2(x,\mathscr{A}) \, d\mathcal{H}^1 \geq \tau^4 r_j^3\beta_j 
\qquad \forall j, \ \forall \mathscr{A} \in \mathcal{A}\, 
\end{equation}
 (we can assume \eqref{e:up-to-rotations} thanks the the fact that our statement is invariant under rotations).
To apply Proposition~\ref{p:lip-approx} let $\tau_1 \leq \frac{\sigma}{2}$, $\sigma$ being defined there. 
For $j$ sufficiently large, let $f_j:[-\sigma r_j,\sigma r_j]\to\mathbb{R}$ be the $1$-Lipschitz function provided by Proposition~\ref{p:lip-approx}, and denote by $\Gamma_j$ its graph.  
In light of conclusion (ii) in Proposition~\ref{p:lip-approx}, we can assume $\Gamma_j$ to be contained inside 
the rectangle $[-\sigma r_j, \sigma r_j] \times [-\frac{\sigma}{2}r_j,\frac{\sigma}{2}r_j]$. 
For any $\eta \in C^{\infty}_c((-\sigma r_j, \sigma r_j)^2;\R^2)$ consider the corresponding inner variation \eqref{e:inner} to infer from the density upper bound in \eqref{e:upper bound} and \eqref{e:d-beta}
\begin{align*}
|\delta \Gamma_j(\eta)| &:= \left|\int_{\Gamma_j} e_j(x)^T \cdot D \eta \, e_j(x) \, d\mathcal{H}^1\right|\\ 
&\leq C\|\eta\|_{C^1} \int_{[-\sigma r_j,\sigma r_j]^2\setminus K_j} |\nabla u_j|^2+C \|\eta\|_{C^0} \lambda_j \|g_j\|_{\infty }^2 r_j\\
&\leq C\bar\delta^{-1}\|\eta\|_{C^1}\beta_jr_j
+C\bar\delta^{-1}\|\eta\|_{C^0}\beta_jr_j^{\frac12}\, ,
\end{align*}
where  $e_j$ is the unitary tangent vector field to $\Gamma_j$, and $C>0$ is a constant. 

Note that by choosing $\eta(x_1,x_2)=(0, \varphi(x_1)\psi(x_2))$, where 
$\varphi, \psi \in C^{\infty}_c((-\sigma r_j,\sigma r_j))$ 
and $\psi \equiv 1$ on $\left(-\frac{\sigma}{2}r_j,\frac{\sigma}{2}r_j\right)$, 
we can use the classical first variation formula for the length of the graph of a function to compute
\[
\delta \Gamma_j(\eta)=\int\frac{f'_j \varphi'}{\sqrt{1 + |f'_j|^2}}.
\]
In addition, by Proposition~\ref{p:lip-approx} (iii) and \eqref{e:d-beta}, we have that 
\[
\int |f_j'|^2 \leq Cr_j( d_j + \beta_j + \lambda_j \|g_j\|_{\infty }^2 r_j) 
\leq C r_j( \beta_j + \beta_j r_j^{\frac12}) \leq C\bar\delta^{-1} r_j\beta_j\,,
\] 
where $C>0$ is a constant. Thus, we set 
\[
h_j (x_1) := \frac{f_j(r_jx_1)-f_j(0)}{r_j\beta_j^{\frac12}},
\]
and conclude that, passing possibly to a subsequence, the $h_j$'s converge weakly in the Sobolev space $W^{1,2}((-\sigma,\sigma))$ to a function $h$. Note that, because of the embedding $W^{1,2} ((-\sigma, \sigma))\subset C^{1/2} ((-\sigma, \sigma))$ the convergence is uniform, and in particular we conclude that $h(0)=0$. Moreover
\[
\int h'\zeta'=0 \qquad \forall \zeta\in C^{\infty}_c ((-\sigma,\sigma)).
\]
Hence, $h(x)=ax_1$ for some constant $a \in \R$. Consider $\mathscr{V}_j:= \{(x_1,f_j(0) + \beta_j^{\frac12} a x_1) \mid x_1 \in \R \}$, then from  Proposition~\ref{p:lip-approx} (ii), recalling that $\tau<\tau_1\leq\frac\sigma2$, we conclude  
\begin{equation} \label{e:1}
\int_{B_{\tau  r_j} \cap \Gamma_j} \dist^2(x,\mathscr{V}_j) \, d\mathcal{H}^1 = o(r_j^3\beta_j).
\end{equation}
On the other hand, from Proposition~\ref{p:lip-approx} (i), (iii) and \eqref{e:d-beta} we have (assuming $\alpha$ there to be such that $2\alpha\leq 1$)
\begin{equation} \label{e:2}
\int_{B_{\tau r_j} \cap (K_j \setminus \Gamma_j)} \dist^2(x,\mathscr{V}_j)\,d\mathcal{H}^1
\leq Cr_j^2\beta_j^{2\alpha} \mathcal{H}^1(K_j\setminus \Gamma_j) \leq C\bar\delta^{-1} r_j^3\beta_j^{2\alpha+1}\,,
\end{equation}
where $C>0$ is a constant.
Hence, putting together \eqref{e:1} and \eqref{e:2}, 
\[
\int_{B_{\tau r_j} \cap K_j} \dist^2(x,\mathscr{V}_j) \, d\mathcal{H}^1 = o(r_j^3\beta_j)
\]
contradicting \eqref{e:contradiction}, and so we are done
with the proof of \eqref{decay1_eq2-g}.

\subsection{Proof of Lemma~\ref{l:decay2}}
We argue by contradiction.
Let $\tau_2>0$, to be suitably chosen afterwards: then there are $\tau \in(0,\tau_2)$, $\bar\delta>0$, 
and sequences $(u_j, K_j)$ of minimizers of $E_{\lambda_j}$, real numbers $c_j$ and radii $r_j$ such that
\begin{equation*}
d_j:= \frac1{r_j}\int_{B_{r_j} \setminus K_j} |\nabla u_j|^2 \to 0\, , 
\end{equation*}
\begin{equation} \label{e:beta-limit}
\bar\delta\max\{\beta_j, \lambda_j \|g_j\|^2_{\infty } r_j^{\frac12}\}:=\bar\delta\max
\left\{ \frac1{r_j^3}\int_{B_{r_j}\cap K_j}|x_2-c_j|^2, \lambda_j \|g_j\|^2_{\infty} r_j^{\frac12}\right\}
\leq d_j\, ,
\end{equation}
\[
\min_{\mathcal{A}}\int_{B_{r_j}\cap K_j}\dist^2(x,\mathscr{A})d\mathcal{H}^1
=\int_{B_{r_j}\cap K_j} |x_2-c_j|^2 \, d\mathcal{H}^1\, ,
\]
and
\begin{equation}\label{e:contra dj}
\int_{B_{\tau r_j} \setminus K_j} |\nabla u_j|^2 \geq \tau^{\frac{3}{2}} r_jd_j.
\end{equation}
Let $(v_j, \bar K_j)$ be defined by $\bar K_j:=r_j^{-1}K_j$, and $v_j(y):= (d_jr_j)^{-\frac12}u_j(r_jy)$. 
In view of Proposition~\ref{p:lip-approx}, 
the assumptions above and the density lower bound yield that 
$\bar K_j \cap \left[-\frac{1}{2},\frac{1}{2}\right]^2 \to \{x_2=0, |x_1| \leq \frac{1}{2}\}$ 
in Hausdorff convergence.
In addition,
\begin{equation}\label{e:Dir normalization}
\int_{B_1 \setminus \bar K_j} |\nabla v_j|^2=1,
\end{equation}
thus by compactness for harmonic functions there exist harmonic functions $v^{\pm}$ and constants $\kappa_j$ such that, 
up to subsequences,
\begin{equation} \label{e:vj convergence to v}
v_j-\kappa_j \to v^{\pm} \qquad \text{in $W^{1,2}_{{\rm loc}}(B^{\pm}_{1})$,}
\end{equation}
where $B^{\pm}_{\rho}:= B_{\rho} \cap \{\pm x_2 >0\}$, $\rho>0$. The heart of the proof is to show that
\begin{equation} \label{e:heart}
\lim_{j \to \infty} \int_{B_{\frac\sigma2} \setminus \bar K_j} |\nabla v_j|^2 =
\int_{B_{\frac\sigma2}^{+}} |\nabla v^{+}|^2 + \int_{B_{\frac\sigma2}^{-}} |\nabla v^{-}|^2
\end{equation}
($\sigma$ the constant introduced in Proposition~\ref{p:lip-approx}).
In fact, once we have \eqref{e:heart} we can easily conclude 
from \eqref{e:contra dj} and \eqref{e:Dir normalization} that
\[
\tau^{\frac32} \Big(\int_{B_{\frac\sigma2}^{+}} |\nabla v^{+}|^2 + \int_{B_{\frac\sigma2}^{-}} 
|\nabla v^{-}|^2\Big)\leq\tau^{\frac32}\leq\int_{B_{\tau}^{+}} 
|\nabla v^{+}|^2 + \int_{B_{\tau}^{-}} |\nabla v^{-}|^2.
\]
On the other hand, from the harmonicity of $v^{\pm}$ we have for $\tau<\frac\sigma2$
\[
\int_{B_{\tau}^{\pm}} |\nabla v^{\pm}|^2 \leq \frac{4\tau^2}{\sigma^2}\int_{B_{\frac\sigma2}^{\pm}}|\nabla v^{\pm}|^2,
\]
and we get a contradiction as long as $ \frac{4\tau^2}{\sigma^2} < \tau^{\frac{3}{2}}$, for instance by choosing $\tau_2=\frac{\sigma^4}{32}$.

We are now left to establish \eqref{e:heart}. By \eqref{e:vj convergence to v}, 
it is enough to prove that the Dirichlet energy of $v_j$ on $B_{\frac\sigma2}^\pm$ does not concentrate on 
$\left[-\frac{\sigma}{2},\frac{\sigma}{2}\right]\times\{0\}$. 
For the sake of contradiction, if the energy concentrates, there would exist a constant $\theta >0$ 
and a sequence $\varepsilon_j \to 0$ such that
\begin{equation} \label{e:concentration}
\int_{ \left(-\frac{\sigma}{2},\frac{\sigma}{2}\right)\times
\left(-\frac{\varepsilon_j}{2},\frac{\varepsilon_j}{2}\right)\setminus\bar K_j} |\nabla v_j|^2 
\geq \theta. 
\end{equation}
Up to replacing $\varepsilon_j$ with $\max\{\varepsilon_j, C (\bar\delta^{-1} d_j)^{\alpha}\}$, 
where $C$ is the constant in (iv) of Proposition~\ref{p:lip-approx},
we may assume $\varepsilon_j\geq C (\bar\delta^{-1} d_j)^{\alpha}$.
Thus, using \eqref{e:beta-limit}, we infer that $\varepsilon_j \geq C \beta_j^{\alpha}$.

To get a contradiction, we first note that being $(u_j, K_j)$ a minimizer of $E_{\lambda_j}$ on $B_{r_j}$, $(v_j, \bar K_j)$ minimizes on $B_1$ the functional
\begin{equation}\label{e:Fj ch3}
F_j(v,\bar K):=\int_{B_1}|\nabla v|^2+\frac1{d_j}\mathcal{H}^1(\bar K\cap B_1)+\lambda_j r_j^2\int_{B_1}|v-\bar g_j|^2
\end{equation}
where $\bar g_j(x):=(d_jr_j)^{-\frac12} g_j (r_j x)$. We next use \eqref{e:concentration} to build a competitor 
$(w_j,K_j')$ for $(v_j,\bar K_j)$ with $F_j(w_j, K_j')<F_j(v_j,\bar K_j)$ for $j$ large, which contradicts the minimality of $(v_j,\bar K_j)$. For the sake of simplicity we assume that the minimizers are absolute minimizers, the reader can anyway easily check that the competitor exhibited below is allowed also in the case of restricted and generalized minimizers.

To this aim consider the $1$-Lipschitz function $f_j:[-\sigma r_j,\sigma r_j]\to\mathbb{R}$ given by Proposition~\ref{p:lip-approx} applied to $(u_j,K_j)$. Note that by item (ii) there $\|f_j\|_{C^0}\leq Cr_j(\bar\delta^{-1}d_j)^\alpha$, and 
by item (iii) there and by \eqref{e:beta-limit} we have
\[
\int_{-\sigma r_j}^{\sigma r_j}|f_j'|^2 \leq C\bar\delta^{-1} r_jd_j,\qquad
\mathcal{H}^1(({\rm gr}\,(f_j)\triangle K_j)\triangle[-\sigma r_j,\sigma r_j]^2)\leq C\bar\delta^{-1}r_j d_j,
\]
where $C>0$ is a constant, and by item (iv) for any $\Lambda,\varepsilon>0$ for $j$ sufficiently large
\[
\mathcal{H}^1([-\sigma r_j,\sigma r_j]\setminus
\pi_{\ello}(K_j\cap\{|x_2|<\varepsilon r_j\}))\leq \Lambda r_jd_j.
\]
In particular, the function $\bar{f}_j(t):=r_j^{-1}f_j(r_jt)$, $t\in[-\sigma,\sigma]$, 
is $1$-Lipschitz with $\|\bar{f}_j\|_{C^0}\leq C(\bar\delta^{-1} d_j)^\alpha<\varepsilon_j$, and
\begin{equation}\label{e:barfj 1}
\int_{-\sigma}^{\sigma}|\bar{f}_j'|^2 \leq C\bar\delta^{-1}d_j,\qquad
\mathcal{H}^1(({\rm gr}\,(\bar f_j) \triangle \bar K_j) \triangle[-\sigma,\sigma]^2)
\leq C\bar\delta^{-1}d_j,
\end{equation}
and
\begin{equation}\label{e:barfj 2}
\mathcal{H}^1([-\sigma,\sigma]\setminus
\pi_{\ello}(\bar K_j\cap\{|x_2|<\varepsilon \}))\leq \Lambda d_j.
\end{equation}
In view of \eqref{e:Dir normalization} and \eqref{e:barfj 2}, an elementary averaging argument 
implies that we can find two points, $a_j \in [-\sigma, -\frac34 \sigma]$ 
and $b_j \in [\frac34 \sigma, \sigma]$, with the property that on the region
\[
R_j:= ([a_j-\varepsilon_j, a_j + 2\varepsilon_j]\cup [b_j-\varepsilon_j, b_j + 2\varepsilon_j])\times \left[-\frac34\sigma, \frac34\sigma\right]
\]
we have
\begin{equation}\label{e:little-energy}
\int_{R_j\setminus\bar K_j} |\nabla v_j|^2\leq C\varepsilon_j,\qquad 
\mathcal{H}^1( (\bar K_j \triangle 
{\rm gr}(\bar f_j))\cap R_j)\leq C \varepsilon_j d_j,
\end{equation}
for some positive constant $C$ depending on $\sigma$ and $\bar\delta$.
Let $h_j$ be the affine function such that 
$h_j(a_j)  = \bar f_j(a_j)$ and
$h_j(b_j)  = \bar f_j(b_j)$.
Note that $\|h_j\|_{C^0}\leq C\bar\delta_j^{-\alpha}d_j^\alpha$, and $|h_j'| \leq C d_j^{\frac12}$ by \eqref{e:barfj 1}. 
Next, for simplicity of exposition we assume that the 
graph of $h_j$ is horizontal (this is clearly always true up to a rotation). 
Let $h_j \equiv \bar h_j\in\R$ under the new choice of coordinates. In particular, $\bar h_j\to 0$ as $j\uparrow\infty$.

Set $Q_j:=[\baraj,\barbj]\times \left[-\frac\sigma2, \frac\sigma2 \right]$,
assuming $\sigma$ sufficiently small we construct a map $\Phi_j \colon B_1 \to B_1$ as follows (see Figure~\ref{figura-2} for a visual description of $\Phi_j$):
\begin{itemize}
\item $\Phi_j(x_1,x_2) = (x_1,x_2)$, if $(x_1,x_2) \in 
B_1\setminus Q_j$;
\item $\Phi_j(x_1,x_2) = (x_1, \varphi_j(x_1,x_2))$ if $(x_1,x_2) \in 
Q_j$; 
\end{itemize}
where $\varphi_j:Q_j\to Q_j$ is defined by
\begin{itemize}
\item If $\baraj + \varepsilon_j \leq x_1 \leq \barbj -\varepsilon_j$, the map $x_2 \mapsto \varphi_j(x_1,x_2)$
\begin{itemize}
\item[i)] is identically $\bar h_j$ for $|x_2-\bar h_j| \leq 2\varepsilon_j$;
\item[ii)] maps linearly the segments $\{x_1\}\times[\bar h_j + 2\varepsilon_j, \frac\sigma2]$ 
and $\{x_1\}\times[-\frac\sigma2,\bar h_j - 2\varepsilon_j]$ onto 
$\{x_1\}\times[\bar h_j,\frac\sigma2]$ and $\{x_1\}\times[-\frac\sigma2,\bar h_j]$, 
respectively.
\end{itemize}

\item If $\baraj \leq x_1 \leq \baraj + \varepsilon_j$, the map  $x_2 \mapsto \varphi_j(x_1,x_2)$
\begin{itemize}
\item[i)] is identically $\bar h_j$ if $|x_2-\bar h_j| \leq 2 (x_1 - \baraj) $;
\item[ii)] maps linearly the segments $\{x_1\}\times[\bar h_j + 2 (x_1-\baraj), \frac\sigma2]$ and 
$\{x_1\}\times[-\frac\sigma2, \bar h_j - 2 (x_1-\baraj)]$ onto $\{x_1\}\times[\bar h_j, \frac\sigma2]$ 
and $\{x_1\}\times[-\frac\sigma2,\bar h_j]$, respectively.
\end{itemize}
\item If $\barbj-\varepsilon_j  \leq x_1 \leq \barbj$, the map  $x_2 \mapsto \varphi_j(x_1,x_2)$
\begin{itemize}
\item[i)] is identically $\bar h_j$ if $|x_2-\bar h_j| \leq  2 (\barbj-x_1)$;
\item[ii)] maps linearly the segments $\{x_1\}\times[\bar h_j + 2 (\barbj-x_1), \frac\sigma2]$ and 
$\{x_1\}\times[-\frac\sigma2,\bar h_j - 2 ( \barbj-x_1)]$ onto $\{x_1\}\times[\bar h_j, \frac\sigma2]$ 
and $\{x_1\}\times[-\frac\sigma2,\bar h_j]$, respectively.
\end{itemize}
\end{itemize}
Let $S_j:=[\baraj,\barbj ] \times \{\bar h_j\}$, $\Sigma_j:=\Phi_j^{-1}(S_j)$,
$T_j:=\{(x_1,x_2)\in Q_j:\,\baraj+\varepsilon_j\leq x_1\leq\barbj-\varepsilon_j,\,2\varepsilon_j\leq|x_2-\bar{h}_j|\leq\frac\sigma2\}$,
and $U_j:=[\baraj+\varepsilon_j,\barbj-\varepsilon_j]\times[-\frac\sigma2,\frac\sigma2]$. 
Then
\begin{itemize}
\item[(a)] $\Phi_j|_{\Sigma_j}$ is the projection onto $\mathscr{V}_j:=\{(x_1,\bar h_j):x_1\in\mathbb{R}\}$
with $\Phi_j(\Sigma_j)=S_j$, 
$\Phi_j \colon B_1 \to B_1$ is Lipschitz, 
and $\Phi_j \colon B_1 \setminus \Sigma_j \to B_1 \setminus S_j$ is bi-Lipschitz; 
\item[(b)] 
$\|D\Phi_j\|_{L^\infty((B_1\setminus Q_j)\cup\Sigma_j,B_1)}=1$, 
$\|D\Phi_j\|_{L^\infty(Q_j\setminus \Sigma_j,B_1)} \leq 1+C \varepsilon_j$, and\\
$\|D(\Phi_j^{-1})\|_{L^\infty(B_1\setminus S_j, B_1\setminus \Sigma_j)}\leq 1+C\varepsilon_j$,  for some universal $C>0$, and where we are using the operator norm on $D\Phi_j$;
\item[(c)] ${\rm gr}\,(\bar f_j) \cap Q_j
\subset \Sigma_j$.
\end{itemize}

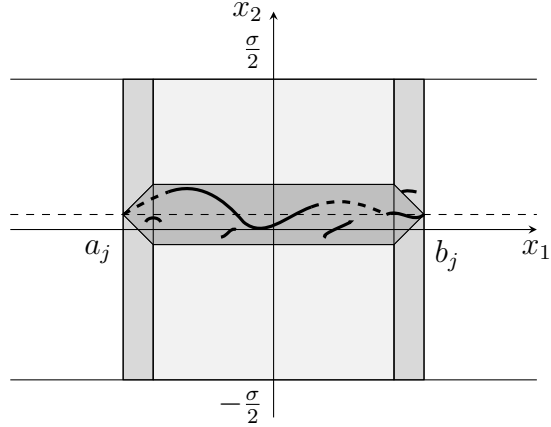
\begin{figure}
\begin{tikzpicture}
\draw[fill=gray!30] (-2,0.2) -- (-1.6,0.6) -- (-1.6,2) -- (-2,2) -- (-2,0.2);
\draw[fill=gray!30] (2,0.2) -- (2,2) -- (1.6,2) -- (1.6,0.6) -- (2,0.2);
\draw[fill=gray!10] (-1.6,0.6) -- (1.6, 0.6) -- (1.6,2) -- (-1.6,2) -- (-1.6,0.6);
\draw[fill=gray!30] (-2,0.2) -- (-1.6,-0.2) -- (-1.6,-2) -- (-2,-2) -- (-2,0.2);
\draw[fill=gray!30] (2,0.2) -- (2,-2) -- (1.6,-2) -- (1.6,-0.2) -- (2,0.2);
\draw[fill=gray!10] (-1.6,-0.2) -- (1.6, -0.2) -- (1.6,-2) -- (-1.6,-2) -- (-1.6,-0.2);
\draw[fill=gray!50] (-2,0.2) -- (-1.6,0.6)-- (1.6,0.6) -- (2,0.2) -- (1.6,-0.2) -- (-1.6,-0.2) -- (-2,0.2);
\draw[dashed, very thick] (-2,0.2) to [out=30,in=200] (-1.4,0.5);
\draw[very thick] (-1.4,0.5) to [out=20, in =130] (-0.4,0.1) to [out=320,in=200] (0.5,0.3);
\draw[dashed, very thick] (0.5,0.3) to [out=20,in=155] (1.5,0.2);
\draw[very thick] (1.5,0.2) to [out=25, in=220] (2,0.2);
\draw[very thick] (1,0.1) to [out=10, in=170] (0.7,-0.1);
\draw[very thick] (-0.7,-0.1) to [out=5,in=160] (-0.5,0);
\draw[very thick] (-1.7,0.1) to [out=70,in=110] (-1.5,0.1);
\draw[dashed] (-3.5,0.2) -- (3.5,0.2);
\draw[>=stealth,->] (-3.5,0) -- (3.5,0);
\node[below] at (3.5,0) {$x_1$};
\draw[>=stealth,->] (0,-2.5) -- (0,2.9);
\node[left] at (0,2.9) {$x_2$};
\draw (-2,-2) -- (-2,2);
\draw (2,2)-- (2,-2);
\draw (-3.5,2) -- (3.5,2);
\draw (-3.5,-2) -- (3.5,-2);
\draw (-1.6,-2) -- (-1.6,2);
\draw (1.6,-2) -- (1.6,2);
\node[above left] at (0,2) {${\textstyle{\frac{\sigma}{2}}}$};
\node[below left] at (0,-2) {${\textstyle{-\frac{\sigma}{2}}}$};
\node[below left] at (-2,0) {$\baraj$};
\node[below right] at (2,0) {$\barbj$};
\draw[very thick] (1.7,0.5) to [out=20,in=170] (1.9,0.5);
\end{tikzpicture}
\caption{A visual description of the map $\Phi_j$ in the rotated coordinates $(x_1, x_2)$. The map is the identity outside the gray zones. The diamond shaped dark gray zone is the region $\Sigma_j$: $\Phi_j$ ``squeezes'' it on the horizontal central segment $S_j$ (which lies on the dashed horizontal line). The $6$ lighter gray zones are consequently stretched by $\Phi_j$: in $T_j$, the two central lightest ones, and in the four remaining lateral zones the Lipschitz constant is controlled by $1+C\varepsilon_j$. The set $\bar K_j$ is depicted by the thick arcs, while the graph of $\bar f_j$ is thick and dashed (and has a large overlap with $\bar K_j$). While the graph of 
$\bar f_j$ must lie in the region $\Sigma_j$, there might well be portions of $\bar K_j$ which lie outside.\label{figura-2}}
\end{figure}

The competitor $(w_j, K'_j)$ to test the minimality of $(v_j, \bar K_j)$ for $F_j$ is then given by
\begin{gather*}
K_j':= (\bar K_j \setminus Q_j) \cup S_j\cup \Phi_j(\bar K_j \setminus \Sigma_j),\qquad
w_j := v_j \circ \Phi_j^{-1} \quad \text{on $B_1 \setminus K'_j$}
\end{gather*}
In particular, note that the number of connected components of $K'_j$ is less than or equal to that of $\bar K_j$.

We begin by estimating the length of $K'_j$. 
First note that $K'_j\setminus Q_j
= \bar K_j \setminus Q_j$. 
Furthermore, we have
\begin{align*}
\mathcal{H}^1(S_j)&=\mathcal{H}^1(\Phi_j(\bar K_j\cap\Sigma_j))+
\mathcal{H}^1(S_j\setminus\Phi_j(\bar K_j\cap\Sigma_j))\\ &
\stackrel{(a),(b)}{\leq} \mathcal{H}^1(\bar K_j \cap \Sigma_j) +
\mathcal{H}^1([\baraj, \barbj]\times\{\bar h_j\} \setminus \pi_{\mathscr{V}_j}(\bar K_j))
\stackrel{\eqref{e:barfj 2}}{\leq}\mathcal{H}^1(\bar K_j \cap \Sigma_j) + \Lambda d_j
\end{align*}
where we recall that $\mathscr{V}_j=\{(x_1,\bar h_j):x_1\in\mathbb{R}\}$. 
On the other hand, the very definition of $K'_j$ gives
\begin{align*}
\mathcal{H}^1&((K'_j\cap Q_j)\setminus S_j)
=\mathcal{H}^1(\Phi_j((\bar K_j\cap Q_j)\setminus \Sigma_j))
\stackrel{(b)}{\leq}(1+C\varepsilon_j)
\mathcal{H}^1((\bar K_j\cap Q_j)\setminus \Sigma_j)
\\
&\stackrel{(c)}{\leq}\mathcal{H}^1((\bar K_j\cap Q_j)\setminus \Sigma_j)+C\varepsilon_j\mathcal{H}^1((\bar K_j\cap Q_j
) \setminus {\rm gr}\,(\bar f_j))\\&
\stackrel{\eqref{e:barfj 1}}{\leq}
\mathcal{H}^1((\bar K_j\cap Q_j)\setminus \Sigma_j)+C\varepsilon_jd_j\,.
\end{align*}
Hence, we conclude that
\begin{equation} \label{e:comp-length}
\mathcal{H}^1(K'_j\cap B_1) \leq \mathcal{H}^1(\bar K_j\cap
B_1) + \Lambda d_j+o(d_j).
\end{equation}

Now, we estimate the Dirichlet energy. 
Note that $w_j = v_j$ on $B_1 \setminus 
Q_j$, and in addition that
\begin{align*}
\int_{U_j\setminus K'_j} |\nabla w_j|^2 & \stackrel{(a),(b)}{\leq} 
(1 + C \varepsilon_j)^2 \int_{
T_j\setminus \bar K_j} |\nabla v_j|^2
\stackrel{\eqref{e:concentration}}{\leq} 
(1 + C \varepsilon_j)^2
\left(\int_{Q_j\setminus \bar K_j
} |\nabla v_j|^2-\theta\right).
\end{align*}
and that 
\begin{align*}
\int_{Q_j\setminus(U_j\cup K'_j)}|\nabla w_j|^2\leq &
\|D(\Phi_j^{-1})\|^2_{L^\infty(B_1\setminus S_j,B_1\setminus \Sigma_j)}\int_{Q_j\setminus(U_j\cup \bar K_j)} |\nabla v_j|^2
\stackrel{\eqref{e:little-energy},\,(b)}{\leq} 
C\varepsilon_j\,.
\end{align*}
Therefore, we conclude by \eqref{e:Dir normalization} that 
\begin{equation} \label{e:comp-energy}
\int_{B_1 \setminus K'_j} |\nabla w_j|^2 \leq \int_{B_1 \setminus \bar K_j}
|\nabla v_j|^2 - \theta+ o(1) .
\end{equation}
By taking into account that $\|w_j\|_{L^\infty(B_1)}=\|v_j\|_{L^\infty(B_1)}
\leq\|\bar g_j\|_{L^\infty (B_1)}=
(d_jr_j)^{-\frac12}\|g_j\|_{L^\infty(B_{r_j})}$,
and that $\lambda_j \|g_j\|_{L^\infty (B_{1})}^2 r_j\leq\bar\delta^{-1}r_j^{\sfrac12}d_j$ in view of \eqref{e:beta-limit},
adding up \eqref{e:comp-length} and \eqref{e:comp-energy} we get from the very definition
of $F_j$ in \eqref{e:Fj ch3}
\[
F_j(w_j, K_j)\leq F_j(v_j,\bar K_j)+\Lambda - \theta + o (1),
\]
which, for $j$ large enough and $\Lambda<\theta$, contradicts the fact that $(\bar K_j,v_j)$ minimizes $F_j$.

\section{Regularity at pure jumps: conclusion}\label{s:salto-conclusione}

We are now ready to show how the $\varepsilon$-regularity Theorem~\ref{t:eps_salto_puro} follows from the decay Proposition~\ref{p:salto_decay}.

\begin{proof}[Proof of Theorem~\ref{t:eps_salto_puro}] 
Without loss of generality assume that $x=0$ and that $\theta=0$. First note that the very definition of $\beta$ in \eqref{e:mean flatness} gives
\begin{align*}
\beta(0,2r)&\leq (2r)^{-3}\dist_H^2 (K \cap \bar B_{2r}, \ello\cap \bar B_{2r})
\mathcal{H}^1(K\cap \bar B_{2r})\\&
\leq C(\Omega^j(0,0,r))^2(1+\lambda\|g\|^2_\infty r^2)\leq C\Omega^j(0,0,r)
\end{align*}
where $C>0$ is a universal constant, assuming $\varepsilon\in(0,1)$, and using assumption (ii) and the energy upper bound in \eqref{e:upper bound} for the last inequality. In turn, from this we deduce that  
\[
\beta(0,2r) + d(0,2r)+\lambda \|g\|_\infty^2 (2r)^{\frac12} 
\leq C(\Omega^j(0,0,r)+\lambda \|g\|_\infty^2 r^{\frac12})=:\varepsilon(r)<C\varepsilon\,.  
\]
Then from assumption (ii), \eqref{beta_scaling} and \eqref{d_scaling} we get
\[
\beta(z,r) + d(z,r)+\lambda \|g\|_\infty^2 r^{\frac12} 
< C\varepsilon(r)
\qquad \forall z\in B_{r}\, .
\]
Using Proposition~\ref{p:salto_decay}, if $\varepsilon$ is chosen sufficiently small, we get for all $k\in\mathbb{N}
$
\[
\beta(z,\tau^k r) + d(z,\tau^k r) + \lambda \|g\|_\infty^2 (\tau^k r)^{\frac12}
< 
C\varepsilon(r)\tau^{\frac k2}
\qquad \forall z\in K\cap B_{r}\, .
\]
In particular, we can easily conclude
\begin{equation}\label{e: beta d r decay}
\beta(z,\rho) +  d(z,\rho) + \lambda \|g\|_\infty^2 \rho^{\frac12}
\leq 
C\varepsilon(r)\rho^{\frac{1}{2}} 
\qquad \forall \rho < r,\, 
\forall z\in K\cap B_r 
\end{equation}
(where from now on we stop keeping track of geometric constants).
From the Tilt Lemma~\ref{l:tilt} we infer that 
\begin{equation}\label{e: exc decay}
\exc(z,\rho) \leq 
C\varepsilon(r)\rho^{\frac{1}{2}}
\qquad \forall \rho < \frac{r}{4}\, , \forall z\in K\cap B_r.
\end{equation}
For each $z$ and $\rho$ as above, let now $\VV (z, \rho)\in\mathcal{L}$ be such that
\[
\exc(z,\rho) = \exc_{\VV (z, \rho)}(z,\rho),
\]
(cf. \eqref{e:excess} and \eqref{e:excess V} for the definition of excess), 
then observe that for all $\rho<t<\frac r8$ and $t\le2\rho$, using the density lower bound in Theorem~\ref{t:dlb}, we have
\begin{align}
|\VV (z, \rho) - \VV (z, t)|^2 &\leq \frac{C}{\rho} \int_{K\cap B_{\rho} (z)} |\VV (z, \rho) - \VV (z, t)|^2 d\mathcal{H}^1 (y)\nonumber\\
&\leq C (\exc (z,\rho) + \exc (z,t)) 
\stackrel{\eqref{e: exc decay}}{\leq}
C\varepsilon(r)\rho^{\frac{1}{2}}\, . \label{e:scales}
\end{align}
Additionally, with a similar argument we infer that
\begin{align*}
|\VV (z, \rho) - \VV (y, \rho)|^2 &\leq 
C\varepsilon(r)
\rho^{\frac{1}{2}} \qquad \mbox{$\forall z,y\in K\cap B_{r}$ with $|z-y|\leq \frac{\rho}{4}$, $\rho< \frac r8$}\, .
\end{align*}
Combining both estimates with \eqref{e: exc decay}, an elementary summation argument on dyadic scales yields that
\[
|\VV (z, \rho) - \VV (0, \textstyle{\frac{r}{16}})|^2 \leq 
C\varepsilon(r)
r^{\frac12} \qquad \forall z\in K \cap B_{\frac{r}{16}}\, , \forall \rho < \textstyle{\frac{r}{8}}\, .
\]
If we rotate the coordinates so that $\VV (0, \frac{r}{16})$ is the horizontal line
$\VV_0$, the density lower bound in Theorem~\ref{t:dlb},
\eqref{e: exc decay} and the latter inequality imply that
\[
\exc_{\ello}(z,\rho)=
\frac{1}{\rho} \int_{K\cap B_\rho(z)} |T_y K - \VV_0|^2\, d\mathcal{H}^1 (y) \leq 
C\varepsilon(r)r^{\frac12}\qquad 
\forall z\in K\cap B_{\frac r{16}}\, , \forall \rho <\textstyle{\frac{r}{8}}\,.
\]
Therefore, thanks to \eqref{e: beta d r decay} and \eqref{e: tilde K} in Proposition~\ref{p:lip-approx}, we know that for a choice of $\varepsilon$ sufficiently small there is a $1$-Lipschitz function 
$f:[-\sigma r,\sigma r]\to\mathbb{R}$ such that 
$K\cap B_{\sigma r}\subseteq {\rm gr}\,(f)$, where $\sigma$ is the geometric constant in Proposition
\ref{p:lip-approx}. 
This shows the graphicality of $K$ over $\mathscr{V} (0, \frac{r}{16})$, but since the angle between $\mathscr{V} (0, \frac{r}{16})$ and $\mathscr{V}_0$ is small, we have in fact shown also graphicality over $\mathscr{V}_0$.

In addition, estimate
\eqref{e: beta d r decay} yields that 
\[
\lim_{\rho\downarrow 0} \frac{1}{\rho} \int_{B_\rho(z)} |\nabla u|^2 = 0
\]
for every $z\in K\cap B_{\frac r{16}}$. Hence, by Theorem~\ref{t:class-global} any blow-up at every such $z$ is either a pure jump or a triple junction. On the other hand, since $\beta(z,\rho)\downarrow 0$ as $\rho\downarrow 0$ (cf. \eqref{e: beta d r decay}), we infer that any blow-up is in fact a pure jump. In particular, 
we conclude that
\begin{equation}\label{e: full H1 density}
\lim_{\rho\downarrow 0} \frac{\mathcal{H}^1 (K\cap B_\rho (z))}{2\rho} = 1\, .
\end{equation}
Note that, choosing $\varepsilon$ sufficiently small, we can assume that $K\cap B_{\frac{\sigma r}4}$ is not empty (cf. item (i) in Proposition~\ref{p:lip-approx}). In particular, $\|f\|_{C^0([-\frac{\sigma r}4, \frac{\sigma r}4])}\leq \frac34 \sigma r$. Denote by $R$ the open rectangle 
\[
R = \big(-\frac{\sigma r}4,\frac{\sigma r}4\big)\times \big(-\frac34 \sigma r, \frac34 \sigma r\big)\, .
\]
We show next that $K\cap R={\rm gr}\,(f)\cap R$. 
 We know that $\pi_{\ello}(K \cap R)$ is a relatively closed set inside $I=(-\frac{\sigma r}4, \frac{\sigma r}4)$ and it is not empty. If $I \setminus \pi_{\ello} (K)$ is not empty, then it contains at least an open interval $(a,b)$ with one extremum, say $b$, which belongs to $\pi_{\ello} (K)$. Namely $y=(b,f(b)) \in K\cap R$, but the density of the set $K$ at $y$ would be strictly less than $1$, a contradiction to 
 \eqref{e: full H1 density}. 
Hence $K \cap R= {\rm gr}\,(f) \cap R$, in particular $K\cap B_{\frac{\sigma r}4} = {\rm gr}\, (f)\cap B_{\frac{\sigma r}4}$.

Finally, observe that the limit of $\VV (z, \rho)$ for $\rho\downarrow 0$ exists at every $z= (s, f(s)) \in  K\cap {\rm gr}\, (f)$ by \eqref{e:scales}, that together with \eqref{e: exc decay} yield that any $s\in I$
is a Lebesgue point of $f'$. In turn, this fact implies 
the differentiability of $f$ at every $s\in I$, with the limit being the graph of the linear map $t\mapsto f' (s) t$
(cf., for instance, the proof of Rademacher's theorem in \cite[Theorem 2.14]{AFP00}).
In particular, 
the decay of $\exc(z,\cdot)$ in \eqref{e: exc decay} can be translated into
\[
\fint_{s-\delta}^{s + \delta} |f'(t)-f'(s)|^2\,dt \leq 
C\varepsilon(r)
\delta^{\frac{1}{2}} \qquad \forall \delta \leq \sigma\rho_0-|s|,
\]
and so the classical Campanato's theorem implies that $ f \in C^{1,\frac14}(I)$ with $[f']_{\frac14} \leq 
(C\varepsilon(r))^{\frac{1}{2}}$
(cf. \cite[Theorem 7.51]{AFP00}). Estimate \eqref{e:estimate_jump Eg} follows from this and Proposition~\ref{p:lip-approx}.
This proves the claim of the theorem except that the graphicality of $K$ has been shown in $B_{\frac{\sigma r}4}$ rather than in $B_r$. Note however that we can proceed with a simple covering argument to achieve graphicality in $B_r$ (provided we choose $\varepsilon$ even smaller). 
\end{proof}


\section{Triple junctions, closeness at all scales}\label{s:tripunto-inizio}

We turn to the proof of Theorem~\ref{t:eps_tripunto}.
We start off proving the following fact: if around a point $x$ a minimizer is close to a triple junction at some scale $r$, then there is a nearby point $y$ such that the minimizer is close to a triple junction at {\em every} scale $\rho<r$ around $y$. In order to formulate our conclusion more precisely, we recall the notation introduced in Theorem~\ref{t:eps_tripunto}
\[
\Omega^t (\theta,x,r) = r^{-1}\dist_H \big(K \cap \overline{B}_{2r} (x), (x+\mathcal{R}_\theta (\mathscr{T}_0))\cap \overline{B}_{2r} (x)\big) +r^{-1} \int_{B_{2r} (x)\setminus K} |\nabla u |^2\,,
\]
where $\theta \in [0, 2\pi]$, $\mathcal{R}_\theta$ is the corresponding rotation, and $\mathscr{T}_0$ is defined in \eqref{e:To}. Furthermore we define
\begin{equation}\label{e:Omega_t}
\Omega^t (x,r) := \inf_\theta \Omega^t (\theta, x,r)\,.
\end{equation}
\index[simb]{aagZ^t(x,r)@$\Omega^t(x,r)$}

\begin{lemma}\label{l:C1smooth}
For every $\delta>0$ there is $\eta>0$ such that the following holds. Assume that 
$(u,K)$ is a minimizer of $E_\lambda$ in $B_{2r} (x)$ such that
\[
\Omega^t (x,r)+ \lambda \|g\|_\infty^2 r^{\frac 12} < \eta \, .
\]
Then there is a $y\in B_{\delta r} (x)$ such that
\begin{align*}
\Omega^t (y, \rho) + \lambda \|g\|_\infty^2 \rho^{\frac 12} \leq \delta \qquad \forall \rho < r\, .
\end{align*}
\end{lemma} 

Lemma~\ref{l:C1smooth} follows in fact easily from the following more technical statement.

\begin{lemma}\label{l:decay iteration}
There exists $\gamma_0>0$ such that for every $\gamma\in(0,\gamma_0)$ there is $\varepsilon_0 (\gamma)>0$ with the following property.
For every $\varepsilon\in(0,\varepsilon_0)$ there is $N= N (\varepsilon) \in\mathbb{N}$ such that for all $r\in (0,1]$, for all $(u,K)$ minimizer of $E_\lambda$ in $B_{2r}(x)$, and for all $(N+1)$-ple of points 
$x_0 =x, x_1, \ldots, x_N$ in $B_{2r}(x)$ such that
\begin{align*}
&\Omega^t (x_k, 2^{-k} r) + \lambda \|g\|_\infty^2 (2^{-k} r)^{\frac{1}{2}} \leq \varepsilon 
\qquad \forall k\in \{0, \ldots, N\}\\
& |x_{k+1} - x_k| \leq \gamma 2^{-k} r \qquad \forall k\in \{0, \ldots , N-1\}\, ,
\end{align*}
then there is a point $x_{N+1}\in B_{2r}(x)$ such that 
\begin{align*}
&\Omega^t (x_{N+1}, 2^{-N-1} r) + \lambda \|g\|_\infty^2 (2^{-N-1} r)^{\frac{1}{2}} \leq \varepsilon\\
&|x_{N+1}-x_N|\leq \gamma 2^{-N} r\, .
\end{align*}
\end{lemma}

\subsection{Proof of Lemma~\ref{l:C1smooth}} Without loss of generality $x=0$. Fix $\delta>0$, and additionally choose $\varepsilon$ and $\gamma$ sufficiently small, whose choice will be specified later, so that Lemma~\ref{l:decay iteration} is applicable. Let $N$ be given by Lemma~\ref{l:decay iteration} and notice that, if $\eta$ is chosen sufficiently small, the assumption of that Lemma holds with $x=x_0=x_1= \ldots = x_N =0$. We thus find $x_{N+1}$ as in the conclusion there. Observe therefore that we can apply the lemma again in $B_{\frac r2} (x_1)$, but this time the points $x_0, \ldots, x_N$ substituted by $x_1, \ldots, x_{N+1}$. Proceeding inductively we find a sequence of points $\{x_k\}$ with:
\begin{itemize}
    \item $x_0=0$;
    \item $|x_{k+1} - x_k|\leq \gamma 2^{-k} r$;
    \item $\Omega^t (x_k, 2^{-k} r) + \lambda\|g\|^2_{\infty}(2^{-k}r)^{\frac12}\leq \varepsilon$.
\end{itemize}
Since $\{x_k\}$ is a Cauchy sequence, it has a limit $y$. Observe that $|y- x_k|\leq \gamma 2^{-k+1} r$. Fix $\rho\leq r$ and choose $k$ such that $2^{-k-1} r < \rho \leq 2^{-k} r$. In particular $B_\rho (y) \subset B_{2^{-k+1} r} (x_k)$. Let $\theta_k$ be such that
\[
\Omega^t (\theta_k, x_k, 2^{-k} r) + \lambda \|g\|_\infty^2 (2^{-k}r)^{\frac12} \leq \varepsilon\, .
\]
Observe that
\begin{align*}
\rho^{-1}&\dist_H \big(K \cap \overline{B}_{2\rho} (y), (y+  \mathcal{R}_{\theta_k} (\mathscr{T}_0))\cap \overline{B}_{2\rho} (y)\big)
+ \rho^{-1}\int_{B_{2\rho} (y)} |\nabla u|^2 + \lambda \|g\|_\infty^2 \rho^{\frac12}\\
\leq & C \gamma+ C(2^{-k} r)^{-1}\dist_H \big(K\cap \overline{B}_{2^{-k+1} r}, (x_k + \mathcal{R}_{\theta_k} (\mathscr{T}_0))\cap 
\overline{B}_{2^{-k+1} r} (x_k)\big)\\
& + C(2^{-k} r)^{-1}
\int_{B_{2^{-k+1} r} (x_k)} |\nabla u|^2 +C \lambda \|g\|_\infty^2 (2^{-k}r)^{\frac12}
\leq C( \gamma + \varepsilon)\, ,
\end{align*}
where $C\geq 1$ is a geometric constant.
We choose first $C\gamma \leq \frac{\delta}{4}$. Having fixed $\gamma$, we can take $\varepsilon_0 (\gamma)$ as in Lemma~\ref{l:decay iteration} and hence impose that $C\varepsilon < \min \{\varepsilon_0 (\gamma), \frac{\delta}{4}\}$. This ensures the applicability of Lemma~\ref{l:decay iteration} in the argument above, and the inequality $C (\gamma + \varepsilon) < \delta$. 
We thus conclude
\[
\Omega^t (y, \rho) + \lambda \|g\|_\infty^2 \rho^{\frac12} \leq \Omega^t (\theta, y, \rho) +  \lambda \|g\|_\infty^2 \rho^{\frac12} < \delta \, .
\]

\subsection{Proof of Lemma~\ref{l:decay iteration} for absolute and generalized minimizers} The argument for the two cases is entirely analogous and for simplicity we focus on absolute minimizers.

Without loss of generality $x=0$.
We argue by contradiction. We assume that, for some $\gamma >0$ and $\varepsilon >0$ sufficiently small (the smallness will be specified later) and for every 
$N\in\N$ there are 
\begin{itemize}
\item[(a)] A family of numbers $\lambda_{N}\in [0,1]$;
\item[(b)] A family of fidelity functions $g_{N}$ with $\|g_{N}\|_{\infty}\leq M_0$;
\item[(c)] A family of radii $r_{N}\in (0,1]$;
\item[(d)] A family of points $x_{k,N}$, with 
$k\in\{0,\ldots,N\}$, and
\begin{align}
&x_{0,N}=0\\    
&|x_{k+1,N}-x_{k,N}|\leq \gamma 2^{-k} r_{N}\qquad \forall k\in\{0,\ldots,N-1\}\, ;
\label{e:contra2}
\end{align}
\item[(e)] An absolute minimizing pair $(u_{N},K_{N})$ of $E_{\lambda_{N}} (\cdot, \cdot, B_{2r_{N}}, g_{N})$ for which
\begin{align}\label{e:contra1}
\Omega^t (x_{k,N}, 2^{-k} r_{N}) + \lambda_{N} \|g_{N}\|_\infty^2 (2^{-k} r_{N})^{\frac12}
\leq\varepsilon\,,
\end{align}
for all $k\in\{0,\ldots,N\}$;
\item[(f)] For every $y\in B_{\gamma 2^{-N} r_{N}} (x_{N,N})$
\begin{align}\label{e:contra3}
\Omega^t (y, 2^{-N-1} r_{N}) + \lambda_{N} \|g_{N}\|_\infty^2 (2^{-N-1} r_{N})^{\frac12} 
>\varepsilon\,.
\end{align}
\end{itemize}
For each $N$ we consider the rescaled pairs
\begin{align}
v_{N} (x) &:= (2^{-N} r_{N})^{-\frac{1}{2}} u_{N} \big(x_{N,N} + 2^{-N} r_{N} x\big)\, ,\\
J_{N} &:= (2^{-N}r_{N})^{-1}(K_{N}-x_{N,N})\, .
\end{align}
Next, observe that from items (a)-(c) above
we achieve
\begin{equation}\label{e:gN infinitesimal}
\lim_{N\to\infty} \lambda_N\|g_{N}\|_\infty^2
(2^{-N} r_{N})^{\frac12}= 0\, .
\end{equation}
We can therefore apply Theorem~\ref{t:minimizers compactness} to conclude the convergence, up to subsequences, of $(v_{N}, J_{N})$ to a generalized minimizer $(v, J, \{p_{kl})\})$ of $E_0$. 

Note that, the points $x_{k,N}$ are mapped to points
$y_{k,N} := (2^{-N} r_{N})^{-1}(x_{k,N} - x_{N,N})$.
From \eqref{e:contra2} we immediately get
\[
|y_{N-k,N}| \leq \gamma \sum_{j=1}^{k} 2^{j} \leq \gamma 2^{k+1}\, . 
\]
In particular, for every fixed $k\geq 1$, up to extraction of an appropriate subsequence, we can assume that $y_{N-k,N}$ converges to a point $y_{k}$ with $|y_{k}|\leq \gamma 2^{k+1}$. 
In particular for $(v, J)$, from \eqref{e:contra1} we immediately conclude
\begin{equation}\label{e:contra10}
\Omega^t (y_{k}, 2^k) \leq \varepsilon \qquad \forall k\in\N\, . 
\end{equation}
But from the convergence in Theorem~\ref{t:minimizers compactness},
from \eqref{e:contra3} and \eqref{e:gN infinitesimal}, 
we also conclude that 
\begin{equation}\label{e:contra11}
\Omega^t (z, \sfrac 12) \geq \varepsilon\, \qquad \forall z \in B_\gamma\, .
\end{equation}

For each $k$ choose $\theta_k$ such that
\[
\dist_H (J\cap B_{2^{k+1}} (y_k), y_k+ \mathcal{R}_{\theta_k} (\mathscr{T}_0))
= \min_\theta \dist_H (J\cap B_{{k+1}} (y_k), 
y_k+ \mathcal{R}_{\theta} (\mathscr{T}_0))\,.
\]
We fix $\delta>0$ and we wish now apply Corollary~\ref{c:salto-puro-curvo} to the pair
$(2^{-k}(J-y_k), 2^{-\frac k2} v (y_k + 2^k \cdot))$ choosing $(u,K)$ equal to $(\mathcal{R}_{\theta_k} (\mathscr{T}_0)), 0)$, while $U= B_1\setminus \bar B_\gamma$ and $V= B_{1-\gamma}\setminus B_{2\gamma}$. This just requires $\varepsilon$ to be chosen sufficiently small depending on $\gamma$ and $\delta$. In particular we conclude that $2^{-k}(J-y_k)\cap (B_{1-\gamma}\setminus B_{2\gamma})$ is the union of three $C^1$ arcs, $\delta$-close to the three segments $\mathcal{R}_{\theta_k} (\mathscr{T}_0))\cap (B_{1-\gamma}\setminus B_{2\gamma})$. Scaling back, we conclude that in $A_k := J\cap (B_{2^k (1-\gamma)}\setminus B_{2^{k+1} \gamma})$ the set $J$ consists of three $C^1$ arcs with Hausdorff distance less than $\delta 2^k$ from the three segments $y_k+\mathcal{R}_{\theta_k} (\mathscr{T}_0)$. Choosing $\gamma$ smaller than a geometric constant, it is easy to see that each of the three $C^1$ curves $A_k \cap J$ coincide with one of the three $C^1$ curves $A_{k+1}\cap J$ in $A_k\cap A_{k+1}$, and viceversa, see 
Figure~\ref{f:figura-tripunto}.

\begin{figure}
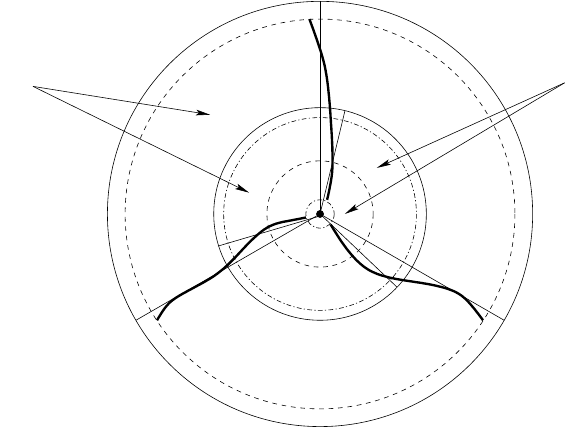
\caption{ The region $A_1$ is delimited by the two dashed circles, while the region $A_0$ is delimited by the two dashed-and-dotted circles. In each of this regions the set $J$ consists of three $C^1$ arcs which are close to three straight lines meeting at $120$ degrees. The arcs must coincide where the regions $A_0$ and $A_1$ overlap, hence $J\cap (A_0 \cup A_1)$ consists of three $C^1$ arcs.}\label{f:figura-tripunto}
\end{figure}

Thus $J \cap (\mathbb R^2 \setminus B_{\frac12})$ consists of three distinct nonselfintersecting $C^1$ infinite curves $\Gamma_i$, which subdivide $\mathbb R^2\setminus B_{\frac12}$ in three infinite connected components. Note that, moreover, for each $r\geq \frac{1}{2}$, each curve $\Gamma_i$ intersects each $\partial B_r$ in exactly one point $p_i (r)$ and that, the three points $p_1 (r), p_2 (r), p_3 (r)$ subdivide $\partial B_r$ in three arcs of length $\frac{2\pi}{3} r (1-\delta) \leq l_i (r) \leq \frac{2\pi}{3} r (1+\delta)$. 

Therefore, we are in a position to apply Proposition~\ref{p:monotonia-1} and deduce from \eqref{e:DLMS monotonia} that
\begin{align*}
\int_{2^{k}}^{2^{k+1}}&
\frac1\rho\int_{\partial B_\rho}|\nabla v|^2d\rho\leq
C \left(F (2^{k+1}) - F (2^{k})\right) \qquad \forall k \geq -1\,,
\end{align*}
where we recall that $F(r):=2d(r)+ \frac{\ell (r)}{r}$.
Summing for $k=-1$ to $k_0-1$, using that $F(r)$ is positive and thanks to the density upper bound \eqref{e:upper bound}, we conclude
\begin{equation}
\int_{\frac12}^{2^{k_0}} \frac{1}{\rho} \int_{\partial B_\rho}|\nabla v|^2d\rho \leq C F (2^{k_0}) \leq  \frac{C}{2^{k_0}} E_0 (v, J, B_{2^{k_0}}) \leq C\,.
\end{equation}
Letting $k_0$ to infinity we achieve
\[
\int_{\frac12}^\infty\frac1\rho\int_{\partial B_r}|\nabla v|^2d\rho\leq C\,,
\]
or equivalently by integration on dyadic intervals, 
\begin{equation}\label{e:dyadic}
\sum_{k\geq0}2^{-k}\int_{B_{2^k}\setminus B_{2^{k-1}}}|\nabla v|^2\leq C\,.
\end{equation}
We claim that $(v, J,\{p_{ij}\})$ is indeed a triple junction, with $z_0$ being the point where the three half-lines meet together. This obviously would give a contradiction to \eqref{e:contra11}, because $\Omega^t (0, 2) \leq \varepsilon$, with $\varepsilon \leq \varepsilon_0 (\gamma)$, would imply, for an appropriately chosen $\varepsilon_0 (\gamma)$, that $z_0$ belongs to $B_\gamma$.

To this aim we start off showing that for all $R\geq 1$ the connected components of $\partial B_R \setminus J$ belong to three distinct connected components of $\mathbb R^2 \setminus J$. As a first step we show that each ``blow-down'' of $(v,J,\{p_{ij}\})$, i.e. any limit as $R\to \infty$ of a subsequence of the rescalings $(v_{0,R}, J_{0,R}, \{p_{ij}\})$, is a triple junction. For the sake of notational simplicity we drop $0$ in the previous subscripts.

First of all, since $v$ is harmonic in 
$\mathbb{R}^2\setminus J$ and being $J\setminus B_{\frac12}$ smooth, standard estimates for harmonic functions give for every $k\geq-1$
\begin{equation}\label{e:LinftyL2}
\|\nabla v\|^2_{L^\infty(B_{2^{k+1}}\setminus (B_{2^k}\cup J))}\leq
C\Big(2^{-2k}\int_{B_{2^{k+2}}\setminus (B_{2^{k-1}}\cup J)}|\nabla v|\Big)^2
\leq C2^{-2k}\int_{B_{2^{k+2}}\setminus (B_{2^{k-1}}\cup J)}|\nabla v|^2
\end{equation}
$C>0$ independent of $k$.
Recall that $\Gamma_i$ is locally $C^{1,1}$ by item (b) in Proposition~\ref{p:variational identities 2}, and that by \eqref{e:Euler curvature g}
\[
\kappa_i=-(|\nabla v^+|^2-|\nabla v^-|^2)\qquad\mathcal{H}^1\textrm{-a.e. on }\Gamma_i
\]
where $\kappa_i$ denotes the curvature of $\Gamma_i$. Thus
we conclude from the energy upper bound \eqref{e:upper bound}, and estimates \eqref{e:dyadic} and \eqref{e:LinftyL2}
\begin{align*}
\int_{J\setminus  B_{\frac12}}&|\kappa_i|d\mathcal{H}^1\leq 2
\sum_{k\geq 0}\|\nabla v\|^2_{L^\infty(B_{2^{k}}\setminus (B_{2^{k-1}}\cup J))} \mathcal{H}^1\big(\Gamma_i\cap(B_{2^{k}}\setminus (B_{2^{k-1}}\cup J))\big)\\
&\leq C \sum_{k\geq 0}2^{-k}\int_{B_{2^{k+2}}\setminus (B_{2^{k-1}}\cup J)}|\nabla v|^2\leq C\,.
\end{align*}
In particular, each $\Gamma_i$ is asymptotic at $\infty$ to some straight line $\ell_i$, and in particular as $R\uparrow \infty$, $J_R$ converges, locally in the sense of Hausdorff, to a set $J_\infty$ which is the union of three half-lines meeting at the origin. Apply Theorem~\ref{t:compactness-2} and let $(v_\infty, J_\infty, \{q_{kl}\})$ be any limit of a subsequence $R_j\uparrow \infty$ of $(v_{R_j}, J_{R_j}, \{p_{kl}\})$. It follows that $v_\infty$ is harmonic in each of the three sectors which form the connected components of $\mathbb R\setminus J_\infty$ and that $\frac{\partial v_\infty}{\partial \nu} =0$ on $J_\infty$. We also know that the angles formed by the three half-lines $\ell_i$ are all close to $\frac{2\pi}{3}$. Using Theorem~\ref{t:class-global}(iii) we conclude that $(v_\infty, J_\infty, \{q_{kl}\})$ is a triple junction.

We next argue that the three connected components of $\mathbb R^2 \setminus (B_{\frac12}\cup J)$ belong to different connected components of $\mathbb R^2\setminus J$. In turn, this implies the claim that $\partial B_R\setminus J$ belongs to three different connected components of $\mathbb R^2 \setminus J$ for all $R\geq1$. Indeed, by contradiction fix two connected components $\Omega_1$ and $\Omega_2$ of $\mathbb R^2 \setminus (B_{\frac12}\cup J)$ and two arbitrary points $q_1\in\Omega_1$ and $q_2\in\Omega_2$ that are connected by a $C^1$ arc in $\mathbb R^2 \setminus J$. 
The boundaries of $\Omega_1$ and $\Omega_2$ have one of the $\Gamma_i$ in common and to fix ideas let us assume it is $\Gamma_1$. The latter is asymptotic to the half $\ell_1$, which is common to the boundaries of two of the sectors of $\mathbb R^2\setminus J_\infty$. We denote $\Lambda_1$ and $\Lambda_2$ these two sectors and we fix two points $\hat{q}_1\in\Lambda_1$ and $\hat{q}_2\in\Lambda_2$ with $|\hat{q}_1|=|\hat{q}_2|=1$. It is then easy to see that there are two $C^1$ curves $\sigma_1, \sigma_2$
contained in $\mathbb R^2\setminus (B_{\frac12}\cup J)$ and containing, respectively, $q_1$ and $q_2$, and
such that for $R$ sufficiently large, $\sigma_i\setminus B_R = 
\{\rho \hat{q}_i : \rho\geq R\}$. 
Thus by \eqref{e:dyadic} and \eqref{e:LinftyL2} we can write, for $2^{k-1}\geq R$,
\begin{align*}
|v (2^k \hat{q}_i)-v (2^{k-1} \hat{q}_i)|&\leq
\int_{\sigma_i\cap(B_{2^k}\setminus B_{2^{k-1}})}|\nabla v|\,d\mathcal{H}^1\\
&\leq \mathcal{H}^1\big(\sigma_i\cap(B_{2^{k}}\setminus B_{2^{k-1}})\big)\|\nabla v\|_{L^\infty(B_{2^{k}}\setminus B_{2^{k-1}})}\leq C\,2^{\frac k2}
\end{align*}
Hence we can estimate 
\[
|v (2^k \hat{q}_1) - v (2^k \hat{q}_2)| \leq C 2^{\frac k2}
\, ,
\]
given that there is a curve connecting $R\hat{q}_1$ and $R\hat{q}_2$ which does not intersect $J$.

On the other hand, since all blow-downs of $(v, J, \{p_{kl}\})$ are triple junctions, from Theorem~\ref{t:compactness-2} we necessarily conclude that 
\[
\lim_{k\to\infty} \frac{|v (2^k \hat{q}_1) - v (2^k \hat{q}_2)|}{2^{\frac k2}} =\infty\, ,
\]
which would be a contradiction.

In particular, having shown that each connected component of $\partial B_R\setminus J$ belongs to distinct connected components of $\mathbb R^2\setminus J$, we can apply Proposition~\ref{p:monotonia-0}, and thus we conclude that \[
[1,\infty)\ni R\mapsto
\frac{1}{R} \int_{B_R\setminus J} |\nabla v|^2 
\] 
is nondecreasing in $R$. Since however we know that the blow-downs of $(v,J,\{p_{kl}\})$ are triple junctions, 
we have
\[
\lim_{R\uparrow \infty} \frac{1}{R} \int_{B_R\setminus J} |\nabla v|^2 = 0\, .
\]
We thus conclude that $\int_{B_R\setminus J} |\nabla v|^2 =0$ for every $R\geq 1$. In particular by Theorem~\ref{t:class-global} 
$(v,J,\{p_{kl}\})$ is itself a triple junction. 

\subsection{Proof of Lemma~\ref{l:decay iteration} for restricted minimizers} In the case of restricted (and generalized restricted) minimizers, we apply the same procedure above. Note, however, that we cannot immediately conclude that the blow-downs are triple junctions, because the proof that $(v_\infty, J_\infty, \{q_{kl}\})$ are triple junctions, i.e. $|q_{kl}| = \infty$, relies on a comparison with a competitor which increases the number of connected components. Now, assume that for every fixed $R>1$ the number of connected components of the sets $J_{R_j}$ in the blow-down sequence is larger than $1$, namely at least two, for an infinite number of $j$. Then, for the limiting blow-down we can use a competitor which has two connected components, because the cut-and-paste argument which from this competitor yields the competitors for the approximating sequence keeps the same number of connected components. On the other hand the argument $|q_{kl}|=\infty$ relies indeed in using a better competitor which has two connected components. Hence in this case the conclusion is valid and in particular the set $J$ is connected. If instead there is an $R$ for which the number of connected components of the $J_{R_j}$ is $1$ for an infinite set of $j$'s, then we conclude directly that $J$ is connected. At any rate, in both cases $J$ must be connected. Hence we can apply the monotonicity formula in Proposition~\ref{p:monotonia-0} and conclude that $J$ is the union of three half lines, while $v$ is locally constant. Note that the proof that the three half lines must meet at 120 degrees holds for restricted minimizers as well, since it is based on exhibiting competitors that consist of a single connected component.

\section{Proof of Theorem~\ref{t:eps_tripunto}}\label{s:triple-junction-final}

Denote by $\varepsilon_0$ the constant in assumption (ii) of Theorem~\ref{t:eps_salto_puro} and fix a $\delta>0$, whose choice will be specified later. Let then the constant $\varepsilon$ in the statement of Theorem~\ref{t:eps_tripunto} be smaller than the constant $\eta$ provided by the conclusion of Lemma~\ref{l:C1smooth}. Let $y$ be the point in $B_{\delta r} (x)$ provided by Lemma~\ref{l:C1smooth} itself. Arguing as in Lemma~\ref{l:decay iteration}, if $\delta$ is chosen sufficiently small with respect to $\varepsilon_0$, then $K \cap B_r (y)\setminus \{y\}$ consists of three arcs $\Gamma_1, \Gamma_2, \Gamma_3$ with the following properties:
\begin{itemize}
    \item[(i)] Each $\Gamma_i$ is a $C^{1,\frac12}$ arc in $B_{r(1-s)} (y) \setminus B_{rs} (y)$ for every $s\in(0,\frac12)$.
    \item[(ii)] For every $k\in\N$, $k\geq 1$, there is an angle $\theta_k$ such that 
    $K\cap (B_{2^{-k} r} (y)\setminus B_{2^{-k-1} r} (y))$ is $\delta 2^{-k} r$ close in the Hausdorff distance to $(y + \mathcal{R}_{\theta_k} (\mathscr{T}_0)))\cap (B_{2^{-k} r}(y)\setminus B_{2^{-k-1} r} (y))$.
    \item[(iii)] Each $\Gamma_i$ is, in $B_{2^{-k} r} (y)\setminus B_{2^{-k-1} r} (y)$, a $C^{1,\frac12}$ graph over the corresponding straight line $\ell_{k,i}$ of $y + \mathcal{R}_{\theta_k} (\mathscr{T}_0)$ of a function $\psi_{k,i}$ with $\|\psi_{k,i}'\|_\infty \leq C \delta$.
\end{itemize}
From now on, in order to simplify our notation, we assume without loss of generality that $y=0$.

We are now in the position to apply Proposition~\ref{p:monotonia-1} for all radii $\rho\in(0,r)$. Thus, from \eqref{e:DLMS monotonia} we infer that 
\[
\frac{D'(\rho)}{\rho}\leq C\,F'(\rho)+C
\]
for $\mathcal{L}^1$ a.e. $\rho\in(0,r)$. 
Therefore, by direct integration and the density upper bound 
\eqref{e:upper bound} we deduce that 
\begin{equation*}
\int_0^{r}\frac{D'(\rho)}{\rho}d\rho\leq 
CF(r)+Cr\leq \frac{C}{r}E_\lambda(u,K,B_r,g)+Cr\leq C\,,
\end{equation*}
for a dimensional constant $C>0$, or equivalently by passing to dyadic intervals, 
\begin{equation}\label{e:dyadic bis}
\sum_{k\geq 2}\frac1{2^{-k} r}\int_{B_{2^{-k+2} r}\setminus 
B_{2^{-k+1}r}}|\nabla u|^2\leq C\,.
\end{equation}
Since $\Delta u = \lambda(u- g)$ on $B_{2r}\setminus K$, $K\cap (B_{2^{-k+2} r} \setminus B_{2^{-k+1} r})$ is $C^{1,\frac12}$ and $u$ satisfies the Neumann boundary conditions on $K$, by elliptic regularity we have
\begin{align}
\|\nabla u\|^2_{L^\infty(B_{2^{-k+1} r} \setminus (B_{2^{-k} r}\cup K))}& \leq \frac{C}{2^{-2k} r^2}  
\int_{B_{2^{-k+2}}\setminus (B_{2^{-k-1}r}\cup K)}|\nabla u|^2
 + C  \lambda 2^{-2k}r^2(\|u\|_\infty^2 + \|g\|_\infty^2)
\label{e:LinftyL2 bis}
\end{align}
for some constant $C>0$. 

Denote by $\kappa_i$ the curvature of $\Gamma_i$ and recall that
\[
\kappa_i=-(|\nabla u^+|^2-|\nabla u^-|^2)
-\lambda(|u^+-g_K|^2-|u^--g_K|^2)
\] 
$\mathcal{H}^1$-a.e. on 
$\Gamma_i$ by \eqref{e:Euler curvature g} for an appropriate trace $g_K$ of $g$ (which enjoys the same upper bound on the $L^\infty$ norm). From \eqref{e:LinftyL2 bis} we thus conclude
\begin{align}\label{e:stima kik}
&\int_{\Gamma_i\cap(B_{2^{-k+1} r} \setminus B_{2^{-k-1} r})}|\kappa_i|d\mathcal{H}^1
\leq \frac{C}{2^{-k} r} \int_{B_{2^{-k+2} r}\setminus B_{2^{-k-1} r}} |\nabla u|^2 + C 2^{-k} r 
\,
\end{align}
for some constant $C>0$.
We therefore conclude from \eqref{e:dyadic bis}
\begin{equation*}
\int_{\Gamma_i \cap B_r \setminus \{0\}}|\kappa_i|d\mathcal{H}^1\leq C \,.
\end{equation*}
In turn, from this we deduce that $\overline{\Gamma_i}$ is a $C^1$ graph up to the origin. Moreover, a blow-up argument shows that the tangents to $\Gamma_i$ in the origin form equal angles and thus, up to rotations, we can assume that they are given by $\{\theta=0\}$, $\{\theta=\frac{2\pi}{3}\}$
and $\{\theta=\frac{4\pi}{3}\}$.

We shall prove next that $\overline{\Gamma_i}$ is a $C^{1,\gamma}$ graph up to the origin, $\gamma\in(0,1)$, by means of a suitable monotonicity formula. First notice that for $\rho\in(0,r)$ we have $\mathcal{H}^0(\partial B_\rho\cap K)=3$. Then, if we consider the set $K'$ which is the union of the three segments obtained by joining each point in $\partial B_\rho\cap K$ with the origin we have
\[
\mathcal{H}^1(B_\rho\cap K)\leq 3\rho=\mathcal{H}^1(B_\rho\cap K')\,.
\]
In addition, $B_\rho\setminus K' = \Omega_1\cup\Omega_2\cup\Omega_3$, where each $\Omega_i$ is a convex cone with vertex the origin and opening $\alpha_i$, with $|\alpha_i-\frac23\pi|\leq C\delta(\rho)$, with 
$\delta(\rho)\to 0$ if $\rho\to0$. 
Therefore if $C\delta(\rho)<\delta_0$ for $\rho\leq\rho_0$ ($\delta_0$ can be chosen as small as we want up to reducing $\rho_0$) and if 
$w_i$ is the (harmonic) function provided by Lemma~\ref{l:extension}, 
we have
\begin{equation}\label{e:enrg harmonic 2}
\int_{\Omega_i}|\nabla w_i|^2
\leq\frac{\alpha_i}\pi\rho\int_{\partial\Omega_i\cap \partial B_\rho}
\left(\frac{\partial u}{\partial\tau}\right)^2.
\end{equation}
If $w$ is defined as $w|_{\Omega_i}=w_i$, $w=u$ on $B_{2r}\setminus B_\rho$, then testing the minimality of $(u,K)$ with the competitor $(w,K')$ we get for $\alpha_0:=\frac{2}{3}\pi+\delta_0$ (note that 
$\max\{\alpha_1,\alpha_2,\alpha_3\}\leq\alpha_0$)
\[
\int_{B_\rho}|\nabla u|^2+\mathcal{H}^1(B_\rho\cap K)+\lambda\int_{B_\rho}|u-g|^2\leq
\frac{\alpha_0}\pi\rho\int_{\partial B_\rho}
\left(\frac{\partial u}{\partial\tau}\right)^2+3\rho
+\lambda\int_{B_\rho}|w-g|^2\,,
\]
from which we deduce straightforwardly that for $\mathcal{L}^1$-a.e. $\rho\in(0,\rho_0)$
\[
D(\rho)\leq \frac{\alpha_0}\pi\rho D'(\rho)
+ 4\pi\lambda\|g\|^2_\infty\rho^2.
\]
In turn, the latter inequality and the energy upper bound in 
\eqref{e:upper bound} imply that for all $\rho\in(0,\rho_0)$
\begin{equation}\label{e:decay ki}
D(\rho)\leq C
\rho_0^{-\frac\pi{\alpha_0}}(D(\rho_0)+\rho_0^2)\rho^{\frac\pi{\alpha_0}}
\leq C \rho_0^{1-\frac\pi{\alpha_0}}\rho^{\frac\pi{\alpha_0}}\,,
\end{equation} 
for some constant $C>0$ depending on $\lambda$, $\|g\|_\infty$
and $\alpha_0$.
Finally, estimates \eqref{e:stima kik} and \eqref{e:decay ki} yield
\[
\int_{\Gamma_i\cap B_\rho}|\kappa_i|d\mathcal{H}^1\leq C\rho^{\frac\pi{\alpha_0}-1}\,,
\]
and the claimed $C^{1,\gamma}$ regularity of $\overline{\Gamma_i}$ up to the origin then follows at once
choosing $\delta_0$ sufficiently sufficiently small so that $\alpha_0<\pi$.

We now come to the construction of the diffeomorphism $\Phi$ which is given as the composition $\Phi_0 \circ \Phi_1$ of two other diffeomorphisms. We note that due to our assumption
we have actually set $\theta=0$. Without loss of generality we assume $r=1$. $\Phi_0$ maps $y$ (which without loss of generality has been assumed to be $0$) into the origin. If we let $\varphi\in C^\infty_c (B_1)$ be a function which is identically equal to $1$ on $B_{\frac12}$, $\Phi_0$ is then given by the formula
\[
\Phi_0 (z) := z + (1- \varphi (z)) x \, .  
\]
(observe that $B_{\delta}(x)$ contains the origin, in particular $|x|\leq \delta$: thus if $\delta$ is sufficiently small, the latter map can be seen to be $C^1$ close to the identity, and hence a diffeomorphism).
$\Phi_1$ is then the inverse of a map
$\Psi$ which maps $K' :=\Phi_0^{-1} (K)$ onto three straight half-lines emanating from $0$ and forming equal angles. We know indeed that $K'$ consists of three $C^{1,\alpha}$ arcs $\gamma_1$, $\gamma_2$, and $\gamma_3$ meeting at $0$, where they form equal angles. We use polar coordinates $(\theta, \rho)$ to define the map $\Phi_1$. Upon a suitable rotation we can also assume the tangents to $\gamma_1$, $\gamma_2$, and $\gamma_3$ at the origin are given by $\{\theta =0\}$, $\{\theta = \frac{2\pi}{3}\}$, and $\{\theta = \frac{4\pi}{3}\}$. In particular each $\gamma_i$ is given in polar coordinates by $\{(\rho, \theta_i (\rho))\}$, where $\theta_i: (0, 2r) \to \mathbb S^1$ is a $C^{1,\alpha}$ function and satisfies 
$\|\theta_i - (i-1) \frac{2\pi}{3}\| \leq \delta$ and $|\theta_i' (r)|\leq \delta r^{-1}$. In particular we can assume that $\delta<\frac{1}{2}.$ Let now $\psi\in C^\infty_c (-1,1)$ be a function which is identically equal to $1$ on $(-\frac{1}{2}, \frac{1}{2})$. In polar coordinates the map $\Phi_1$ is then given by
\[
(\theta, \rho) \mapsto \Big(\theta + \sum_i \psi\left(\theta - (i-1) {\textstyle{\frac{2\pi}{3}}}\right) \left( \theta_i (\rho) - (i-1) {\textstyle{\frac{2\pi}{3}}}\right), \rho\Big)\,.
\]
The estimate \eqref{e:estimate_triple} follows from the corresponding estimates for the functions $\theta_i$ and for $|x|$ and is left to the reader.

\subsection{Proof of Corollary~\ref{c:curvatura=0 nel tripunto}}
\begin{proof}
We first note that $K\setminus\{\bar{x}\}$ is the union of three $C^{1,1}$ arcs in view of Proposition~\ref{p:variational identities 2}.
We use the notation and arguments in the proof of Theorem~\ref{t:eps_tripunto}, in particular $\bar{x}=0$ by translation. 
We build upon the conclusions of Theorem~\ref{t:eps_tripunto} and substitute the monotonicity formula employed there by a sharper one inspired by Proposition~\ref{p:monotonia-0}.

Let $\rho\in(0,2r)$, then $\mathcal{H}^0(\partial B_\rho\cap K)=3$.
Let $B_\rho\setminus K=S_1\cup S_2\cup S_3$ with 
$\partial B_\rho\setminus K=\gamma_1\cup\gamma_2\cup\gamma_3$,
$\gamma_i\subset\partial S_i$, and  
$|\mathcal{H}^1(\gamma_i)-\frac23\pi\rho|\leq C\varepsilon\rho^{1+\gamma}$.
Denote by $u_i$ the mean value of $u$ on $\gamma_i$ and recall that
$\int_{\gamma_i}\frac{\partial u}{\partial\nu}=0$ by \eqref{e:Bonnet}
(here we use $\lambda=0$).
For every $i\in\{1,2,3\}$, by the sharp Poincar\'e-Wirtinger inequality 
we then get that
\begin{align*}
\int_{S_i}|\nabla u|^2&\stackrel{\eqref{e:int by parts}}{=}
\int_{\gamma_i}u\frac{\partial u}{\partial\nu}\stackrel{\eqref{e:Bonnet}}{=}
\int_{\gamma_i}(u-u_i)\frac{\partial u}{\partial\nu}
\leq\int_{\gamma_i}\left(\frac{\delta}{2}(u-u_i)^2
+\frac1{2\delta}\left(\frac{\partial u}{\partial\nu}\right)^2\right)\\
&\leq\int_{\gamma_i}
\left(\frac{\delta}{2}\left(\frac{\mathcal{H}^1(\gamma_i)}{\pi}\right)^2
\left(\frac{\partial u}{\partial\tau}\right)^2
+\frac1{2\delta}\left(\frac{\partial u}{\partial\nu}\right)^2\right)
= \frac{\mathcal{H}^1(\gamma_i)}{2\pi}\int_{\gamma_i}|\nabla u|^2
\end{align*}
having chosen $\delta=\frac{\pi}{\mathcal{H}^1(\gamma_i)}$ in the last inequality. In turn, from this we conclude that for $\mathcal L^1$-a.e. $\rho\in(0,\rho_0)$, $\rho_0<2r$ sufficiently small, we have
\begin{align*}
D(\rho)&=\int_{B_\rho\setminus K}|\nabla u|^2=\sum_{i=1}^3\int_{S_i}|\nabla u|^2\leq
\max_{i\in\{1,2,3\}}\frac{\mathcal{H}^1(\gamma_i)}{2\pi}
\int_{\partial B_\rho\setminus K}|\nabla u|^2\\
&=\max_{i\in\{1,2,3\}}\frac{\mathcal{H}^1(\gamma_i)}{2\pi} D'(\rho)
\leq \frac5{12}\rho D'(\rho).
\end{align*}
By direct integration and the energy upper bound in \eqref{e:upper bound} we deduce that for all $\rho\in(0,\rho_0)$
\[
D(\rho)\leq D(\rho_0)\Big(\frac{\rho}{\rho_0}\Big)^{\frac5{12}}
\leq 2\pi \rho_0^{-\frac75}\rho^{\frac{12}5},
\]
and thus estimate \eqref{e:stima kik} yields
\[
\int_{\Gamma_i\cap B_\rho}|\kappa_i|d\mathcal{H}^1
\leq C\rho^{\frac75}.
\]
In particular, since $\Gamma_i$ is $C^{1,\gamma}$ up to the origin, 
we deduce that it is actually $C^2$ in the origin itself with $\kappa_i(0)=0$. 

Since $u$ has $C^{1,\alpha}$ extensions on each side of $\Gamma_i\cap A$ (cf. item (a) in 
Proposition~\ref{p:variational identities 2}) and $\Gamma_i$ is $C^{1,\gamma}$, the conclusion follows at once using
\eqref{e:Euler curvature g} for $\lambda=0$.
\end{proof}
\begin{remark}
Actually, the arcs composing $K\setminus\{\bar{x}\}$ are $C^\infty$ (resp. analytic) if $g$ is $C^\infty$ (resp. analytic) in view of the higher regularity theory contained in \cite[Theorem 7.42]{AFP00} (resp. \cite{KLM}).
\end{remark}

\section{Proof of Theorem~\ref{t:main} for pure jumps and triple junctions}\label{ss:proof of (b)-(c) Theorem main}

We start by proving conclusion (i). First of all we observe that, by a standard covering argument, it is enough to prove the conclusion in a ball $B_{\delta r}(x)$ rather than in $B_r (x)$, where $\delta$ is a fixed geometric constant. We focus on the case of a pure jump and of absolute minimizers. The various other possibilities are all treated with the same idea, with minor changes. Without loss of generality assume $x=0$.

We wish to apply Theorem~\ref{t:eps_salto_puro} and in order to do it we claim that
\begin{itemize}
    \item[(Cl)] For any $\varepsilon >0$ there are constants $\delta>0$ and $\eta>0$ such that, if $(u,K)$ is an absolute minimizer and, for some $r>0$, ${\rm dist}_H (\mathscr{V}_0\cap \overline{B}_{2r} , K\cap \overline{B}_{2r} )\leq \eta r$, then
    \[
    \int_{B_{2\delta r}\setminus K} |\nabla u|^2 + \lambda \|g\|_\infty^2 (2\delta r)^{\frac{3}{2}} < 2 \varepsilon \delta r\, .
    \]
\end{itemize}
First of all, since $\|g\|_\infty \leq M_0$, $\lambda \leq 1$, and $r\leq 1$, it is easy to see that, for $\delta$ sufficiently small,
\[
\lambda \|g\|_\infty^2 (2\delta r)^{\frac{3}{2}} < \varepsilon \delta r\, .
\]
We therefore focus on the Dirichlet energy and argue by contradiction. In particular, if our claim is false, we find a sequence of radii $r_j\leq 1$ and of absolute minimizers $(u_j, K_j)$ in $B_{2r_j}$ of $E_{\lambda_j}$ with fidelity terms $g_j$, satisfying the following properties:
\begin{itemize}
    \item[(1)] ${\rm dist}_H (\mathscr{V}_0\cap \overline{B}_{2r_j} , K_j\cap \overline{B}_{2r_j})
    \leq 2^{-j} r_j$;
    \item[(2)] $\int_{B_{2 j^{-1} r_j}\setminus K_j} 
    |\nabla u_j|^2 \geq \varepsilon j^{-1} r_j$ for some positive $\varepsilon$. 
\end{itemize}
We now consider the rescalings $(v_j, J_j)$ of $(u_j, K_j)$ given by
\begin{align*}
J_j &:= \frac{K_j}{j^{-1} r_j}\\
v_j (x) &:= (j^{-1} r_j)^{-\frac{1}{2}} u_j (j^{-1} r_j x)\, .
\end{align*}
Note that $\lambda_j \|g_j\|_\infty^2j^{-1} r_j \to 0$ as $j\uparrow \infty$. In particular we can apply Theorem~\ref{t:minimizers compactness} and assume that, up to a subsequence not relabeled, the pairs $(v_j, J_j)$ converge to a generalized global minimizer $(v,J)$. Observe that on any disk $B_R$ we have 
\[
{\rm dist}\, (\mathscr{V}_0 \cap \overline{B}_R ,
J_j \cap \overline{B}_R) \leq j 2^{-j}\, .
\]
Hence $J$ must coincide with $\mathscr{V}_0$. But then from Theorem~\ref{t:class-global} it follows that $\nabla v\equiv 0$ on $\mathbb R^2\setminus \mathscr{V}_0$. In particular, by the convergence proved in Theorem~\ref{t:minimizers compactness}, 
\[
\lim_{j\to \infty} \int_{B_1\setminus J_j} |\nabla v_j|^2 = 0\, .
\]
On the other hand
\[
\int_{B_1\setminus J_j} |\nabla v_j|^2 = \frac{1}{j^{-1} r_j} \int_{B_{2 j^{-1} r_j\setminus K_j}} |\nabla u_j|^2\, ,
\]
and we thus reach a contradiction with (2). 

\medskip

We next come to point (ii). In the case of pure jumps observe that if $\lambda=0$ or $g\in C^0$, $\Delta u$ is continuous and from classical estimates for the Neumann problem we infer that $\nabla u$ has a $C^{0,\alpha}$ extensions up to $K$ on each side of it. In particular, from the Euler-Lagrange conditions of Proposition~\ref{p:variational identities 2} we conclude that the distributional curvature of $K$ is continuous, which in turn implies its $C^2$ regularity. 

As for point (iii), it follows from Corollary~\ref{c:curvatura=0 nel tripunto} if the extremum is a triple junction.

\chapter{The Bonnet-David rigidity theorem for cracktip}\label{c:Bonnet David}

\section{Main statement and consequences}

This part of our notes is devoted to prove the following rigidity theorem of Bonnet and David, which in turn is the first step towards the proof of case (c) in Theorem~\ref{t:main}, namely Corollary~\ref{c:connectedness}. We start by stating both these facts.

\begin{theorem}\label{t:Bonnet-David}\label{T:BONNET-DAVID}
Let $(u,K,\{p_{kl}\})$ be a global generalized minimizer and assume that, for a sufficiently large radius $R$:
\begin{itemize}
    \item[(a)] $K\setminus B_R$ consists of a single unbounded connected component;
    \item[(b)] $K\cap \partial B_R$ consists of a single point.
\end{itemize}
Then $(u,K,\{p_{kl}\})$ is a cracktip. 
\end{theorem}

\begin{corollary}\label{c:connectedness}
There is a $\delta >0$ with the following property. 
Assume that
\begin{itemize}
    \item \eqref{e:g-and-lambda} holds;
    \item $(u,K)$ is an absolute minimizer of $E_\lambda$ in $B_{4r} (x)$ for some $4r \leq 1$;
    \item ${\rm dist}_H\, (K\cap \overline{B}_{4r} (x), (x+\mathcal{R_\theta}) (\mathscr{V}_0^+)\cap \overline{B}_{4r}(x)) < \delta r$. 
\end{itemize}    
Then $B_{2r} (x)\cap K$ consists of a single continuous nonselfintersecting arc with an endpoint $y\in B_r (x)$ (which according to our terminology is a terminal point of $K$) and an endpoint in $\partial B_{2r} (x)$. Moreover the arc is $C^{1,1}$ in $B_{2r} (x)\setminus \{y\}$.
\end{corollary}
Note that the conclusion of Corollary~\ref{c:connectedness} does not prevent the possibility that $K$ spirals around $y$. This fact needs further analysis and will be actually ruled out in Chapter~\ref{ch:crack}.

In this chapter we will not only prove Theorem~\ref{t:Bonnet-David} but also give some general properties of global generalized minimizers $(u, K, \{p_{kl}\})$ which are not elementary. In order to simplify our terminology, we will call them {\em nonelementary} global minimizers.\index{nonelementary global minimizer@nonelementary global minimizer}\index{global minimizer, nonelementary@global minimizer, nonelementary} An important result of David and L\'eger, presented below, shows that for a {\em nonelementary} global minimizer the set $\mathbb R^2\setminus K$ is in fact connected. Due to this we can (and will) omit to mention the ``normalizations'' $\{p_{kl}\}$ for nonelementary global minimizers.

We conclude this section by showing how Corollary~\ref{c:connectedness} follows from Theorem~\ref{t:Bonnet-David}.

\subsection{Proof of Corollary~\ref{c:connectedness}}
The argument is similar to that used in Section~\ref{s:tripunto-inizio}. Following a similar path we introduce the quantities 
\[
\Omega^c (\theta, x, r) := r^{-1}{\rm dist}_H (K\cap \overline{B}_{2r} (x), (x+\mathcal{R}_\theta (\mathscr{V}_0^+))\cap \overline{B}_{2r}(x))
\]
and\index[simb]{aagZ^c(t,x,r)@$\Omega^c(\theta,x,r)$}\index[simb]{aagZ^c(x,r)@$\Omega^c(x,r)$}
\[
\Omega^c (x,r) := \min_\theta \Omega^c (\theta, x, r)\, .
\]
The following lemma is then the analog of Lemma~\ref{l:decay iteration}.

\begin{lemma}\label{l:decay iteration 2}
For every $\gamma>0$ sufficiently small there exists $\varepsilon_0 (\gamma)$ with the following property. Assume $\varepsilon \in (0, \varepsilon_0)$ and let $N = N (\varepsilon)\in \mathbb N$ be sufficiently large. Let $r\in (0,1]$ and assume that $(u,K)$ is an absolute minimizer of $E_\lambda$ in $B_{2r} (x)$, while $x=x_0, x_1, \ldots, x_N$ are points such that
\begin{align}
&\Omega^c (x_k, 2^{-k} r) + \lambda \|g\|_\infty^2 (2^{-k} r)^{\frac{1}{2}} \leq \varepsilon 
\qquad \forall k\in \{0, 1, \ldots, N\}\\
&|x_{k+1} -x_k|\leq \gamma 2^{-k} r \qquad \qquad \forall k\in \{0, 1, \ldots , N-1\}\, .
\end{align}
Then there is a point $x_{N+1}\in B_{2r}(x)$ such that $|x_{N+1} -x_N| \leq \gamma 2^{-N} r$ and 
\[
\Omega^c (x_{N+1}, 2^{-N-1} r) + \lambda \|g\|_\infty^2 (2^{-N-1}r)^{\frac{1}{2}} \leq \varepsilon\, . 
\]
\end{lemma}
\begin{proof} We argue by contradiction and assume that, for some $\gamma>0$ and $\varepsilon >0$ sufficiently small (the smallness will be specified later) there are 
\begin{itemize}
    \item[(a)] A family of numbers $\lambda_N \in [0,1]$;
    \item[(b)] A family of fidelity functions $g_N$ with $\|g_N\|_\infty \leq M_0$;
    \item[(c)] A family of radii $r_N\in (0,1]$;
    \item[(d)] A family of points $x_{k,N}$, for $k\in \{0, \ldots , N\}$, with  
    \begin{align}
    &x_{0,N} =x
    &|x_{k+1,N}-x_{k,N}| \leq \gamma 2^{-k} r_N \qquad 
    \forall k\in \{0, \ldots, N-1\};
    \end{align}
    \item[(e)] An absolute minimizing pair $(u_N, K_N)$ of $E_{\lambda_N} (\cdot, \cdot, B_{2r_N}, g_N)$ for which
    \begin{align}\label{e:Omega c N}
    \Omega^c (x_{k,N}, 2^{-k} r_N) + \lambda_N \|g_N\|_\infty^2 (2^{-k} r_N)^{\frac{1}{2}} \leq \varepsilon\,, 
    \end{align}
    for all $k\in \{0, \ldots , N\}$;
    \item[(f)] For every $y\in B_{\gamma 2^{-N}r_N} (x_{N,N})$
    \[
    \Omega^c (y, 2^{-N-1} r_N) + \lambda_N \|g_N\|_\infty^2 (2^{-N-1} r_N)^{\frac{1}{2}} > \varepsilon\, . 
    \]
\end{itemize}
For each $N$ we consider the rescaled pairs
\begin{align*}
v_N (x) &:= (2^{-N} r_N)^{-\frac{1}{2}} u_N \left( x_{N,N} + 2^{-N} r_N x\right) \\
J_N &:= (2^{-N}r_N)^{-1} (K_N - x_{N,N})\, .
\end{align*}
Next observe that from \eqref{e:Omega c N} we get
\[
\lambda_N \|g_N\|_\infty^2 2^{-N} r_N  \leq 
\varepsilon (2^{-N} r_N)^{\sfrac12}\, .
\]
and thus in particular $\lambda_N \|g_N\|_\infty^2 2^{-N} r_N\to 0$. We can therefore apply Theorem~\ref{t:minimizers compactness} to conclude the convergence, up to subsequences, of $(v_N, J_N)$ to a generalized minimizer $(v, J, \{p_{kl}\})$. 
Note that the points $x_{k,N}$ are mapped to the points
\[
y_{k,N}:= (2^{-N} r_N)^{-1}(x_{k,N}-x_{N,N})
\]
We thus infer $y_{N,N}=0$ for all $N$, and for $k\in\{1,\ldots,N\}$
\[
|y_{N-k,N}| \leq \gamma \sum_{j=1}^k 2^{j} \leq \gamma 2^{k+1}\, .
\]
Up to extraction of a subsequence we can assume that, for each fixed $k\geq 1$, $y_{N-k,N}$ converges, as $N\uparrow \infty$, to some $y_k$ with $|y_k|\leq \gamma 2^{k+1}$, $k\geq 1$. Set moreover $y_0=0$.
In particular for $(v,J)$ we have
\[
\Omega^c (y_{k},2^k) \leq \varepsilon \qquad \forall k\in\N\, , 
\] 
On the other hand our contradiction assumption implies as well
\begin{equation}\label{e:contradiction-crack}
\inf_{z\in B_{\gamma}} \Omega^c (z, \sfrac12) \geq \varepsilon\, .
\end{equation}

Next observe that, by taking $\varepsilon$ smaller than a suitably chosen $\varepsilon_0 (\gamma)$, the $\varepsilon$-regularity theory at pure jumps would imply that, in the corona
$A_k= B_{(1-\gamma) 2^{k+1}} (y_k) \setminus B_{\gamma 2^{k+1}} (y_k)$ the set $K$ consists of a single arc with endpoints on the circles $\partial B_{(1-\gamma) 2^{k+1}}$ and $\partial B_{\gamma 2^{k+1}}$. If $\gamma$ is smaller than a geometric constant, the coronas $A_{k+1}$ and $A_k$ have a large overlap. It then follows that $K$ consists of a single unbounded curve in $\cup_{k\geq 1} A_k = \mathbb R^2 \setminus B_{2\gamma}$. 
We can thus apply Theorem~\ref{t:Bonnet-David} to conclude that $(v,J)$ is a cracktip. If we denote by $z_0$ the starting point of the half-line $J$, the condition $\Omega^c (0, 1) \leq \varepsilon$ will imply that $z_0\in B_\gamma$, provided $\varepsilon$ is smaller than a suitably chosen positive $\varepsilon_0 (\gamma)$. This however contradicts \eqref{e:contradiction-crack} and hence completes the proof.
\end{proof}

\begin{proof}[Proof of Corollary~\ref{c:connectedness}]
The proof is entirely similar to that of Lemma~\ref{l:C1smooth}. Like in there we assume without loss of generality that $x=0$. We fix then $\delta$, $\varepsilon$, and $\gamma$ whose choice will be specified later. First of all $\varepsilon$ is assumed to be smaller than the $\varepsilon_0 (\gamma)$ given by Lemma~\ref{l:decay iteration 2}, so that the latter is applicable. We then let $N$ be the natural number given by the conclusion of the latter lemma. If we let $r_0\leq 1$ be such that 
\[
\lambda M_0^2 r_0^{\sfrac12} = \frac{\varepsilon}{2}\, , 
\]
we observe that, by choosing $\varepsilon$ sufficiently small, we can at the same time ensure that
\begin{itemize}
    \item[(1)] Lemma~\ref{l:decay iteration 2} is applicable in $B_{2 \bar r}$ with $\bar r = \min \{r, r_0\}$;
    \item[(2)] $K\cap (B_{2r}\setminus B_{\bar r/2})$ is a single $C^1$ arc $\gamma$ with endpoints in $\partial B_{\bar r/2}$ and $\partial B_{2r}$. 
\end{itemize}
Indeed, the first point is simply because, by choosing $\delta$ sufficiently small, the condition of the Lemma is satisfied with $x_{k,N}$ all equal to $0$, $k\in\{0,\ldots,N\}$ (note that $N$ and $\varepsilon$ are fixed at this point). As for the second point, it is just a consequence of the $\varepsilon$-regularity theory at pure jumps. 

We now note that, after having applied Lemma~\ref{l:decay iteration 2} to find $x_{N+1}$, we can actually apply it again to $B_{\bar r/2} (x_1)$, but this time the points $x_0, \ldots, x_N$ would be substituted by $x_1, \ldots , x_{N+1}$. We then proceed inductively to produce a sequence of points $x_k$ with 
\begin{align*}
&x_0 =0\\
&|x_{k+1}-x_k|\leq \gamma 2^{-k} \bar r\\
&\Omega^c (x_k, 2^{-k} \bar r) + \lambda \|g\|_\infty^2 (2^{-k} \bar r)^{\frac{1}{2}} \leq \varepsilon \, .
\end{align*}
Since $\{x_k\}$ is a Cauchy sequence, it has a limit $y$. Observe that by choosing $\gamma$ smaller than a geometric constant we can then ensure $y\in B_{\bar r}$. Moreover, it is easy to see that 
\[
\Omega^c (y, \rho) \leq 4 (\varepsilon+\gamma)  
\]
for every $\rho \leq \bar r$ simply by choosing $k$ so that $2^{-k-2} \bar r \leq \rho \leq 2^{-k-1} \bar r$ and comparing $\Omega^c (y, \rho)$ with $\Omega^c (x_k, 2^{1-k} \bar r)$,
if $\gamma\leq 1$.

In particular, if $\varepsilon$ and $\gamma$ are chosen sufficiently small, $K\cap (B_{2\rho/3} \setminus B_{\rho/3})$ consists of a single $C^1$ arc with endpoints in the respective circles $\partial B_{2\rho/3}$ and $\partial B_{\rho/3}$. This shows that $B_{2r}\cap K$ consists of a single continuous arc joining $y$ with a point in $\partial B_{2r}$, which moreover is $C^{1,\alpha}$ in $B_{2r}\setminus \{y\}$. But then by the Euler-Lagrange conditions in Proposition~\ref{p:variational identities 2} we conclude that the arc is $C^{1,1}$ in $B_{2r}\setminus \{y\}$.
\end{proof}

\section{An overview of the ideas in the proof of Theorem~\ref{t:Bonnet-David}}

Even though for some steps we give independent proofs, the overall strategy and the main ideas for proving Theorem~\ref{t:Bonnet-David} are all taken from the book \cite{bonnetdavid2001} of Bonnet and David. In this section we give an overview of the whole argument, which is very ingenious. Before coming to it, we first notice that in Section~\ref{s:monotonia-2} we prove another Liouville-type Theorem, namely Theorem~\ref{t:David-Leger}, which is due to David and L\'eger in the work \cite{DL02}, posterior to \cite{bonnetdavid2001}. However, one big advantage of having Theorem~\ref{t:David-Leger} at disposal is that it implies a series of useful corollaries (above all Corollary~\ref{c:non-terminal=regular}) which cut a lot of technicalities of \cite{bonnetdavid2001}. The proof of Theorem~\ref{t:David-Leger} are based on the monotonicity formulas in Propositions~\ref{p:monotonia-1} and \ref{p:monotonia-2}. 

\medskip

One main player in the proof of Theorem~\ref{t:Bonnet-David} is the ``harmonic conjugate'' \index{harmonic conjugate@harmonic conjugate} $v$, a function with the property that $\nabla v = \nabla u^\perp$ on $\mathbb R^2\setminus K$, and which can be shown (for a nonelementary global minimizer) to have a unique continuous extension to $K$. $v$ is obviously harmonic on $\mathbb R^2\setminus K$. Its existence and some preliminary properties are given in Section~\ref{s:coniugata}: one pivotal property is that $v$ is constant on each connected component of $K$.

Much of the technical work for the proof of Theorem~\ref{t:Bonnet-David} goes into describing the structure of the level sets of $v$. The sections~\ref{s:livelli-coniugata} and \ref{s:livelli-coniugata-2} will prove facts which are valid for all nonelementary global minimizers. The most important are that:
\begin{itemize}
    \item no level set of $v$ contains a loop (a particular case of this statement will actually be shown in the next section: $K$ itself has no loops);
    \item most of them do not have ``terminal points''. 
\end{itemize}    
Section~\ref{s:livelli-coniugata-3} then uses the additional assumption of Theorem~\ref{t:Bonnet-David} (namely that outside a sufficiently large ball $K$ consists of a single connected component) together with Bonnet's monotonicity formula (cf. Proposition~\ref{p:monotonia-0}) to infer that asymptotically at infinity $(u,K)$ is a cracktip. In particular any level set of $v$ has precisely ``two infinite ends'' outside of a sufficiently large disk. This information, combined with the previous analysis allows to conclude that:
\begin{itemize}
    \item up to a change of sign in $u$ the function $v$ achieves its absolute minimum $m_0$ exactly on the unbounded connected component of $K$;
    \item Most of the level sets of $v$ are just single (i.e. nonintersecting) unbounded curves.
\end{itemize}
The first information allows us to define the maximum $\bar m$ of $v$ on $K$, while another simple argument shows that if $\bar m = m_0$ then $K$ is connected and hence it is the cracktip.

\medskip

Arguing by contradiction that $(u,K)$ is not a cracktip we can then introduce the pivotal object in the final argument for Theorem~\ref{t:Bonnet-David}: a connected component $G$ of $K$ where the value of $v$ equals $\bar m$. With a variant of the Bonnet's monotonicity formula we can show that $K$ must have positive length and that no terminal point of $G$ can actually be contained in the closure of $\{v> \bar m\}$, cf. Section~\ref{s:G}.

The final argument comes now from looking at the trace of $u$ as we go ``around $G$'' (which will be shown to be continuous) and is explained in Section~\ref{s:gira} and it is particularly simple to explain when $G$ does not have triple points and it is thus, topologically, a segment with two terminal points. 

\medskip

\begin{figure}
\begin{tikzpicture}
\draw[very thick] (-3, 0) to (3,0);
\draw (-1.5,0) to (-1.5, 3);
\draw (1.5, 0) to (1.5, 3);
\node at (0, 1.5) {$\{v>\bar m\}$};
\node[below] at (-1.5,0) {$p^-$};
\node[below] at (1.5,0) {$p^+$};
\node[below] at (0,0) {$\sigma$};
\end{tikzpicture}
\caption{\label{f:G-semplice} The picture is a visualization of the final argument for Theorem~\ref{t:Bonnet-David} under the simplifying assumption that $G$ does not have triple junctions. The analysis in the Sections~\ref{s:livelli-coniugata}-\ref{s:G} implies that the global minimum and global maximum of the trace of $u$ on $G$ are assumed at two point $p^-$ and $p^+$ ``on the same side of $G$'': the picture shows in particular the level set $\{v= \bar m\}$ departing at those points and delimiting the upper level set $\{v > \bar m\}$, whose closure cannot contain the terminal points of $G$.
The range of the lower trace of $u$ on the segment $\sigma = [p^+,p^-]\subset G$ is thus necessarily contained in the range of the upper trace over the same segment $\sigma$: the intermediate value theorem provides then a point $q\in \sigma$ where the two traces have the same value.}
\end{figure}
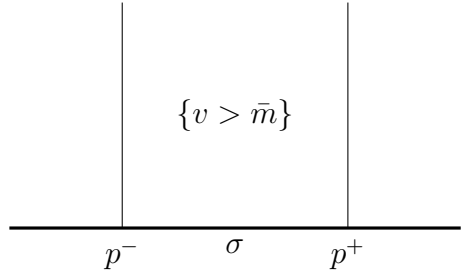

One outcome of the topological description of the level sets of $v$ is that while going around $G$ the trace of $u$ has exactly two local extrema, which are clearly the global minimum point and the global maximum point. This happens because to each local extremum of the trace of $u$ at $G$ corresponds to a distinct end of the level set $\{v=\bar m\}$ at infinity: having proved that such ends are  precisely 2, the trace must have precisely one local minimum and one local maximum. At the global maximum and the global minimum the level set $\{v=\bar m\}$ departs from $G$ as two infinite half-lines, which delimit one connected component of $\mathbb R^2$ where $v$ is above $\bar m$ and one connected component where $v$ is below $\bar m$. 

The other fundamental outcome of the previous analysis is that, since the terminal points of $G$ are not in the closure of $\{v>\bar m\}$, the global minimum and maximum points of the trace of $u$ on $G$ (which for simplicity we denote by $p^-$ and $p^+$) are on the ``same side'' of $G$: a schematic picture of what happens is given in Figure~\ref{f:G-semplice}.
$p^\pm$ delimit a segment $\sigma$ in $G$. Referring to the Figure~\ref{f:G-semplice}, the range of the trace of $u$ on the ``upper side'' of $\sigma$ is $[m,M]$, where $m$ denotes the global minimum and $M$ the global maximum. So the range of the trace in the lower side must be a segment $[m',M']$ strictly contained in $[m,M]$. By the intermediate value theorem there must then be a point $q\in \sigma$ where the upper and lower traces coincide. However this point must be a jump point (cf. Corollary~\ref{c:non-terminal=regular}) and at those points upper and lower traces must necessarily differ. This gives the desired contradiction.

\section{The absence of pockets}\label{s:pockets}

In this section we prove a simple fact valid for all concepts of minimizers of $E_0$, the absence of ``internal pockets''\index{pocket@pocket}, see Figure~\ref{f:pocket}.

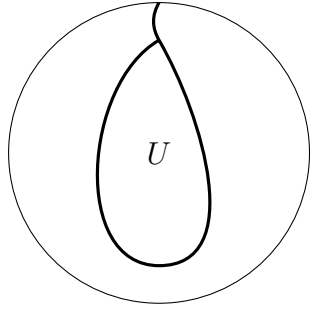
\begin{figure}
\begin{tikzpicture}
\draw (-6,0) circle [radius =2];
\draw[very thick] (-6, 2) to [out = 240, in = 120] (-6, 1.5) to [out = 300, in = 0] (-6, -1.5) to [out = 180, in = 210] (-6, 1.5);
\node at (-6,0) {$U$};
\end{tikzpicture}
\caption{\label{f:pocket}$U$ is a ``pocket'', namely a bounded connected component of $\Omega\setminus K$ which does not intersect the boundary $\partial \Omega$ of the domain of $(u,K)$.}
\end{figure}

\begin{lemma}\label{l:pockets}
Consider a minimizer $(u,K)$ of $E_0$ in $\Omega$. If $U$ is a connected component of $\Omega\setminus K$, then either it is unbounded, or its closure has to intersect $\partial \Omega$.
\end{lemma}
\begin{proof} If the statement were false, then there would be a connected component $U$ of $\Omega\setminus K$ with the property that $U\subset\subset \Omega$. In particular it turns out that $u$ must necessarily be constant on $U$. Moreover, for any given constant $c$, if we define the function
\[
u_c := \left\{
\begin{array}{ll}
u &\mbox{on $\Omega\setminus \overline{U}$}\\
c &\mbox{on $U$,}
\end{array}\right.
\]
then $(u_c,K)$ has the same energy has $(u,K)$ and it is not difficult to check that it must have the same minimizing property of $(u,K)$. Now, observe that $\mathcal{H}^1 (\partial U)>0$ and thus, by Corollary~\ref{c:a.e.regularity}, there is at least one regular jump point $x\in K$, i.e. there is a neighborhood of $x$ where $K$ is a $C^1$ arc. Assume without loss of generality that $x=0$ and that the tangent to $K$ at $0$ is
$\{(x_1, x_2): x_2 =0\}$. In particular, for a sufficiently small $\delta>0$ we have that
$K\cap [-\delta,\delta]^2 = {\rm gr}\, (f)$ for some $C^{1,\alpha}$ function $f:[-\delta, \delta]\to [-\delta, \delta]$ with $f(0) = f'(0)=0$. We next assume, without loss of generality, that 
\[
U\cap [-\delta, \delta]^2
=\{(t,s): s>f(t), s\in [-\delta, \delta]\}\, .
\]
At the same time we know that the restriction of $u$ to $[-\delta, \delta]^2\setminus {\rm gr}\, (f)$ coincide with two $C^{1,\alpha}$ functions $u^\pm$ which can in fact be extended $C^{1,\alpha}$ up to the boundary $K\cap [-\delta, \delta]^2$. $u^+$ in fact is a constant while, by a classical extension theorem, we assume that $u^-$ is extended $C^1$ to the full square $[-\delta, \delta]^2$.

Let now $c= u^-(0)$ and consider $(u_c, K)$. 
We now consider a new
function defined in the following way. Let $\chi: \mathbb R\to [0,1]$ be a smooth function which is identically equal to $1$ on $[1,\infty)$ and identically equal to $0$ on $(-\infty, -1]$ and 
let $\varphi\in C^\infty_c ((-1,1)^2)$ be a function which is identically $1$ on $(-\sfrac12,\sfrac12)^2$. Then we first define
\[
v (x) = c \chi \left(\frac{x_2}{\delta}\right) + u^- (x) \left(1-\chi \left(\frac{x_2}{\delta}\right)\right)\, ,
\]
which is a $C^1$ function on $[-\delta, \delta]^2$. Then on $U$ we define 
\[
w (x) = c \left(1- \varphi \left(\frac{x}{\delta}\right)\right) +
v (x) \varphi \left(\frac{x}{\delta}\right)\, ,
\]
while on the complement of $U\cup K$ we define 
\[
w (x) = u^- (x) \left(1- \varphi \left(\frac{x}{\delta}\right)\right) + v (x) \varphi \left(\frac{x}{\delta}\right)\, .
\]
Observe that the function $w$ coincides with a $C^1$ function on $ Q:= [-\frac{\delta}{2}, \frac{\delta}{2}]\times [-\delta, \delta]$ and we can thus consider $J:= K \setminus Q$, so that $(w, J)$ is a pair in the domain of $E_0$. Note that the pair is obviously a competitor for absolute minimizers. For restricted minimizers observe that we are not creating an additional connected component, while for generalized and generalized restricted minimizers it is easy to see that, if $V$ is a bounded open set such that $V\supset\supset U$, there are no pair of points belonging to distinct connected components of $\Omega\setminus K$ and lying outside $V$ which would belong to the same connected component of $\Omega\setminus J$. 
Finally, observe that $|\nabla w|\leq C$ in $Q$, for some constant $C$ independent 
of $\delta$. So, while
\[
\int_V |\nabla w|^2 \leq \int_V |\nabla u|^2 + C \delta^2
\]
we have
\[
\mathcal{H}^1 (J\cap V) \leq \mathcal{H}^1 (K\cap V) - \delta\, .
\]
For a sufficiently small $\delta$ we then contradict the minimizing property of $(u_c,K)$.
\end{proof}

\section{The David-L\'eger rigidity theorem and consequences}\label{s:monotonia-2}

In this section we prove the second part of Theorem~\ref{t:uniqueness}, which we recall here for the reader's convenience. 

\begin{theorem}\label{t:David-Leger}
Assume $(u, K, \{p_{kl}\})$ is a global generalized minimizer of $E_0$ and $K$ disconnects the plane. Then $(u,K, \{p_{kl}\})$ is either a pure jump or a triple junction.
\end{theorem}

Before coming to the proof of Theorem~\ref{t:David-Leger}, we register some very important corollaries.

\begin{corollary}\label{c:non-terminal=regular}
Statement (ii) of Theorem~\ref{t:structure} holds. 
\end{corollary}

\begin{corollary}\label{c:global-two-connected-components}
If $(u,K, \{p_{kl}\})$ is a global generalized minimizer of $E_0$, then $K$ cannot have two distinct unbounded connected components.
\end{corollary}

\begin{corollary}\label{c:no-vanishing}
If $(u, K, \{p_{kl}\})$ is a global generalized minimizer and $\nabla u$ vanishes on some open set, then $(u, K, \{p_{kl}\})$ is an elementary global minimizer.
\end{corollary}

We start by giving a proof of the theorem, which relies on Proposition~\ref{p:monotonia-2}. Since the latter has not yet been proved, we will follow with its argument, and hence we will prove the three corollaries above.

\begin{proof}[Proof of Theorem~\ref{t:David-Leger}]
Let $U_1$ and $U_2$ be two connected components of $\mathbb R^2\setminus K$. First of all observe that, by Lemma~\ref{l:pockets}, both $U_i$ are unbounded. In particular, for every $r$ sufficiently large $B_r\cap U_i\neq \emptyset$ for both $i$'s. Again by Lemma~\ref{l:pockets}, $\partial B_r \cap U_i\neq \emptyset$ for both $i$'s as well. It thus turns out that $K\cap \partial B_r$ has cardinality at least $2$ for all $r$ sufficiently large. If the cardinality is precisely $2$, then $\partial B_r \setminus K$ consists of two arcs $\gamma_1$ and $\gamma_2$ and it turns out that one of them belongs to $U_1$ while the other to $U_2$. But then $K\cap \partial B_r$ must belong to the same connected component of $K\cap \overline{B}_r$.

Thus, in any case for all $r$ sufficiently large we fall under the assumption of Proposition~\ref{p:monotonia-2}, which implies that the map $r\mapsto F(r)$ is monotone for $r$ sufficiently large. Let $F_0$ be its limit as $r\uparrow \infty$, and consider the ``blow-downs'' of $(u,K)$, namely any generalized global minimizer $(v,J, \{q_{kl}\})$ which is the limit of a sequence of rescalings $(u_{0,r_j}, K_{0, r_j})$ for some $r_j\uparrow \infty$. We then conclude that the function
\[
r\mapsto \frac{2}{r} \int_{B_r\setminus J} |\nabla v|^2 + \frac{\mathcal{H}^1 (B_r \cap J)}{r}
\]
has the constant value $F_0$. We will show below that:
\begin{itemize}
\item[(Cl)] $J$ disconnects $\mathbb R^2$, and there is no $r>0$ for which $B_r \setminus J$ belongs to the same connected component of $\mathbb R^2\setminus J$.
\end{itemize}
Assuming the claim, we can argue as above to conclude that one of the two assumptions (i) and (ii) of Proposition~\ref{p:monotonia-2} holds for a.e. $r>0$. From the second part of the proposition we infer that $(v,J)$ is an elementary global minimizer, with $J$ that is either a straight line passing through the origin, or the union of three half lines meeting at the origin at equal angles. However, in one case $F_0 =2$ and in the other $F_0=3$. We could now apply Corollary~\ref{c:unicita-blow-down} to conclude that $(u,K)$ itself is an elementary global minimizer.

In order to show (Cl), consider the set $V := {\rm Int}\, \overline{U}_1$ and the (local) Hausdorff limit $\tilde{J}$ of $\partial V_{0,r_j}$ (which exists up to subsequences). Clearly $\tilde{J}\subset J$. Fix moreover $\rho>0$ and observe that for every sufficiently large $j$ the set $B_{\rho/2} \setminus V_{0, r_j}$ is nonempty because $B_{\rho r_j/2} \cap U_2$ is not empty.
So, for each $\rho >0$ there is at least one point $p=p(\rho)\in \tilde{J}\cap B_r$. By Corollary~\ref{c:a.e.regularity}, a.e. such $p$ is a pure jump point. So there is a $\delta>0$ such that $\tilde{J}\cap B_{2\delta} (p)$ divides $B_{2\delta} (p)$ in two regions
$W^+$ and $W^-$. Again by Corollary~\ref{c:a.e.regularity}, for $j$ sufficiently large the same statement is correct for 
$K_{0, r_j} \cap B_\delta (p)$ and moreover $K_{0, r_j} \cap B_\delta (p)$ converges smoothly to $J\cap B_r$. Note that $K_{0, r_j}\cap B_\delta (p)$ must contain a point of $\partial V_{0, r_j}$ for all $j$ large enough. But then the smoothness implies that $K_{0, r_j}\cap B_\delta (p)\subset \partial V_{0,r_j}$. One of the two connected components of $K_{0, r_j}$ belongs then to $V_{0, r_j}$, while the other is contained in a different connected component of $\mathbb R^2\setminus K_{0,j}$: this is a consequence of the definition of $V$ as formed by the interior points of the closure of $U_1$. We thus infer that $W^+$ and $W^-$ cannot belong to the same connected component of $\mathbb R^2\setminus J$. Otherwise there would be an arc $\gamma$ with endpoints $p^\pm \in W^{\pm}$, which is at positive distance from $J$. By the Hausdorff convergence, such path would be contained in $\mathbb R^2\setminus K_{0, r_j}$ for all $j$ large enough, implying then that $B_\delta (p)\setminus K_{0,r_j}$ is contained in the same connected component of $\mathbb R^2\setminus K_{0, r_j}$. This is a contradiction and hence completes the proof.
\end{proof}

\subsection{Proof of Proposition~\ref{p:monotonia-2}} 
Without loss of generality we assume $x=0$, and as usual we will drop the base point in the notation of all the relevant quantities.

First of all we argue as in the proof of Proposition~\ref{p:monotonia-1} to conclude
\begin{equation}\label{e:F'-2}
r^2 F' (r) = 2 r \int_{\partial B_r\setminus K} |\nabla u|^2 + r \sum_{p\in \partial B_r \cap K} \frac{1}{e(p)\cdot n (p)} - 2 D(r) - \ell (r)\, 
\end{equation}
and
\begin{align}
r^2 F' (r) 
\geq & 3 r \int_{\partial B_r \setminus K} \left(\frac{\partial u}{\partial \tau}\right)^2 + r \int_{\partial B_r\setminus K} \left(\frac{\partial u}{\partial n}\right)^2
+ 2 r N(r) - 2 E_0 (u,K,B_r)\, ,\label{e:3-tau-bound-2}
\end{align}
Note however that \eqref{e:DLMS} can be used to derive also
\begin{align}
r^2 F' (r) 
\geq & r \int_{\partial B_r \setminus K} \left(\frac{\partial u}{\partial \tau}\right)^2 + 3r \int_{\partial B_r\setminus K} \left(\frac{\partial u}{\partial n}\right)^2 - 2D(r)\, ,\label{e:3-nu-bound}
\end{align}
In the rest of the proof we consider several competitors for $(u,K)$ in $B_r$, all constructed in the following fashion:
\begin{itemize}
    \item $(w,J)$ coincides with $(u,K)$ on $\Omega\setminus \overline{B}_r$;
    \item $J\cap \overline{B}_r$ is connected and consists of finitely many segments;
    \item $J\cap \partial B_r\supset K\cap \partial B_r$;
    \item $J\cap B_r$ partitions $B_r$ in finitely many connected components $\Omega_i$; 
    \item either $\mathcal{H}^1 (\overline{\Omega_i}\cap \partial B_r)=0$, and $w$ is defined to be constant on $\Omega_i$, or $\overline{\Omega}_i\cap \partial B_r$ is a closed arc $\beta_i$ which intersects only one connected component of $\partial B_r\setminus K$, and in that case we use Lemma~\ref{l:extension} to extend $u|_{\beta_i}$ to $\Omega_i$.
\end{itemize}
Since $J$ does not increase the number of connected components of $K$ and we can apply Lemma~\ref{l:ciao_componenti}, the pair $(w,J)$ is a valid competitor under all our assumptions and thus we can infer
\begin{equation}\label{e:competitore}
E_0 (u,K,B_r) \leq E_0 (w,J, B_r) \leq \mathcal{H}^1 (J\cap \overline{B}_r) +
\frac{\alpha (J)}{\pi} r \int_{\partial B_r \setminus K} \left(\frac{\partial u}{\partial \tau}\right)^2\, ,
\end{equation}
where $\alpha (J)$ is the least real number such that each $\Omega_i$ is contained in an {\em open} circular sector of angle $\alpha (J)$ centred in the origin. Observe that in the extreme case $\alpha (J)=2\pi$, in order to apply Lemma~\ref{l:ciao_componenti}, we need each $\Omega_i$ to be contained in a ``slit domain''. i.e. in $B_r\setminus [0,q_i]$ for some $q_i\in \partial B_r$.

In what follows, when invoking \eqref{e:competitore} we will then just need to specify $J\cap\overline{B}_r$ and estimate $\alpha (J)$ and $\mathcal{H}^1 (J\cap \overline{B}_r)$. Moreover, since $J$ coincides with $K$ outside the closed disk $\overline{B}_r$, with a slight abuse of notation we will just refer to $J$ for the piece inside the disk itself. 

We next focus on proving $F'(r) \geq 0$, leaving the discussion of the implications of $F'\equiv 0$ to the very end. Observe that we can always resort to Proposition~\ref{p:monotonia-2} to prove $F' (r) \geq 0$ when the largest arc of $\partial B_r\setminus K$ has length at most $\frac{3}{2} \pi r$. We therefore assume from now on that:
\begin{itemize}
    \item[(A)] $\partial B_r \cap K$ is contained in a subarc of $\partial B_r$ which has length no larger than $\frac{\pi r}{2}$.
\end{itemize}

\medskip

{\bf Case $N (r)\geq 4$.} By applying a rotation (and because of (A)) we can assume that $K\cap \partial B_r=\{p_1, \ldots, p_N\}$ are contained in the quadrant $\{x_1\geq 0, x_2\geq 0\}$. Without loss of generality assume $p_1 = (r,0)$ and order $p_2, \ldots, p_N$ counterclockwise, moreover denote by $q$ the point $(0,r)$. We then use \eqref{e:competitore} with $J=[0,p_1]\cup [p_1,p_2]\cup \ldots \cup [p_{N-1}, p_N]\cup [p_N,q]\cup [q,0]$ (cf. Figure~\ref{f:N>=4}). Observe that $\mathcal{H}^1 (J) < 2r+\frac{\pi r}{2} < 4r \leq N(r)r$ and that $\alpha (J) = \frac{3\pi}{2}$, so that 
\[
E_0 (u, K, B_r) < N (r) r + \frac{3}{2} \int_{\partial B_r \setminus K} \left(\frac{\partial u}{\partial \tau}\right)^2
\]
We then conclude from \eqref{e:3-tau-bound-2} that $r^2 F'(r)>0$.

\begin{figure}
\begin{tikzpicture}
\draw (0,0) circle [radius = 3];
\draw[very thick] (0,0) -- (0,3) -- (1,{sqrt(8)}) -- (2,{sqrt(5)}) -- (2.7,{sqrt(9-2.7^2)})-- (3,0) -- (0,0);
\node[right] at (3,0) {$p_1$};
\node[above] at (0,3) {$q$};
\node[above right] at (1,{sqrt(8)}) {$p_4$};
\node[above] at (2,{sqrt(5)}) {$p_3$};
\node[above right] at (2.7,{sqrt(9-2.7^2)}) {$p_2$};
\end{tikzpicture}
\caption{The set $J$ in an example where $N(r)=4$ and $K\cap \partial B_r$ lies in the first quadrant.\label{f:N>=4}}
\end{figure}
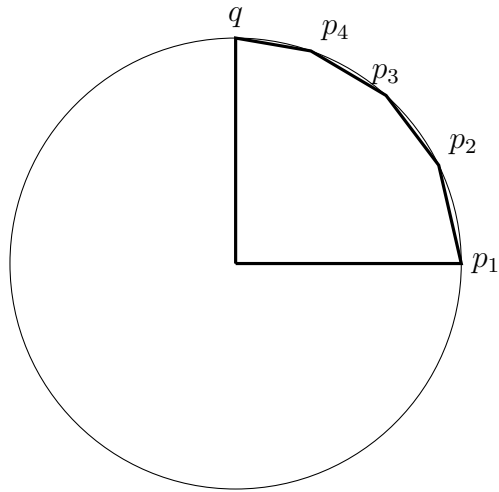

\medskip

{\bf Case $N(r)=3$.} Let $\{p_+, p_-, q\}= K \cap \partial B_r$ and once again we invoke (A) and the possibility of rotating the domain to assume that $p_\pm = ( c,\pm\sqrt{r^2-c^2})$ for some $r\frac{\sqrt{2}}{2}\leq c \leq r$ and that $q = (d,\sqrt{r^2-d^2})$ for some $0< c\leq d$, cf. Figure~\ref{f:N=3}. We now use \eqref{e:competitore} with two distinct $(w,J)$: the corresponding sets $J_1$ and $J_2$ are specified in the picture Figure~\ref{f:N=3} 

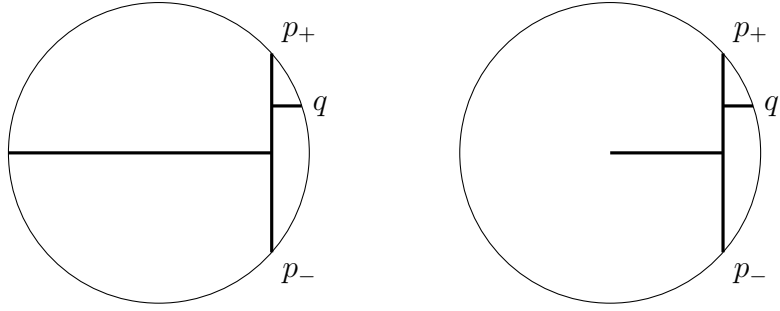
\begin{figure}
\begin{tikzpicture}
\draw (-3,0) circle [radius = 2];
\draw (3,0) circle [radius=2];
\draw[very thick] (-1.5,{sqrt(4-1.5^2)}) -- (-1.5,{-sqrt(4-1.5^2)});
\draw[very thick] (4.5,{sqrt(4-1.5^2)}) -- (4.5,{-sqrt(4-1.5^2)});
\draw[very thick] (-5,0) -- (-1.5,0);
\draw[very thick] (3,0) -- (4.5,0);
\draw[very thick] (-1.5, {sqrt(4-1.9^2)}) -- (-1.1,{sqrt(4-1.9^2)});
\draw[very thick] (4.5, {sqrt(4-1.9^2)}) -- (4.9,{sqrt(4-1.9^2)});
\node[above right] at ((-1.5,{sqrt(4-1.5^2)}) {$p_+$};
\node[right] at (-1.1,{sqrt(4-1.9^2)}) {$q$};
\node[below right] at (-1.5,{-sqrt(4-1.5^2)}) {$p_-$};
\node[above right] at ((4.5,{sqrt(4-1.5^2)}) {$p_+$};
\node[right] at (4.9,{sqrt(4-1.9^2)}) {$q$};
\node[below right] at (4.5,{-sqrt(4-1.5^2)}) {$p_-$};
\end{tikzpicture}
\caption{The sets $J_1$ and $J_2$ when $N(r)=3$.\label{f:N=3}}
\end{figure}

Observe that
\begin{align*}
\mathcal{H}^1 (J_1) & \leq 2r + \sqrt{2}r\\
\mathcal{H}^1 (J_2) &\leq r+\sqrt{2}r\\
\alpha (J_1) &=\pi\\
\alpha (J_2) &=2\pi\, .
\end{align*}
Apply now \eqref{e:competitore} with the two distinct competitors, average the inequalities and use \eqref{e:3-tau-bound-2} to conclude $F'(r)>0$.

\medskip

{\bf Case $N(r)=2$.} Let $\gamma_1$ and $\gamma_2$ be the two arcs delimited by $\partial B_r \cap K$ and again assume without loss of generality that $\omega := \frac{1}{\pi r}\max \{\ell (\gamma_i)\}\geq \frac{3}{2}$. Observe first that, by \eqref{e:outer},
\[
D (r) = \sum_i \int_{\gamma_i} u \frac{\partial u}{\partial n}\, .
\]
Next, let $c_i$ be the average of $u$ on $\gamma_i$ and recall Corollary~\ref{c:Bonnet} to estimate
\begin{align*}
D(r) = & \sum_i \int_{\gamma_i} (u-c_i) \frac{\partial u}{\partial n}
\leq \left(\sum_i \int_{\gamma_i} (u-c_i)^2\right)^{\frac{1}{2}} \left(\int_{\partial B_r\setminus K} \left(\frac{\partial u}{\partial n}\right)^2\right)^{\frac{1}{2}}\\
\leq & \omega r\underbrace{\left(\int_{\partial B_r\setminus K} \left(\frac{\partial u}{\partial \tau}\right)^2\right)^{\frac{1}{2}}}_{=: a}
\underbrace{\left(\int_{\partial B_r\setminus K} \left(\frac{\partial u}{\partial n} \right)^2\right)^{\frac{1}{2}}}_{=:b}\, .
\end{align*}
In particular, using \eqref{e:3-nu-bound} we conclude
\begin{equation}\label{e:poli-1}
r F' (r) \geq a^2 + 3b^2 - 2 \omega ab =: P_1 (\omega,a,b)\, .
\end{equation}
Next we use \eqref{e:3-tau-bound-2} with two competitors $J_1$ and $J_2$ constructed in the following fashion.
Without loss of generality we assume $\partial B_r = \{p_+, p_-\}$ where $p_\pm$ are, as in the previous step, symmetric with respect to the $x_1$-axis and in the quadrant $\{x_1>0, \pm x_2>0\}$, respectively. We then let $J_1$ be the connected set of minimal length joining the origin with the two points $p_\pm$ while, after setting $\bar c=(-r,0)$, we let $J_2 = J_1 \cup [\bar c,0]$, cf. Figure~\ref{f:N=2}. Observe that $J_1$ is easily determined as the union of the three segments joining $0, p_+$ and $p_-$ to the unique point $c$ where such segments form equal angles.

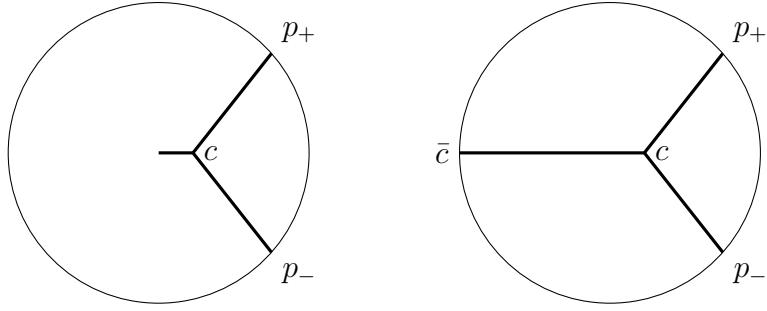
\begin{figure}
\begin{tikzpicture}
\draw (-3,0) circle [radius = 2];
\draw (3,0) circle [radius=2];
\node[above right] at ((-1.5,{sqrt(4-1.5^2)}) {$p_+$};
\node[below right] at (-1.5,{-sqrt(4-1.5^2)}) {$p_-$};
\draw[very thick] (-3,0) -- ({-3 + sqrt(4-1.5^2) -1.5/sqrt(3)},0);
\draw[very thick] (-1.5,{sqrt(4-1.5^2)}) -- ({-3 + sqrt(4-1.5^2) -1.5/sqrt(3)},0) -- (-1.5,{-sqrt(4-1.5^2)});
\node[right] at ({-3 + sqrt(4-1.5^2) -1.5/sqrt(3)},0) {$c$};
\draw[very thick] (1,0) -- ({3 + sqrt(4-1.5^2) -1.5/sqrt(3)},0);
\draw[very thick] (4.5,{sqrt(4-1.5^2)}) -- ({3 + sqrt(4-1.5^2) -1.5/sqrt(3)},0) -- (4.5,{-sqrt(4-1.5^2)});
\node[above right] at ((4.5,{sqrt(4-1.5^2)}) {$p_+$};
\node[below right] at (4.5,{-sqrt(4-1.5^2)}) {$p_-$};
\node[right] at ({3 + sqrt(4-1.5^2) -1.5/sqrt(3)},0) {$c$};
\node[left] at (1,0) {$\bar c$};
\end{tikzpicture}
\caption{The sets $J_1$ and $J_2$ when $N(r)=2$.\label{f:N=2}}
\end{figure}

Observe that $\alpha (J_1)=2\pi$ and $\alpha (J_2) =\pi$, while $\mathcal{H}^1 (J_2) = r +\mathcal{H}^1 (J_1)$. In order to apply \eqref{e:competitore}, we now wish to compute $\mathcal{H}^1 (J_1)$, relating its length to the angle $2\pi-\omega\pi$, which is twice the angle formed by $p_\pm$ with the $x_1$ axis. Thus we can explicitly calculate
\begin{align}
\mathcal{H}^1 (J_1) =& \mathcal{H}^1 ([0,c]) + 2\mathcal{H}^1 ([c,p_+])\nonumber\\ 
= & r\big(\cos (\pi (1-{\textstyle{\frac{\omega}{2}}}) - {\textstyle{\frac 1{\sqrt{3}}}}
\sin(\pi (1-{\textstyle{\frac{\omega}{2}}}))\big)
+ {\textstyle{\frac{4r}{\sqrt{3}}}}\sin(\pi (1-{\textstyle{\frac{\omega}{2}}}))\nonumber\\
= & 2 r \big({\textstyle{\frac{1}{2}}}\cos (\pi (1-{\textstyle{\frac{\omega}{2}}}))
+ {\textstyle{\frac{\sqrt{3}}{2}}}\sin (\pi (1-{\textstyle{\frac{\omega}{2}}}))\big)
= 2r \sin ({\textstyle{\frac{7\pi}{6}}-\frac{\omega \pi}{2}}) =: r f(\omega)\, .
\end{align}
Applying now \eqref{e:3-tau-bound-2} we thus get
\begin{align*}
r F' (r) &\geq -a^2 + b^2 + 4-2f(\omega)=: P_2 (\omega, a,b)\\
r F' (r) &\geq a^2 + b^2 + 2-2f(\omega)\, .
\end{align*}
Averaging between the two we then get
\begin{equation}\label{e:poli3}
r F'(r) \geq b^2 + 3-2f(\omega) =: P_3 (\omega, a, b)\, .
\end{equation}
Introduce now the function $h (\omega,a,b) =\max \{P_1, P_2, P_3\} (a,b, \omega)$ (recall that $P_1$ is defined in \eqref{e:poli-1}). We are thus reduced to show that $h\geq 0$ on the interval $[\frac{3}{2}, 2)$. First of all, $P_1 \geq 0$ if $\omega \leq \sqrt{3}$. Moreover, the function $f$ is clearly a decreasing function of $\omega$, which equals $\frac32$ at the point $\omega_0$ determined by
\[
\omega_0 := \frac{7}{3} - \frac{2}{\pi} \arcsin \frac{3}{4}\, . 
\]
Hence $P_3\geq 0$ for $\omega \geq \omega_0$. We thus have to show that $h$ is positive when $\omega\in I:=(\sqrt{3} , \omega_0)$. Observe that
\begin{itemize}
    \item $h\geq P_3 \geq 0$ unless
    \begin{equation}\label{e:condizione-1}
    b^2 < 2f(\omega) -3\, \, ;
    \end{equation}
    \item $h\geq P_2 \geq 0$ unless
    \begin{equation}\label{e:condizione-2}
    b^2 + 4-2 f(\omega) < a^2\, ;
    \end{equation}
    \item $h\geq P_1 \geq 0$ unless
    \begin{equation}\label{e:condizione-3}
    (\omega - \sqrt{\omega^2-3})^2 b^2< a^2 < (\omega + \sqrt{\omega^2-3})^2 b^2\, .
    \end{equation}
\end{itemize}
Now, \eqref{e:condizione-2} and \eqref{e:condizione-3} would imply
\[
4 - 2 f(\omega) < ((\omega + \sqrt{\omega^2-3})^2-1) b^2\, ,
\]
which combined with \eqref{e:condizione-1} would imply
\[
4 - 2 f(\omega) < ((\omega + \sqrt{\omega^2-3})^2-1) (2 f(\omega) -3))\, ,
\]
which in turn becomes
\begin{equation}\label{e:g1g2}
1 < (\omega + \sqrt{\omega^2-3})^2 (2 f(\omega) -3)\, .
\end{equation}
Recall that $I= (\sqrt{3}, \omega_0)$, and define 
\begin{align}
f_1 (\omega) &:=  (\omega + \sqrt{\omega^2-3})^2\, ,\\
f_2 (\omega) &:= 2 f(\omega) -3\, ,
\end{align}
we then just need to show that $f_1 f_2 \leq 1$ on $I$. Now, we already observed that $f_2$ is monotone decreasing, while it is easy to see that $f_1$ is monotone increasing. It can be explicitly computed that $f_2 (\sqrt{3}) < \frac{1}{4}$, which in turn implies $\sup_I f_2 < \frac{1}{4}$. Since it can be readily checked that $f_1 (\frac{7}{4}) = 4$, we conclude $\sup_{(\sqrt{3}, 7/4]} f_1 \leq 4$, in turn implying $\sup_{(\sqrt{3}, 7/4]} f_1 f_2 <1$. On the other hand it can be readily checked that $\omega_0 < 1.8$ and $f_1 (1.8) < 5.3$, which in turn implies $\sup_I f_1 < 5.3$. Since $f_2 (\frac{7}{4}) < 0.18$, we conclude $\max_{[7/4,\omega_0]} f_2 < 0.18$, from which we infer $\max_{[7/4, \omega_0]} f_1 f_2 < 0.18\cdot 5.3 < 1$.   

We also observe that in all these cases $F'(r)$ would actually result positive unless either $a$ and $b$ vanish.

\medskip

{\bf Case $N (r) =0$.} In this case we compare $(u,K)$ in $B_r$ with the harmonic extension $w$ of $u|_{\partial B_r}$. Recalling that the harmonic extension satisfies the estimate
\[
\int_{B_r} |\nabla w|^2 \leq r \int_{\partial B_r} \left(\frac{\partial u}{\partial \tau}\right)^2\, , 
\]
from \eqref{e:F'-2} we immediately conclude
\[
r^2 F' (r) \geq 2 \int_{\partial B_r} \left(\frac{\partial u}{\partial n}\right)^2\, .
\]

\medskip

{\bf $F$ constant.} Observe that if $F'(r)=0$ and we are in the assumptions of Proposition~\ref{p:monotonia-1} then 
\[
\int_{\partial B_r} \left(\frac{\partial u}{\partial n}\right)^2 = 0\, .
\]
However, the same conclusion can be drawn as well in all the cases examined above. But then the arguments of the final part of the proof of Proposition~\ref{p:monotonia-1} apply here as well and we reach the conclusion that, if $F$ is constant on $(0, r_0)$ and for a.a. $r\in (0,r_0)$ one of the assumptions (i) and (ii) holds, then $K\cap B_{r_0}$ coincides with one of three elementary global minimizers of Theorem~\ref{t:class-global}, with the additional information that, if $K$ is not empty, $0\in K$ and that, if $K$ is a triple junction, then $0$ is the point of junction. 

\subsection{Proof of Corollary~\ref{c:non-terminal=regular}} We assume without loss of generality that the nonterminal point \index{nonterminal point@nonterminal point} is the origin. By definition there is an injective continuous map $\gamma:[-1,1]\to K$ such that $\gamma(0) = 0$. Set 
\begin{equation}\label{e:def_raggio}
2r := \min \{|\gamma (-1)|, |\gamma (1)|\}
\end{equation}
and observe that there must be at least one negative $s_-$ and one positive $s_+$ such that $|\gamma (s_\pm)|= r$. We let $s_-$ be the largest negative number and $s_+$ the smallest positive number with the latter property. Using the Jordan curve theorem we then conclude that $B_r\setminus \gamma ((s_-, s_+))$ consists of two connected components $\Omega^\pm$. Consider now the limit $(u_\infty, K_\infty, \{p_{kl}\})$ of some sequence $\{(u_{0,r_k}, K_{0, r_k})\}$ with $r_k\downarrow 0$, as in Theorem~\ref{t:minimizers compactness}. If we can show that $K_\infty$ disconnects $\mathbb R^2$, then Theorem~\ref{t:David-Leger} would imply that $(u_\infty, K_\infty)$ is an elementary global minimizer and thus, by the cases (b) and (c) of Theorem~\ref{t:main} (resp Theorem~\ref{t:main-generalized}), $0$ would be a regular point, i.e. a point for which there is a disk $B_r$ in which $K$ is diffeomorphic either to a diameter, or to three radii joining at the origin.

In order to show that $K_\infty$ disconnects the plane, we consider the rescalings 
\begin{align*}
\Omega^\pm_k &:= \{x: r_k x \in \Omega^\pm\}\cap B_1\\
\Gamma_k &:= \{x: r_k x \in \gamma (s_-, s_+)\}\cap B_1\, .
\end{align*}
Up to extraction of a further subsequence, we can assume that $\Omega^\pm_k$ converge to some open sets $\Omega^\pm_\infty$ and $\Gamma_k$ converges (locally in the Hausdorff distance) to some $\Gamma\subset K_\infty$. Since $\partial \Omega^\pm_k \subset \Gamma_k\cup \partial B_1$, the convergence of the open sets means that 
\begin{itemize}
    \item[(i)] $|\Omega^+_\infty\Delta \Omega^+_k| + |\Omega^-_\infty \Delta \Omega^-_k| \to 0$;
    \item[(ii)] $x\in \Omega^\pm_\infty$ if and only if there is a $\rho>0$ and a $k_0\in \mathbb N$ such that $B_\rho (x)\subset \Omega^\pm_k$ for $k\geq k_0$.
\end{itemize}
Notice also that each $\Gamma_k$ contains at least one arc which connects the origin with $\partial B_1$. Hence $\Gamma$ as well contains at least one arc which connects the origin with $\partial B_1$ and in particular $\Gamma$ is not empty.

Pick now a point $q\in \Gamma$ which is a pure jump point (and whose existence is guaranteed by Corollary~\ref{c:a.e.regularity}). Then, for sufficiently small $\rho$, $B_\rho (q) \cap \Gamma$ is a $C^{1,\alpha}$ arc which divides $B_\rho (q)$ in two regions. On the other hand, by the regularity theory developed in the first part of the notes, for a sufficiently large $k$ we have that $K_{0,r_k}\cap B_\rho (q)$ consists also of a $C^{1,\alpha}$ arc which divides $B_\rho (q)$ into two regions. Since $\Gamma_k \to \Gamma$, for $k$ sufficiently large we must have that $K_{0,r_k}\cap B_\rho (q) = \Gamma_k \cap B_\rho (q)$. In particular $\Gamma_k$ divides $B_\rho (q)$ in two regions. By construction, one of the regions must be in $\Omega^+_k$ and the other must be in $\Omega^-_k$. It then turns out that $\liminf_k \min \{|\Omega^+_k|, |\Omega^-_k|\}>0$. Thus (i) above implies that both $\Omega^+_\infty$ and $\Omega^-_\infty$ are not empty. 

We now use the latter fact to show that $K_\infty$ disconnects $\mathbb R^2$. Consider indeed two points $q^\pm \in \Omega^\pm_\infty\setminus K_\infty$ whose existence is guaranteed by the fact that both $\Omega^\pm_\infty$ are not empty. If these points belonged to the same connected component of $\mathbb R^2\setminus K_\infty$, then there would be a continuous curve $\eta:[0,1] \to \mathbb R^2\setminus K_\infty$ with $\eta ([0,1])\cap K_\infty = \emptyset$, $\eta (0) = q^-$ and $\eta (1) = q^+$. For $k$ sufficiently large we have
\begin{itemize}
    \item $\eta ([0,1])$ does not intersect $K_{0,r_k}$ (by local Hausdorff convergence of $K_{0,r_k}$ to $K_\infty$);
    \item $q^\pm\in \Omega^\pm_k$ (by condition (ii) above);
    \item $\eta ([0,1])]\subset B_{r/r_k}$, where $r$ is the radius on \eqref{e:def_raggio}.
\end{itemize}    
Fix now a $k$ sufficiently large so that all the conclusions above apply. Then:
\begin{itemize}
    \item $p^-:= r_k q^- \in \Omega^+$, $p^+:= r_k q^+\in \Omega^-$;
    \item $t\mapsto r_k \eta (s)$ is a continuous curve contained in $B_r$ which does not intersect $K$ and hence does not intersect $\gamma ((s^-,s^+))$. 
\end{itemize}
However, since $\Omega^\pm$ are the two distinct connected components of $B_r\setminus \gamma ((s^-,s^+))$, we have reached a contradiction.

\subsection{Proof of Corollary~\ref{c:global-two-connected-components}} Assume by contradiction that $(u, K)$ is a global generalized minimizer such that $K$ has two unbounded connected components $K_1$ and $K_2$. Without loss of generality we can assume that both components intersect the unit disk $B_1$ and we fix two points $p_1\in K_1\cap B_1, p_2\in K_2 \cap B_1$.

Fix now $k\in \mathbb N\setminus \{1,2\}$ and consider the class of Lipschitz curves 
\[
\{\beta: [0, \ell (\beta)]\to K_1 :|\dot\beta|=1, \beta (0)=p_1, |\beta (\ell (\beta))|\geq k^2\}\, .
\]
By Lemma~\ref{l:connessi_per_archi} such class is nonempty.
It is easy to see that there is a curve $\beta$ in this class whose domain of definition has minimal length (the argument is the same as in Step 3 of the proof of Lemma~\ref{l:connessi_per_archi}). Such curve must then be injective and $|\beta (\ell (\beta))| = k^2$, otherwise it would not have the minimizing property just described. We then consider the largest $s$ such that $|\beta (s)|=1$. The arc $\gamma_k : [s, \ell (\beta)] \ni \sigma \mapsto \beta (\sigma)\in K_1$ is thus an injective arc connecting a point in $\partial B_1$ to a point in $\partial B_{k^2}$. We find likewise an injective arc $\eta_k$ in $K_2$ with the same properties. By the Jordan curve theorem the two paths $\gamma_k\cup \eta_k$ subdivide $B_{k^2}\setminus \overline{B}_1$ into two connected components $U^\pm_k$. 

We consider now the rescaled pairs $(u_{0, k}, K_{0,k})$ and the corresponding domains $\Omega_k^\pm \subset B_k\setminus \overline{B}_{1/k}$, resulting from appropriately scaling the domains $U^\pm_k$.  
As in the proof of Corollary~\ref{c:non-terminal=regular}, up to subsequences we can assume that $(u_{0, k}, K_{0,k})$ converge to a global generalized minimizer $(u_\infty,K_\infty)$, that the paths $\overline{\eta}_k\cup \overline{\gamma}_k$ converge locally in the Hausdorff topology to a connected closed set $\Gamma\subset K_\infty$ and that the sets $\Omega_k^\pm \cap B_1$ converge to open sets $\Omega^\pm_\infty$. As above $\Gamma\cap B_1$ is not empty and we can pick a point $q\in \Gamma\cap B_1$ which is a pure jump point for $K_\infty$. Again the regularity theory will imply that there is a $B_\rho (q)$ and a $k_0$ such that for all $k\geq k_0$ $K_{0,k}\cap B_\rho (q)$ is a smooth arc subdividing $B_\rho (q)$ into two open subsets which have roughly the same area. Such arc is then either a subset of the rescaled $\eta_k$ or a subset of the rescaled $\gamma_k$ and in both cases we conclude that one the two regions in which $B_\rho (q)$ is subdivided by $\Gamma$ is a subset of $\Omega^+_\infty$, while the other is a subset of $\Omega^-_\infty$. 

Similarly to the proof of Corollary~\ref{c:non-terminal=regular} we argue then that any pair of arbitry points $q^\pm \in \Omega^\pm_\infty\setminus K_\infty$ must belong to different connected components of $\mathbb R^2\setminus K_\infty$. This in turn allows us to apply Theorem~\ref{t:David-Leger} and conclude that $(u_\infty, K_\infty)$ is an elementary global generalized minimizer. But then Corollary~\ref{c:unicita-blow-down} would imply that $(u,K)$ was itself an elementary global generalized minimizer. In that case $K$ must be connected, contradicting the initial assumption that $K$ contains at least two distinct connected components.

\subsection{Proof of Corollary~\ref{c:no-vanishing}} Assume $(u, K, \{p_{kl}\})$ is a global generalized minimizer and that $\nabla u$ vanishes on some open set. If $K$ disconnects $\mathbb R^2$ then we know from Theorem~\ref{t:David-Leger} that the minimizer is elementary. We can thus assume that $\mathbb R^2\setminus K$ is connected. But then, by the unique continuation for harmonic functions we conclude that $\nabla u =0$ on the whole $\mathbb R^2\setminus K$ and in particular that $u$ is constant. This however is only possible if $K$ is empty. 

\section{The harmonic conjugate}\label{s:coniugata}

This and the next remaining four sections of the chapter will be dedicated to prove Theorem~\ref{t:Bonnet-David}. A fundamental role in the proof will be played by the harmonic conjugate \index{harmonic conjugate@harmonic conjugate} $v$ of the function $u$ in a minimizing pair $(u,K)$, which will be introduced in this section. $v$ is, locally in $\mathbb R^2\setminus K$, a classical harmonic conjugate of $u$, i.e. a function $v$ whose gradient $\nabla v$ is the counterclockwise rotation of $\nabla u$ by 
90 degrees (which we will denote by $\nabla u^\perp$), or alternatively such that the map $\zeta (x+iy) = u (x,y) + i v (x,y)$ is holomorphic. The main point is that $v$ exists globally on $\mathbb R^2\setminus K$, it is unique up to addition of a constant (if the minimizer is nonelementary), and it can be uniquely extended to a H\"older continuous function on the whole plane.\index{harmonic conjugate@harmonic conjugate}\index{conjugate, harmonic@conjugate, harmonic}\index[simb]{aalDu^p@$\nabla u^\perp$}

\begin{proposition}\label{p:armonica-coniugata}
Let $(u,K)$ be a nonelementary global minimizer of $E_0$. Then there is a function $v:\mathbb R^2\to \mathbb R$ with the following properties:
\begin{itemize}
    \item[(i)] $v$ is harmonic and smooth on $\mathbb R^2\setminus K$, where $\nabla v = \nabla u^\perp$;
    \item[(ii)] $v\in W^{1,2}_{{\rm loc}} (\mathbb R^2)$, it is H\"older continuous with exponent $\frac{1}{2}$ and the H\"older seminorm is globally bounded, i.e.
    \[
    \sup_{x\neq y} \frac{|v(x)-v(y)|}{|x-y|^{\frac{1}{2}}} < \infty\, ;
    \]
    \item[(iii)] $v$ is constant on each connected component of $K$;
    \item[(iv)] $v$ is unique up to addition of a constant.
\end{itemize}
\end{proposition}

Before coming to the proof of the proposition we observe a simple consequence of Corollary~\ref{c:non-terminal=regular}, which will be useful in other occasions.

\begin{lemma}\label{l:smooth connections}
Let $(u,K)$ be a restricted, absolute, or generalized minimizer of $E_\lambda$. Let $x,y$ be two nonterminal points in the same connected component of $K$. There is then an injective Lipschitz curve 
$\gamma:[0,1]\to K$ with $\gamma (0) =x$ and $\gamma (1)=y$ and a partition $0= s_0 <s_1 < \ldots <s_{N+1}=1$ such that:
\begin{itemize}
    \item[(i)] $\gamma|_{[s_i, s_{i+1}]}\in C^{1,\alpha}$;
    \item[(ii)] Each $\gamma (s_i)$ with $i\in \{1, \ldots , N\}$ is a triple junction, while $\gamma (t)$ is a pure jump for every $t\neq s_i$.\footnote{Observe that $0$ and $1$ can be triple junctions or pure jumps.} 
\end{itemize}
In particular, if $K'\subset K$ is a connected component, the subset of nonterminal points of $K'$ is still arc connected.
\end{lemma}

\begin{proof}
Consider $\gamma$ as in Lemma~\ref{l:connessi_per_archi}. Since it is injective, $\gamma (s)$ is a nonterminal point for every $s\in (0,1)$. On the other hand $\gamma (0)$ and $\gamma (1)$ are nonterminal points by assumption. In particular, by the regularity theory developed so far, for each $s\in [0,1]$ there is a neighborhood $U$ of $s$ such that $\gamma (U\setminus \{s\})$ consist of pure-jump points. It thus turns out that there are at most finitely many $s$'s such that $\gamma (s)$ is a triple junction. If we then parametrize $\gamma$ with constant speed, assertion (i) follows from the regularity theory.
\end{proof}

\begin{proof}[Proof of Proposition~\ref{p:armonica-coniugata}]
Consider the $L^2$ vector field $\nabla u^\perp$. Observe that \eqref{e:outer} implies that the distributional curl of $\nabla u^\perp$ vanishes. In particular $\nabla u^\perp$ must have a potential $v\in W^{1,2}_{{\rm loc}} (\mathbb R^2)$. A simple direct proof of the latter claim can be found by convolving with a standard family of mollifiers $\varphi_\varepsilon$: clearly ${\rm curl} (\nabla u^\perp)*\varphi_\varepsilon$ is a smooth curl-free vector field by \eqref{e:outer} and as such it has a potential $v_\varepsilon$ which we can normalize to $v_\varepsilon (0) =0$. We can then use the compact embedding of $W^{1,2}$ to extract a sequence $\varepsilon_k\downarrow 0$ such that $v_{\varepsilon_k}$ has a local weak limit in $W^{1,2}$. Claim (i) is then an obvious consequence of the smoothness of $v$ in $\mathbb R^2\setminus K$, while the uniqueness up to constant is obvious because, by Theorem~\ref{t:David-Leger}, $\mathbb R^2\setminus K$ is connected. Observe next that
\[
\int_{B_r (x)} |\nabla v|^2 = \int_{B_r (x)\setminus K} |\nabla u|^2 \leq 2\pi r
\]
for every disk $B_r (x)$. Hence (ii) follows from the standard Morrey's estimate.

Finally, in order to prove (iii) fix a connected $K'\subset K$. Since by Corollary~\ref{c:a.e.regularity} pure jump points (which clearly are nonterminal) are dense, it suffices to prove that $v (x) = v(y)$ for every 
pair of pure jump points $x,y\in K'$, which from now on we assume to be fixed. Let $\gamma$ be a curve as in Lemma~\ref{l:smooth connections} and observe that the corresponding partition must satisfy $0<s_1< s_2< \ldots < s_N <1$. Observe moreover that $v|_K$ must be constant in the neighborhood of every pure jump point by the regularity theory developed thus far: indeed at each jump point we have that $u$ has $C^1$ extensions $u^+$ and $u^-$ on both sides of $K$, while the Euler-Lagrange conditions for $u$ imply $\frac{\partial u^\pm}{\partial \nu}=0$. This in turn implies that $v$ as well has $C^1$ extensions $v^+$ and $v^-$ on both sides of $K$, and since $\nabla v$ is a counterclockwise 90 degree rotation of $\nabla u$, we conclude that both $\nabla v^+$ and $\nabla v^-$ are orthogonal to $K$. Hence $s\mapsto v (\gamma (s))$ is constant on the arcs $[0, s_1), (s_1, s_2),\ldots , (s_{N-1}, s_N), (s_N, 1]$. By continuity we conclude that $v\circ \gamma$ is constant over the whole interval $[0,1]$, which in turn implies $v (x)= v (\gamma (0)) = v (\gamma (1))=y$ as desired.
\end{proof}

\section{The level sets of the harmonic conjugate: Part I}\label{s:livelli-coniugata}

In this section we study the level sets of the harmonic conjugate introduced in Proposition~\ref{p:armonica-coniugata}. Our main conclusion is Proposition~\ref{p:struttura-1} below. Observe that for both $v$ and $u$ we have a quite good description of their behavior at every point of $K$ which is either a pure jump point or a triple junction. We thus wish to first gain some more information on points which belong to the remaining set $K^\sharp$, i.e. points which either turn out to form a connected component of $K$ by themselves or which are terminal points of a connected component with positive length (cf. Definition~\ref{d:internal points}). At all such points we claim that $u$ can be extended continuously. 

\begin{proposition}\label{p:continuity}
There is a universal constant $C$ with the following property. Let $(u,K)$ be a nonelementary global minimizer and assume that $x\in K^\sharp$ (i.e. that $x\in K$ is neither a triple junction nor a pure jump point). Then there is a $\bar{u}\in \mathbb R$ such that
\begin{equation}\label{e:Hoelder-1/2}
|\bar u - u(y)|\leq C |x-y|^{\frac{1}{2}} \qquad \forall y\not \in K\, .
\end{equation}
\end{proposition}

In particular there is a unique continuous extension of $u$ to $K^\sharp$ and from now on we will just use the same notation $u$ for such extension. We are now ready to state the following structural proposition.

\begin{proposition}\label{p:struttura-1}
Let $(u,K)$ be a nonelementary global minimizer and let $v$ be as in Proposition~\ref{p:armonica-coniugata}. Then for a.e. $m\in \mathbb R$ we have the following properties:
\begin{itemize}
    \item[(i)] $u$ is continuous at every point of $\{v=m\}$;
    \item[(ii)] $\mathcal{H}^{\frac12} (K\cap \{v=m\})=0$;
    \item[(iii)] $\{v=m\}\setminus K$ is the union of countably many real analytic arcs;
    \item[(iv)] $\{v=m\}$ has locally finite $\mathcal{H}^1$ measure;
    \item[(v)] for every injective Lipschitz curve $\gamma: [0,1]\to \{v=m\}$ the function $u\circ \gamma$ is absolutely continuous, $u$ is differentiable at $\gamma (s)$ for a.e. $s$ such that $\dot\gamma (s)\neq 0$ and
    \begin{equation}\label{e:derivata}
    \frac{d}{dt} (u\circ \gamma) = ((\nabla u) \circ \gamma) \cdot \dot\gamma
    \end{equation}
    where the right hand side is defined as $0$ at every point $s$ where $\dot \gamma (s) =0$, irrespectively of whether $u$ is differentiable or not at that point.
\end{itemize}
\end{proposition}

\begin{remark} Note that, as a consequence of (v) we have the integral identity
\begin{equation}\label{e:integrata}
u (\gamma (1)) - u (\gamma (0)) = \int_0^1 \nabla u (\gamma (t))\cdot \dot{\gamma} (t)\, dt\, .
\end{equation}
\end{remark}

\subsection{Proof of Proposition~\ref{p:continuity}} Without loss of generality assume that $x=0$. Given an open set $U$ we define 
\[
{\rm osc}\, (u, U) := \sup \{|u(x)- u(y)|: x,y\in U\setminus K\}\, .
\]
We just need to show that there is a universal constant $C$ such that ${\rm osc}\, (u, B_r (0)) \leq C r^{\frac{1}{2}}$
and then we can simply define
\[
\bar{u} := \lim_{y\not\in K, y\to 0} u (x)\, 
\]
to conclude \eqref{e:Hoelder-1/2}.
Since we are claiming that the constant $C$ in our estimate is independent of $(u,K)$, we can assume by scaling that $r=1$ and we are thus reduced to show
\begin{equation}
\sup \{|u(x)-u(y)|: x,y\in B_1\setminus K\} \leq C\, .
\end{equation}
We will accomplish the latter estimate in three steps.

\medskip

{\bf Step 1} Firs of all consider $F:= \{r\in (1,2): \sharp (K\cap \partial B_r) <\infty\}$ and for each $r\in F$ we define
\[
d (K, r):= \min \{|x-y|: x\neq y\;\mbox{and}\; x,y\in K \cap \partial B_r\}\, . 
\]
We then claim that there is a $k\in \mathbb N$ such that 
\begin{equation}\label{e:spaziati}
|\{r\in F: d(K, r)> k^{-1}\}|\geq 2 k^{-1}\, .
\end{equation}
Assume indeed that the statement is false for every $k$ and let $(u_k,K_k)$ be a counterexample for $k$. Hence a subsequence $(u_{k_j}, K_{k_j})$ converges to a global generalized minimizer $(u_\infty, K_\infty)$. By the coarea formula \cite[Theorem 2.93]{AFP00} there must be at least one radius $r\in (5/4,7/4)$ such that $\partial B_r \cap K_\infty$ consists of finitely many pure jump points $x_1, \ldots, x_N$, intersecting $\partial B_r$ transversally. But then there is $\delta>0$ such that $(B_{r+2\delta}\setminus \overline{B}_{r-2\delta})\cap K_\infty$ consists of $N$ arcs, intersecting all circles at an angle no smaller than $2\delta$. In turn the regularity theory developed in the first part implies the existence of a positive $\eta$ such that, for sufficiently large $j$, $d(K_{k_j}, \rho)> \eta$ for all $\rho\in (r-\eta, r+\eta)$, which is a contradiction. 

Fix now a $k\in \mathbb N$ such that \eqref{e:spaziati} holds and set 
\[
\mathcal{R}:= \left\{r\in F: d(K, r)> k^{-1}\quad \mbox{and}\quad \int_{\partial B_r} |\nabla u|^2 \leq 8\pi k\right\}\, .
\]
Observe that $|\mathcal{R}|\geq k^{-1}$

For each $r\in \mathcal{R}$ we let $I_1 (r), \ldots , I_{N (r)} (r)$ be the connected components of $\partial B_r\setminus K$. Observe that the length of each $I_j (r)$ is at least $\frac{1}{k}$. For each $I_j (r)$ let
\[
\delta_j (r) := \max\{{\rm dist}\, (x,K): x\in I_j (r)\}
\]
and
\[
\delta (r):= \min_j \delta_j (r)\, .
\]
We claim that there is some $M\in \mathbb N$ universal constant such that $\delta (r)> M^{-1}$ for at least one $r\in \mathcal{R}$. Assume the contrary and for each $M>0$ let $(u_M, K_M)$ be such that $\delta (r) \leq \frac{1}{M}$ for every $r$ in the set
\[
\mathcal{R}_M :=\{r: d(K_M,r)> k^{-1}\}\, .
\]
Let then $(u_{M_j}, K_{M_j})$ be a sequence which converges to $(u_\infty, K_\infty)$. Let $\mathcal{R}_\infty$ be the set of all $\rho\in (1,2)$ which are cluster points of a sequence $\{r_j\}$ with $r_j\in \mathcal{R}_{M_j}$. Then $\mathcal{R}_\infty$ cannot be finite. Fix now $\rho\in \mathcal{R}_\infty$. We then easily conclude that $\partial B_\rho\cap K_\infty$ contains at least one interval of length larger than $\frac{1}{k}$. However this would imply that $\mathcal{H}^1 (K_\infty\cap B_2) = \infty$.

\medskip

{\bf Step 2} Let now $\Omega (K, \delta, R):= \{x\in B_R: {\rm dist}\, (x,K)>\delta\}$. We then claim that for every $\delta>0$ there is $L\in \mathbb N$ such that $\Omega (K, \delta, 2)$ is contained in the same connected component of $\Omega (K, \frac{1}{L}, L)$. Indeed, assume the claim is false and for every $L$ let $(u_L, K_L)$ be a counterexample. Extract a subsequence $(u_{L_j}, K_{L_j})$ which is converging to a global generalized minimizer $(u_\infty, K_\infty)$. Note that $K_\infty$ must then disconnect $\mathbb R^2$ and thus, by Theorem~\ref{t:David-Leger}, $(u_\infty, K_\infty)$ is a generalized global minimizer. Recall that by assumption $0\in K_L$ for every $L\in \mathbb N$ and hence $0\in K_\infty$. In turn, by the regularity theory $0\in K_{L_j}$ must be a triple junction or a pure jump point, while we are assuming that $0\in K^\sharp$.

\medskip

{\bf Step 3} Fix now $M$ and $k$ as in Step 1 and let $r\in \mathcal{R}$ be such that 
$\delta(r)>M^{-1}$.
Let $L$ be as in Step 2. Let 
$I_i$, $i\in\{1,\ldots, N\}$, be the connected components of $\partial B_r\setminus K$, and for each $i$ select $x_i\in I_i$ such that $B_{M^{-1}} (x_i)\subset \mathbb R^2\setminus K$. Fix now $\ell$ and consider that $x_\ell$ and $x_{\ell+1}$ are in the same connected component of $\Omega (K, \frac1L, L)$, which we denote by $\Omega'$. Consider a maximal subset of points $\mathscr{S} = \{y_j\}\subset \Omega'$ with the property that $|y_i - y_j|\geq \frac{1}{8L}$ for each distinct pair $y_i, y_j$. We then have that $\{B_{(4L)^{-1}} (y_j)\}$ covers $\Omega'$. At the same time the cardinality of $\mathscr{S}$ is bounded by a constant since $B_{(16L)^{-1}} (y_j)$ are pairwise disjoint. 

A chain of $\mathscr{S}$ is given by a choice of balls $\{B_{(4L)^{-1}} (y_{i(j)})\}_{j\in \{1, \ldots N\}}$ where the $i(j)$ are all distinct and $B_{(4L)^{-1}} (y_{i(j)})\cap B_{(4L)^{-1}} (y_{i(j+1)}) \neq \emptyset$. We say that $B_{(4L)^{-1}} (y_K)$ and $B_{(4L)^{-1}} (x_J)$ are chain-connected if there is a chain such that $I = i(1)$ 
and $J = i(N)$. Assume now, without loss of generality, that $B_{(4L)^{-1}} (y_1)$ contains $x_\ell$ and let $\mathscr{C}\subset \mathscr{S}$ be the subset of points such that $B_{(4L)^{-1}} (y_j)$ is chain-connected to $B_{(4L)^{-1}} (y_1)$. $\mathscr{C}$ must coincide with $\mathscr{S}$ otherwise the two open sets 
\begin{align}
U &:= \bigcup_{i\in \mathscr{C}} B_{(4L)^{-1}} (y_i)\\
V &:= \bigcup_{i\in \mathscr{S}\setminus \mathscr{C}} B_{(4L)^{-1}} (y_i)
\end{align}
would be disjoint and would disconnect 
$\Omega^\prime$. 
Upon reindexing our balls we can thus assume that $\{B_{(4L)^{-1}} (y_i)\}_{i\in \{1, \ldots , N\}}$ is a chain such that $x_\ell\in B_{(4L)^{-1}} (y_1)$ and $x_{\ell+1}\in B_{(4L)^{-1}} (y_N)$. Set $z_0=x_\ell$, $z_{N} = x_{\ell+1}$, and choose $z_i \in B_{(4L)^{-1}} (y_i)\cap B_{(4L)^{-1}} (y_{i+1})$ for every other $1\leq i\leq N-1$. Consider then the piecewise linear curve $\gamma$ consisting of joining the segments $[z_i, z_{i+1}]$. Since $|z_{i+1}-z_i|\leq 1/(2L)$, the length of the curve is bounded by a universal constant. Furthermore, each point $z$ on the curve is at distance at least $1/{2L}$ from $K$. In turn this 
and the energy upper bound in \eqref{e:upper bound} imply
\[
\int_{B_{1/(2L)} (z)} |\nabla u|^2 \leq \frac{\pi}{L}\, .  
\]
By the mean-value property of harmonic functions, we conclude that $|\nabla u|$ is bounded by a universal constant on the curve $\gamma$. This then implies that $|u (x_\ell) - u (x_{\ell+1})|$ is bounded by a universal constant too.

Next recall that 
\[
\int_{\partial B_r} |\nabla u|^2 \leq 8\pi k\, .
\]
Hence, by Morrey's embedding, we conclude that ${\rm osc}\, (u, \partial B_r\setminus K)$ is bounded by a universal constant. At this point the maximum principle of Lemma~\ref{l:maximum} implies that ${\rm osc}\, (u, B_r\setminus K)$ is bounded as well by the same constant, concluding the proof.

\subsection{Proof of Proposition~\ref{p:struttura-1}} Consider first the union $\tilde{K}$ of the connected components of $K$ with positive length. Since $K\setminus \tilde{K}$ consists of irregular points, $\mathcal{H}^1 (K\setminus \tilde{K})=0$ and thus Proposition~\ref{p:Jonas} implies that $\mathcal{H}^{\frac12} ((K\setminus \tilde{K})\cap \{v=t\}) =0$ for a.e. $t$
since $v\in C^{\frac12}$  by (ii) in Proposition~\ref{p:armonica-coniugata}. 
On the other hand, $v (\tilde{K})$ is a countable set, because the connected components of $K$ with positive length are countaly many, and on each such component the function $v$ is constant.
Thus, except for a countable values of $t$'s we have $\tilde{K}\cap \{v=t\}=\emptyset$. In particular, (ii) follows. Moreover, Proposition~\ref{p:continuity} implies that $u$ is continuous at every point $x\in \{v=t\}$ if $\{v=t\}$ does not intersect $\tilde{K}$, hence giving (i). 

The fact that $\{v=t\}\setminus K$ consists of a locally finite union of real analytic arcs is a direct consequence of the harmonicity of $u$ on $\mathbb R^2\setminus K$. Next, for every $N\in \mathbb N$ consider that by the coarea formula \cite[Theorem 2.93]{AFP00} it is true that
\[
\int \mathcal{H}^1 (\{v=t\}\cap (B_N\setminus K))\, dt =
\int_{B_N\setminus K} |\nabla v| = \int_{B_N\setminus K} |\nabla u|<\infty\, .
\]
Therefore for a.e. $t$ we have that $\mathcal{H}^1 (\{v=t\}\cap (B_N\setminus K)) < \infty$ for every $N\in \mathbb N$, thus implying (iv).

Again by the coarea formula \cite[Theorem 2.93]{AFP00} we know that
\[
\int \int_{B_N \cap \{v=t\}\setminus K} |\nabla u| d\mathcal{H}^1 \, dt
= \int_{B_N\setminus K} |\nabla v| |\nabla u| = \int_{B_N\setminus K} |\nabla u|^2<\infty\, .
\]
We thus infer that for a.e. $t$ 
\begin{equation}\label{e:nabla u in L^1}
\int_{B_R\cap\{v=t\}\setminus K} |\nabla u| d\mathcal{H}^1 < \infty \qquad \forall R>0\, .
\end{equation}
Fix now a $t$ which satisfies (i), (ii), and \eqref{e:nabla u in L^1} and let $\gamma: [0,1]\to \{v=t\}$ be an injective Lipschitz curve. We want to show that $u\circ \gamma \in W^{1,1}$. In order to do that we set
\[
h(s) := \left\{
\begin{array}{ll}
|\nabla u (\gamma (s))||\dot\gamma (s)| \quad &\mbox{if $\gamma (s)\not \in K$}\\
0 &\mbox{otherwise.}
\end{array}\right.
\]
Observe that, by injectivity of $\gamma$ and the $1$-dimensional area formula
\[
\int_0^1 h(s)\, ds = \int_{\gamma ([0,1])\setminus K} |\nabla u| < \infty\, .
\]
The absolute continuity of $u\circ \gamma$ will then follow once we prove that, for all
$\sigma, s\in[0,1]$,
\begin{equation}\label{e:abs continuity}
|u (\gamma (s)) - u (\gamma (\sigma))| \leq \int_\sigma^s h(\tau)\, d\tau\, .   
\end{equation}
First of all, \eqref{e:abs continuity} is obvious if $\gamma ([\sigma, s])\subset \mathbb R^2\setminus K$, because in that case $u$ is smooth on $[\sigma,s]$. On the other hand it is not possible that $\gamma ([\sigma, s])\subset K$ for $\sigma<s$ because 
$\mathcal{H}^{\frac{1}{2}} (K\cap \{v=t\})=0$, hence $K\cap \{v=t\}$ is totally disconnected and $\gamma ([\sigma,s])\subset K$ would imply that $\gamma$ is constant on $[\sigma,s]$, contradicting the injectivity of $\gamma$. 

Fix now an arbitrary $\sigma<s$. By the above argument we know that there is a sequence $\sigma_k\downarrow \sigma$ and a sequence $s_k\uparrow s$ such that $\gamma (\sigma_k), \gamma (s_k)\not \in K$, hence it suffices to prove \eqref{e:abs continuity} when $\gamma (s), \gamma (\sigma)\not \in K$. However, if $\gamma([\sigma , s])\subset \mathbb R^2\setminus K$, then we already observed that the inequality is correct.
Thus, we can assume the existence of at least one $\tau\in (\sigma, s)$ such that $\gamma (\tau)\in K$.

Fix now $\delta>0$ and using the fact that $\mathcal{H}^{\frac{1}{2}} (K\cap \{v=t\})=0$ cover $\gamma ([0,1])\cap K$ with a collection $\mathscr{C}$ of open disks $B_{r_i} (x_i)$ such that $x_i\in \gamma ([0,1])\cap K$ and 
\[
\sum_i r_i^{\frac{1}{2}} < \delta\, .
\]
Observe that $\gamma ([0,1])\cap K$ is compact and thus we can assume that the cover is finite. Moreover, by taking the disks suitably small, we can assume that $\gamma (\sigma)$ and $\gamma (s)$ do not belong to any of them. Finally, since $\mathcal{H}^{\frac{1}{2}} (\gamma ([0,1])\cap K)=0$ we can choose the disks so that $\partial B_{r_i} (x_i) \cap (\gamma ([0,1])\cap K) = \emptyset$, which in turn implies that 
\begin{equation}\label{e:sparso}
\gamma ([0,1])\cap K \cap B_{r_i} (x_i) \qquad \mbox{is compact for every $i$.}
\end{equation}
Select now the smallest $s_1\in (\sigma, s)$ such that $\gamma (s_1)\in K$. Then clearly $\gamma (s_1)$ belongs to some disk of the cover $\mathscr{C}$ and by reindexing it we can assume it is $B_{r_1} (x_1)$. We then let $\sigma_2$ be the largest number such that $\gamma (\sigma_2)\in K \cap B_{r_1} (x_1)$, which exists by \eqref{e:sparso} and is smaller than $s$ because $\gamma (s)\not\in K$. If $\gamma ((\sigma_2, s])\cap K =\emptyset$ we stop the procedure, otherwise we select the smallest $s_2>\sigma_2$ such that $\gamma (s_2)\in K$. Then $\gamma (s_2)$ must belong to a disk of $\mathscr{C}$ which however cannot be $B_{r_1} (x_1)$, and upon reindexing we can assume is $B_{r_2} (x_2)$. We then proceed as in the first step and since $\mathscr{C}$ is finite the procedure stops in finite time. We let $\sigma_N$ be the last chosen number and set $\sigma_1 = \sigma$ and $s_N = s$. We can then partition $[\sigma, s]$ as
\[
[\sigma_1, s_1)\cup [s_1, \sigma_2] \cup (\sigma_2, s_2) \cup \ldots \cup [s_{N-1}, \sigma_N] \cup (\sigma_N, s_N]\, .
\]
Each $(\sigma_i, s_i)$ is contained in $\mathbb R^2\setminus K$. Thus $u$ is smooth on the arc $\gamma ((\sigma_i, s_i))$ and for every sufficiently small $\varepsilon$ we can write
\[
u (\gamma(s_i - \varepsilon)) - u (\gamma(\sigma_i +\varepsilon)) = \int_{\sigma_i+\varepsilon}^{s_i-\varepsilon} \frac{d}{d\tau} (u\circ \gamma) (\tau)\, d\tau\, .
\]
In particular
\[
|u (\gamma(s_i)) - u (\gamma(\sigma_i))|\leq \int_{\sigma_i}^{s_i} h(\tau)\, d\tau\, .
\]
On the other hand $u (\gamma(s_i)), u (\gamma(\sigma_{i+1}))\in B_{r_i} (x_i)$. Recall that $x_i \in \gamma ([0,1])\cap K$ and since $\gamma ([0,1])$ does not intersect any connected component of $K$ with positive length, we necessarily have $x_i\in K^\sharp$. Hence by Proposition~\ref{p:armonica-coniugata} 
\[
|u (\gamma(\sigma_{i+1})) - u(\gamma(s_i))| \leq C r_i^{\frac{1}{2}}\, .
\]
We can thus estimate
\begin{align*}
| u (\gamma(s))-u (\gamma(\sigma))| &\leq \sum_{i=1}^N |u(\gamma(s_i))- u (\gamma(\sigma_i))| + \sum_{i=1}^{N-1} |u (\gamma(\sigma_{i+1}))- u (\gamma(s_i))|\\
&\leq \sum_{i=1}^N \int_{\sigma_i}^{s_i} h(\tau)\, d\tau + C \sum_{i=1}^{N-1} r_i^{\frac{1}{2}}
\leq \int_\sigma^s h(\tau)\, d\tau + C \delta\, .
\end{align*}
Since $\delta$ is arbitrary, we then conclude \eqref{e:abs continuity}. 

As for \eqref{e:derivata} recall first that, since $u\circ \gamma$ is in $W^{1,1}$, it is a.e. differentiable. The derivative at a.e. $\tau\in [0,1]\setminus \gamma^{-1} (K)$ can be computed using the chain rule because $u$ is smooth on $\mathbb R^2\setminus K$. Nex observe that, by \eqref{e:abs continuity},
\[
\left|\frac{d}{d\tau} (u\circ \gamma)\right|\leq h \qquad \mbox{a.e.}
\]
and since $g$ vanishes on $\gamma^{-1} (K)$, we conclude that 
\[
\frac{d}{d\tau} (u\circ \gamma) = 0\,  \qquad \mbox{a.e. on $\gamma^{-1} (K)$.}
\]
However, as already noticed we have that $\dot\gamma = 0$ a.e. on $\gamma^{-1} (K)$. Hence we can claim \eqref{e:derivata} a.e. on $[0,1]$ if we interpret the expression as $0$ at every point where $\dot\gamma$ exists and vanishes.

\section{The level sets of the harmonic conjugate: Part II}\label{s:livelli-coniugata-2}

In this section we deepen the analysis of the level sets of the harmonic conjugate. We will make use of the following terminology:

\begin{definition}\label{d:terminal points}
A set $J\subset \mathbb R^2$ contains no loops if any injective continuous map $\gamma: \mathbb S^1 \to J$ is necessarily constant. A point $p\in J$ is terminal if there is no injective continuous map: $\gamma:(-1,1) \to J$ such that $\gamma (0) =p$. Points $p\in J$ which are not terminal will be called nonterminal.
\end{definition}
Note that the terminology is consistent with the terms terminal and nonterminal points 
introduced in Definition~\ref{d:internal points} and used thus far for points of $K$.

\begin{proposition}\label{p:struttura-2}
Let $K,u$, and $v$ be as in Proposition~\ref{p:struttura-1}. Then:
\begin{itemize}
    \item[(i)] For every $m$ the level set $\{v=m\}$ does not disconnect $\mathbb R^2$ and in particular contains no loops.
    \item[(ii)] $v$ satisfies the maximum and minimum principle, i.e. 
    \begin{align}
    \max_{\overline U} v &= \max_{\partial U} v\, ,\\
    \min_{\overline U} v &= \min_{\partial U} v\, ,
    \end{align}
    for every bounded open set $U$. Moreover, if $x\in U$ is a local minimum (resp. maximum), then necessarily $x\in K$.
    \item[(iii)] For every $m$ to which all
    the conclusions of Proposition~\ref{p:struttura-1} apply, the level set $\{v=m\}$ contains no terminal points.
\end{itemize}
\end{proposition}
\begin{proof}{\bf Proof of (i)} If $\{v=m_0\}$ contains a loop, then it disconnects $\mathbb R^2$ and at least one connected component of $\{v \neq m_0\}$, which we denote by $U$, must be bounded. Clearly $v$ on $U$ must either be strictly larger than $m_0$ or strictly smaller than $m_0$. To fix ideas let us assume that it is strictly larger and let $\delta>0$ be such that $U_m:=\{v>m\}\cap U$ is not empty for every $m\in (m_0, m_0+\delta)$. Consider that $\partial U_m$ is rectifiable for a.e. $m$ by Proposition~\ref{p:struttura-1} and that 
\[
\int_{m_0}^{m_0+\delta} \mathcal{H}^1 (\partial U_m)\, dm < C\, .
\]
We fix next a suitable regularization $v_k$ of $v$, for instance by convolution, so that 
\[
\limsup_k \int_{m_0}^{m_0+\delta} \mathcal{H}^1 (\partial\{v_k>m\}\cap U)\, dm < \infty\, .
\]
By a standard diagonalization procedure we can find a sequence of values $m_k$ converging to some $m\in (m_0, m_0+\delta)$ with the properties that:
\begin{itemize}
    \item[(a)] All the conclusions of Proposition~\ref{p:struttura-1} apply.
    \item[(b)] $\{v_k = m_k\}\cap U$ consist of non degenerate points (i.e. $\nabla v_k\neq 0$ everywhere on $\{v_k= m_k\}\cap U$).
    \item[(c)] $\mathcal{H}^1 (\{v_k = m_k\}\cap U)$ converge to $\mathcal{H}^1 (\{v = m\}\cap U)$. 
\end{itemize}
Since $\{v_k=m_k\}\cap U$ is (for sufficiently large $k$) compactly contained in $U$, it consists of finitely many loops. Each such loop bounds a corresponding disk. Fix a point $p$ in $\{v>m\}\cap U$: for sufficiently large $k$'s $p$ must belong to one such disk. We denote it by $D_k$ and we let $\gamma_k$ be the corresponding loop. Note that the length of $\gamma_k$ is uniformly bounded in $k$. We can thus parametrize $\gamma_k$ by a constant multiple of the arc length over $\mathbb S^1$. Without loss of generality we keep denoting by $\gamma_k$ such parametrization. By Ascoli-Arzel\`a we can assume that $\gamma_k$ converges to some Lipschitz map $\gamma: \mathbb S^1 \to U$. Observe that $\gamma (\mathbb S^1)\subset \{v=m\}\cap U$. But observe also that the length of $\gamma (\mathbb S^1)$ must be equal to the limits of the lengths of $\gamma_k (\mathbb S^1)$ (which we might assume exist by extraction of a subsequence), otherwise 
item (c) above would be violated. In particular we obtain that 
\[
\int |\dot\gamma| = 
\mathcal H^1 (\gamma (\mathbb S^1))\, .
\]
Observe also that we can choose $\delta$ positive so that $\overline{B}_\delta (p)\subset \{v>m\}\cap U$. In particular, for all sufficiently large $k$ the disk $\overline{B}_\delta (p)$ must be contained inside the disk $D_k$. We thus conclude from the isoperimetric inequality that 
\[
\mathcal H^1(\gamma (\mathbb S^1)) > 0\, .
\]
Choose now the orientation of $\gamma_k$ so that 
\[
\frac{\dot{\gamma}_k^\perp}{|\dot \gamma_k|} = \lambda_k \nabla v_k\circ \gamma_k
\]
for some positive function $\lambda_k$. In particular 
\[
\frac{\dot{\gamma}_k^\perp}{|\dot \gamma_k|}
\]
is the inward unit normal to $D_k$ and $v_k$ increases in its direction. It is then easy to see that 
\[
\frac{\dot{\gamma}^\perp}{|\dot\gamma|}\cdot \nabla v \circ \gamma \geq 0
\]
holds a.e., since the convergence of $v_k$ to $v$ is in fact smooth on $\gamma (\mathbb S^1)\setminus K$, which has measure zero. 
But then we conclude that in fact 
\[
\frac{\dot{\gamma}^\perp}{|\dot\gamma|}\cdot \nabla v \circ \gamma > 0
\]
a.e., since both vectors are collinear and nonzero a.e. on $\gamma (\mathbb S^1)$ (recall that $v$ is harmonic outside $K$ and as such the set of points not contained in $K$ where its gradient vanishes must be countable). 

In particular, we conclude
\[
\dot{\gamma}\cdot \nabla u \circ \gamma > 0\, \qquad \mbox{a.e..}
\]
However, by Proposition~\ref{p:struttura-1} the latter would imply
\[
0 = u (\gamma (2\pi)) - u (\gamma (0)) = \int_0^{2\pi} \nabla u (\gamma (t)) \cdot \dot \gamma (t)\, dt > 0\, , 
\]
which is a contradiction.

\medskip

{\bf Proof of (ii)} The two cases are entirely analogous and we focus on the case of maxima for simplicity. Consider, by contradiction, a bounded open set $V$ for which $\max_{\overline V} v > \max_{\partial V} v$. Then there is an $m_0$ such that $\{v = m_0\}$ does not intersect $\partial V$ and $\{v>m_0\} \cap V$ is not empty. Setting $U:= \{v>m_0\}$ we can argue as in the previous step to obtain a contradiction. Next, if $x\in U$ is a local maximum and $x\not \in K$, by the harmonicity of $v$ we conclude that $\nabla v$ must vanish in a neighborhood of $x$. But this is not possible because Corollary~\ref{c:no-vanishing} would imply that $(u, K)$ is an elementary minimizer.

\medskip

{\bf Proof of (iii)} We will prove the claim for any $m$ such that all the conclusions of Proposition~\ref{p:struttura-1} apply. Fix then a point $x_0$ in $\{v=m\}$ and assume, without loss of generality, that $x_0 =0$. The proof will be split in several step

\medskip

{\bf Step 1. Choice of a good radius.}
Let $r>0$, and $G$ be any connected component of $\{v=m\}\cap \overline{B}_r$ and observe that it must necessarily intersect $\partial B_r$. Otherwise, using Lemma~\ref{l:little-loop} we find a Jordan curve $\gamma$ which does not intersect $\{v=m\}$ and bounds a disk $D$ which contains $G$. Clearly we must have either $v<m$ or $v>m$ on $\partial D$ and in particular we would violate (ii).

We thus infer that $\{v=m\}\cap \partial B_r \neq\emptyset$ for every $r$, because some connected component of $\{v=m\}\cap \overline{B}_r$ must contain the origin (recall that $x_0=0\in\{v=m\}$).

We next appeal to the coarea formula \cite[Theorem 2.93]{AFP00} to choose an $r\in (1,2)$ with the properties that
\begin{itemize}
    \item[(1)] $\{v=m\}\cap \partial B_r$ is finite;
    \item[(2)] $\{v=m\}\cap \partial B_r\cap K = \emptyset$;
    \item[(3)] $\nabla v (y) \neq 0$ for every $y\in \{v=m\}\cap \partial B_r$ (hence $\{v=m\}$ is a smooth arc in a sufficiently small neighborhood of any such $y$);
    \item[(4)] $\{v=m\}$ intersects $\partial B_r$ transversally at any $y\in \partial B_r\cap \{v=m\}$.
\end{itemize}
Fix now any point $y\in \partial B_r \cap \{v=m\}$.
By (4) the normal to $\{v=m\}$ at $y$ cannot be perpendicular to $\partial B_r$: recall that this normal is colinear with $\nabla v (y)$, which by (3) does not vanish. We thus conclude that $\nabla v (y)$ cannot be perpendicular to $\partial B_r$. Hence, if we consider the restriction of the function $v-m$ to $\partial B_r$, such function is changing sign at each point $y\in \partial B_r \cap \{v=m\}$. Therefore $\{v=m\}\cap \partial B_r$ consists of an even number of points which divide the cirle $\partial B_r$ into an even number of arcs $\gamma_i$: on half of them $v-m$ is positive, while on the other half it is negative, cf. Figure~\ref{f:finite-arcs}.

\begin{figure}
\begin{tikzpicture}
\draw (-6,0) circle [radius =2];
\draw[very thick] (-6, 2.5) to [out = 240, in = 120] (-6, 1.5);
\draw[very thick] (-3.5, 0) to [out = 240, in = 120] (-4.5, 0);
\draw[very thick] (-6, -2.5) to [out = 240, in = 120] (-6, -1.5);
\draw[very thick] (-7.5, 0) to [out = 240, in = 120] (-8.5, 0);
\node at (-6,0) {$B_r$};
\node at (-4.3,1.7) {$\gamma_1$};
\node at (-4.3,-1.7) {$\gamma_2$};
\node at (-7.7,-1.7) {$\gamma_3$};
\node at (-7.7,1.7) {$\gamma_4$};
\end{tikzpicture}
\caption{\label{f:finite-arcs} $\partial B_r\setminus \{v=m\}$ consists of an even number of arcs, on which the function $v-m$ takes alternating signs. The picture depicts $\{v=m\}$ only in a neighborhood of $\partial B_r$.}
\end{figure}
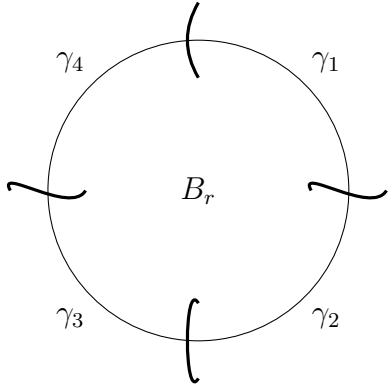

\medskip

{\bf Step 2. $0\in \partial A$ for some connected component $A$ of $B_r \setminus \{v=m\}$.} Consider now the connected components $\{A_i\}$ of $B_r \setminus \{v=m\}$. $v-m$ does not change sign nor vanishes on each $A_i$.
$\partial A_i$ must then intersect $\partial B_r$: otherwise we would have $v\equiv m$ on $\partial A_i$, violating (ii). Since each $A_i$ is connected, if $\partial A_i$ intersects a given $\gamma_j$, 
the latter being one of the arcs introduced in the previous step, then $\gamma_j \subset A_i$. In particular each $\gamma_j$ is contained in at most one $\partial A_i$, while each $\partial A_i$ contains at least one $\gamma_j$.
Hence the $A_i$'s are finitely many. 

Next, since $\{v=m\}$ contains no interior point, every $z\in \{v=m\}\cap B_r$ must be contained in $\partial A_i$ for some $i$. This holds for $0$ as well. We therefore assume that $0\in \partial A_1$ and for simplicity we set $A=A_1$.

\medskip

{\bf Step 3. $A$ is simply connected.} Furthermore,
to simplify our discussion we assume that $v>m$ on $A$. Note next that $A$ must be simply connected: if $\gamma$ is a smooth simple Jordan curve in $A$, then it bounds a disk $D$ and the latter must be contained in $A$ otherwise $D$ would contain a portion of $\partial A$ which does not intersect $\gamma$, and from the latter we would get a connected component of $\{v=m\}$ which does not intersect $\partial B_r$. Since $A$ contains $D$, $\gamma$ is contractible in $A$.

\medskip

{\bf Step 4. Finding a suitable Jordan curve in $\partial A$.} Having established that $A$ is simply connected, i.e. it is a topological disk, we infer that its boundary $\partial A$ cannot be disconnected by a point $p\in \partial A$. We next devise a suitable algorithm to generate a suitable Jordan curve contained in $\partial A$. 

Fix first a $\gamma_i$ contained in $\partial A$, relabel it so that $i=1$ and consider its two endpoints $a_1$ and $b_1$. If we remove the (open arc) $\gamma_1$ from $\partial A$, the remainder is a closed connected set with finite Hausdorff measure. We can thus apply Lemma~\ref{l:connessi_per_archi} to conclude that there is an injective curve $\eta$ in $\partial A\setminus \gamma_1$ which connects $b_1$ and $a_1$.  Consider it as a parametrized curve $\eta: [0,1]\to \partial A$ with $\eta (0) = b_1$ and $\eta (1) = a_1$. Since in a neighborhood $U$ of $b_1$ $\{v=m\} \cap U$ consists of a single smooth arc crossing $\{v=m\}$ precisely in $b_1$, for a sufficiently small $\delta>0$ $\eta ([0, \delta))$ must be contained in this arc. In particular $\eta ((0, \delta))$ does not intersect $\partial B_r$. On the other hand $\eta (1) = a_1\in \partial B_r$. Consider therefore the smallest positive $s$ such that $\eta (s) \in \partial B_r$: the arc $\eta_1 = \eta ([0, s])$ is then a simple arc joining $b_1$ with either $a_1$ or with one extremum of some other arc $\gamma_j$, and $\eta_1 ((0,s))$ is contained in $\partial A \setminus \partial B_r$. If $\eta_1 (s)= a_1$ we stop and we conclude that $\gamma_1\cup \eta_1$ is a Jordan curve. 

Otherwise $\eta_1 (s)$ must be an extremum of a second $\gamma_j$ contained in $\partial A$. We relabel so that $j=2$ and denote by $a_2$ and $b_2$ its extrema, with the convention that $a_2 = \eta_1 (s)$. We now remove $\gamma_1\cup \eta_1 \cap \gamma_2$ from $\partial A$ (where we follow the convention that the arcs $\gamma_i$ are open while the arcs $\eta_i$ are closed). The remaining set is still connected, has finite $\mathcal{H}^1$ measure, and contains $b_2$ and $a_1$. We can thus repeat the procedure above to produce a second simple arc $\eta_2\subset \partial A\setminus (\gamma_1\cup \eta_1\cup \gamma_2)$ which connects $b_2$ with the endpoint of some other arc $\gamma_k$ contained in $\partial B_r$, taking care that $\eta_2$ lies in $B_r$ except for its two extrema. Recall that one extremum of $\eta_2$ is $b_2$: if the second is $a_1$ we stop the procedure. Otherwise note that it cannot be an extremum of either $\gamma_1$ or $\gamma_2$: it must be the extremum of some $\gamma_k$ with $k\neq 1,2$, cf. Figure~\ref{f:find-Jordan}. We then let $k=3$ after relabeling and proceed as above.

Since the $\gamma_j$'s are finitely many, the procedure must end. Assume it ends after $N$ steps. If we string together $\gamma_1\cup \eta_1\cup \gamma_2 \cup \ldots \cup \eta_N$, we achieve a Jordan curve. 

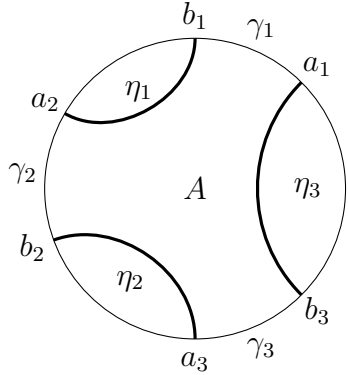
\begin{figure}
\begin{tikzpicture}
\draw (0,0) circle [radius =2];
\draw[very thick] (0,2) to [out = 270, in = 330] ({2*cos{150}},{2*sin{150}});
\draw[very thick] ({2*cos{200}},{2*sin(200)}) to [out = 20, in =90] (0,-2);
\draw[very thick] ({2*cos{-45}},{2*sin{-45}}) to [out =135, in=225] ({2*cos{45}},{2*sin{45}});
\node at ({2.3*cos{45}}, {2.3*sin{45}}) {$a_1$};
\node at (0,2.3) {$b_1$};
\node at ({2.3*cos{150}},{2.3*sin{150}}) {$a_2$};
\node at ({2.3*cos{200}},{2.3*sin{200}}) {$b_2$};
\node at (0,-2.3) {$a_3$};
\node at ({2.3*cos{-45}},{2.3*sin{-45}}) {$b_3$};
\node at ({2.3*cos{67.5}},{2.3*sin{67.5}}) {$\gamma_1$};
\node at ({2.3*cos{175}},{2.3*sin{175}}) {$\gamma_2$};
\node at ({2.3*cos{-67.5}},{2.3*sin{-67.5}}) {$\gamma_3$};
\node at ({1.5*cos{120}},{1.5*sin{120}}) {$\eta_1$};
\node at ({1.5*cos{235}},{1.5*sin{235}}) {$\eta_2$};
\node at (1.5,0) {$\eta_3$};
\node at (0,0) {$A$};
\end{tikzpicture}
\caption{\label{f:find-Jordan} The picture is an illustration of the algorithm to find a Jordan curve inside $\partial A$. The algorithm finds inductively the pairs of Jordan arcs $(\gamma_i, \eta_i)$ starting from $\gamma_1$.}
\end{figure}

\medskip

{\bf Step 5. Describing $\partial A$.} The Jordan curve $\beta:\mathbb S^1\to\partial A$ found by the previous algorithm bounds then a topological disk $D$ which must be contained in $B_r$ and which contains $0$. Observe that $A$ must be contained in $D$. Indeed $D$ is a connected component of $B_r\setminus \beta(\mathbb S^1)$ and, since $\beta(\mathbb S^1) \subset \partial A \subset \{v=m\}$, any connected component of $B_r\setminus \{v=m\}$ either does not intersect $D$, or it is contained in $D$. On the other hand both $A$ and $D$ contains $0$, so $A\cap D \neq \emptyset$. Next $v=m$ on all the $\eta_i$'s, while $v>m$ on all the $\gamma_i$'s ($v\geq m$ on each of them because they are in $\partial A$ and the strict inequality is because none of them intersects $\{v=m\}$). But then we conclude that $v\geq m$ on $D$, by (ii). We now want to exclude that 
$v=m$ somewhere in $D$. Indeed, if $z\in \{v=m\}\cap D$ notice that there is a connected component $G$ of $\{v=m\}$ which contains $z$ and intersects $\partial B_r$. But then $\mathcal{H}^1 (\{v=m\}\cap D) >0$. Since $\mathcal{H}^1 (\{v=m\} \cap K) = 0$, we must have some $z\in \{v=m\}\cap D \setminus K$. But then, by Proposition~\ref{p:struttura-1}(ii) $v-m$ takes both signs in any neighborhood of $z$, contradicting $v\geq m$ in $D$. Having concluded that $v>m$ in $D$, we infer that $D\cap \{v=m\}=\emptyset$. In particular $D$ is a connected component of $B_r\setminus \{v=m\}$. Since it contains $A$, it actually must be equal to $A$. 

We are now ready to conclude: recalling that $0\in \partial A$, it must be contained in some $\eta_j$. The latter is an injective curve, it is contained in $\{v=m\}$, and $0$ is not one of its two extrema: in particular $0$ is a nonterminal point of $\{v=m\}$. 
\end{proof}

\section{The level sets of the harmonic conjugate: Part III}\label{s:livelli-coniugata-3}

We are now ready to conclude our analysis of the level sets of $v$. While so far all the conclusions apply for a global generalized minimizer which is not elementary, in this section we use substantially the additional assumption that all but one connected component of $K$ are contained in a fixed disk $B_R$.
Before coming to the relevant statement, let us point out that, given any global generalized minimizer $(u,K, \{p_{kl}\})$ and any constant $c$, $(\pm (u-c), K, \{\pm p_{kl}\})$ is also a global generalized minimizer.

\begin{proposition}\label{p:struttura-3}
Let $(u,K)$ be a nonelementary global minimizer satisfying (a) and (b) in Theorem~\ref{t:Bonnet-David} and let $v$ be as in Proposition~\ref{p:armonica-coniugata}. Then:
\begin{itemize}
    \item[(i)] every blow-down of $(u,K)$ (i.e. any limit of $(u_{0,r_j}, K_{0, r_j})$ for $r_j\uparrow \infty$) is the cracktip;
    \item[(ii)] $|\nabla v (x)|\geq c_0 |x|^{-\frac12}$ for all $x\not \in K$ large enough and for a positive geometric $c_0$;
    \item[(iii)] up to changing the sign of $u$, $v$ achieves its global minimum $\min v$;
    \item[(iv)] for every $m> \min v$ there is a radius $R= R(m)$ such that $\{v=m\} \setminus B_r$ is smooth and $\{v=m\}$ intersects each $\partial B_r$ transversally in exactly two points for every $r\geq R(m)$;
    \item[(v)] for a.e. $m> \min v$ the level set $\{v=m\}$ is a properly embedded unbounded line;
    \item[(vi)] the unbounded connected component of $K$ is $\{v=\min v\}$.
\end{itemize}
\end{proposition}

\begin{proof}{\bf Proof of (i).} Since $K\setminus B_{R_0}$ is a single connected component for a sufficiently large $R_0$, by Proposition~\ref{p:monotonia-0} the quantity
\[
\frac{1}{r}\int_{B_r} |\nabla u|^2
\]
is monotone for $r> R_0$.
Let $(w, J)$ be any blow-down. It then turns out that $J$ consists of a single unbounded connected component and that 
\[
\frac{1}{r} \int_{B_r} |\nabla w|^2
\]
is a nonzero constant thanks to Corollary~\ref{c:no-vanishing}, being $(u, K)$ 
nonelementary. 
But then Proposition~\ref{p:monotonia-0} implies that it is a cracktip. Observe that, while we lack at this point any uniqueness statement and the half-line of the cracktip could change depending on the blow-down sequence, it is easy to verify that the sign of the constant $b$ in \eqref{e:radice} must always be the same, independently of the subsequence (cf. Proposition~\ref{p:constant-crack-tip}). 

\medskip

{\bf Proof of (ii), (iii), and (iv).} Fix $r>0$ sufficiently large and rotate the coordinates by an angle $\theta (r)$ so that $\mathcal{R}_{\theta (r)} (K)\cap \partial B_r = \{(r, 0)\}$. We consider the pair
\begin{align}
K_r &:= \frac{1}{r} \mathcal{R}_{\theta (r)} (K)\\
u_r (x) &:= r^{-\sfrac{1}{2}} u ( r \mathcal{R}_{-\theta (r)} (x))
\end{align}
and the function
\[
v_r (x) := r^{-\sfrac{1}{2}} v (r \mathcal{R}_{-\theta (r)} (x))\, .
\]
Up to a change of the sign of $u$ we can assume $(u_r, K_r)$ is converging to the following specific cracktip $(w_\infty,K_\infty)$:
\begin{align}
K_\infty &= \{(t,0): t\geq 0\}\\
w_\infty (x) &= \sqrt{\frac{2}{\pi}} r^{\sfrac{1}{2}} \cos \frac{\theta}{2}\, .
\end{align}
On the other hand $v_r$ is converging (locally uniformly, thanks to the H\"older estimate of Proposition~\ref{p:armonica-coniugata}) to a H\"older function $v_\infty$ which is harmonic on $\mathbb R^2\setminus K_\infty$ and satisfies $\nabla v_\infty = \nabla u_\infty^\perp$ and $v_\infty (0) =0$. But then it turns out that 
\[
v_\infty (x) = \sqrt{\frac{2}{\pi}} r^{\sfrac{1}{2}} \sin \frac{\theta}{2}\, .
\]
We next appeal to the $\varepsilon$-regularity theory at pure jumps to argue that the convergence of $K_r$ to $K_\infty$ is in $C^{1,\alpha}$ on $B_4\setminus B_{\frac12}$. Hence we can appeal to the regularity theory coming from the Euler-Lagrange conditions of Proposition~\ref{p:variational identities 2} to conclude that the convergence of $u_r$ and $v_r$ to $u_\infty$ and $v_\infty$ is also in $C^{1,\alpha}$ ``up to the discontinuity set $K_r$'': more precisely, there are $C^{1,\alpha}$ diffeomorphisms $\Phi_r$ of $B_4\setminus (B_{1/4} \cup K_r)$ onto $B_4\setminus (B_{1/4}\cup K_\infty)$ such that
\begin{itemize}
    \item $\Phi_r$ converges to the identity in $C^{1,\alpha}$ as $r\uparrow\infty$;
    \item $u_r\circ \Phi_r^{-1}$, $v_r\circ \Phi_r^{-1}$ converge in $C^{1,\alpha}$ to $u_\infty$, $v_\infty$ on $B_2\setminus (B_{\frac12}\cup K_\infty)$.
\end{itemize}
We can then easily draw the following conclusions:
\begin{itemize}
    \item[(1)] $|\nabla u_r|\geq c >0$ on $B_2\setminus (B_{\frac12}\cup K_r)$ for some suitable positive constant $c$ and every sufficiently large $r$. This clearly implies (ii).
    \item[(2)] Since $|\nabla v_r|= |\nabla u_r|$ we infer that the level sets of $v_r$ are smooth on $B_2\setminus (B_{\frac12} \cup K_r)$.
    \item[(3)] Since $\pm \frac{\partial v_\infty}{\partial \theta} (1,\pm \theta)\geq \frac{1}{2\sqrt{2\pi}}$ for $\theta \in (0, \frac{\pi}{2})$, we conclude $\pm \frac{\partial v_r}{\partial \theta} (1, \pm \theta) \geq \frac{1}{4\sqrt{2\pi}}$ for $\theta \in (0,\frac{\pi}{2})$ and $r$ large enough.
    \item[(4)] Since $v_\infty (1, \theta)\geq \pi^{-\sfrac{1}{2}}$ for $\theta \in (\frac{\pi}{2}, \frac{3\pi}{2})$, clearly $v_r (1, \theta) \geq \frac{1}{2}\pi^{-\sfrac{1}{2}}$ for $\theta \in (\frac{\pi}{2}, \frac{3\pi}{2})$ and $r$ large enough.
\end{itemize}
The last two facts imply that, if $r$ is large enough, then $v_r (1, \theta) > v_r (1,0)$ for all $\theta\in (0, 2\pi)$. Translating this information back to $v$, we conclude that, if we denote by $\bar m$ the constant value achieved by $v$ on the unbounded connected component of $K$, then $v> \bar m$ on $\partial B_r\setminus K$. By the maximum principle of Proposition~\ref{p:struttura-2} the latter implies that $\bar m$ is the absolute minimum of $v$.

Fix now $m> \bar m$, then 
\[
\{v=m\}\setminus B_{\frac r2}=
r (\{v_r = r^{-\sfrac{1}{2}} m
\}\setminus B_{\frac12})\, .
\]
If $r$ is sufficiently large,  the set $\{v_r = r^{-\sfrac{1}{2}} m 
\}$ intersects $\partial B_1$ transversally in exactly two points, by (3) and (4). This shows (iv)  

\medskip

{\bf Proof of (v).} Pick $m> \bar m$ to which all the conclusions of Proposition~\ref{p:struttura-2} apply and let $r$ be large enough so that $\partial B_r$ intersects $\{v=m\}$ transversally in exactly two points. Since $B_r \cap \{v=m\}$ has no terminal points and no loops, $\{v=m\}\cap B_r$ must be a Jordan arc. Indeed we could even notice that the very argument given in the previous section for Proposition~\ref{p:struttura-2}(iii) shows this fact directly.

\medskip

{\bf Proof of (vi).} Let $L$ be the unbounded connected component of $K$. Observe that we have already shown that $L\subset \{v=\bar m\}$, where $\bar m = \min v$. Next observe that, for a.e. $m> \bar m$, the level set $\{v\leq m\}$ is connected, because of (v). Since $\{v=\bar m\} = \bigcap_{m>\bar m} \{v\leq m\}$, we conclude that $\{v= \bar m\}$ is connected as well. Next notice that $\{v= \bar m\}\subset K$, because by Proposition~\ref{p:struttura-1}(ii) any point of $\{v=\bar m\}$ is a global minimum for $v$.  We have thus inferred that $L \subset \{v=\bar m\}\subset K$ and that $\{v=\bar m\}$ is connected. Recalling that $L$ is a connected component of $K$, we must have $L = \{v= \bar m\}$.
\end{proof}

\section{A special bounded connected component of \texorpdfstring{$K$}{K}}\label{s:G}

From now on we will establish a set of conditions which we can assume for a couple $(u,K)$ 
as in Theorem~\ref{t:Bonnet-David} without loss of generality, thanks to Proposition~\ref{p:struttura-3}.

\begin{ipotesi}\label{a:Bonnet-David}
$(u,K)$ is a nonelementary global generalized minimizer (in particular $\R^2\setminus K$ is connected) and $v$ is an harmonic conjugate of $u$ as in Proposition~\ref{p:armonica-coniugata}, which satisfy in addition the following properties.
\begin{itemize}
    \item[(i)] $v\geq 0$.
    \item[(ii)] $L:= \{v=0\}$ is the only unbounded connected component of $K$.
    \item[(iii)] There is $R>0$ s.t. $K\setminus B_R\subset L$ and $L\cap \partial B_R$ has cardinality $1$.
\end{itemize}
We then let $m_0$ be $\max_K v$.
\end{ipotesi}

Observe that $m_0$ is well defined because it is in fact $\max_{K\cap \overline{B}_R} v$ by (ii) and (iii) above, while $v$ is continuous and $K\cap \overline{B}_R$ compact. We next point out a pivotal (yet simple) outcome of the Bonnet monotonicity formula.

\begin{corollary}\label{c:m_0=0}
Let $(u,K)$ be as in Assumption~\ref{a:Bonnet-David}. If $m_0=0$, then $(u,K)$ is a cracktip.
\end{corollary}

We postpone its simple proof, for the moment, and note that it gives us a route to the proof of Theorem~\ref{t:Bonnet-David}: we assume that $m_0>0$ and wish to derive a contradiction. To that end it will be useful to make the following preliminary analysis.

\begin{proposition}\label{p:componente-speciale}
Let $(u,K)$ be as in Assumption~\ref{a:Bonnet-David} and assume that $m_0 >0$. Then any $x\in K\cap \partial \{v> m_0\}$ is either a pure jump or a triple junction. In particular, if $G$ is any connected component of $K$ contained in $\{v=m_0\}$, then
\begin{itemize}
    \item[(i)] $G$ is bounded and $\mathcal{H}^1 (G)>0$; 
    \item[(ii)] if $x\in G$ is a terminal point, then there is a neighborhood of $x$ in which $v\leq m_0$.
\end{itemize}
\end{proposition}

\subsection{Proof of Corollary~\ref{c:m_0=0}}
If $m_0=0$ it follows immediately from our assumptions that $K = \{v=0\}$, it is connected, and it is unbounded. Since $K$ does not disconnect $\mathbb R^2$, $K$ must have at least one terminal point, which without loss of generality we can assume to be $0$. By Proposition~\ref{p:monotonia-0} we know that
\[
\frac{D(r)}{r} = \frac{1}{r} \int_{B_r} |\nabla u|^2
\]
is monotone nondecreasing. Recall that the blow-downs of $(u,K)$ are cracktips by (i)  Proposition~\ref{p:struttura-3}. Observe that, if $\lim_{r\downarrow 0} r^{-1} D(r) =0$, then the blow-ups of $(u,K)$ at $0$ are elementary minimizers. By the density lower bound none of them can be a constant, so they must be either pure jumps or triple junctions. But then the $\varepsilon$-regularity theory developed thus far would imply that $K\cap B_r$ is diffeomorphic either to a diameter of $B_r$ or to three radii meeting at the origin, contradicting the hypothesis that $0$ is a terminal point of $K$. We thus conclude that
$c := \lim_{r\downarrow 0} r^{-1} D(r) > 0$. But then any blow-up $(u_0, K_0)$ at $0$ must have $K_0$ connected and must satisfy
\[
\frac{1}{r} \int_{B_r} |\nabla u_0|^2 \equiv c >0
\]
Thus the second part of Proposition~\ref{p:monotonia-0} implies that $(u_0, K_0)$ is a cracktip. We thus conclude that $r^{-1} D(r)$ is as well constant and nonzero. Hence we can once again appeal to Proposition~\ref{p:monotonia-0} to conclude that $(u,K)$ itself is a cracktip.  

\subsection{Proof of Proposition~\ref{p:componente-speciale}} First of all we address the two consequences (i) and (ii) of the main statement. (ii) is indeed obvious because of the main part of the Proposition. 

Next consider any connected component $G$ of $K$ in $\{v=m_0\}$, 
and note that $G$ is necessarily bounded in view of Assumption~\ref{a:Bonnet-David}.
Let $x\in G$. 
Recall that the connected component $J$ of $\{v=m_0\}\cap \overline{B}_r (x)$ which contains 
$x$ must intersect $\partial B_r (x)$, by Proposition~\ref{p:struttura-2} (otherwise Lemma~\ref{l:little-loop} would provide an open disk $D$ with smooth boundary $\gamma$ containing $J$ and with $\gamma \cap \{v=m_0\}=\emptyset$: such a $D$ would then violate the maximum principle of Proposition~\ref{p:struttura-2}). Therefore,
if $\{x\}$ were a connected component of $K$, then there would be a sequence $x_k \in J\setminus K$ converging to $x$. $v$ would be harmonic in a neighborhood of these points, and since it is nowhere constant because of Corollary~\ref{c:no-vanishing}, $v-m_0$ would need to change sign in these neighborhoods. In particular we could provide a sequence $y_k$ such that $|x_k-y_k|\leq \frac{1}{k}$ with $v (y_k)>m_0$, thus showing that $x\in K\cap \partial \{v>m_0\}$. But then by the first part of the Proposition $x$ would be a pure jump point or a triple junction, showing that $K\cap B_\rho (x)$ is connected in some disk $B_\rho (x)$,
a contradiction.

We next come to the main part of the Proposition, which will be split in several steps. 

\medskip

{\bf Step 1.} Consider the set $V := \{v>m_0\}$ and fix any point $x$ 
in $\R^2$. In this step we show that the functional
\[
O (x,r) := \frac{1}{r} \int_{V \cap B_r (x)} |\nabla u|^2
\]
is monotone in $r$. In this step for convenience of notation we assume that $x=0$ and, 
arguing as in the proof of Proposition~\ref{p:monotonia-0}, we observe that $O(0,\cdot)$ is 
absolutely continuous and its derivative is given by
\[
O' (0,r) = \frac{1}{r} \int_{V \cap \partial B_r} |\nabla u|^2 - \frac{1}{r^2} \int_{V \cap B_r} |\nabla u|^2\, .
\]
We next claim that 
\begin{equation}\label{e:integra-per-parti}
\int_{V \cap B_r} |\nabla u|^2 = \int_{V \cap \partial B_r\setminus K} u \frac{\partial u}{\partial n}\, .
\end{equation}
In order to prove it we observe first that the identity can be justified if $V$ is replaced by $\{v>m\}$ for some $m$ bigger than $m_0$ for which the conclusion of (v) of Proposition~\ref{p:struttura-3} applies, because then we can let $m\downarrow m_0$. Next, for any such $m$ the set $\{v=m\}$ is a nonintersecting infinite curve which does not intersect $K$ (because by definition $v \leq m_0$ on $K$). It then turns out that $\nabla v$ cannot vanish on it: as it is well known, since $v$ is harmonic, around a point $p$ where $\nabla v$ vanishes the level set $\{v= v(p)\}$ cannot be a Jordan arc, and it is rather diffeomorphic to the union of $N\geq 4$ segments joining at $p$. This follows simply from the fact that, in a sufficiently small neighborhood of $p$, the level set $\{v= v(p)\}$ is diffeomorphic to the zero set of the first nontrivial harmonic polynomial in the Taylor expansion of $v-v(p)$ at $p$: up to rotations, the latter is necessarily of the form ${\rm Re} (z-z_0)^k$ for some $k\geq 2$, where $z_0 = x_0+i y_0$ for $p=(x_0, y_0)$, and $z=x+iy$. It thus turns out that $\{v>m\}$ is a smooth set and we can use the relation $\frac{\partial u}{\partial n} =0$ on $\partial \{v>m\}=\{v=m\}$. 

Next observe that 
\[
\int_{\gamma} 
\frac{\partial u}{\partial n} = 0
\]
for every connected component $\gamma$ of $\{v>m_0\}\cap \partial B_{r}$.
Again it suffices to show it first for every connected component of $\{v>m\}\cap \partial B_r$ for the $m>m_0$ to which we can apply Proposition~\ref{p:struttura-3} and then pass to the limit. In fact by varying $m$ we can assume that the intersection of $\{v=m\}$ with $\partial B_r$ is transversal. Then it suffices to observe that for each such $\gamma$ we can find an arc $\eta$ in $\{v=m\}$ with the same endpoints as $\gamma$ such that $\eta\cup \gamma$ bounds a disk $D$. We then use 
\[
0=\int_{\gamma\cup \eta} \frac{\partial u}{\partial n} = \int_\gamma \frac{\partial u}{\partial n}\, ,
\]
where the second identity is due to $\frac{\partial u}{\partial n}=0$ on $\eta$, and the first is implied by the harmonicity of $u$ in $D$, if $D\cap K = \emptyset$. More generally we can use Proposition~\ref{p:variational identities} if $K\cap D \neq \emptyset$, because (taking into account that $\lambda =0$) \eqref{e:outer} is the weak formulation of the Neumann condition $\frac{\partial u^\pm}{\partial n}=0$ on $K$. However in this particular instance we can show that $D\cap K= \emptyset$. Indeed $\{v\geq m\}\subset\R^2\setminus K$ as $m>m_0$ and $m_0\geq v$ on $K$, therefore $K\cap D$ is contained in the endpoints of $\gamma$, which are actually points of $\eta$ and $K\cap\eta=\emptyset$. At any rate we are now in the position of arguing as in Proposition~\ref{p:monotonia-0} to show that 
\[
\frac{1}{r} \int_{V\cap\partial B_r} |\nabla u|^2 \geq \frac{1}{r^2} \int_{V \cap B_r}|\nabla u|^2\, .
\]

\medskip

{\bf Step 2.} Fix now any $x\in \mathbb R^2$ and observe that
$V=\{v>m_0\}\subset\R^2\setminus K$, and that by Proposition~\ref{p:struttura-3}(i), 
\[
\lim_{r\uparrow \infty} O(x,r) = \lim_{r\uparrow \infty} \frac{1}{r} \int_{B_r(x)\setminus K} |\nabla u|^2\, ,
\]
while the latter value equals the constant value
\[
C_\infty := \frac{1}{r} \int_{B_r\setminus K_c} |\nabla u_c|^2 = 1
\]
for a cracktip $(u_c, K_c)$ with terminal point at the origin. Indeed, observe first that 
\[
B_{r-|x|} (0) \subset B_r (x) \subset B_{r+|x|} (0)
\]
and therefore it suffices to show that 
\[
\lim_{r\to\infty} \frac{1}{r} \int_{(\{v> m_0\} \cap B_r) \setminus K} |\nabla u|^2 = C_\infty\, .
\]
Next, introduce $u_r (x) := r^{-1/2} u (rx)$, $K_r := r^{-1} K$, and $v_r (x) := r^{-1/2} v (rx)$ and consider that 
\[
\frac{1}{r} \int_{(\{v> m_0\} \cap B_r) \setminus K} |\nabla u|^2
= \int_{(\{v_r > r^{-1/2} m_0\}\cap B_1)\setminus K_r} |\nabla u_r|^2\, . 
\]
However, up to extraction of subsequences (which for the sake of keeping our notation simple we ignore) $\nabla u_r$ converges strongly in $L^2$ to $\nabla u_c$, while $v_r$ converges (uniformly, given that $\|v_r\|_{C^{1/2}} \leq C$) to the harmonic conjugate $v_c$ of $u_c$ normalized to be $0$ at $0$. Depending on the sign of the cracktip $u_c$, $v_c$ is either nonnegative, or nonpositive. In the first case we have $\{v_c>0\} = \mathbb R^2\setminus K_c$, while in the second we have $\{v_c<0\} = \mathbb R^2\setminus K_c$. In particular $\{v_r > r^{-1/2} m_0\}$ converges to $\{v_c>0\}$. The second case, namely $v_c\leq 0$, is then excluded because the limit of the integrals would be $0$. We thus must have $\{v_c>0\} = \mathbb R^2\setminus K_c$, which in turn implies
\[
\lim_{r\to \infty} \int_{(\{v_r>r^{-1/2} m_0\} \cap B_1)\setminus K_r} |\nabla u_r|^2 
= \int_{B_1\setminus K_c} |\nabla u_c|^2\, .
\]
In particular, if we set 
\[
C (x) := \lim_{r\downarrow 0} O (x,r)\, ,
\]
we then conclude that $C (x)\in [0, 1]$. 
Observe that, arguing as in Proposition~\ref{p:monotonia-0}, if $C(x) = 1$,
then $O'(x,\cdot) \equiv 0$, which in turn implies that:
\begin{itemize}
    \item either $V \cap B_r (x)$ has Lebesgue measure zero for every $r$;
    \item or $B_r(x)\setminus V$ is a straight segment originating at $x$ and $u$ is $\frac{1}{2}$-homogeneous on $V$. 
\end{itemize}
It is immediate to see that the first case is excluded as it would be $C(x)=0$, while we are assuming that $C(x)=1$.
In the second case $(u, K)$ is a cracktip with terminal point at $x$. But then $m_0$ could not be larger than $0$, so the latter conclusion is excluded as well. 
Thus, $C (x)\in [0, 1)$.

Assume now that $C (x)\in (0,1)$. 
Consider then any blow-up $(u_0, K_0, \{p_{kl}\})$ at $x$, limit of some sequence of rescalings $u_{x,r_j}, K_{x,r_j})$ with $r_j\downarrow 0$, and consider the limit $v_0$ of the rescaled harmonic conjugate function 
\[
v_{x,r_j} (y) := r_j^{-\sfrac{1}{2}} (v (r_j y +x) - m_0)\, .
\]
In view of Step~1, it then turns out that 
\[
\frac{1}{r} \int_{B_r \cap \{v_0 >0\}} |\nabla u_0|^2 = C (x) > 0\, \qquad \mbox{for all $r>0$}.
\]
We again argue as in Step 1 and conclude that the last identity yields that either $\{v_0> 0\} \cap B_r$ has measure zero for all $r$ (which is excluded from the fact that $C (x)>0$), or that $u_0$ is the cracktip. In this case, however $v_0$ is the harmonic conjugate of the cracktip, normalized to be $0$ at $0$. Arguing as above, it follows again that either $v_0$ is nonnegative or it is nonpositive. But in the second case $\{v_0>0\}$ would be empty, which in turns gives $C(x)=0$. It thus follows that $v_0$ is nonnegative, which in turn implies $\{v_0>0\} = \mathbb R^2 \setminus K_0$. In particular we have $C (x) =1$, which is also a contradiction.

We thus have established that 
\[
\lim_{r\downarrow 0} O (x,r) = 0
\]
for every $x\in\R^2$.

\medskip

{\bf Step 3.} Consider now $x\in K$ and let $(u_0, K_0, \{p_{kl}\})$ be a blow-up 
of $(u,K)$ at $x$. Let moreover $v_0$ be as above. We distinguish two possibilities:
\begin{itemize}
    \item[(a)] $\{v_0>0\}$ is not empty;
    \item[(b)] $\{v_0>0\}$ is empty.
\end{itemize}
In case (a), being $C(x)=0$, we would have $|\nabla v_0| = |\nabla u_0|=0$ on a nontrivial open set. It therefore follows that $\nabla u_0$ vanishes identically on at least one connected component of $\mathbb R^2\setminus K_0$. Since $K_0$ is nontrivial by the density lower bound, $K_0$ must disconnect $\mathbb R^2$ (cf. Corollary~\ref{c:no-vanishing}). In particular it follows that $(u_0, K_0, \{p_{kl}\})$ is either a pure jump or a triple junction. The $\varepsilon$-regularity theory for these two cases allows then to conclude the main claim of the proposition. 

\medskip
{\bf Step 4.} We are now ready to conclude the proof of the Proposition: we just need to handle case (b). Consider $(u_0, K_0, \{p_{kl}\})$ and $v_0$ as above. Consider in addition the closure $H$ of $\{v> m_0\}$ and observe that $H\cap B_r(x)$ is connected for every $r>0$ by Proposition~\ref{p:struttura-3}, while it also contains $x$. Consider now its rescalings $H_j := r_j^{-1} (H-x)$ and assume, up to subsequences, that it converges locally in the sense of Hausdorff to some closed set $H_0$. $H_0$ is connected 
and moreover it contains the origin. Observe that $v_0\geq 0$ on $H_0$. Observe that, since we are in case (b) above, $v_0 \leq 0$. Therefore, $v_0=0$ on $H_0$.
Now, if exists $z\in H_0\setminus K_0$, then $v_0$ is harmonic in a neighborhood of $z$, and hence constant. We can thus argue as in the previous step to conclude that $(u_0, K_0,\{p_{kl}\})$ is either a pure jump or a triple junction. Otherwise we have $H_0\subset K_0$. Since $H_0$ must be unbounded, we conclude that it contains some regular jump point $z$. Let $B_\rho (z)$ be such that $B_\rho (z)\cap H_0$ is a smooth arc $\gamma$ and denote by $B_\rho^+ (z)$ and $B_\rho^- (z)$ the connected components of $B_\rho (z)\setminus H_0$. By the unique continuation for harmonic functions at smooth boundaries, $\nabla u_0^+$ cannot vanish identically on $\gamma$, unless $u_0$ is constant on $B_\rho^+ (z)$. The latter situation however falls back in what already analyzed in Step 3 and we can ignore it. We thus conclude that $\nabla v_0^+$ does not vanish indentically on $\gamma$ and likewise $\nabla v_0^-$ does not vanish indentically on $\gamma$ either, which means of course that $\frac{\partial v_0^\pm}{\partial n}$ do not vanish identically on $\gamma$. By possibly changing $z$ and making $\rho$ smaller we can assume that they do not vanish on $\gamma$ at all
(recall that $v_0=0$ on $H_0$).

Consider now that, by the $\varepsilon$-regularity theory, $K_{x,r_j} \cap B_\rho (z)$ is converging smoothly to $\gamma\cap B_\rho (z)$. So for a sufficiently large $j$ it is an arc $\gamma_j$ very close to $\gamma$. Now, the value of $v_j$ over $\gamma_j$ is converging to $0$. On the other hand we also know that $v_j \leq 0$ on $\gamma_j$, because $v\leq m_0$ on $K$.

Let now $B_j^+$ and $B_j^-$ be the two connected components in which $B_\rho (z)$ is divided by $\gamma_j$.
Consider that $|\frac{\partial v_j^{\pm}}{\partial n}|\geq c>0$ on $\gamma_j$ for some constant $c$. At the same time the second derivatives $D^2 v_j$ are uniformly bounded on the $B_j^\pm$. Observe that $\frac{\partial v_j^\pm}{\partial n}$ cannot change sign on $\gamma_j$, which is connected. We thus can examine the following cases, depending on their signs:
\begin{itemize}
\item $\frac{\partial v_j^+}{\partial n}\geq c>0$ on $\gamma_j$ or $\frac{\partial v_j^-}{\partial n}\geq c>0$ on $\gamma_j$, for $j$ sufficiently large. Thanks to the uniform bound on the second derivatives and to the smooth convergence towards $\gamma$, we can choose $\rho$ sufficiently small but independent of $j$, so that $B_\rho^+ (z)$ (or $B_\rho^- (z)$), is contained in $H_0$, which is a contradiction, because $H_0$ does not contain interior points. 
\item $\frac{\partial v_j^+}{\partial n}\leq - c$ and $\frac{\partial v_j^-}{\partial n}\leq -c <0$ on $\gamma_j$ for $j$ sufficiently large. Again thanks to the bound on the second derivatives, if we then choose $\rho$ sufficiently small and $j$ sufficiently large, we must have $v_j < 0$ on $B_j^+$ and $B_j^-$. In particular, since $H_j$ is the closure of $\{v_j>0\}$, we conclude that $H_j \cap B_\rho (z) = \emptyset$. But that is also a contradiction because $z$ is in $H_0$, which is the Hausdorff limit of $H_j$ in $\overline{B}_\rho (z)$.
\end{itemize}

\section{Proof of Theorem~\ref{t:Bonnet-David}}\label{s:gira}

In this section we complete the proof of Theorem~\ref{t:Bonnet-David}. We thus fix a nonelementary global minimizer $(u,K)$, an harmonic conjugate $v$ as in Assumption~\ref{a:Bonnet-David}, and set $m_0:= \max_K v$. Because of Corollary~\ref{c:m_0=0} we just need to show that $m_0=0$. Towards a contradiction we assume $m_0>0$ and we can now use Proposition~\ref{p:componente-speciale} to fix a bounded connected component $G$ of $K$ with $\mathcal{H}^1 (G) >0$ and on which $v$ takes the constant value $m_0$. We next recall that, because of Proposition~\ref{p:componente-speciale} and Corollary~\ref{c:non-terminal=regular}, 
every point $p\in G$ can be assigned to one of the following three categories:
\begin{itemize}
    \item Pure jump points $p$: for each of them there is a disk $B_\rho (p)$ such that $B_\rho (p)\cap K = B_\rho (p)\cap G$ is a smooth arc subdviding $B_\rho (p)$ into two topological disks. From Proposition~\ref{p:variational identities 2} we gather immediately that this arc is smooth.
    \item Triple junctions $p$: for each of them there is a disk $B_\rho (p)$ such that $G\cap B_\rho (p)= K\cap B_\rho (p)$ is diffeomorphic to three radii of $B_\rho (p)$ joining at $p$ at $120$ degrees.
    \item Terminal points $p$ (according to Corollary~\ref{c:non-terminal=regular} then $p\notin \partial\{v>m_0\}$ in this case).
\end{itemize}
Obviously the triple junctions are countably many, because they form a discrete set: if $x$ is a triple junction, by the $\varepsilon$-regularity theorem already established there is a punctured disk $B_r (x)\setminus \{x\}$ in which the set $K$ is regular, in particular $x$ is the only triple junction in $B_r (x)$.
However, we caution the reader that for the terminal points the only information available is that they form a compact set of $\mathcal{H}^1$ measure zero.

\subsection{Touring \texorpdfstring{$G$}{G} counterclockwise}
Without loss of generality we will assume that $\mathcal{H}^1 (G) = \pi$ (this can be achieved by a simple rescaling argument). We now wish to find a surjective Lipschitz map $\alpha: \mathbb S^1\to G$ with the following properties:
\begin{itemize}
    \item[(i)] Each jump point of $G$ has exactly two counterimages;
    \item[(ii)] Each triple junction point of $G$ has exactly three counterimages;
    \item[(iii)] Each terminal point has exactly one counterimage.
\end{itemize}
The idea is that such a map ``goes around $G$''. In fact it can be shown that, up to a change of phase in $\mathbb S^1$ and of orientation, the map is unique. We want then to select the one which ``goes around $G$ clockwise''.

Before detailing the construction of this map, it is useful to break down $G$ as the union of maximal smooth open subarcs: a $C^\infty$ open subarc of $G$ is the image of a smooth injective $\eta:(0, \ell)\to K$, parametrized by arc-length, and a maximal one is a $C^\infty$ closed subarc of $G$ which is not a strict subset of any other closed subarc. Observe that any maximal smooth open subarc consist only of jump points and has a unique arc-length parametrization up to orientation. If $\eta: (0, \ell)\to K$ is one such arc-length parametrization, arbitrarily chosen, then $\eta$ can be extended continuously to both $0$ and $\ell$. The corresponding values, which we will call extrema of the subarc, cannot be equal ($G$ contains no loops by Lemma~\ref{l:pockets}) and cannot be jump points (otherwise the subarc would not be a maximal smooth one). Therefore, they are either terminal points, or triple junctions (and of course one of them could be a triple junction, while the other is a terminal point). From now on when treating a maximal smooth subarc of $G$ we will always understand that its parametrization includes the two endpoints, even though the extension is not at all guaranteed to be smooth (the best we can say is that it is Lipschitz continuous at extrema which are terminal points, and $C^2$ at extrema which are triple junctions). The underlying idea about the map $\alpha$ is quite simple to define directly when the number of smooth maximal arcs is finite. We do not give a precise definition of the map in this case, since it can be easily inferred from the general construction below, we rather refer to reader to Figure~\ref{f:giro-attorno} for a visual description of what $\alpha$ does.

\begin{figure}
\begin{tikzpicture}
\draw[very thick] (-0.5,1) to [out=270,in=120] (0,0);
\draw[very thick] (0,0) to [out = 0, in = 180] (0.8,0.3) to [out = 0, in=180] (2,0);
\draw[very thick] (2,0) to [out=60,in = 90] (3,0.4);
\draw[very thick] (-0.75,-1.5) to [out=90, in=240] (0,0);
\draw[very thick] (2,0) -- (2.5,-1);
\draw[-stealth] (-0.1,0.5) -- ({-0.1+0.3*cos{-60}},{0.5+0.3*sin{-60}});
\draw[-stealth] (0.7,0.5) -- (1,0.5);
\draw[-stealth] (2.2,0.4) -- ({2.2+0.3*cos{40}},{0.4+0.3*sin{40}});
\draw[stealth-] (2.3,0.1) -- ({2.3+0.3*cos{40}},{0.1+0.3*sin{40}});
\draw[-stealth] (2.5,-0.5) -- ({2.5+0.3*cos{-60}},{-0.5+0.3*sin{-60}});
\draw[stealth-] (2.1,-0.8) -- ({2.1+0.3*cos{-60}},{-0.8+0.3*sin{-60}});
\draw[stealth-] (1.5,-0.1) -- ({1.5+0.3*cos{-15}},{-0.1+0.3*sin{-15}});
\draw[stealth-] (0.2,-0.2) -- ({0.2+0.3*cos{25}},{-0.2+0.3*sin{25}});
\draw[stealth-] (-0.4,-1.2) -- ({-0.4+0.3*cos{60}},{-1.2+0.3*sin{60}});
\draw[-stealth] (-0.8,-0.8) -- ({-0.8+0.3*cos{60}},{-0.8+0.3*sin{60}});
\draw[stealth-] (-0.6,0.5) -- ({-0.6+0.3*cos{-55}},{0.5+0.3*sin{-55}});
\end{tikzpicture}
\caption{\label{f:giro-attorno} The picture is a visualization of the map $\alpha$ in case the number of maximal smooth open subarcs of $G$ is finite. The map is particularly easy to define if the arcs are finitely many and very regular: in that case $\alpha$ is the limit as $t\downarrow 0$ of the clockwise arc-length parametrizations of the sets $\{x: {\rm dist}\, (x, G) = t\}$.}
\end{figure}
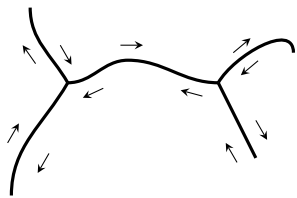

We now describe in the detail the algorithm which produces the map $\alpha$.

\medskip

{\bf Construction of the map $\alpha$.} Fix a maximal smooth subarc (its choice is not important). We denote it by $G_0$ and let $p_0$ the point which divides it into two arcs of equal length. Without loss of generality we assume that $p_0=0$ and that the tangent to $\eta_0$ at $0$ is horizontal. We then fix an arclength parametrization $\eta_0: [-L_0, L_0]\to\R^2$ with $\eta_0 (0) =0$ and $\dot{\eta}_0 (0) = (1,0)$. We start defining a map $\alpha_0: [-L_0, 3L_0] \to \mathbb R^2$ as follows:
\[
\alpha_0 (\theta) = \left\{
\begin{array}{ll}
\eta_0 (\theta) \qquad &\mbox{if $-L_0 \leq \theta \leq L_0$}\\
\eta_0 (2L_0 - \theta) \qquad &\mbox{otherwise}\, .
\end{array}\right.
\]
$\alpha_0$ is a map as in Figure~\ref{f:giro-attorno-2}. 

If both extrema of $G_0$ are terminal points, then $G= \overline{G}_0$ and thus $L_0 = \frac{\pi}{2}$. In this case the map $\alpha$ is given by $\alpha_0$, after we indentify $-\frac{\pi}{2}$ with $\frac{3\pi}{2}$ so that the domain of $\alpha$ is $\mathbb S^1$.

Otherwise, at least one of the extrema of $\eta_0$ is a triple junction. If only one of them is, we consider the two additional maximal subarcs which have that particular extremum in common, if both of them are, then we consider the four additional subarcs. Either way, 
the union of these subarcs and the initial one $\eta_0$ will be denoted by $G_1$. $G_1$ is, schematically, a tree with finitely many nodes, in the second alternative it will look like the set in Figure~\ref{f:giro-attorno}. We now define a new map $\alpha_1: [-a_1, b_1] \to \mathbb R^2$ which is ``going around $G_1$ clockwise'' by putting together pairs of oppositely oriented parametrizations of each maximal subarc forming $G_1$. The parametrizations are chosen and joined in a unique way, once we prescribe that $[-L_0, L_0]\subset [-a_1, b_1]$ and that it must coincide with $\alpha_0$ on $[-L_0, L_0]$. 
Note that $G_1$ has a finite number of terminal points. If all of them are terminal points of $K$ we then stop the procedure. Otherwise we select all the ones which are triple junctions and add to $G_1$ all the maximal subarcs that have one of them as extrema. This forms $G_2$ and we can then construct the corresponding map $\alpha_2$. We keep iterating this algorithm: if it stops after finitely many times $N$, we then set $\alpha =\alpha_N$. If the procedure never stops, then we can easily check that $\alpha_i$ converges to a unique map $\alpha$
(recall that $\mathcal H^1(G)<\infty$, and that the maximal subarcs are parametrized by arc-length).

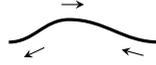
\begin{figure}
\begin{tikzpicture}
\draw[very thick] (0,0) to [out = 0, in = 180] (0.8,0.3) to [out = 0, in=180] (2,0);
\draw[-stealth] (0.7,0.5) -- (1,0.5);
\draw[stealth-] (1.5,-0.1) -- ({1.5+0.3*cos{-15}},{-0.1+0.3*sin{-15}});
\draw[stealth-] (0.2,-0.2) -- ({0.2+0.3*cos{25}},{-0.2+0.3*sin{25}});
\end{tikzpicture}
\caption{\label{f:giro-attorno-2} The map $\alpha_0$ going around $G_0$.}
\end{figure}

\medskip

{\bf Surjectivity of $\alpha$.} Once we show that $\alpha$ is surjective on $G$, it is a simple fact that it satisfies all the properties above. 

In order to show its surjectivity, let $\tilde{G}$ be the image of $\alpha$ and observe at first the following property:
\begin{itemize}
    \item[(R1)] if $p,q\in G$ are two triple junctions extrema of the same maximal open subarc of $G$, then $p\in \tilde{G}$ if and only if $q\in \tilde{G}$. 
\end{itemize}
In fact, assume $p\in \tilde{G}$. If $p\in G_i$ for some $i$, then certainly $q\in G_{i+1}\subset \tilde{G}$ by construction. On the other hand if $p$ were not included in some $G_i$, then there would be a sequence of points $p_i\in G_i\setminus G_{i-1}$ converging to $p$. $p_i$ would belong to an arc $\eta_i$ added at the $i$-th stage. Since the total length of $\tilde{G}$ is finite, the length $L_i$ must converge to $0$ as $i\uparrow \infty$. Note however that each such $L_i$ must have at least one extremum which is a triple junction. But then $p$ would be a limit of triple junctions, which is not possible since the latter are isolated. 
We next observe the following further simple consequence of the same idea:
\begin{itemize}
    \item[(R2)] If $p\in \tilde{G}$ is a triple junction, then $\tilde{G}$ contains all three maximal smooth open subarcs of $K$ which have $p$ as one endpoint.
\end{itemize}
Fix now a point $q\in G$ and observe that, since $G$ is a connected, there is an injective arc $\gamma:[0,L]\to G$ such that $\gamma (0) =0$ and $\gamma (L) = q$, which we can assume to be parametrized by arc-length. If $\gamma ([0,L])$ contains no triple junctions, then $\gamma$ is contained in $\overline{G}_0$ and hence $q\in \overline{G}_0\subset \tilde{G}$. Otherwise let $P\subset [0,L]$ be the subset of points $s$ such that $\gamma (s)$ is a triple junction. This set must be discrete in $[0, L)$: we can order its elements as $p_1<p_2 < \ldots$. Notice that $p_1$ is an extremum of $G_0$ and thus contained in $\tilde{G}$ by (R1). In particular by (R1) and (R2) we must have that $p_2\in \tilde{G}$ as well. By induction all $p_k$ belong to $\tilde{G}$. But then we have two possibilities:
\begin{itemize}
    \item The number of $p_k$'s is finite: if $p_N$ is the largest of them, then $q=\gamma (L)$ is in the closure of a maximal smooth open subarc with endpoint $p_N$: by (R2) this subarc belongs to $\tilde{G}$ and since the latter is closed, $q\in \tilde{G}$;
    \item The number is infinite: in this case $p_k\to q$ and since $\tilde{G}$ is closed and $p_k\in \tilde{G}$ for every $k$, we conclude again $q\in \tilde{G}$.
\end{itemize}

\subsection{Final contradiction} With the map $\alpha$ at hand, we define 
\[
J := \{t\in \mathbb S^1 : \mbox{$\alpha (t)$ is a jump point}\}\, .
\]
For each $t\in J$ we then let $e (t) := \dot\alpha (t)$. Recall that the latter is well defined (as $\alpha$ is indeed smooth on $J$) and has unit norm. We then denote by $n (t)$ its counterclockwise rotation by 90 degrees. Recall that for every $t\in J$ there is one (and only one) other element $s\in J$ such that $\alpha (s) = \alpha (t)$. By construction we then have $e(s)= - e(t)$ and $n(s) = - n (t)$. 
Next recall that for every $t\in J$, if $p= \alpha (t)$, then there is a disk $B_\rho (p)$ in which $G$ is a smooth arc dividing $B_\rho (p)$ in two topological disks. Observe that the restrictions of $v$ and $u$ to any of these two open sets have smooth extensions to its closure (in particular smooth extensions up to $G\cap B_\rho (p)$). We can therefore define
\begin{align}\label{e:tracce}
f(t) &:= \lim_{\delta\downarrow 0} n (t) \cdot \nabla v (\alpha (t) + \delta n (t))\, ,\\
h (t) &:= \lim_{\delta\downarrow 0} u (\alpha (t) + \delta n (t))\, .
\end{align}
Clearly $f(t)$ tells us whether at the point $p$ the function $v$ is decreasing or increasing in the direction $n(t)$. But it is also easy to check that $h' (t) = f(t)$
(recall that $e(t)=\dot{\alpha}(t)$). Observe in addition that $h$ has a continuous extension to $\mathbb S^1$. Indeed for a point $t\not \in J$ we just have two possibilities.
\begin{itemize}
    \item $\alpha (t)$ is a triple junction. In this case $t$ is an isolated point of $\mathbb S^1\setminus J$ and $h$ can be continuously extended at $t$ by the regularity theory at triple junctions.
    \item $\alpha (t)$ is a terminal point. But then by Proposition~\ref{p:continuity} $u$ is continuous at $\alpha (t)$ and 
    \[
    |u (\alpha (t)) - h (s)|\leq C |t-s|^{\sfrac{1}{2}}
    \]
    for all $s\in J$ and a universal constant $C$.
\end{itemize}

The final contradiction argument then hinges on the following pivotal lemma.

\begin{lemma}\label{l:BD-pivotal}
The continuous function $f: J \to \mathbb R$ defined above has the following properties:
\begin{itemize}
\item[(i)] $\{f=0\}$ has empty interior;
\item[(ii)] The closure $S^+$ of $\{f>0\}$ in $\mathbb S^1$ is a connected arc.
\end{itemize}
\end{lemma}

Before coming to the proof of Lemma~\ref{l:BD-pivotal} we show how to conclude the proof of Theorem~\ref{t:Bonnet-David} from it. First of all, since $v (\alpha (t)) = m_0$ for every $t$, when $t\in \{f>0\}$ clearly there is a $\varepsilon >0$ with the property that 
$v (\alpha (t) + \sigma n (t))> m_0$ for every $\sigma\in (0, \varepsilon)$. In particular we conclude that $\alpha (t)\in \partial \{v> m_0\}$. Moreover $\{f>0\}$ is dense in $S^+$ and therefore $\alpha (S^+)\subset \partial \{v>m_0\}$. Proposition~\ref{p:componente-speciale} (ii) implies that $\alpha (S^+)$ consists all of jump points and triple junctions. Since $G$ has at least two terminal points (cf. (i) Proposition~\ref{p:struttura-2}), we conclude that the complement of $S^+$ is not empty, and we will denote it by $S^-$. 

At this point it is convenient to change phase to the parametrization $\alpha$ so that $\alpha([0, M]) = S^+$ and $\alpha(M, 2\pi) = S^-$ (cf. (ii) Lemma~\ref{l:BD-pivotal}). 
Clearly, as $h'=f$, the function $h$ is increasing on $S^+$ and decreasing on $S^-$: $h(0)$ is then the minimum and $h (M)$ is the maximum. Observe also that in both intervals $(0, M)$ and $(M, 2\pi)$ the monotonicity of $h$ is strict, because the zero set of $h'$ has empty interior.

Since $[0, M]$ contains no terminal points, we can subdivide $[0, M]$ as $[0, t_1]\cup [t_1,t_2]\cup \ldots \cup [t_N, M]$, where each $t_i$ is a triple junction and each open interval $(t_0, t_1), \ldots , (t_N, M)$ contains no triple junctions. Observe that $\alpha$ is injective on $[0,M]$ (because $\alpha ([0,M])$ contains no terminal point).

The situation is particularly easy when $(0, M)$ itself contains no triple junctions. $\alpha ((0, M))$ has a second, disjoint, counterimage $(s_0, s_0+M)$. This counterimage must be contained in $S^-$. We therefore can draw the following conclusions:
\begin{itemize}
    \item $h$ is strictly decreasing on $[s_0, s_0+M]$. 
    \item $h (0) < h(s_0+M) < h (s_0) < h (M)$ (the inequalities are justified as  $h$ is strictly decreasing on $S^-$, for the first we also use that $h(0)=h(2\pi)$);
    \item $\alpha (s) = \alpha (s_0+M-s)$ (because recall that $\alpha|_{[0,M]}$ and $\alpha|_{[s_0, s_0+M]}$ are arclength parametrizations of the same smooth arc with opposite orientations.
\end{itemize}
Now, define the function
\begin{equation}\label{e:def-k}
k (s) := h (s) - h (s_0+M-s)
\end{equation}
and note that $k$ is strictly increasing on $[0, M]$ while $k (0) < 0 < k(M)$.
In particular, $k(0)=[u](\alpha(0))$ and $k(M)=-[u](\alpha(M))$.
Therefore $k$ must have a zero $z_0$ in $(0, M)$. Such zero $z_0$ corresponds to a pure jump point $p = \alpha (z_0)$ with the property that the two traces $u^+ (p)$ and $u^- (p)$ on the two sides of $K$ at $p$ are equal. But since $u$ is an {\em absolute minimizer} this is not possible
(note that here we are using the property that at a jump point the one-sided traces of $u$ have to differ: this property does not hold for restricted minimizers).

We now analyze the more general case in which $(0, M)$ is not contained in $J$. We then let $t_1, \ldots, t_N$ be as above and set $t_0 =0$ and $t_{N+1} = M$. We then find points $s_{N+1} < s_N^+ < s_N^- < s_{N-1}^+ < s_{N-1}^- < \ldots < s_1^- < s_0$ such that:
\begin{itemize}
    \item $\alpha$ maps $[s_{N+1}, s_N^+]$ onto $\alpha ([t_N, t_{N+1}])$, $[s_1^-, s_0]$ onto $\alpha ([t_0, t_1])$, and $[s_j^-, s_{j-1}^+]$ onto $\alpha ([t_{j-1}, t_j])$;
    \item $\alpha (s_{N+1}) = \alpha (t_{N+1})$, $\alpha (s_0) = \alpha (t_0)$, and $\alpha (s_j^-)= \alpha (s_j^+) = \alpha (t_j)$.
\end{itemize}
Moreover by (ii) Lemma~\ref{l:BD-pivotal}:
\begin{itemize}
    \item $h(t_0) < h (s_0)$;
    \item $h (s_{N+1}) < h (t_{N+1})$;
    \item $h (t_0) < h (t_1) < \ldots < h (t_{N+1})$;
    \item $h (s_0) < h (s_1^-) < h (s_1^+) < \ldots < h (s_N^-) < h (s_N^+) < h (s_{N+1})$.
\end{itemize}
In order to streamline the rest of the discussion we use the convention that $s_0^+=s_0^-=s_0$.

Assume now there is $j\leq N$ such that $h (s_j^-) \leq h (t_j)$. In that case we let $j$ be the smallest of them. Observe then that $h (s_{j-1}^+) \geq h (s_{j-1}^-) > h (t_{j-1})$. We then set $a:= t_{j-1}$, $a+d := t_j$, $b:= s_{j-1}^+$ and observe that $s_j^- = b+d$. Our function $k$ is now defined on $[a, a+d]$ as 
\begin{equation}\label{e:up-and-down}
k (s) = h (s) - h (b+d - (s-a))
\end{equation}
Once again $k$ is strictly increasing on $[a, a+d]$ and moreover $h (a) < 0 \leq h (a+d)$. Then $h$ must have a zero $z_0$ in $(a, a+d]$. If this zero is smaller than $a+d$, then we are in the exact same situation as in the case analyzed when $\alpha ((0, M_0))$ contains no triple junction. In case the zero is $a+d$, we then find a triple junction point at which two of the three traces of $u$ coincide: this again (thanks to the regularity theory at triple junctions) contradicts the absolute minimality of $u$. 

If there is no $j\leq N$ such that $h (s_j^-) \leq h (t_j)$, then obviously $h (s_N^+) > h (s_N^-) > h (t_N)$. But then the exact argument just given can be replicated defining $a=t_N$, $b+d = t_{N+1}$, $s_N^+ = b$. Then $s_{N+1} = b+d$ and defining $k$ as in \eqref{e:up-and-down} we find this time $h (a) < 0 < h (a+d)$. Thus we can repeat the very same argument above and conclude that this case too leads to a contradiction.

\subsection{Proof of Lemma~\ref{l:BD-pivotal}} In order to conclude our proof of Theorem~\ref{t:Bonnet-David} we are thus left with giving an argument for Lemma~\ref{l:BD-pivotal}.
First of all, if $\{f=0\}$ contains an interior point, then there is a jump point $p\in G$ and a neighborhood $B_\rho (p)$ with the properties that:
\begin{itemize}
    \item $G\cap B_\rho (p)$ divides $B_\rho (p)$ into two topological disks, $B^+$ and $B^-$;
    \item the trace of $\nabla u$ on $G\cap B_\rho (p)$ from one of the two sides $B^+$ or $B^-$ vanishes identically (as $\frac{\partial u}{\partial n}^\pm=0$ on $K$). 
\end{itemize}
Assume without loss of generality that the side is $B^+$. The unique continuation of harmonic functions then implies that $\nabla v$ vanishes identically on $B^+$. But then $u$ would be constant on $B^+$ and we know this is not possible in view of Corollary~\ref{c:no-vanishing} being $(u,K)$ nonelementary.

In order to prove the second statement of the lemma, consider first $t\in J$ with $f(t)>0$ and let $s$ be the only other point on $\mathbb S^1$ such that $\alpha (s) = \alpha (t)$. We want to show that $f(s) \leq 0$. Assume indeed $f(s)>0$. We can then select $\delta >0$ such that $v$ is strictly increasing on the two segments $[\alpha (s), \alpha (s) + \delta n (s)]$ and $[\alpha (t), \alpha (t)+\delta n (t)]$. For any $m> m_0= v (\alpha (s))= v (\alpha (t)$ sufficiently close to $m_0$ we find then $a, b\in (0, \delta)$ such that
\[
v (\alpha (t) + a n(t)) = v (\alpha (s) + b n (s)) = m\, .
\]
For an appropriately chosen $m$ we can apply the conclusion (v) of Proposition~\ref{p:struttura-3}. So the two points $p= \alpha (t) + a n(t)$ and $q = \alpha (s) + b n(s)$ lie in the properly embedded unbounded line $\{v=m\}$. In particular $p$ and $q$ determine a bounded Jordan arc $\beta$ on $\{v=m\}$. The union of the arc $\beta$ with the segment $[p,q]$ (which is directed along $n(t)=-n(s)$ and hence contains the point $\alpha (t)=\alpha (s)$) is a simple curve, which by the Jordan's Theorem bounds a disk $D$. Note that on $\partial D$ we have $v\geq m_0$ and thus $v\geq m_0$ on $D$. In particular $D\cap K\subset \partial \{v>m_0\}$. But $K\cap \partial D$ consists only of the point $\alpha (s)$. In particular it would follow that $D\cap K$ contains a terminal point, contradicting Proposition~\ref{p:componente-speciale}. 

Having established the claim above, we are now ready to prove the second conclusion of the lemma. Towards a contradiction we assume that there are points $t_1, t_2, t_3, t_4$ in $J\subset \mathbb S^1$ with the property that $t_2$ and $t_4$ belong to the two distinct arcs of $\mathbb S^1$ delimited by $t_1$ and $t_3$ and at the same time $f(t_3), f(t_1)>0$ and $f(t_2), f(t_4) < 0$. Moreover, because of the first statement of the lemma, by perturbing $t_1$ and $t_3$ we can assume that, if $s_1$ and $s_3$ are the two other points such that $\alpha (s_1) = \alpha (t_1)$ and $\alpha (s_3)= \alpha (t_3)$, then $f(s_1), f(s_3)<0$

Consider now as above a $\delta >0$ such that $v$ is strictly increasing on both segments $S_1 = [\alpha (t_1), \alpha (t_1) + \delta n (t_1)]$ and $S_3 = [\alpha (t_3), \alpha (t_3) + \delta n (t_3)]$. Fix as above an $m>m_0$ to which the conclusion (v) of Proposition~\ref{p:struttura-3} apply and let $p_1$, $p_3$ be the only intersections of $S_1$ and $S_3$ with $\{v=m\}$. Again as above let $\beta\subset \{v=m\}$ be the Jordan arc delimited by $p_1$ and $p_3$. Meanwhile let $\gamma$ be a Jordan arc connecting $\alpha (t_1)$ and $\alpha (t_3)$ in $G$ (which exists by Lemma~\ref{l:connessi_per_archi}). The curve $\gamma \cup \beta \cup [\alpha (t_1), p_1)] \cup [\alpha (t_2), p_2]$ is simple and therefore it delimits a disk $D$. Arguing as above, $v\geq m_0$ in $D$ and thus $D$ cannot intersect $K$. In particular it follows that $\gamma = K\cap \partial D$ contains a finite number of triple junctions. So we can chop $\gamma$ as the union of closed arcs
\[
\gamma_0\cup \gamma_1 \cup \ldots \cup \gamma_N\cup \gamma_{N+1}
\]
such that:
\begin{itemize}
    \item The right endpoint $p_i$ of $\gamma_i$ is the left endpoint $q_{i+1}$ of $\gamma_{i+1}$;
    \item For $i\in \{1, \ldots , N\}$ each $\gamma_i$ is the closure of a maximal smooth open subarc $\gamma_i$ joining the triple junctions $q_i$ and $p_i$;
    \item $\gamma_0$ joins $q_0 = \alpha (t_1)$ to the triple junction $p_0=q_1$ while $\gamma_{N+1}$ joins the triple junction $p_N = q_{N+1}$ the $p_{N+1} = \alpha (t_3)$, but both $\gamma_0$  and $\gamma_{N+1}$ contain no other jump points.
\end{itemize}
We parametrize $\gamma$ by arclength so that, while we are following it from $q_0$ to $p_{N+1}$, the counterclockwise rotation of $\dot\gamma$ by 90 degrees points ``inwards'' with respect to $D$

It is now easy to see that there is an interval $I\subset \mathbb S^1$ over which $\alpha$ is injective and $\alpha (I) = \gamma$ and that we can choose it so that $n(s)$ always agrees with the counterclockwise rotation of $\dot{\gamma} (\sigma)$ for the only $\sigma$ with $\gamma (\sigma) = \alpha (s)$. But then clearly the two extrema of the intervals must be $t_1$ and $t_3$, because $n(s_1)$ and $n(s_3)$ point outwards. We conclude that $g$ is never negative
over the interval $I$. So neither $t_2$ nor $t_4$ can belong to it, and rather we have $t_2, t_4\in \mathbb S^1\setminus I$. This is however precisely in contradiction to our initial assumption.

\chapter{Epsilon regularity at the cracktip}\label{ch:crack}

This chapter is devoted to the final part in the proof of Theorem~\ref{t:main}. 
In view of the previous chapters, and in particular of Corollary~\ref{c:connectedness} and of the $\varepsilon$-regularity of pure jumps, to establish the case (a) of Theorem~\ref{t:main} it will suffice to prove the following statement.
\begin{theorem}\label{t:final-cracktip}\label{T:FINAL-CRACKTIP}
There are positive constants $\varepsilon_0,\alpha_0$ and $r_0$ with the following property. Assume that $(u, K)$ is a critical point of $E_\lambda$ in $B_3$ and that:
\begin{itemize}
    \item[(i)] \eqref{e:g-and-lambda} holds;
    \item[(ii)] $K\cap B_2$ consists of a single Jordan arc $\gamma$ with endpoints $0$ and $p\in \partial B_2$;
    \item[(iii)] $\gamma$ is $C^{1,1}_{{\rm loc}}$ in $B_2\setminus \{0\}$;
    \item[(iv)] $\dist_H (K, [0, p]) \leq \varepsilon_0$.
\end{itemize}
Then, up to a suitable rotation of coordinates, $K\cap [0, r_0]^2$ is given by $\{(t, \psi (t)): t \in [0, r_0]\}$ for some $C^{1,\alpha_0}$ function $\psi: [0, r_0]\to \mathbb R$ with $\psi (0) = \psi' (0) =0$. If $\lambda =0$ then, in addition, $\psi\in C^{2,\alpha_0}$ and $\psi'' (0) =0$.
\end{theorem}
{Let us now briefly describe the contents of the Chapter. In Section~\ref{s:exp-rep} we  
rescale and reparametrize $u$ and $\gamma$ to obtain in the new variables a nonlinear evolution system from the outer and inner variation identities satisfied by $(u,K)$. 
Theorem~\ref{t:final-cracktip} is then reduced to prove suitable decay estimates for solutions of the latter. 
To this aim, Section~\ref{s:first linearization} contains a first linearization of the evolution system around the cracktip, the tangent couple to $(u,K)$ dictated by the assumptions in Theorem~\ref{t:final-cracktip}.
Section~\ref{s:spectral analysis} is then devoted to the spectral analysis of the linearized system in order to rewrite conveniently its solutions. Taking advantage of this, in Section~\ref{s:tre-anelli} we establish a notable property for such solutions, called in what follows linear three annuli property, 
that implies a suitable decay of the corresponding coefficients. 
A crucial ingredient for this, is the linearized version of the boundary variations in Proposition~\ref{p:boundary-variations} in order to exclude some slowly decaying solution of the linearized system itself (cf. \eqref{e:zeta_0}). 
In turn, a nonlinear three annuli property for solutions to the nonlinear system is used in Section \ref{s:second linearization} to improve upon the above mentioned first linearization, and get the appropriate decay estimates 
to finally prove Theorem~\ref{t:final-cracktip} in Section \ref{s:exp-rep}. 

For the sake of simplifying the calculations, we will actually not work with $u$, but rather with its harmonic conjugate $w$ if $\lambda=0$, and a suitable substitute if $\lambda>0$.\index{harmonic conjugate@harmonic conjugate}\index{conjugate, harmonic@conjugate, harmonic}}
Therefore, we need some preparatory work.
Thanks to Lemma~\ref{l:V-tecnico1} $u$ is continuous at $0$ and we may assume without loss of generality $u(0)=0$. Thus, by Corollary~\ref{c:connectedness} and the $\varepsilon$-regularity at pure jumps we infer (cf. again Lemma~\ref{l:V-tecnico1}) that\footnote{Here and in what follows we adopt the notation $\varphi\lesssim\psi$ if there is $C>0$ such that $\varphi\leq C\psi$.}  
\begin{equation}\label{e:stima tip 1}
|u(x)|+|x||\nabla u(x)|\lesssim|x|^{\sfrac12}\,.
\end{equation}

Moreover, in view of Proposition~\ref{p:variational identities 2}
$u$ has $C^{1,1^-}$ extensions on each side of $K\cap B_2$\footnote{namely, $u$ has extensions on each side of $K\cap B_2$ that are $C^{1,\alpha}$ for every $\alpha<1$.} and the variational identities \eqref{e:outer}-\eqref{e:inner}  imply the following three conditions
\begin{align}
&\Delta u = \lambda (u-g) \qquad \mbox{on $B_2\setminus K$}\label{e:Euler harmonic g 1}\\
&\frac{\partial u}{\partial \nu} = 0 \qquad \mbox{on $K$}\label{e:Euler Neumann g 1}\\
& \kappa = - (|\nabla u^+|^2- |\nabla u^-|^2) - \lambda (|u^+-g_K|^2-|u^--g_K|^2)\qquad 
\mbox{$\mathcal{H}^1$ a.e. on $K\cap B_2$.} \label{e:Euler curvature g 1}
\end{align}
We stress that the equivalence in Proposition~\ref{p:variational identities 2} has been obtained assuming $K$ to be locally a $C^{1,1}$ graph. We do not know that this property holds at the origin, the loose end, but we know it on $K\setminus \{0\}$.
Indeed, Corollary~\ref{c:connectedness} provides only a parametrization of $K$ in polar coordinates smooth up to the tip excluded, i.e. $\gamma\in C^{1,1}((0,2))$ such that 
\begin{equation}\label{e:K param}
K=\{\gamma:[0,2]\to B_2:\,\gamma(r)=r(\cos\alpha(r),\sin\alpha(r))\}\,.
\end{equation}
On the other hand, it is reductive to focus only on \eqref{e:Euler Neumann g 1}-\eqref{e:Euler curvature g 1} and it turns out that we can deduce more pieces of information from \eqref{e:outer}-\eqref{e:inner} using  
Proposition~\ref{p:boundary-variations}, which will play a key role in our analysis. 
In particular, we need the following simple consequence of 
Corollary~\ref{c:boundary-variations} in which we consider the situation described
in item (ii) above of Theorem~\ref{t:final-cracktip}, i.e.~$K\cap \partial B_r$ consists of a single point. We can then take a suitable linear combination of \eqref{e:translations} and \eqref{e:rotations} (in fact we subtract the second from the first) to derive a boundary integral identity, which does not involve the set $K$ in case $\lambda=0$. Recall the notation $n(x)=\sfrac{x}{|x|}$ and $\tau(x)=\sfrac{x^\perp}{|x|}$ for all $x\in\mathbb{R}^2\setminus\{0\}$.
\begin{corollary}\label{c:AM-variations}
Let $(u,K)$ be as in Theorem~\ref{t:final-cracktip}, and assume that $K\cap \partial B_r=\{p\}$. Then, for a.e. $r\in(0,\dist(0,\partial\Omega))$ it is true that
\begin{align}\label{e:AM-identity}
&0=\int_{\partial B_r\setminus \{p\}} \left(|\nabla u |^2 n \cdot \tau (p) + 2 \frac{\partial u}{\partial n} \nabla u \cdot \left(\tau - \tau (p)\right)\right)\, d\mathcal{H}^1 \notag\\
&-2\param \int_{B_r\setminus K}(u-g)(\tau-\tau(p))\cdot\nabla u\,dx
+\param \int_{B_r\cap K}\big(|u^+-g_K|^2-|u^--g_K|^2\big)\tau(p)\cdot \nu\, d\mathcal{H}^1\,.
\end{align}
\end{corollary}

It is also convenient to define a natural substitute of the harmonic conjugate in case $\lambda>0$. \index{harmonic conjugate@harmonic conjugate}\index{conjugate, harmonic@conjugate, harmonic} We argue similarly to Proposition~\ref{p:armonica-coniugata} and
consider an auxiliary function $u_a$ which satisfies
\[
\Delta u_a = -\lambda (u-g) \qquad \mbox{on $B_2$}
\]
and
\begin{align}\label{e:u aux}
u_a(0)=0\,,\quad\nabla u_a(0)=0.
\end{align}
In order to obtain a canonical choice, we fix a cut-off function $\varphi$ supported in $B_3$ and identically $1$ on $B_2$ and we let 
\[
\bar u_a = \lambda \Gamma*(\varphi (u-g))\, ,
\]
where $\Gamma (x) = - \frac{1}{2\pi} \log |x|$ is the fundamental solution of the Laplace equation. 
Clearly $\bar u_a\in H^2 \cap C^{1,1^-} (B_3)$ by elliptic regularity (namely $\bar u_a\in C^{1,\alpha} (B_2)$ for all $\alpha\in(0,1)$).
We then set $u_a (x) = \bar u_a (x) - \bar u_a (0) - \nabla \bar u_a (0) \cdot x$. Observe that, with this canonical choice, $u_a =0$ if $\lambda =0$.

By the very definition of $u_a$, the $L^2$ vector field $\nabla(u+u_a)^\perp$ is curl free 
on $B_2$ in the sense of distributions. By mollification we find a potential $w\in H^1_{{\rm loc}}(B_2)$, i.e.~$\nabla w=\nabla(u+u_a)^\perp$, which is harmonic and smooth on $B_2\setminus K$. The following facts follow from simple modifications of arguments already used in the previous sections, which we anyway report for the reader's convenience. 

\begin{lemma}\label{l:V-tecnico2}\label{L:V-TECNICO2}
$w\in C^{0,\sfrac12} (B_2)$ and we can normalize it so that $w (0) =0$. Let $\gamma\in C^{1,1}((0,2))$ be as in \eqref{e:K param},
and define
\begin{equation}\label{e:h1}
    h_1(r):=\int_0^{r}\nabla u_a(\gamma(\rho))\cdot\dot\gamma^\perp(\rho) d\rho
\end{equation}
and
\begin{align}\label{e:h2}
    h_2&:= 2(\nabla u^+-\nabla u^-)\cdot \nabla u_a - \lambda (|u^+-g_K|^2-|u^--g_K|^2)\notag\\
    &=2(\nabla u^+-\nabla u^-)\cdot \nabla u_a - \lambda (|u^+|^2-|u^-|^2-2(u^+-u^-)g_K)\,.
\end{align}
Then
\begin{equation}\label{e:u +u aux}
\begin{cases}
\Delta w = 0 \qquad \mbox{on $B_2\setminus K$}\\   
w=h_1\qquad \mbox{on $K$}\\
\kappa=-(|\nabla w^+|^2-|\nabla w^-|^2)+h_2\qquad \mbox{$\mathcal{H}^1$ a.e. on $K\cap B_2$,}
\end{cases}
\end{equation}
and the functions $h_1$ and $h_2$ satisfy the growth estimates\footnote{Here and in what follows we write 
$0\leq \varphi(x)\lesssim (\psi(x))^{\beta^-}$ if $\varphi(x)\lesssim (\psi(x))^{\alpha}$ for all $\alpha<\beta$, with a constant depending on $\alpha$ in the last inequality.} 
\begin{align}
& h_1\in C^{1,1^-}_{\mathrm{loc}}([0,2)),
\text{ and}\qquad |h_1|+r|h_1'|\lesssim r^{2^-}
\label{e:proprieta h1}\\
& h_2\in L^\infty_{\mathrm{loc}}(B_2),\text{ and}\qquad |h_2(x)|\lesssim |x|^{\sfrac12^-}.\label{e:proprieta h2}
\end{align}
when $\lambda >0$, while they vanish identically if $\lambda=0$. 

Finally, if $r\in(0,2)$ and
$\partial B_r \cap K = \{p\}$, then the following identity holds
\begin{align}\label{e:AM-identity ua}
& 2\int_{B_r}\nabla u_a^T \cdot D\tau\,\nabla u_a\notag\\
&=\int_{\partial B_r}\Big(|\nabla u_a|^2 n\cdot\tau(p) 
+2\frac{\partial u_a}{\partial n}\nabla u_a\cdot(\tau-\tau(p))\Big)d\mathcal{H}^1
+2\param\int_{B_r}(u-g)(\tau-\tau(p))\cdot\nabla u_a\,,
\end{align}
where $\tau(x)=\frac{x^\perp}{|x|}$ (a function which belongs to $ W^{1,p}_{{\rm loc}}(\R^2;\R^2)$ for every $p\in[1,2)$).
\end{lemma}

\begin{proof}
Observe that, since $\|\nabla u_a\|_{C^0}\leq C$, we have
\[
\int_{B_r (x)} |\nabla w|^2 = \int_{B_r (x)} |\nabla u + \nabla u_a|^2 \leq 2 \int_{B_r (x)} |\nabla u|^2 + C r^2 \leq C r\, .
\]
In particular $w\in C^{0,\sfrac12}$ follows from the usual Morrey's embedding. The equation $\Delta w =0$ on $B_2\setminus K$ is obvious from the definition, while the Dirichlet condition $w|_K = h_1$ follows from $w(0)=0$ integrating $\nabla w$ along $K$. The last equation in \eqref{e:u +u aux} follows immediately from \eqref{e:Euler curvature g 1} using $|\nabla w^\pm|^2 = |\nabla u^\pm + \nabla u_a|^2$. The estimates \eqref{e:proprieta h1}-\eqref{e:proprieta h2} follow immediately from \eqref{e:stima tip 1} and the formulas for $h_1$ and $h_2$. Finally, \eqref{e:AM-identity ua} is a simple integration by parts using $\Delta u_a = - \lambda (u-g)$.
\end{proof}

\begin{lemma}\label{l:V-tecnico1}\label{L:V-TECNICO1}
Under the assumptions of Theorem~\ref{t:final-cracktip} $u$ is continuous at $0$ and moreover there are constants $C, r_1>0$ such that 
\begin{equation}
|u (x) - u(0)| + |x| |\nabla u (x)| \leq C |x|^{\sfrac12} \qquad \forall x\in B_{r_1}\setminus K\, .    
\end{equation}
\end{lemma}

\begin{proof} Consider a point $x\in (B_r\setminus B_{\sfrac r2})\setminus K$. 
We distinguish two situations:
\begin{itemize}
    \item $B_{\sfrac r8} (x) \cap K = \emptyset$. Recalling that $\int_{B_{\sfrac r8} (x)} |\nabla u|^2 \leq C r$ and using the mean value property for harmonic functions we immediately infer 
    \[
    |\nabla u (x)|\leq \frac{8^2}{\pi r^2} \int_{B_{\sfrac r8} (x)} |\nabla u|
    \leq 
    \frac8r\left(\int_{B_{\sfrac r8} (x)} |\nabla u|^2\right)^{\sfrac12} \leq C r^{-\sfrac12}
    \leq C |x|^{-\sfrac12}\, .
    \]
    \item $B_{\sfrac r8} (x) \cap K \neq \emptyset$. We can then fix a point $\bar{x} \in B_{\sfrac r8} (x) \cap K$. We consider therefore $B_{\sfrac r4} (\bar x)\subseteq B_{\sfrac{3r}2} \setminus B_{\sfrac r8} $ and notice that we can apply the jump case of the $\varepsilon$-regularity theory to $B_{\sfrac r4} (\bar x)$ because of Corollary~\ref{c:connectedness}, provided $|x|\leq r_1$ for a sufficiently small $r_1$. In particular the conclusion of the $\varepsilon$-regularity theory, classical estimates for the PDE \eqref{e:Euler harmonic g 1} and \eqref{e:Euler Neumann g 1}, and the estimate 
    \[
    \int_{B_{\sfrac r4} (\bar x)} |\nabla u|^2 \leq C r\, ,
    \]
    we conclude again
    \[
    |\nabla u (x)|\leq C r^{-\sfrac12} \leq C |x|^{-\sfrac12}\, .
    \]
\end{itemize}
We observe next that, again by the $\varepsilon$-regularity theory at pure jumps, for every $0<r<r_1$ $\partial B_r\cap K$ consists of a single point. In particular, by the estimate on $|\nabla u|$, we conclude that ${\rm osc}\, (u, \partial B_r) \leq C r^{\sfrac12}$. In particular, by the maximum principle Lemma~\ref{l:maximum}, we conclude that 
\[
{\rm osc}\, (u, B_r) \leq C r^{\sfrac12}\, .
\]
In particular $u$ is continuous at $0$ and $|u(x)-u(0)|\leq C |x|^{\sfrac12}$.
\end{proof}

\section{Rescaling and reparametrization}\label{s:exp-rep}

Before starting our considerations, we introduce the model ``tangent function'' of a
minimizer at a loose end. By Theorem~\ref{t:Bonnet-David}, 
in polar coordinates the latter is given by the function \index[simb]{aalRad(phi,r)@$\Rad(\phi,r)$}
\begin{align}
\Rad(\phi, r) &:=\sqrt{\textstyle{\frac{2r}{\pi}}}\cos\textstyle{\frac{\phi}{2}}\label{e:Rad}\, 
\end{align}
with jump set $K_{\Rad}$ equal to the open half line $\{(t, 0): t\in \R^+\}$ (in cartesian coordinates). Observe that $\Rad$ is, up to the prefactor $\sqrt{\frac{2}{\pi}}$, the real part of a branch of the complex square root. We will likewise use the notation for its harmonic conjugate \index[simb]{aalIsq(phi,r)@$\Isq(\phi,r)$}$\Isq$, which is the imaginary part of the same branch, multiplied by the same prefactor, namely  \index[simb]{aalIsq(phi,r)@$\Isq(\phi,r)$}
\begin{equation}\label{e:Isq}
\Isq (\phi, r) := \sqrt{\textstyle{\frac{2r}{\pi}}}\sin\textstyle{\frac{\phi}{2}}\, .
\end{equation}

\subsection{Rescalings}

From now until the very last section, $(u,K)$ will always denote a critical point of $E_\lambda$ 
in $B_2$ satisfying the assumptions of Theorem~\ref{t:final-cracktip}. 
Keeping the notation introduced in \eqref{e:K param}, for $\rho>0$ set
\begin{align}
u^\rho(\phi, r)&:=\rho^{-\sfrac12}\,u(\phi+ \alpha(\rho\,r), \rho\, r),
\label{e:riscala1}\\
\alpha^\rho (r) &:= \alpha (\rho\, r)\, .\label{e:riscala2}
\end{align}
\begin{lemma}\label{l:Bonnet}
 For every $\delta>0$ the following holds.
 
If $\lambda>0$, then for every $\varepsilon>0$ there is $\e_1>0$ such that, 
if $(u,K)$ is as in Theorem~\ref{t:final-cracktip} and $\alpha$ is as in 
\eqref{e:K param} with $\e_0\leq\e_1$, then
\begin{equation}\label{e:Bonnet0}
 \|u^\rho-\Rad\|_{C^{1,1-\varepsilon}([0,2\pi]\times[\sfrac12,2])} 
+ \|\alpha^\rho\|_{C^{1,1-\varepsilon} ([\sfrac 12, 2])}\leq\delta\qquad 
 \forall\,\rho \leq \frac{1}{4}\,\, .
\end{equation}
If $\lambda = 0$, then for every $k\in \mathbb N$ there is $\e_1>0$ such that, under the very same assumptions, 
 \begin{equation}\label{e:Bonnet0 bis}
 \|u^\rho-\Rad\|_{C^k([0,2\pi]\times[\sfrac12,2])} 
+ \|\alpha^\rho\|_{C^k ([\sfrac 12, 2])}\leq\delta\qquad 
 \forall\,\rho \leq \frac{1}{4}\, .
\end{equation}

\end{lemma}
\begin{proof} First of all, we introduce the function $u^\rho (x) := \rho^{-\sfrac12} u (\rho x)$ and the rescaled singular set $K_\rho := \frac{K}{\rho}$. Recall that $(u_\rho, K^\rho)$ converges, up to subsequences, to a global minimizer of $E_0$, which we denote by $(u_0 , K_0)$. We also know that $|\nabla u (x)|\leq C |x|^{-\sfrac12}$ from Lemma~\ref{l:V-tecnico2}. But then it follows easily that the rescaled set $K_\rho$ satisfies uniform $C^{1,1}$ estimates on $B_1\setminus B_{\sfrac14}$. Hence the estimate of $\alpha^\rho$ in \eqref{e:Bonnet0} simply follows from the fact that the curve parametrized by $\alpha^\rho$ is converging to the straight segment $[0, (1,0)]$, while $\alpha^\rho$ has a uniform $C^{1,1}$ bound. As for the other estimate in \eqref{e:Bonnet0} it follows from classical regularity for the Neumann problem, using $\Delta u = \lambda (u-g)$ (and the uniform $L^\infty$ bounds on $u$ and $g$).

Next, observe that, for the case $\lambda =0$, the rescaled pair $(u_\rho , K^\rho)$ is still a minimizer of the Mumford-Shah functional, and the estimate \eqref{e:Bonnet0 bis} would follow (by compactness) once we show uniform $C^{k, \sfrac12}$ estimates for both $u^\rho$ and $K_\rho$ in $B_1\setminus B_{\sfrac12}$. However, the latter is a simple bootstrapping process using the equation. Note for instance that Theorem~\ref{t:eps_salto_puro} gives a uniform $C^{1,\alpha}$ bound on $K_\rho\cap (B_{2-\mu}\setminus B_\mu)$ for every $\mu\in(0,1)$. But then using the uniform $L^2$ bound for $\nabla u$ and classical estimates for the Neumann problem, we conclude uniform $C^{\alpha}$ estimates for $\nabla u$ in $B_{2-2\mu}\setminus (B_{2\mu}\cup K_\rho)$ for every $\mu\in(0,\sfrac12)$. We can now use the equation for the curvature $\kappa$ to conclude uniform $C^{2,\alpha}$ estimates for $K_\rho \cap (B_{2-3\mu}\setminus B_{3\mu})$ for every $\mu\in(0,\sfrac13)$. In turn this implies uniform $C^{1,\alpha}$ estimates for $\nabla u$ in $B_{2-4\mu}\setminus (B_{4\mu}\cup K_\rho)$
for every $\mu\in(0,\sfrac14)$. This bootstrap argument can be repeated finitely many times, until we achieve $C^{k+1,\alpha}$ estimates. 
\end{proof}

\begin{corollary}\label{c:Bonnet-2}
 For every $\delta>0$ the following holds.
\begin{itemize}
     \item[(a)] If $\lambda>0$, then for every $\varepsilon>0$ there is $\varepsilon_1>0$ such that, if $(u,K)$ and $\alpha$ satisfy the assumptions of Theorem~\ref{t:final-cracktip} with $\e_0\leq\e_1$, then 
 \begin{equation}\label{e:Bonnet2}
\|r^{i-\sfrac12}\partial_\phi^{j}\partial_r^{i}(u(\phi + \alpha (r) ,r)-\Rad(\phi,r))\|_{C^{0,1-\e}([0,2\pi]\times (0,\sfrac12))}
\leq \delta\quad \text{$\forall\, i+j\leq 1$},
\end{equation}
\begin{equation}\label{e:Bonnet2bis}
\|r\alpha^{\prime}(r)\|_{C^{0,1-\e}((0,\sfrac 12))}\leq \delta\, ,
\end{equation}
\item[(b)] If $\lambda=0$, then for every $k\in\N$ there is $\e_1>0$ such that, if $(u,K)$ and $\alpha$ satisfy the assumptions of Theorem~\ref{t:final-cracktip} with $\e_0\leq\e_1$, then 
 \begin{equation}\label{e:Bonnet2 bis}
\sup_{[0,2\pi]\times(0,\sfrac12)}
r^{i-\sfrac12}|\partial_\phi^{j}\partial_r^{i}(u(\phi + \alpha (r),r) -\Rad(\phi,r))|
\leq \delta\quad \text{$\forall\, i+j\leq k$},
\end{equation}
\begin{equation}\label{e:Bonnet2bis bis}
\sup_{(0,\sfrac 12)}r^i|\alpha^{(i)}(r)|\leq \delta 
\quad\qquad\qquad\qquad\qquad \forall\, i\leq k\, . 
\end{equation}
\end{itemize}
\end{corollary}
\begin{proof}
We discuss only the case $\lambda=0$, the other being analogous. Observe first that
\[
(\alpha^\rho)^{(i)}  (r) = \rho^i \alpha^{(i)} (\rho r)\, .
\]
Taking the supremum in $r\in [\sfrac 12,2]$ in the latter identity, we easily infer 
\[
\rho^i \|\alpha^{(i)}\|_{C^0 ([\sfrac \rho 2, 2\rho])} = \|(\alpha^{\rho})^{(i)}\|_{C^0 ([\sfrac 12, 2])}\,, 
\]
and hence conclude \eqref{e:Bonnet2bis bis} 
from Lemma~\ref{l:Bonnet}. 

Next, from \eqref{e:riscala1} and the $\sfrac{1}{2}$-homogeneity of $\Rad$ we conclude
\begin{align*}
 u (\phi + \alpha (r) , r) - \Rad (\phi, r) = \rho^{\sfrac{1}{2}} \left( u^\rho \left(\phi, \textstyle{\frac{r}{\rho}}\right)
- \Rad \left(\phi, \textstyle{\frac{r}{\rho}}\right)\right)\, .
\end{align*}
Differentiating the latter identity $j$ times in $\theta$ and $i$ times in $r$, we conclude
\begin{align*}
\partial_r^i\partial_\phi^j\left(u(\phi+\alpha(r) , r) -\Rad (\phi, r)\right)
= \rho^{\sfrac{1}{2} -i} \partial_r^i\partial_\phi^j \left(u^\rho - \Rad\right) \left(\phi, \textstyle{\frac{r}{\rho}}\right)
\end{align*}
Substitute first $\rho=r$ and take then the supremum in 
$\phi$ and $r$ to achieve  \eqref{e:Bonnet2}, 
again from Lemma~\ref{l:Bonnet}.
\end{proof}

\subsection{Reparametrization}
Following Simon's insight for studying the uniqueness of tangent cones to minimal surfaces 
\cite{Simon83} we next introduce the functions 
\begin{align}
\vartheta(t):=&\alpha(e^{-t}), 
\label{e:vartheta}\\
\varrho(t) :=& e^{-t}(\cos\vartheta(t),\sin\vartheta(t)),\label{e:varrho},\\
 f (\phi, t):=& e^{\sfrac{t}2}\, w (\phi+\vartheta(t), e^{-t})
=w^{e^{-t}} (\phi, 1)\, ,\label{e:f}\\
\rad (\phi) := &\Rad (\phi, 1),\label{e:radino}\\
\isq (\phi) := &\Isq (\phi, 1),\label{e:radinow}
\end{align}
{where $\Rad$, $\Isq$ are defined respectively in \eqref{e:Rad}, \eqref{e:Isq}.}
\index[simb]{aalrsq(phi)@$\rad(\phi)$}
\index[simb]{aalisq(phi)@$\isq(\phi)$}
Note that the change of phase in the definition of $f$ in \eqref{e:f} maps the set $K$ 
onto the halfline $\{0\}\times[0,\infty)$.

In the next lemma we derive a system of partial differential equations for the functions $f$ and $\vartheta$, exploiting the Euler-Lagrange conditions \index{Euler-Lagrange conditions@Euler-Lagrange conditions} satisfied by $u$ and $K$ (cf. \eqref{e:outer} and \eqref{e:inner}). 
The effect of the negative exponential reparametrization is that we will get an evolution equation for $f$. The claimed regularity for $\gamma$ corresponds to exponential decay estimates in time $t$ for $\varthetad$,
which will be the object of study in the next sections.

We also rewrite the estimates of Corollary~\ref{c:Bonnet-2} in terms of the new functions.
It is more convenient to work with $w$ rather than $u$. This is clear when $\lambda=0$ 
because of the homogeneous Dirichlet boundary condition satisfied by $w$ on $K$ instead 
of its Neumann counterpart satisfied by $u$.

\begin{lemma}\label{l:nonlinear}
 If $(u,K)$ satisfies the assumptions of Theorem~\ref{t:final-cracktip} and 
 $\vartheta, f$ are given by \eqref{e:vartheta} 
 and \eqref{e:f}, then
\begin{equation}
  \begin{cases}\label{e:SIS}
 \displaystyle{f_t =\frac f4+ f_{\phi\phi}+ f_{tt}+\big(\varthetad f_\phi 
 +\varthetad^2 f_{\phi\phi}-2\varthetad f_{t\phi}-\varthetadd f_\phi \big)}\cr\cr
f(0,t)=f(2\pi,t)=H_1(t)
 \cr\cr
 \displaystyle{ \frac{\varthetadd - \varthetad-\varthetad^3}{(1+\varthetad^2)^{\sfrac52}}
 =f_\phi^2(2\pi,t)-f_\phi^2(0,t)+H_2(t)\,}
 \end{cases}
\end{equation}
where (recalling the definition of $h_1$ in \eqref{e:h1} and $h_2$ in \eqref{e:h2}) the functions $H_i$ are given by the following formulas:
\begin{equation}\label{e:H1}
H_1(t):=e^{\frac t2}h_1(e^{-t})\,,
\end{equation}
\begin{equation}\label{e:H2}
H_2(t):=\frac{2\varthetad(t)}{1+\varthetad^2(t)}\Big(\frac{H_1(t)}{2}-\dot{H}_1(t)\Big)
(f_\phi(2\pi,t)-f_\phi(0,t))+\frac{e^{-t}}{1+\varthetad^2(t)}h_2(\vartheta(t),e^{-t})\,.
\end{equation}
Moreover, it is true that 
\begin{equation}\label{e:reg Hi}
H_1\in H^{\sfrac32}\cap C^{1,1^-}((1,+\infty)),\text{ and}\qquad
H_2\in L^\infty((1,+\infty))\,, 
\end{equation}
and for every $\e>0$ there is a constant $C_\e>0$ such that for all $t>0$ and 
$s\geq t$
\begin{align}
& 
|H_1(t)|+|\dot{H}_1(t)|+|H_2(t)|\lesssim (e^{-t})^{\sfrac32^-}\label{e:bounds Hi}\\
&|\dot{H}_1(t)-\dot{H}_1(s)|\leq C_\e e^{-(\frac32-\varepsilon)t}|s-t|^{1-\e}\,.\label{e:C11- dotH1}
\end{align}
Finally, for every fixed $\sigma,\,\delta>0$, $\e\in(0,\e_0)$ and $k\in\N$, the following 
estimates hold provided $\e_0$ in Theorem~\ref{t:final-cracktip} is sufficiently small:
if $\lambda>0$
\begin{equation}\label{e:decay_g lambda>0}
 \|\varthetad\|_{C^{0,1-\e}([\sigma , \infty))}\leq\delta\,,
\end{equation}
\begin{equation}\label{e:decay_f lambda>0}
\|\partial_\phi^i\partial_t^j (f - \isq)\|_{C^{0,1-\e}([0,2\pi]\times[\sigma,\infty))} \leq \delta\,
\qquad\mbox{for all 
$i+j\leq 1$,}
\end{equation}
and if $\lambda=0$ 
\begin{equation}\label{e:decay_g}
 \|\vartheta^{(i)}\|_{C^0([\sigma , \infty))}\leq\delta \qquad \mbox{for all $i\leq k$,}
\end{equation}
\begin{equation}\label{e:decay_f}
\|\partial_\phi^i\partial_t^j (f - \isq)\|_{C^0([0,2\pi]\times[\sigma,\infty))} \leq \delta\,
\qquad\mbox{for all 
$i+j\leq k$}.
\end{equation}
\end{lemma}
\begin{proof}
 Let us first introduce the unit tangent and normal vector fields to $K$ denoted by $e (t)$ 
 and for the normal vector $\nu (t)$, the latter is obtained from $e(t)$ by a counterclockwise rotation of $90$ degrees, that is:
 \begin{align*}
  e (t):= \frac{\etad(t)}{|\etad(t)|},\quad \nu (t):=e^\perp(t).
 \end{align*}
Moreover, we will denote by $\nabla u^+$ and $\nabla u^-$ the traces of $\nabla u$ on $K$ where $\pm$ is identified by the direction in which the vector $\nu$ is pointing. More precisely, if $p\in K$, then
\begin{align*}
\nabla u^+ (p) & = \lim_{s\downarrow 0} \nabla u (p + s \nu (p))\,,\\
\nabla u^- (p) & = \lim_{s\downarrow 0} \nabla u (p - s \nu (p))\, .
\end{align*}
Observe that, under the assumptions of Lemma~\ref{l:Bonnet}, $e (t)$ is pointing ``inward'', i.e. towards the origin, and hence
for $p = \varrho (t) = (e^{-t} (\cos (\vartheta (t)), \sin (\vartheta (t)))$ (cf. \eqref{e:K param}) we have
\begin{align}
\nabla u^+ (p) &= \lim_{\phi \uparrow 2\pi} \nabla u (e^{-t} (\cos (\vartheta (t) + \phi), \sin (\vartheta (t) +\phi))\, ,\\
\nabla u^- (p) &= \lim_{\phi\downarrow 0} \nabla u (e^{-t} (\cos (\vartheta (t) + \phi), \sin (\vartheta (t) +\phi))\, .
\end{align}
We refer to Figure~\ref{figura-5.1} for a visual illustration.

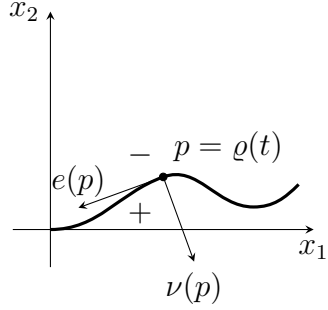
\begin{figure}
\centering
\begin{tikzpicture}
\draw[>=stealth,->] (-0.5,0) -- (3.5,0);
\node[below] at (3.5,0) {$x_1$};
\draw[>=stealth,->] (0,-0.5) -- (0,2.9);
\node[left] at (0,2.9) {$x_2$};
\draw[very thick] (0,0) to [out=0, in=200] (1.5,0.7) to [out=20, in=180] (2.7, 0.3) to [out = 0, in=225] (3.3,0.6);
\draw[>=stealth,->] (1.5,0.7) -- ({1.5-1.2*cos(20)},{0.7-1.2*sin(20)});
\draw[fill] (1.5,0.7) circle [radius=0.05];
\draw[>=stealth,->] (1.5,0.7) -- ({1.5+1.2*sin(20)},{0.7-1.2*cos(20)});
\node[right] at (1.5,1.1) {$p = \varrho (t)$};
\node[above] at ({1.5-1.2*cos(20)},{0.7-1.2*sin(20)}) {$e(p)$};
\node[below] at  ({1.5+1.2*sin(20)},{0.7-1.2*cos(20)}) {$\nu(p)$};
\node[below left] at (1.5,0.5) {$+$};
\node[above left] at (1.5,0.7) {$-$};
\end{tikzpicture}
\caption{The tangent vector $e(p)$ and the normal vector $e(p)$ and a point $p\in K$. Since $t\mapsto |\varrho (t)|$ is a decreasing function, $e(p)$ points towards the origin. Consequently the convention for the symbols $\pm$ on traces of functions is as illustrated in the picture.}
\label{figura-5.1}
\end{figure}

Since $(u,K)$ is a critical point of the $E_\lambda$ energy on $B_2$, the identities
\eqref{e:Euler harmonic g 1}-\eqref{e:Euler curvature g 1} are true. 
Note that the curvature $\kappa$ of $K$ is given by 
\begin{align*}
\kappa(t) =& \frac{1}{|\etad(t)|}\dot e(t)\cdot \nu(t)\, .
\end{align*}
In particular, the auxiliary function $w$ defined in Lemma~\ref{l:V-tecnico2} 
satisfies \eqref{e:u +u aux}, 
which we rewrite for the readers' convenience
\begin{equation}\label{e:firstvar2}
 \begin{cases}
 \triangle w=0 & \mbox{on } B_1\cr
 w=h_1 & \mbox{on } K\cr
 \kappa=-|\nabla w^+|^2+|\nabla w^-|^2+h_2\qquad&\mbox{on }  K\,,
 \end{cases}
\end{equation}
where $h_1$ and $h_2$ are defined in \eqref{e:h1} and \eqref{e:h2}, respectively.
Recalling that
 \begin{equation}\label{e:ftr}
 w(\phi ,r)=r^{\sfrac12}f(\phi-\vartheta(-\ln r) , - \ln r),
  \end{equation}
we compute
\begin{equation}\label{e:utr}
 w_r =r^{-\sfrac12}\left(\frac f2- f_t +\varthetad f_\phi\right),\quad 
 w_\phi =r^{\sfrac 12} f_\phi .
 \end{equation}
Next we recall the formula for the Laplacian in polar coordinates:
\[
\triangle w=0\quad\Longleftrightarrow\quad r^{-2} w_{\phi\phi} +r^{-1} (r w_r )_r =0.
\]
By means of \eqref{e:utr} we get
\[
 r^{-2} w_{\phi\phi} =r^{-\sfrac32} f_{\phi\phi}\, ,
\]
and
\begin{align*}
 r^{-1} (r w_r )_r
=& r^{-1} \left(r^{\sfrac12}\left(\frac f2- f_t +\varthetad f_\phi\right)\right)_r\\
=& r^{-\sfrac32}\left(\frac f4-\frac{f_t}{2}+\frac{\varthetad f_\phi}{2} \right)
+r^{-\sfrac12}\left(- r^{-1}\frac{f_t}{2} +r^{-1} \varthetad\frac{f_\phi}{2} \right)\\
&+r^{-\sfrac{1}{2}}\left(r^{-1} f_{tt} -2r^{-1}\varthetad f_{t\phi} - r^{-1}\varthetadd f_\phi 
+r^{-1}\varthetad^2 f_{\phi\phi}  \right)\\
= &r^{-\sfrac32}\left(\frac f4- f_t+\varthetad f_\phi+ f_{tt} -2\varthetad f_{t\phi}-\varthetadd f_\phi 
+\varthetad^2 f_{\phi\phi} \right).
\end{align*}
In conclusion, we get
\begin{equation}\label{e:laplr}
f_t =\frac f4+ f_{\phi\phi}+ f_{tt}+\big(\varthetad f_\phi 
+\varthetad^2 f_{\phi\phi}-2\varthetad f_{t\phi}-\varthetadd f_\phi \big).
\end{equation}

Next, recalling equality \eqref {e:ftr}, we may  rewrite the Dirichlet condition in the new coordinates simply as 
\begin{equation}\label{e:dir f}
f (0, t)= f(2\pi,t)=e^{\frac t2}h_1(e^{-t})=H_1(t)\, . 
\end{equation}

Finally, we derive the equation satisfied by the scalar curvature $\kappa$. To this end take into account that
\begin{align}\label{e:doteta}
\dot{\varrho} (t) 
=& - \varrho (t) + \dot\vartheta (t) \varrho^\perp (t)\,, 
\end{align}
and thus differentiating \eqref{e:doteta} we get
\begin{align}\label{e:ddoteta}
\etadd (t) =& - \etad + \ddot\vartheta \varrho^\perp + \dot\vartheta \etad^\perp\, .
\end{align}
On the other hand, explicitly we have
\begin{align}
\dot{\varrho} (t)^\perp &= - e^{-t} (-\sin \vartheta (t), \cos \vartheta (t)) - e^{-t} \dot\vartheta (t) (\cos \vartheta (t), \sin \vartheta (t))\nonumber\\
&= - \varrho^\perp (t) - \dot\vartheta (t) \varrho (t) 
\, .
\end{align}
Hence, we conclude
\begin{align}
\kappa(t) =& \frac{1}{|\etad(t)|}\left(\frac{d}{dt}\frac{\etad(t)}{|\etad(t)|}\right)
\cdot\frac{\etad^\perp(t)}{|\etad(t)|}=\frac{\etadd(t)\cdot\etad^\perp(t)}{|\etad(t)|^3}\nonumber\\
=& \frac{(\varthetad+\varthetad^3-\varthetadd) |\varrho (t)|^2}{(1+\varthetad^2)^{\sfrac32}|\varrho (t)|^3}
= e^{t}\,\frac{\varthetad+\varthetad^3-\varthetadd}{(1+\varthetad^2)^{\sfrac32}}.\label{e:formula_curvatura}
\end{align}
As
\[
|\nabla u|^2=|\nabla w|^2=(w_r)^2+r^{-2}(w_\phi)^2=
r^{-1}\left(\frac f2+\varthetad f_\phi - f_t \right)^2+r^{-1} f_\phi^2
\]
we get
\begin{equation}\label{e:curvr}
 \frac{\varthetad+\varthetad^3-\varthetadd}{(1+\varthetad^2)^{\sfrac32}}= -
 \left. \left[\left(\frac f2+\varthetad f_\phi - f_t \right)^2 
 + f_\phi^2\right]\right|_0^{2\pi}+e^{-t}h_2(\vartheta(t),e^{-t})\, .
\end{equation}
Thus, by taking into account \eqref{e:dir f} and \eqref{e:curvr} we conclude 
the third equation in \eqref{e:SIS}.

If $\lambda=0$ we note that in terms of $\vartheta$ the bound of $\alpha$ in \eqref{e:Bonnet2bis bis} reads as
\[
 \sup_{t\in[\sigma,\infty)}
 |\vartheta^{(i)}(t)|\leq C_i\,\delta \qquad \mbox{for every 
 $i\leq k$.}
\]
Indeed, differentiating $i$ times the identity $\vartheta (t) = \alpha (e^{-t})$ we get
\[
\vartheta^{(i)} (t) = \sum_{j=1}^i b_{i,j} e^{-jt} \alpha^{(j)} (e^{-t})\, ,
\]
with $b_{i,j}\in \mathbb R$. Then, \eqref{e:decay_g} follows at once.

Instead, the bound \eqref{e:decay_f} is a consequence of the linearity and elementary arguments, together with the decay \eqref{e:Bonnet2 bis}. Indeed, the latter translates into 
\[
\sup_\phi|\partial_\phi^i\partial_t^j (\xi - \rad)|\leq C_j\,\delta\, 
\qquad\mbox{for every $t\in[\sigma,\infty)$ and $i+j\leq k$,}
\]
having set $\xi (\phi, t):= e^{\sfrac{t}2} u (\phi+\vartheta(t), e^{-t})$.
To prove the latter estimate we argue as follows.
Using the $\sfrac{1}{2}$-homogeneity of $\Rad$, we infer
\begin{align}\label{e:g}
\xi(\phi, t) - \rad (\phi) &=\;  e^{\sfrac{t}{2}} \left(u (\phi + \alpha (e^{-t}), e^{-t}) - \Rad (\phi, e^{-t})\right)
\notag\\
&=:\,  e^{\sfrac{t}{2}} h (\phi, e^{-t})\, .
\end{align}
We conclude that \eqref{e:Bonnet2 bis} can be reformulated as
\[
\sup_{r\in(0,\sfrac12)}r^{i-\sfrac12}\|\partial^j_\theta \partial^i_r h (\cdot, r)\|_{C^0} \leq C_i\,\delta 
\qquad\mbox{for every $t\in[\sigma,\infty)$ and $i+j\leq k$.}
\]
On the other hand, differentiating \eqref{e:g} yields
\[
\partial_\phi^j \partial_t^i (\xi(\phi, t) - \rad (\phi)) = 
\sum_{\ell=0}^i b_{i, \ell}\, e^{\sfrac{t}{2} - \ell t} [\partial_\phi^j \partial_t^\ell h] (\phi, e^{-t})\,, 
\]
for some $b_{i,\ell}\in \mathbb R$. Setting $r= e^{-t}$, we then conclude 
\[
\|\partial_\phi^i\partial_t^j (\xi - \rad)\|_{C^0([0,2\pi]\times[\sigma,\infty))} \leq \delta\,
\qquad\mbox{for all 
$i+j\leq k$}\,,
\]
and thus \eqref{e:decay_f} follows at once, using that the gradient of the harmonic conjugate is the counterclockwise rotation by $90$ degrees of the gradient of $u$.

Instead, if $\lambda>0$ we argue analogously to infer \eqref{e:decay_g lambda>0} and \eqref{e:decay_f lambda>0}, using in addition classical estimates for the function $u_a$.
Finally, being $u_a\in C^{1,1^-}_{\mathrm loc}(B_2)$, in view of the bounds in  \eqref{e:proprieta h1}-\eqref{e:proprieta h2}, 
we infer \eqref{e:reg Hi}-\eqref{e:C11- dotH1}. 
\end{proof}

\section{First linearization}\label{s:first linearization}

In this section we consider a sequence $(u_j, \alpha_j)$ as in Theorem~\ref{t:final-cracktip}
where condition (iv) holds for a vanishing sequence $\e_0 (j)\downarrow 0$. 
Without loss of generality we assume $\alpha_j (1) =0$. 
With fixed $\hat\e\in(0,\sfrac32)$ we define $\delta_j$, $\theta_j$,  and $v_j$ thanks to \eqref{e:vartheta}-\eqref{e:f} as follows:
\begin{align}
\delta_j &:= \|f_j(\cdot,\cdot+j)-\isq\|_{H^2 ((0,2\pi)\times (0,3))} + 
\|\dot\vartheta_j(\cdot+j)\|_{H^1 ((0,3))}
+\lambda e^{-(\sfrac32-\hat\e)j}\,,\label{e:delta}\\
\theta_j (t) &:= \delta_j^{-1} \vartheta_j (t+j)\label{e:theta_j}\,,\\
v_j (\phi, t) &:= \delta_j^{-1} ( f_j (\phi,t+j) - \isq(\phi))\label{e:v_j}\,.
\end{align}
Next, we show that the limit of $(v_j,\theta_j)$ solves a linearization of \eqref{e:SIS},
and in addition it satisfies the linearization of \eqref{e:int-variation-bdry} (actually it suffices to consider \eqref{e:AM-identity}). The latter remark is crucial for our purposes. 
\begin{proposition}\label{p:linearizzazione}
Let $(u_j, \alpha_j)$ as in Theorem~\ref{t:final-cracktip} where the smallness condition 
in item (iv) holds for a vanishing sequence $\e_0 (j)\downarrow 0$. 
Assume $\alpha_j (1) =0$ and define
$\vartheta_j$ and $f_j$ as in \eqref{e:vartheta}-\eqref{e:f} and $v_j$ and 
$\theta_j$ as above. Then, up to subsequences,
\begin{itemize}
\item[(a)] $v_j$ converge to some function $v$ weakly in 
$H^2 ((0, 2\pi)\times (0,3))$ and strongly in $W^{1,p}((0, 2\pi)\times (0,3))$
for all $p\geq 1$;
\item[(b)] $\theta_j$ converge to some $\theta$;
weakly in $H^2 ((0, 2\pi))$ and in $C^{1,\frac12-\varepsilon}([0,3])$ 
for all $\varepsilon \in (0,1)$.
\end{itemize} 
More precisely, the convergences are  
\begin{itemize}
\item[(c)] either if $\lambda=0$: in $C^{2,\alpha} ([0,2\pi]\times [\sigma, 3-\sigma])$ 
for $v_j$, and in $C^{2,\alpha}([\sigma,3-\sigma])$ for $\theta_j$, for all $\sigma\in (0, \frac32)$, $\alpha \in (0,1)$, respectively;

\item[(d)] or if $\lambda>0$: in $C^{1,1-\varepsilon} ([0,2\pi]\times [\sigma, 3-\sigma])$ and strongly in $H^2((0,2\pi)\times (\sigma, 3-\sigma))$ for $v_j$, and strongly in $H^2((\sigma,3-\sigma))$ for $\theta_j$, for all $\sigma\in (0, \frac32)$, $\varepsilon \in (0,1)$, respectively.
\end{itemize}
Moreover, the pair $(v, \theta)$ solves the following linear system of PDEs in 
$(0,2\pi)\times (0,3)$
\begin{equation}\label{e:lineare}
\left\{
\begin{array}{l}
v_t - v_{tt} = \frac{v}{4} + v_{\phi\phi} + (\dot\theta - \ddot\theta) \isq_\phi\\ \\
v (0, t) = v(2\pi, t) = 0\\ \\
\dot\theta (t) - \ddot\theta (t) =  \sqrt{{\textstyle{\frac{2}{\pi}}}} ( v_\phi(2\pi,t)+ v_\phi(0,t))\\ \\
\theta (0) = 0\,,
\end{array}\right.
\end{equation}
and satisfies the following integral condition for every $t\in (0, 3)$:
\begin{equation}\label{e:condizione_extra}
\int_0^{2\pi} v_\phi(\phi,t) \sin\frac{\phi}{2}\, d\phi=0\,.
\end{equation}
\end{proposition}
\begin{proof} The statements in (a) and (b) are obvious consequences of the bounds on $(v_j,\theta_j)$ (and of the fact that 
$H^2 ((0,2\pi)\times(0,3))$, resp. $H^2 ((0,3))$, embeds compactly in $W^{1,p} ((0,2\pi)\times(0,3))$ for all $p\geq 1$, resp. $C^{1,\frac12-\e}([0,3])$ for all $\e\in(0,\frac12)$). 
Observe that, by assumption, $\theta_j (0) =0$ and thus $\theta (0) =0$ is a consequence of the uniform convergence. 

We next observe that the PDE in \eqref{e:SIS} is linear in the unknown $f$. Hence, recalling that $(f_j(\phi,t+j),\vartheta_j(t+j)) = (\isq(\phi)+\delta_j v_j(\phi,t),\delta_j\theta_j(t))$, we infer
\begin{equation}\label{e:odd-PDE}
v_{j,tt} + v_{j,\phi\phi} = -\frac{v_j}{4} + v_{j,t} + (\ddot\theta_j - \dot\theta_j) (\isq_\phi+\delta_j v_{j,\phi}) + 2 \delta_j\dot\theta_j v_{j, t\phi}  - \delta_j \dot\theta^2_j (\isq_{\phi\phi}+\delta_j v_{j,\phi\phi})
\end{equation}
and 
\begin{equation}\label{e:bdry cndntn vj}
v_j(0,t)=v_j(2\pi,t)=\delta_j^{-1}H_1(t+j)\,.
\end{equation}
Passing into the limit we therefore conclude easily that $v$ solves the PDE in the first 
line of \eqref{e:lineare}.
Likewise the boundary condition $v (0, \cdot) = v (2\pi, 0) =0$ is also a consequence of uniform convergence, of the bounds \eqref{e:bounds Hi} and \eqref{e:C11- dotH1}, and of the very definition of $\delta_j$ in
\eqref{e:delta}, in turn implying
for all $\e\in(0,\hat\e)$
\begin{equation}\label{e:Hi scaled}
\|\delta_j^{-1}H_1(\cdot+j)\|_{C^{1,1-\e}([0,3])}+
\|\delta_j^{-1}H_2(\cdot+j)\|_{L^{\infty}((0,3))}\to 0\qquad j\to+\infty\,.
\end{equation}
Moreover, note that again thanks to the bounds in \eqref{e:bounds Hi} and \eqref{e:C11- dotH1} we can also deduce that for all $j\geq 1$
\begin{align*}
[\dot H_1(\cdot+j)]_{H^{\sfrac12}((0,\infty))}^2
=&2\int_{j}^{\infty}dt\int_{t}^{\infty}\frac{|\dot H_1(s)-\dot H_1(t)|^2}{|s-t|^2}ds\\
=&2\int_{j}^{\infty}dt\sum_{k=0}^\infty
\int_{t+k}^{t+k+1}\frac{|\dot H_1(s)-\dot H_1(t)|^2}{|s-t|^2}ds\\
\leq& 2C_\varepsilon^2\int_{j}^{\infty}dt
\sum_{k=0}^\infty\int_k^{k+1}\frac{e^{-(3-2\varepsilon)(t+k)}}{\tau^{2\varepsilon}}d\tau
\leq\tilde C_\varepsilon e^{-(3-2\varepsilon)j}\,.
\end{align*}
Therefore, 
\begin{equation}\label{e:Hi scaled bis}
\|\delta_j^{-1}H_1(\cdot+j)\|_{H^{\sfrac32}((0,\infty))}\to 0\,.
\end{equation}

We next write the third equation in \eqref{e:SIS} in terms of $\theta_j$ and $v_j$:
\begin{align}\label{e:transmission}
\ddot\theta_j -\dot\theta_j = &\delta_j^2 \dot\theta_j^3 + 4 (1+\delta_j^2 \dot\theta_j^2)^{\sfrac{5}{2}}\big(\delta_j( v_{j,\phi}^2 (2\pi,t)- v_{j,\phi}^2 (0,t))
\notag\\ & 
- \sqrt{{\textstyle{\frac{2}{\pi}}}}( v_{j,\phi}(2\pi,t)+ v_{j,\phi}(0,t)) +\delta_j^{-1}H_2(t+j)\big)\, . 
\end{align}
Observe that, by the trace theorems, $v_{j,\phi} (2\pi, \cdot)$ and $v_{j,\phi} (0, \cdot)$ enjoy uniform bounds in $H^{\frac12}$. Clearly, by \eqref{e:Hi scaled} we get that the third equation in \eqref{e:lineare} holds. 

Moreover, by \eqref{e:Hi scaled} and by the Sobolev embedding, we conclude that the right hand side of \eqref{e:transmission} has a uniform control in $L^q$ for every $q<\infty$, in particular the same bound is enjoyed by $\ddot\theta_j - \dot\theta_j$ and, using that $\|\dot \theta_j\|_{C^0}$ is bounded, we conclude that $\dot\theta_j$ has a uniform $W^{1,q}$ bound for every $q<\infty$. 

Next, we distinguish the two cases $\lambda=0$ and $\lambda>0$, the former being considerably simpler 
at least for what computations are concerned. 

\noindent {\bf Case $\lambda=0$.} We start off rewriting the equation in \eqref{e:odd-PDE} above in the following way:
\begin{align}
(1+ \delta_j^2 \dot \theta_j^2) v_{j,\phi\phi} + v_{j,tt} 
-2\delta_j \dot\theta_j v_{j,t\phi}
= \underbrace{-\frac{v_j}{4} + v_{j,t} + (\ddot\theta_j - \dot\theta_j) (\isq_\phi+\delta_j v_{j,\phi})
- \delta_j \dot\theta^2_j \isq_{\phi\phi}}_{=:F_j}
\label{e:bootstrap}
\end{align}
Observe that the left hand side is an elliptic operator with a uniform bound on the ellipticity constants and a uniform bound on the $C^{\frac12-\varepsilon}$ norm of the coefficients, for all $\varepsilon\in(0,\hat\varepsilon)$. 
Thanks to the uniform $W^{1,q}$ bound on $\theta_j$ and $v_j$ for every $q<\infty$, we infer a uniform bound on $\|F_j\|_{L^q ((0,2\pi)\times (0,3))}
$ for every $q<\infty$.

In addition, as $v_j=0$ on $\{0,2\pi\}\times(0,3)$, using elliptic regularity we conclude a uniform bound for $\|v_j\|_{W^{2,q} ((0,2\pi)\times (\sigma, 3-\sigma))}$. We now can use Morrey's embedding to get a uniform estimate on $\|v_j\|_{C^{1,\alpha} ([0,2\pi]\times [2\sigma, 3-2\sigma])}$ for every $\alpha <1$. We now turn again to \eqref{e:transmission}, to conclude that the right hand side has a uniform $C^\alpha$ bound in $[2\sigma, 3-2\sigma]$ for every $\alpha >0$. This gives uniform $C^{1,\alpha}$ bounds on the coefficient of the elliptic operator in the left hand side of \eqref{e:bootstrap} and uniform $C^\alpha$ bounds on the right hand side of  \eqref{e:bootstrap}. We can thus infer a uniform $C^{2,\alpha}$ bound in $[0,2\pi]\times [3\sigma, 3-3\sigma]$ on $v_j$ from elliptic regularity.

It thus remains to prove \eqref{e:condizione_extra}. The latter will come from \eqref{e:AM-identity}. First of all, we fix $t\in(0,3)$, set $t_j := e^{-t-j}$ and observe that $\partial B_{t_j} \cap K_j$ consists of a single point $p_j$. We can thus apply Corollary~\ref{c:AM-variations}. Hence using the relation between harmonic conjugates, we rewrite \eqref{e:AM-identity} as 
\begin{align*}
0=\int_{\partial B_{t_j}\setminus \{p_j\}} \left(|\nabla w_j|^2 n  \cdot \tau (p_j) - 2 \frac{\partial w_j}{\partial \tau} \nabla w_j  \cdot (n - n (p_j))\right)\, d\mathcal{H}^1\, . 
\end{align*}
Next, we assume without loss of generality that $p_j =(t_j,0)$ and rewrite the latter equality using polar coordinates:
\begin{equation}\label{e:A+B=0}
\underbrace{t_j\, \int_0^{2\pi} \Big(t_j w_{j,r}^2 -  \frac{1}{t_j} w_{j,\phi}^2\Big) (\phi, t_j) \sin \phi\, d\phi}_{=:A_{1,j}}
- \underbrace{2 \int_0^{2\pi} \left(w_{j,r} w_{j,\phi}\right) (\phi, t_j) (1 - \cos \phi)\, d\phi}_{=:A_{2,j}} = 0
\end{equation}
We next write $w_j$ in terms of $v_j$, $\gamma_j$ and $\theta_j$  as
\[
w_j (\phi,r) = r^{\sfrac{1}{2}} 
\isq(\phi - \delta_j \theta_j (-\ln r-j))
+ \delta_j r^{\sfrac{1}{2}} v_j (\phi - \delta_j \theta_j (-\ln r-j), -\ln r-j)\, \, ,
\]
Note that, having normalized so that $p_j =(t_j,0)$, we conclude that $\theta_j (-\ln t_j) = \theta_j (t+j) = 0$. Using the latter we compute:
\begin{align}
\frac{\partial w_j}{\partial r} (\phi, t_j) = & \underbrace{t_j^{-\sfrac{1}{2}} \frac{\isq(\phi)}{2}}_{=:a_j (\phi)}
+ \delta_j \underbrace{t_j^{-\sfrac{1}{2}} \left(\frac{\dot \theta_j (t)}{\sqrt{2\pi}} 
\cos \frac{\phi}{2}
+ \frac{v_j (\phi, t)}{2} - v_{j,t} (\phi, t)\right)}_{=: b_j (\phi)} + o (\delta_j)\label{e:dr10}\\
\frac{\partial w_j}{\partial \phi} (\phi, t_j) = & \underbrace{t_j^{\sfrac{1}{2}}\frac{\rad(\phi)}{2}}_{=: c_j (\phi)} + \delta_j\underbrace{t_j^{\sfrac{1}{2}} 
v_{j,\phi} (\phi, t)}_{=:d_j (\phi)}\label{e:dphi10}\, .
\end{align}
Note now that the function $a_j^2 - \frac{1}{t_j} c_j^2$ is even. Since $\sin \phi$ is odd, we thus conclude
\begin{align*}
A_{1,j} &= 2 \delta_j \int_0^{2\pi} \left(t_j a_j (\phi) b_j (\phi) - t_j^{-1} c_j (\phi) d_j (\phi)\right)\, \sin \phi\, d\phi\,  + o (\delta_j)\, . 
\end{align*}
Letting $j\to\infty$ we obtain
\begin{align}
\lim_{j\to\infty} \delta_j^{-1} A_{1,j} = &\sqrt{\frac{2}{\pi}} \int_0^{2\pi} 
 \left(\frac{\dot \theta (t)}{\sqrt{2\pi}} 
 \cos \frac{\phi}{2} +\frac{v (\phi, t)}{2} - v_t (\phi, t)\right)
\sin \frac{\phi}{2}\sin \phi\, d\phi\notag\\
 &-\sqrt{\frac{2}{\pi}} \int_0^{2\pi}   
 v_\phi (\phi, t)
 \cos \frac{\phi}{2}\sin \phi\, d\phi\, .\label{e:limite-A}
\end{align}
By direct computation it follows that 
\[
\int_0^{2\pi} a_j(\phi)  c_j(\phi)  (1 - \cos \phi)\, d\phi 
=\frac18 \int_0^{2\pi} \sin \phi 
  (1 - \cos \phi)\, d\phi = 0\, ,
\]
from which we conclude
\begin{align*}
A_{2,j} & = 2 \delta_j \int_0^{2\pi} (a_j (\phi) d_j (\phi) + b_j (\phi) c_j (\phi)) (1 - \cos \phi)\, d\phi + o (\delta_j)\, .  
\end{align*}
Hence 
\begin{align}
 \lim_{j\to \infty} \delta_j^{-1} A_{2,j} 
= & \sqrt{\frac{2}{\pi}} \int_0^{2\pi} 
v_\phi (\phi, t)
\sin \frac{\phi}{2}(1-\cos \phi)\, d\phi\, \nonumber\\
&+ \sqrt{\frac{2}{\pi}} \int_0^{2\pi} \left(\frac{\dot \theta (t)}{\sqrt{2\pi}} \cos \frac{\phi}{2} +\frac{v (\phi, t)}{2} - v_t (\phi, t)\right)
\cos \frac{\phi}{2} (1-\cos \phi)\, d\phi\, .\label{e:limite-B}
\end{align}
Combining \eqref{e:limite-A} and \eqref{e:limite-B} with \eqref{e:A+B=0} we conclude
\begin{align*}
0 = & \frac{\dot \theta (t)}{2\pi} 
\int_0^{2\pi} \left(\sin^2 \phi - 2\cos^2\frac{\phi}{2} (1-\cos \phi)\right)\, d\phi\\
 & +\sqrt{\frac{2}{\pi}} \int_0^{2\pi} \left( \frac{v(\phi, t)}{2}  - v_t (\phi, t)\right) \left(\sin \frac{\phi}{2}\sin \phi - \cos\frac{\phi}{2} (1-\cos \phi)\right)\, d\phi\\ & - \sqrt{\frac{2}{\pi}} \int_0^{2\pi} 
v_\phi (\phi, t)
\left( \cos\frac{\phi}{2} \sin \phi+ \sin \frac{\phi}{2}(1-\cos \phi) \right)\, d\phi\, .
\end{align*}
Using the identities
\begin{align*}
&\sin^2 \phi - 2\cos^2\frac{\phi}{2} (1-\cos \phi)=
\sin^2 \phi - (1 + \cos\phi)(1-\cos \phi)=0\, , \\
&\sin \frac{\phi}{2}\sin \phi - \cos\frac{\phi}{2} (1-\cos \phi)=
\sin \frac{\phi}{2}\sin \phi +\cos\frac{\phi}{2}\cos \phi - \cos\frac{\phi}{2} 
=0\, , \\
&\cos\frac{\phi}{2} \sin \phi+ \sin \frac{\phi}{2}(1-\cos \phi)=
\cos\frac{\phi}{2} \sin \phi - \sin \frac{\phi}{2} \cos \phi+\sin \frac{\phi}{2}
=2\sin \frac{\phi}{2}\, ,
\end{align*}
we conclude \eqref{e:condizione_extra}.

\noindent{\bf Case $\lambda>0$.} We rewrite the equation in \eqref{e:odd-PDE} above in the following way:
\begin{align}
v_{j,\phi\phi} + v_{j,tt} 
=& \underbrace{-\frac{v_j}{4} + v_{j,t} 
+ (\ddot\theta_j - \dot\theta_j) (\isq_\phi+\delta_j v_{j,\phi})
 - \delta_j \dot\theta^2_j \isq_{\phi\phi}
 - \delta_j^2 \dot \theta_j^2 v_{j,\phi\phi}
 +2\delta_j \dot\theta_j v_{j,t\phi}}_{=:F_j}\,, 
\label{e:bootstrap lambda>0}
\end{align}
and observe that, as before in case $\lambda=0$, we can now get a 
uniform bound on the norms $\|F_j\|_{L^q ((0,2\pi)\times (0, 3))}$ for every $q<\infty$. 

With fixed $\sigma\in(0,\frac32)$ consider a domain $\Lambda$ 
diffeomorphic to the circle such that $(0,2\pi)\times(0,3)\subset\Lambda
\subset(0,2\pi)\times(-\sigma,3+\sigma)$, and extend it to $\Lambda$ with $C^{1,1-\varepsilon}(\partial\Lambda)$ and 
$H^{\sfrac32}(\partial\Lambda)$ norms bounded by a multiple of the corresponding one for $\delta_j^{-1}H_1(\cdot+j)$. Denote by $\widetilde{H}_j$ the extended function, and define the auxiliary functions $v_j^{(1)}$ and $v_j^{(2)}$ to be the solutions respectively of
\begin{equation*}
    \begin{cases}
v_{j,\phi\phi}^{(1)} + v_{j,tt}^{(1)}=0 & \hbox{on $\Lambda$}\cr\cr
v_j^{(1)}=\widetilde{H}_j & \hbox{on $\partial\Lambda$}\, ,
    \end{cases}
\end{equation*}
and of
\begin{equation*}
    \begin{cases}
v_{j,\phi\phi}^{(2)} + v_{j,tt}^{(2)} = F_j\chi_{(0,2\pi)\times(0,3)} & \hbox{on $\Lambda$}\cr\cr
v_j^{(2)}=0 & \hbox{on $\partial\Lambda$}\, .
    \end{cases}
\end{equation*}
\cite[Corollary 8.36, Theorems 8.29 and 8.16]{GT} imply that $v_j^{(1)}\in C^{1,1-\varepsilon}(\overline{\Lambda})$ with  
\[
\|v_j^{(1)}\|_{C^{1,1-\varepsilon}(\overline{\Lambda})}\leq c\|\delta_j^{-1}H_1(\cdot+j)\|_{C^{1,1-\varepsilon}([0,3])}\, ,
\]
where the constant $c$ depends on $\Lambda$. 
In particular, from \eqref{e:Hi scaled} we infer that $v_j^{(1)}$ 
converges strongly to zero in $C^{1,1-\varepsilon}(\overline{\Lambda})$.
Moreover, \cite[Theorem~5.1]{LM68-I} imply that $v_j^{(1)}\in H^2(\Lambda)$ with  
\[
\|v_j^{(1)}\|_{H^2(\Lambda)}\leq c\|\delta_j^{-1}H_1(\cdot+j)\|_{H^{\sfrac32}([0,3])}\, ,
\]
where the constant $c$ depends on $\Lambda$. Thus, from \eqref{e:Hi scaled bis} we infer that $v_j^{(1)}$ converges to zero strongly in $H^2((0,2\pi)\times (0,3))$, as well. 
In addition, \cite[Theorems 9.13 and 9.15]{GT} implies that 
$v_j^{(2)}\in W^{2,p}((0,2\pi)\times (\sigma,3-\sigma))$ for all $p>2$.

The function $v_j-v_j^{(1)}-v_j^{(2)}$ is then harmonic on 
$(0,2\pi)\times(0,3)$ with null Dirichlet boundary conditions on
$\{0,2\pi\}\times(0,3)$. By odd reflection with respect to those segments we find an harmonic function on $(-2\pi,4\pi)\times(0,3)$. By interior estimates for harmonic functions (see for instance \cite[Theorem 2.10]{GT}) we conclude that $v_j-v_j^{(1)}-v_j^{(2)}\in C^{\infty}([0,2\pi]\times[\sigma,3-\sigma])$ with
\[
\|D^\beta(v_j-v_j^{(1)}-v_j^{(2)})\|_{C^0([0,2\pi]\times[\sigma,3-\sigma])}\leq
({n|\beta|}{\sigma^{-1}})^{|\beta|}
\|v_j-v_j^{(1)}-v_j^{(2)}\|_{C^0([0,2\pi]\times[\sigma,3-\sigma])}\, ,
\]
for any multi-index $\beta$, where $|\beta|$ denotes the length of 
$\beta$. 
From this we conclude item (d). 

Hence, $v_j$ turns out to converge in 
$C^{1,1-\varepsilon}([0,2\pi]\times[\sigma,3-\sigma])$ and strongly in 
$H^2((0,2\pi)\times(\sigma,3-\sigma))$.
Finally, from \eqref{e:transmission} 
we then conclude that $\theta_j$ converge strongly in $H^2((\sigma,3-\sigma))$.

Let us now turn to the proof of \eqref{e:condizione_extra}.
Assume $p_j=(r_j,0))$ where $r_j:=e^{-t-j}$
for some $t\in(0,3)$. Then we rewrite  \eqref{e:AM-identity} by taking advantage of \eqref{e:AM-identity ua}, recalling that $\nabla^\perp (u_j+u_{j,a})=\nabla w_j$ we conclude
\begin{align*}
\underbrace{\int_{\partial B_{r_j}\setminus \{p_j\}} \left(|\nabla w_j|^2 n \cdot \tau (p_j) - 2 \frac{\partial w_j}{\partial \tau} \nabla w_j \cdot (n - n (p_j))\right)\, d\mathcal{H}^1}_{=:B_j}
= C_{1,j}+ C_{2,j} + C_{3, j}
\end{align*}
where
\begin{align*}
C_{1,j} &= 2\int_{\partial B_{r_j}\setminus \{p_j\}}\Big(
(\nabla w_j\cdot\nabla^\perp u_{j,a} ) n \cdot \tau (p_j)+  \frac{\partial w_j}{\partial \tau}
\nabla u_{j,a} \cdot(\tau-\tau(p_j))\\
&\qquad\qquad
-\frac{\partial u_{j,a} }{\partial n}\nabla w_j\cdot(n - n (p_j))\Big)d\mathcal{H}^1\\
C_{2,j}&= 2\param \int_{B_{r_j}\setminus K_j}(u_j-g_j)
\nabla w_j\cdot(n-n(p_j))
-2\int_{B_{r_j}}\nabla^T u_{j,a}\cdot D\tau\nabla u_{j,a}\\
C_{3,j} &= -2\param \int_{B_{r_j}\cap K_j}
\big(|u^+_j-g_{j,K_j}|^2-|u^-_j-g_{j,K_j}|^2\big)\tau(p_j)\cdot \nu\, d\mathcal{H}^1\, .
\end{align*}
Since the term $B_j$ equals the sum of the terms $A_{1,j}$ and $A_{2,j}$ already discussed in the proof of the analogous identity when $\lambda=0$, we immediately conclude that 
\begin{equation}\label{e:Aj}
\lim_{j\to+\infty}\delta_j^{-1}B_j =
-2\sqrt{\frac2\pi}\int_0^{2\pi}v_\phi (\phi, t ) \sin\frac \phi 2\, d\phi\,.
\end{equation}
Next, recall that $\nabla u_{j,a} (0) = 0$ while $u_{j,a}\in C^{1,1^-}$, so that in particular we can achieve $\|\nabla u_{j,a}\|_{L^\infty (B_r)} \leq C r^{1-\sfrac{\hat\varepsilon}{2}}$, where $\hat\varepsilon>0$ has been chosen in the definition of $\delta_j$
(cf. \eqref{e:delta}).
Since  $|\nabla w_j (x)|\leq C |x|^{-\sfrac12}$, we thus can estimate 
\[
| C_{1,j}|\leq C\|\nabla w_j\|_{L^\infty(B_{r_j})}
\|\nabla u_{j,a}\|_{L^\infty(B_{r_j})}r_j
\leq Cr_j^{\sfrac32-\sfrac{\hat\varepsilon}{2}}
\]
where $C$ is a universal constant. On the other hand, 
$\delta_j \geq C r_j^{\sfrac32-\hat\varepsilon}$ for some positive constant $C$ and so
\begin{equation}\label{e:Bj}
\lim_{j\to+\infty}\delta_j^{-1} C_{1,j}=0\,.
\end{equation}
As for $C_{2,j}$ we write 
\begin{align*}
|C_{2,j}| &\leq \left|\int_{B_{r_j}\setminus K}(u_j-g_j)
\nabla w_j\cdot(n-n(p_j))\right|
+\left|2\int_{B_{r_j}}\nabla^T u_{j,a}\cdot D\tau\nabla u_{j,a}\right|\\
&\leq C \int_{B_{r_j}} |x|^{-\sfrac12}
\leq C r_j^{\sfrac32}\, . 
\end{align*}
Likewise, 
\begin{align*}
&\left|\int_{B_{r_j}\cap K}\big(|u^+_j-g_{j,K_j}|^2-|u^-_j-g_{j,K_j}|^2\big)\tau(p_j)\cdot \nu\, d\mathcal{H}^1\right|\\
&=\left|\int_{B_{r_j}\cap K}\big(|u^+_j|^2-|u^-_j|^2
-2(u^+_j-u^-_j)g_{j,K_j}\big)\tau(p_j)\cdot \nu\, d\mathcal{H}^1\right|\leq C \|g_j\|_{\infty}r_j^{\sfrac32}
\leq C r_j^{\sfrac32}\, ,
\end{align*}
where we have used that $|u_j^\pm (x)|\leq C |x|^{\sfrac12}$ (cf. \eqref{e:stima tip 1}), 
the estimate $\|g_{j,K_j}\|_{L^\infty(\Omega.\mathcal H^1\restrsmall K_j)}\leq\|g_j\|_{\infty}$, and the density upper bound in \eqref{e:upper bound}).
Therefore, thanks to the very definition of $\delta_j$ in \eqref{e:delta} we conclude that
\begin{equation}\label{e:Cj}
\lim_{j\to+\infty}\delta_j^{-1}(C_{2,j}+C_{3,j})=0\,.
\end{equation}
Collecting \eqref{e:Aj}-\eqref{e:Cj} we deduce again \eqref{e:condizione_extra}.
\end{proof}

\section{Spectral analysis}\label{s:spectral analysis}

In this section we will find a suitable representation of solutions to \eqref{e:lineare} on $[0,2\pi]\times (0, 3)$, based on the spectral analysis of a closely related linear PDE. 

To this aim we introduce the following terminology: a function $h$ on $(0,2\pi)\times (0,3)$ will be called even if $h (\phi, t) = h (2\pi-\phi, t)$ and odd if $h (\phi, t) = - h (2\pi - \phi, t)$. Moreover, a general $h$ can be split into the sum of its odd part $h_j^o (\phi, t) := \frac{h (\phi, t) - h (2\pi-\phi, t)}{2}$ and its even part $h_j^e (\phi, t) = \frac{h (\phi, t) + h (2\pi-\phi, t)}{2}$. Note finally that, if $h$ is even (resp. odd), then $\partial_\phi^j \partial_t^k h$ is odd (resp. even) for $j$ odd, and even (resp. odd) 
for $j$ even for every $k$. \index{odd part@odd part}\index{even part@even part}
\index{part, odd@part, odd}\index{part, even@part, even}
\index[simb]{aalh^o@$h^o$}\index[simb]{aalh^e@$h^e$}\index[simb]{aalv^o@$v^o$}\index[simb]{aalv^e@$v^e$}

In what follows it will be convenient to consider the change of variables
\begin{equation}\label{e:zeta}
\zeta (\phi, t) := v^o (\phi, t) - \theta (t) \isq_\phi (t) = 
v^o (\phi, t) - \frac{\theta (t)}{\sqrt{2\pi}} \cos \frac{\phi}{2}\, .
\end{equation}

\begin{lemma}\label{l:Ventsell}
The pair $(v, \theta)\in H^2((0,2\pi)\times(0,3))\times H^3((0,2\pi))$ solves \eqref{e:lineare} if and only if $v^e,\,\zeta\in H^2((0,2\pi)\times(0,3))$ are such that $v^e(\cdot,t)$ is even for every $t\in(0,3)$ and it solves the partial differential equation with homogeneous boundary conditions 
\begin{equation}\label{e:ve}
\left\{
\begin{array}{l}
v^e_{tt} + v^e_{\phi\phi} + \frac{v^e}{4} - v^e_t = 0\\ \\
v^e (0, t) = v^e (2\pi, t) = 0\,,
\end{array}
\right.
\end{equation} 
$\zeta(\cdot,t)$ is odd for every for $t\in(0,3)$, $\zeta (0,0) = 0$,
$\zeta(0,t)=-\frac{\theta(t)}{\sqrt{2\pi}}$, and $\zeta$ solves
the partial differential equation with Ventsel boundary conditions\index{Ventsel boundary conditions@Ventsel boundary conditions}
\begin{equation}\label{e:Ventsell}
\left\{
\begin{array}{l}
\zeta_{tt} + \zeta_{\phi\phi} + \frac{\zeta}{4} - \zeta_t = 0\\ \\
\zeta_\phi (0, t) + \frac{\pi}{2} \left(\frac{\zeta}{4} (0,t) + \zeta_{\phi\phi} (0,t)\right) = 0\, ,
\end{array}
\right.
\end{equation} 
and in addition
$\theta\in H^3((0,2\pi))$ solves
\begin{equation}\label{e:theta eq}
\left\{
\begin{array}{l}
\dot\theta (t) -\ddot\theta (t)
= 2\sqrt{\frac2\pi}\zeta_\phi(0,t)\\ \\
\theta(0) = 0\, .
\end{array}
\right.
\end{equation} 

\end{lemma}
Taking into account standard regularity theory, the lemma is reduced to elementary computations which are left to the reader.

From \eqref{e:theta eq} it is evident that the decay properties of $\theta$ are related only to the odd part $v^o$ of the solution $v$ to \eqref{e:lineare} (cf. \eqref{e:zeta}). 
The analysis of the latter would suffice in case $\lambda=0$. Instead, for $\lambda>0$ it is necessary to discuss the spectral property of the even part $v^e$ of the solutions to the linearized system, as well.

We aim at a representation for $v^e$ and $\zeta$, i.e. a representation as a series of functions in $\phi$ with coefficients depending on $t$, for which we can reduce \eqref{e:ve}, \eqref{e:Ventsell}, respectively, to an independent system of ODEs for the coefficients. To that aim we introduce the spaces 
\begin{equation}\label{e:E}
\mathscr{E} := \{\psi\in H^1 ((0, 2\pi)): 
\psi (\phi) =  \psi (2\pi -\phi)\}\, ,
\end{equation}
and
\begin{equation}\label{e:O}
\mathscr{O} := \{\psi\in H^1 ((0, 2\pi)): 
\psi (\phi) = - \psi (2\pi -\phi)\}\, .
\end{equation}
We start off discussing the spectral analysis for the even part 
which is standard. Instead, that for the odd part is more subtle in view of the Ventsel boundary conditions.
\index[simb]{aalEscr@$\mathscr{E}$}\index[simb]{aalOscr@$\mathscr{O}$}
\index{odd part@odd part}\index{even part@even part}
\index{part, odd@part, odd}\index{part, even@part, even}

\subsection{Spectral analysis of the even part}
The representation for $v^e$ is given in the next proposition.
\begin{proposition}\label{p:even}
If $v^e\in H^2 ((0,2\pi)\times (0, 3))$ is even and $v^e(0,t)=0$ for every $t\in(0,3)$ then
\begin{equation}\label{e:representation_even}
v^e(\phi,t)= \sum_{k=0}^\infty \frac{a_k(t)}{\sqrt\pi}\sin\big((\textstyle{k+\frac 12}) \phi\big)
\end{equation}
where:
\begin{itemize}
\item[(a)] $C^{-1} \sum_k k^2 a_k^2 (t) \leq \|v^e (\cdot, t)\|^2_{H^1}\leq C \sum_k k^2 a_k^2 (t)$ for a universal constant $C$; 
\item[(b)] For $k\geq 0$, the coefficients $a_k$ satisfy $a_k (t) = 
\langle v^e (\cdot, t), \frac1{\sqrt\pi}\sin((k+\frac 12) \cdot) \rangle_{L^2}$. 
\end{itemize}
If additionally $v^e$ solves \eqref{e:ve} then $v^e\in C^\infty ([0,2\pi]\times (0, 3))$ and
for every $k\geq 0$ the coefficients $a_k (t)$ in the expansion satisfy 
\begin{equation}\label{e:a_k(t) 0 eq}
\ddot a_k(t) - \dot a_k (t) -  (k^2 +k) a_k (t)=0\,.
\end{equation}
From \eqref{e:a_k(t) 0 eq} it follows in particular that, for every $k\geq 0$
\begin{equation}\label{e:a_k(t) 0}
 a_{k}(t)=C_{k}\,e^{(k+1) t}+D_{k}\,e^{-k t}
\end{equation}
for some $C_k$, $D_k\in\mathbb R$.
\end{proposition}

\begin{remark}
Observe that, as a consequence of the regularity of $v_e$, the coefficients $C_k$ and $D_k$ will satisfy appropriate decay estimates in $k$ as $k\uparrow \infty$, even though this is not explicitly stated.
\end{remark}

\begin{proof}
Consider for $h\in\mathscr{E}$ the eigenvalue problem 
\[
\begin{cases}
h_{\phi\phi} = \beta h\cr\cr
h (0) = 0\,.
\end{cases}
\]
By solving the ODE and imposing the Dirichlet boundary condition it is easy 
to show that $\beta=-\nu^2<0$ and that $2\nu$ is odd. From this one finds 
that necessarily $h(\phi)=A\sin\big((k+\frac 12) \phi\big)$ for some $k\in\mathbb{N}$ and $A\in\mathbb{R}\setminus\{0\}$, and that $a_k$ 
are as in item (b).

The rest of the statement follows easily from standard Fourier analysis, and explicit calculations.
\end{proof}

\subsection{Spectral analysis of the odd part}

The representation for $\zeta$ is detailed in the following 

\begin{proposition}\label{p:odd}
If $\zeta\in H^2 ((0,2\pi)\times (0, 3))$ is odd then
\begin{equation}\label{e:representation_odd}
\zeta(\phi,t)= \sum_{k=0}^\infty a_k(t)\zeta_k (\phi)
\end{equation}
where:
\begin{itemize}
\item[(a)] $C^{-1} \sum_k \nu_k^2 a_k^2 (t) \leq \|\zeta (\cdot, t)\|^2_{H^1((0,2\pi))}\leq C \sum_k \nu_k^2 a_k^2 (t)$ for a universal constant $C$; 
\item[(b)] The functions $\zeta_k$ are defined in Section~\ref{s:spectral}
(cf. \eqref{e:zeta_1}-\eqref{e:zeta_0});
\item[(c)] For $k\geq 2$, the coefficients $a_k$ satisfy $a_k (t) = \langle \zeta (\cdot, t), \zeta_k \rangle$ for the bilinear symmetric form $\langle\cdot, \cdot\rangle$\index[simb]{aalquasiprodotto@$\langle \cdot, \cdot \rangle$} defined in Section~\ref{s:ventsel} (cf. \eqref{e:quasi_prodotto}), 
\item[(d)] the coefficients $a_0 (t)$ and $a_1 (t)$ are given by $a_0 (t) = \mathcal{L}_0 (\zeta (\cdot, t))$ and $a_1 (t) = \mathcal{L}_1 (\zeta (\cdot, t))$ for appropriately defined linear bounded functionals $\mathcal{L}_0, \mathcal{L}_1 : \mathscr{O} \to \mathbb R$. 
\end{itemize}
Next, if additionally $\zeta$ solves \eqref{e:Ventsell} then $\zeta\in C^\infty ([0,2\pi]\times (0, 3))$ and
for every $k\geq 2$ the coefficients $a_k (t)$ in the expansion satisfy 
\begin{equation}\label{e:a_k(t) eq}
\ddot a_k(t) - \dot a_k (t) - (\textstyle{\nu_k^2 - \frac{1}{4}}) a_k (t)=0    
\end{equation}
where the number $\nu_k$'s are given in Lemma~\ref{l:autovalori} 
(cf. \eqref{e:zeros}). In particular, for every $k\geq 2$
\begin{equation}\label{e:a_k(t)}
a_k(t)=C_k e^{\mu_{k,-} t}+D_k e^{\mu_{k,+} t}
\end{equation}
where the $C_k$ and $D_k$ are constants and, 
\[
\mu_{k,\pm} = \frac12\pm \nu_k\,.
\]
\end{proposition}

The proof is an obvious consequence of Proposition~\ref{p:spettro} below, which will be the main focus of the rest of this section.

\subsection{The Ventsel boundary condition}\label{s:ventsel}
For every $\psi\in \mathscr{O}$ we look for solutions 
$h\in \mathscr{O}$ of the following equation:
\begin{equation}\label{e:ventsel}
 \begin{cases}
 \displaystyle{h_{\phi\phi} = \psi}\cr\cr
 \displaystyle{h_\phi (0) = - \frac{\pi}{2} \left(\frac{h (0)}{4} + h_{\phi\phi} (0)\right) \, .}
\end{cases}
\end{equation}
The following is an elementary fact of which we include the proof for the reader's convenience.
\begin{lemma}\label{l:risolvente}
For every $\psi \in \mathscr{O}$ there is a unique solution $h:= \mathscr{K} (\psi)\in \mathscr{O}$ of 
\eqref{e:ventsel}. In fact the linear operator $\mathscr{K} : \mathscr{O} \to \mathscr{O}$ is compact.\index[simb]{aalKscr@$\mathscr{K}$}
\end{lemma}
\begin{proof}
$h\in \mathscr{O}$ solves the first equation in \eqref{e:ventsel} if and only if
\begin{equation}\label{e:sol_esplicita}
h (\phi) = h_\phi (\pi) (\phi -\pi) + \underbrace{\int_\pi^\phi \int_\pi^\tau \psi (s)\, ds\, d\tau}_{=: \Psi (\phi)}\, .
\end{equation}
On the other hand the initial condition holds if and only if 
\[
h_\phi(\pi) \left(\frac{\pi^2}{8}-1\right) =  \Psi' (0) + \frac{\pi}{2} \left(\frac{\Psi(0)}{4} + \Psi'' (0)\right)\, .
\]
Since $\Psi$ is determined by $\psi$, the latter determines uniquely $h_\phi(\pi)$ and thus shows that there is one
and only one solution $h = \mathscr{K} (\psi) \in \mathscr{O}$ of \eqref{e:ventsel}. Moreover, we obviously have
\[
\|\mathscr{K} (\psi)\|_{H^3} \leq C \|\psi\|_{H^1}\, ,
\]
which shows that the operator is compact.
\end{proof}

We next introduce in $\mathscr{O}$ a continuous bilinear map
\begin{equation}\label{e:quasi_prodotto}
\langle u, v \rangle := \int_0^{2\pi} u_\phi v_\phi - \frac{1}{4} \int_0^{2\pi} uv\, .
\end{equation}
\index[simb]{aalquasiprodotto@$\langle \cdot, \cdot \rangle$}
If $\langle \cdot, \cdot \rangle$ were a scalar product on $\mathscr{O}$, $\mathscr{K}$ would be a self-adjoint operator on $\mathscr{O}$ with respect to it and we would conclude that there is an 
orthonormal base made by eigenfunctions of $\mathscr{K}$. Unfortunately $\langle \cdot, \cdot \rangle$ is 
{\em not} positive definite, it has a one dimensional radical. This causes some technical complications.

\begin{lemma}\label{l:autoaggiunto} 
The bilinear map $\langle \cdot, \cdot\rangle$ satisfies the following properties:
\begin{itemize}
\item[(a)] $\langle v, v \rangle \geq 0$ for every $v\in \mathscr{O}$;
\item[(b)] $\langle v, v\rangle = 0$ if and only if $v (\phi) = \mu \cos \frac{\phi}{2}$ for some constant $\mu$;
\item[(c)] $\langle v, \cos \frac{\phi}{2}\rangle = 0$ for every $v\in \mathscr{O}$;
\item[(d)] $\langle \mathscr{K} (v), w\rangle = \langle v, \mathscr{K} (w)\rangle$ for every $v,w\in \mathscr{O}$.
\end{itemize}
\end{lemma}
\begin{proof} {\bf (a) \& (b)}
First observe that (a) is equivalent to
\begin{equation}\label{e:sharp}
\frac{1}{4} \int_0^{2\pi} v^2 \leq \int_0^{2\pi} v_\phi^2\, .
\end{equation}
If we write $v$ using the Fourier series expansion
$v (\phi) = \sum_{k=1}^\infty \alpha_k \cos \frac{k\phi}{2}$,
the inequality becomes obvious and
it is also clear that equality holds if and only if $\alpha_k =0$ for every $k\geq 2$. 

\medskip

{\bf (c)} Let $z (\phi) := \cos \frac{\phi}{2}$ and observe that $\frac{z}{4} + z_{\phi\phi}=0$ and that
$z_\phi (0)= z_\phi (2\pi) = 0$. We therefore compute
\begin{align*}
\langle w,z\rangle = & \int_0^{2\pi} w_\phi z_\phi - \frac{1}{4} \int_0^{2\pi} zw =  w z_\phi \Big|_0^{2\pi} - \int_0^{2\pi} w \left(z_{\phi\phi} + \frac{z}{4}\right) = 0\, .
\end{align*}

\medskip

{\bf (d)} Consider $z = \mathscr{K} (v)$ and $u = \mathscr{K} (w)$. We then compute
\begin{align*}
\langle \mathscr{K} (v), w \rangle = & \langle z, u_{\phi\phi}\rangle = \int_0^{2\pi} z_\phi u_{\phi\phi\phi} - \frac{1}{4} \int_0^{2\pi} z u_{\phi\phi}\\
=& z_\phi u_{\phi\phi} \Big|_0^{2\pi} - \int_0^{2\pi} z_{\phi\phi} u_{\phi\phi} - \frac{1}{4} z u_\phi \Big|_0^{2\pi} +
\frac{1}{4} \int_0^{2\pi} z_\phi u_\phi\\
=& z_\phi u_{\phi\phi} \Big|_0^{2\pi} - z_{\phi\phi} u_{\phi}\Big|_0^{2\pi} + \int_0^{2\pi} z_{\phi\phi\phi} u_\phi  - \frac{1}{4} z u_\phi \Big|_0^{2\pi}
+  \frac{1}{4} z_\phi u \Big|_0^{2\pi}- \frac{1}{4} \int_0^{2\pi} z_{\phi\phi} u\\
= & \left. z_\phi \left(u_{\phi\phi} + \frac{u}{4}\right)\right|_0^{2\pi} - \left. u_\phi \left(z_{\phi\phi} + \frac{z}{4}\right)\right|_0^{2\pi} + \langle z_{\phi\phi} , u\rangle\\
= & - \frac{2}{\pi} z_\phi u_\phi\Big|_0^{2\pi} + \frac{2}{\pi} z_\phi u_\phi \Big|_0^{2\pi} + \langle v, \mathscr{K} (w)\rangle =
\langle v , \mathscr{K} (w)\rangle\, . \qedhere
\end{align*}
\end{proof}

\subsection{Spectral decomposition}\label{s:spectral} We are now ready to prove the following spectral analysis. First of all we start with the following

\begin{lemma}\label{l:autovalori} If $\beta$ is a real number and 
$h\in \mathscr{O}$ a solution of the following eigenvalue problem
\begin{equation}\label{e:autovalore}
\left\{
\begin{array}{l}
h_{\phi\phi} = \beta h\\ \\
h_\phi (0) = - \frac{\pi}{2} \left(\frac{h (0)}{4}+ h_{\phi\phi} (0) \right)
\end{array}\right.
\end{equation}
then
\begin{itemize}
\item[(a)] $\beta < 0$ and if we set $\beta = -\nu^2$ for $\nu>0$, then $\nu$ is a positive solution of
\begin{equation}\label{e:zeros}
\nu \cos \nu \pi = \frac{\pi}{2} \left(\frac{1}{4} - \nu^2\right) \sin \nu \pi\, .
\end{equation}
\item[(b)] $h$ is a constant multiple of $\sin (\nu (\phi -\pi))$. 
\item[(c)] The positive solutions of \eqref{e:zeros} are given by an increasing sequence $\{\nu_k\}_k\in \mathbb N$ in which $\nu_1 = \frac{1}{2}$, $\nu_2>\frac{3}{2}$ and 
\begin{equation}\label{e:asintotico}
\lim_{k\to \infty} \frac{\nu_k}{k} = 1\, 
\end{equation}
\end{itemize}
\end{lemma}

We will postpone the proof of the lemma and introduce instead the following notation. For $k=1$ we set 
\begin{equation}\label{e:zeta_1}
\zeta_1 (\phi) = \cos \frac{\phi}{2}\,, 
\end{equation}
while for $k>1$ we let 
\begin{equation}\label{e:zeta_k}
\zeta_k(\phi) := c_k \sin(\nu_k(\phi-\pi))\,,
\end{equation}
where $c_k$ is chosen so that $\langle \zeta_k, \zeta_k\rangle = 1$.  Furthermore, we set 
\begin{equation}\label{e:zeta_0}
\zeta_0 (\phi) := (\phi-\pi) \sin \frac{\phi}{2}\,,    
\end{equation} 
the relevance of the latter function is that it solves 
\begin{equation}\label{e:blocchetto}
\left\{
\begin{array}{l}
\zeta_{\phi\phi} = - \frac{\zeta}{4} + \zeta_1\\ \\
\zeta_\phi (0) = - \frac{\pi}{2} \left(\frac{\zeta (0)}{4} + \zeta_{\phi\phi} (0)\right)\, .
\end{array}\right.
\end{equation}
In particular, if we restrict the second derivative operator on the $2$-dimensional vector space generated by $\zeta_1$ and $\zeta_0$, its matrix representation is given by
\[
\left(
\begin{array}{ll}
-\sfrac14 & 0\\
1 & -\sfrac14
\end{array}
\right)\,.
\]
Consequently, the operator $\mathscr{K}$ is not diagonalizable in $\mathscr{O}$, which is the reason why its spectral analysis is somewhat complicated.

\begin{proposition}\label{p:spettro}
The set $\{\zeta_k \} \subset \mathscr{O}$ is an Hilbert basis for $\mathscr{O}$, namely for every $\zeta \in \mathscr{O}$ there is a unique 
choice of coefficients $\{a_k\}$ such that
\begin{equation}\label{e:espansione}
\zeta = \sum_{k=0}^\infty a_k \zeta_k \, ,
\end{equation}
where the series converges in $H^1$. The coefficients $a_k$ in \eqref{e:espansione} are determined by
\begin{equation}\label{e:coefficienti}
a _k = \langle \zeta, \zeta_k\rangle\qquad \mbox{for all $k\geq 2$,}
\end{equation}
while $a_0$ and $a_1$ are continuous linear functionals on $\mathscr{O}$. 
\end{proposition} 

\begin{proof}[Proof of Lemma~\ref{l:autovalori}]
First of all, consider $\beta =0$. An odd solution of \eqref{e:autovalore} must then take necessarily the form $c (\phi - \pi)$ and the boundary condition would imply $c=0$. If $\beta >0$ observe that a nontrivial function $\zeta\in \mathscr{O}$ solving \eqref{e:autovalore} would also satisfy $\mathscr{K} (\zeta) = \frac{\zeta}{\beta}$. If $ \beta = \nu^2 >0$ for $\nu>0$, then $h(\phi ) = c (e^{\nu (\phi-\pi)} - e^{-\nu (\phi-\pi)})$ for some constant $c$. If $c\neq 0$ the boundary condition becomes
\begin{equation}
\nu \left(e^{-\nu \pi} + e^{\nu \pi}\right) = - \frac{\pi}{2} \left(\frac{1}{4} + \nu^2\right) \left(e^{-\nu \pi} - e^{\nu \pi}\right)\, .
\end{equation}
The latter identity is equivalent to
\begin{equation}
e^{2\pi \nu} ({\pi + 4\pi \nu^2 - 8 \nu})= {\pi + 4\pi \nu^2 + 8\nu}\, .
\end{equation}
If we make the substitution $x= 2\pi \nu$, we then are seeking for zeros of the function
\[
\Phi (x) = e^x (\pi^2 + x^2 - 4x) - \pi^2 - x^2 - 4x = 0\, .
\]
The derivative is given by 
\[
\Phi' (x) = e^x (x^2 -2x + \pi^2 -4) - 2 (2+x)\, ,
\]
the second derivative by
\[
\Phi'' (x) = e^x (x^2 + \pi^2 -6) -2 \geq 3 e^x -2 > 0\, .
\]
In particular $\Phi$ is convex and $\Phi' (0) = \pi^2 - 8 > 0$. Thus $\Phi$ is strictly increasing and, since $\Phi (0) = 0$, it cannot have positive zeros.

Consider now $\mu = - \beta^2$ for $\nu>0$. A solution of the PDE in \eqref{e:autovalore} must then be a linear combination of $\sin \nu (\phi-\pi)$ and 
$\cos \nu (\phi-\pi)$: the requirement that $h\in \mathscr{O}$ excludes the multiples of $\cos \nu (\phi-\pi)$ in the linear combination. 

For $h (\phi) = \sin \nu (\phi-\pi)$ the boundary condition becomes 
\begin{equation}\label{e:autovalori}
\nu \cos (-\nu \pi) = - \frac{\pi}{2} \left(\frac{1}{4} - \nu^2\right) \sin (-\nu \pi)\, ,
\end{equation}
which is equivalent to \eqref{e:zeros}. If we introduce the unknown $x = \pi \nu$, then the equation becomes
\[
\Psi (x) := 8 x \cos x - \left(\pi^2-4 x^2\right) \sin x = 0\, .
\]
Since $\Psi' (x) = (4x^2 + 8 - \pi^2) \cos x$,
$\Psi'$ has a single zero in the open interval $(0, \frac{\pi}{2})$. Since $\Psi (0) = \Psi (\frac{\pi}{2})=0$, 
we infer that there is no zero of $\Psi$ in the open interval $(0, \frac{\pi}{2})$, i.e. any positive $\nu$ 
satisfying \eqref{e:autovalori} cannot be smaller than $\frac{1}{2}$. Moreover, as $\Psi'$ is strictly negative on 
$(\frac\pi 2,\frac32\pi)$ and $\Psi(\frac32\pi)<0<\Psi(2\pi)$, the next solution $\nu$ lies in 
$(\frac 32,2)$. 

Next, there is a unique solution $\nu_k\in(k-1,k)$, for every $k\geq 3$.
Indeed, $\Psi((k-1)\,\pi)\cdot\Psi(k\,\pi)<0$ and $\Psi'$ has a single zero in the open interval $((k-1)\,\pi, k\pi)$.
Therefore $(\nu_k)_k$ satisfies \eqref{e:asintotico}.
\end{proof}

\begin{proof}[Proof of Proposition~\ref{p:spettro}]
Let $Y$ be the closure in $H^1$ of the vector space $V$ generated by $\{\zeta_k\}_{k\geq 2}$. First of all observe that, for some constant $C$ independent of $k$, 
\begin{equation}\label{e:comparable_1}
1 = \langle \zeta_k, \zeta_k \rangle \geq C^{-1} \|\zeta_k\|_{H^1}^2\qquad \forall k\geq 2\, .
\end{equation}
Indeed set $g_k:= \sin \nu_k (\phi-\pi)$: \eqref{e:comparable_1} is then equivalent to say that the $g_k$'s satisfy the same inequality. An explicit computation shows that this is equivalent to
\begin{align*}
& \int_0^{2\pi} \cos^2 \nu_k (\phi - \pi)\, d\phi  - \frac{1}{4 \nu_k^2} \int_0^{2\pi} \sin^2 \nu_k (\phi-\pi)\, d\phi\\
\geq &\; C^{-1} \left(\int_0^{2\pi} \cos^2 \nu_k (\phi-\pi)\, d\phi +  \frac{1}{\nu_k^2} \int_0^{2\pi} \sin^2 \nu_k (\phi-\pi)\, d\phi\right)\, .
\end{align*}
For each fixed $\nu_k$ the fact that the inequality holds for a sufficiently large constant is an easy consequence of the fact that $\int \cos^2 \nu_k (\phi-\pi)$ is positive while 
$\int \sin^2 \nu_k (\phi-\pi)$ is finite, and $\nu_k\ge\nu_2>\frac32$ for $k\geq2$. 
On the other hand by \eqref{e:asintotico} both integrals converge to $\pi$ as $k\uparrow \infty$ and thus for a sufficiently large $k$ the inequality holds for $C\geq 2$. Now, for $k\neq j$ we have 
\[
\langle \zeta_k, \zeta_j\rangle = - \nu_k^2 \langle \mathscr{K} (\zeta_k), \zeta_j\rangle = - \nu_k^2 \langle \zeta_k, \mathscr{K} (\zeta_j)\rangle 
= \frac{\nu_k^2}{\nu_j^2} \langle \zeta_k, \zeta_j\rangle
\]
implying that $\langle \zeta_k, \zeta_j \rangle =0$.

We next claim that $\zeta_1 (\phi) = \cos \frac{\phi}{2}\not\in Y$. Otherwise there is a sequence $\{v_n\}\subset V$ such that $v_n\to \zeta_1$ strongly in $H^1$. $v_n$ takes therefore the form $v_n = \sum_{k=2}^{N (n)} a_{n,k} \zeta_k$. Using that $\langle v_n, v_n \rangle$ converges to $\langle \zeta_1, \zeta_1\rangle = 0$. Thus we have
\begin{equation}\label{e:si_annulla}
\lim_{n\to \infty} \sum_{k=2}^{N(n)} a_{n, k}^2 = 0\, .
\end{equation}
Now, given that the operator $\mathscr{K}$ is compact we also have that $z_n := \frac{\mathscr{K} (v_n)}{4}$ converges strongly in $H^1$ to 
$\frac{\mathscr{K} (\zeta_1)}{4} = - \cos \frac{\phi}{2}$. On the other hand
\[
z_n = - \sum_{k=2}^{N(n)} \frac{1}{4\nu_k^2} a_{n, k} \zeta_k \, .
\]
We then would have by item (c) of Lemma~\ref{l:autovalori} and \eqref{e:comparable_1}
\begin{align*}
0 < & \|\zeta_1\|^2_{H^1} = \lim_{n\to \infty} \|z_n\|_{H^1}^2 \leq \lim_{n\to \infty} \sum_{k, j=2}^{N(n)} \frac{|a_{n,j}| 
|a_{n,k}|}{16\nu_j^2 \nu_k^2} \|\zeta_k\|_{H^1}\|\zeta_j\|_{H^1}\\
\leq & C \limsup_{n\to\infty} \left(\sum_{k=2}^{N(n)} \frac{|a_{n,j}|}{j^2}\right)^2
\leq C \limsup_{n\to\infty} \sum_{k=2}^{N(n)} \frac{1}{k^4} \sum_{j=2}^{N(n)} a_{n,j}^2
\leq  C \limsup_{n\to\infty}\sum_{j=2}^{N(n)} a_{n,j}^2\stackrel{\eqref{e:si_annulla}}{=} 0\, ,
\end{align*}
Consider now the standard $H^1$ scalar product $(\cdot, \cdot)$ on $\mathscr{O}$ and for every $\zeta\in Y$ let 
$\zeta = \zeta^\perp + \zeta^\parallel$ be the decomposition of $\zeta$ into a multiple of $\zeta_1$ and an element $\zeta^\perp$ orthogonal in the scalar product $(\cdot, \cdot)$ to $\zeta_1$. Since $\zeta_1\not \in Y$ and $Y$ is closed in $H^1$, there is a constant $\alpha >0$ such that $\|\zeta^\perp\|^2_{H^1} \geq \alpha \|\zeta\|^2_{H^1}$. On the other hand using the Fourier expansion of $\zeta$ we easily see that $\langle \zeta, \zeta\rangle = \langle \zeta^\perp, \zeta^\perp\rangle \geq C^{-1} \|\zeta^\perp\|^2_{H^1}$ for some universal constant $C>0$. In particular $\mathscr{K}$ is a compact self-adjoint operator on $Y$, which implies that $\{\zeta_k\}_{k\geq 2}$ is an orthonormal basis on the Hilbert space $Y$ (endowed with the scalar product $\langle \cdot, \cdot \rangle$).

Consider now the $2$-dimensional vector space 
$Z := \{a_0 \zeta_0 +a_1 \zeta_1 : a_i \in \mathbb R\}$. If $a_0 \zeta_0 + a_1\zeta _1 = z\in Z \cap Y$, using Lemma~\ref{l:autoaggiunto} and the fact that $\langle y, \zeta_1\rangle =0$ for every $y\in Y$, we can compute
\[
\langle z, \zeta_j\rangle = a_0 \langle \zeta_0, \zeta_j\rangle = - \nu_j^2 \langle a_0 \zeta_0, \mathscr{K} (\zeta_j)\rangle = -\nu_j^2 \langle a_0 \mathscr{K} (\zeta_0), \zeta_j\rangle
= 4\nu_j^2 \langle a_0 \zeta_0, \zeta_j\rangle = 4 \nu_j^2 \langle z, \zeta_j\rangle\, 
\] 
for every $j\geq 2$. Since $\nu_j> \frac{3}{2}$ we infer that $\langle z, \zeta_j\rangle =0$, i.e. that $z=0$, since $\{\zeta_j\}_{j\geq 2}$ is an orthonormal Hilbert basis of $Y$ with respect to the scalar product $\langle \cdot, \cdot\rangle$. We have thus concluded that $Z\cap Y = \{0\}$. The proof of the proposition will be completed once we show that $Z+Y = \mathscr{O}$. Consider an element $\zeta\in \mathscr{O}$ and define 
\[
\bar\zeta := \frac{\langle \zeta_0, \zeta\rangle}{\langle \zeta_0, \zeta_0\rangle} \zeta_0 + \sum_{j\geq 2} \langle \zeta_j, \zeta\rangle \zeta_j\, .
\]
It turns out that $\bar\zeta \in Z+Y$ and that $\hat{\zeta} := \zeta - \bar\zeta$ satisfies the condition $\langle \hat\zeta, z\rangle = 0$ for every element $z\in Z+Y=:X$. We claim that the latter condition implies that $\hat\zeta$ is a constant multiple of $\cos \frac{\phi}{2}$. Indeed set $X^\perp := \{v : \langle v, w\rangle = 0 \quad \forall w\in X\}$. Then clearly $\mathscr{K} (X^\perp) \subset X^\perp$. Moreover $\mathscr{K}$ on $X^\perp$ has only one eigenvalue, namely $-4$. Consider now $X^\perp \ni v \mapsto Q (v,v) = \langle \mathscr{K} (v), \mathscr{K} (v)\rangle = \langle \mathscr{K}^2 (v), v \rangle$ and set
\begin{equation}\label{e:massimo}
m:= \sup \{Q (v,v): v\in  X^\perp \quad \mbox{and}\quad \langle v, v\rangle =1\}\, ,
\end{equation}
where at the moment $m$ is allowed to be $\infty$ as well. If $m=0$ we then have that $\mathscr{K} (v)$ is a multiple of $\zeta_1$ for every $v$ and this would imply that $v$ itself is a multiple of $\zeta_1$. We therefore assume that $m$ is nonzero.
Using the fact that $Q (v, \zeta_1) = 0$ for every $v$, we can find a maximizing sequence with Fourier expansion
\[
v_k := \sum_{j\geq 1} c_{k,j} \cos \big((\textstyle{j+\frac{1}{2}}) \phi\big)  
\]
for which we easily see that $\langle v_k, v_k \rangle \geq C^{-1} \|v_k\|_{H^1}^2$. We can thus extract a subsequence converging weakly to some $v$. $v$ clearly belongs to $X^\perp$ and, by the compactness of the operator $\mathscr{K}$ is actually a maximizer of \eqref{e:massimo}. The Euler-Lagrange condition implies then that $\mathscr{K}^2 (v) = m v + b \zeta_1$ for some real coefficients $b$. Consider now the vector space $W$ generated by $\zeta_1, v$ and $\mathscr{K} (v)$. $W$ is then either $2$-dimensional or $3$-dimensional and $\mathscr{K}$ maps it onto itself. If $W$ were three-dimensional, then the matrix representation of $\mathscr{K}|_W$ in the basis  $\zeta_1, v$ and $\mathscr{K} (v)$ would be
\[
\left(
\begin{array}{lll}
-4 & 0 & 0\\
0 & 1 & 0\\
\alpha & 0 & m\\
\end{array}
\right)
\]
Since the characteristic polynomial of the latter matrix is $(x-1) (x-m)(x+4)$, $\mathscr{K}$ would have an eigenvalue different from $-4$ on $W\subset X^\perp$, which is not possible. 
On the other hand if $W$ were $2$-dimensional, then $v$ and $\cos \frac{\phi}{2}$ would be a basis and the matrix representation of $\mathscr{K}|_W$ in that basis would be
\[
\left(
\begin{array}{ll}
-4 & 0\\
\alpha & \beta
\end{array}
\right)
\]
Since $\mathscr{K}|_W$ cannot have an eigenvalue different than $-4$ this would force $\beta = -4$. We then would have $\mathscr{K} (v) = -4v + \alpha \zeta_0$. This would imply that $v$ is an odd solution of $v_{\phi\phi} +\frac{v}{4} = \alpha \cos \frac{\phi}{2}$. The general solution of the latter equation is given by $c_1 \cos \frac{\phi}{2} + c_2 \sin \frac{\phi}{2} + \alpha (\phi - \pi) \sin \frac{\phi}{2}$, for real coefficients $c_1$ and $c_2$. The fact that $v$ is odd implies $c_2=0$, namely $c_1 \zeta_1 + \alpha \zeta_0$. The fact that $v$ is not colinear with $\zeta_1$ implies that $\alpha \neq 0$, but on the other hand since $v\in X^\perp$, $\langle v, \zeta_1\rangle =0$, which implies $\alpha =0$. We have reached a contradiction: $X^\perp$ was thus the line generated by $\zeta_1$, proving that indeed $X=\mathscr{O}$.  
\end{proof}

\section{The linear three annuli property}\label{s:tre-anelli}

We now define a functional which will be instrumental in proving a suitable decay property 
for the coefficients of solutions to \eqref{e:lineare} and hence to \eqref{e:SIS}. 
The latter called three annuli property, is a way to encode the presence of positive exponentials among the coefficients $a_k$ defined by \eqref{e:a_k(t) 0}, \eqref{e:a_k(t) eq} 
in the representations of $v^e$, $v^o$, respectively. 

We separate the behaviour of the even and odd parts. We start off with the former.
\begin{definition}\label{d:functionals even}
Consider any $0<\sigma<s<3$ real numbers and functions $v$ such that
$v$ is even, 
$v\in H^2 ((0,2\pi)\times (\sigma,s))$ and $v (0, t)=v(2\pi,t)=0$ for every $t$.

We then define the functional\index[simb]{aalGcal_e@$\mathcal{G}_e$}
\begin{align}\label{e:def_G even}
\mathcal{G}_{e} (v, \sigma,s) &:= 
\int_\sigma^s\|v_{\phi\phi}(\cdot,t)\|^2_{L^2((0,2\pi))}\, dt\,
\end{align}
\end{definition}
\begin{proposition}\label{p:tre-anelli_even}
There is a constant $C>0$ such that for all $v\in H^2 ((0,2\pi)\times (0,3))$ even with $v(0, t) =0$
\begin{equation}\label{e:Ge coerc}
 C^{-1}\int_\sigma^s\|v(\cdot,t)\|_{H^2((0,2\pi))}^2\,dt\leq 
 \mathcal{G}_e(v,\sigma,s)\leq 
 C\int_\sigma^s \|v(\cdot,t)\|_{H^2((0,2\pi))}^2\,dt\, .
\end{equation}
Moreover, there is a constant $\eta\in(0,1)$ such that the following property holds 
\begin{equation}\label{e:Ge annuli}
\mbox{If 
$\mathcal{G}_e (v, 1,2) \geq (1-\eta) \mathcal{G}_e (v, 0,1)$
then 
$\mathcal{G}_e (v, 2,3)\geq (1+\eta) \mathcal{G}_e (v, 1,2)$,}
\end{equation}
for every even solution $v\in H^2((0,2\pi)\times(0,3))$ of \eqref{e:lineare}  satisfying:
\begin{equation}\label{e:mean-zero}
\int_0^{2\pi} v (\phi, t) \sin \frac{\phi}{2}d\phi =0 \,.
\end{equation}
\end{proposition}

\begin{proof} Since $v(0, t) =v(2\pi, t) =0$, we have by the Poincar\'e inequality
\[
\|v (\cdot,t )\|_{L^2 ((0, 2\pi))} \leq C \|v_\phi (\cdot, t)\|_{L^2 ((0, 2\pi))}\, .
\]
Moreover, using that $v_\phi(\pi, t)=0 $, we conclude
\[
\|v (\cdot, t)\|_{H^2 ((0, 2\pi))}\leq C \|v_{\phi\phi} (\cdot, t)\|_{L^2 ((0, 2\pi))}\, ,
\]
which clearly implies \eqref{e:Ge coerc}.
 
We now establish \eqref{e:Ge annuli}. Recall that, since $v$ is even
and satisfies \eqref{e:mean-zero}, the Fourier decomposition of $v$ 
reads as (cf. \eqref{e:representation_even})
 \[
 v(\phi,t)=\sum_{k=1}^\infty \frac{a_{k}(t)}{\sqrt\pi}\sin\left(\textstyle{(k+\frac 12)}\phi\right),
\]
where by \eqref{e:a_k(t) 0} the coefficients $a_{k}$'s satisfy
for all $k\in\mathbb{N}$ with $k\geq1$
\[
 a_{k}(t)=C_{k}\,e^{(k+1) t}+D_{k}\,e^{-k t}
\]
for some $C_k$, $D_k\in\mathbb R$. 
A simple calculation then gives for all $k\geq 1$ 
\begin{equation}\label{e:alphakconv}
\frac{d^2}{dt^2}(a_k^2(t))\geq a_k^2(t)\geq 0,
\end{equation}
establishing the convexity of each $a_k^2$ for $k\geq 1$. 
Moreover, note that 
\begin{equation*}
\mathcal{G}_e(v,\sigma,s)=
\int_\sigma^s\underbrace{\sum_{k=1}^\infty \textstyle{(k+\frac12)^4} a_{k}^2(t)}_{h(t):=}\,dt.
\end{equation*}
We now want to argue that, there is a constant $\eta>0$ with the following property. 
If $h\geq 0$ is a nontrivial $L^1$ function such that $\ddot h\geq h$ on $(0, 3)$, in particular $h$ is convex, then 
\[
\int_1^{2} h(t)\, dt \geq (1- \eta) \int_0^1 h(t)\, dt \quad
\Longrightarrow \quad \int_2^3 h(t)\, dt \geq (1+\eta)\int_1^2 h(t)\, dt\, .
\]
Indeed, assume by contradiction this were not true and let $h_j$ be a sequence of nontrivial functions such that
$\ddot h_j \geq h_j\geq 0$ and
 \[
\int_1^2 h_j (t)\, dt \geq \max \left\{(1- \sfrac{1}{j}) \int_0^1 h_j (t)\, dt,
(1+\sfrac{1}{j})^{-1} \int_2^3 h_j(t)\, dt \right\}\, .
\]
After multiplying by a suitable constant we can then assume 
\[
\int_1^2 h_j (t)\, dt = 1\, .
\]
The convexity of the $h_j$ and the uniform bound on $\|h_j\|_{L^1((0,3))}$ implies easily a uniform bound on
$\|h_j\|_{L^\infty((\sigma,s))}$ for any $0<\sigma<s<3$ and therefore (again by convexity) a uniform Lipschitz bound on any compact subset of $(0, 3)$. This ensures the local uniform convergence of a subsequence of $h_j$ (not relabeled) to a nonnegative convex function $h$, which is $L^1$ (and thus locally finite) on the open interval $(0,3)$. 
In particular, $\int_1^2 h (t)\,dt =1$. On the other hand it is also easy to see that 
\[
\int_0^1 h(t)\, dt \leq 1,\quad\text{and}\quad \int_2^3 h(t)\, dt \leq 1\, .
\]
By the mean-value theorem this implies the existence of three points $0 <t_1 < 1<t_2<2<t_3 < 3$ where $h (t_2) \geq 1 \geq
\max \{h (t_1), h(t_3)\}$. But then the convexity of $h$ implies that $h$ must be constantly equal to $1$ on $[t_1, t_3]$.
Since the inequality $\ddot h \geq h$ is verified in the limit in the sense of distributions, this is a contradiction.
\end{proof}
Next we consider the odd part. 
\begin{definition}\label{d:functionals odd}
Fix a constant $c_0>0$ appropriately small (whose choice will be specified in Proposition~\ref{p:tre-anelli_odd} below). Consider now any couple of real numbers $0\leq\sigma<s \leq 3$ and a pair of functions $(v, \theta)$ such that
\begin{itemize}
\item[(i)] $v$ is odd, $v\in H^2 ((0,2\pi)\times (\sigma,s))$ and $v (0, t)=v(2\pi,t)=0$ for every $t$;
\item[(ii)] $\theta\in H^2 ([\sigma,s])$.
\end{itemize}
Define $\zeta$ as in \eqref{e:zeta} and let $a_k (t)$ be the coefficients in the representation \eqref{e:representation_odd} and $\nu_k$ the numbers in Lemma~\ref{l:autovalori}. We then define the functionals\index[simb]{aalGcal_o@$\mathcal{G}_o$}\index[simb]{aalEcal@$\mathcal{E}$}\index[simb]{aalFcal@$\mathcal{F}$}
\begin{align}
\mathcal{E} (v, \theta, \sigma,s) &:= \sum_{k=2}^\infty \int_\sigma^s  ({\nu_k^4 a_k^2 (t) + \ddot a_k^2 (t)})\, dt\, \label{e:def_E}\\
\mathcal{F} (v, \theta,\sigma,s) &:= \int_\sigma^s (\dot\theta^2 (t) + \ddot\theta^2 (t) + a_0^2 (t) + a_1^2 (t) + \ddot a_0^2 (t) + \ddot a_1^2 (t))\, dt\\
\mathcal{G}_o (v, \theta, \sigma,s) &:= \max\{\mathcal{E} (v,\theta, \sigma,s), c_0 \mathcal{F} (v, \theta,\sigma,s)\}
\label{e:def_G}
\end{align}
\end{definition}

\begin{proposition}\label{p:tre-anelli_odd}
There is a constant $\eta\in(0,1)$ such that the following property holds for every solution $(v,\theta)\in H^2((0,2\pi)\times (0,3))\times H^3((0,2\pi))$ of \eqref{e:lineare} with $v$ odd:
\begin{itemize}
\item[(a)] If $\mathcal{E} (v, \theta, 1,2) \geq (1-\eta) \mathcal{E} (v, \theta,0,1)$ then $\mathcal{E} (v, \theta,2,3)\geq (1+\eta) \mathcal{E} (v, \theta,1,2)$.
\end{itemize}
Furthermore, there are positive constants $C$ and $c_0$ such that the following properties hold for every solution $(v,\theta)$ of \eqref{e:lineare} with $v$ odd which satisfies in addition \eqref{e:condizione_extra}:
\begin{itemize}
\item[(b)] If $\mathcal{G}_o (v, \theta, 1,2) \geq (1-\eta) \mathcal{G}_o (v, \theta,0,1)$ then $\mathcal{G}_o (v, \theta,2,3)\geq (1+\eta) \mathcal{G}_o (v, \theta,1,2)$.
\end{itemize}
and
\begin{equation}\label{e:Go growth}
C^{-1} (\|v\|_{H^2 ((0,2\pi)\times(\sigma,s))}^2 + \|\dot\theta\|_{H^1 ((\sigma,s))}^2) \leq \mathcal{G}_o (v, \theta, \sigma,s)
\leq C (\|v\|_{H^2 ((0,2\pi)\times(\sigma,s))}^2 + \|\dot\theta\|_{H^1 ((\sigma,s))}^2)\,,
\end{equation}
for all $0\leq\sigma<s\leq 3$.

\end{proposition}

\begin{proof} In order to prove claim (a) consider any of the functions $a_k (t)$ and $\ddot a_k (t)$ and call it $\omega (t)$, and observe we know $k\geq 2$ by assumption. From Proposition~\ref{p:odd} and Lemma~\ref{l:autovalori} it follows that $\omega$ solves then the ODE
\[
\ddot \omega (t) - \dot \omega (t) - c\, \omega (t) = 0\, ,
\]
where $c$ is a constant which depends on $k$, but it satisfies the bound $c \geq \bar c>0$ for some positive $\bar c$ independent of $k$. The polynomial $x^2 -x - c$ has then a positive and a negative solution $\alpha^+$ and $-\alpha^-$ (also depending on $k$) with $\alpha^\pm \geq \alpha_0 >0$. The function $\omega (t)$ is then given by $D e^{\alpha^+ t} + C e^{-\alpha^- t}$. A simple computations shows that
\[
\frac{d^2}{dt^2} (\omega^2 (t)) \geq \hat{c}\, 
\omega^2 (t)\, ,
\]
where the positive constant $\hat{c}$ can be chosen to depend on $\alpha_0$ and in particular independent of $k$. Summing the square of all the coefficients involved in the computation of $\mathcal{E}$ we find a non negative function $h(t)$ with the property that $\ddot h (t) \geq \hat{c}\, h(t)$ and $\mathcal{E} (v, \theta, s, \sigma) = \int_s^\sigma h(t)\, dt$. 
The rest of the argument follows the lines of that employed for proving the analogous property for $\mathcal{G}_e$ in Proposition~\ref{p:tre-anelli_even}, to which we refer.

Having shown (a) we now turn to (b).  We claim that (b) holds for $c_0$ sufficiently small.
Observe that if $\mathcal{E} (v, \theta, 1,2)\geq c_0\mathcal{F} (v, \theta,1,2)$, then (b) is simply implied by (a). 
Thus we may assume $\mathcal{G}_o (v, \theta,1,2) = c_0\mathcal{F} (v, \theta,1,2)$.
We argue by contradiction: for $c_0 = \sfrac{1}{j}$ choose $(v_j, \theta_j)$ such that 
\[
 \mathcal{G}_o (v_j, \theta_j, 1,2) \geq \max \{ (1-\eta) \mathcal{G}_o (v_j, \theta_j, 0,1),(1+\eta)^{-1} \mathcal{G}_o (v_j, \theta_, 2,3)\}\,.
\]
Using the linearity we can normalize it so that $\mathcal{F} (v_j, \theta_j,1,2) =1$. Observe that we have the inequalities
\begin{align}
1 = \mathcal{F} (v_j, \theta_j,1,2) &\geq \max \{ (1-\eta) \mathcal{F} (v_j, \theta_j, 0,1), (1+\eta)^{-1} \mathcal{F} (v_j, \theta_j, 2,3)\}\label{e:preponderante 0}\\
1 = \mathcal{F} (v_j, \theta_j,1,2) &\geq j \max \{\mathcal{E} (v_j, \theta_j, 1,2),  (1-\eta) \mathcal{E} (v_j, \theta_j, 0,1), 
(1+\eta)^{-1} \mathcal{E} (v_j, \theta_j, 2,3)\}\, .\label{e:preponderante}
\end{align}
From Proposition~\ref{p:odd} we gain a uniform bound on $\|v_j\|_{H^2 ([0,2\pi)\times (0,3))}
$ and $\|\dot\theta_j\|_{H^1 ((0,3))}$ and consequently (since $\theta_j (0) =0$) on $\|\theta_j\|_{H^2 ((0,3))}$. We then extract a sequence converging weakly to $(v, \theta)\in H^2$ which satisfies \eqref{e:lineare} and \eqref{e:condizione_extra}. Consider the functions $v$ and $\zeta$, which are the limit of the corresponding maps constructed from $v_j$. From \eqref{e:preponderante} and \eqref{e:def_E} we 
conclude that $\zeta (\phi, t) = a_0 (t) \zeta_0 (\phi) + a_1 (t) \zeta_1 (\phi)$. Unraveling the definition of $\zeta$ we infer 
\[
v (\phi, t) = a_0 (t) (\phi -\pi) \sin \frac{\phi}{2} + \bar{a}_1 (t) \cos \frac{\phi}{2}\, ,
\]
where $\bar{a}_1 (t) = a_1 (t) + \frac{\theta (t)}{\sqrt{2\pi}}$. 
However the boundary conditions $v (0, t) = v (2\pi, t) =0$ imply $\bar{a}_1 \equiv 0$. We are thus left with the formula $v (\phi, t) = a_0 (t) (\phi -\pi) \sin \frac{\phi}{2}$. Inserting in \eqref{e:lineare} we get:
\begin{equation}\label{e:a0 theta}
\left\{
\begin{array}{l}
\ddot a_0(t) - \dot a_0 (t) =0\\
\dot\theta (t) - \ddot\theta (t) = -\sqrt{2\pi} a_0 (t)\\
\theta (0) =0\, .
\end{array}\right.
\end{equation}
From the first equation we find $a_0 (t) = c_1 + c_2 e^t$, while from the second we find $\theta (t) = d_1 - \sqrt{2\pi} c_1 t + d_2 e^t -c_2\sqrt{2\pi} t e^t$, 
i.e. $\theta (t) = -\sqrt{2\pi} t a_0 (t) + d_1+d_2 e^t$. Using $\theta (0) =0$ we thus get $\theta (t) = - \sqrt{2\pi} t a_0 (t) + d (e^t-1)$. We next use \eqref{e:condizione_extra} to derive that $a_0$ is actually identically null. 
Indeed, the latter reads as
\[
a_0 (t) \underbrace{\int_0^{2\pi} \left(\sin \frac{\phi}{2} + \frac{\phi-\pi}{2} 
\cos \frac{\phi}{2} \right) \sin \frac{\phi}{2}\, d\phi}_{=:I} = 0\, .
\]
We compute the integral $I$ as 
\begin{align*}
I &= \int_0^{2\pi} \left(\sin^2 \frac{\phi}{2} + \frac{\phi-\pi}{2} \cos \frac{\phi}{2} \sin \frac{\phi}{2} \right)\, d\phi\\
&= \int_0^{2\pi} \left(\frac{1}{2} (1-\cos \phi) + \frac{\phi - \pi}{4} \sin \phi\right)\, d\phi\\
&= \pi - \left. \frac{\phi-\pi}{4} \cos \phi \right|^{2\pi}_0 + \frac{1}{4} \int_0^{2\pi} \cos \phi\, d\phi= \frac{\pi}{2}\, . 
\end{align*}
So we actually infer $a_0 (t)=0$, which in turn implies $\theta (t) = d (e^t-1)$. 

Using the convergences established for $(v_j,\theta_j)$ and passing to the limit into \eqref{e:preponderante 0} we find
\[
1 = \mathcal{F} (v, \theta,1,2) \geq \max \{ (1-\eta) \mathcal{F} (v, \theta, 0,1), (1+\eta)^{-1} \mathcal{F} (v, \theta, 2,3)\}\, ,
\]
to get by an explicit computation
\[
1 = 2d^2 \int_1^2 e^{2t}\, dt \geq 2d^2 \max \left\{ (1-\eta) \int_0^1 e^{2t}\, dt, (1+\eta)^{-1} \int_2^3 e^{2t}\, dt\right\}\, .
\]
Thus, $d\neq 0$ and the latter inequality is equivalent to
\[
e^4-e^2 \geq \max \{(1-\eta) (e^2-1), (1+\eta)^{-1} (e^6-e^4)\}\, ,
\]
which in turn is equivalent to
\[
e^2 \geq \max \{(1-\eta), (1+\eta)^{-1} e^4\}\, .
\]
Since $0<\eta<1$, the latter would imply $e^2 \geq \frac{e^4}{2}$, which is clearly a contradiction.

The growth conditions in \eqref{e:Go growth} easily follow from Proposition~\ref{p:odd} by taking into account that $(v,\theta)$ solves \eqref{e:lineare}.
\end{proof}

\section{Second linearization and proof of Theorem~\ref{t:final-cracktip}}\label{s:second linearization}

The three annuli property of the previous section allows us to improve upon Proposition~\ref{p:linearizzazione}. 

\begin{proposition}\label{p:linearizzazione-2}
Let $v_j$ and $\theta_j$ be as in Proposition~\ref{p:linearizzazione}. 
Then, there is a pair $(v, \theta)\in C^2_{{\rm loc}} ([0,2\pi]\times [0,\infty))$ and a subsequence, not relabeled, such that $(v_j, \theta_j)$ converges in $C^2 ([0,2\pi]\times [0, T))$ to $(v, \theta)$ for every $T>0$. Moreover, $(v, \theta)$ solves \eqref{e:lineare}, satisfies \eqref{e:condizione_extra} and there are positive constants $\varpi$ and $C$ such that:
\[
\|v^e-\isq\|_{C^2 ( (0,2\pi)\times [k,k+1])} +
\|v^o\|_{C^2 ((0,2\pi)\times[k,k+1])} + 
\|\dot\theta\|_{C^1 ([k,k+1])} \leq C e^{-\varpi k} 
\] 
for all $k\in \mathbb N\setminus \{0\}$ if $\lambda =0$, while 
\[
\|v^e-\isq\|_{H^2 ((0,2\pi)\times(k,k+1))} +
\|v^o\|_{H^2 ((0,2\pi)\times(k,k+1))} + 
\|\dot\theta\|_{H^1 ((k,k+1))} 
\leq C e^{-\varpi k}
\] 
for all $k\in \mathbb N\setminus \{0\}$ if $\lambda>0$.
\end{proposition}
\begin{proof}
We prove the estimates claimed above  separately for the odd and even parts by showing in both cases a nonlinear three annuli property. 

We start with the case $\lambda>0$ by observing that by Proposition~\ref{p:linearizzazione} and by Proposition~\ref{p:tre-anelli_odd} we can prove: given $\beta>\ln 2$ (to be chosen suitably in what follows), if $\varepsilon_0$ in Theorem~\ref{t:final-cracktip} is sufficiently small and 
$(u,K)$ satisfies the assumptions of Theorem~\ref{t:final-cracktip}, for every 
$k\in \mathbb N$ we have
that if 
\begin{align}\label{e:Gfo}
&\max\{\mathcal{G}_o (f^o(\cdot,\cdot+k), \vartheta(\cdot+k),1,2), \lambda e^{-\beta(1+k)}\}\notag\\
&\qquad\geq (1-{\textstyle{\frac\eta 2}}) \max\{\mathcal{G}_o (f^o(\cdot,\cdot+k), \vartheta(\cdot+k),0,1), \lambda e^{-\beta k}\}\notag\\
\Longrightarrow&\max\{\mathcal{G}_o (f^o(\cdot,\cdot+k), \vartheta(\cdot+k), 2,3), 
\lambda e^{-\beta(2+k)}\}\notag\\
&\qquad\geq (1+{\textstyle{\frac\eta 2}}) 
\max\{\mathcal{G}_o (f^o(\cdot,\cdot+k),\vartheta(\cdot+k), 1,2), 
\lambda e^{-\beta(1+k)}\}\,.
\end{align}
where $\mathcal G_o$ is the functional defined in \eqref{e:def_G}, $\vartheta, f$ are given by \eqref{e:vartheta} and \eqref{e:f}, respectively. 

Indeed, assume the claim is false, no matter how small $\varepsilon_0$ in Theorem~\ref{t:final-cracktip} is chosen, and 
let thus $f^o_j,\, \vartheta_j$ be a sequence which violates it for some $k_j\in\mathbb{N}$ when we choose $\varepsilon_0 = \frac{1}{j}$. By rescaling the time variable $t$ (which just implies a rescaling of the variable $r$ in the 
original problem), we can assume $k_j=j$. Furthermore, by adding a rotation we can assume that $\vartheta_j (0) =0$, so that we can apply Proposition~\ref{p:linearizzazione}. 
Assume by contradiction that
\begin{align}\label{e:contra G bis}
\max&\{\mathcal{G}_o (f_{j}^o(\cdot,\cdot+j), \vartheta_{j}(\cdot+j), 1,2), \lambda e^{-\beta(1+j)}\}\notag \\
&\geq \max \left\{\left(1-{\textstyle{\frac\eta 2}}\right) 
\max\{\mathcal{G}_o (f_j^o(\cdot,\cdot+j), \vartheta_j(\cdot+j), 0,1), \lambda e^{-\beta j}\},\right.\notag\\ 
&\qquad\qquad\left.\left(1+{\textstyle{\frac\eta 2}}\right)^{-1} \max\{\mathcal{G}_o (f_j^o(\cdot,\cdot+j), \vartheta_{j}(\cdot+j), 2,3), \lambda e^{-\beta(2+j)}\}\right\}\, .
\end{align}
Note that if $\lambda >0$ and $\beta>\ln 2$, then necessarily 
\begin{equation}\label{e:G_o vs exp}
\mathcal{G}_o (f_j^o(\cdot,\cdot+j), \vartheta_j(\cdot+j), (1,2))
\geq \lambda e^{-\beta (1+j)}
\end{equation}
for all $j\geq 0$, since otherwise from \eqref{e:contra G bis} 
we would conclude for some $i\geq 0$ that  
\[
e^{-\beta(1+i)}\geq \max\left\{\left(1-{\textstyle{\frac\eta 2}}\right) e^{-\beta i}, 
\left(1+{\textstyle{\frac\eta 2}}\right)^{-1} e^{-\beta(2+i)}\right\}\,,
\]
which is equivalent to  
\[
\left(1+{\textstyle{\frac\eta 2}}\right)^{-1}\leq e^{\beta}\leq
\left(1-{\textstyle{\frac\eta 2}}\right)^{-1}\,, 
\]
in turn implying $e^\beta\in(\sfrac23,2)$, as $\eta\in(0,1)$.
Clearly, this is a contradiction in view of the choice $\beta>\ln 2$.

Therefore, from this observation, from the definition of $v_j$ in \eqref{e:v_j},
from \eqref{e:contra G bis}, and since the functional $\mathcal{G}_o$ is quadratic, 
we immediately obtain that
\begin{equation}\label{e:contra G ter}
\mathcal{G}_o (v_j^o, \theta_j, 1,2)\geq \max \left\{\left(1-{\textstyle{\frac\eta 2}}\right) 
\mathcal{G}_o (v_j^o, \theta_j, 0,1),\left(1+{\textstyle{\frac\eta 2}}\right)^{-1} 
\mathcal{G}_o (v_j^o, \theta_j, 2,3)\right\}\, .
\end{equation}
Upon choosing $\beta>\sfrac32-\hat\varepsilon(>\ln2)$, where $\hat\varepsilon$ has been fixed in the definition of $\delta_j$ (cf. \eqref{e:delta}), from \eqref{e:G_o vs exp}
we conclude that $\liminf_j \|v_j^o\|_{H^2 ((0,2\pi)\times [1,2])} > 0$,
because of \eqref{e:Go growth}.

Thus, we can apply Proposition~\ref{p:linearizzazione}
(d) (recall that $\theta_j (0)=0$) to extract a subsequence converging to some $(v, \theta)$ strongly in 
$H^2((0,2\pi)\times(\sigma,3-\sigma))\times H^2((\sigma,3-\sigma))$
for all $\sigma\in(0,\sfrac32)$. 
Therefore, we have that 
\[
\mathcal{G}_o (v^o,\theta, 1,2) = \lim_j \mathcal{G}_o (v_j^o, \theta_j, 1,2)\, ,
\]
and in particular we conclude that the pair $(v^o, \theta)$ is nontrivial.
On the other hand, the functional $\mathcal{G}_o$ is lower semicontinuous with respect to the mentioned convergences, and we thus infer from \eqref{e:contra G ter} 
\[
\mathcal{G}_o (v^o, \theta, 1,2) \geq \max \left\{\left(1-{\textstyle{\frac\eta 2}}\right) \mathcal{G}_o (v^o, \theta, 0,1), 
\left(1+{\textstyle{\frac\eta 2}}\right)^{-1} \mathcal{G}_o(v^o, \theta, 2,3)\right\}\,,
\]
contradicting Proposition~\ref{p:tre-anelli_odd} 
(b) being $(v^o,\theta)$ nontrivial. 

Therefore, having completed the proof of \eqref{e:Gfo}, if for some $k_0\in\N$ we were to have
\begin{align*}
\max&\{\mathcal{G}_o (f^o(\cdot,\cdot+k_0), \vartheta(\cdot+k_0), 1,2), 
\lambda e^{-\beta(1+k_0)}\}\\ &\geq \left(1-{\textstyle\frac\eta2}\right) 
\max\{\mathcal{G}_o (f^o(\cdot,\cdot+k_0), \vartheta(\cdot+k_0), 0,1),
\lambda e^{-\beta k_0}\}\, ,
\end{align*}
then from \eqref{e:Gfo} itself, and the choice $\beta>\ln 2$, we would infer that for all $j \geq k_0+1$
\[
\mathcal{G}_o (f^o(\cdot,\cdot+j), \vartheta, 0,1) \geq \left(1+{\textstyle\frac\eta2}\right)^{j-(k_0+1)}  \mathcal{G}_o (f^o(\cdot,\cdot+k_0), \vartheta(\cdot+k_0), 0,1) \, .
\]
However the latter contradicts the fact that $f^o (\cdot,\cdot+j)$ and $\dot\vartheta (\cdot+j)$ converge smoothly to $0$ for $j\to \infty$.

We thus conclude that for every $k\in\N$
\begin{align*}
\max&\{\mathcal{G}_o (f^o(\cdot,\cdot+k), \vartheta(\cdot+k), 1,2), 
\lambda e^{-\beta(k+1)}\} \\
&\leq \left(1-{\textstyle\frac\eta2}\right) 
\max\{\mathcal{G}_o (f^o(\cdot,\cdot+k), \vartheta(\cdot+k), 0,1),
\lambda e^{-\beta k}\}
\, ,
\end{align*}
in turn implying, by iteration and by \eqref{e:Go growth}, the existence of constants $C>0$ and $\varpi\in(0,\ln 2)$ such that
\[
\|f^o\|_{H^2 ((0,2\pi)\times (k,k+1))}^2 + \|\dot\vartheta\|_{H^1 ((k,k+1))}^2 \leq C e^{-\varpi k} \left(\|f^o\|_{H^2 ((0,2\pi)\times (0,1))}^2 + \|\dot\vartheta\|_{H^1 ((0,1))}^2+\lambda\right)\, .
\]
In turn, if $(v_j, \theta_j)$ are as in the statement of the proposition, we infer
\begin{align}
\|v_j^o\|_{H^2 ((0,2\pi)\times (k,k+1))}^2 &+ \|\dot\theta_j\|_{H^1 ((k,k+1))}^2\notag\\ 
\leq & C e^{-\varpi k} \left(\|v_j^o\|_{H^2 ((0,2\pi)\times (0,1))}^2 + \|\dot\theta_j\|_{H^1 ({(0,1)})}^2+\lambda\right) 
= & C e^{-\varpi k}\, .\label{e:decay v_j^o}
\end{align}

The estimate for the even part follows analogously. Indeed, 
one first shows the nonlinear three annuli property for
\[
g^e(\phi,t+k):=f^e(\phi,t+k)-\langle f^e(\cdot,t+k),\textstyle{\frac{1}{\sqrt{\pi}}}\sin({\textstyle{\frac{\phi}{2}}})\rangle_{L^2}
\textstyle{\frac{1}{\sqrt{\pi}}}\sin({\textstyle{\frac{\phi}{2}}})\,.
\]
Note that $g^e(\cdot,\cdot+j)$ is still even and rescaled by $\delta_j$ is converging to an even solution to \eqref{e:lineare} satisfying \eqref{e:mean-zero}. Hence, by 
using Proposition~\ref{p:even} and by arguing as above one deduces that if 
\begin{align*}
\max&\{\mathcal{G}_e (g^e(\cdot,\cdot+k),1,2), \lambda e^{-\beta(1+k)}\}\geq (1-{\textstyle{\frac\eta 2}}) \max\{\mathcal{G}_e (g^e(\cdot,\cdot+k), \lambda e^{-\beta k}\}\notag\\
\end{align*}
then
\begin{align*}
\max&\{\mathcal{G}_e (g^e(\cdot,\cdot+k),  2,3), 
\lambda e^{-\beta(2+k)}\}\geq 
(1+{\textstyle{\frac\eta 2}}) 
\max\{\mathcal{G}_e (g^e(\cdot,\cdot+k), 1,2), 
\lambda e^{-\beta(1+k)}\}\,.
\end{align*}
where $\mathcal G_e$ is the functional defined in \eqref{e:def_G even}. 

By assumption $f^e (\cdot,\cdot+j)$ converges smoothly to $\isq$ for $j\to \infty$, so that for every $k\in\N$
\[
\max\{\mathcal{G}_e (g^e(\cdot,\cdot+k), 1,2), 
\lambda e^{-\beta(k+1)}\} 
\leq \left(1-{\textstyle\frac\eta2}\right) 
\max\{\mathcal{G}_e (g^e(\cdot,\cdot+k), \vartheta(\cdot+k), 0,1),
\lambda e^{-\beta k}\}
\, .
\]
In turn implying, by iteration and by \eqref{e:Ge coerc}, the existence of constants $C>0$ and $\varpi\in(0,\ln 2)$ such that
\[
\|g^e\|_{H^2 ((0,2\pi)\times (k,k+1))}^2\leq C e^{-\varpi k} \left(\|g^e\|_{H^2 ((0,2\pi)\times (0,1))}^2 +\lambda\right)\, .
\]
Therefore, we conclude that 
\begin{align}\label{e:decay v_j^e}
&\|v_j^e-\langle v_j^e,\textstyle{\frac{1}{\sqrt{\pi}}}\sin({\textstyle{\frac{\phi}{2}}})\rangle_{L^2}\textstyle{\frac{1}{\sqrt{\pi}}}\sin({\textstyle{\frac{\phi}{2}}})\|_{H^2 ((0,2\pi)\times (k,k+1))}^2\notag\\
&\leq C e^{-\varpi k} \left(\|v_j^e-\langle v_j^e,\textstyle{\frac{1}{\sqrt{\pi}}}\sin({\textstyle{\frac{\phi}{2}}})\rangle_{L^2}\textstyle{\frac{1}{\sqrt{\pi}}}\sin({\textstyle{\frac{\phi}{2}}})\|_{H^2 ((0,2\pi)\times (0,1))}^2 +\lambda\right)\, .
\end{align}

Having fixed $k\in\mathbb N$, in view of \eqref{e:decay v_j^o} and 
\eqref{e:decay v_j^e}, the conclusion of the proposition is then a simple application of Proposition~\ref{p:linearizzazione} with $[0,1]$ replaced by $[k,k+1]$, that is 
equivalently by applying it to $f(\cdot,\cdot+k+j)$ and $\vartheta(\cdot+k+j)$ on $[0,1]$, together with a  diagonal argument over $k$ and $j$.

The proof of the case $\lambda=0$ proceeds similarly: the inequality analogous to \eqref{e:Gfo} in this setting is obtained by revisiting the argument outlined above, upon normalizing $\mathcal{G}_o (v_j^o, \theta_j, 1,2)=1$ for every $j$. Then, the conclusion follows by taking advantage of \eqref{e:decay v_j^o}, of elliptic regularity and of the stronger 
convergences described in Proposition~\ref{p:linearizzazione} (c).
\end{proof}

Using the second linearization procedure 
in Proposition~\ref{p:linearizzazione-2} and again the spectral analysis for solutions of \eqref{e:lineare} we will then conclude the decay for the curvature at the tip when $\lambda = 0$.

\begin{corollary}\label{c:decay_curvature}
There is a constant $\delta_0$ with the following property. 
Assume $(u,K)$ is as in Theorem~\ref{t:final-cracktip} and $\vartheta$ 
as in \eqref{e:vartheta}. 
Then there are constants
$C,\,\delta_0>0$ and
$\delta_1\in(0,1)$ such that, if $\lambda=0$, 
\begin{align}
\|f^o(\cdot,t)\|_{C^2([0,2\pi])}+
|\dot \vartheta (t)| + |\ddot \vartheta (t)| &\leq C e^{-(1+\delta_0) t}\, , 
\label{e:decay f^o theta lambda=0}\\
\|f^e(\cdot,t)-\isq\|_{C^2([0,2\pi])} &\leq C e^{-(1-\delta_1) t} \,, \label{e:decay f^e lambda=0}
\end{align}
while, if $\lambda>0$ for every $\varepsilon\in(0,1)$ there is a constant $C_\varepsilon>0$ such that
\begin{align}
\|f^o(\cdot,t)\|_{C^{1,1-\varepsilon}([0,2\pi])}+
|\dot \vartheta (t)| 
& \leq C_\varepsilon e^{-\delta_0 t} \, ,\label{e:decay f^o theta lambda>0}\\
\|f^e(\cdot,t)-\isq\|_{C^{1,1-\varepsilon}([0,2\pi])} &\leq C_\varepsilon e^{-\delta_0 t} \, .\label{e:decay f^e lambda>0}
\end{align}
In particular, in case $\lambda =0$ we also have
\begin{equation}\label{e:decay kappa t}
\left|\displaystyle{ \frac{\varthetadd (t)- \varthetad(t)-\varthetad^3(t)}{(1+\varthetad^2(t))^{\sfrac52}}}\right|
\leq C e^{-(1+\delta_0) t}\, .
\end{equation}
\end{corollary}
\begin{proof}
First of all consider any limit $(v, \theta)$ as in Proposition~\ref{p:linearizzazione-2}. We discuss separately the behavior of the odd and even parts of $v$. Let $\zeta$ be as in \eqref{e:zeta}. 
Recalling that $v$ satisfies \eqref{e:condizione_extra},
Proposition~\ref{p:odd} yields the expansion
\begin{align*}
\zeta (\phi, t)  =
 \bar{a}_1(t) \cos \frac{\phi}{2} + 
\sum_{k=2}^\infty (\bar{a}_{k,-} e^{\mu_{k,-} t}+\bar{a}_{k,+} e^{\mu_{k,+} t})\zeta_k (\phi)\, ,
\end{align*}
where $\zeta_k(\phi)=c_k\sin(\nu_k(\phi-\pi))$, $c_k$ is such that $\langle \zeta_k,\zeta_k\rangle=1$
for every $k\geq 2$ (cf. \eqref{e:zeta_k}), where the $\bar{a}_{k,\pm}$'s 
are constants, and $\mu_{k,\pm}$ is the 
positive/negative zero of the quadratic polynomial $x^2-x-(\nu_k^2-\frac{1}{4})$, i.e. (cf. \eqref{e:a_k(t) eq}) 
\[
\mu_{k,\pm} = \frac12\pm \nu_k\,.
\]
Recalling item (c) in Lemma~\ref{l:autovalori}, we have $\nu_k \geq \nu_2 > \frac{3}{2}$ when $k\geq 2$ and thus we conclude that $\mu_{k,+} \geq \mu_{2,+} > 2$ and $\mu_{k,-}\leq\mu_{2,-}<-1$ for all $k\geq 2$.
Therefore, by the decay properties of $v^o$ and $\dot \theta$ in Proposition~\ref{p:linearizzazione-2},  
we easily infer that $\bar{a}_{k,+}=0$ for every $k\geq 2$, so that
\begin{equation*}
v^o (\phi, t) =
\zeta (\phi, t)  + \frac{\theta(t)}{\sqrt{2\pi}}\cos\frac\phi2
= a_1(t) \cos \frac{\phi}{2} + \sum_{k=2}^\infty \bar{a}_k e^{-\mu_k t} \zeta_k (\phi) 
\,, 
\end{equation*}
where we have set $\mu_k=|\mu_k^-|>1$,
$\bar{a}_k:=\bar{a}_{k,-}$, and $a_1(t):=\bar{a}_1(t)+\frac{\theta(t)}{\sqrt{2\pi}}$.
In particular, note that $\mu_k^2+\mu_k=\nu_k^2-\frac{1}{4}$.
From the boundary conditions $v^o(0,t)=v^o(2\pi,t)=0$ we conclude that 
\begin{equation}\label{e:a_1(t)}
a_1(t) = -\sum_{k=2}^\infty \bar{a}_k \zeta_k (0) e^{-\mu_k t}\,.
\end{equation}
An elementary computation together with the Ventsel boundary condition satisfied by the $\zeta_k$'s (cf. \eqref{e:autovalore}) and the definition of $\nu_k$ 
(cf. Lemma~\ref{l:autovalori} (a)) imply that 
\begin{equation}\label{e:v^o(phi,t)}
v^o_\phi(0,t)=\frac\pi2\sum_{k=2}^\infty
\bar{a}_k\textstyle{(\nu_k^2-\frac14)}\zeta_k (0)
e^{-\mu_k t}\,.
\end{equation}
Thus, using the ODE in \eqref{e:theta eq} 
and $\mu_k^2+\mu_k=\nu_k^2-\frac{1}{4}$
we get that
\begin{align*}
    \theta(t)=A + B e^t -\sqrt{2\pi}\sum_{k=2}^\infty
    \bar{a}_k \zeta_k (0) e^{-\mu_k t}\,.
\end{align*}
As $\lim_{t\to+\infty}\theta(t)\in\mathbb R$ we infer that $B=0$, 
in turn implying $A=\lim_{t\to+\infty}\theta(t)$.
Finally, as $\theta(0)=0$ we deduce that 
\begin{align}\label{e:theta(t)}
    \theta(t)=\sqrt{2\pi} \sum_{k=2}^\infty
    \bar{a}_k \zeta_k (0) (1-e^{-\mu_k t})\,.
\end{align}
Note that the latter together with \eqref{e:a_1(t)} easily imply that
$\bar{a}_1(t)=-\frac{A}{\sqrt{2\pi}}$.

We argue similarly for the even part.
Namely, Proposition~\ref{p:even} yields the expansion
\begin{equation*}
v^e (\phi, t)  =
\sum_{k=0}^\infty
(C_k e^{(k+1)t} + D_k e^{-kt})
{\textstyle{\frac{1}{\sqrt{\pi}}}}
\sin({\textstyle{(k+\frac12)\phi}})\, ,
\end{equation*}
and by the decay properties of $v^e$ in Proposition~\ref{p:linearizzazione-2},
we infer that $C_k=0$ for all $k\geq 0$, and moreover that $D_0=\sqrt{2}$, namely
\begin{equation}\label{e:v^e(phi,t)}
v^e (\phi, t)  = \isq(\phi)
+\frac{1}{\sqrt{\pi}}\sum_{k=1}^\infty D_k e^{-kt}
\sin({\textstyle{(k+\frac12)\phi}})\,.
\end{equation}

\medskip

{\bf Case $\lambda =0$.}
From \eqref{e:a_1(t)}, \eqref{e:v^o(phi,t)}, \eqref{e:theta(t)} and \eqref{e:v^e(phi,t)}
it is then easy to check that for every $T>0$ we have the estimate
\begin{align*}
\|v^o\|_{C^2 ([0,2\pi]\times [T, 2T])} + \|\dot\theta\|_{C^1 ([T, 2T])} 
\leq C e^{-\mu_2 T} \left(
\|v^o\|_{H^2 ([0,2\pi]\times[0,1])} + \|\dot\theta\|_{H^1 ([0,1])}\right)\, ,
\end{align*}
and
\begin{align*}
\|v^e-&\isq\|_{C^2 ([0,2\pi]\times [T, 2T])}
\leq C e^{-T}
\|v^e-\isq\|_{H^2 ([0,2\pi]\times [0,1])}\,,
\end{align*}
where $C$ is a constant independent of $T$. Fix now $T$, whose choice will be specified a few paragraphs below. Using the conclusions of Proposition~\ref{p:linearizzazione-2} if $\lambda=0$ we then conclude that, if $u$ is as in Theorem~\ref{t:final-cracktip} and $\vartheta$ and $f$ as in Lemma~\ref{l:nonlinear} and $\varepsilon_0$ sufficiently small (depending on $T$), then
\begin{align*}
\|f^o\|_{C^2 ([0,2\pi]\times [T, 2T])} + \|\dot\vartheta\|_{C^1 ([T, 2T])} 
&\leq 2 C e^{-\mu_2 T}\left(
\|f^o\|_{H^2 ([0,2\pi]\times[0,1])} + \|\dot\vartheta\|_{H^1 ([0,1])}\right)\\
&\leq \bar{C} e^{-\mu_2 T}\left(
\|f^o\|_{C^2 ([0,2\pi]\times[0,T])} + \|\dot\vartheta\|_{C^1 ([0,T])}\right)\,,
\end{align*}
and
\begin{align*}
\|f^e-&
\isq\|_{C^2 ([0,2\pi]\times [T, 2T])} \leq 2 C e^{-T}
\|f^e-
\isq\|_{H^2 ([0,2\pi]\times [0,1])} \leq
\bar{C} e^{-T}\|f^e-
\isq\|_{C^2 ([0,2\pi]\times [0,T])}\, , 
\end{align*}
where the constant $\bar C$ is independent of $T$.
By a simple rescaling argument, this actually implies that for all $k\in \mathbb N$ 
\begin{align*}
\|f^o\|_{C^2 ([0,2\pi]\times [(k+1) T, (k+2)T]} &+ \|\dot\vartheta\|_{C^1 ([(k+1)T, (k+2) T])} \\
\leq & \bar{C} e^{-\mu_2 T}\left(
\|f^o\|_{C^2 ([0,2\pi]\times[kT,(k+1)T])} + \|\dot\vartheta\|_{C^1 ([kT,(k+1)T])}\right)\,.
\end{align*}
and
\begin{align*}
\|f^e&-
\isq\|_{C^2 ([0,2\pi]\times [(k+1) T, (k+2)T])} \leq
\bar{C} e^{-T}\|f^e-
\isq\|_{C^2 ([0,2\pi]\times [kT, (k+1)T])}\,.    
\end{align*}
We stress that the constant $\bar C$ is independent of $T$. On the other hand, recalling that $\mu_2>1$ 
(because $\mu_2=\nu_2-\frac12>1$, cf. (c) Lemma~\ref{l:autovalori}), while
given $\delta_0\in(0,\mu_2-1)$ and $\delta_1\in
(0,1)$,
we can choose $T$ itself large enough so that
\[
\bar C e^{-\mu_2 T} \leq e^{-(1+\delta_0) T},\qquad
\bar C e^{- T}\leq e^{-(1-\delta_1) T}\,.
\] 
We then can iterate the latter inequalities to infer
\begin{align*}
 \|f^o\|_{C^2 ([0,2\pi]\times [(k+1) T, (k+2)T])} +& \|\dot\vartheta\|_{C^1 ([(k+1)T, (k+2) T])}\\  
\leq &e^{-(1+\delta_0) kT}\left(
\|f^o\|_{C^2 ([0,2\pi]\times[0,T])} + \|\dot\vartheta\|_{C^1 ([0,T])}\right)\,, \end{align*}
and 
\begin{align*}
\|f^e&-
 \isq\|_{C^2 ([0,2\pi]\times [(k+1) T, (k+2)T])}\leq
e^{-(1-\delta_1) kT}\|f^e-
\isq\|_{C^2 ([0,2\pi]\times [0, T])}\, .
\end{align*}
This easily gives the conclusions \eqref{e:decay f^o theta lambda=0},
\eqref{e:decay f^e lambda=0} and, in particular, \eqref{e:decay kappa t}.
\medskip 

\noindent{\bf Case $\lambda>0$.}

The proof of the estimates in \eqref{e:decay f^o theta lambda>0} and \eqref{e:decay f^e lambda>0} follows as in the previous case by using the conclusions of  Proposition~\ref{p:linearizzazione-2} for $\lambda>0$ rather than those for $\lambda=0$ 
there. 
\end{proof}

The latter Corollary~\ref{c:decay_curvature} implies easily Theorem~\ref{t:final-cracktip}.

\subsection{Proof of Theorem~\ref{t:final-cracktip}}
Corollary~\ref{c:decay_curvature} gives a $C^{2, \delta_0}$ estimate for 
the parametrization of $K\cap B_2$ if $\lambda=0$ and a $C^{1, \delta_0}$ 
estimate if $\lambda>0$. 
More precisely, the unit tangent $\tau (r)$ to $K\cap B_2$ at the point $\gamma(r)=r (\cos \alpha (r), \sin \alpha (r))$ (cf. \eqref{e:K param})
is given by the expression
\[
\tau (r) = \frac{1}{\sqrt{1+ r^2\alpha' (r)^2}} \big( (\cos \alpha (r), \sin \alpha (r)) +
r\, \alpha' (r)\,(-\sin \alpha (r), \cos \alpha (r)) \big)\,.
\]
using the relation $r= e^{-t}$ and $\alpha(r)=\vartheta(|\ln r|)$ (cf. \eqref{e:vartheta}) and
\eqref{e:decay f^o theta lambda=0} if $\lambda=0$, respectively \eqref{e:decay f^o theta lambda>0} if $\lambda>0$, we easily check that 
$|\tau'' (r)|\leq C r^{\delta_0-1}$, respectively $|\tau' (r)|\leq C r^{\delta_0-1}$. Integrating the latter inequalities between 
$r_1$ and $r_2$ we reach the estimates
\begin{equation*}
|\tau' (r_2) - \tau' (r_1)| \leq C(r_2-r_1)^{\delta_0} \qquad \forall\, 0<r_1<r_2< \sfrac{1}{2}\, ,
\end{equation*}
for $\lambda=0$, and for $\lambda>0$
\begin{equation*}
|\tau (r_2) - \tau (r_1)| \leq C(r_2-r_1)^{\delta_0} \qquad \forall\, 0<r_1<r_2< \sfrac{1}{2}\, .
\end{equation*}
In turn, the latters imply respectively $C^{1, \delta_0}$ and 
$C^{0, \delta_0}$ estimates on the tangent $\tau (r)$ to $K$ at the point $\gamma(r)$,
and moreover that $\tau(r)$ has a limit $\tau_0\in\mathbb S^1$ as $r\to 0^+$ in both cases.

In addition, we get a decay estimate for the curvature when $\lambda =0$. Indeed, using \eqref{e:formula_curvatura}, namely
\[
\kappa(r)=r^{-1}\,
\frac{\varthetad(|\ln r|)+\varthetad^3(|\ln r|)-\varthetadd(|\ln r|)}{(1+\varthetad^2(|\ln r|))^{\sfrac32}}\, , 
\]
from estimate \eqref{e:decay kappa t} in Corollary~\ref{c:decay_curvature} 
we conclude that
\[
|\kappa(r)|\leq Cr^{\delta_0}\, 
\]
when $\lambda =0$. Note that the denimonator of the quantity estimated in \eqref{e:decay kappa t} differs from the denominator appearing in the formula of the curvature by a multiplicative factor which is $1+\dot{\theta}^2 (|\ln r|)$: the latter however converges to $1$ as $r\to 0$, and in fact according to our estimates is bounded above by an absolute constant on the interval of interest. 

Moreover, it follows easily that there is an $\eta>0$, depending only upon $C$ and $\varpi$, such 
that $B_\eta \cap K$ is a graph $\{t\,\tau_0  + \psi (t)\, \tau_0^\perp\}\cap B_\eta$
for some function $\psi: [0, \eta] \to \mathbb R$. If $\lambda =0$ the latter is $C^{2, \delta_0}$ smooth,
with $\psi (0)= \psi' (0) = \psi'' (0)= 0$ and $\|\psi\|_{C^{2,\delta_0}}\leq \bar{C}$. When $\lambda >0$, $\psi$ in $C^{1, \delta_0}$ with $\psi (0)= \psi' (0) = 0$ and $\|\psi\|_{C^{1,\delta_0}} \leq \bar{C}$. 

On the other hand, if $\varepsilon_0$ is sufficiently small, since 
$|\alpha'(r)| \leq \varepsilon_0/\eta$ for $r\in (\eta, \sfrac12)$
(recall that $\gamma\in C^{1,1}((0,2))$), we conclude that $K\cap B_{\sfrac 14}$ is a graph in the coordinates induced by the orthonormal base 
$\{\tau_0,\tau_0^\perp\}$.
Finally, since such graph will have
to be sufficiently close to the line $\{(s,0): s\geq 0\}$ by assumption (iv), we conclude that $\tau_0$
must be close to $(1,0)$. 
Therefore, $K \cap B_{\sfrac 14}$ is in fact a graph in the standard coordinates, as claimed in Theorem~\ref{t:final-cracktip}.

\chapter{Some consequences of the epsilon-regularity theory}\label{ch:finale}

\newcommand{\dimh}{\mathrm{dim}_{\mathcal{H}}}
\newcommand{\diam}{{\rm diam}}
\newcommand{\Sigmauno}{K^{(t)}}
\newcommand{\Sigmadue}{K^{(c)}}
\newcommand{\Sigmatre}{\Sigma\setminus (K^{(t)}\cup K^{(c)})}
\def\Per{\mathrm{Per}}
\newcommand\res{\mathop{\hbox{\vrule height 7pt width .3pt depth 0pt
\vrule height .3pt width 5pt depth 0pt}}\nolimits}

\section{Main statements}

In this chapter, we use the $\varepsilon$-regularity theory and some more ideas to prove several structural results about the set $K$ and a more refined analysis of the behavior 
of $K$ around some particular points. 

{More precisely, \index{subset of jump points@subset of jump points}\index{jump points, subset of@jump points, subset of}
\index[simb]{aalK^j@$K^{(j)}$} consider the subset $K^{(j)}\subset K$ of points $p$ such that $K$ is a regular arc in a neighborhood of $p$ and $p$ is not an endpoint of the arc. This is in fact the set of pure jumps, also called in what follows jump points\index{pure jump@pure jump}\index{jump point@jump point}. Through the $\epsilon$-regularity criterion of Theorem~\ref{t:eps_salto_puro} we can characterize its complement $K\setminus K^{(j)}$ in $K$ according to whether the scaled Dirichlet energy or the scaled mean flatness or none of them is infinitesimal. In this respect, we shall show that triple junctions \index{triple junction@triple junction} are the only points where the scaled Dirichlet energy decays while the scaled mean flatness does not, \index{regular loose end@regular loose end} \index{loose end, regular@loose end, regular} and \index{cracktip@cracktip} regular loose ends (or cracktips) are the only points where the scaled mean flatness decays but the scaled Dirichlet energy does not. On the remaining set, which turns out to coincide with $K^{(i)}$, both the scaled mean flatness and the scaled Dirichlet energy are above some positive thresholds {\em at all scales}. 

\medskip

The Mumford-Shah conjecture is thus equivalent to showing that $K^{(i)}$ is indeed the empty set. While this is still open,
it is possible to estimate the size of the latter set in Section~\ref{s:higher integrability} assuming that some $L^p$ norm of $\nabla u$ is finite, 
following the approach introduced by Ambrosio, Fusco, and Hutchinson in \cite{AFH}. Indeed a suitable implementation of the latter idea shows that  the Mumford Shah conjecture is equivalent to a sharp integrability property of $\nabla u$, namely that it belongs to the weak $L^4$ space. 

That $\nabla u$ is higher integrable (namely in $L^{2+\varepsilon}$) was in fact conjectured by De Giorgi for minimizers in all dimensions. We shall give two proofs of this fact. In Section \ref{ss:reverse holder} we show how it follows from a suitable reverse H\"older inequality, which can be inferred from the $\varepsilon$-regularity theory and the compactness result of Theorem~\ref{t:minimizers compactness}. The validity of the latter argument has not been explored in higher dimensions. A second proof, due to De Philippis and Figalli and which can be extended to higher dimensions, establishes first the porosity of $K$, in the spirit of David, again as a consequence of the $\varepsilon$-regularity at jump points,  cf. Section~\ref{ss:porosity}. It then shows the higher integrability of $\nabla u$ leveraging the latter property.

\medskip

The final Section~\ref{s:Bonnet-plus} is dedicated to proving that the Mumford-Shah conjecture holds when the number of connected components of $K$ is limited apriori. This follows a classical argument of Bonnet, who proved in \cite{B96} that each connected components is formed by finitely many arcs whose common endpoints are triple junctions. Note however that the $\varepsilon$-regularity theorem for cracktips was missing at the time and therefore Bonnet could not exclude that such arcs could form slowly rotating spirals at an endpoint which is not in common with another arc.}

\subsection{Structure of \texorpdfstring{$K$}{K}}\label{s:structure}

We consider the set $K^{(j)}$ of pure jump points and as discussed above note that, according to Theorem~\ref{t:eps_salto_puro}, they can be characterized as those points where the scaled Dirichlet energy and the flatness of $K$ both converge to $0$. Its complement 
$\Sigma:=K\setminus K^{(j)}$, usually called in the literature the 
\emph{singular set}\index{singular set@singular set}\index[simb]{aagSigma@$\Sigma$},
can then be decomposed as follows as the disjoint union of $\Sigmauno$, $\Sigmadue$, and $\Sigmatre$, 
\index{subset of triple junctions@subset of triple junctions}\index[simb]{aalK^t@$\Sigmauno$} \index{subset of cracktips@subset of cracktips}\index[simb]{aalK^c@$\Sigmadue$}\index{cracktips, subset of@cracktips, subset of}\index{triple junctions, subset of@triple junctions, subset of}
where 
\begin{eqnarray*}
&&\Sigmauno:=\{x\in\Sigma:\,\liminf_{r\downarrow 0} 
d(x,r)=0\},\\
&&\Sigmadue:=\{x\in\Sigma:\,\liminf_{r\downarrow 0} 
\beta(x,r)=0\},\\
&&\Sigmatre:=\{x\in\Sigma:\,
\limsup_{r\downarrow 0}d(x,r)>0,\, \mbox{and}\, 
\limsup_{r\downarrow 0}\beta(x,r)>0\}.
\end{eqnarray*}
According to the Mumford-Shah conjecture, we should have $\Sigma\setminus (\Sigmauno\cup\Sigmadue)=\emptyset$. Note also that the set $K^\sharp$ 
of high energy points is in fact given by $\Sigma\setminus \Sigmauno$. 

We will use this subdivision to prove Theorem~\ref{t:structure}. A starting point will be the following one.

\begin{theorem}\label{t:structure-2}\label{T:STRUCTURE-2}
Let $(u,K)$ be an absolute, or generalized global minimizer of $E_\param$ in $\Omega$.
Then the following holds:
\begin{itemize}
    \item[(i)] $\Sigmauno$ consists of the triple junctions of $K$ and is discrete.
    \item[(ii)] $\Sigmadue$ consists of the regular loose ends of $K$ and is discrete. 
    \item[(iii)] $\Sigmatre$ consists of those connected components of $K$ which are singletons and of irregular terminal points of connected components with positive length. Every $\Sigmatre$ can also be characterized as an accumulation point of infinitely many connected components of $K$. 
   \end{itemize}
\end{theorem}

In particular, we conclude that $\Sigma\setminus (\Sigmauno\cup\Sigmadue)$ is the set of irregular points $K^{(i)}$ and coincides with the set $K^\sharp\setminus \Sigmadue$, i.e. those high energy points which are not regular loose ends. In view of the above theorem we will refer to $\Sigmauno$ as the set of triple junctions and to $\Sigmadue$ 
as the set of cracktips, or regular loose ends. Observe that the discreteness of the sets immediately implies their countability. Note also that the case of restricted minimizers is not covered in Theorem~\ref{t:structure-2} for the following reason:
\begin{itemize}
    \item If the number of connected components of a minimizer is infinite, then the ``restrictedness'' is in fact an empty condition, and the minimizer is actually an absolute minimizer;
    \item If the number of connected components is finite, it follows from the theory exposed so far that the conclusion of the Mumford-Shah conjecture holds since any blow-up at a point of $K$ is either a global pure jump, or a global triple junction, or a cracktip (this is the content of Bonnet's theorem).  
\end{itemize}
\begin{remark}\label{r:struttura quantitativa}
Finally, we remark that the conclusions of the structure theorem can be slightly strengthened, in the sense that $\Sigmauno$ (resp. $\Sigmadue$) can be also characterized as those points in $K\setminus K^{(j)}$ where the liminf of $d(x,r)$ (resp. of $\beta (x,r)$) is below a certain universal threshold $\varepsilon$, cf. Proposition~\ref{p:structure-3} below.
\end{remark}

\subsection{Higher integrability of the gradient and the size of the singular set}\label{s:higher integrability}

As already remarked, the Mumford-Shah conjecture is equivalent to proving that the set $K^{(i)}$ is empty for an absolute minimizer. An obvious consequence of the identity $K^{(i)} = \Sigmatre = K^\sharp\setminus \Sigmadue$ is the following

\begin{corollary}\label{c:semplice}
$\mathcal{H}^1 (K^{(i)}) = \mathcal{H}^1 (K^\sharp) = 0$.
\end{corollary}

Using a porosity argument it was first shown by David \cite{david1996} that indeed the Hausdorff dimension of $K^\sharp$ must be strictly smaller than $1$. 
We will show the latter theorem differently. Following the approach of Ambrosio, Fusco, and Hutchinson in \cite{AFH}, we relate a higher integrability estimate of $\nabla u$ to a suitable dimension estimate for $K^\sharp$. More precisely, using again $K^{(i)}= K^\sharp \setminus \Sigmadue$ we can immediately conclude the following

\begin{corollary}\label{c:semplice-2}\label{C:SEMPLICE-2}
If $\nabla u \in L^p$, then $\mathcal{H}^{2-\frac{p}{2}} (K^{(i)}) =0$.
\end{corollary}

It was indeed conjectured by De~Giorgi (in all space dimensions) that 
$\nabla u\in L^p_{{\rm loc}}$ for all $p<4$ (cf. with \cite[conjecture 1]{DG}).
So far only a first step into this direction has been established. 
\begin{theorem}\label{t:hi}
There is $p>2$ such that $\nabla u\in L^p_{\mathrm{loc}}(\Omega)$ for all
$(u,K)$ either absolute minimizers of 
$E_\param$ in $\Omega$, and for all open sets $\Omega\subseteq\R^2$. The same conclusion holds for generalized global minimizers in $\mathbb R^2$.
\end{theorem}
Theorem~\ref{t:hi} has been proven first by the authors in \cite{DLF13} using a reverse H\"older inequality (cf. Section~\ref{ss:reverse holder}). Shortly after 
De~Philippis and Figalli established the result without any dimensional limitation in 
\cite{DePFig14}, as a more direct application of the $\varepsilon$-regularity theory at pure jumps (cf. Section~\ref{ss:porosity}). In this chapter we will present both proofs.

\subsection{An equivalent formulation of the Mumford-Shah conjecture}

Note that, because cracktip is indeed a minimizer, we certainly cannot expect $\nabla u \in L^4$, while on the other hand the Mumford-Shah conjecture would easily imply $\nabla u\in L^p_{{\rm loc}}$ for every $p<4$.
Introducing a finer scale of spaces, it is in fact possible to characterize the conjecture in terms of a sharp integrability result for $\nabla u$. The relevant statement was given in the introduction in Theorem~\ref{t:summability}, which for the reader convenience we recall here.
\index[simb]{aalL^4,infty@$L^{4,\infty}$}\index{weakL4@weak $L^4$}
\begin{theorem}\label{t:summability-2}\label{T:SUMMABILITY-2}
Let $(u,K)$ be an absolute, or generalized global minimizer of $E_\param$ in $\Omega$. The set $K^{(i)}$ is empty if and only if $\nabla u \in L^{4,\infty}_{{\rm loc}}$, namely if and only if for every compact set $U\subset \Omega$ there is a constant $C = C(U)$ such that
\begin{equation}\label{e:L4infty}
|\{x\in U : |\nabla u (x)|^4\leq M\}|\leq \frac{C}{M}
\qquad \forall M \geq 1\, .
\end{equation}
\end{theorem}

\subsection{The case of finitely many connected components} Finally, we recall the statement of Corollary \ref{c:Bonnet-plus}, 
namely that the Mumford-Shah conjecture holds when $K$ has finitely many connected components, even for reduced minimizers.

\begin{corollary}\label{c:Bonnet-plus-2}\label{C:BONNET-PLUS-2}
Conjecture~\ref{c:MS} hold for restricted minimizers $(u,K)$ such that $K$ has a finite number of connected components. 
\end{corollary}

It is in fact a consequence of the work of Bonnet \cite{B96} that Theorem \ref{t:main} is applicable at all points of $K$ under the latter assumption and we include a proof in Section \ref{s:Bonnet-plus}.


\section{Proof of Theorem~\ref{t:structure-2}}

By the $\varepsilon$-regularity theory, in order to show (i) and (ii) it suffices to show, respectively, that:
\begin{itemize}
    \item[(i')] If $x\in \Sigmauno$, then there is one blow-up $(u_\infty, \Sigma_\infty)$ at $x$ which is a triple junction centered at the origin.
    \item[(ii')] If $x\in \Sigmadue$, then there is one blow-up $(u_\infty, \Sigma_\infty)$ at $x$ which is a cracktip centered at the origin.
\end{itemize}
In fact, in both cases it would follow that $K$ is regular in some punctured disk $B_\rho (x)\setminus \{x\}$, implying that both $\Sigmauno$ and $\Sigmadue$ are discrete sets. 

In both cases, we will assume, without loss of generality, that $x=0$. 
We start with (i'). Hence, we fix a sequence of radii $r_j\downarrow 0$ with the property that 
\[
\lim_{j\to\infty} \frac{1}{r_j} \int_{B_{r_j}} |\nabla u|^2=0\, 
\]
and assume, upon extraction of a subsequence, that the corresponding rescaled functions and sets as in Theorem~\ref{t:minimizers compactness} converge to a generalized global minimizer $(u_\infty, K_\infty)$. By Theorem~\ref{t:minimizers compactness}(i) we know that the Dirichlet energy of $u_\infty$ vanishes identically, and so by Theorem~\ref{t:class-global} we know that $(u_\infty, K_\infty)$ is either a constant, or a global pure jump, or a global triple junction. On the other hand, by the Hausdorff convergence of the rescaled sets $K_{0, r_j}$ to $K_\infty$, $K_\infty$ must contain the origin, hence the constant is excluded. If it were a global pure jump, then the $\varepsilon$-regularity theory would imply that $0\in K^{(j)}$, a contradiction to the assumption $0\in\Sigma$. Hence, $(u_\infty, K_\infty)$ must be a global triple junction. Now, if $0$ were not the meeting point of the three half lines forming $K_\infty$, then it would be a pure jump point of $K_\infty$, and a simple diagonal argument would imply the existence of a blow-up of $(u,K)$ at $0$ which is a pure jump, a possibility which has already been excluded. 

In case (ii') we choose our sequence so that $\beta (r_k)$ converges to $0$ and we see immediately that the corresponding blow-up $(u_\infty, K_\infty)$ has the property that $K_\infty$ is contained in a line. As above observe also that $K_\infty$ contains the origin. A result by L\'eger, cf. \cite{Leger}, implies then that $K_\infty$ is either the whole line or a half-line. In Chapter 4 (cf. Theorems~\ref{t:Bonnet-David} and \ref{t:David-Leger}) we already saw the arguments leading to the conclusion that in the first case, we have a pure jump and in the second we have a cracktip. On the other hand, the pure jump case is excluded because the $\varepsilon$-regularity theory would imply that $0$ is a pure jump of $(u,K)$.

For completeness, we give a simple self-contained argument that $K_\infty$ is either a line or a half-line. Even though we follow L\'eger's ideas, rather than relying on his ``magic formula'' in Proposition~\ref{p:Leger-magic}, we give a direct elementary derivation from the inner variations.

Without loss of generality we assume that $K_\infty \subset \{x: x_2=0\}$. We will show below that
\begin{itemize}
    \item[(Cl)] any nontrivial connected component of $K_\infty$ is infinite.  
\end{itemize}
Taking (Cl) for granted, observe that, since the $\varepsilon$-regularity theory implies that regular pure jumps are dense in $K_\infty$, the existence of a connected component of $K_\infty$ which is a singleton implies the existence of infinitely many connected components of $K_\infty$ which are closed intervals of bounded positive length. Therefore, (Cl) implies that $K_\infty$ is either a line or a half-line or the union of two half-lines. We then just need to exclude the possibility that $K_\infty$ contains two connected components which are half-lines. To exclude the latter, observe that we can 
use Theorem~\ref{t:compactness-2} to ``blow-down'' the pair $(u_\infty, K_\infty)$ introducing 
\[
K'_R := \{R^{-1} x: x\in K_\infty\} \qquad u'_R (x):= R^{-\sfrac12} u (Rx)
\]
and letting $R\uparrow \infty$. For an appropriate sequence $R_k$ the pair $(u'_{R_k}, K'_{R_k})$ converges then to a global minimizer $(v_\infty, J_\infty)$ for which $J_\infty$ is a full line, which implies that the global minimizer is a pure jump. But then the $\varepsilon$-regularity theory would imply that for all sufficiently large $k$ the set $K'_{R_k}\cap B_1$ is connected, which is a contradiction. 

We now come to the core of the argument for (ii'), which is the proof of (Cl). 
Consider a connected component $K'$ of $K_\infty$ with positive length: it is either a closed interval or a closed half-line or the full line. By symmetry, it suffices to show that, if $K'$ has a left extremum, which we will denote by $a$, then it cannot have a right extremum. Observe that any point $p= (x_1, 0)\in K'$ which is not an extremum is necessarily a regular jump point (since any blow-up at such points would be a global pure jump) and we define
\[
u^\pm_\infty (p) = \lim_{t\downarrow 0} u_\infty (x_1, \pm t)\, .
\]
It is also the case that, since we are assuming that $K'$ is not the full line, $K$ does not disconnect $\mathbb R^2$. It then follows from Proposition~\ref{p:continuity} that $u_\infty$ is continuous at $a$. Recall that $\frac{\partial u^\pm_\infty}{\partial x_2} = 0$ at all $p\in K'$ which are not extrema, while (since the curvature of $K_\infty$ is $0$),
\[
\left|\frac{\partial u^+_\infty}{\partial x_1} (p)\right| =
\left|\frac{\partial u^-_\infty}{\partial x_1} (p)\right|\, . 
\]
The above derivatives can only vanish at isolated points: if the zeros were to accumulate to a point $q$ which is not an extremum, the classical theory of harmonic functions would imply that both $u^\pm$ would be constant in a neighborhood of $q$. But then unique continuation would imply that $u_\infty$ is actually constant on $\mathbb R^2\setminus K$. If we next fix a point $p$ where
\[
\frac{\partial u^+_\infty}{\partial x_1} (p) \neq 0\, ,
\]
then in a neighborhood of it, we either have 
\begin{equation}\label{e:wrong-alternative}
\frac{\partial u^+_\infty}{\partial x_1} = \frac{\partial u^-_\infty}{\partial x_1}  
\end{equation}
or we have
\begin{equation}\label{e:correct-alternative}
\frac{\partial u^+_\infty}{\partial x_1} = - \frac{\partial u^-_\infty}{\partial x_1}\, .  
\end{equation}
But then, again the unique continuation theory for harmonic function would tell that one of the two alternatives holds on $K'$ with the exception of its extrema. However, if the first alternative were to hold, integrating it starting from the left extremum $a$ we would conclude that $u^+_\infty=u^-_\infty$ on $K'$: we could then remove it and decrease the energy, violating the minimality. We must therefore have \eqref{e:correct-alternative}. We next claim that $\frac{\partial u^+_\infty}{\partial x_1}$ never vanishes in the interior of $K'$: that, and \eqref{e:correct-alternative} would imply that there cannot be a right extremum of $K'$, because at such extremum $u_\infty$ would be continuous, and hence $u^+_\infty$ and $u^-_\infty$ would have to coincide, while integrating \eqref{e:correct-alternative} from $a$ to $b$ we would conclude that $u^+_\infty (b)\neq u^-_\infty (b)$. 

We thus are left to show that, if $(x_1, 0)$ is an interior point of $K_\infty$, then $\frac{\partial u^+_\infty}{\partial x_1} (x_1, 0) \neq 0$. By translation we can assume that $(x_1, 0) = (0,0)$ and we will then show that 
\begin{equation}\label{e:Leger-magic-12}
2 \pi \left(\frac{\partial u^+_\infty}{\partial x_1}\right)^2 (0)
= \int_{\{(t, 0)\in (\mathbb R\times \{0\}) \setminus K_\infty\}} \frac{dt}{t^2}\, .
\end{equation}
L\'eger in \cite{Leger} derives \eqref{e:Leger-magic-12} directly from his ``magic formula'' \eqref{e:Leger-magic}, while here we will show how it follows immediately from the inner variation formula \eqref{e:inner-C1}, using the same test of the proof of Proposition~\ref{p:Leger-magic}. 

We fix positive radii $\rho<R$ and consider the vector field $\psi (x,y) = \varphi (|(x,y)|) (x, -y)$ where 
\[
\varphi (t) = \left\{
\begin{array}{ll}
\rho^{-2} - R^{-2} \qquad &\mbox{if $t\leq \rho$}\\
t^{-2} - R^{-2} \qquad &\mbox{if $\rho \leq t \leq R$}\\
0 &\mbox{otherwise}\, .
\end{array}\right.
\]
Strictly speaking, the latter is not a valid test in the inner variation formula, which in our case would read
\begin{equation}\label{e:internal-again}
\int_{\mathbb R^2\setminus K} (2 \nabla^T u_\infty \cdot D\psi\, \nabla u_\infty - |\nabla u_\infty|^2 {\rm div}\, \psi) = \int_{K_\infty} e^T\cdot D\psi\, e\, d\mathcal{H}^1\, ,
\end{equation}
because $\psi$ is not $C^1$.
However, if we assume that $\mathcal{H}^1 (K_\infty \cap (\partial B_\rho \cup \partial B_R))=0$, it is easily seen that the left hand side \eqref{e:internal-again} makes sense because $\psi$ is $\mathcal{H}^1$-a.e. differentiable on $K_\infty$, while a standard regularization argument shows the validity of the formula. Next we compute $D\psi$ in the two relevant domains where it does not vanish:
\begin{align}
D\psi & = \left(\frac{1}{\rho^2}- \frac{1}{R^2}\right)
\left(\begin{array}{ll}
1 & 0\\
0 & -1
\end{array}\right)
\qquad \mbox{on $B_\rho$}\\
D\psi &=  -\frac{1}{R^2} \left(\begin{array}{ll}
1 & 0\\
0 & -1
\end{array}\right) + \frac{1}{(x^2+y^2)^2} 
\left(\begin{array}{ll}
y^2-x^2 & 2xy\\
-2xy & y^2-x^2
\end{array}
\right) \qquad \mbox{on $B_R\setminus \bar B_\rho$}\, .
\end{align}
Inserting on \eqref{e:internal-again} we immediately get
\begin{align*}
& \frac{2}{\rho^2} \int_{B_\rho\setminus K_\infty} \left(\left(\frac{\partial u_\infty}{\partial x_1}\right)^2 - \left(\frac{\partial u_\infty}{\partial x_2}\right)^2\right) -
\frac{2}{R^2} \int_{B_R\setminus K_\infty} \left(\left(\frac{\partial u_\infty}{\partial x_1}\right)^2 - \left(\frac{\partial u_\infty}{\partial x_2}\right)^2\right)\\
= &\frac{1}{\rho^2} \mathcal{H}^1 (K_\infty\cap B_\rho) - \frac{1}{R^2} \mathcal{H}^1 (K_\infty\cap B_R) - \int_{\{(t,0)\in (B_R\setminus B_\rho)\cap K_\infty\}} \frac{dt}{t^2}\, .
\end{align*}
We choose $\rho$ sufficiently small so that $B_\rho \cap K_\infty = B_\rho \cap 
(\mathbb R \times \{0\} )$. Then we let $R\uparrow \infty$ and use
\[
\int_{B_R\setminus K_\infty} |\nabla u_\infty|^2 + \mathcal{H}^1 (B_R\cap K_\infty) \leq 2\pi R\, ,
\]
to conclude
\[
\frac{2}{\rho^2} \int_{B_\rho\setminus K_\infty} \left(\left(\frac{\partial u_\infty}{\partial x_1}\right)^2 - \left(\frac{\partial u_\infty}{\partial x_2}\right)^2\right) = \frac{2}{\rho} - \int_{\{(t,0)\in K_\infty\setminus B_\rho\}} \frac{dt}{t^2}\, .
\]
Since, however, the choice of $\rho$ yields
\[
\frac{2}{\rho} = \int_{\{(t, 0)\in (\mathbb R\times \{0\})\setminus B_\rho\}} 
\frac{dt}{t^2}\, ,
\]
we arrive to
\begin{equation}\label{e:Leger-magic-13}
\frac{2}{\rho^2} \int_{B_\rho\setminus K_\infty} \left(\left(\frac{\partial u_\infty}{\partial x_1}\right)^2 - \left(\frac{\partial u_\infty}{\partial x_2}\right)^2\right)= \int_{\{(t, 0)\in (\mathbb R\times \{0\}) \setminus K_\infty\}} \frac{dt}{t^2}\, .
\end{equation}
We now observe that, because $0$ is a regular jump point, the function $\left(\frac{\partial u_\infty}{\partial x_1}\right)^2 - \left(\frac{\partial u_\infty}{\partial x_2}\right)^2$
is actually continuous over $B_\rho$ and its value at $0$ is
$\left(\frac{\partial u^+_\infty}{\partial x_1}\right)^2 (0)$.
So, letting $\rho\downarrow 0$ in \eqref{e:Leger-magic-13}, we obtain \eqref{e:Leger-magic-12}.

It finally remains to show (iii). It is however a straightforward consequence of the material in Chapter 4 that, if a point $p\in K$ is not the accumulation point of infinitely many connected components of $K$, then the blow-ups at $K$ are either cracktips or global pure jumps or triple junctions, while it is a straightforward consequence of the $\varepsilon$-regularity theory that if one of the blow-ups belong to these subsets, then $K\cap B_\rho (p)$ is connected for all sufficiently small radii $\rho$.  

\begin{remark}\label{r:Leger}
Along the proof of the previous result we have given a different argument for a rigidity property originally 
due to L\'eger \cite{Leger}: if the jump set of a global generalized minimizer is contained in a line, 
then necessarily it is either a cracktip or a pure jump. 
An analogous rigidity property in higher dimensions is not yet known as it is pointed out at 
the end of David's book \cite{DavidBook}, though some interesting results in this direction have been obtained 
by Lemenant in \cite{Lemenant2015}.
\end{remark}

\section{Proofs of Corollary~\ref{c:semplice-2} and Theorem~\ref{t:summability-2}}

Corollary~\ref{c:semplice-2} is a simple consequence of the structure theorem and standard results in measure theory. 

\begin{proof}[Proof of Corollary~\ref{c:semplice-2}]
Suppose that $|\nabla u|\in L^p_{{\rm loc}}(\Omega)$ for some $p> 2$. For every $s\in [0,2]$ consider the set 
\[
\Lambda_s:=\left\{x\in\Omega:\,
\limsup_{r\to 0}r^{-s}\int_{B_r(x)}|\nabla u|^p > 0\right\}\,,
\]
then $\Lambda_s$ is a subset of $K$. Choose $s:= 2-\frac{p}{2}$, then H\"older's inequality, yields that
\[
\lim_{r\downarrow 0} \frac{1}{r} \int_{B_r (x)\setminus K} |\nabla u|^2 = 0\,
\]
 for every $x\in K\setminus \Lambda_s$. 
In particular, Theorem~\ref{t:structure-2} implies that every point $x\in K\setminus \Lambda_s$ is either a pure jump point or a triple junction. We thus conclude that 
$K^{(i)}\subset \Lambda_s$.
Consider the Radon measure defined on any subset $E\subseteq\Omega$ by
\[
\mu(E):=\int_E |\nabla u|^p\,.
\]
As $\Lambda_s\subset K$, then $\mu (\Lambda_s) =0$. We can employ
\cite[Theorem~2.56]{AFP00} to infer that $\mathcal{H}^s (\Lambda_s) = 0$, 
which concludes the proof.  
\end{proof}

For the proof of Theorem~\ref{t:summability-2} we first make the following preliminary observation.

\begin{lemma}\label{l:L4infty}
Let $f\in L^{4,\infty}_{{\rm loc}}(\Omega)$, $\Omega\subseteq\R^2$, then for all $\e>0$ the set
\begin{equation}\label{e:ctip}\
D_\e:=\left\{x\in\Omega:\,\liminf_r\frac 1r\int_{B_r(x)}f^2(y)\ge\e\right\}
\end{equation}
is locally finite.
\end{lemma}
\begin{proof}
We shall show in what follows that if $f\in L^{4,\infty}(\Omega)$ then $D_\e$
is finite. An obvious localization argument then gives the general case.

Let $\e>0$ and consider the set $D_\e$ in \eqref{e:ctip} above.
First note that for any $B_r (x)\subset \Omega$ and any $\lambda>0$ we have the estimate
\begin{align}\label{e:stimasublev}
\int_{\{y\in B_r(x):\,|f(y)|\geq \lambda\}}f^2(y)&\leq
\int_{\{y\in\Omega:\,|f(y)|\geq \lambda\}}f^2(y)\notag\\
&=2\int_\lambda^{\infty}t\, \left|\{y\in\Omega:\,|f(y)|\geq t\}\right|\, dt
\leq\int_\lambda^{\infty}\frac{2C}{t^3}dt=\frac C{\lambda^2}\,,
\end{align}
where $C>0$ is the constant introduced in \eqref{e:L4infty}.
If $x\in D_\e$ and $r>0$ satisfy 
\begin{equation}\label{e:stimamisura}
\int_{B_r(x)}f^2(y)\geq \frac \e2 r,
\end{equation}
by choosing $\lambda=2(C/r\e)^{\sfrac12}$ in \eqref{e:stimasublev} we conclude 
\begin{equation}\label{e:stimasuplev}
\int_{\{y\in B_r(x):\,|f(y)|<\,2( \frac C{r\e})^{\sfrac12}\}}f^2(y)\geq
\frac \e4 r.
\end{equation}
Furthermore, the trivial estimate
\[
\int_{\{y\in B_r(x):\,|f(y)|<\lambda\}}f^2(y)<\pi\lambda^2r^2,
\]
implies for $\lambda=(\e/8\pi r)^{\sfrac12}$
\begin{equation}\label{e:stimasublev1}
\int_{\{y\in B_r(x):\,|f(y)|<(\frac \e{8\pi r})^{\sfrac12}\}}f^2(y)<\frac\e 8r.
\end{equation}
By collecting \eqref{e:stimasuplev} and \eqref{e:stimasublev1}
we infer
\[
\int_{\{y\in B_r(x):\,(\frac \e{8\pi r})^{\sfrac12}\leq
|f(y)|<\,2(\frac C{r\e})^{\sfrac12}\}}f^2(y)\geq\frac\e 8r,
\]
that in turn implies
\begin{equation}\label{e:stimafond}
\big|\{y\in B_r(x):\,|f(y)|\geq (\frac \e{8\pi r})^{\sfrac12}\}\big|
\geq\frac{\e^2r^2}{32C}.
\end{equation}
Let $\{x_1,\ldots,x_N\}\subseteq D_\e$ and $r>0$ be a radius such 
that the balls $B_r(x_i)\subseteq\Omega$ are disjoint and 
\eqref{e:stimamisura} holds for each $x_i$. Then, from \eqref{e:stimafond} 
and the fact that $f\in L^{4,\infty}(\Omega)$, we infer
\[
N\frac{\e^2r^2}{32C}\leq  \left|\{y\in\Omega:\,
|f(y)|\geq(\frac \e{8\pi r})^{\sfrac12}\}\right|
\leq\frac {C(8\pi r)^2}{\e^2}\Longrightarrow 
N\leq\frac {2^{11}C^2\pi^2}{\e^4},
\]
and the conclusion follows at once.
\end{proof}

Another ingredient in the proof of Theorem~\ref{t:summability-2} is the following strengthened version of the first conclusion in Theorem~\ref{t:structure-2}.
We in fact only need \eqref{e:soglia-per-dir}, but for completeness we also prove statement \eqref{e:soglia-per-flat}, which has its own independent interest.

\begin{proposition}\label{p:structure-3}
There is a universal constant $\varepsilon >0$ with the following property. If $(u,K)$ is an absolute or generalized minimizer of $E_\lambda$, then 
\begin{align}
\Sigmauno\cup K^{(j)} &= \left\{x\in K: \liminf_{r\downarrow 0} d(x,r) < \varepsilon\right\}\,\label{e:soglia-per-dir} \\
\Sigmadue\cup K^{(j)} & = \left\{x\in K: \liminf_{r\downarrow 0} \beta (x,r) < \varepsilon\right\}\, .\label{e:soglia-per-flat}
\end{align}
\end{proposition}
\begin{proof}
We start by addressing \eqref{e:soglia-per-dir}.
We first of all observe that there is an absolute constant $\bar \varepsilon >0$ such that, if $(u, K)$ is a generalized global minimizer with $0\in K$ and $\int_{B_8\setminus K} |\nabla u|^2 < \bar \varepsilon$, then either $B_1\cap K$ is diffeomorphic to a pure jump or $B_4\cap K$ is diffeomorphic to a triple junction. This can be done using the compactness Theorem~\ref{t:compactness-2} and the $\varepsilon$-regularity theory. Indeed, assume by contradiction $(u_j, K_j)$ is a sequence of generalized global minimizers which satisfy
\[
\int_{B_8\setminus K_j} |\nabla u_j|^2 < \frac{1}{j}\, 
\]
but violate our claim. Then up to subsequences they converge to a global generalized minimizer $(u_\infty, K_\infty)$ with $0\in K_\infty$ and such that 
\[
\int_{B_8\setminus K_\infty} |\nabla u_\infty|^2=0\,.
\]
It follows from the theory of Chapter 5 that $u_\infty$ is an elementary global minimizer. If it is a pure jump, then it would follow from the $\varepsilon$-regularity theory that $K_j\cap B_1$ is diffeomorphic to a line for a sufficiently large $j$. If it is a triple junction, but the meeting point of the three half-lines is not contained in $B_2$, then the very same conclusion can again be drawn. If on the other hand, it is a triple junction with a meeting point at $x\in B_2$, then we can use the $\varepsilon$-regularity theory to say that $B_6 (x) \cap K_j$ is diffeomorphic to a triple junction for a sufficiently large $j$.

Given the first part, consider now $\varepsilon := \frac{\bar \varepsilon}{8}$ and fix a point where $\liminf_{r\downarrow 0} d(x,r) < \varepsilon$ for some absolute minimizer or generalized global minimizer $(u,K)$. Then we draw from the above argument that there is at least one blow-up $(u_\infty, K_\infty)$ at $x$ with the property that $0\in K_\infty$ is either a pure jump point or a triple junction. A further blow-up with a simple diagonal argument implies then that there is a blow-up of $(u,K)$ at $x$ which is either a pure jump or a triple junction. This concludes the proof.

\medskip

{The argument for the second statement is analogous with a slight twist. We claim that, for an appropriate choice of $\varepsilon>0$ sufficiently small, if $\beta (0, 1)<\varepsilon^4$ and $0\in K$, then one of the following (non-exclusive) alternatives holds:
\begin{itemize}
    \item[(i)] $K\cap B_\varepsilon$ is a $C^1$ arc with both endpoints contained in $\partial B_\varepsilon$;
    \item[(ii)] $K\cap B_{5\varepsilon}$ is a $C^1$ arc with one endpoint in $\partial B_{5\varepsilon}$ and the other endpoint in $\overline{B}_{3\varepsilon}$. 
\end{itemize}
Obviously this is enough to prove \eqref{e:soglia-per-flat}.

We again argue by contradiction and assume to have a sequence $(u_j, K_j)$ which goes against the latter statement but where each pair satisfies the assumption with $\eta=\frac{1}{j}$. Let $(v_j, \tilde{K}_j)$ be suitable rescalings of the pairs by a factor $\eta^{-1}$ and note that, looking at the definition of $\beta$, we find a sequence of affine lines $\ell_j$ with the property that 
\[
\int_{B_{\eta}^{-1}\cap \tilde{K}_j} \dist^2 (x, \ell_j) d\mathcal{H}^1 (x) \leq \frac{1}{j}\, .
\]
In particular, since $0\in \ell_j$ and using the density lower bound, $\dist (0, \ell_j)$ converges to $0$.
Extracting suitable subsequences, we can assume that $(v_j, \tilde{K}_j)$ converges to a global generalized minimizer $(u_\infty, K_\infty)$ with the property that $K_\infty$ is a subset of a line $\ell_\infty$. The latter line must, moreover, pass through $0$, which is also an element of $K$. However by Remark \ref{r:Leger} this implies that either $K_\infty=\ell_\infty$, and hence the global minimizer is a pure jump, or that the global minimizer is a cracktip ending at some point $p_0$, in particular $K_\infty$ is an half-line originating at $p_0$ and containing the origin. 

If we are in the first case, or we are in the second case but the point $p_0\not\in B_3$, then clearly the sets $\tilde{K}_j\cap \overline{B}_2$ converge, in the Haudorff distance, to $\ell_\infty \cap \overline{B}_2$. But then, from the $\varepsilon$-regularity theory it follows that, for $j$ sufficiently large, the sets $\tilde{K}_j \cap B_{3/2}$ are $C^1$ arcs converging in $C^1$ to the segment $\ell_\infty \cap B_{3/2}$, in particular $K_j \cap B_{1/j}$ is a $C^1$ arc with both endpoints in $\partial B_{1/j}$, falling in alternative (i). If we are in the second alternative and the point $p_0\in B_3$, we now conclude that $\tilde{K}_j \cap \overline{B}_{20} (p_0)$ is converging in the Hausdorff distance to a segment $[p_0, p_1]$ containing the origin and with $p_1\in \partial B_{20} (p_0)$. In particular, again by the $\varepsilon$ regularity theorem, the sets $\tilde{K}_j\cap B_{10} (p_0)$ are, for $j$ sufficiently large, $C^1$ arcs converging in $C^1$ to $[p_0, p_1]\cap B_{10} (p_0)$. This implies that $\tilde{K}_j\cap B_5$ is, for $j$ large enough, a $C^1$ arc with one endpoint in $B_4$ and the other endpoint in $\partial B_5$. Scaling back, $\tilde{K}_j$ falls under case (ii) above. This provides a contradiction and completes the proof.} 
\end{proof}

\begin{proof}[Proof of Theorem~\ref{t:summability-2}]
First of all assume that $(u,K)$ is an absolute minimizer of $E_\lambda$ (or a generalized minimizer of $E$) in some $\Omega$ and that $\nabla u\in L^{4,\infty}_{{\rm loc}}(\Omega)$. Without loss of generality we can assume that $\Omega$ is the unit ball $B_1$ and that $\nabla u\in L^{4, \infty} (B_1)$. Let $\varepsilon$ be the positive number of Proposition~\ref{p:structure-3} and define 
\[
D_\varepsilon := \left\{x: \liminf_{r\downarrow 0} \frac{1}{r} \int_{B_r (x)\setminus K} |\nabla u|^2 \geq \varepsilon \right\}\, .
\]
By Lemma~\ref{l:L4infty} $D_\varepsilon$ is a finite set, which we enumerate as $\{p_1, \ldots , p_N\}$ and by Proposition~\ref{p:structure-3} any point $p\in K\setminus \{p_1, \ldots , p_N\}$ is either a pure jump or a triple junction. It then turns out that, if $p\in K\setminus \{p_1, \ldots, p_N\}$ there are only two possible alternatives: either there is a Lipschitz curve $\gamma\subset K$ connecting $p$ with a point $q\in \partial B_1$, or there is a Lipschitz curve $\gamma \subset K$ connecting $p$ to one of the $p_i$'s. So, we can partition $K$ into a collection of those finitely many connected components which contain at least one of the $p_i$'s and a countable number of connected components which do not, but whose closure must contain a point of $\partial B_1$. A point $q\in \partial B_1$ cannot be the accumulation point of an infinite number of the latter, because otherwise, we would have $\mathcal{H}^1 (K)=\infty$. It follows therefore from Theorem~\ref{t:structure-2} that each $p_i$ is necessarily a regular loose end of $K$.

Consider on the contrary an absolute minimizer $(u,K)$ in $\Omega$ for which the Mumford-Shah conjecture holds. Let $U\subset\subset \Omega$. Then there is a finite subset $\{p_1, \ldots, p_N\} \subset U\cap K$ of regular loose ends, while all the other points of $K\cap U$ are regular jump points or triple junctions. From the regularity theorem at regular loose ends, we infer the existence of balls $B_{r_i} (p_i)$ and of an absolute constant $C$ with the properties that $|\nabla u(x)| \leq C |x-p_i|^{-\sfrac12}$. On the other hand by the regularity theory at triple junctions and regular jump points, $|\nabla u|$ is bounded on $U\setminus \bigcup_i B_{r_i} (p_i)$. Hence, we conclude that there is a constant $C (U)$ such that 
\[
|\nabla u(x)|\leq C \max \left\{ |x-p_i|^{-\sfrac12}\right\}\, .
\]
Since the function on the right-hand side belongs to $L^{4,\infty} (U)$, this completes the proof.
\end{proof}

\section{Reverse H\"older inequality for \texorpdfstring{$\nabla u$}{Du}}\label{ss:reverse holder}

Following a classical path, the key ingredient used in \cite[Theorem 1.3]{DLF13} 
to establish Theorem~\ref{t:hi} (for $E_0$) is a reverse H\"older inequality\index{reverse H\"older inequality@reverse H\"older inequality} for the gradient, which we state independently. 
\begin{theorem}\label{t:rH}
For all 
$q\in(1,2)$ there exist $C_0>0$, $r_0\in(0,1)$ such that if $(u,K)$ is either a restricted, or an absolute, or a generalized global minimizer of $E_\param(\cdot,\cdot,B_1,g)$, with $\lambda\leq 1$, $\|g\|_{\infty}\leq M_0$ (we use the notation introduced in  Assumption~\ref{a:blow-up}), then for every $x\in B_{\sfrac12}$ and $r\in(0,r_0)$
\begin{equation}\label{e:rH}
\left(\fint_{B_r(x)\setminus K}|\nabla u|^2\right)^{\sfrac 12}
\leq C_0\left(\fint_{B_{2r}(x)\setminus K}|\nabla u|^q
\right)^{\sfrac1q}+C_0\|g\|_{\infty}\,.
\end{equation}
\end{theorem}
Theorem~\ref{t:hi} is then a consequence of a by now classical result 
by Giaquinta and Modica \cite{GM}. 
\begin{theorem}
Let $v\in L^q_{{\rm loc}}(\Omega)$, $q>1$, be nonnegative such that for some
constants $\beta>0$, $t\geq 1$ and $R_0>0$
 \[
  \left(\fint_{B_r(z)}v^q\right)^{\sfrac 1q}\leq 
  \beta\fint_{B_{t\,r}(z)}v+\fint_{B_{t\,r}(z)}h
 \]
for all $z\in\Omega$, $r\in\left(0,R_0\wedge \mathrm{dist}(z,\partial\Omega)\right)$, 
and with $h\in L^s(\Omega)$ for some $s>1$.

Then $v\in L^p_{{\rm loc}}(\Omega)$ for some $p>q$ and there is $C=C(\beta,n,q,p,\lambda)>0$
 such that 
 \[
 \left(\fint_{B_r(z)}v^p\right)^{\sfrac 1p}\leq C\,
 \left(\fint_{B_{2r}(z)}v^q\right)^{\sfrac 1q}
   +C\left(\fint_{B_{2r}(z)}h^q\right)^{\sfrac 1q}.
 \]
\end{theorem}
We provide here a new proof of Theorem~\ref{t:rH} relying on 
the compactness properties established in Theorem~\ref{t:minimizers compactness} rather than on the theory of Caccioppoli partitions as originally done in \cite{DLF13}. Nevertheless, the overall strategy is the same. 
\begin{proof}[Proof of Theorem~\ref{t:rH}]
Assume by contradiction that there is $q\in(1,2)$ such that for every $j\in\N$ one can find $(u_j,K_j)$ either restricted, or absolute, or generalized global minimizer of $E_{\param_j}(\cdot,\cdot,B_1,g_j)$ with ${\param_j}\leq 1$ and $\|g_j\|_{\infty}\leq M_0$, radii $r_j\downarrow 0$, and 
points $x_j\in B_{\sfrac12}$ such that
\begin{equation}\label{e:hpcontrad}
j\left(\fint_{B_{2r_j}(x_j)\setminus K_j}
|\nabla u_j|^q\right)^{\sfrac1q}+j\|g_j\|_{\infty}\leq
\left(\fint_{B_{r_j}(x_j)\setminus K_j}|\nabla u_j|^2\right)^{\sfrac 12}
\end{equation}
Consider the rescalings $(\widetilde{u}_j,\widetilde{K}_j):=((u_j)_{x_j,r_j}, (K_j)_{x_j,r_j})$, the conditions 
in Assumption~\ref{a:blow-up} are satisfied, then by Theorem~\ref{t:minimizers compactness} $(\widetilde{u}_j,\widetilde{K}_j)$ converge up to a subsequence to 
a blow-up limit $(u_\infty,K_\infty)$ which is a global generalized (restricted) minimizer of $E_0$ (in case $(u_j,K_j)$ are restricted minimizers). 
Moreover, by \eqref{e:hpcontrad} and the density upper bound \eqref{e:upper bound} 
we get
\[
\left(\fint_{B_2\setminus \widetilde{K}_j}|\nabla \widetilde{u}_j|^q\right)^{\sfrac1q}\leq j^{-1}
\left(\fint_{B_1\setminus \widetilde{K}_j}|\nabla \widetilde{u}_j|^2\right)^{\sfrac 12}
\leq j^{-1}(2\pi+\lambda_j\pi M_0^2r_j)^{\sfrac 12}\,.
\]
The latter inequality and Theorem~\ref{t:minimizers compactness} 
yield that
\[
\lim_j\int_{B_2\setminus \widetilde{K}_j}|\nabla \widetilde{u}_j|^2=
\int_{B_2\setminus K_\infty}|\nabla u_\infty|^2=0.
\]
Therefore, Corollary~\ref{c:no-vanishing} implies that $(u_\infty,K_\infty)$ is an elementary global minimizer; in turn, this and Theorem~\ref{t:minimizers compactness} imply that for every $R>0$
\begin{equation}\label{e:enrg conv to glob min}
    \lim_j\int_{B_R\setminus \widetilde{K}_j}|\nabla \widetilde{u}_j|^2=0\,.
\end{equation}
In particular, $(u_\infty,K_\infty)$ is either a constant, or a global pure jump or a global triple junction (cf. Theorem~\ref{t:class-global}). In any case, for all $R>0$
\begin{equation*}\label{e:enrg conv to glob min 2}
\mathcal{H}^0(K_\infty\cap\partial B_R)\leq 3,\qquad
\mathcal{H}^1(K_\infty\cap B_R)\leq 3R\,.
\end{equation*}
With fixed $R>0$ we shall now choose a radius $\rho\in[\frac R4,R]$ conveniently as outlined in what follows
\begin{itemize}
    \item[(a$_1$)] if $\mathcal{H}^0(K_\infty\cap\partial B_R)=0$ choose $\rho=R$; 
    \item[(a$_2$)] if $\mathcal{H}^0(K_\infty\cap\partial B_R)=2$ then $K_\infty\cap B_R$ is a segment and $\partial B_R\setminus K_\infty$ is the union of two arcs. Then, either each of them has a length less or equal to $\frac{4\pi}3R$ or $K\cap B_{\frac R2}=\emptyset$. In the first instance set $\rho=R$, in the second $\rho=\frac R2$;
    \item[(a$_3$)] if $\mathcal{H}^0(K_\infty\cap\partial B_R)=3$, then $K_\infty\cap B_R$ is a (possibly off-centered) propeller and $\partial B_R\setminus K_\infty$ is the union of three arcs. 
    Then, either each of them has length less or equal to $(2\pi-\frac18)R$, in this case set $\rho=R$, or $\mathcal{H}^0(K_\infty\cap\partial B_{\frac R2})=2$.
    In the last event, we are back to the setting of item (a2) with the radius $\frac R2$ playing the role of $R$. Thus, $K_\infty\cap B_{\frac R2}$ is a segment, and $\partial B_{\frac R2}\setminus K_\infty$ is the union of two arcs, and either each of them has a length less or equal to $\frac{2\pi}3R$ or $K\cap B_{\frac R4}=\emptyset$. In the first instance set $\rho=\frac R2$, in the second $\rho=\frac R4$.
\end{itemize}
By \eqref{e:enrg conv to glob min} and Theorem~\ref{t:minimizers compactness} we may select $j_0\in\N$ such that for all $j\geq j_0$
\[
\int_{B_\rho\setminus \widetilde{K}_j}|\nabla \widetilde{u}_j|^2+\dist^2_H(\widetilde{K}_j\cap B_\rho,K_\infty\cap B_\rho)\leq\e \rho\,.
\]
Therefore, thanks to Theorem~\ref{t:eps_salto_puro} and
Theorem~\ref{t:eps_tripunto} for all $j\geq j_0$ and for some $\beta\in(0,\frac13)$ one of the following alternatives is true
\begin{itemize}
    \item[(b$_1$)] $\widetilde{K}_j\cap B_\rho=\emptyset$;
    \item[(b$_2$)] For each $t\in ((1-\beta)\rho,\rho)$,  
  $\partial B_{t}\setminus \widetilde{K}_j$ is the union of two arcs 
  $\gamma_1^j$ and $\gamma_2^j$ each with length $< (2\pi-\frac{1}{9})t$, 
  whereas $\widetilde{K}_j\cap B_t$ is connected and divides $B_t$ in two components 
  $B_1^j$, $B_2^j$ with $\partial B_i^j=\gamma_i^j\cup 
  (\widetilde{K}_j\cap\overline{B_t})$;

  \item[(b$_3$)] For each $t\in ((1-\beta)\rho, \rho)$, $\partial B_{t}\setminus \widetilde{K}_j$ 
is the union of three arcs 
  $\gamma_1^j$, $\gamma_2^j$ and $\gamma_3^j$ each with length $<(2\pi-\frac{1}{9})t$, whereas
  $B_t\cap \widetilde{K}_j$ is connected and divides $B_t$ in three connected 
  components $B_1^j$, $B_2^j$ and $B_3^j$ with 
  $\partial B_i^j\subset\gamma_i^j\cup  (\widetilde{K}_j\cap \overline{B_t})$.
\end{itemize}
We finally choose $r\in((1-\beta)\rho,\rho)$ and a subsequence, not relabeled, such that
\begin{itemize}
 \item[(A)] $h_j:=\widetilde{u}_j|_{\partial B_r}$ belongs to $W^{1,2}(\gamma)$ 
 for any connected component $\gamma$ of $\partial B_r\setminus \widetilde{K}_j$;

 \item[(B)] $h_j$ satisfies 
 \[
\int_{\partial B_r\setminus \widetilde{K}_j}
|h_j^\prime|^q d\HH^1
\leq\frac 1{\beta \rho}\int_{B_{\rho}}|\nabla \widetilde{u}_j|^q\,.
 \]
\end{itemize}
Let us conclude our argument by showing that \eqref{e:hpcontrad} is 
violated for $j$ sufficiently big. To this aim, we note first that 
the choice $\beta<\frac13$ yields that $r>\frac23\rho$.

In case (b$_1$) holds, $\partial B_r\cap \widetilde{K}_j=\emptyset$ and $\widetilde{u}_j$ 
is the harmonic extension of its boundary trace $h_j$. Hence, the embedding $W^{1,q}(\partial B_r)\to H^{\sfrac12}(\partial B_r)$ implies for some constant $C>0$ (independent of $j$)
\begin{align*}
\int_{B_{\frac23\rho}} |\nabla \widetilde{u}_j|^2& \leq \int_{B_r}|\nabla \widetilde{u}_j|^2
\leq C\min_{c\in\mathbb R} \|h_j-c\|_{H^{\sfrac12}(\partial B_r)}^2\\ 
&\leq C\,\left(\int_{\partial B_r}|h_j^\prime|^q\,d\HH^1\right)^{\sfrac2q}
\stackrel{(B)}{\leq} C\,
\left(\frac 1{\beta \rho}\int_{B_\rho}
|\nabla \widetilde{u}_j|^q\right)^{\sfrac2q}\,.
\end{align*}
Choosing $R=\rho\in[\frac32,2]$
and scaling back, we get a contradiction to \eqref{e:hpcontrad}.

In case (b$_2$) or (b$_3$) hold the construction is similar.
Denote by $J_j$ a \emph{minimal connection} 
\index{minimal connection@ minimal connection} relative to $\widetilde{K}_j\cap\partial B_r$, i.e. a closed and connected set in $\overline{B}_r$ containing 
$\widetilde{K}_j\cap\partial B_r$ with shortest length (cf. \cite[Theorem~4.4.20]{AmbTil}).
Then $J_j$ splits $\overline{B_r}$ into two (case (b$_2$)) or three (case (b$_3$)) regions denoted by $B^i_r$.
Let $\gamma^i$ be the arc of $\partial B_r$ contained in the boundary of $B^i_r$. 
The ensuing Lemma~\ref{l:traceext} provides a function $w_j^i\in W^{1,2}(B_r)$ with boundary trace $h_j$ and satisfying for some absolute constant $C>0$
\begin{equation}\label{e:wij}
\int_{B_r}|\nabla w_j^i|^2\,\leq C
\left(\int_{\gamma^i}|h_j^\prime|^q\,d\HH^1\right)^{\sfrac2q}.
\end{equation}
Denote by $w_j$ the function equal to $w^i_j$ on $B^i_j$, equal to $\widetilde{u}_j$ on $B_{r_j^{-1}}(-x_j)\setminus \cup_i B^i_j$, and extend $J_j$ as 
$J_j\setminus B_r=\widetilde{K}_j\setminus B_r$. 
Clearly, $w_j\in H^1(B_{r_j^{-1}}(-x_j)\setminus J_j)$. 
Scaling back $(w_j,J_j)$ to $B_1$ we obtain a pair admissible to test the 
minimality of $(u_j,K_j)$. 
On setting $\widetilde{g}_j:=r_j^{-\sfrac12}g_j(x_j+r_j \cdot)$, 
for some constant $C>0$ we then deduce that
\begin{align*}
\int_{B_{\frac23\rho}\setminus \widetilde{K}_j}& |\nabla \widetilde{u}_j|^2 
\leq \int_{B_r\setminus J_j}|\nabla \widetilde{u}_j|^2
\notag\\ &
\leq\int_{B_r\setminus J_j}|\nabla w_j|^2 +
\underbrace{\HH^1 (J_j\cap B_r) - \HH^1 (\widetilde{K}_j\cap B_r)}_{\leq 0}\notag\\
&\qquad\qquad
+\lambda r_j^2\int_{B_r\setminus J_j}|w_j-\widetilde{g}_j|^2
-\lambda r_j^2\int_{B_r\setminus \widetilde{K}_j}|\widetilde{u}_j-\widetilde{g}_j|^2\\
&\leq\int_{B_r\setminus J_j}|\nabla w_j|^2+C
\|\widetilde{g}_j\|_{\infty}^2\rho^2\,r_j^2
\stackrel{(\ref{e:wij})}{\leq} C\left(\int_{\partial B_r\setminus \widetilde{K}_j}|h_j^\prime|^q\,d\HH^1\right)^{\sfrac2q}
+C\|g_j\|_{\infty}^2\rho^2r_j\notag\\
&\stackrel{(B)}{\leq} 
C\left(\frac 1{\beta \rho}\int_{B_\rho\setminus\widetilde{K}_j}|\nabla \widetilde{u}_j|^q\,\right)^{\sfrac2q}
+C\|g_j\|_{\infty}^2\rho^2r_j\,.
\end{align*}
Choosing $\rho\in[\frac32,2]$ and $R$ accordingly, scaling back 
we find a contradiction to \eqref{e:hpcontrad}.
\end{proof}
In the proof above we have used an elementary extension result that we prove in what follows.
\begin{lemma}\label{l:traceext}
For any $q\in(1,2)$ there exists $C=C(q)> 0$ such that
the following holds. For any arc $\gamma\subseteq\partial B_1$ 
and any $h\in W^{1,q}(\gamma)$, there exists 
$w\in W^{1,2}(B_1)$ with trace $g$ on $\gamma$ and
\begin{equation}\label{e:extension}
\|\nabla u_\infty\|_{L^{2}(B_1)}\leq \frac{C}{(2\pi - \HH^1(\gamma))^{1-\frac{1}{q}}}\, 
\|h^\prime\|_{L^{q}(\gamma)}.
\end{equation}
In addition, if $h\in L^\infty(\gamma)$, then $\|w\|_{L^\infty(B_1)}\leq 2\|h\|_{L^\infty(\gamma)}$.
\end{lemma}
\begin{proof}
Let $\alpha$, $\beta\in\partial B_1$ denote the extreme points of 
$\gamma$. By the H\"older inequality
\[
|h(\alpha)-h(\beta)|=\left|\int_\gamma h^\prime d\HH^1\right|
\leq(\HH^1(\gamma))^{1-\frac 1q}\|h^\prime\|_{L^q(\gamma)}\,.
\]
Linearly interpolating $h$ on $\partial B_1\setminus\gamma$, 
we get an extension $\widetilde{h}\in W^{1,q}(\partial B_1)$ of $h$ satisfying the estimate
\begin{equation}\label{e:una}
\|\widetilde{h}^\prime\|^q_{L^q(\partial B_1\setminus \gamma)}
=(2\pi - \HH^1(\gamma))^{1-q}|h(\alpha)-h(\beta)|^q
\leq\left(\frac{\HH^1(\gamma)}{2\pi - \HH^1(\gamma)}\right)^{q-1}
\|h^\prime\|^q_{L^q(\gamma)}\,.
\end{equation}
In turn, if we set $\hat{h}:=\widetilde{h}-\fint_{\partial B_1} \widetilde{h}$, the 
Poincar\'e inequality and (\ref{e:una}) yield
\begin{equation}\label{e:due}
\|\hat{h}\|^q_{L^q(\partial B_1)}\leq C\|\widetilde{h}^\prime\|^q_{L^q(\partial B_1)}
\leq C\left(\frac{2\pi}{2\pi - \HH^1(\gamma)}\right)^{q-1}
\|h^\prime\|^q_{L^q(\gamma)}\,.
\end{equation}
The embedding $W^{1,q}(\partial B_1)\to H^{\sfrac12}(\partial B_1)$ 
provides us with a function $v\in W^{1,2}(B_1)$ with boundary trace $\hat{h}$ and
such that
\[
\|\nabla v\|_{L^2(B_1)}
\leq C\|\hat{h}\|_{H^{\sfrac12}(\partial B_1)}\leq
C\|\hat{h}\|_{W^{1,q}(\partial B_1)}\stackrel{(\ref{e:due})}{\leq}
\frac{C}{\left(2\pi - \HH^1(\gamma)\right)^{1-\frac 1q}}\,
\|h^\prime\|_{L^q(\gamma)}.
\]
By the latter inequality the function 
$w:=v+\fint_{\partial B_1}\widetilde{h}$ fulfills the assertions of the Lemma. 
\end{proof}

\section{Higher integrability of \texorpdfstring{$\nabla u$}{Du} via the porosity of \texorpdfstring{$K$}{K}}\label{ss:porosity}

A central role in the approach by De~Philippis and Figalli \cite{DePFig14} to establish 
the higher integrability of the gradient in any dimension is played by 
the ensuing improvement of Theorem~\ref{t:eps_salto_puro} due to David \cite{david1996}, and in higher dimension to Rigot \cite{Rig00}, and Maddalena and Solimini \cite{MadSol01}. Here, we provide a proof in the two-dimensional setting using again Theorem~\ref{t:minimizers compactness}.

\begin{theorem}\label{t:RMS}
There are constants $C_1,\,r_1,\,\e_1>0$ such that for every 
$\e\in(0,\e_1)$ there exists $t_\e\in(0,\frac 12)$ with the following property. If $(u,K)$ is either a restricted, or an absolute, or a generalized global minimizer of $E_\param(\cdot,\cdot,B_2,g)$, with $\lambda\leq 1$, $\|g\|_{\infty}\leq M_0$ (we use the notation introduced 
in  Assumption~\ref{a:blow-up}), then for every $x\in K\cap B_{\sfrac12}$ and for every $r\in(0,r_1)$ there exist $y\in K\cap B_r(x)$ and $\rho\in(t_\e r,r)$ such that $B_\rho(y)\subset B_r(x)$ and the assumptions of Theorem~\ref{t:eps_salto_puro} are satisfied in $B_\rho(y)$, namely for some $\theta\in[0,2\pi]$
\[
\Omega^j(\theta,y,\rho)+\lambda\|g\|_{\infty}^2 \rho^{\frac12}< \e_0
\]
($\e_0>0$ being the constant in Theorem~\ref{t:eps_salto_puro}).

In addition, it is true that
\begin{itemize}
 \item[(i)] $ K\cap B_\rho(y)$ is a $C^{1,\alpha}$ graph;
 \item[(ii)] and,
\begin{equation}\label{e:neureg}
 \rho\|\nabla u\|^2_{L^\infty(B_\rho(y))}\leq C_1\, \e.
\end{equation} 
\end{itemize}
\end{theorem}
\begin{proof}
Assume by contradiction that 
we can find sequences $\e_j\downarrow 0$, $r_j\downarrow 0$, $(u_j,K_j)$ of either restricted, or absolute, or generalized global minimizers of $E_{\param_j}(\cdot,\cdot,B_2,g_j)$, with $\param_j\leq 1$, $\|g_j\|_{\infty}\leq M_0$, 
points $x_j\in K_j\cap B_{\sfrac12}$, and
scalars $t_j\in(0,\frac12)$ such that for every 
$y\in K_j\cap B_{r_j}(x_j)$ and $\rho\in(t_jr_j,r_j)$ with $B_{\rho}(y)\subset B_{r_j}(x_j)$, then  
\begin{equation}\label{e:contra porosity}
    \rho^{-2}\dist_H^2 (K_j \cap B_{2\rho}(y), \mathcal{R}_{\theta} (\ello))\cap B_{2\rho}(y)) + \rho^{-1} \int_{B_{2\rho}(y)\setminus K_j} |\nabla u_j |^2+\lambda_j\|g_j\|_{\infty}^2 \rho^{\frac12}\geq \e_0\,
\end{equation}
for every $\theta\in[0,2\pi]$ (we are using the notation introduced at the beginning of Section~\ref{s:epsilon reg}).

Consider the rescalings $(\widetilde{u}_j,\widetilde{K}_j):=((u_j)_{x_j,r_j},(K_j)_{x_j,r_j})$, being the conditions 
in Assumption~\ref{a:blow-up} satisfied, by Theorem~\ref{t:minimizers compactness} $(\widetilde{u}_j,\widetilde{K}_j)$ converge up to a subsequence to 
a blow-up limit $(u_\infty,K_\infty)$ which is a global generalized (restricted) minimizer of $E_0$ (in case $(u_j,K_j)$ are restricted minimizers). 
Moreover, Corollary~\ref{c:a.e.regularity} yields that the subset of $K_\infty$ of pure jump points is relatively open and has full $\mathcal{H}^1$ measure. 
Let $z\in K_\infty$ be one such point, then necessarily 
Theorem~\ref{t:eps_salto_puro} implies that 
\[
 t^{-2}\dist_H^2 (K_\infty \cap B_{2t}(z), \mathcal{R}_{\theta_z} (\ello))\cap B_{2t}(z)) + 
 t^{-1} \int_{B_{2t}(z)\setminus K_\infty} 
 |\nabla u_\infty|^2< \e_0
 \]
for some $\theta_z\in[0,2\pi]$ and for every $t>0$ sufficiently small. 
In turn, Theorem~\ref{t:minimizers compactness} implies 
the existence of points $z_j\in K_j\cap B_{r_j}(x_j)$ such that
\begin{align*}
    (t r_j)^{-2}&\dist_H^2 (K_j \cap B_{2{t r_j}}(z_j), \mathcal{R}_{\theta_z} (\ello))\cap B_{2t r_j}(z_j)) \\ 
    &+ (t r_j)^{-1} 
    \int_{B_{2t r_j}(z_j)\setminus K_j} |\nabla u_j |^2+\lambda_j\|g_j\|_{\infty}^2 (t r_j)^{\frac12}< \e_0\,
\end{align*}
contradicting \eqref{e:contra porosity}.

Item (i) follows directly from Theorem~\ref{t:eps_salto_puro}; instead estimate \eqref{e:neureg} is a consequence of item (i) and classical elliptic regularity results for the Neumann problem since $u$ satisfies \eqref{e:Euler Neumann g} and \eqref{e:Euler curvature g} (cf. \cite[Theorem 7.53]{AFP00}).
\end{proof}

Theorem~\ref{t:RMS} can be rephrased by asserting that the set $K\setminus K^{(j)}$ is $(t_\e,1)$-\emph{porous} in $K$ (cf. for instance \cite{mattila}).
\index{porosity @porosity}
Hence, following the papers by David \cite{david1996}, Rigot \cite{Rig00} and Maddalena and Solimini \cite{MadSol01}, one can estimate the Minkowski dimension
\index{Minkowski dimension @Minkowski dimension}
(and thus the Hausdorff dimension\index{Hausdorff dimension@Hausdorff dimension}) of the set $K\setminus K^{(i)}$ using
Theorem~\ref{t:RMS}, a by now classical result by David and Semmes~\cite{DavSem97}, and the Alfohrs regularity of $K$
(cf. Lemma~\ref{l:upper-bound} and Theorem~\ref{t:dlb}), which we restate for convenience: if $(u,K)$ is either an absolute or a restricted, or a generalized minimizer of $E_\param$ on $B_2$ then
\begin{equation}\label{e:SuAR}
\epsilon\,r \leq \HH^1( K\cap B_r(x)) \leq \e_\param r
\end{equation}
for some constant $\epsilon>0$, for $\e_\param:=(2  + \lambda \|g\|_\infty^2 )\pi$, for every $x\in  K$, and for every $r\in(0,1\wedge \mathrm{dist}(x,\partial\Omega))$.

Instead, in Proposition~\ref{p:alfpor} below, an estimate of the Minkowski dimension will be obtained directly. The key technical results to prove the higher integrability of $\nabla u$ are contained in Proposition~\ref{p:alfpor} for which we need two preparatory lemmas. 
The first one provides the selection of suitable families of radii obtained via the De~Giorgi's slicing/averaging principle.
\begin{lemma}\label{l:technical}
 There are positive constants $M_1$, $C_1$ such that if $M\geq M_1$ for every $(u,K)$ (absolute, restricted, generalized, or generalized restricted) minimizer of $E_\param$ on $B_2$ we can find three decreasing sequences of radii such that 
 \begin{itemize}
  \item[(i)] $1\geq R_h> S_h> T_h> R_{h+1}$;
  \item[(ii)] $8M^{-(h+1)}\leq R_h-R_{h+1}\leq M^{-\frac{(h+1)}2}$, and $S_h-T_h=T_h-R_{h+1}=4M^{-(h+1)}$;
  \item[(iii)] $\HH^1( K\cap( \overline{B}_{S_h}\setminus \overline{B}_{R_{h+1}}))\leq C_1M^{-\frac{(h+1)}2}$;
  \item[(iv)] $R_\infty=S_\infty=T_\infty\geq \sfrac12$.
  \end{itemize}
\end{lemma}
\begin{proof}
Let $R_1=1$, given $R_h$ we construct $S_h$, $T_h$ and $R_{h+1}$ as follows.

Set $N_h:=\lfloor \frac{M^{\frac{h+1}2}}8\rfloor\in\N$ and fix $M_1\in\N$ such that 
$N_h\geq \lfloor \frac{M^{\frac{h+1}2}}{16}\rfloor$ for $M\geq M_1$. Here, 
$\lfloor t\rfloor$ denotes the integer part of $t\in\R$. 

The annulus $B_{R_h}\setminus  \overline{B}_{R_h-8M^{-\frac{h+1}2}}$ contains the $N_h$ disjoint sub annuli 
$ \overline{B}_{R_h-8(i-1)M^{-(h+1)}}\setminus  \overline{B}_{R_h-8iM^{-(h+1)}}$, $i\in\{1,\ldots,N_h\}$, of equal 
width $8\,M^{-(h+1)}$. By averaging we can find an index $i_h\in\{1,\ldots,N_h\}$ such that
\begin{multline*}
\HH^1\big(K\cap( \overline{B}_{R_h-8(i_h-1)M^{-\frac{h+1}2}}\setminus \overline{B}_{R_h-8i_h\,M^{-\frac{h+1}2}})\big)\\
\leq \frac 1{N_h}\,\HH^1\big(K\cap ( \overline{B}_{R_h}\setminus  \overline{B}_{R_h-8M^{-\frac{h+1}2}})\big)
\stackrel{\eqref{e:SuAR}}{\leq} 
\varepsilon_\lambda\,\frac{R_h}{N_h}
\leq 8 \varepsilon_\lambda\,M^{-\frac{h+1}2},
\end{multline*}
so that (iii) is established with $C_1:=8\varepsilon_\lambda$. Finally, set 
\[
 S_h:=R_h-8(i_h-1)M^{-(h+1)},\,\,R_{h+1}:=R_h-8i_h\,M^{-(h+1)},\,\,T_h:=\frac12(S_h+R_{h+1}),
\]
then items (i) and (ii) follow by the very definition, and item (iv) from (ii) if 
$M_1$ is sufficiently big.
\end{proof}
The second lemma has a geometric flavor.
\begin{lemma}\label{l:grafico}
Let $f:\R^{n-1}\to\R$ be Lipschitz with
\begin{equation}\label{e:lip}
f(0)=0,\quad\|\nabla f\|_{L^\infty(\R^{n-1},\R^{n-1})}\leq\eta.
\end{equation}
If $G:=\mathrm{graph}(f)\cap B_{2}$ and $\eta\in(0,\sfrac{1}{15}]$, then 
for all $\delta\in(0,\sfrac12)$ and $x\in (\overline{B}_{1+\delta}\setminus B_{1})\cap G$
\[
\mathrm{dist}(x,(\overline{B}_{1+2\delta}\setminus B_{1+\delta})\cap G)\leq \frac 32 \delta.
\]
\end{lemma}
\begin{proof}
Clearly by \eqref{e:lip} we get
\[
 \|f\|_{W^{1,\infty}(B_2)}\leq 3\eta.
\]
Let $x=(y,f(y))\in (\overline{B}_{1+\delta}\setminus B_{1})\cap G$ and $\hat{x}:=(\lambda\,y,f(\lambda\,y))$,
with $\lambda$ to be chosen suitably. Note that as $|x|\geq 1$ we have
\begin{align*}
 |f(\lambda\,y)-\lambda\,f(y)|&\leq |f(\lambda\,y)-f(y)|+|\lambda-1|\,|f(y)|\\
 &\leq |\lambda-1|\,(\|\nabla f\|_{L^\infty(\R^{n-1},\R^{n-1})}\,|y|+\|f\|_{L^\infty(\R^{n-1})})\leq 3\eta|\lambda-1|\,|x|.
\end{align*}
Hence,
\[
 |\hat{x}-x|\leq |\hat{x}-\lambda\,x|+|\lambda-1||x| \leq(3\eta+1)|\lambda-1|\,|x|.
\]
It is easy to check that the choice $\lambda=1+\frac{5}{4}\delta|x|^{-1}$ gives the conclusion.
\end{proof}

Following De~Philippis and Figalli \cite{DePFig14}, we prove next a version of 
the mentioned porosity result by David and Semmes~\cite{DavSem97} that is suitable for our purposes.
In what follows, $(E)_r$ denotes the open $r$-neighborhood of a given 
set $E$, i.e. 
$(E)_r=\{x\in\R^2:\,\dist(x,E)< r\}$.
\index[simb]{aalEr@$(E)_r$}
\begin{proposition}\label{p:alfpor}
There exist constants $C_2,\, M_2>0$ and $\alpha,\beta\in(0,\sfrac 14)$ 
such that for every $M\geq M_2$, for every $(u,K)$ (absolute, restricted, generalized, or generalized restricted) minimizer of $E_\param$ on $B_2$, 
we can find families $ \mathcal{F}_j$ of disjoint balls
 \[
  \mathcal{F}_j=\left\{B_{\alpha M^{-j}}(y_i):\,y_i\in  K,\,1\leq i\leq N_j\right\}
 \]
such that for all $h\in\N$ 
\begin{itemize}
 \item[(i)] if $B$, $B^\prime\in\cup_{j=1}^h \mathcal{F}_j$ are distinct balls, then 
 $(B)_{4M^{-(h+1)}}\cap(B^\prime)_{4M^{-(h+1)}}=\emptyset$;
 \item[(ii)] if $B_{\alpha M^{-j}}(y_i)\in \mathcal{F}_j$, then 
 $ K\cap B_{2\alpha M^{-j}}(y_i)$ is a $C^{1,\gamma}$ graph, $\gamma\in(0,1)$, 
 containing $y_i$,
 \[
\inf_\theta\Omega^j(\theta,y_i,2\alpha M^{-j})
+\lambda\|g\|^2_\infty(\alpha M^{-j})^{\frac12} 
< \e_0\,,
 \]
\begin{equation}\label{e:gradbdd}
\|\nabla u\|_{L^\infty(B_{2\alpha M^{-j}}(y_i))}<M^{j+1}; \end{equation}
 \item[(iii)] let $\{R_h\}_{h\in\N}$, $\{S_h\}_{h\in\N}$, $\{T_h\}_{h\in\N}$ be the sequences of radii in 
 Lemma~\ref{l:technical}, and let
 \[
  K_h:=( K\cap\overline{B}_{S_h})\setminus\left(\cup_{j=1}^h\cup_{ \mathcal{F}_j}B\right),
 \]
 (by construction $K_{h+1}\subset K_h\setminus \cup_{ \mathcal{F}_{h+1}}B$), and 
 \begin{equation}\label{e:widetilde K h}
 \widetilde{K}_h:= ( K\cap\overline{B}_{T_h})\setminus
\left(\cup_{j=1}^h\cup_{ \mathcal{F}_j}(B)_{2M^{-(h+1)}}\right)\subset K_h.
 \end{equation}
 Then, there exists a finite set of points $ \mathcal{C}_h:=\{x_i\}_{i\in I_h}\subset\widetilde{K}_h$ such that
 \begin{equation}\label{e:centres}
  |x_j-x_k|\geq 3M^{-(h+1)}\quad\forall j,k\in I_h,\,j\neq k;
 \end{equation}
 \begin{equation}\label{e:ingro}
  (K_h\cap  \overline{B}_{R_{h+1}})_{M^{-(h+1)}}\subset 
	\cup_{i\in I_h}B_{8M^{-(h+1)}}(x_i);
 \end{equation}
 \begin{equation}\label{e:stimaKh}
  \HH^1(K_h)\leq C_1\,h\,M^{-2h\beta};
 \end{equation}
 \begin{equation}\label{e:measureingro}
 \left|(K_h\cap  \overline{B}_{R_{h+1}})_{M^{-(h+1)}}\right|\leq C_2\,h\,M^{-h(1+2\beta)-1}.
 \end{equation}
 
 \item[(iv)] $\Sigma\cap B_{\sfrac12}\subset K_h$ for all $h\in\N$ and 
 \begin{equation}\label{e:sigmaur}
\left|(\Sigma\cap B_{\sfrac12})_{r}\right|\leq C_2\,r^{1+\beta}\qquad\forall r\in(0,\sfrac 12].
 \end{equation}
In particular, 
$\dim_{\MM}(\Sigma\cap B_{\sfrac12})\leq 1-\beta$.
 \end{itemize}
\end{proposition}
\begin{proof}
In what follows we shall repeatedly use Theorem~\ref{t:RMS} 
with $ \e\in(0,1)$ fixed and sufficiently small. 
We split the proof into several steps.

In addition, consider the constants $\epsilon,\,C_1,\,M_1$ introduced in \eqref{e:SuAR} and 
Lemma~\ref{l:technical}, respectively.

\smallskip

\noindent{\bf Step 1.} \emph{Inductive definition of the families $ \mathcal{F}_h$.}

For $h=1$ we define 
 \[
 \mathcal{F}_1:=\emptyset,\,\, K_1=K\cap \overline{B}_{S_1},\,\,\widetilde{K}_1=K\cap \overline{B}_{T_1}, 
\] 
and choose $ \mathcal{C}_1$ to be a maximal family of points in $\widetilde{K}_1$ 
at distance $3M^{-2}$ from each other. 
Of course, properties in items (i) and (ii) and in \eqref{e:centres} are satisfied. 
To check the others, one can argue as in the verification below.
 
 Suppose that we have built the families $\{ \mathcal{F}_j\}_{j=1}^h$ as in the statement, to construct $\mathcal{F}_{h+1}$ we argue as follows. 
 Let $\mathcal{C}_h=\{x_i\}_{i\in I_h}\subset\widetilde{K}_h$ be a maximal family of points satisfying \eqref{e:centres}, i.e.~$|x_i-x_k|\geq 3M^{-(h+1)}$ for all $j,\,k\in I_h$ 
 with $j\neq k$. Consider 
 \[
 \mathcal{G}_{h+1}:=\{B_{M^{-(h+1)}}(x_i)\}_{i\in I_h}.
 \]
For every ball $B_{M^{-(h+1)}}(x_i)\in \mathcal{G}_{h+1}$ we can find a sub-ball $B_{2\alpha\,M^{-(h+1)}}(y_i)\subset B_{M^{-(h+1)}}(x_i)$,
 $\alpha\in(0,\sfrac 14)$ for which the theses of Theorem~\ref{t:RMS} are satisfied. 
 Then, define for $M\geq M_1$
 \[
  \mathcal{F}_{h+1}:=\{B_{\alpha\,M^{-(h+1)}}(y_i)\}_{i\in I_h}.
 \]
 By condition \eqref{e:centres}, the balls $B_{\frac32 M^{-(h+1)}}(x_i)$ are disjoint and
 do not intersect
 \[
  \cup_{j=1}^h\cup_{ \mathcal{F}_j}(B)_{\frac12M^{-(h+1)}}
 \]
 by the very definition of $\widetilde{K}_h$ in \eqref{e:widetilde K h}. 
 Thus, item (i) and (ii) are satisfied for $M_2$ sufficiently large. 
  Hence, we can define $K_{h+1}$, $\widetilde{K}_{h+1}$ and $ \mathcal{C}_{h+1}$ as in the statement. 
  
  We verify next the rest of the conclusions.  
\smallskip

 \noindent{{\bf Step 2.}} \emph{Proof of \eqref{e:ingro}.}

Let $x\in (K_{h+1}\cap  \overline{B}_{R_{h+2}})_{M^{-(h+2)}}$ and let $z$ be a point of minimal distance from $K_{h+1}\cap  \overline{B}_{R_{h+2}}$. 
In case $z\in \widetilde{K}_{h+1}$, by maximality of $\mathcal{C}_{h+1}$ there is some point $x_i\in  \mathcal{C}_{h+1}$ such that 
$|z-x_i|< 3M^{-(h+2)}$ and thus we conclude $x\in B_{4M^{-(h+2)}}(x_i)$.
Instead, if $z\in(K_{h+1}\cap  \overline{B}_{R_{h+2}})\setminus\widetilde{K}_{h+1}$, the definitions of ${K_{h+1}}$ and $\widetilde{K}_{h+1}$ yield the existence of a ball $\widetilde{B}\in\cup_{j=1}^{h+1} \mathcal{F}_j$ 
for which $z\in (K\cap(\widetilde{B})_{2M^{-(h+2)}})\setminus \widetilde{B}$. 
Item (ii) and a scaled version of Lemma~\ref{l:grafico} 
(applied with $\delta=2M^{-(h+2)}$, and radii $\rho$ (that of $\widetilde{B}$) and $\rho+2M^{-(h+2)}$ in place of $1$ and $1+\delta$) give a point $y$ satisfying
\[
y\in \left(K\cap(\widetilde{B})_{4M^{-(h+2)}}\right)\setminus (\widetilde{B})_{2M^{-(h+2)}},\quad
\text{and}\quad |z-y|\leq 3M^{-(h+2)}.
\]
Therefore, as $z\in \overline{B}_{R_{h+2}}$ and $T_{h+1}=R_{h+2}+4M^{-(h+2)}$ 
we get by property (i) and the definition of $\widetilde{K}_{h+1}$ 
\[
y\in \left(K\cap B_{T_{h+1}}\cap(\widetilde{B})_{3M^{-(h+2)}}\right)
\setminus (\widetilde{B})_{2M^{-(h+2)}}\subset \widetilde{K}_{h+1}.
\]
Finally, by maximality of $\mathcal{C}_{h+1}$ we may find $x_i\in \mathcal{C}_{h+1}$ 
such that $|y-x_i|< 3M^{-(h+2)}$. In conclusion, we have
\[
|x-x_i|\leq|x-z|+|z-y|+|y-x_i|\leq 7M^{-(h+2)},
\]
so that  \eqref{e:ingro} follows at once.
\smallskip

\noindent{{\bf Step 3.}} \emph{For every $h\in\N$ and for every $x_i\in  \mathcal{C}_{h}$
\begin{equation}\label{e:dlbKh}
 K_{h}\cap B_{M^{-(h+1)}}(x_i)=K\cap B_{M^{-(h+1)}}(x_i)\,.
\end{equation}}
Indeed, assume by contradiction that we can find $x_i\in  \mathcal{C}_{h}$ and $x\in (K\setminus K_{h})\cap B_{M^{-(h+1)}}(x_i)$. 
As $x_i\in\widetilde{K}_{h}$ then $x_i\in B_{T_{h}}$, in turn implying 
$x\in B_{S_{h}}$ since $S_h-T_h=4M^{-(h+1)}$.
Therefore $x\in(K\setminus K_{h})\cap \overline{B}_{S_{h}}$, and by definition of 
$K_{h}$ we can find a ball $B\in  \mathcal{F}_{j}$, $j\leq h$, such that $x\in B$. 
We conclude that
\[
\textrm{dist}(x_i,B)\leq|x-x_i|\leq M^{-(h+1)},
\] 
contradicting the assumption that $x_i\in\widetilde{K}_{h}$.
\smallskip

\noindent{{\bf Step 4.}} \emph{Proof of \eqref{e:stimaKh}.}

We get first a lower bound for $\mathcal{H}^0(I_h)$: use \eqref{e:ingro} and the density upper bound in \eqref{e:SuAR} to get
 \[
  \HH^1(K_h\cap  \overline{B}_{R_{h+1}})=
  \HH^1(K_h\cap  \overline{B}_{R_{h+1}}
  \cap\cup_{i\in I_h}B_{8M^{-(h+1)}}(x_i))
  \leq \e_\param 8M^{-(h+1)}\,\mathcal{H}^0(I_h)
 \]
where we recall that $\e_\param=(2+\lambda \|g\|_\infty^2)\pi$.
Equivalently, we have that 
 \begin{equation}\label{e:Ihbound}
 \mathcal{H}^0(I_h)M^{-(h+1)}\geq \frac1{8\e_\param} 
 \HH^1(K_h\cap  \overline{B}_{R_{h+1}}).
 \end{equation}
 Thus, we estimate as follows by taking into account that by item (i) the balls in the family $\mathcal{F}_{h+1}$ are disjoint 
 \begin{align}\label{e:measure}
  \HH^1(K_{h+1})&\leq\HH^1\left(K_h\setminus\cup_{ \mathcal{F}_{h+1}}B\right)
  =\HH^1(K_h)-\sum_{ \mathcal{F}_{h+1}}\HH^1(K_h\cap B)\notag\\
 &\stackrel{\eqref{e:SuAR},\,\eqref{e:dlbKh}}{\leq}
	\HH^1(K_h)-\epsilon \alpha M^{-(h+1)}\mathcal{H}^0(I_h) 
  \stackrel{\eqref{e:Ihbound}}{\leq}\HH^1(K_h)-
  \frac\epsilon{8\e_\param}\alpha
  \HH^1(K_h\cap  \overline{B}_{R_{h+1}})\notag\\
  &=(1-\eta)\HH^1(K_h)+\eta\left(\HH^1(K_h)-\HH^1(K_h\cap  \overline{B}_{R_{h+1}})\right)\notag\\
  &\leq(1-\eta)\HH^1(K_h)+\eta\,\HH^1(K\cap(\overline{B}_{S_h} \setminus  \overline{B}_{R_{h+1}}))\leq(1-\eta)\HH^1(K_h)+C_1M^{-\frac{h+1}2},
 \end{align}
where we have set $\eta:=\frac\epsilon{8\e_\param}\alpha\in(0,1)$, and we have used the very definition of $K_h$ in the last but one inequality, and item of (iii) Lemma~\ref{l:technical} in the last one.
By taking into account (iii) in Lemma~\ref{l:technical}
and \eqref{e:Ihbound}, an iteration of \eqref{e:measure} implies
\[
 \HH^1(K_h)\leq C_1\sum_{i=1}^h(1-\eta)^{h-i}M^{-\sfrac i2}
 \leq C_1\,h\,(\max\{1-\eta,M^{-\sfrac 12}\})^h.
\]
Choose $\beta\in(0,\sfrac 14)$ such that $(1-\eta)\leq M^{-2\beta}$, 
the previous estimate then yields \eqref{e:stimaKh}, 
\begin{equation*}
 \HH^1(K_h)\leq C_1\,h\,\max\{M^{-2h\beta},M^{-\sfrac h2}\}=C_1\,h\,M^{-2h\beta}.
\end{equation*}

\smallskip

 \noindent{{\bf Step 5.}} \emph{Proof of \eqref{e:measureingro}.}

We use \eqref{e:ingro} to get for some dimensional constant $C$
 \begin{align*}
\left|(K_{h+1}\cap  \overline{B}_{R_{h+2}})_{M^{-(h+2)}}\right| &\leq 
 \left|\cup_{i\in I_{h+1}}B_{8M^{-(h+2)}}(x_i)\right|\leq 
	C M^{-2(h+2)}\mathcal{H}^0(I_{h+1})\notag\\
  &\stackrel{\eqref{e:SuAR},\, \eqref{e:dlbKh}}{\leq} C\epsilon M^{-(h+2)}\sum_{i\in 
	I_{h+1}}\HH^1(K_{h+1}\cap B_{M^{-(h+2)}}(x_i))\notag\\
  &\leq C\epsilon M^{-(h+2)}\HH^1(K_{h+1})
  \stackrel{\eqref{e:stimaKh}}{\leq} C\epsilon C_1\,(h+1)\,M^{-2(h+1)\beta-(h+2)}.
  \end{align*}
where in the last but one inequality we have used that the balls $B_{M^{-(h+2)}}(x_i)$ 
are disjoint by construction.
\smallskip

\noindent{{\bf Step 6.}} \emph{Proof of \eqref{e:sigmaur}}
  
By construction we have that $\Sigma\cap B_{\sfrac 12}\subseteq K_h$. 
Therefore, 
\eqref{e:measureingro} gives as $R_h\geq R_\infty\geq\sfrac 12$
\begin{equation*}
\left|(\Sigma\cap B_{\sfrac12})_{M^{-(h+1)}}\right|\leq
\left|(K_{h}\cap\overline{B}_{R_{h+1}})_{M^{-(h+1)}}\right|
\leq C_2\,h\,M^{-h(1+2\beta)-1}. 
\end{equation*}
Hence, if $r\in (M^{-(h+2)},M^{-(h+1)}]$ we get
\[
\left|(\Sigma\cap B_{\sfrac12})_{r}\right|\leq 
 C_2\,h\,M^{-h(1+2\beta)-1}\leq  C_2\,M^{-h(1+\beta)-1}\leq C_2\,r^{1+\beta}.\qedhere
\]
\end{proof}

We are now ready to establish the higher integrability of 
the gradient.
\begin{theorem}\label{t:hiDePF}

There is $p>2$ such that for all open sets $\Omega\subseteq\R^2$ and for all 
$(u,K)$ either a restricted, or an absolute, or a generalized global minimizer of $E_\param(\cdot,\cdot,\Omega,g)$ then $\nabla u\in L^p_{\mathrm{loc}}(\Omega\setminus K)$ .
\end{theorem}
\begin{proof} 
Clearly, it is sufficient for our purposes to prove a localized estimate. Hence,
for the sake of simplicity, we suppose that $ \Omega=B_2$. 

We keep the notation of Proposition~\ref{p:alfpor} and furthermore denote for all $h\in\N$ 
 \begin{equation}\label{e:Ah}
  A_h:=\{x\in B_2\setminus K:\,|\nabla u(x)|^2>M^{h+1}\}\,.
 \end{equation}
We claim that 
\begin{equation}\label{e:claimAh}
A_{h+2}\cap B_{R_{h+2}}\subset (K_h\cap B_{R_{h+1}})_{M^{-(h+1)}}.
\end{equation}
Given this, for granted we conclude as follows: we use \eqref{e:measureingro} to deduce that 
\begin{equation}\label{e:estmeas}
|A_{h+2}\cap B_{R_{h+2}}|\leq \left|(K_h\cap B_{R_{h+1}})_{M^{-(h+1)}}\right|\leq
C_2\,h\,M^{-h(1+2\beta)-1}.
\end{equation}
Therefore, recalling that $\frac 12\leq R_\infty\leq R_h$, in view of \eqref{e:estmeas} 
and Cavalieri's formula for $q>1$ we get that 
\begin{align*}
 \int_{B_{\frac12}}|\nabla u|^{2q}&=q\int_0^{\infty}t^{q-1}
|\{x\in B_{\sfrac12}\setminus K:\,|\nabla u(x)|^2>t\}|\, dt\\
&\leq q\sum_{h\geq 3}\int_{M^h}^{M^{h+1}}t^{q-1}
|\{x\in B_{\sfrac12}\setminus K:\,|\nabla u(x)|^2>t\}|\, dt+M^{3q} |B_{\sfrac12}|\\
&\leq \sum_{h\geq 0}M^{(h+4)q} |A_{h+2}\cap B_{\frac 12}|
+M^{3q} |B_{\frac12}|\\
&\leq C_2\,\sum_{h\geq 0}h\,M^{(h+4)q-h(1+2\beta)-1}+M^{3q} |B_{\frac12}|.
 \end{align*}
 The conclusion follows at once by taking $q\in(1,1+2\beta)$ and $p=2q$.
\smallskip

 To conclude we prove formula \eqref{e:claimAh} in two steps. 
 
 \noindent{{\bf Step 1.}} \emph{For every $M> 0$ large enough, for every $h\in\N$ and for every $R\in(0,1]$ we have that}
 \begin{equation}\label{e:prelimAh}
  A_h\cap B_{R-2M^{-h}}\subset(K\cap {B_R})_{M^{-h}}\,. 
 \end{equation}
Indeed, for $x\in A_h\cap B_{R-2M^{-h}}$ let  $z\in K$ be such that $\mathrm{dist}(x,K)=|x-z|$.
If $|x-z|>M^{-h}$ then $B_{M^{-h}}(x)\cap K=\emptyset$ so that $u$ 
solves \eqref{e:Euler harmonic g}, i.e. $\triangle u=\lambda(u-g)$, on $B_{M^{-h}}(x)$.
Therefore, being the right-hand side of the PDE in $L^\infty$, by standard Lipschitz bounds in elliptic regularity \cite[estimate (4.45)]{GT} and the density upper bound in \eqref{e:SuAR}, as $x\in A_h$ we infer that
\[
 M^{h+1}\leq|\nabla u(x)|^2\leq 
 C\,M^h
\]
with $C$ depending on $\|g\|_{\infty}$. The latter estimate is clearly impossible for $M$ large enough. 
Furthermore, as $x\in B_{R-2M^{-h}}$ and $|x-z|\leq M^{-h}$ we conclude that $z\in B_R$.
\smallskip

\noindent{{\bf Step 2.}} \emph{Proof of \eqref{e:claimAh}.}

Since $R_{h+1}-R_{h+2}\geq 8M^{-(h+2)}$ (cf. (i) Lemma~\ref{l:technical}), we apply Step~1 to $A_{h+2}$ 
and $R=R_{h+1}$ and then \eqref{e:prelimAh} implies that 
\[
A_{h+2}\cap B_{R_{h+2}}\subset\big(K\cap B_{R_{h+1}}\big)_{M^{-(h+1)}}.
\]
Let $x\in A_{h+2}\cap B_{R_{h+2}}$, $z\in K\cap B_{R_{h+1}}$ be a point of minimal distance, and suppose that $z\in K\setminus K_h$. 
Being $R_{h+1}\leq S_h$, by the very definition of $K_h$ there is a ball 
$B\in\cup_{j=1}^h \mathcal{F}_j$ such that $z\in B$. In turn, since $B=B_\rho(y)$ for some $y$ and radius $\rho\geq t\,M^{-h}$, then $x\in B_{2t}(y)$ as $|x-z|\leq M^{-(h+1)}$ for $M$ sufficiently large. 
Thus, estimate $|\nabla u(x)|^2<M^{h+1}$ follows from \eqref{e:gradbdd} in item (ii) of Proposition~\ref{p:alfpor}. 
This is in contradiction with $x\in A_{h+2}$.
\end{proof}

\section{Proof of Corollary \ref{c:Bonnet-plus-2}}\label{s:Bonnet-plus}

Consider a pair $(u,K)$ which is a restricted minimizer in some open set $\Omega$. We wish to show that at every point $x_0\in K$ any possible blow up $(v, J)$ is either an elementary minimizer, or a cracktip. Fix such an $x_0$ and, without loss of generality, assume that it is the origin. We select the connected component $K^0$ of $K$ which contains $0$ and observe that, since $K$ consists of finitely many connected components, we find $r_0$ such that $B_{r_0} \cap K = B_{r_0}\cap K^0$. By the monotonicity formula of Proposition \ref{p:monotonia-0}, $r\mapsto \frac{1}{r} \int_{B_r\setminus K} |\nabla u|^2$ is then monotone nondecreasing for $r\in (0, r_0]$. In particular, the function has a limit as $r\downarrow 0$, which in turn is the constant value of the function $r\mapsto \frac{1}{r} \int_{B_r\setminus J} |\nabla v|^2$. If we can argue that $J$ is itself connected, we can then apply Proposition \ref{p:monotonia-0} again and conclude, from its analysis of the equality case, that either $\nabla v =0$ (and hence $(v, J)$ is an elementary minimizer) or $(v,J)$ is a cracktip. 

We thus focus on proving the connectedness of $J$, which is the limit, in the local Hausdorff topology, of the rescaled sets $(K^0)_{0,r_j}$ for some sequence $r_j\downarrow 0$. The argument follows closely that of \cite[Lemma 2.8]{B96}.

\medskip

Recall the classical fact that closed connected sets with finite $\mathcal{H}^1$ measure are arcwise connected (see Lemma \ref{l:connessi_per_archi}). In particular, for every pair of points $x,y\in K^0$ we can introduce the intrinsic distance $d (x,y)$, defined as the length of a shortest path in $K^0$ connecting $x$ and $y$. We wish to show the following chord-arc property:
\begin{itemize}
\item[(Ch)] There are $\bar{r}>0$ and $C>0$ such that $d(x,0)\leq C |x|$ for every $x\in K^0\cap B_{\bar{r}}$.  
\end{itemize}
With the latter property at hand we conclude immediately that $J$ is connected: if we fix any $y\in J\setminus \{0\}$, we know it is the limit of points $y_j \in (K^0)_{0, r_j}$. It turn this means that there are $\bar y_j \in K^0$ such that $y_j = \frac{\bar y_j}{r_j}$. In particular $|\bar y_j|\leq 2|y| r_j$ for sufficiently large $j$. We then find arcs $\bar \gamma_j\subset K^0$ connecting $0$ and $\bar y_j$, with length bounded by $2C|y| r_j$. In turn this means that we find arcs $\gamma_j\subset (K^0)_{0, r_j}$ connecting $0$ and $y_j$ of length bounded by $2C|y|$. In particular, by extraction of a converging subsequence, we find an arc $\gamma\subset J$ connecting $0$ and $y$.

To prove (Ch) we argue by contradiction and assume that there is a sequence of points $\{y_j\}\subset K^0$ with $y_j\to 0$ and such that $|y_j|^{-1} d (0, y_j) \to \infty$. We then distinguish two cases, to which we can restrict after extraction of a subsequence: \begin{itemize}
    \item[(a)] $d (0, y_j)\downarrow 0$
    \item[(b)] $\liminf_j d(0, y_j) > 0$.
\end{itemize}

\medskip

{\bf Case (a).} We let $\gamma_j\subset K^0$ be an arc connecting $0$ and $y_j$ with minimal length. Because of the minimality, $\gamma_j$ is an injective arc. Let $z_j\in \gamma_j$ be a point which is furthest away from the origin and note that it therefore subdvides $\gamma_j$ in two arcs $\gamma^1_j$ and $\gamma^2_j$, one connecting $0$ and $z_j$, and the other connecting $z_j$ and $y_j$, with no other point in common except for $z_j$. Since 
$|y_j|^{-1} d (0, y_j) \to \infty$ and 
$\mathcal{H}^1 (B_r \cap \gamma_j) \leq \mathcal{H}^1 (B_r \cap K)$ for $r>0$ sufficiently small, by the energy upper bound $\frac{|z_j|}{|y_j|}$ must converge to infinity. Let $\rho_j := |z_j|$ and consider the rescaled maps $(u_{0, \rho_j}, K_{0, \rho_j})$, the rescaled points $\bar z_j = \frac{z_j}{\rho_j}$ and $\bar y_j := \frac{y_j}{\rho_j}$, and the rescaled arcs $\bar{\gamma}^1_j$ and $\bar{\gamma}^2_j$. Without loss of generality, we can extract a subsequence and assume that 
\begin{itemize}
\item $(u_{0, \rho_j}, K_{0, \rho_j})$ converges to a global generalized restricted minimizer $(\bar u, \bar K)$;
\item $\bar{\gamma}^1_j$ and $\bar{\gamma}^2_j$ converge to two arcs, $\bar\gamma^1$ and $\bar\gamma^2$, included in $\bar K$ and connecting $0$ with some point $\bar{z}\in \partial B_1$.
\end{itemize}
There are then two possibilities: 
\begin{itemize}
\item[(i)] $\bar K$ contains a Jordan curve. The latter would however bound a ``pocket'', i.e. a bounded connected component of $\mathbb R^2\setminus \bar K$. One could then argue as in Lemma~\ref{l:pockets} to get a contradiction. Note that the argument in there, which is written for an absolute minimizer, is, strictly speaking, not applicable to $(\bar u, \bar K)$ which is a restricted minimizer: however it is easy to see that the idea leads to the construction of a competitor which does not increase the number of connected components of $\bar K$.
\item[(ii)] $\bar{\gamma}^1$ and $\bar{\gamma}^2$ are the same arc $\bar\gamma$. In this case we know that $\bar\gamma$ must contain a regular jump point $y'$ with $0<|y'|<1$. In a sufficiently small ball $B_\tau (y')$ the set $\bar K$ is then a single smooth arc. But then, by case (a) of Theorem~\ref{t:main}, for a sufficiently large $j$ and a sufficiently small $\sigma$, $K_{0, \rho_j}\cap B_{\sigma} (y')$ is also a single smooth arc. However, this is not possible since, for a sufficiently large $j$, the two arcs $\bar\gamma^1_j$ and $\bar\gamma^2_j$ both intersect $B_{\sigma} (y')$. 
\end{itemize}

\medskip

{\bf Case (b)} Again we let $\gamma_j\subset K^0$ be an arc connecting $0$ and $y_j$ with minimal length, concluding that $\gamma_j$ is an injective arc. Parametrize it by arclength so that $\gamma_j (0) = y_j$, fix an arbitrary positive $\varepsilon$ and focus on the arc  $\gamma_j ([0, \varepsilon])$. $\gamma_j$ converges to a map $\gamma : [0, \varepsilon] \to K$ with $\gamma (0)=0$, which however cannot be the constant map, because otherwise $\gamma_j ([0, \varepsilon]) \subset B_{\delta_j}$ for some sequence $\delta_j\downarrow 0$, contradicting the energy upper bound. Hence $\mathcal{H}^1 (\gamma ([0, \varepsilon]))>0$ and so there must be a regular jump point $y'\in \gamma ([0, \varepsilon])$ with the property that, in a sufficiently small ball $B_\sigma (y')$, $K\cap B_\sigma (y')$ is a smooth arc and $K\cap B_\sigma (y')\subset \gamma ([0, \varepsilon])$. But then for $j$ sufficiently large we must necessarily have that $K\cap B_{\sigma/2} (y') \subset \gamma_j ([0, \varepsilon])$. In other words, the arcs $\gamma ([0, \varepsilon])$  and $\gamma_j ([0, \varepsilon])$ share a point, for $j$ sufficiently large. But then there is an arc in $K$ connecting $0$ and $y_j$ and whose length is at most $2\varepsilon$. Since $\varepsilon$ is arbitrary and $\liminf_j d (0, y_j)>0$, this would contradict the minimality of $\gamma_j$.

\appendix

\chapter{Variational identities}\label{a:variational identities}
We give here the proofs of Proposition~\ref{p:variational identities} and Proposition~\ref{p:variational identities 2}. 
We recall that $\nu$ denotes the counterclockwise rotation by $90$ degrees of a $C^0$ unit tangent vector $e$ locally orienting $K$, while $\kappa$ is the curvature of the curve, namely $\ddot\gamma \cdot \nu$, for an arclength parametrization $\gamma$ such that $\dot\gamma = e$. 
Finally, $w^+$ and $w^-$ are the one-sided traces of the relevant function $w$ on $K$ (following the obvious convention that $w^+$ is the trace on the side which $\nu$ is pointing to).

\begin{proof}[Proof of Proposition~\ref{p:variational identities}]
The proofs of the outer variations formula \eqref{e:outer} is standard, 
and we leave it to the reader.

Let $\psi \in C^1_c (\Omega, \mathbb R^2)$, and consider the 
Cauchy problems (parametrized in terms of the initial condition $x\in\Omega$)
\[
\begin{cases}
\displaystyle{\frac{\partial}{\partial t}\Phi_t(x)=\psi(\Phi_t(x))}\cr
  \Phi_0(x)=x\,,
\end{cases}
\]
then the map $\Phi_t(x)$ 
is a diffeomorphism of $\Omega$ onto itself for $t\in(-t_0,t_0)$, $t_0>0$ 
sufficiently small. 
We recall the notation $(u_t,K_t(x))=(\Phi_t(K),u(\Phi^{-1}_t(x)))$, and define $f(t):=E_0(u_t,K_t)$.
We claim that $f$ is differentiable on $(-t_0,t_0)$ with
\begin{equation}\label{e:f'}
f'(t)=\int_{\Omega\setminus \Phi_t(K)} \big(|\nabla u_t|^2 {\rm div}\,\psi 
-2 \nabla u_t^T \cdot D\psi\, 
\nabla u_t\big)
+\int_{\Phi_t(K)} e_t^T \cdot D\psi\, e_t \, d\mathcal{H}^1\,,
\end{equation}
where $e_t:\Phi_t(K)\to \mathbb S^1$ is a Borel vector field tangent to 
the rectifiable set $\Phi_t(K)$. 
Given this, if $(u,K)$ is either an absolute or a restricted or a generalized minimizer 
of $E_0$ we get \eqref{e:inner} (for $\param =0$) as necessarily $f'(0)=0$.

First, note that for all $\varepsilon$ small 
\begin{align*}
&\nabla\Phi_\varepsilon^{-1}(\Phi_\varepsilon(x))
=[\mathrm{Id}+\varepsilon D\psi(x)]^{-1}=
\mathrm{Id}-\varepsilon D\psi(x)+o(\varepsilon)\\
&\det\nabla\Phi_\varepsilon(x)=\det[\mathrm{Id}+\varepsilon D\psi(x)]
=1+\varepsilon\,\mathrm{div}\psi(x)+o(\varepsilon),
\end{align*}
$o(\varepsilon)$ uniform with respect to $x\in\Omega$.
Thus, by changing variables as $u_{t+\varepsilon}=u_t( \Phi_\varepsilon^{-1})$,
for $t\in(-t_0,t_0)$ and $\varepsilon\in\mathbb{R}$ sufficiently small, we get 
\begin{align}\label{e:volume t}
&\int_{\Omega\setminus \Phi_{t+\varepsilon}(K)}|\nabla u_{t+\varepsilon}|^2=
\int_{\Omega\setminus \Phi_t(K)}|\nabla u_t\cdot D \Phi_{\varepsilon}^{-1}(\Phi_{\varepsilon})|^2|\det D\Phi_\varepsilon|\notag\\
&=\int_{\Omega\setminus \Phi_t(K)}|\nabla u_t-\varepsilon\nabla u_t\cdot D\psi|^2(1+\varepsilon\,\mathrm{div}\psi)+o(\varepsilon)\notag\\
&=\int_{\Omega\setminus \Phi_t(K)}|\nabla u_t|^2+\varepsilon\int_{\Omega\setminus \Phi_t(K)} \big(|\nabla u_t|^2 {\rm div}\, \psi -2 \nabla u_t^T \cdot D \psi\, \nabla u_t\big)+o(\varepsilon)\,.
\end{align}
In addition, as $\Phi_{t+\varepsilon}(K)=\Phi_t(\Phi_\varepsilon(K))$, 
from the coarea formula \cite[Theorem 2.93]{AFP00} we infer that
\begin{align}\label{e:surface t}
\mathcal{H}^1(\Phi_{t+\varepsilon}(K))=
\mathcal{H}^1(\Phi_{t}(K))
+\varepsilon e_t^T\cdot D\psi\, e_t+o(\varepsilon)
\end{align}
In conclusion, we deduce from \eqref{e:volume t} and \eqref{e:surface t} that $f$ is differentiable in $t\in(-t_0,t_0)$ and that \eqref{e:f'} holds.

To prove \eqref{e:inner} in the general case we need to consider further the fidelity term. To this aim, for $t\in(-t_0,t_0)$ set
\[
h_g(t):=\int_{\Omega\setminus\Phi_t(K)}|u_t-g|^2\,.
\]
Let us first assume that $g\in C^1(\Omega)$, and prove that $h_g$ is differentiable on $(-t_0,t_0)$ with 
\begin{equation}\label{e:h'g}
h'_g(t)=-2\int_{\Omega}(u_{t}-g)\psi\cdot \nabla u_t
-\int_{\Phi_t(K)}(|u_{t}^+-g|^2-|u_{t}^--g|^2)
\psi\cdot\nu_{\Phi_t(K)}d\mathcal{H}^1\,,
\end{equation}
where $\nu_{\Phi_t(K)}$ denotes the counterclockwise rotation by $90$ degrees of the 
Borel unit tangent vector $e_t$ to $\Phi_t(K)$.

Indeed, by changing variables, the fact that 
$g\in C^1(\Omega)$ implies that, for some function $o(\varepsilon)$ that is uniform 
with respect to $x\in\Omega$,
\begin{align*}
&h_g(t+\varepsilon)=\int_{\Omega\setminus\Phi_{t+\varepsilon}(K)}|u_{t+\varepsilon}-g|^2=
\int_{\Omega\setminus\Phi_t(K)}|u_{t}-g(\Phi_\varepsilon)|^2|\det D\Phi_\varepsilon|\\
&=\int_{\Omega\setminus\Phi_t(K)}|u_{t}-g-\varepsilon\nabla g\cdot\psi+o(\varepsilon)|^2 (1+\varepsilon\,\mathrm{div}\psi)+o(\varepsilon)\\
&=h_g(t)+\varepsilon\int_{\Omega\setminus\Phi_t(K)}\big(|u_{t}-g|^2\mathrm{div}\psi-2(u_{t}-g)\nabla g\cdot\psi\big)+o(\varepsilon).
\end{align*}
Therefore, 
\[
h'_g(t)=\int_{\Omega\setminus\Phi_t(K)}\big(|u_{t}-g|^2\mathrm{div}\psi-2(u_{t}-g)\nabla g\cdot\psi\big)\,,
\]
and the identity in \eqref{e:h'g} follows from an integration by parts by taking into account that, under the standing assumptions, $|u-g|^2\in BV_{{\rm loc}}\cap L^\infty_{{\rm loc}}(\Omega)$.

If $g\in L^\infty(\Omega)$, let $g_\varepsilon\in C^1\cap L^\infty(\Omega)$ be such that 
$g_\varepsilon\to g$ in $L^2(\Omega)$. Consider therefore the functions
\[
f_{\varepsilon} (t) := f (t) + \param h_{g_\varepsilon} (t)
\]
and observe that, by the above formulas $|f'_\varepsilon (t)|$ is uniformly bounded. Thus $f_\varepsilon (t)$ converge uniformly to the Lipschitz function $f_0(t) = f(t) + \param h_g (t)$, which has a minimum in $0$. We rewrite
formula \eqref{e:h'g} for the derivative of $h_{g_\varepsilon}$ as follows
\begin{align*}
h'_{g_\varepsilon} (t) :=&\underbrace{-2\int_{\Omega\setminus\Phi_t(K)}(u_t-g_\varepsilon)\psi\cdot \nabla u_t }_{=:\Lambda_\varepsilon (t)}
\underbrace{-\int_{\Phi_t(K)}((u^+_t)^2-(u^-_t)^2)\psi\cdot \nu_{\Phi_t(K)}d\mathcal{H}^1}_{=: \Gamma (t)}\\
&+\underbrace{2\int_{\Phi_t(K)}(u^+_t-u^-_t)g_\varepsilon\,\psi\cdot \nu_{\Phi_t(K)}d\mathcal{H}^1}_{=: L_{t, \varepsilon} (\psi)}
\,.
\end{align*}
Note that $\Lambda_\varepsilon(t)$ converge uniformly as $\varepsilon\to 0$ to 
\[
\Lambda (t) = -2\int_{\Omega\setminus\Phi_t(K)}(u_t-g)\psi\cdot \nabla u_t \, ,
\]
while $\Gamma\in C^0((-t_0,t_0))$, because
$u\in W^{1,2}(\Omega\setminus K)$ and we can use the Generalized Area Formula \cite[Theorem 2.91]{AFP00}.
We use the Generalized Area Formula also to rewrite
\[
L_{t, \varepsilon} (\psi) 
= 2 \int_K (u^+ - u^- ) g_\varepsilon (\Phi_t )
\psi (\Phi_t ) \cdot \nu_{\Phi_t (K)} (\Phi_t )\,
J_t \, d\mathcal{H}^1 \, ,
\]
where $J_t (x) = |D\Phi_t (x) (e (x))|$, and $e(x)$ is the unit tangent to $K$ at $x$.

If we consider the vector-valued measures 
\[
\mu_t := 2 (u^+-u^-)\, \nu_{\Phi_t (K)} \circ \Phi_t\, J_t\, \mathcal{H}^1 \restr K\, ,
\]
and the pushforward $\alpha_t := (\Phi_t)_\sharp\mu_t$,
we can then write 
\[
L_{t, \varepsilon} (\psi) = \int g_\varepsilon\, \psi \cdot d\alpha_t \, .
\]
Next, consider the measures $g_\varepsilon d\alpha_t \otimes dt$ on $\Omega \times (-t_0, t_0)$ and, given the uniform boundedness of $g_\varepsilon$ (in a pointwise sense), we can assume that a suitable subsequence, not relabeled, converges to a measure of the form
\[
\bar g d\alpha_t\otimes dt\, ,
\]
for some Borel function $\bar g$ with $\|\bar g\|_{L^\infty(\Omega\times (-t_0, t_0),d\alpha_t\otimes dt)} \leq \|g\|_{\infty}$.
So we can rewrite 
\[
f_0 (t) = f_0 (0) + \int_0^t f' (s)\, ds + \param \int_0^t (\Lambda (s) + \Gamma (s))\, ds + \param \int_0^t \int \bar{g} \psi\cdot d\alpha_s\, ds\, .
\]
Since $f_0$ is Lipschitz and has a minimum in $0$, while $f'$, $\Lambda$, and $\Gamma$ are continuous, we conclude that, for every positive $s_0$
\[
|f' (0) + \param (\Lambda (0) + \Gamma (0))|\leq 
\param \|g\|_{\infty}  \sup_{|s|\leq s_0}  \left|\int 
\psi \cdot d\alpha_s\right|
\]
We now use the fact that $\psi$ is continuous and that so are the maps $t\mapsto \nu_{\Phi_t (K)} (\Phi_t (x)) = \frac{D\Phi_t (\nu (x))}{|D\Phi_t (\nu (x))|}$ and $t \mapsto J_t (x)$ to conclude that
\[
\liminf_{s_0\downarrow 0} \sup_{|s|\leq s_0} \left|\int 
\psi \cdot d\alpha_s\right| \leq
2 \int_K |u^+-u^-| |\psi \cdot \nu| d\mathcal{H}^1\, .
\]
Since the map $\psi \mapsto (f'(0) + \param (\Lambda (0)+\Gamma (0)))$ is linear in $\psi$, we conclude from the Riesz representation theorem that there exists a function $g_K \in L^\infty(\Omega,\mathcal{H}^1\restr K)$ with 
$\|g_K\|_{L^\infty(\Omega,\mathcal{H}^1\restrsmall K)} \leq \|g\|_\infty$ such that 
\[
f' (0) + \param (\Lambda (0) + \Gamma (0)) = 2\param \int_K (u^+-u^-) g_K \psi \cdot \nu \, d\mathcal{H}^1\, .
\]
\end{proof}
We show next the second form of the Euler-Lagrange conditions 
together with the higher regularity of $K$.
\begin{proof}[Proof Proposition~\ref{p:variational identities 2}]
Assume that $K\cap A$ is a graph for some open subset $A\subseteq\Omega$. Up to a rotation, there are an open interval $I\subset\mathbb{R}$ and $\phi:I\to \mathbb{R}$
such that 
\[
K\cap A=\{(t,\phi(t)):\,t\in I\}\,.
\]
Let $A^\pm:=\{(t,s)\in A:\,\pm s>\phi(t)\}$, and let $\varphi\in C^1(\bar{A})$ be such that $\varphi=0$ in a neighbourhood of $\partial A^+\setminus K$. Let
$v=u$ on $A^-$ and $v=u+\varepsilon\varphi$ on $A^+$, then from the (restricted) minimality of $(u,K)$ for $E_\param$ we infer 
\[
\int_{A^+}\big(\nabla u\cdot\nabla \varphi
+\param \varphi(u-g)\big)=0\,.
\]
Clearly, a similar identity can be obtained on $A^-$.
Therefore, $u$ is a weak solution to 
\begin{align}\label{e:weak solution}
\begin{cases}
\triangle u= \param (u-g)\qquad A^\pm\\
\frac{\partial u}{\partial\nu}=0\qquad K\cap A\,.
\end{cases}
\end{align}
From elliptic regularity theory we infer that if 
$\phi\in C^{1,\alpha}(I)$ then $u\in C^{1,\alpha}(A^\pm)$
(see \cite[Theorem 7.49]{AFP00}). The conclusions in item (a), in \eqref{e:Euler harmonic g} and in 
\eqref{e:Euler Neumann g} then follow at once.

We assume next that $u\in W^{2,2}(A^+\cup A^-)$, and moreover that $\nu$ is the interior normal vector to $A^+$. Let $\psi\in C^1_c(A;\mathbb{R}^2)$, then integrating twice by parts give
\begin{align*}
\int_{A^\pm}|\nabla u|^2\mathrm{div}\psi&=
-2\int_{A^\pm}\nabla u\cdot\nabla^2 u\,\psi
\mp\int_{K\cap A}|\nabla u^\pm|^2\psi\cdot\nu d\mathcal{H}^1\\
&=2\int_{A^\pm}\big(\triangle u\,\nabla u\cdot\psi+
\nabla u^T\cdot D\psi\,\nabla u\big)\\
&\qquad \mp\int_{K\cap A}\big((\nabla u\cdot\psi)
\frac{\partial u}{\partial\nu}
+|\nabla u^\pm|^2\psi\cdot\nu \big)d\mathcal{H}^1\\
&\stackrel{\eqref{e:weak solution}}{=}
2\int_{A^\pm}\big(\param(u-g)\nabla u\cdot\psi+
\nabla u^T\cdot D\psi\,\nabla u\big)
\mp\int_{K\cap A}|\nabla u^\pm|^2\psi\cdot\nu d\mathcal{H}^1\,.
\end{align*}
Therefore, from \eqref{e:inner} we infer that 
\begin{align}\label{e:curvature}
\int_{K\cap A}&e^T\cdot D\psi\, e\, d\mathcal{H}^1\notag\\
&=\int_{K\cap A}(|\nabla u^+|^2+\param|u^+-g_K|^2-|\nabla u^-|^2-\param|u^--g_K|^2)
\psi\cdot\nu d\mathcal{H}^1\,.
\end{align}
The last formula still holds even if $u\in W^{2,2}_{{\rm loc}}(A^+\cup A^-)$, which is actually 
the known regularity for $u$. Indeed, it suffices to deform smoothly $K$ inside 
$A^\pm$ and to take into account \eqref{e:weak solution} and the fact that both $u$ and 
$\nabla u$ are bounded to conclude.

Consider $\varphi\in C^1_c(A)$ and let $\psi(x)=(0,\varphi(x))$, as 
$K\cap A=\{(t,\phi(t)):\, t\in I\}$, we infer that 
\[
e(t,\phi(t))^T\cdot D\psi(t,\phi(t))\, e(t,\phi(t))
=\frac{d}{dt}\left(\varphi(t,\phi(t))\right)\frac{\phi'(t)}{1+|\phi'(t)|^2}\,,
\]
and thus
\begin{align*}
\int_{K\cap A}&e^T\cdot D\psi\, e\,d\mathcal{H}^1=
\int_I\frac{d}{dt}\left(\varphi(t,\phi(t))\right)\frac{\phi'(t)}{\sqrt{1+|\phi'(t)|^2}}dt\,.
\end{align*}
Moreover, if $\varphi(t,s)=\zeta(t)\eta(s)$, where $\zeta\in C^1_c(I)$ and $\eta\in C^1_c(\mathbb{R})$ such that $\eta=1$
on $[-2\|\phi\|_{L^\infty(I)},2\|\phi\|_{L^\infty(I)}]$,  
on setting $H:=|\nabla u^+|^2+\param |u^+-g_K|^2-|\nabla u^-|^2-\param |u^--g_K|^2\in L^{\infty}(K\cap A)$, \eqref{e:curvature} rewrites as 
\begin{align}\label{e:distributional curvature}
\int_I&\zeta'(t)\frac{\phi'(t)}{\sqrt{1+|\phi'(t)|^2}}dt
=\int_IH(t,\phi(t))\,\zeta(t)dt\,.
\end{align}
In particular, $\frac{\phi'(t)}{\sqrt{1+|\phi'(t)|^2}}\in W^{1,\infty}(I)$, and 
in turn this implies $\phi'\in W^{1,\infty}(I)$, as $\phi'\in L^\infty(I)$. 
Item (b) is then established.

Eventually, \eqref{e:distributional curvature} yields that 
the distributional curvature of the graph of $\phi$ satisfies
\[
\frac{d}{dt}\left(\frac{\phi'(t)}{\sqrt{1+|\phi'(t)|^2}}\right)=-H(t,\phi(t))\qquad
\textrm{$\mathcal{L}^1$ a.e. on $I$},
\] 
\eqref{e:Euler curvature g} then follows at once.
\end{proof}

\chapter{Equivalence of \texorpdfstring{$SBV$}{SBV} and classical formulations}\label{a:strong vs SBV}

In this section, we prove the density lower bound for $K$, where $(u,K)$ is a minimizer of $E_\param$, as well as
the same property for the jump set $S_u$ of a minimizer $u$ 
of the $SBV$ counterpart $\widetilde{E}_\lambda$. 
The equivalence of the weak and strong formulations of the problem then follows easily from the latter property. 
We also add one important consequence on the local compactness of bounded minimizers which will then be used to prove Theorem \ref{t:minimizers compactness}.

\section{Interior density lower bound}

This subsection aims at giving a direct simple proof of the density lower bound estimate valid for (restricted) minimizers of the functionals $E_\param$, and also for minimizers of the corresponding $SBV$ relaxed versions.
In addition, we prove such property also for minimizers of $E_\param$ subject 
to Dirichlet boundary conditions. The latter result will be crucial for the applications in section~\ref{ss:compactness through SBV}, in turn instrumental for the proof of the compactness result  Theorem~\ref{t:minimizers compactness}.

We start with proving the following decay result.
\begin{lemma}\label{l:decay}
For every $\tau>0$ there are constants $\varepsilon(\tau)>0$ and $\vartheta(\tau)\in(0,1)$ such that 
if $(u,K)$ is a (restricted) minimizer of $E_\param$ on $\Omega$ then the following property holds. 

If $\mathcal{H}^1(K\cap B_\rho(x))\leq\varepsilon\rho$ for some $x\in\Omega$ and 
$\rho\in(0,\mathrm{dist}(x,\partial\Omega))$ with $\param \in [0,1]$ and $\|g\|_\infty \leq M_0$, 
then
\begin{equation}\label{e:Eg decay}
E_0(u,K,B_{\tau\rho}(x))\leq \max\{\tau^{\sfrac 32}E_0(u,K,B_{\rho}(x)),
\textstyle{\frac{8\pi \param}\vartheta}
M_0^2\rho^2\}\,.
\end{equation}
\end{lemma}
\begin{proof}
We distinguish the case $\param=0$ from the case $\param>0$.

\noindent{\bf Case $\param=0$.}
Assume by contradiction that for some $\tau>0$ there is a sequence of (restricted) minimizers $(u_j,K_j)$ on $\Omega$,
points $x_j\in\Omega$ and radii $\rho_j$ such that $\rho_j^{-1}\mathcal{H}^1(K_j\cap B_{\rho_j}(x_j))$ is infinitesimal 
and, on setting $E_j:=E_0 (u_j,K_j,B_{\rho_j}(x_j))$,
\[
E_0 (u_j,K_j,B_{\tau\rho_j}(x_j))> \tau^{\sfrac 32}E_j\,.
\]
Define next $v_j:B_1\to\mathbb{R}$ by $v_j(x):=E_j^{-\sfrac12}u_j(x_j+\rho_jx)$ and 
$H_j:=B_1\cap\textstyle{\frac 1{\rho_j}}(K_j-x_j)$. By scaling we get
\[
\int_{B_1\setminus H_j}|\nabla v_j|^2=E_j^{-1}\int_{B_{\rho_j}(x_j)\setminus K_j}|\nabla u_j|^2\leq 1,\qquad
\mathcal{H}^1(H_j)=\rho_j^{-1}\mathcal{H}^1(K_j\cap B_{\rho_j}(x_j))\,.
 \]
As $\mathcal{H}^1(H_j)$ is infinitesimal, for a subsequence not relabeled, by the coarea formula  
\cite[Theorem 2.93]{AFP00} and by the Mean-value theorem we may find a radius $r\in(\tau^{\sfrac14},1)$ 
(independent from $j$) such that 
\begin{equation}\label{e:raggio r}
H_j\cap\partial B_r=\emptyset\,,\qquad 
\int_{\partial B_r}|\nabla v_j|^2\leq (1-\tau^{\sfrac14})^{-1}\,.
\end{equation}
In particular, $v_j\in W^{1,2}(\partial B_r)$ with $\mathrm{osc}_{\partial B_r}v_j\leq C$ (independent 
from $j$). 
Therefore, up to subtracting the mean value of $v_j$ over $\partial B_r$ and relabeling the sequence $v_j$, 
we may conclude that $\|v_j\|_{L^\infty(\partial B_r)}\leq C$. 
In turn, scaling back to $u_j$ and using Lemma~\ref{l:maximum} yields that $\|v_j\|_{L^\infty(B_r)}\leq C$.
Define 
for an admissible pair $(w,J)$ in $B_1$ and for all open subsets $A\subseteq B_1$ 
\begin{equation}\label{e:Fj}
 F_j(w,J,A):=\int_{A\setminus J}|\nabla w|^2+\frac{\rho_j}{E_j}\mathcal{H}^1(J\cap A)\,.
\end{equation}
Then, the conditions above are rewritten as 
\begin{equation}\label{e:bounds vj}
F_j(v_j,H_j,B_1)= 1,\qquad F_j(w,J,B_1)\geq F_j(v_j,H_j,B_1)\,,
\end{equation}
for all admissible pairs $(w,J)$ such that $\{v_j\neq w\}\cup(H_j\triangle J)\subset\subset B_1$
(with the number of connected components less than that of $K_j$ in case $u_j$ is a restricted minimizer of $E_0$), and 
\begin{equation}\label{e:contra}
F_j(v_j,H_j,B_\tau)
>\tau^{\sfrac32}\,.
\end{equation}
Moreover, $v_j\in SBV(B_1)$ with $S_{v_j}\subseteq H_j$ being $v_j$ harmonic on $B_1\setminus H_j$.
Hence, thanks to the first condition in \eqref{e:bounds vj}, and the $L^\infty(B_r)$ bound on $v_j$, Ambrosio's compactness theorem \cite[Theorem 4.8]{AFP00} implies 
that there exists a subsequence (not relabeled) and a function $v\in SBV(B_r)$ such that $v_j$ converge to 
$v$ in $L^2(B_r)$, and that for all open subsets $A\subseteq B_r$
\begin{equation}\label{e:v Sobolev}
\mathcal{H}^1(S_v\cap A)\leq
\liminf_j\mathcal{H}^1(S_{v_j}\cap A)=0,\quad
\int_A|\nabla v|^2\leq\liminf_j\int_A|\nabla v_j|^2\leq 1\,.
\end{equation}
In particular, $v\in W^{1,2}(B_r)$. Actually, we claim that $v$ turns out to be harmonic on $B_r$ and
satisfying for all $s\in(0,r)$
\begin{equation}\label{e:energy convergence v_j}
\lim_jF_j(v_j,H_j,B_s)=
\lim_j\int_{B_s}|\nabla v_j|^2=\int_{B_s}|\nabla v|^2\,.
\end{equation}
Given this for granted, we conclude as follows: on one hand from \eqref{e:contra} and the energy convergence 
in \eqref{e:energy convergence v_j} and being $\tau<r^4<r<1$, we infer that 
\begin{equation}\label{e:contradiction tau}
\int_{B_\tau}|\nabla v|^2\geq\tau^{\sfrac32}\,;
\end{equation}
but on the other hand being $v$ harmonic on $B_r$ and using the first condition in \eqref{e:bounds vj}, 
we conclude that
\[
\int_{B_\tau}|\nabla v|^2\leq\frac{\tau^2}{r^2}
\int_{B_r}|\nabla v|^2\leq\frac{\tau^2}{r^2}\,,
\] 
therefore we get a contradiction recalling that $\tau<r^4<r<1$.

We finally establish \eqref{e:energy convergence v_j} together with the harmonicity of $v$ on $B_r$. 
To this aim let $w\in W^{1,2}(B_r)$ with $\{v\neq w\}\subset\subset B_s$, $s\in(0,r)$. 
Let $s<t\in(0,r)$ and $\varphi\in C_c^\infty(B_t)$ be such that $\varphi=1$ on $B_s$.
Define functions $\zeta_j=\varphi w+(1-\varphi) v_j$ and sets $J_j$ by 
$J_j\cap B_s=\emptyset$, $J_j\cap(B_t\setminus\overline{B_s})=H_j\cap(B_t\setminus\overline{B_s})$.
Then, $(\zeta_j,J_j)$ is a pair to test the (restricted) minimality of $(v_j,H_j)$ for $F_j$
(note that the number of connected components of $J_j$ is less than that of $H_j$, i.e. of $K_j$).
The locality of the energy leads to
\begin{align*}
F_j&(v_j,H_j,B_s)\leq F_j(v_j,H_j,B_t)
\leq F_j(\zeta_j,J_j,B_t)\\
&\leq F_j(w,\emptyset,B_s)
+C\,F_j(w,H_j,B_t\setminus\overline{B_s})+C\,F_j(v_j,H_j,B_t\setminus\overline{B_s})
+\frac C{(t-s)^2}\int_{B_t\setminus\overline{B_s}}|v_j-w|^2\\
&\leq\int_{B_s}|\nabla w|^2+C\int_{B_t\setminus\overline{B_s}}|\nabla v|^2
+ C F_j(v_j,H_j,B_t\setminus\overline{B_s})
+\frac C{(t-s)^2}\int_{B_t\setminus\overline{B_s}}|v_j-v|^2\,.
\end{align*}
As the sequence of Radon measures $(F_j(v_j,H_j,\cdot))_{j\in\N}$
is equi-bounded in mass on $B_1$ in view of \eqref{e:bounds vj}, 
it converges to some Radon measure $\mu$ on $B_1$ up to a subsequence not relabeled for convenience. 
Assume that $\mu(\partial B_s)=\mu(\partial B_t)=0$, by passing to the limit as 
$j\uparrow\infty$ and by Ambrosio's lower semicontinuity result we find
\begin{align*}
\int_{B_s}|\nabla v|^2&\leq\liminf_jF_j(v_j,H_j,B_s)\leq
\limsup_jF_j(v_j,H_j,B_s)\\
&\leq\int_{B_s}|\nabla w|^2+C\int_{B_t\setminus\overline{B_s}}|\nabla v|^2
+C\,\mu(B_t\setminus\overline{B_s})\,.
\end{align*}
Thus, by letting $t\downarrow s^+$ along values satisfying $\mu(\partial B_t)=0$ 
we conclude that for all but a countable set of radii in $(0,r)$ we have
\begin{equation}\label{e:harmonicity}
\int_{B_s}|\nabla v|^2\leq\liminf_jF_j(v_j,H_j,B_s)\leq
\limsup_jF_j(v_j,H_j,B_s)\leq\int_{B_s}|\nabla w|^2\,.
\end{equation}
Moreover, the latter inequality is extended to all radii $s\in(0,r)$ as the Dirichlet energies 
of $v$ and of $w$, as set functions, are the trace of a Radon measure on open sets.

Eventually, equality \eqref{e:energy convergence v_j} follows by taking $w=v$ in \eqref{e:harmonicity}, and by taking into account \eqref{e:v Sobolev}.
\medskip
 
\noindent{\bf Case $\param>0$.}
The proof of the general case needs a further argument in addition to those used for $\param=0$. We fix $\param \in (0,1]$ and 
we claim that 
\begin{equation}\label{e:decay E}
E_0(u,K,B_{\tau\rho}(x))\leq \tau^{\sfrac 32}E_0(u,K,B_{\rho}(x)),
\end{equation}
if in addition
\begin{equation}\label{e:additional}
(1-\vartheta)E_0(u,K,B_\rho(x))\leq E_0 (w,J,B_\rho(x))
\end{equation}
for all $(w,J)$ with $\{u\neq w\}\cup(K\triangle J)\subset\subset B_\rho(x)$ and $\|w\|_{\infty}\leq M_0$
(where $\vartheta(\tau)\in(0,1)$ is the parameter in the statement).
In case the condition in \eqref{e:additional} is not satisfied for some admissible pair $(w,J)$ with $\{u\neq w\}\cup(K\triangle J)\subset\subset B_\rho$  and $\|w\|_{\infty}\leq M_0
$, then we have
\[
E_0(u,K,B_\rho(x))<\textstyle{\frac1\vartheta}(E_0(u,K,B_\rho(x))-E_0(w,J,B_\rho(x)))\,.
\]
Using that $(u,K)$ is a (restricted) minimizer of $E_\param$ on $\Omega$ and that
$\|w\|_{\infty}\leq M_0
$, we get from the latter inequality
\begin{align*}
E_0&(u,K,B_{\tau\rho}(x)) \leq E_0(u,K,B_\rho(x))
\leq {\textstyle{\frac1\vartheta}} (E_0 (u,K, B_\rho (x)) -
E_0 (w,J, B_\rho (x))\\
&\leq {\textstyle{\frac1\vartheta}} (E_\param (u,K, B_\rho (x)) -
E_\param (w,J, B_\rho (x))) + {\textstyle{\frac\param\vartheta}} 
\int_{B_{\rho} (x)} |w-g|^2
\leq\textstyle{\frac{8\pi\param}\vartheta} M_0^2
\rho^2\,.
\end{align*}
Thus, to conclude \eqref{e:Eg decay} we are left with proving \eqref{e:decay E} 
under the further condition \eqref{e:additional}. 
In what follows we highlight only the necessary changes in the argument used 
for the case $\param=0$.
We argue by contradiction and consider $\tau>0$ for which we can find (restricted) minimizers $(u_j,K_j)$ of 
$E_{\param_j}(\cdot,\cdot,\Omega,g_j)$, with $\param_j\in[0,1]$ and $\|g_j\|_\infty\leq M_0$ (cf. Assumption \ref{a:blow-up}), and infinitesimal sequences $\varepsilon_j$ and $\vartheta_j$, points $x_j\in\Omega$ and 
radii $\rho_j$ so that $\mathcal{H}^1(K_j\cap B_{\rho_j}(x_j))\leq\varepsilon_j\rho_j$ and, on setting $E_j:=E_0 (u_j,K_j,B_{\rho_j}(x_j))$, we have
\begin{equation}\label{e:quasi min}
(1-\vartheta_j)E_j 
\leq E_0(w,J,B_{\rho_j}(x_j)),\qquad E_0 (u_j,K_j,B_{\tau\rho_j}(x_j))
> \tau^{\sfrac 32}E_j\,.
\end{equation}
for all $(w,J)$ with $\{u_j\neq w\}\cup(K_j\triangle J)\subset\subset B_{\rho_j}(x_j)$ and $\|w\|_{\infty} \leq M_0
$.
Consider next the functions $v_j(x):=E_j^{-\sfrac12}u_j(x_j+\rho_j x)$, the sets 
$H_j:=B_1\cap\textstyle{\frac1{\rho_j}}(K_j-x_j)$, and the functionals $F_j$ as defined in \eqref{e:Fj}.
By rescaling \eqref{e:quasi min}
\begin{equation}\label{e:quasi min vj}
F_j(v_j,H_j,B_1)
=1\,,\qquad 1-\vartheta_j\leq F_j(w,J,B_1)
\end{equation}
for all $(w,J)$ with $\{v_j\neq w\}\cup(H_j\triangle J)\subset\subset B_1$ and 
$\|w\|_{\infty} \leq E_j^{-\sfrac12}M_0$, and moreover
\begin{equation}\label{e:contra bis}
F_j(v_j,H_j,B_\tau)
> \tau^{\sfrac 32}\,,
\end{equation}
(cf. with \eqref{e:bounds vj} and \eqref{e:contra}). 
Note that $E_j$ is actually infinitesimal by the energy upper bound 
in Lemma~\ref{l:upper-bound} and $\rho_j\to 0$ as $j\to\infty$.
Thus, we may find a radius $r\in(\tau^{\sfrac14},1)$ such that \eqref{e:raggio r} hold, 
$v_j\in W^{1,2}(\partial B_r)$ with $\mathrm{osc}_{\partial B_r}v_j\leq C$, and, up to subtracting the mean value 
of $v_j$ over $\partial B_r$ and relabeling the sequence $v_j$, we infer that $\|v_j\|_{L^\infty(\partial B_r)}\leq C$. 
Consider next the function
\[
\tilde{v}_j:=\begin{cases}
              \min\{\|v_j\|_{L^\infty(\partial B_r)},\max\{v_j,-\|v_j\|_{L^\infty(\partial B_r)}\}\} & B_r \cr
              v_j & B_1\setminus B_r
             \end{cases}
\]
Thus, $\tilde{v}_j\in SBV(B_1)$ with  $\mathcal{H}^1(S_{\tilde{v}_j}\setminus S_{v_j})=0$. 
In particular, $\mathcal{H}^1(S_{\tilde{v}_j})\leq\mathcal{H}^1(S_{v_j})\leq \mathcal{H}^1(H_j)$.
Ambrosio's compactness theorem implies that there exists a subsequence (not relabeled) and a function $v\in SBV(B_r)$
such that $\tilde{v}_j$ converge to $v$ in $L^2(B_r)$, and that for all open subsets $A\subseteq B_r$
\[
\mathcal{H}^1(S_v\cap A)\leq
\liminf_j\mathcal{H}^1(S_{\tilde{v}_j}\cap A)=0,\quad
\int_A|\nabla v|^2\leq\liminf_j\int_A|\nabla \tilde{v}_j|^2\leq 1\,.
\]
Therefore, $v\in W^{1,2}(B_r)$. 
Moreover, for every $t\in(0,r)$ we have
\begin{equation}\label{e:energy v_j tildev_j}
\lim_j\Big(\int_{B_t}|\nabla v_j|^2-\int_{B_t}|\nabla \tilde{v}_j|^2\Big)=0\,.
\end{equation}
Indeed, on one hand being $\tilde{v}_j$ a truncation of $v_j$, we have for every Borel subset $B$ of $B_1$
\[
\int_B|\nabla\tilde{v}_j|^2\leq \int_B|\nabla v_j|^2\,.
\]
In turn, from this and \eqref{e:quasi min vj} applied to the pair $(\tilde{v}_j,H_j)$,
we deduce that for all $t\in(0,r)$  
\[
\int_{B_t}|\nabla v_j|^2\leq \int_{B_t}|\nabla\tilde{v}_j|^2+\vartheta_j\,.
\]
Therefore, we conclude \eqref{e:energy v_j tildev_j}.
In turn, arguing as for the case $\param=0$, through \eqref{e:energy v_j tildev_j} 
one can prove that for all $s\in(0,r)$
\[
\lim_jF_j(v_j,H_j,B_s)=
\lim_j\int_{B_s}|\nabla v_j|^2=\int_{B_s}|\nabla v|^2
\] 
together with the fact that $v$ is harmonic on $B_r$ by exploiting the (restricted) almost minimality of $v_j$ for $F_j$ on $B_1$, namely the second condition in \eqref{e:quasi min vj}. Eventually, the final contradiction follows again from \eqref{e:contra bis}.
\end{proof}
An immediate consequence of Lemma~\ref{l:decay} are the density lower bound estimates for (restricted) minimizers 
of $E_\param$. We follow the argument in \cite[Theorem~7.21]{AFP00}. 

\begin{proof}[Proof of Theorem~\ref{t:dlb}] We establish the estimates separately in the two cases $\param=0$ and $\param >0$. 
In both cases we denote by $\epsilon:=\varepsilon(\sfrac12)$ the parameter 
in Lemma~\ref{l:decay} having fixed $\tau=\sfrac12$.

\noindent{\bf Case $\param=0$.}
Let $\sigma>0$ be such that $2\pi\sigma^{\sfrac12}<\epsilon$, and consider the set 
\[
\Omega_u:=\{x\in\Omega:\,\mathcal{H}^1(K\cap B_\rho(x))<\varepsilon(\sigma)\rho\textrm{ for some $\rho\in(0,\mathrm{dist}(x,\partial\Omega))$}\}.
\]
Note that if $x\in\Omega_u$ and $\rho\in(0,\mathrm{dist}(x,\partial\Omega))$ is the corresponding radius, then 
$\mathcal{H}^1(K\cap B_\rho(x))<(1-\mu)\varepsilon(\sigma)\rho$ for some $\mu\in(0,1)$. 
It is then easy to check that $B_{(1-\sfrac\mu2)\rho}(x)\subset\Omega_u$, 
so that $\Omega_u$ is an open subset of $\Omega$.

We claim that if $x\in\Omega_u$, namely
$\mathcal{H}^1(K\cap B_\rho(x))<\varepsilon(\sigma)\rho$ for some $\rho\in(0,\mathrm{dist}(x,\partial\Omega))$, 
then for all $j\in\mathbb{N}$
\begin{equation}\label{e:decay iteration}
 E_0(u,K,B_{\sigma 2^{-j}\rho}(x))\leq\epsilon 2^{-\sfrac j2}\,(\sigma 2^{-j}\rho)\,.
\end{equation}
In particular, given this for granted, we get 
\[
 \lim_{j\to\infty}\frac{E_0(u,K,B_{\sigma 2^{-j}\rho}(x))}{\sigma2^{1-j}\rho}=0,
\]
so that $x\notin K_1$, where $K_1:=\{x\in K: \lim_r\textstyle{\frac{1}{2r}}\mathcal{H}^1(K\cap B_r(x))=1\}$.
As the rectifiability of $K$ implies that $K_1$ is dense in $K$, we conclude that 
$K=\overline{K_1}\subseteq\Omega\setminus\Omega_u$, being $\Omega_u$ open.
On the other hand, $u$ is harmonic on $\Omega\setminus K$, so that $\Omega\setminus K\subseteq \Omega_u$, and 
\eqref{e:dlbE} thus follows at once.

Let us now prove \eqref{e:decay iteration} by an induction argument. Indeed, the case $j=0$ follows immediately by from Lemma~\ref{l:decay} (applied with $\tau=\sigma$), the energy upper bound and the choice of $\sigma$
\[
E_0(u,K,B_{\sigma\rho}(x))\leq \sigma^{\sfrac32} E_0(u,K,B_\rho(x))\leq 2\pi\sigma^{\sfrac32}\rho<\epsilon\sigma\rho\,.
\]
Assume now that \eqref{e:decay iteration} is true for some $j\in\mathbb{N}$, we show then that it holds for 
$j+1$. Indeed, in such a case Lemma~\ref{l:decay} (applied with $\tau=\sfrac12$) and the induction assumption imply that 
\[
E_0(u,K,B_{\sigma 2^{-j-1}\rho}(x))\leq 2^{-\sfrac32} E_0(u,K,B_{\sigma 2^{-j}\rho}(x))\leq
\epsilon 2^{-\sfrac{(j+1)}2}\,(\sigma 2^{-j-1}\rho)\,.
\]
\medskip

\noindent{\bf Case $\param>0$.}
The proof is very similar to that in the case $\param=0$, we highlight only the necessary changes.
Moreover, we will prove the estimate, without loss of generality, for radii smaller than some suitably chosen constant $R$.

Choose $\sigma\in(0,1)$ such that 
$2\pi\sigma^{\sfrac12}<\textstyle{\frac\epsilon2}$, and $R>0$ such that  
$\textstyle{\frac{8\pi\param}{\vartheta(\sfrac12)}}M_0^2
R<\epsilon\sigma$.
Note that, since $\|g\|_\infty \leq M_0$, $\param\leq 1$ and the parameter $\vartheta(\sfrac12)$ 
is fixed from Lemma~\ref{l:decay}, $R$ can be indeed chosen to be an absolute constant. 

We claim that if $\mathcal{H}^1(K\cap B_\rho(x))<\varepsilon(\sigma)\rho$ for some 
$\rho\in(0,\min\{R,\mathrm{dist}(x,\partial\Omega)\})$, 
then \eqref{e:decay iteration} holds true. Indeed, if $j=0$, Lemma~\ref{l:decay}, the energy upper bound in Lemma~\ref{l:upper-bound}, and the choices of $\sigma$ and $R$ give
\begin{align*}
 E_0(u,K,B_{\sigma\rho}(x))&\leq \max\{\sigma^{\sfrac32}E_0(u,K,B_{\rho}(x)),
\textstyle{\frac{8\pi\param}{\vartheta(\sfrac12)}}M_0^2
\rho^2\}\\
&\leq \max\{\sigma^{\sfrac32}E_\param(u,K,B_{\rho}(x)),
\textstyle{\frac{8\pi\param}{\vartheta(\sfrac12)}}M_0^2
R\rho\}\\
&\leq\max\{\sigma^{\sfrac32}2\pi(1+\param\|g\|_\infty^2R),
\textstyle{\frac{8\pi\param}{\vartheta(\sfrac12)}}M_0^2
R\}\rho<\epsilon\sigma\rho\,.
\end{align*}
Assume now that \eqref{e:decay iteration} is true for some $j\in\mathbb{N}$, we show that then it holds for 
$j+1$. Indeed, in such a case the induction assumption, Lemma~\ref{l:decay} and the choices of $\sigma$ 
and $R$ imply that 
\begin{align*}
E_0(u,K,B_{\sigma 2^{-j-1}\rho}(x))&\leq \max\{2^{-\sfrac32}E_0(u,K,B_{\sigma 2^{-j}\rho}(x)),
\textstyle{\frac{8\pi\param}{\vartheta(\sfrac12)}}M_0^2
(\sigma2^{-j}\rho)^2\}\\
&\leq\max\{2^{-\sfrac32} 2^{-\sfrac j2}\,(\sigma 2^{-j}\rho),
\textstyle{\frac{8\pi\param}{\vartheta(\sfrac12)}}M_0^2
R(\sigma 2^{-j})^2\rho\}\\
&\leq\max\{\epsilon 2^{-\sfrac{(j+1)}2}\,(\sigma 2^{-j-1}\rho),
\epsilon\sigma(\sigma 2^{-j})^2\rho\}\leq\epsilon 2^{-\sfrac{(j+1)}2}\,(\sigma 2^{-j-1}\rho)\,.
\end{align*}
Therefore, \eqref{e:decay iteration} holds. Set now 
\[
\Omega_u:=\{x\in\Omega:\,\mathcal{H}^1(K\cap B_\rho(x))<\varepsilon(\sigma)\rho
\textrm{ for some $\rho\in(0,\min\{R,\mathrm{dist}(x,\partial\Omega))\}$}\}\,.
\]
It is easy to check that $\Omega_u$ is an open subset of $\Omega$, that together with 
\eqref{e:decay iteration} give $K\subseteq\Omega\setminus\Omega_u$.
Eventually, by taking into account that $u\in W^{2,p}_{{\rm loc}}(\Omega\setminus K)$ by elliptic regularity we conclude \eqref{e:dlbEg}. 
\end{proof}
Recall that the density lower bound in Theorem~\ref{t:dlb} together with the energy upper bound in Lemma~\ref{l:upper-bound} establish the Ahlfors-David regularity of $K$ (cf. \eqref{e:Ahlfors-David regularity}).

\section{Boundary density lower bound}

We prove next similar results for minimizers of $E_\param$ subject to 
Dirichlet boundary conditions. This fact has been first established by Carriero and 
Leaci for $\param=0$ in \cite{CL90} under slightly more general assumptions on the reference set $U$. 
\index{boundary density lower bound@boundary density lower bound}
\index{density lower bound, boundary@density lower bound, boundary}
More precisely, given a bounded open set $U$ with $C^1$ boundary, a closed subset 
$\Sigma$ of $\partial U$ with $\mathcal{H}^1(\Sigma)=0$, a boundary datum $w\in C^1\cap L^\infty
(\mathbb{R}^2\setminus \Sigma)$, we assume that there exists $(u,K)$ minimizing $E_\param(\cdot,\cdot,\overline{U},g)$ over the set of admissible pairs $(v,J)$, $J$ closed and $1$-rectifiable in $\overline{U}$ and $v\in C^1(U\setminus J)\cap C^0(\overline{U}\setminus(\Sigma\cup J))$, such that
\begin{equation}\label{e:Dirichlet pb}
v=w \text{ on }\partial U\setminus(J\cup \Sigma)\,.
\end{equation}
In fact, the existence of such an optimal pair can be inferred from the analogous density lower bound property for minimizers of the weak formulation of the problem, namely for minimizers of $\widetilde{E}_\param$ (the relaxed counterpart of $E_\param$, cf. \eqref{e:tilde Eg}) satisfying the Dirichlet boundary condition, as will be discussed in Section~\ref{ss:dlb SBV}. In particular, the $L^\infty$ assumption on $w$ is needed to show the existence of a minimizer for $\widetilde{E}_\param$. 

We start fixing some notation. For any admissible pair $(v,J)$ satisfying \eqref{e:Dirichlet pb} we
extend $v$ to $\R^2\setminus U$ by $w$ and, for simplicity, denote such an extension still by $v$.
With this agreement, the quantity 
$E_\param(v,J,B_r(x),g)$ is well defined for every $x\in\partial U$ and $r>0$.

Being $U$ bounded and open with $C^1$ boundary, for every $x\in\partial U$ there 
is a radius $\rho_x>0$ such that $\partial U\cap B_{\rho_x}(x)$ is the graph of a function 
$\varphi_x\in C^1(\mathbb{R})$ is some system of coordinates. 
In case $x\in\partial U\setminus \Sigma$ it is not restrictive 
to assume that $B_{\rho_x}(x)\cap \Sigma=\emptyset$, being $\Sigma$ closed.
Finally, fixed $x\in\partial U$, up to rigid motions, we may always reduce 
to the following situation: $\varphi_x(0)=\varphi_x'(0)=0$, and
\[
U\cap B_{\rho_x}(x)=x+\{y=(y_1,y_2)\in B_{\rho_x}:\, y_2>\varphi_x(y_1)\}\,.
\] 
We prove first a decay result for the energy $E_0$ at boundary points, analogous to the interior case treated in Lemma~\ref{l:decay}. We keep the notation introduced there and outline only the necessary changes in the argument.
\begin{lemma}\label{l:decay Dirichlet}
For every $\tau>0$ and $M_1>0$ there are constants $\varepsilon(\tau,M_1)>0$ and $\vartheta(\tau,M_1)\in(0,1)$ such that if $(u,K)$ is a minimizer of $E_\lambda(\cdot,\cdot,\overline{U},g)$ subject to the Dirichlet boundary condition \eqref{e:Dirichlet pb}, then the following property holds. 

If $\mathcal{H}^1(K\cap B_\rho(x))\leq\varepsilon\rho$ and $\|\nabla w\|_{L^\infty(B_\rho(x))}\leq M_1$
for some $x\in\partial U\setminus \Sigma$ and $\rho\in(0,\min\{1,\rho_x\})$, with $\param \in [0,1]$ and $\|g\|_\infty \leq M_0$, then
\begin{align}\label{e:Eg decay Dirichlet}
E_0&(u,K,B_{\tau\rho}(x))\notag\\
&\leq \max\{\tau^{\sfrac 32}E_0(u,K,B_{\rho}(x)),
\textstyle{\frac{\tau^{\sfrac32}}\vartheta}
\|\nabla w\|^2_{L^\infty(B_{\rho}(x))},
\textstyle{\frac{\tau^{\sfrac32}}\vartheta}
\|\varphi_x'\|_{L^\infty((-\sfrac{\rho}{2},\sfrac{\rho}{2}))},
\textstyle{\frac{8\pi \param}\vartheta}M_0^2\rho^2\}\,.
\end{align}
\end{lemma}
\begin{proof}
\noindent{\bf Case $\param=0$.}
Fixed $x\in\partial U\setminus \Sigma$, we proceed with the same proof of 
Lemma~\ref{l:decay} up to \eqref{e:energy convergence v_j}, for 
which now we claim that $v$ minimizes the Dirichlet energy on the set 
$v+W^{1,2}_0(B_r\cap\{x_n>0\})$. Clearly, the claim implies that $v$ is harmonic on 
$B_r\cap\{x_n>0\}$. 
To prove such a claim, note first that $v$ is constant on $B_r\cap\{x_n<0\}$. 
Indeed, besides the conditions in \eqref{e:bounds vj} and \eqref{e:contra}, contradicting
\eqref{e:Eg decay Dirichlet} gives also
\begin{equation}\label{e:contra wj}
\tau^{\sfrac32}\|\nabla w_j\|^2_{L^\infty(B_{\rho_j}(x_j))}
+\tau^{\sfrac32}\|\varphi_{x_j}'\|_{L^\infty((-\sfrac{\rho_j}2,\sfrac{\rho_j}2))}
\leq E_0(u_j,K_j,B_{\tau\rho_j}(x_j))\vartheta_j
\leq E_j\vartheta_j\,,
\end{equation}
where we recall that $E_j=E_0(u_j,K_j,B_{\rho_j}(x_j))$. Note that
$E_j\leq C\rho_j$, for some constant independent from $j$, using 
$\|\nabla w_j\|_{L^\infty(B_{\rho_j}(x_j))}\leq M_1$ and arguing as for the energy upper bound. 
Thus, since $v_j=E_j^{-\sfrac12}w_j(x_j+\rho_j\cdot)$ on $B_1\setminus \frac{U-x_j}{\rho_j}$, 
$\rho_j\leq1$ and \eqref{e:contra wj} yield that $\|\nabla v_j\|_{L^\infty(B_1\setminus\frac{U-x_j}{\vartheta_j})}=O(\vartheta_j)\to 0$ as $j\to\infty$. 
On the other hand, $B_1\setminus\frac{U-x_j}{\rho_j}=\{z\in B_1:\,z_2\leq\rho_j^{-1}\varphi_{x_j}(\rho_jz_1)\}$, so that $\varphi_{x_j}(0)=0$ and \eqref{e:contra wj} imply that 
$\rho_j^{-1}\varphi_{x_j}(\rho_j\cdot)$ converges uniformly to $0$ on $(-\sfrac12,\sfrac12)$, and thus we conclude that $v$ is constant on $B_1\cap\{z_2<0\}$.

Given any function $w\in v+C^\infty_c(B_r\cap\{x_n>0\})$ with 
$\{w\neq v\}\subset\subset B_s\cap\{x_n\geq\frac1i\}=:B_{s,i}$, $s\in(0,r)$ 
and $i\in\mathbb{N}$ such that $i>s^{-1}$, consider $\varphi\in C_c^\infty(B_{t,i+1})$,
$t\in(s,r)$, with $\varphi=1$ on $B_{s,i}$.
Note that $B_{s,i}\subset\subset B_{t,i+1}\subset\subset\frac{U-x_j}{\rho_j}$ for $j$ sufficiently large,
therefore the pair $(\zeta_j,J_j)$ defined by $\zeta_j:=\varphi w+(1-\varphi)v_j$, $J_i\cap B_{s,i}:=\emptyset$ and $J_j\cap (B_r\setminus B_{s,i})=H_j\cap(B_r\setminus B_{s,i})$, is admissible  
to test the minimality of $v_j$ for $F_j$.
Thus, inequalities \eqref{e:energy convergence v_j} and \eqref{e:harmonicity} for 
$w\in v+C^\infty_c(B_r\cap\{x_n>0\})$ follows by arguing as in Lemma~\ref{l:decay}. 
A density argument concludes the proof for every $w\in v+W^{1,2}_0(B_r\cap\{x_n>0\})$. 
Finally, being $v\in W^{1,2}(B_r)$, harmonic on $B_r\cap\{x_n>0\}$ and equal to a constant $C$ on $B_r\cap\{x_n<0\}$, $v\in C^0(B_r)$ by elliptic regularity. Therefore, the odd extension $\tilde{v}$ of $v-C$ across $B_r\cap\{x_n=0\}$ is an harmonic function on $B_r$ satisfying 
\[
2\tau^{\sfrac32}\leq2\int_{B_\tau\cap\{x_n>0\}}|\nabla v|^2=
\int_{B_\tau}|\nabla \tilde{v}|^2\leq2\frac{\tau^2}{r^2}\int_{B_r}|\nabla v|^2\leq 2\frac{\tau^2}{r^2}\,,
\]
thanks to \eqref{e:bounds vj}, \eqref{e:contra} and \eqref{e:energy convergence v_j}
(cf. \eqref{e:contradiction tau}). The contradiction follows by the choice $\tau<r^4<r<1$. 
\medskip

\noindent{\bf Case $\param>0$.} The proof in this case is completely analogous to that in 
Lemma~\ref{l:decay}. The necessary changes are related to the arguments outlined above for $\param=0$.
\end{proof}
We are now ready to establish the density lower bound for minimizers of $E_\param$ with Dirichlet 
boundary conditions.
\begin{theorem}\label{t:dlb Dirichlet}
There exists a geometric constant $\epsilon>0$ with the following property.
If $(u,K)$ is a minimizer of $E_\lambda(\cdot,\cdot,\overline{U},g)$ 
with Dirichlet boundary conditions as in \eqref{e:Dirichlet pb}, then 
\begin{equation}\label{e:dlbEg Dirichlet}
\mathcal{H}^1(K\cap B_{\rho_x}(x))\geq\epsilon\rho\qquad
\text{for all $x\in K\cap(\partial U\setminus\Sigma)$, $\rho\in(0,\min\{1,\rho_x\})$\,.}
\end{equation}
\end{theorem}
\begin{proof}
We follow the proof of Theorem~\ref{t:dlb}. If $\param=0$, let 
$x\in K\cap(\partial U\setminus\Sigma)$ 
and choose $\sigma>0$ such that 
\[
\max\{\sigma^{\sfrac 32}E_0(u,K,B_{\rho}(x)),
\textstyle{\frac{\sigma^{\sfrac32}}{\vartheta(\sfrac12)}}
\|\nabla w\|^2_{L^\infty(B_{\rho_x}(x))},
\textstyle{\frac{\sigma^{\sfrac32}}{\vartheta(\sfrac12)}}
\|\varphi_x'\|_{L^\infty((-\sfrac{\rho_x}{2},\sfrac{\rho_x}{2}))}\}<\epsilon\sigma\rho\,,
\]
where $\epsilon:=\varepsilon(\sfrac12)$. To this aim, note that arguing as in the energy 
upper bound, there is a constant $C>0$ depending on $U$ and 
$\|\nabla w\|_{L^\infty(B_{\rho_x}(x))}$ such that $E_0(u,K,B_{\rho}(x))\leq C\rho_x$.
Define the set
\[
\widehat{\Omega}_u:=\{x\in\partial U\setminus\Sigma:\,\mathcal{H}^1(K\cap B_\rho(x))<\varepsilon(\sigma)\rho\quad
\text{ for some }\rho\in(0,\min\{1,\rho_x\})\}\,
\]
and note that $\widehat{\Omega}_u$ is relatively open in $\partial U\setminus\Sigma$, and 
$K\cap(\partial U\setminus\Sigma)\subset\partial U\setminus \widehat{\Omega}_u$. 
Indeed, if $x\in\widehat{\Omega}_u$ one can use Lemma~\ref{l:decay Dirichlet} inductively with $M_1:=\|\nabla w\|_{L^\infty(B_{\rho_x}(x))}$. The first induction step is a consequence of the definition of $\widehat{\Omega}_u$ itself (Lemma~\ref{l:decay Dirichlet} is applied with $\tau=\sigma$), the others of the choice of $\sigma$ 
(Lemma~\ref{l:decay Dirichlet} is applied with $\tau=\sfrac12$). In addition, 
$\partial U\setminus(\Sigma\cup K)\subset\widehat{\Omega}_u$ by elliptic regularity.

Finally, the proof of the case $\lambda>0$ follows analogously by mimicking that  
in Theorem~\ref{t:dlb} for $\lambda>0$ using Lemma~\ref{l:decay Dirichlet} in place of 
Lemma~\ref{l:decay}. 
\end{proof}

\section{Density lower bounds for \texorpdfstring{$SBV$}{SBV} minimizers}\label{ss:dlb SBV}

Results similar to those in the previous section hold for minimizers of the corresponding weak formulations of the problem.
To this aim we define for every $\param\geq 0$, for every Borel subset $B$ of $\Omega$, and 
for every $v\in SBV_{{\rm loc}}(\Omega)$ the (formal) extensions of the functionals $E_\param$ as follows:
\index{Mumford Shah energy SBV@ Mumford Shah energy SBV}\index[simb]{aalE_p@$\widetilde{E}_\param$}\index[simb]{aalE_0@$\widetilde{E}_0$}
\begin{equation}\label{e:tilde Eg}
\widetilde{E}_\param (v,B):=\int_B|\nabla v|^2+\mathcal{H}^1(S_v\cap B)+\param \int_B|v-g|^2\,.
\end{equation}
We say that a function $u\in SBV(\Omega)$ is an (absolute) minimizer of $\widetilde{E}_\param$ 
on $\Omega$, if $\widetilde{E}_\param(u,A)<\infty$ and $\widetilde{E}_\param(u,A)\leq \widetilde{E}_\param(v,A)$ for every $v\in SBV_{{\rm loc}}(\Omega)$ with $\{v\neq u\}\subset\subset A\subset\subset\Omega$, $A$ open. We warn the reader that in literature the terminology local 
minimizer is rather used (cf. \cite{AFP00}), we employ the chosen one for the sake of consistency with Definition~\ref{d:definitions}. 

Dirichlet boundary value problems in the $SBV$ setting are introduced as follows: 
if $U$ is a bounded open set with $C^1$ boundary in $\Omega$, and if $w$ is the boundary datum in \eqref{e:Dirichlet pb}, we consider any function $\widetilde{w}$ in 
$W^{1,1}\cap L^\infty(\mathbb{R}^2\setminus U)$ extending $w$.  
A function $u\in SBV(\mathbb{R}^2)$ such that $u=\widetilde{w}$ $\mathcal{L}^2$-a.e. 
on $\mathbb{R}^2\setminus\overline{U}$ is an absolute minimizer of $\widetilde{E}_\param$ 
on $U$ subject to Dirichlet boundary conditions if
\[
\widetilde{E}_\param(u,\overline{U})=\inf\{\widetilde{E}_\param(v,\overline{U}):\,
v\in SBV(\mathbb{R}^2),\,v=\widetilde{w}\, \text{ $\mathcal{L}^2$-a.e. on }
\mathbb{R}^2\setminus\overline{U}\}\,,
\]
which is clearly equivalent to minimizing
\[
\int_U|\nabla v|^2+\mathcal{H}^1(S_v)+\param\int_U|v-g|^2
\]
on the same class of functions.
The same arguments in Lemma~\ref{l:decay} and Theorem~\ref{t:dlb} can be used to infer 
corresponding density lower bound estimates.
\begin{theorem}\label{t:dlb SBV}
Let $u\in SBV(\Omega)$. There exists a constant $\epsilon>0$ such that 
\begin{itemize}
 \item[(a)] 
 if $u$ is an absolute minimizer of 
 $\widetilde{E}_0$ on $\Omega$, then for all $x\in \overline{S_u}$
\begin{equation}\label{e:dlb tilde E}
\mathcal{H}^1(\overline{S_u}\cap B_\rho(x))\geq\epsilon\rho\qquad
\text{for all $\rho\in(0,\mathrm{dist}(x,\partial\Omega))$.}
\end{equation}
\item[(b)] 
if $u$ is an absolute minimizer of 
$\widetilde{E}_\param$ on $\Omega$, $\param>0$, then for all $x\in \overline{S_u}$
\begin{equation}\label{e:dlb tilde Eg}
\mathcal{H}^1(\overline{S_u}\cap B_\rho(x))\geq\epsilon\rho\qquad
\text{for all $\rho\in(0,\min\{1,\mathrm{dist}(x,\partial\Omega)\})$.}
\end{equation}
\item[(c)] 
if $u$ is an absolute minimizer of $\widetilde{E}_\param$ on $U$ with Dirichlet boundary conditions, 
$\param\geq0$, with $u=\widetilde{w}$ $\mathcal{L}^2$-a.e. on $\mathbb{R}^2\setminus\overline{U}$, 
then for all $x\in \overline{S_u}\setminus\Sigma$
\begin{equation}\label{e:dlb tilde Eg Dirichlet}
\mathcal{H}^1(\overline{S_u}\cap B_\rho(x))\geq\epsilon\rho\qquad
\text{for all $\rho\in(0,\min\{1,\rho_x\})$.}
\end{equation}
\end{itemize}
\end{theorem}
The conclusions in Theorem~\ref{t:dlb SBV} and standard density estimates for measures imply the essential closure of the jump set in the $SBV$ setting (cf. \cite[Theorem~2.56]{AFP00}). 
\begin{corollary}\label{c:essential closure}
The following properties hold.
\begin{itemize}
  \item[(a)] if $u\in SBV(\Omega)$ is a local minimizer of $\widetilde{E}_\param$ on $\Omega$, 
  then
 \[
  \mathcal{H}^1((\overline{S_u}\setminus S_u)\cap\Omega)=0\,.
 \]
\item[(b)] if $u\in SBV(U)$ is a minimizer of $\widetilde{E}_\param$ on $U$ subject to 
Dirichlet boundary conditions, then 
 \[
  \mathcal{H}^1((\overline{S_u}\setminus S_u)\cap\overline{U})=0\,.
 \]
 \end{itemize}
\end{corollary}

\section{Equivalence of \texorpdfstring{$SBV$}{SBV} and classical formulations}\label{ss:strong vs SBV}

The starting point to compare the strong and weak formulations of the problem is the following result due to De~Giorgi, Carriero, and Leaci \cite[Lemma~2.3]{DGCL89} (see also \cite[Proposition~4.4]{AFP00}, and \cite[Lemma~2.3]{CL90} for the Dirichlet problem).
\begin{lemma}\label{l:DGCL}
Let $(v,J)$ be an admissible pair with $v\in L^\infty_{{\rm loc}}(\Omega)$. Then, $v\in SBV_{{\rm loc}}(\Omega)$ and
$\overline{S_v}\cap \Omega\subseteq J$. 
\end{lemma}
{
\begin{proof}
Since $J$ is $1$-rectifiable with finite Hausdorff measure,
for every $j\in\N$ there is a covering of $J$, either finite or countable, with balls $B_{\rho_{i,j}}^{j}$ such that $\rho_{i,j}\leq \sfrac1j$ and 
\[
2\sum_i\rho_{i,j}\leq \mathcal{H}^1(J)+1.
\]
Define $v_j:=v\chi_{\Omega\setminus \cup_iB_{\rho_{i,j}}^j}$, then $v_j\in BV_{{\rm loc}}(\Omega)$ with 
\[
|Dv_j|(A)\leq \int_{A\setminus K}|\nabla v| +2\pi\|v\|_{L^\infty(A)} ( \mathcal{H}^1(J)+1)\,,
\]
for every open set $A$ compactly contained in $\Omega$. 
Moreover, by construction $v_j\to v$ in $L^1_{{\rm loc}}(\Omega)$, so that by the lower semicontinuity of the total variation
$v\in  BV_{{\rm loc}}(\Omega)$, and $S_v\subseteq J$ as $v\in W^{1,2}_{\rm loc}(\Omega\setminus J)$. In turn, being $J$ closed, this yields that $\overline{S_v}\cap \Omega\subseteq J$, and that 
\begin{align*}
|Dv|(A)&=\int_{A\setminus J}|\nabla v|+
|Dv|(A\cap J)=\int_{A}|\nabla v|+ \int_{A\cap S_v}|v^+-v^-|d\mathcal H^1\\
&\leq \int_{A}|\nabla v|+ \int_{A\cap S_v}|v^+-v^-|d\mathcal H^1
\end{align*}
where in the second equality we have used that 
$|Dv|(A\cap J\setminus S_v)=0$. This clearly implies that 
$v\in SBV_{\rm loc}(\Omega)$.
\end{proof}
}
From the latter lemma and the density lower bound, it is not difficult to derive that,
for bounded open sets and when the $L^\infty$ norm of the minimizer can be controlled a-priori, minimizing over pairs $(u,K)$ in the classical formulation or in the $SBV$ formulation is equivalent. A subtle point for our purposes is that this equivalence holds even if we impose Dirichlet boundary conditions on a bounded set $U$ which is sufficiently smooth.

We report next the contributions of De~Giorgi, Carriero, and Leaci in \cite[Theorem~1.1]{DGCL89} and \cite[Remark~4.3]{CL90}.
\begin{proposition}\label{p:SBV implies strong DGCL}
Assume that $\param\in[0,1]$ and $g\in L^\infty(U)$, with $U$ open and bounded, then 
\begin{itemize}
\item[(a)] for every $\param>0$ there exists $u\in SBV\cap L^\infty(U)$ such that 
$\displaystyle \widetilde{E}_\param(u,U)=\inf_{SBV(U)}\widetilde{E}_\param(\cdot,U)$,
and 
\begin{align*}
E_\param(u,\overline{S_u}\cap U,U,g)=\inf\{&E_\param(v,J,U,g):\,\\
&v\in W^{1,2}_{{\rm loc}}(U\setminus K),\,K\subset U\text{ closed, $1$-rectifiable}\}\,.
\end{align*}
(in particular, $(u,\overline{S_u}\cap U)$ is an absolute minimizer for $E_\param$). 
\item[(b)] for every $\param>0$, if $(u,K)$, with $u\in L^\infty(U)$, satisfies
\begin{align*}
E_\param(u,K,U,g)=\inf\{&E_\param(v,J,U,g):\,\\
&v\in W^{1,2}_{{\rm loc}}(U\setminus K),\,K\subset U\text{ closed, $1$-rectifiable}\}\,,
\end{align*}
then $\displaystyle \widetilde{E}_\param(u,U)=\inf_{SBV(U)}\widetilde{E}_\param(\cdot,U)$,
and moreover $K=\overline{S_u}\cap U$.
\item[(c)] for every $\param\geq0$ analogous conclusions as those in items (a) (and (b)) hold
for minimizers (in $L^\infty$) of $E_\param(\cdot,\cdot,\overline{U},g)$ ($\widetilde{E}_\param$) 
on $U$ subject to Dirichlet boundary conditions, provided $U$ has $C^1$ boundary.
\end{itemize}
\end{proposition}
\begin{proof}
An elementary truncation argument yields that the $L^\infty(U)$ norms of minimizing sequences 
of $\widetilde{E}_\param$ are equi-bounded by $\|g\|_\infty$. 
Next, the lower semicontinuity of the functional $\widetilde{E}_\param$ in $SBV$ and the compactness theorem of Ambrosio (cf. \cite[Theorems~4.7 and 4.8]{AFP00}) yield the existence 
of a minimizer $u$ for $\widetilde{E}_\param$ on $SBV(U)$.
From item (a) in Corollary~\ref{c:essential closure} we deduce that $u\in W^{1,2}_{{\rm loc}}(U\setminus\overline{S_u})$, in turn implying that $u\in C^1(U\setminus\overline{S_u})$.
Indeed, $u\in W^{2,p}_{{\rm loc}}(U\setminus \overline{S_u})$ for every $p<\infty$ and solves 
$\Delta u = \param (u-g)$ on each open subset $U'\subset\subset U\setminus \overline{S_u}$ using the outer variation equation \eqref{e:outer} and elliptic regularity theory.
Finally, item (a) in Corollary~\ref{c:essential closure} and Lemma~\ref{l:DGCL} imply that 
$(u,\overline{S_u}\cap U)$ is a ``classical'' minimum for $E_\param$ with 
$\widetilde{E}_\param(u)=E_\param(u,\overline{S_u}\cap U)$. 
This proves all the conclusions in (a).

For what item (b) is concerned, if $\tilde{u}\in SBV(U)$ were such that $\widetilde{E}_\param(\tilde{u};U)<\widetilde{E}_\param(u,U)$, then by item (a) we would conclude that
\[
E_\param(\tilde{u},\overline{S_{\tilde{u}}}\cap U,U,g)=
\widetilde{E}_\param(\tilde{u},U)<\widetilde{E}_\param(u,U)\leq E_\lambda(u,K,U,g)\,,
\]
where in the second inequality we have used Lemma~\ref{l:DGCL}. This is a contradiction. 

We show next that $K = \overline{S_u}\cap U$. Recall that we already noticed that 
$\overline{S_u}\cap U\subseteq K$. Clearly, $\mathcal{H}^1(K\setminus \overline{S_u}) = 0$, since otherwise if $\mathcal{H}^1(K\setminus\overline{S_u})>0$, we would have $E_\param(u,\overline{S_u})<E_\param(u,K)$. 
Thus, as $\overline{S_u}$ is closed by definition, $1$-rectifiable thanks to Corollary~\ref{c:essential closure}, and $u\in C^1(U\setminus\overline{S_u})$, arguing as in item (a) we would get a contradiction.
Finally, assume that $x\in K\setminus \overline{S_u}$ and let $r\in(0,\mathrm{dist}(x,\partial U))$ be such that $B_r(x)$ does not intersect $\overline{S_u}$. Then, we have
\[
\mathcal{H}^1(K\cap B_r(x))\leq \mathcal{H}^1(K\setminus\overline{S_u})+
\mathcal{H}^1(\overline{S_u}\cap B_r(x))=0,
\]
which is a contradiction because of the density lower bound for $K$ in Theorem~\ref{t:dlb}.

The proof of the claims in item (c) are completely analogous using item (b) in Corollary~\ref{c:essential closure}.
\end{proof}
We establish next the analogous statement for absolute minimizers of $E_\param$. 
This property will be used in the proof of Corollary~\ref{c:SBV-convergence} below, 
which, in turn, is a crucial ingredient in the proof of Theorem~\ref{t:minimizers compactness}.
\begin{theorem}\label{t:equivalence-2}
Assume $(u, K)$ is an absolute minimizer of $E_\param$ in some open set $\Omega$, with $u\in L^\infty_{{\rm loc}}(\Omega)$. Then $u\in SBV_{{\rm loc}}(\Omega)$ and 
for every $C^1$ domain $U\subset\subset \Omega$ with the property that $\mathcal{H}^1 (K\cap \partial U)=0$ and let $v\in SBV_{{\rm loc}}(\Omega)$ be a function such that $v=u$  $\mathcal{L}^2$-a.e. on $\Omega\setminus \overline{U}$, then 
\[
\int_U |\nabla v|^2 + \mathcal{H}^1 (S_v \cap \overline{U}) + \param \int_U |v-g|^2 
\geq E_\param (u, K, U, g)\, .
\]
Moreover, $(K\triangle \overline{S_u})\cap\Omega=0$.
\end{theorem}
\begin{proof} 
That $u\in SBV_{{\rm loc}}(\Omega)$ follows immediately from Lemma~\ref{l:DGCL}.
Assume now that the inequality above fails. By Proposition~\ref{p:SBV implies strong DGCL} we can find one 
such $v$ which minimizes the energy in the space of $SBV_{{\rm loc}}(\Omega)$ functions which coincide 
with $u$ on $\mathcal{L}^2$-a.e. on $\Omega\setminus \overline{U}$. 

It follows from item (b) in Corollary~\ref{c:essential closure} that 
$\mathcal{H}^1 ((\overline{S_v}\setminus S_v)\cap \overline{U})=0$, while at the same time $v\in C^1(U\setminus \overline{S_v})\cap C^0(\overline{U}\setminus \overline{S_v})$. 
Indeed, $u\in C^1(\Omega\setminus K)$ and by assumption $U\subset\subset \Omega$ with 
$\mathcal{H}^1 (K\cap \partial U)=0$, therefore the conclusions in Proposition~\ref{p:SBV implies strong DGCL} hold.
If we let $K':= (\overline{S_v}\cap \overline{U}) \cup (K \cap (\Omega\setminus \overline{U}))$ and $u':= v \mathbf{1}_{U\setminus K'} + u \mathbf{1}_{\Omega\setminus (\overline{U}\cup K)}$, we then conclude that 
$K'$ is closed and $u'$ is locally Lipschitz on $\Omega\setminus K'$, while $\nabla u'\in L^2_{{\rm loc}} (\Omega\setminus K')$. In particular, the pair $(u', K')$ would contradict the minimizing property of $(u,K)$. 

The last property follows from Proposition \ref{p:SBV implies strong DGCL} for an invading sequence of domains compactly contained in $\Omega$ and satisfying the assumptions in the statement.
\end{proof}

{
\section{Proof of Corollary \ref{c:normalized}}\label{s:SBV-equivalence}
Fix $(u,K)$ as in the first part of the statement. 
The fact that $u\in SBV$ and that $\overline{S_u}\subset K$ follows from Lemma~\ref{l:DGCL}.
The fact that $K\triangle \overline{S_u}=\emptyset$ for absolute minimizers is
established in item (b) of Proposition~\ref{p:SBV implies strong DGCL}.


Eventually, let $(u,K)$ be a (restricted) minimizer of $E_\param$
with $\mathcal{H}^1 (K\cap U) = 0$ in an open set $U\subset\subset\Omega$. 
Then $\mathcal{H}^1 (\overline{S_u}\cap U) = 0$, therefore
$u\in W^{1,2}_{{\rm loc}}(U)$. 
Hence, $u$ extends to a harmonic function when $\param =0$, resp. to a function in $W^{2,p}_{{\rm loc}}(U)$ when $\param >0$, in view of \eqref{e:outer} and standard elliptic regularity theory.
}

\section{Compactness of bounded minimizers through the \texorpdfstring{$SBV$}{SBV} formulation}\label{ss:compactness through SBV}

As it is customary in the literature which addresses the Mumford-Shah variational problem using the ``classical'' formulation 
through pairs $(u, K)$, the compactness statements which are relevant to the purposes of this book could be proved using the 
``uniform concentration property'' of Dal~Maso, Morel, and Solimini, see \cite{DME92}.
A simple corollary of Theorem~\ref{t:equivalence-2} is the fact that, on any region where there is a suitable $L^\infty$ 
bound for the sequence $(u_j, K_j)$ in Theorem~\ref{t:minimizers compactness}, it converges in the $SBV$ sense to a minimizer 
of $E_0$ which at the same time is also a classical minimizer. 
\begin{corollary}\label{c:SBV-convergence}
Assume $(u_j, K_j)$ is a sequence as in Theorem~\ref{t:minimizers compactness}, let $K$ be the (local) Hausdorff limit of $K_j$, 
$u$ be the local $W^{1,2}$ limit of $u_j$ on the complement of $K$, and let $U$ be any bounded $C^2$ open set satisfying the following properties:
\begin{itemize}
\item[(a)] $K\cap \partial U$ is finite and the cardinality of $K_j\cap \partial U$ is uniformly bounded;
\item[(b)] $\displaystyle\int_{\partial U\setminus K} |\nabla u|^2 < \infty$ and $\limsup_j \displaystyle\int_{\partial U \setminus K_j} |\nabla u_j|^2<\infty$;
\item[(c)] $\|u_j - m_j\|_\infty$ is uniformly bounded for some sequence of constants $m_j$.
\end{itemize}
Then $u$ is an $SBV$ absolute minimizer of $\widetilde{E}_0$ in $U$ with $K\cap U=\overline{S_u}\cap U$, and moreover 
$\nabla u_j \to \nabla u$ strongly in $L^2 (U)$, while $\mathcal{H}^1 (K_j\cap U)\to \mathcal{H}^1 (K\cap U)$.
\end{corollary}
\begin{proof} We give the proof under the assumption that $U$ is the disk $B_1$, the general case is a straightforward modification of the idea using a tubular neighborhood of $\partial U$.
First of all $u$ is certainly an $SBV$ function in $U$ with $\overline{S_u} \cap U \subset K$. Moreover, passing to a subsequence we can assume, without loss of generality, that $u_j-m_j$ converges to $u$ in $L^2 (\partial B_1)$. Using a simple interpolation formula between the functions $u_j-m_j$ and $u$ on the disk $B_1\setminus B_{1-\varepsilon}$ it is possible to find closed sets $J_j$ and functions $w_j\in W^{1,2}_{{\rm loc}} (B_1\setminus B_{1-\varepsilon})$ with the properties that 
\[
\limsup_j \mathcal{H}^1 (J_j) +
\limsup_j \int_{B_1\setminus B_{1-\varepsilon}} |\nabla w_j|^2 \leq C \varepsilon\, 
\] 
and
$w_j  = u_j $ on $\partial B_1$ while $w_j = u+m_j$ on $\partial B_{1-\varepsilon}$. 
This is achieved using the explicit formula
\[
w_j (x) = \tfrac{1-|x|}{\varepsilon} \left[u \left(\tfrac{x}{|x|}\right) +m_j\right]+ \tfrac{|x|- (1-\varepsilon)}{\varepsilon} u_j \left(\tfrac{x}{|x|}\right)
\]
and letting $J_j$ be the union of the segments $\{\param p: \param \in [1-\varepsilon, 1]\}$ where $p$ ranges in the set $S_j := (K\cup K_j)\cap \partial B_1$.
Note indeed that:
\begin{itemize}
\item[(i)] $J_j$ is formed by finitely many segments of length $\varepsilon$, whose number is uniformly bounded in $j$ by assumption (a).
\item[(ii)] The Dirichlet energy of $w_j$ on $B_1\setminus (B_{1-\varepsilon}\cup J_j)$ is estimated by
\[
\frac{C_0}{\varepsilon} \int_{\partial B_1} |u_j -u|^2 + C_0 \varepsilon \int_{\partial B_1\setminus S_j} (|\nabla u_j|^2 + |\nabla u|^2)\, ,
\]
where $C_0$ is a geometric constant. Note that the first term converges to $0$ as $j\uparrow \infty$, while the second is bounded by $C\varepsilon$ for a constant $C$ independent of $j$ by assumption (b).
\end{itemize}
In order to handle the fidelity term, recall that $\|u\|_\infty \leq\sup_j\|u_j -m_j\|_\infty \leq C$ by assumption (c). On the other hand, $|m_j| \leq \|u_j\|_\infty + \|u_j - m_j\|_\infty \leq \frac{C}{r_j^{\sfrac{1}{2}}} \|u^j\|_\infty + \|u_j -m_j\|_\infty$. Recall that $\|u^j\|_\infty$ is uniformly bounded by the assumptions of Theorem \ref{t:minimizers compactness}. A simple rescaling argument shows that $u_j$ is a minimizer of the functional
\[
E_0(\cdot,\cdot,B_1) + \param_j r_j^2 \int_{B_1} \left|\cdot - \tfrac{g_j}{r_j^{\sfrac12}}\right|^2\, . 
\]
However, given the above estimates, we immediately see that 
\[
\param_j r_j^2 \int_{B_1} \left|w_j - \tfrac{g_j}{r_j^{\sfrac12}}\right|^2
\leq C \param_j r_j\, 
\]
vanishes in the limit.
 In particular, we conclude that $u$ must be an absolute minimum of $\widetilde{E}_0$: a competitor $\tilde{u}$ which gains energy (and does not increase the $L^\infty$ norm) would lead to a competitor $\tilde{u}_j$ for $u_j$ by setting $\tilde{u}_j = w_j$ on the corona $B_1\setminus B_{1-\varepsilon}$ and $\tilde{u}_j = \tilde{u}(\frac{\cdot}{1-\varepsilon})$ in $B_{1-\varepsilon}$. Tuning the parameter $\varepsilon$ appropriately would then show that $\tilde{u}_j$ improves the energy $\widetilde{E}_{\lambda_j}$ of $u_j$ for a sufficiently large $j$, contradicting 
 Theorem~\ref{t:equivalence-2}.

This very argument also implies convergence of the energies $\widetilde{E}_{\lambda_j}$ of $u_j$ to $\widetilde{E}_0$ of $u$, which in turn, using the lower semicontinuity theorem of Ambrosio \cite[Theorem~4.7]{AFP00}, implies in addition that 
\[
\nabla u_j \to \nabla u \mbox{ in $L^2(B_1)$ }
\]
and that
\begin{equation}\label{e:length-convergence}
\lim_j \mathcal{H}^1 (S_{u_j}\cap B_1) =  \mathcal{H}^1 (S_u\cap B_1)\, .
\end{equation}
Theorem \ref{t:equivalence-2} also implies that $K_j = \overline{S_{u_j}}\cap U$, so it remains to show that $K = \overline{S_u}\cap U$. However, we already noticed that $\overline{S_u}\cap U\subseteq K$, so, given the density lower bound, 
$\mathcal{H}^1 ((K\setminus \overline{S_u})\cap B_1) = 0$. Assume that $x\in K\setminus \overline{S_u}$ and let $\bar r$ be such that $B_{\bar r} (x)$ does not intersect $\overline{S_u}$. 
By the density lower bound in Theorem \ref{t:dlb SBV} we have 
\[
\liminf_j \mathcal{H}^1 (S_{u_j}\cap B_r (x))>0\, 
\]
for every $r< \bar r$.
On the other hand, the arguments above apply with $B_r (x)$ in place of $B_1$ provided we choose $r< \bar r$ appropriately. Indeed, for $\mathcal{L}^1$-a.e. $r\in(0,\mathrm{dist}(x,\partial B_1))$ the set $B_r (x)$ will satisfy the assumptions (a), (b) and (c) of the lemma. But we then reach a contradiction to \eqref{e:length-convergence}.
\end{proof}

\chapter{Useful results from elementary topology}

This section collects some useful elementary results, mostly of topological nature, which have been used extensively in the notes

\begin{lemma}\label{l:little-loop}
Consider a closed set $K\subset \mathbb R^2$ and let $J$ be a bounded connected component of $K$. Then, for every $\delta>0$ there is a smooth Jordan curve $\gamma$ such that
\begin{itemize}
\item[(a)] ${\rm dist}\, (y, J) < \delta$ for every $y\in \gamma$;
\item[(b)] $\gamma \cap K = \emptyset$;
\item[(c)] $J$ is contained in the bounded connected component of $\mathbb R^2\setminus \gamma$.
\end{itemize}
\end{lemma}

\begin{lemma}\label{l:connessi_per_archi}
Let $K\subset \mathbb R^2$ be a closed connected set with locally finite Hausdorff measure. Then $K$ is arcwise connected. Moreover, for every $x,y\in K$ there is an injective Lipschitz path $\gamma : [0,1]\to K$ such that $\gamma (0) =x$, $\gamma (1) = y$, and its length is minimal among all paths in $K$ connecting $x$ and $y$.
\end{lemma}

\begin{proof}[Proof of Lemma~\ref{l:little-loop}]
Consider the open set $U_\eta := \{y: {\rm dist}\, (y, K) < \eta\}$ and let $V_\eta$ be the connected component of $U_\eta$ which contains $J$. Clearly, each point $y\in \partial V_\eta$ has distance $\eta$ from $K$. We next claim that for $\eta$ sufficiently small $\overline{V}_\eta$ is compact. Indeed first of all observe that there is an open set $Z$ such that $J \subset Z$ and
$\partial Z \cap K = \emptyset$. Next notice that, since $J$ is bounded, it is contained in some open ball $B_R (0)$. If we take $Z' := Z\cap B_R (0)$, then $J\subset Z'$ and moreover 
$\partial Z' \subset (B_R (0) \cap \partial Z) \cup (Z\cap \partial B_R (0))$, hence it does not intersect $K$. Since $\partial Z'$ is compact, it means that there is some positive constant $c_0$ such that ${\rm dist}\, (y, K) \geq c_0$ for every $y\in Z'$. In particular, for $\eta$ sufficiently small, $\partial U_\eta \cap \partial Z' = \emptyset$. Thus $Z'\cap U_\eta$ is an open subset of $U_\eta$ which contains $J$, implying that $V_\eta \subset Z'$.  

Next we claim that, if $\eta$ is sufficiently small, then we have 
\begin{equation}\label{e:V eta}
V_\eta \subset \{z: {\rm dist}\, (z, J) < \frac{\delta}{2}\}\,.
\end{equation}
Indeed, if the latter were not true, then $\{\overline{V}_\eta\}_{\eta > 0}$ would be a nested family of connected closed sets, each of which contains a point $z_\delta\in K$ at distance at least $\frac{\delta}{2}$ from $J$. Moreover, for $\eta$ sufficiently small the sets would be bounded and hence compact. Their intersection $K_\infty$ would then be a compact connected subset of $K$ and since it contains $J$, it must be $J$. On the other hand, any accumulation point of $\{z_\delta\}_{\delta >0}$ would be an element of $K_\infty$ at a positive distance from $J$, which is a contradiction.
 
Consider next $\eta>0$ for which \eqref{e:V eta} holds and 
a standard family of mollifiers $\varphi_\varepsilon$, the function $\psi_\varepsilon := \mathbf{1}_{V_\eta} * \varphi_\varepsilon$ and the open sets $V_{\varepsilon, t}:=\{y : \psi_\varepsilon (y) > t\}$. For $\varepsilon$ sufficiently small we have that:
\begin{itemize}
\item $J\subset V_{\varepsilon, t}$ for every $1> t\geq 0$;
\item $\partial V_{\varepsilon, t} \cap K = \emptyset$ for every $1 > t\geq \frac{1}{4}$;
\item $V_{\varepsilon, t} \subset \{y: {\rm dist}\, (y, J) < \delta\}$ for every $t\geq \frac{1}{4}$.
\end{itemize}
Fix such a small $\varepsilon$ use Sard's theorem to select a $t\in (\frac{1}{4}, 1)$ such that $V_{\varepsilon, t}$ has smooth boundary. Furthermore pick the connected component $W$ of $V_{\varepsilon, t}$ which contains $J$. By smoothness $\partial W$ consists of a finite number of disjoint smooth Jordan curves and by classical differential topology there is one which is uttermost, i.e. such that $W$ is contained in the topological disk bounded by it. The latter is a curve $\gamma$ which satisfies all the requirements of the lemma. 
\end{proof}

\begin{proof}[Proof of Lemma~\ref{l:connessi_per_archi}]
{\bf Step 1} We prove the first claim of the lemma when $K$ is compact. In that case $\mathcal{H}^1 (K)$ is finite. Fix now $\delta>0$ and let $\mathscr{C}:= \{B_{r_i} (x_i)\}_{i\in \mathbb N}$ be a cover of $K$ such that $\sup_i r_i < \delta$ and $\sum_i 2 r_i < \mathcal{H}^1 (K)+\delta$. By compactness, we can assume, without loss of generality that $\mathscr{C}$ is finite and we can also assume that $B_{r_i} (x_i)\cap K\neq\emptyset$ for all $i$. A chain of $\mathscr{C}$ is given by a choice of balls in $\mathscr{C}$ with radii $\{r_{i(j)}\}_{j\in \{1, \ldots N\}}$ where the $i(j)$ are all distinct and $B_{r_{i(j)}}\cap B_{r_{i(j+1)}} \neq \emptyset$. We say that $B_{r_K} (x_K)$ and $B_{r_J} (x_J)$ are chain-connected if there is a chain such that $K = i(1)$ and $J = i(N)$. Assume now, without loss of generality, that $B_{r_1} (x_1)$ contains $x$ and let $\mathscr{C}'\subset \mathscr{C}$ be the set of balls $B_{r_i} (x_i)$ which are chain connected to $B_{r_1} (x_1)$. $\mathscr{C}'$ must coincide with $\mathscr{C}$ otherwise the two open sets 
\begin{align}
U &:= \bigcup_{i\in \mathscr{C}'} B_{r_i} (x_i)\\
V &:= \bigcup_{i\in \mathscr{C}\setminus \mathscr{C}'} B_{r_i} (x_i)
\end{align}
would be disjoint and would disconnect $K$. Upon reindexing our balls we can thus assume that $\{B_{r_i} (x_i)\}_{i\in \{1, \ldots , N\}}$ is a chain such that $x\in B_{r_1} (x_1)$ and $y\in B_{r_N} (x_N)$. Set $z_0=x$, $z_N = y$, and choose $z_i \in B_{r_i} (x_i)\cap K$ for every other $i$. Consider then the piecewise linear curve consisting of joining the segments $[z_i, z_{i+1}]$. Since $|z_{i+1}-z_i|\leq 2 (r_{i+1} + r_i)$, such curve has length at most $2 \mathcal{H}^1 (K) + 2\delta$. Observe moreover that, since $r_i < \delta$ for all $i$, each point of the curve has distance at most $2\delta$ from $K$.

Let now $\delta:= 1/j$ and let $\gamma_j: [0,1] \to \mathbb R^2$ be a constant speed parametrization of the latter curve. It turns out that $\Lip (\gamma_j)\leq 2 \mathcal{H}^1 (K) +2$ and thus by Ascoli-Arzel\`a we can extract a subsequence converging uniformly to a Lipschitz curve $\gamma:[0,1]\to K$ such that $\gamma (0) =x$ and $\gamma (1) = y$. 

\medskip

{\bf Step 2} We next prove the first claim for $K$ closed. We fix $x,y\in K$ and we seek for an arc connecting $x$ and $y$ in $K$. By Step 1 it suffices to show that for a sufficiently large $R>0$ the points $x$ and $y$ must be contained in the same connected component of $K\cap \overline{B}_R$. Now, for each $N\in \mathbb N$ with $N\geq \max \{|x|, |y|\}$ let $K_N$ be the connected component of $K\cap \overline{B}_N$ which contains $x$ and assume by contradiction that $y\not\in K_N$ for every $N$. We set $\tilde{K}:= \bigcup_N K_N$. The latter set does not contain $y$: if we can show that it is at the same time open and closed in $K$, we have contradicted the connectedness of $K$.

First of all observe that, by Step 1, $K_N \subset K_{N+1}$, which in turn implies that $\tilde{K}$ is open in $K$. We next claim that:
\begin{itemize}
\item[(S)] for every fixed $R> |x|+1$ there is $N(R)> R$ such that
$K_{N'}\cap \overline{B}_R = K_{N(R)} \cap \overline{B}_R$ for all $N'> N(R)$.
\end{itemize}
This would imply that 
$\overline{B}_R\cap \tilde{K} = \overline{B}_R \cap K_{N(R)}$ is closed, in turn implying that $\tilde{K}$ is closed.

If (S) is false for some $R$, then we can find a monotone sequence $N_k\geq R$ such that $K_{N_{k+1}} \cap \overline{B}_R$ is strictly larger than $K_{N_k}\cap \overline{B}_R$. In particular, there is a point $x_{k+1}\in K_{N_{k+1}}\cap \overline{B}_R$ which is not contained in $X_{N_k}\cap \overline{B}_R$. Let $\gamma_k: [0,1]\to K_{N_{k+1}}$ be a curve which connects $x$ to $x_{k+1}$ in $K_{N_{k+1}}$. Such curve cannot be contained in $\overline{B}_{R}$ and cannot intersect $K_{N_k}$, otherwise $x_{k+1}$ would belong to $K_{N_k}$. Let now $s_k$ be the smallest positive number such that $|\gamma_k (s_k)| = R$. It then turns out that the arc $\gamma_{k+1} ([0,s_k])$ is contained in $K_{N_{k+1}}\cap \overline{B}_{R}$, but it has empty intersection with $K_{N_k}\cap \overline{B}_{R}$. Since such arc has length at least $1$ (recall that $|x|\leq R-1$), we conclude
\[
\mathcal{H}^1 (N_{k+1}\cap \overline{B}_{R}) \geq \mathcal{H}^1 (N_k \cap \overline{B}_R)+1\, .
\]
Letting $k\uparrow \infty$ we would conclude $\mathcal{H}^1 (K\cap \overline{B}_{R}) = \infty$.

\medskip

{\bf Step 3} To prove the last claim, fix $x,y\in K$. By Step 1 and 2 we know the existence of a Lipschitz curve $\gamma:[0,1]\to K$ joining $x$ and $y$ in $K$. Let $L (\gamma):= \int |\dot \gamma (t)|\, dt$ and assume without loss of generality that it is parametrized at constant speed. Consider now the set $\mathscr{S}$ of Lipschitz curves $\gamma:[0,1]\to K$ joining $x$ and $y$ and parametrized at constant speed $L (\gamma)$. If $L_0$ is the infimum of $L (\gamma)$ for $\gamma \in \mathscr{S}$, then there is indeed a curve $\bar\gamma$ which attains it: this follows easily from Ascoli-Arzel\`a and the fact that every Lipschitz curve can be reparametrized to constant speed. $\bar\gamma$ is easily seen to be injective, otherwise it could not be a minimizer of the functional $L$.
\end{proof}

\chapter{Proof of Theorem~\ref{t:minimizers compactness}}

With the
tools in Appendix~\ref{a:strong vs SBV} at hand, the conclusions of Theorem~\ref{t:minimizers compactness} would be immediate as soon as we had at our disposal some apriori $L^\infty$ bound. This would be the case if, for instance, the complement of $K$ were connected. However in many situations this is certainly not the case. For this reason we include a detailed proof with the necessary workarounds. 

\begin{proof}[Proof of Theorem~\ref{t:minimizers compactness}]
{\bf Step 0}
Assume $(u_j, K_j)$ is as in Theorem~\ref{t:minimizers compactness} and let $K$ be the local Hausdorff limit of the sequence $K_j$. Then 
\begin{equation}\label{e:twosided-Hausdorff}
\frac{1}{C} \limsup_{j\to\infty} \mathcal{H}^1 (K_j\cap B_{r} (x)) \leq \mathcal{H}^1 (K\cap \overline{B}_r (x)) \leq 
C\liminf_{j\to\infty} \mathcal{H}^1 (K_j\cap B_{2r} (x))
\end{equation}
for every ball $B_r (x) \subset \mathbb R^2$, which is a direct consequence of the density lower bound. 

Fix a positive $\delta\in(0,r]$ and consider any covering of disks $\{B_{r_k} (x_k)\}$ of $K\cap \overline{B}_r (x)$ with $\sup_k r_k \leq \delta$ and $x_k\in K\cap \overline{B}_r (x)$. By compactness extract a finite subcover and using Vitali's covering Lemma, let $\{B_{r_\ell} (x_\ell)\}$ be a subfamily of pairwise disjoint disks such that $\{B_{5r_\ell} (x_\ell)\}$ is still a cover. Since the family is finite, for every $j$ large enough we can find $y_\ell^j\in K_j\cap B_{r_\ell/2} (x_\ell)$. In particular, using the density lower bound (which under our assumption holds with a uniform constant independent of $j$) we have
\begin{align*}
\mathcal{H}^1_\delta (K\cap B_r (x) ) &\leq 10 \pi \sum_\ell r_\ell \leq 
C \sum_\ell \mathcal{H}^1 (K_j\cap B_{r_\ell/2} (y_\ell^j) )\\
&\leq C \mathcal{H}^1 (K_j\cap B_{2r} (x))\, .
\end{align*}
Since $\delta$ is arbitrary, this implies the right inequality in \eqref{e:twosided-Hausdorff}. 
Next, fix $\delta>0$ and let $\{E_k\}$ be a family of closed sets covering $K\cap \overline{B}_r (x)$ such that
\[
\sum_k \textrm{diam}\, (E_k) \leq \mathcal{H}^1 (K \cap \overline{B}_r (x)) + \delta\, .
\]
Let $r_k:= \textrm{diam}\, (E_k)$ and consider disks $B_{2r_k} (x_k)$ containing $E_k$. By Hausdorff convergence, for $j$ large enough $K_j \cap B_{r} (x) \subset \bigcup_\ell B_{2r_\ell} (x_\ell)$. We can thus estimate
\[
\mathcal{H}^1 (K_j\cap B_{r} (x)) \leq \sum_\ell \mathcal{H}^1 (K_j \cap B_{r_k} (x_k))
\leq 4\pi \sum_k r_k \leq C \mathcal{H}^1 (K\cap \overline{B}_r (x)) + C \delta\, .
\]
The arbitrariness of $\delta$ completes the proof. 

\medskip

{\bf Step 1}
Our 
second step is to show that $K$ is rectifiable and we first give the argument for $(u_j,K_j)$ absolute minimizers.
Assume that $W\subset\subset U$ is an open set whose closure is a closed topological disk with the property that $\partial W \cap K = \emptyset$. By a standard approximation theorem there is a second topological disk $W \subset V\subset\subset U$ with $\partial V \cap K = \emptyset$ and which has smooth boundary. Then for a sufficiently large $j$ we have as well that $\partial V \cap K_j=\emptyset$. Set $m_j := \min_{\partial V} u_j$ and $M_j:= \max_{\partial V} u_j$ and use the PDE $\Delta u_j = \param_j(u_j- g_j)$ in a neighborhood of $\partial V$ (independent of $j$) to conclude that $M_j - m_j$ is bounded uniformly independently of $j$. Moreover, by the maximum principle in Lemma~\ref{l:maximum}  (b) it turns out that $\min\{-\|g_j\|_{\infty} , m_j\} \leq u_j|_V \leq \max\{M_j , \|g_j\|_{\infty}\}$. 
In particular, $u_j - m_j$ has a uniform $BV$ bound on $V$ and, up to subsequences, converges to an $SBV$ function $\tilde{u}$ by the $SBV$ closure theorem \cite[Theorem~4.7]{AFP00}. We can then use Corollary \ref{c:SBV-convergence} and conclude that $\overline{S_{\tilde{u}}}=K$ and coincides with $S_{\tilde{u}}$ up to an $\mathcal{H}^1$-null set, thus proving rectifiability.

We next argue that in fact all of $K$ is rectifiable. Let $K^\sharp$ be the union of the connected components of $K$ which are singletons. If $y\in K^\sharp$, then there is a closed topological disk $D$ with smooth boundary and containing $y$ in the interior and such that $\partial D\cap K =\emptyset$ (cf. Lemma~\ref{l:little-loop}). We can therefore apply the argument above to conclude that $K^\sharp$ is rectifiable. Consider now $K\setminus K^\sharp$. Then every $y\in K\setminus K^\sharp$ is contained in a nontrivial connected component of $K$: since each such component has positive length, there are countably distinct ones. As it is well known each connected closed set with finite Hausdorff measure is rectifiable, this completes the proof of our claim (cf. \cite[Theorem 4.4.7]{AmbTil}).

In the case of restricted minimizers we consider again the situation in which $W\subset\subset U$ is  an open set whose boundary is a closed topological disk with the property that $\partial W \cap K = \emptyset$. Choose $V$ as above and denote by $N (j)$ the number of connected components of $V\cap K_j$. If $\bar N := \liminf_j N (j) < \infty$, then $K\cap V$ consists of at most $\bar N$ components and it is therefore rectifiable. If $\lim_j N (j) = \infty$, it is not difficult to see that we can apply the same argument above, i.e. the $SBV$ function $\tilde{u}$ is a minimizer. The reason is that any $SBV$ function $\tilde{u}$ can be approximated in the Mumford-Shah energy with an $SBV$ function whose jump set is closed and has a finite number of connected components, see for instance \cite{DME92}. 

\medskip

{\bf Step 2} We next wish to show that
\[
\liminf_{j\to \infty} \mathcal{H}^1 (K_j\cap A) \geq \mathcal{H}^1 (K\cap A)
\]
for every open set $A$. Consider the measures $\mu_j (E) := \mathcal{H}^1 (K_j\cap E)$, $E$ Borel, and assume that, up to subsequences, 
$\mu_j\rightharpoonup^\star \mu$ for some measure $\mu$. Observe that, by Step 1,
\[
C^{-1} \mathcal{H}^1 (K\cap E) \leq \mu (E) \leq C \mathcal{H}^1 (K\cap E) \qquad \mbox{for every Borel $E$.}
\]
We thus have
\[
\mu (K\cap E) = \int_{K\cap E} \theta (x) d\mathcal{H}^1 (x)
\]
for some Borel function $\theta$ taking values in $[C^{-1}, C]$. Our claim is thus equivalent to $\theta (x) \geq 1$ for $\mathcal{H}^1$-a.e. $x\in K$. 
Assume not, since $K$ is rectifiable we can choose a point $x\in K$ where the approximate tangent to $K$ exists, the $1$-dimensional density of $K$ 
equals $1$ and moreover
\[
\lim_{\rho\downarrow 0} \frac{\mu (B_\rho (x))}{2\rho} = \theta (x) < 1\, .
\]
Using the density lower bounds for $K$, it then follows that the rescaled sets $K_{x, \rho} := \{\frac{y-x}{\rho}: y\in K\}$ converge locally in the Hausdorff distance to a $1$-dimensional subspace $\ell$ of $\mathbb R^2$. Without loss of generality we can assume that $\ell = \{(x_1, 0): x_1\in \mathbb R\}$.
 In addition, by taking a diagonal sequence, we can assume that
$\tilde{K}_j := (K_j)_{x, \rho_j}$ has the following properties:
\[
\lim_j \mathcal{H}^1 (\tilde{K}_j \cap B_1) = 2 \theta (x) <2\, .
\]
Introducing the appropriate rescalings of the functions $u_j$ we then have the following situation:
\begin{itemize}
\item $(\tilde{u}_j, \tilde{K}_j)$ are minimizers of the Mumford-Shah functional in $B_2$ with appropriate fidelity functions $\tilde{g}_j$ and fidelity constants $\tilde\param_j$;
\item $\tilde{K}_j$ converges in the Hausdorff sense to the line $\{x_2=0\}$;
\item $\mathcal{H}^1 (\tilde{K}_j\cap B_1) \leq 2\theta < 2$ for some $\theta$ and for all $j$;
\item $\tilde \param_j \|\tilde{g}_j\|^2_\infty \to 0$.
\end{itemize}
Observe that the following holds:
\begin{itemize}
    \item[(R)] if we replace the sequence $\rho_j$ with any sequence $\tilde{\rho}_j$ such that $0 < \liminf_j \frac{\rho_j}{\tilde{\rho}_j} \leq \limsup_j \frac{\rho_j}{\tilde{\rho}_j} < \infty$, then all the properties above remain true. 
\end{itemize}
Now, for each $\gamma$ sufficiently small compared to $2-2\theta$,  
it is easy to see that there is $\rho\in (\gamma, 1)$ and a subsequence such that
\begin{align*}
&\lim_{\ell\to \infty} \int_{\partial B_\rho} |\nabla \tilde{u}_{j_\ell}|^2 < \infty\\
&\lim_{\ell\to \infty} \mathcal{H}^0 (\tilde{K}_{j_\ell}\cap \partial B_\rho) \leq 2\,.
\end{align*}
On the other hand, by possibly changing the sequence $\rho_j$ to a new sequence $\tilde{\rho}_j$ satisfying (R) above, and after extracting a further subsequence, we can assume that $\rho \in (1-\gamma, 1)$.

Moreover, again by a Fubini argument, if $\gamma$ is sufficiently small compared to $2-2\theta$, we can find a $t\in (-(1-\gamma), 1-\gamma)$ such that
\begin{align*}
&\lim_{\ell\to\infty} \int_{-1}^1 |\nabla \tilde{u}_j (t, x_2)|^2\, dx_2 < \infty\\
&\tilde{K}_{j_\ell} \cap \{(t, x_2): |x_2|\leq 1\}= \emptyset \quad \forall \ell\, .
\end{align*}
In particular, from this, it is again simple to see that $M_\ell - m_\ell := \sup_{\partial B_\rho} \tilde{u}_{j_\ell} - \inf_{\partial B_\rho} \tilde{u}_{j_\ell}$ is uniformly bounded in $\ell$. We thus conclude that
$\tilde{u}_{j_\ell}$ converges, up to subsequences, to an $SBV$ function $\tilde{u}$. Again the latter is a minimizer of the Mumford-Shah functional and $K\cap B_\rho$ must be the closure of $J_{\tilde{u}}$ and 
\[
2 (1-\gamma) \leq 2\rho = \mathcal{H}^1 (K\cap B_\rho) \leq \liminf_{\ell\to\infty} \mathcal{H}^1 (\tilde{K}_{j_\ell} \cap B_\rho) \leq 2\theta
\, .
\]
Since $\gamma$ is arbitrary, the latter is a contradiction.

In the case of restricted minimizers we again consider the number of connected components $N (j)$ of $\tilde{K}_j\cap B_\rho$. If the latter goes to infinity, we see that, as in the previous step, $\tilde{u}$ is a minimizer of the Mumford-Shah functional. Otherwise there is a uniform upper bound on the number of connected components and then the inequality
\[
\mathcal{H}^1 (K\cap B_\rho) \leq \liminf_{\ell\to\infty} \mathcal{H}^1 (\tilde{K}_{j, \ell} \cap B_\rho) 
\]
follows from the classical Golab's theorem (cf. \cite[Theorem 4.4.17]{AmbTil}).

\medskip

{\bf Step 3} Next observe that 
\[
\lim_{j\to \infty} \int_O |\nabla u_j|^2 = \int_O |\nabla v|^2
\]
for every open set $O\subset\subset U\setminus K$. The latter in fact is a consequence of the Hausdorff convergence of $K_j$ to $K$ and of standard regularity properties of harmonic functions. 

Next, if we fix any open set $O\subset U$, choosing a sequence $O_\ell\uparrow O\setminus K$ of open sets with $\overline{O}_\ell \cap K=\emptyset$ we easily conclude
\[
\int_{O\setminus K} |\nabla v|^2 = \lim_{\ell\to\infty} \int_{O_\ell\setminus K} |\nabla v|^2 \leq \liminf_{j\to\infty} \int_O |\nabla u_j|^2\, .
\]
In particular, from the last inequality and Step 3, in order to conclude point (i) in the statement of the theorem we just need to show
\[
\limsup_{j\to\infty} \left(\int_{O\setminus K_j} |\nabla u_j|^2 +\mathcal{H}^1 (K_j\cap O)\right)\leq
\int_{O\setminus K} |\nabla v|^2 + \mathcal{H}^1 (K\cap O)
\]
under the assumption that $\mathcal{H}^1 (\partial O\cap K) = 0$. Assume that the inequality fails and fix a subsequence, not relabeled, for which the limsup on the left is a limit. After possibly extracting a further subsequence, we can assume that the measures $\mu_j$ defined through
\[
\mu_j (E) := \int_{E\setminus K_j} |\nabla u_j|^2 +\mathcal{H}^1 (K_j\cap E)
\]
converge weakly$^\star$ to some measure $\mu$ and so far we can conclude that
\begin{align*}
\mu (E) &= \int_{E\setminus K} |\nabla v|^2 \qquad \mbox{if $E$ is Borel and $K\cap E=\emptyset$.}\\
\mu (E) &\geq \mathcal{H}^1 (K \cap  E) \qquad \mbox{if $E$ is Borel and $E\subset K$}\\
\mu (F) &> \mathcal{H}^1 (K \cap F) \qquad \mbox{for some $F\subset K$ Borel}.
\end{align*}
Note however that from the upper bound \eqref{e:upper bound} we immediately conclude $\mu (B_r (x))\leq 2\pi r$ for every disk $B_r (x)\subset U$. 
In particular
\[
\mu (E) = \int_{E\setminus K} |\nabla v|^2 + \int_{K\cap E} \theta(x)\, d\mathcal{H}^1 (x)\,,
\]
for some density $\theta$ with $1\leq \theta \leq \pi$, but which must be strictly larger than $1$ on a set of positive $\mathcal{H}^1$ measure. Arguing as in the previous step, we pick a Lebesgue point $x\in K$ for $\theta$ with respect to the measure $\mathcal{H}^1$ and we can assume this is a point where the approximate tangent to the rectifiable set $K$ exists. We can thus use the procedure in the previous step to produce a new sequence $(\tilde{u}_j , \tilde{K}_j)$ in $B_3$ with corresponding limits $\tilde{v}$ and $\tilde{K}$ and corresponding measures $\tilde{\mu}_j$ converging to $\tilde{\mu}$, with the additional properties that $\tilde{K}$ is a straight segment and the corresponding density $\tilde{\theta}$ is a constant strictly larger than $1$. In order to simplify our notation we drop the $\tilde{}$, and we assume that the segment is $\sigma_0:=\{(x_1, 0): |x_1|< 2\}$. Next we choose a vanishing sequence $\varepsilon_j$ with the property that
\begin{align}
&K_j \cap B_{2} \subset \{|x_2|\leq \varepsilon_j\} \label{e: Kj trapped}\\
&\lim_{j\to \infty} \mu_j ((a_j,b_j)\times (-\varepsilon_j, \varepsilon_j)) = \theta (b-a)
\label{e: massa muj}
\end{align}
for sequences $a_j\to a$ and $b_j \to b$ in $[-2,2]$ defined as follows:
for each $j$ choose two points $a_j\in [-2, -\frac32]$ and $b_j \in [\frac32,2]$ with the property that, upon setting
$L_j:= ([a_j, a_j +\varepsilon_j)\cup (b_j -\varepsilon_j, b_j]) \times [-1,1]$, then
\begin{equation}\label{e:piccolezza Lj}
\mu_j (L_j) \leq C \varepsilon_j \mu_j ([-2,2]\times [-1,1]) \leq 6 C \varepsilon_j\, ,
\end{equation}
where $C$ is a geometric constant independent of $j$. 
Construct then the following Lipschitz deformation $\Phi_j:Q_j\to Q_j$, $Q_j := [a_j, b_j]\times [-1,1]$, 
defined as $\Phi_j (x_1, x_2):= (x_1, \varphi_j (x_1, x_2))$ where, upon setting 
$f_j (x_1):= \min\{|x_1-a_j|, |b_j-x_1|, \varepsilon_j\}$, the second component is given by
\begin{align}
\varphi_j (x_1, x_2) =
\left\{
\begin{array}{ll}
0 \qquad &\mbox{if $x\in \Sigma_j:= \{|x_2|\leq f_j (x_1)\}$}\\
\frac{x_2 - f_j (x_1)}{1- f_j (x_1)} \qquad &\mbox{if $x\in T_j^+:=\{f_j (x_1) < x_2\leq 1\}$}\\
\frac{x_2 + f_j (x_1)}{1- f_j (x_1)} \qquad &\mbox{if $x\in T_j^- :=\{-1\leq x_2 < -f_j(x_1)\}$}
\end{array}
\right.
\end{align}
(for a similar construction see the definition of the map $\Phi_j$ in the proof of Lemma~\ref{l:decay2}).
Note that $\Phi_j$ is the identity on $\partial Q_j$. Moreover, if we introduce the regions $Q_j^\pm:= Q_j\cap \{\pm x_2 > 0\}$, then 
\begin{itemize}
\item$\Phi_j$ is a bi-Lipschitz map of $T^\pm_j$ onto $Q_j^\pm$, with a uniform bound of the Lipschitz constant of both the map and its inverse; 
\item $\Phi_j$ maps the rectangles $R_j^\pm:=[a_j+\varepsilon_j,b_j-\varepsilon_j]\times\{\varepsilon_j<\pm x_2\leq 1\}$ onto the rectangles  
$[a_j+\varepsilon_j,b_j-\varepsilon_j]\times\{0<\pm x_2\leq 1\}$
and in these regions both $\Lip (\Phi_j)$ and $\Lip (\Phi_j^{-1})$ are bounded by $1+ C \varepsilon_j$ for a universal 
constant $C$;
\item The restriction of $\Phi_j$ on $\Sigma_j$ is the orthogonal projection onto the horizontal axis.
\end{itemize}
Let thus $(v_j, J_j)$ be the pair:
\begin{itemize}
\item $J_j = ([a_j, b_j]\times \{0\}) \cup \Phi_j (K_j\setminus \Sigma_j)$;
\item $v_j = u_j \circ \Phi_j^{-1}$ on $Q_j \setminus J_j$.
\end{itemize}
It easy to estimate the Mumford-Shah energy of $(v_j, J_j)$ in the rectangle $Q_j$ as
\begin{align*}
E_j &:= \int_{Q_j\setminus J_j} |\nabla v_j|^2 + \mathcal{H}^1 (J_j\cap Q_j)\\
&\leq (1+C\varepsilon_j)^3 \int_{R_j^+\cup R_j^-}
|\nabla u_j|^2 + (b_j-a_j - 2\varepsilon_j) + C \mu_j (L_j)\\
&\leq (1+C\varepsilon_j)^3 \int_{Q_j\setminus K_j}|\nabla u_j|^2 + (b_j-a_j - 2\varepsilon_j) + C \mu_j (L_j)\, .
\end{align*}
Then from \eqref{e: massa muj}, \eqref{e:piccolezza Lj} and the convergences $b_j\to b$ and $a_j\to a$, we easily conclude that
\[
\lim_{j\to\infty}  \left(\int_{Q_j\setminus K_j} |\nabla u_j|^2 + \mathcal{H}^1 (K_j\cap Q_j) - E_j \right)
\geq  (\theta-1) (b-a)  
\]
contradicting the minimality of $(u_j, K_j)$ for $j$ large enough, recalling that $\theta>1$ 
and considering that after rescaling the fidelity terms converge to $0$ for our competitors.

Note that the competitor does not increase the number of connected components and therefore the argument remains valid for restricted minimizers.

\medskip

{\bf Step 4} We now prove property (ii) (and come back later to \eqref{e:varifold-convergence} to complete the proof of (i)). Fix thus an open set $O$ and a competitor $(w,J)$ as in the statement and observe that the competitor keeps the same property if we replace $O$ with a larger open set. In particular, without loss of generality, we can assume that $O$ has a smooth boundary. Consider now a tubular neighborhood of $\partial O$ where the distance function to $\partial O$ gives a smooth foliation. By a Fubini-type argument, we can change slightly the set and assume that for a subsequence, not relabeled, we have
\begin{align*}
&\sup_j \left( \int_{\partial O\setminus K_j} |\nabla u_j|^2 + \mathcal{H}^0 (K_j \cap \partial O)\right) < \infty\\
&\int_{\partial O\setminus K} |\nabla v|^2 + \mathcal{H}^0 (K\cap \partial O) < \infty\, .
\end{align*}
Denote by $O_\delta$ the set $O_\delta:=\{x\in O : \dist (x, \partial O) > \delta\}$. We leave as an exercise to the reader to show that for each $\delta$ sufficiently small there is a diffeomorphism $\Phi_\delta$ of $O_\delta$ onto $O$ with the property that $\|D\Phi_\delta - {\rm Id}\|_{C^0} + \|D\Phi_\delta^{-1} - {\rm Id}\|_{C^0}$
is infinitesimal as $\delta\to 0$. Enumerate the connected components of $U\setminus K$ which intersect $\partial O$ and denote them by $U_1, \ldots , U_N$. 
We leave to the reader to prove the simple fact that, by our definition of the $v^i$'s, there is a choice of constants $p_{ik}$'s  with the property that, if we define the map $\hat{v}_j$ on each $U_i$ to be equal to $v^i+p_{ik}$, then there is a suitable sequence $\delta_j\downarrow 0$ and a suitable interpolating pair $(z_j, J_j)$ on $O\setminus O_{\delta_j}$ such that
\begin{itemize}
\item $z_j = u_j$ on $\partial O$ and $J_j \cap \partial O = K_j \cap \partial O$;
\item $z_j = \hat{v}_j\circ \Phi_{\delta_j}$ on $\partial O_{\delta_j}$ and $J_j \cap \partial O_{\delta_j} = \Phi_{\delta_j}^{-1} (\partial O \cap J)$;
\item $\int_{O\setminus O_{\delta_j}} |\nabla z_j|^2 + \mathcal{H}^1 (J_j \cap(O\setminus O_{\delta_j}))  \to 0$ as $j\uparrow \infty$.
\end{itemize}
Now, the technical condition in item (ii) of Definition~\ref{d:competitors} on the competitor makes sure that, if $x, y\in \partial O$ belong to two distinct connected components $\Omega_i$ and $\Omega_k$ of $U\setminus K$, then $x,y$ are in distinct connected components of $U\setminus J$. Therefore, given $q\in U\setminus J$, we define $\hat{w}_j (q) := w(q) +p_{ik}$ if the connected component of $U\setminus J$ containing $q$ intersects $\partial O$ in $\Omega_i$, while we define $\hat{w}_j (q):= w(q)$ if the connected component of $U\setminus J$ does not intersect $\partial O$. We are now ready to define $(z_j, J_j)$ inside $O_{\delta_j}$ as well: we put $J_j \cap O_{\delta_j} := \Phi_{\delta_j}^{-1} (J \cap O)$ and we define $z_j := \hat{w}_j \circ \Phi_{\delta_j}$. Observe now that, if the energy of $(w, J)$ were smaller than that of $(v,K)$, then for sufficiently large $j$ the energy of $(z_j, J_j)$ in $O$ would be smaller than that of $(u_j, K_j)$, which is a contradiction.

In the case of restricted minimizers we split again into two situations, one in which the number of connected components in $O$ converges to infinity (in which case it can be shown that the limit is an absolute minimizer) or the one in which the number of connected components stays bounded. Since the number of connected components of the limit is at most the limit of the number of connected components of the approximating sequence, the limit is a restricted minimizer.

\medskip

{\bf Step 5} We next prove (iii). To this aim we first draw some conclusions from (ii). First of all consider an open set $O\subset\subset U$ and assume 
there is a connected component $A$ of $O\setminus K$ whose closure does not intersect $\partial O$. Now observe that, if we change the value of the constant in $A$, 
any topological competitor for the new modified function is an topological competitor for the old function, and viceversa. We can thus assume without loss of generality 
that $v$ is identically equal to $0$ on any such connected component, and ultimately on any bounded connected component of $U\setminus K$ whose closure does not have 
some portion of the boundary in common with $\partial U$.

Assume now that $O$ intersects $K$ in a finite number of points, that its boundary $\partial O$ is regular and that the restriction of $v$ on each connected component of 
$\partial O \setminus K$ belongs to $W^{1,2}$. Observe that all these properties imply the boundedness of $v$ on $\partial O\setminus K$. If $O'$ is a connected 
component of $O\setminus K$ which is not compactly contained in $O$, $v|_{O'}$ minimizes the Dirichlet energy among all $W^{1,2}$ functions which agree with it on 
$O'\cap \partial O$. By the maximum principle we conclude that $v$ is bounded on every such $O'$. Since we have normalized $v$ to be $0$ on the remaining ones 
(which do not ``touch'' $\partial O$), we conclude that $v$ is bounded in $O$. Ultimately, since for every open $A\subset\subset U$ we can easily find a slightly
larger $O$ with all the properties above, we conclude that $v$ is locally bounded. 

Consider now $\Omega_{\mathscr{A}}$ as in the statement of point (iii). It then turns out that, for every $A\subset\subset \Omega_{\mathscr{A}}$, the function $u_{\mathscr{A}}$ is bounded and has bounded variation in $A$. Using analogous reasonings, it is not difficult to see that, upon subtraction of a suitable constant 
$c_j$, we conclude as well that $u_j - c_j$ is uniformly bounded and has uniform bound on its $BV$ norm. Moreover $u_j - c_j$ converge to $v-c$ for some suitable 
constant $c$. In particular the conclusion of point (iii) that $u_{\mathscr{A}}$ is a minimizer in $A$ follows from the SBV existence theory of minimizers, as already argued in Step 2.

In the case of restricted minimizers we argue similarly.

\medskip

{\bf Step 6} In order to show \eqref{e:varifold-convergence}, and hence complete the proof of (i), consider the measures $\alpha_j$ on $\mathbb P^1 \mathbb R\times \mathbb R^2$ given by
\[
\int \varphi (\pi, x) d\alpha_j (\pi, x) := \int_{K_j} \varphi (T_x K_j,x) d\mathcal{H}^1 (x)\, ,
\]
for every $\varphi\in C_c(\mathbb{P}^1\R\times U)$. From now on we use the notation $\alpha_j = \delta_{T_x K_j} \otimes \mathcal{H}^1 \restr K_j$.
The convergence in \eqref{e:varifold-convergence} is equivalent to say that $\alpha_j$ converges weakly$^\star$ to the measure 
$\alpha = \delta_{T_x K} \otimes \mathcal{H}^1 \restr K$.
First of all, up to subsequences we can assume that $\alpha_j \rightharpoonup^\star \beta$ for some measure $\beta$. Secondly, by what proved so far 
$\beta (\mathbb P^1 \mathbb R\times E) = \mathcal{H}^1 (K\cap E)$.
In particular, by the classical theorem on disintegration of measures, we can write
$\beta = \beta^x \otimes \mathcal{H}^1 \restr K$,
where $\beta^x$ is a weakly$^\star$ measurable family of probability measures on $\mathbb P^1 \mathbb R$. Thus, \eqref{e:varifold-convergence} is equivalent to $\beta^x = \delta_{T_x K}$ for $\mathcal{H}^1$-a.e. $x\in K$. If the latter conclusion were false, we could then find a point $x_0$ where $\beta^{x_0}\neq \delta_{T_{x_0} K}$, the approximate tangent $T_{x_0} K$ to $K$ at $x_0$ exists, and where the rescaled measures $\beta^{x_0,r}$ defined by
\[
\int \varphi (\pi,y) d \beta^{x_0,r} (\pi,y)= \frac{1}{r} \int \varphi \left(\pi, \frac{y-x_0}{r}\right) d\beta (\pi, y) 
\]
converges to the product measure $\beta^{x_0} \times \mathcal{H}^1 \restr T_{x_0} K$. Observe moreover that, since
\[
\lim_{r\downarrow 0} \frac{1}{r} \int_{B_r (x)} |\nabla u|^2 =0 \qquad \mbox{for $\mathcal{H}^1$-a.e. $x\in K$\,,}
\]
we can assume to have selected $x_0\in K$ such that the last condition holds, as well. 

We can now extract a diagonal sequence and find a new sequence of minimizers $(\tilde{u}_j, \tilde{K}_j)$, with the property that $\tilde{K}_j$ converges to the line $\ell = T_{x_0} K$ uniformly on compact subsets, but the corresponding measures $\tilde{\alpha}_j = \delta_{T_x \tilde{K}_j}\otimes \mathcal{H}^1 \restr \tilde{K}_j$ converge weakly$^\star$ to $\beta^{x_0} \times \mathcal{H}^1 \restr\ell$, with $\beta^{x_0} \neq \delta_{\ell}$. Moreover, up to subsequences, $(\tilde{u}_j, \tilde{K}_j)$ converge to a global generalized minimizer $(\ell, \tilde{u}, p_{12})$. Note that necessarily $|\nabla \tilde{u}|=0$ vanishes identically. Therefore $p_{12} = \pm \infty$ (as it follows from Theorem~\ref{t:class-global} that the global minimizer is a pure jump). 

To fix ideas rotate the coordinates so that $\ell = \{(t,0): t\in \mathbb R\}$. For each $t\in [-2,2]$ consider the segment $\sigma_t := \{(t, s): -1\leq s\leq 1\}$ and define
\[
E_j := \{t\in [-2,2]: \sigma_t \cap \tilde{K}_j \neq \emptyset\}\, .
\]
Since for sufficiently large $j$ the $\tilde{u}_j$ will be harmonic on $[-1,1]\times ([-2, -\frac{1}{2}]\cup [\frac{1}{2}, 2])$ and by Chebyshev there is always a $t\in [-2,2]\setminus E_j$ such that
\[
\int_{\sigma_t} |\nabla u_j|^2 \leq \frac{1}{4-|E_j|} \int_{[-2,2]^2} |\nabla \tilde{u}_j|^2\, ,
\]
we must necessarily have $|E_j|\to 4$, otherwise we would conclude $p_{12} =0$. 

Next for each $x\in \tilde{K}_j$ let $\theta (x)$ be the angle between the lines $\ell$ and $T_x \tilde{K}_j$. 
Recall that, by the generalised coarea formula \cite[Theorem 2.93]{AFP00},
\begin{equation}\label{e:coarea formula}
\mathcal{H}^1 (\tilde{K}_j \cap [-2,2]^2) = \int_{E_j} \sum_{x\in \sigma_t \cap \tilde{K}_j} (\cos \theta (x))^{-1}\, dt\, ,
\end{equation}
while by the conclusions in item (i)
\begin{equation}\label{e: convergenza a 4}
    \lim_{j\to \infty} \mathcal{H}^1 ( \tilde{K}_j \cap [-2,2]^2) =4\,  .
\end{equation}

Define thus $F_j :=\{t\in E_j : \sharp ( \sigma_t \cap \tilde{K}_j) =1 \}$ and $G_j := \tilde{K}_j \cap (F_j\times [-2,2])$. Then 
from \eqref{e:coarea formula} we have
\[
\mathcal{H}^1 (\tilde{K}_j \cap [-2,2]^2 ) \geq |F_j| + \mathcal{H}^1 (\tilde{K}_j \cap ([-2,2]^2\setminus G_j)) 
\geq |F_j| + 2 |E_j\setminus F_j|\, .
\]
In particular, from \eqref{e: convergenza a 4}
and $|E_j|\to 4$, we conclude $|F_j|\to 4$ and $\mathcal{H}^1 ( \tilde{K}_j \cap ([-2,2]^2\setminus G_j)) \to 0$.
Next, for each $\delta >0$ consider the set
\[
H_j := \{x\in G_j : \cos \theta (x) < 1-\delta\} 
\]
and its projection $\pi_\ell (H_j)$ on $\ell$. We then have from \eqref{e:coarea formula}
\[
\mathcal{H}^1 (\tilde{K}_j \cap [-2,2]^2) \geq |F_j\setminus \pi_\ell (H_j)| + \mathcal{H}^1 (H_j) \geq |F_j\setminus \pi_\ell (H_j)| + \frac{1}{1-\delta} |\pi_\ell (H_j)|\, . 
\]
In particular, we again conclude from \eqref{e: convergenza a 4} and $|F_j|\to 4$, that $|\pi_\ell (H_j)|\to 0$ and $\mathcal{H}^1 (H_j) \to 0$. 
Summarizing, for any positive $\eta>0$ we have
\[
\lim_{j\to \infty} \mathcal{H}^1 (\{x\in \tilde{K}_j \cap [-2,2]^2 : |T_x \tilde{K}_j - \ell| > \eta\}) = 0\, .
\]
This however easily implies that $\delta_{T_x \tilde{K}_j} \otimes \mathcal{H}^1 \restr \tilde{K}_j$ converges to $\delta_\ell \times \mathcal{H}^1 \restr \ell$ on the open set $(-2,2)^2$, contradicting our assumption that $\beta^{x_0} \neq \delta_\ell$.
\end{proof}

\chapter{Hirsch's coarea inequality for H\"older maps}

In this section we include an elementary observation by Jonas Hirsch, which leads to a coarea inequality for H\"older maps. The argument is similar to that of \cite[Theorem 2.10.25]{federer} and, in fact, what we need could be concluded directly from the very statement of \cite[Theorem 2.10.25]{federer} by selecting an appropriate target metric space $Y$ in there. However the interesting point is not so much in the argument, but rather in the realization that it is indeed possible to have a coarea inequality for H\"older maps, a fact which we have not seen anywhere in the literature.

\begin{proposition}\label{p:Jonas}
Le $f\in C^\alpha (\mathbb R^m)$ and $A\subset \mathbb R^m$ closed. For every $\beta\geq 0$ there is then a constant $C(\alpha, \beta)$ such that
\begin{equation}\label{e:coarea-Jonas}
\int \mathcal{H}^\beta (f^{-1} (\{t\})\cap A)\, dt \leq C [f]_\alpha \mathcal{H}^{\beta+\alpha} (A)\, .
\end{equation}
\end{proposition}
\begin{proof}
Without loss of generality we assume $\mathcal{H}^{\alpha +\beta} (A)<\infty$. Fix $i\in \mathbb N\setminus \{0\}$ an almost optimal $\frac{1}{i}$ cover $A$ with compact sets $\{B^i_j\}$, i.e.
\begin{align}
{\rm diam}\, (B^i_j) &\leq \frac{1}{i}\\
\sum_j ({\rm diam}\, (B^i_j))^{\alpha+\beta} &\leq C (\alpha, \beta) 
\mathcal{H}^{\alpha+\beta}_{i^{-1}} (A) + \frac{1}{i}\, . 
\end{align}
The functions 
\[
g_j^i (y) := ({\rm diam}\, (B^i_j))^\beta {\mathbf 1}_{f (B^i_j)} (y)
\]
are nonnegative and measurable and so is
\[
g^i (y) := \sum_j g^i_j (y)\, .
\]
So
\begin{align*}
\int g^i (y)\, dy &= \sum_j ({\rm diam}\, (B^i_j))^\beta |f (B^i_j)|
\leq \sum_j [f]_\alpha ({\rm diam}\, (B^i_j))^{\alpha +\beta}\\
&\leq C [f_\alpha] (\mathcal{H}^{\alpha+\beta}_{i^{-1}} (A) + {\textstyle{\frac{1}{i}}})\, .
\end{align*}
Note however that 
\[
\mathcal{H}^\beta_{i^{-1}} (A\cap f^{-1} (\{y\}) \leq g^i (y)\, .
\]
Since $\mathcal{H}^\beta_{i^{-1}} (A\cap f^{-1} (\{y\})\uparrow \mathcal{H}^\beta (A \cap f^{-1} (\{y\})$
and $\mathcal{H}^{\alpha+\beta}_{i^{-1}} (A) \uparrow \mathcal{H}^{\alpha+\beta} (A)$ monotonically, the desired inequality follows from letting $i\uparrow \infty$.
\end{proof}

\backmatter

\bibliography{document}
\bibliographystyle{plain}

\printindex[simb]
\printindex
\end{document}